\documentclass[twoside,11pt]{book}

\usepackage{jmlr2e}
\usepackage{bm}

\usepackage[english]{babel}

\usepackage{imakeidx}
\makeindex[columns=2, title=Alphabetical Index]

\usepackage{stackengine}

\usepackage{enumitem}
\setenumerate[1]{itemsep=0pt,partopsep=0pt,parsep=\parskip,topsep=5pt}
\setitemize[1]{itemsep=0pt,partopsep=0pt,parsep=\parskip,topsep=5pt}
\setdescription{itemsep=0pt,partopsep=0pt,parsep=\parskip,topsep=5pt}

\usepackage{dsfont}

\usepackage[framemethod=tikz]{mdframed}

\usepackage{amsmath}
\usepackage{arydshln}

\usepackage{amsmath}
\usepackage{blkarray}

\usepackage[final]{pdfpages}

\usepackage{color}   
\usepackage[noframe]{showframe}
\usepackage{framed}
\usepackage{lipsum}

{\endMakeFramed}
\definecolor{shadecolor}{gray}{0.75}

\usepackage{xcolor}

\usepackage{tcolorbox}

\usepackage{graphicx}   
\usepackage[hang]{subfigure}

\usepackage{algorithm}
\usepackage{algpseudocode}
\usepackage{amsmath}
\usepackage{graphics}
\usepackage{epsfig}

\usepackage{blkarray}

\usepackage{listings}

\usepackage[Bjornstrup]{fncychap}

\usepackage[colorlinks]{hyperref}
\usepackage{hyperref}
\hypersetup{
    colorlinks=true, 
    linktoc=all,     
    linkcolor=mydarkblue,  
    anchorcolor=blue,
    citecolor=mydarkblue,
}

\usepackage{booktabs}  

\usepackage{changepage}                 

\newlength{\offsetpage}
\newenvironment{widepage}{\begin{adjustwidth}{-1.2cm}{-\offsetpage/2}
		\addtolength{\textwidth}{\offsetpage}}%
	{\end{adjustwidth}}
\newenvironment{widepage1}{\begin{adjustwidth}{-1.0cm}{-\offsetpage/2}
		\addtolength{\textwidth}{\offsetpage}}%
	{\end{adjustwidth}}
\usepackage{tikz}
\usetikzlibrary{mindmap,trees,backgrounds}
\usetikzlibrary{arrows.meta}

\usetikzlibrary{decorations.text}

\usetikzlibrary{calc,backgrounds}

\newcommand*\circled[1]{\tikz[baseline=(char.base)]{
		\node[shape=circle,draw,inner sep=2pt] (char) {#1};}}

\usepackage{hhline}

\usetikzlibrary{positioning,shapes,shadows,arrows}
\tikzstyle{condition}=[rectangle, draw=black, rounded corners, fill=colorqr, drop shadow,
text centered, anchor=north, text=black, text width=3cm]
\tikzstyle{abstract}=[rectangle, draw=black, rounded corners, fill=blue!30, drop shadow,
text centered, anchor=north, text=black, text width=3cm]
\tikzstyle{comment}=[rectangle, draw=black, rounded corners, fill=color1, drop shadow,
text centered, anchor=north, text=black, text width=3cm]
\tikzstyle{myarrow}=[->, >=open triangle 90, thick]
\tikzstyle{line}=[-, thick]

\usepackage{sidecap}
\sidecaptionvpos{figure}{t}
\usepackage{verbatimbox}

\usepackage[labelfont=bf]{caption}
\newcommand\myhrulefill[1]{\leavevmode\leaders\hrule height#1\hfill\kern0pt}
\DeclareCaptionFormat{myformat}{{\color[RGB]{0,0,0}\myhrulefill{0.08em}}\\#1#2#3}
\makeatletter
\renewcommand{\l@section}{\@dottedtocline{1}{1.5em}{2.6em}}
\renewcommand{\l@subsection}{\@dottedtocline{2}{4.0em}{3.6em}}
\renewcommand{\l@subsubsection}{\@dottedtocline{3}{7.4em}{4.5em}}
\makeatother

\numberwithin{equation}{chapter}

\renewcommand\thesection{\thechapter.\arabic{section}}

\usepackage{minitoc}             
\let\cleardoublepage\clearpage
\usepackage[titletoc]{appendix}

\usepackage{fancyhdr, blindtext}

\newcommand{\changefonts}{%
	\fontsize{9}{11}\selectfont
}

\fancypagestyle{plain}{%

	\fancyhf{}  
}

\def\algoalign#1{\parbox[t]{\dimexpr\linewidth-\algorithmicindent}{#1}}

\definecolor{titlepagecolor}{cmyk}{75,68,67,90}
\definecolor{titlepagecolor2}{rgb}{1.0, 0.08, 0.58}
\definecolor{emerald}{rgb}{0.31, 0.78, 0.47}
\definecolor{deeppink}{HTML}{D14064}
\definecolor{lowpink}{HTML}{ffe6ec}
\newcommand{\partcolor}{gray!65} 
\definecolor{lowblue}{HTML}{E1EBFE}

\makeatletter \renewcommand*\cleardoublepage{
	\clearpage
	\if@twoside   
	\ifodd\c@page 
	\hbox{}\newpage
	\if@twocolumn\hbox{}   
	\newpage
	\fi
	\fi
	\fi
} \makeatother
\let\originalpart=\part
\def\part#1{\cleardoublepage\clearpage \pagecolor{\partcolor} \originalpart{#1}\nopagecolor }

\usepackage{yfonts,color,lettrine}
\definecolor{caligraphcolor}{HTML}{B3C2F0}

\usepackage{setspace}

\AtBeginDocument{\setlength{\DefaultFindent}{0.5em}}
\renewenvironment{thebibliography}[1]{
	\begin{oldthebibliography}{#1}
		\setlength{\itemsep}{0.16em}
		\setlength{\parskip}{0.16em}
	}
	{
	\end{oldthebibliography}
}

\newenvironment{sbmatrix}[1]{\def\mysubscript{#1}\mathop\bgroup\begin{bmatrix}}{\end{bmatrix}\egroup_{\textstyle\mathstrut\mysubscript}}

\newenvironment{bmatrixfoot}
{\footnotesize\begin{bmatrix}}
	{\end{bmatrix}\normalsize}
\newenvironment{bmatrixscript}
{\scriptsize\begin{bmatrix}}
	{\end{bmatrix}\normalsize}

\newcommand\comple[1]{#1^C}

\newcommand{\cond}{\text{cond} }

\let\oldforall\forall
\renewcommand{\forall}{\oldforall\, }

\usepackage{color}   

\newcommand{\mdframecolor}{gray!10}
\newcommand{\mdframehideline}{true}

\definecolor{mylightbluetitle}{RGB}{60,113,183}
\definecolor{mylightbluetext}{rgb}{0,0.08,0.45}
\definecolor{structurecolorblue}{RGB}{60,113,183}
\definecolor{structurecolorgreen}{RGB}{63,145,182}
\colorlet{structurecolor}{structurecolorblue}
\definecolor{structurecolorelegant}{RGB}{60,113,183}
\definecolor{structurecolorlt}{RGB}{31,119,185}

\definecolor{colorBlue1}  {RGB}{220,227,248}
\newcommand{\mdframecolorTheorem}{gray!35}  

\definecolor{winestain}{rgb}{0.5,0,0}
\definecolor{mydarkblue}{rgb}{0,0.08,0.45}
\definecolor{mydarkred}{rgb}{0.70,0.00,0.00}
\definecolor{mydarkyellow}{RGB}{197,151,13}
\definecolor{mydarkgreen}{rgb}{0.00,0.30,0.00}
\definecolor{mydarkpurple}{RGB}{90,35,140}
\definecolor{mydarkgray}{RGB}{64,64,64}

\definecolor{color0}  {RGB}{174,225,254} 
\definecolor{color1}  {RGB}{220,227,248} 
\definecolor{color2}  {RGB}{28,130,185} 
\definecolor{color3}  {RGB}{255,253,250} 
\definecolor{colormiddleright}  {RGB}{245,253,250} 
\definecolor{colorbottomleft}  {RGB}{255,243,250} 
\definecolor{coloruppermiddle}  {RGB}{255,253,230} 
\definecolor{colormiddleleft}  {RGB}{255,244,237}
\definecolor{colorcr}  {RGB}{249,253,232} 
\definecolor{colorreduction}  {RGB}{255,235,254} 
\definecolor{colorqr}  {RGB}{254,221,199} 
\definecolor{colorbiconjugate}  {RGB}{251,149,161} 
\definecolor{colorsvd}  {RGB}{215,247,235} 
\definecolor{colorupperright}  {RGB}{239,246,251} 
\definecolor{colorspectral}  {RGB}{206,226,243} 
\definecolor{colorbottomright}  {RGB}{220,224,236} 
\definecolor{coloreigenvalue}  {RGB}{197,203,224} 
\definecolor{colorcp} {RGB}{217, 234, 186} 
\definecolor{colorcpborder} {RGB}{233, 243, 216} 
\definecolor{colorupperleft}  {RGB}{235,243,240} 
\definecolor{colorsemidefinite}  {RGB}{217,232,226} 
\definecolor{colormiddle} {RGB}{235, 240,255}
\definecolor{colorlu}  {RGB}{220,227,255} 
\definecolor{colorals}  {RGB}{240,230,255} 
\definecolor{coloralsbkg}  {RGB}{248,243,255} 
\definecolor{canaryyellow}{rgb}{1.0, 0.75, 0.0}
\definecolor{bluepigment}{rgb}{0.0, 0.0, 1.0}
\definecolor{canarypurple}{RGB}{208, 13, 241}
\definecolor{colorGreenOcre}{RGB}{51,102,0} 
\definecolor{colorBlue2}{RGB}{200,207,248}
\definecolor{shadecolor}{gray}{0.75}

\definecolor{color0}  {RGB}{174,225,254} 
\definecolor{color1}  {RGB}{220,227,248} 
\definecolor{color2}  {RGB}{28,130,185} 
\definecolor{color3}  {RGB}{255,253,250} 
\definecolor{colormiddleright}  {RGB}{245,253,250} 
\definecolor{colorbottomleft}  {RGB}{255,243,250} 
\definecolor{coloruppermiddle}  {RGB}{255,253,230} 
\definecolor{colormiddleleft}  {RGB}{255,253,250}
\definecolor{colorcr}  {RGB}{249,253,232} 
\definecolor{colorreduction}  {RGB}{255,235,254} 
\definecolor{colorqr}  {RGB}{254,221,199} 
\definecolor{colorbiconjugate}  {RGB}{251,149,161} 
\definecolor{colorsvd}  {RGB}{215,247,235} 
\definecolor{colorupperright}  {RGB}{239,246,251} 
\definecolor{colorspectral}  {RGB}{206,226,243} 
\definecolor{colorbottomright}  {RGB}{220,224,236} 
\definecolor{coloreigenvalue}  {RGB}{197,203,224} 
\definecolor{colorupperleft}  {RGB}{235,243,240} 
\definecolor{colorsemidefinite}  {RGB}{217,232,226} 
\definecolor{colormiddle} {RGB}{235, 240,255}
\definecolor{colorlu}  {RGB}{220,227,255} 
\definecolor{colorals}  {RGB}{240,230,255} 
\definecolor{coloralsbkg}  {RGB}{248,243,255} 
\definecolor{canaryyellow}{rgb}{1.0, 0.75, 0.0}
\definecolor{bluepigment}{rgb}{0.0, 0.0, 1.0}
\definecolor{canarypurple}{RGB}{208, 13, 241}
\definecolor{mydarkblue}{rgb}{0,0.08,0.45}
\definecolor{winestain}{rgb}{0.5,0,0}

\newcommand{\hadaprod}{\circledast}

\newcommand{\kronecker}{\otimes}

\newcommand{\leadto}{\qquad\underrightarrow{ \text{leads to} }\qquad}

\newcommand{\gapthree}{\,\,\,}
\newcommand{\gap}{\,\,\,\,\,\,\,\,}  
\mathchardef\mhyphen="2D

\newcommand{\real}{\mathbb{R}}
\newcommand{\complex}{\mathbb{C}}
\newcommand{\integer}{\mathbb{Z}}
\newcommand{\exampbar}{\hfill $\square$\par}
\newcommand{\argmax}{\text{arg max}}
\newcommand{\argmin}{\text{arg min}} 

\newcommand\mathopmax[1]{\mathop{\max}_{#1}}
\newcommand\mathopmin[1]{\mathop{\min}_{#1}}
\newcommand\mathoplim[1]{\mathop{\lim}_{#1}}

\newcommand\abs[1]{\left\lvert#1\right\rvert}
\newcommand\norm[1]{\left\lVert#1\right\rVert}
\newcommand\normtwo[1]{\left\lVert#1\right\rVert_2}
\newcommand\normtwobig[1]{\big\lVert#1\big\rVert_2}

\newcommand\normf[1]{\left\lVert#1\right\rVert_F}

\newcommand{\convd}{\breve d}
\newcommand{\concd}{\widehat d}
\newcommand{\cnstd}{\bar d}

\newcommand{\concG}{\widehat G}

\newcommand{\mathcalC}{\mathcal{C}}

\newcommand{\mathcalO}{\mathcal{O}}
\newcommand{\mathcalP}{\mathcal{P}}

\newcommand{\mathcalV}{\mathcal{V}}

\newcommand{\Cov}{\mathbb{C}\mathrm{ov}}

\newcommand{\Exp}{\mathbb{E}}

\newcommand{\normal}{\mathcal{N}}

\newcommand{\tr}{\mathrm{tr}}

\newcommand{\cspace}{\mathcal{C}}
\newcommand{\nspace}{\mathcal{N}}
\newcommand{\bzero}{\mathbf{0}}
\newcommand{\bone}{\mathbf{1}}
\newcommand{\diag}{\mathrm{diag}}
\newcommand{\rank}{\mathrm{rank}}
\newcommand{\adjugate}{\mathrm{adj}}
\newcommand{\trace}{\mathrm{tr}}

\newcommand{\spn}{\mathrm{span}}

\newcommand{\bOmega}{\boldsymbol\Omega}

\newcommand{\balpha}{\boldsymbol\alpha}

\newcommand{\bdelta}{\boldsymbol\delta}

\newcommand{\bPhi}{\boldsymbol\Phi}
\newcommand{\bPsi}{\boldsymbol\Psi}
\newcommand{\bpsi}{\boldsymbol\psi}
\newcommand{\bgamma}{\boldsymbol\gamma}

\newcommand{\bLambda}{\boldsymbol\Lambda}

\newcommand{\blambda}{\boldsymbol\lambda}
\newcommand{\bphi}{\boldsymbol\phi}

\newcommand{\bmu}{\boldsymbol\mu}
\newcommand{\bsigma}{\boldsymbol\sigma}
\newcommand{\bSigma}{\boldsymbol\Sigma}

\def\sF{{\mathbb{F}}}

\def\sI{{\mathbb{I}}}
\def\sJ{{\mathbb{J}}}
\def\sK{{\mathbb{K}}}
\def\sL{{\mathbb{L}}}

\def\sS{{\mathbb{S}}}

\newcommand{\widebarbW}{\overline{\bm{W}}}

\newcommand{\widebarbZ}{\overline{\bm{Z}}}
\newcommand{\widebarba}{\overline{\bm{a}}}
\newcommand{\widebarbb}{\overline{\bm{b}}}

\newcommand{\widebarbw}{\overline{\bm{w}}}

\newcommand{\widebarbz}{\overline{\bm{z}}}

\newcommand{\widetildebA}{\widetilde{\bm{A}}}

\newcommand{\widetildebC}{\widetilde{\bm{C}}}

\newcommand{\widetildebH}{\widetilde{\bm{H}}}

\newcommand{\widetildebN}{\widetilde{\bm{N}}}

\newcommand{\widetildebP}{\widetilde{\bm{P}}}

\newcommand{\widetildebS}{\widetilde{\bm{S}}}

\newcommand{\widetildebW}{\widetilde{\bm{W}}}
\newcommand{\widetildebX}{\widetilde{\bm{X}}}

\newcommand{\widetildebZ}{\widetilde{\bm{Z}}}
\newcommand{\widetildeba}{\widetilde{\bm{a}}}
\newcommand{\widetildebb}{\widetilde{\bm{b}}}

\newcommand{\widetildebf}{\widetilde{\bm{f}}}

\newcommand{\widetildebj}{\widetilde{\bm{j}}}

\newcommand{\widetildebl}{\widetilde{\bm{l}}}

\newcommand{\widetildebw}{\widetilde{\bm{w}}}
\newcommand{\widetildebx}{\widetilde{\bm{x}}}
\newcommand{\widetildeby}{\widetilde{\bm{y}}}
\newcommand{\widetildebz}{\widetilde{\bm{z}}}

\newcommand{\widetildea}{\widetilde{a}}

\newcommand{\widetildey}{\widetilde{y}}
\newcommand{\widetildez}{\widetilde{z}}

\newcommand{\ba}{\bm{a}}
\newcommand{\bA}{\bm{A}}
\newcommand{\bb}{\bm{b}}
\newcommand{\bB}{\bm{B}}
\newcommand{\bc}{\bm{c}}
\newcommand{\bC}{\bm{C}}
\newcommand{\bd}{\bm{d}}
\newcommand{\bD}{\bm{D}}
\newcommand{\be}{\bm{e}}
\newcommand{\bE}{\bm{E}}
\newcommand{\bff}{\bm{f}}
\newcommand{\bF}{\bm{F}}
\newcommand{\bg}{\bm{g}}
\newcommand{\bG}{\bm{G}}
\newcommand{\bh}{\bm{h}}
\newcommand{\bH}{\bm{H}}

\newcommand{\bI}{\bm{I}}
\newcommand{\bj}{\bm{j}}
\newcommand{\bJ}{\bm{J}}

\newcommand{\bK}{\bm{K}}
\newcommand{\bl}{\bm{l}}
\newcommand{\bL}{\bm{L}}
\newcommand{\bmm}{\bm{m}}
\newcommand{\bM}{\bm{M}}
\newcommand{\bn}{\bm{n}}
\newcommand{\bN}{\bm{N}}
\newcommand{\bo}{\bm{o}}

\newcommand{\bp}{\bm{p}}
\newcommand{\bP}{\bm{P}}
\newcommand{\bq}{\bm{q}}
\newcommand{\bQ}{\bm{Q}}
\newcommand{\br}{\bm{r}}
\newcommand{\bR}{\bm{R}}
\newcommand{\bs}{\bm{s}}
\newcommand{\bS}{\bm{S}}

\newcommand{\bT}{\bm{T}}
\newcommand{\bu}{\bm{u}}
\newcommand{\bU}{\bm{U}}
\newcommand{\bv}{\bm{v}}
\newcommand{\bV}{\bm{V}}
\newcommand{\bw}{\bm{w}}
\newcommand{\bW}{\bm{W}}
\newcommand{\bx}{\bm{x}}
\newcommand{\bX}{\bm{X}}
\newcommand{\by}{\bm{y}}
\newcommand{\bY}{\bm{Y}}
\newcommand{\bz}{\bm{z}}
\newcommand{\bZ}{\bm{Z}}

\newcommand{\brx}{\bm{R}_\gamma^{(x)}}
\newcommand{\bry}{\bm{R}_\gamma^{(y)}} 

\newcommand{\bryt}{\bm{R}_\gamma^{(y)\top}} 
\newcommand{\bxgamma}{\bX_\gamma}
\newcommand{\bygamma}{\bY_\gamma}
\newcommand{\bugamma}{\bU_\gamma}
\newcommand{\bvgamma}{\bV_\gamma}
\newcommand{\bomegagamma}{\bOmega_\gamma}
\newcommand{\svd}{\mathrm{svd}}

\DeclareMathAlphabet{\mathsfit}{\encodingdefault}{\sfdefault}{m}{sl}
\SetMathAlphabet{\mathsfit}{bold}{\encodingdefault}{\sfdefault}{bx}{n}

\def\sF{{\mathbb{F}}}

\def\sI{{\mathbb{I}}}
\def\sJ{{\mathbb{J}}}
\def\sK{{\mathbb{K}}}
\def\sL{{\mathbb{L}}}

\def\sS{{\mathbb{S}}}

\def\rz{{\textnormal{z}}}

\def\rvx{{\mathbf{x}}}

\def\rvz{{\mathbf{z}}}

\usepackage[noframe]{showframe}
\usepackage{framed}
\usepackage{lipsum}

\newcommand{\roundcornertheorem}{0pt}
\newcommand{\linewidththeorem}{0.1pt}
\newcommand{\frametitlerulewidththeorem}{0.1pt}
\newcommand{\innerbottommargintheorem}{2pt}
\newcommand{\innerleftmargintheorem}{2pt}
\newcommand{\innerrightmargintheorem}{2pt}
\newcommand{\innertopmargintheorem}{2pt}
\newcommand{\outerlinewidththeorem}{1pt}

\newenvironment{theoremHigh}[1][]{%
\refstepcounter{theo}%
\ifstrempty{#1}%
{\newcommand{\theoName}{}}
{\newcommand{\theoName}{:~(#1)}}
\mdfsetup{
	backgroundcolor=colorBlue1,
	linecolor=colorBlue1,
	frametitlerulewidth=\frametitlerulewidththeorem,
	roundcorner=\roundcornertheorem,
	linewidth=\linewidththeorem,
	innerbottommargin=\innerbottommargintheorem,
	innerleftmargin=\innerleftmargintheorem,
	innerrightmargin=\innerrightmargintheorem,
	innertopmargin=\innertopmargintheorem,
	outerlinewidth=\outerlinewidththeorem,
	topline=false,
	innertopmargin=-5pt,
	innerbottommargin=1pt,
	linewidth=0,
	startinnercode=\paragraph{{\strut Theorem~\thetheo\theoName}}
}
\begin{mdframed}[]\relax%
}{\end{mdframed}}
\newenvironment{corollaryHigh}[1][]{%
\refstepcounter{theo}%
\ifstrempty{#1}%
{\newcommand{\theoName}{}}
{\newcommand{\theoName}{:~(#1)}}
\mdfsetup{
	backgroundcolor=colorBlue1,
	linecolor=colorBlue1,
	frametitlerulewidth=\frametitlerulewidththeorem,
	roundcorner=\roundcornertheorem,
	linewidth=\linewidththeorem,
	innerbottommargin=\innerbottommargintheorem,
	innerleftmargin=\innerleftmargintheorem,
	innerrightmargin=\innerrightmargintheorem,
	innertopmargin=\innertopmargintheorem,
	outerlinewidth=\outerlinewidththeorem,
	topline=false,
	innertopmargin=-5pt,
	innerbottommargin=1pt,
	linewidth=0,
	startinnercode=\paragraph{{\strut Corollary~\thetheo\theoName}}
}
\begin{mdframed}[]\relax%
}{\end{mdframed}}

\newenvironment{theorem}[1][]{%
	\refstepcounter{theo}%
	\ifstrempty{#1}%
	{\newcommand{\theoName}{}}
	{\newcommand{\theoName}{:~(#1)}}
	\mdfsetup{
		backgroundcolor=\mdframecolorTheorem,
		linecolor=\mdframecolorTheorem,
		frametitlerulewidth=\frametitlerulewidththeorem,
		roundcorner=\roundcornertheorem,
		linewidth=\linewidththeorem,
		innerbottommargin=\innerbottommargintheorem,
		innerleftmargin=\innerleftmargintheorem,
		innerrightmargin=\innerrightmargintheorem,
		innertopmargin=\innertopmargintheorem,
		outerlinewidth=\outerlinewidththeorem,
		topline=false,
		innertopmargin=-5pt,
		innerbottommargin=1pt,
		linewidth=0,
		startinnercode=\paragraph{{\strut Theorem~\thetheo\theoName}}
	}
	\begin{mdframed}[]\relax%
	}{\end{mdframed}}
\newenvironment{corollary}[1][]{%
	\refstepcounter{theo}%
\ifstrempty{#1}%
{\newcommand{\theoName}{}}
{\newcommand{\theoName}{:~(#1)}}
\mdfsetup{
	backgroundcolor=\mdframecolorTheorem,
	linecolor=\mdframecolorTheorem,
	frametitlerulewidth=\frametitlerulewidththeorem,
	roundcorner=\roundcornertheorem,
	linewidth=\linewidththeorem,
	innerbottommargin=\innerbottommargintheorem,
	innerleftmargin=\innerleftmargintheorem,
	innerrightmargin=\innerrightmargintheorem,
	innertopmargin=\innertopmargintheorem,
	outerlinewidth=\outerlinewidththeorem,
	topline=false,
	innertopmargin=-5pt,
	innerbottommargin=1pt,
	linewidth=0,
	startinnercode=\paragraph{{\strut Corollary~\thetheo\theoName}}
}
	\begin{mdframed}[]\relax%
	}{\end{mdframed}}
\newenvironment{lemma}[1][]{%
	\refstepcounter{theo}%
\ifstrempty{#1}%
{\newcommand{\theoName}{}}
{\newcommand{\theoName}{:~(#1)}}
\mdfsetup{
	backgroundcolor=\mdframecolorTheorem,
	linecolor=\mdframecolorTheorem,
	frametitlerulewidth=\frametitlerulewidththeorem,
	roundcorner=\roundcornertheorem,
	linewidth=\linewidththeorem,
	innerbottommargin=\innerbottommargintheorem,
	innerleftmargin=\innerleftmargintheorem,
	innerrightmargin=\innerrightmargintheorem,
	innertopmargin=\innertopmargintheorem,
	outerlinewidth=\outerlinewidththeorem,
	topline=false,
	innertopmargin=-5pt,
	innerbottommargin=1pt,
	linewidth=0,
	startinnercode=\paragraph{{\strut Lemma~\thetheo\theoName}}
}
\begin{mdframed}[]\relax%
}{\end{mdframed}}

\newenvironment{proposition}[1][]{%
\refstepcounter{theo}%
\ifstrempty{#1}%
{\newcommand{\theoName}{}}
{\newcommand{\theoName}{:~(#1)}}
\mdfsetup{
	backgroundcolor=\mdframecolorTheorem,
	linecolor=\mdframecolorTheorem,
	frametitlerulewidth=\frametitlerulewidththeorem,
	roundcorner=\roundcornertheorem,
	linewidth=\linewidththeorem,
	innerbottommargin=\innerbottommargintheorem,
	innerleftmargin=\innerleftmargintheorem,
	innerrightmargin=\innerrightmargintheorem,
	innertopmargin=\innertopmargintheorem,
	outerlinewidth=\outerlinewidththeorem,
	topline=false,
	innertopmargin=-5pt,
	innerbottommargin=1pt,
	linewidth=0,
	startinnercode=\paragraph{{\strut Proposition~\thetheo\theoName}}
}
\begin{mdframed}[]\relax%
}{\end{mdframed}}

\usepackage{amsthm}
\newtheoremstyle{normalfontstyle} 
{3pt}                           
{3pt}                           
{\normalfont}                   
{}                              
{\bfseries}                     
{}                             
{ }                             
{}                              

\usepackage{thmtools}
\declaretheoremstyle[
spaceabove=3pt,
spacebelow=3pt,
headfont=\bfseries,
notefont=\bfseries, 
notebraces={(}{)}, 
bodyfont=\normalfont,
postheadspace=1em, 
]{normalfontboldhead}

\theoremstyle{normalfontstyle}
\newcommand{\BlackBox}{\rule{1.5ex}{1.5ex}}  
\renewenvironment{proof}{\par\noindent{\bf Proof\ }}{\hfill\BlackBox\\[2mm]}

\declaretheorem[style=normalfontboldhead, name=Definition, numberlike=theo]{definitionT}
\newmdenv[skipabove=7pt,
skipbelow=7pt,
rightline=false,
leftline=true,
topline=false,
bottomline=false,
linecolor=mydarkblue,
innerleftmargin=5pt,
innerrightmargin=5pt,
innertopmargin=0pt,
leftmargin=2cm,
rightmargin=0cm,
linewidth=4pt,
innerbottommargin=0pt]{dBox}
\newenvironment{definition}{\begin{dBox}\begin{definitionT}}{\end{definitionT}\end{dBox}}

\declaretheorem[style=normalfontboldhead, name=Exercise, numberlike=theo]{exerciseC}
\newmdenv[skipabove=7pt,
skipbelow=7pt,
rightline=false,
leftline=true,
topline=false,
bottomline=false,
linecolor=mydarkgreen,
innerleftmargin=5pt,
innerrightmargin=5pt,
innertopmargin=0pt,
leftmargin=2cm,
rightmargin=0cm,
linewidth=4pt,
innerbottommargin=0pt]{eBox}
\newenvironment{exercise}{\begin{eBox}\begin{exerciseC}}{\end{exerciseC}\end{eBox}}

\declaretheorem[style=normalfontboldhead, name=Remark, numberlike=theo]{remarekC}
\newmdenv[skipabove=7pt,
skipbelow=7pt,
rightline=false,
leftline=true,
topline=false,
bottomline=false,
linecolor=mydarkgreen,
innerleftmargin=5pt,
innerrightmargin=5pt,
innertopmargin=0pt,
leftmargin=2cm,
rightmargin=0cm,
linewidth=4pt,
innerbottommargin=0pt]{rBox}
\newenvironment{remark}{\begin{rBox}\begin{remarekC}}{\end{remarekC}\end{rBox}}

\declaretheorem[style=normalfontboldhead, name=Example, numberlike=theo]{exampleC}
\newmdenv[skipabove=7pt,
skipbelow=7pt,
rightline=false,
leftline=false,
topline=false,
bottomline=false,
linecolor=mydarkgreen,
innerleftmargin=1pt,
innerrightmargin=5pt,
innertopmargin=0pt,
leftmargin=2cm,
rightmargin=0cm,
linewidth=4pt,
innerbottommargin=0pt]{xBox}
\newenvironment{example}{\begin{xBox}\begin{exampleC}}{\exampbar\end{exampleC}\end{xBox}}

\usepackage{adforn}  
\newcommand{\xchaptertitle}{Chapter~\thechapter~}
\newcommand{\problemname}{Problems}
\newenvironment{problemset}[1][\xchaptertitle~\problemname]{
	\vspace*{10pt}
	\begin{center}
		\phantomsection\addcontentsline{toc}{section}{\texorpdfstring{\xchaptertitle~\problemname}{\problemname}}
		\markright{#1}
		\textcolor{structurecolor}{\Large\bfseries\adftripleflourishleft~#1~\adftripleflourishright}
	\end{center}
	\begin{enumerate}[ref=\thechapter.\theenumi]}{
\end{enumerate}}

\clearpage
\nopagecolor

\newcommand{\mytitle}{Matrix Decomposition and Applications}

\begin{document}
\newpage

\thispagestyle{empty}  
\title{\mytitle}

\author{
\begin{center}
\name Jun Lu \\ 
\email jun.lu.locky@gmail.com \\
\end{center}
}

\frontmatter

\newpage
\maketitle

\chapter*{\centering \begin{normalsize}Preface\end{normalsize}}
The realm of matrices is as vast as it is indispensable, with applications spanning from the minutiae of quantum systems to the expansive challenges of large-scale data analytics. At the heart of matrix analysis lies the transformative process of matrix decomposition---a method of reducing a complex matrix into simpler, constituent parts that illuminate its structure and utility. Far from being merely an abstract mathematical concept, matrix decomposition has become a cornerstone in fields as diverse as computer science, engineering, physics, and economics.

At its essence, matrix decomposition simplifies the representation and manipulation of matrices by breaking them down into manageable components. This process enables efficient solutions to linear systems, reduces computational complexity, and provides insights into data's inherent structure. Its applications are far-reaching, influencing everything from machine learning and optimization to image processing and recommender systems.

The historical roots of matrix decomposition trace back to the foundational work of Alston S. Householder in the mid-20th century, which set the stage for modern numerical analysis. Over the decades, the field has seen tremendous advancements, including innovations like backpropagation for neural networks, dimensionality reduction techniques in machine learning, and the utilization of low-rank matrices in natural language processing and large language models.

Today, matrix decomposition underpins technologies in statistics, optimization, and artificial intelligence. It is fundamental to the functioning of algorithms in deep neural networks, recommendation systems, and high-dimensional data analysis, among others. These applications not only underscore its practical significance but also highlight the evolving complexity of its theoretical underpinnings.

This book seeks to serve as a comprehensive and accessible introduction to matrix decomposition, offering readers a bridge between theoretical concepts and practical applications. It is designed for readers with a foundational understanding of linear algebra and aims to achieve the following objectives:
\begin{itemize}
\item \textit{Explore core principles.} Present the mathematical foundations of matrix decomposition, ranging from basic methods such as LU, Cholesky, and QR decomposition to advanced techniques like SVD, eigenvalue decomposition, and their modern extensions.
\item \textit{Highlight practical applications.} Demonstrate the relevance of decomposition methods in diverse fields, including optimization, machine learning, neural network compression, and data interpretation.
\item \textit{Facilitate problem-solving.} Equip readers with the tools to understand and solve problems involving matrices, emphasizing how decomposition can simplify complex tasks and provide deeper insights.
\end{itemize}

\vspace{5em}
\noindent
\textit{Keywords: }
Existence and computing of matrix decompositions, Low-rank approximation, Pivot, LU decomposition for nonzero leading principal minors, Data distillation, CR decomposition, CUR/Skeleton decomposition, Interpolative decomposition, Biconjugate decomposition, Coordinate transformation, Hessenberg decomposition, ULV decomposition, URV decomposition, Rank decomposition, Gram--Schmidt process, Householder reflector, Givens rotation, Rank-revealing decomposition, Cholesky decomposition and update/downdate, Eigenvalue problems, Alternating least squares.

\vspace{5em}
\noindent
\textit{Acknowledgement: }
We extend our deepest gratitude to Gilbert Strang for posing the problem articulated in Corollary~\ref{corollary:invertible-intersection}, reviewing the manuscript, and providing invaluable insights and references on the three factorizations derived from elimination steps. We are especially thankful for his generosity in sharing the manuscript of \citet{strang2021three}, which greatly enriched our understanding of the subject.
We also extend our heartfelt appreciation to the anonymous professors who offered their consultation, feedback, and expressed interest in adopting this book as course material for college-level instruction.
The author also acknowledges the collaborative contributions of Joerg Osterrieder, Christine P. Chai, and Xuanyu Ye in developing the Bayesian approach for nonnegative matrix factorization and (intervened) interpolative decomposition. Their work has significantly illuminated the structure and content of several sections in this book, providing critical perspectives and innovative methodologies.

\input{chapter-worldmap.tex}

\newpage
\begingroup
\hypersetup{
linkcolor=mydarkblue,
linktoc=page,  
}
\dominitoc
\pdfbookmark{\contentsname}{toc} 

\setcounter{tocdepth}{1}  
\tableofcontents 
\endgroup

\mainmatter

\newpage
\section*{Introduction and Background}
\addcontentsline{toc}{section}{Introduction and Background}

Matrix decomposition is a cornerstone of modern numerical linear algebra, with applications in diverse fields such as statistics  \citep{banerjee2014linear, gentle1998numerical}, optimization \citep{gill2021numerical}, and  machine learning \citep{goodfellow2016deep, bishop2006pattern}, particularly in deep learning.
As an essential computational framework, it simplifies complex matrix operations by breaking a matrix into more manageable components. This approach is critical not only for theoretical insights but also for practical implementations, enabling efficient algorithms and enhancing interpretability.

The prominence of matrix decomposition techniques is largely due to advances like the backpropagation algorithm for neural network training and the use of low-rank neural networks in efficient deep learning architectures \citep{lu2025large}.
The primary goal of this book is to provide a self-contained introduction to the concepts and mathematical tools of  linear algebra and matrix analysis, laying a solid foundation for understanding matrix decomposition techniques and their applications in subsequent sections.
This book explores the fundamental techniques and applications of matrix decomposition. It begins with foundational methods such as LU and Cholesky decomposition, which are integral to solving linear systems and understanding positive definiteness. It then delves into more advanced topics, including QR decomposition, spectral decomposition, and singular value decomposition (SVD), which have broad applications ranging from eigenvalue problems to low-rank approximations and data compression.
This introduction is designed for readers with a foundational knowledge of linear algebra and aims to bridge the gap between theory and application, equipping them with the necessary tools to navigate this critical area of numerical mathematics.

However, we clearly realize our inability to cover all the useful and interesting results concerning matrix decomposition. Given the scope limitations, topics such as the analysis of Euclidean space, Hermitian space, and Hilbert space are not addressed in detail here. For a more comprehensive introduction to these areas, readers are encouraged to consult the literature on linear algebra, including works such as \citet{trefethen1997numerical, strang1993introduction, stewart2000decompositional, gentle2007matrix, higham2002accuracy, quarteroni2010numerical, golub2013matrix, beck2017first, gallier2017fundamentals, boyd2018introduction, strang2019linear, van2020advanced, strang2021every}. 
It is important to note that this book specifically focuses on providing compact proofs for the existence of various matrix decomposition methods. For a more in-depth exploration of topics such as reducing computational complexity, detailed discussions of applications, and insights into tensor decomposition, readers are encouraged to refer to \citet{lu2021numerical}.

A matrix decomposition involves breaking down a complex matrix into its constituent parts, simplifying its representation.
The underlying principle of this approach is that, rather than solving specific problems directly, matrix algorithms focus on simplifying more complex matrix operations. These operations can be performed on the decomposed components, rather than the original matrix itself. At a general level, a matrix decomposition task for a matrix $\bA$ can be formulated as follows:
\begin{itemize}
\item $\bA=\bQ\bU$: Here, $\bQ$ is an orthogonal matrix that contains the same column space as $\bA$, while $\bU$ is a relatively simple and sparse matrix used to reconstruct $\bA$.
\item $\bA=\bQ\bT\bQ^\top$: In this case, $\bQ$ is orthogonal such that $\bA$ and $\bT$ are \textit{similar matrices} \footnote{See Definition~\ref{definition:similar-matrices}  for a rigorous definition.} that share essential properties such as  eigenvalues and sparsity. Additionally, working with $\bT$ is computationally simpler than working with $\bA$.
\item $\bA=\bU\bT\bV$: In this formulation, $\bU$ and $\bV$ are orthogonal matrices such that the columns of $\bU$ and the rows of $\bV$ form  orthonormal bases for the column space and row space of $\bA$, respectively.
\item $\underset{m\times n}{\bA}=\underset{m\times r}{\bB} \gapthree \underset{r\times n}{\bC}$: Here, $\bB$ and $\bC$ are full-rank matrices capable of reducing the memory storage  requirements for  $\bA$. 
In practical applications, 
a low-rank approximation, $\underset{m\times n}{\bA}\approx \underset{m\times k}{\bD}\gapthree \underset{k\times n}{\bF}$, where $k<r$ is the \textit{numerical rank} of the matrix, proves beneficial. This approximation allows for more efficient storage of the matrix $\bA$, requiring only $k(m+n)$ floats instead of $mn$ numbers. Additionally, it facilitates the efficient computation of matrix-vector products, $\bb = \bA\bx$, through intermediate steps involving $\bc = \bF\bx$ and $\bb = \bD\bc$. This approximation method is also valuable for data interpretation and other computational tasks.

\item Although typically computationally demanding, a matrix decomposition can be leveraged to solve  new problems related to the original matrix in various contexts. For instance, once the factorization of $\bA$ is obtained, it can be reused to solve a set of linear systems: $\bb_1=\bA\bx_1, \bb_2=\bA\bx_2, \ldots, \bb_k=\bA\bx_k$.

\item More generally,  matrix decomposition aids in understanding the internal structure and logic of operations involving matrix multiplication. Each component of the decomposition contributes to a geometrical transformation, as discussed in Section~\ref{section:coordinate-transformation}.
\end{itemize}

Matrix decomposition algorithms can be classified into several categories. Below are six fundamental types:
\begin{enumerate}
\item Factorizations based on Gaussian elimination, such as LU decomposition and its positive definite counterpart, Cholesky decomposition.
\item Factorizations achieved by orthogonalizing either the columns or rows of a matrix, enabling effective data representation in an orthonormal basis.
\item Factorizations involving skeleton matrices, where a subset of columns or rows can sufficiently represent the entire dataset with minimal reconstruction error, while preserving sparsity and nonnegativity.
\item Reduction to Hessenberg, tridiagonal, or bidiagonal forms, allowing the properties of the matrix (such as rank and eigenvalues) to be explored within these reduced forms.
\item Factorizations derived from the computation of matrix eigenvalues.
\item Other specialized methods, which involve optimization techniques and high-level concepts. These may not fit neatly into the categories above but still represent important classes of decompositions.
\end{enumerate}

The visual representations of matrix decomposition in Figures~\ref{fig:matrix-decom-world-picture} and \ref{fig:matrix-decom-world-picture2} illustrate the connections between various decomposition methods based on their underlying relationships. These figures also distinguish the methods according to specific criteria or prerequisites. Further details about these visualizations are provided in the accompanying text.

\subsection*{Objectives of This Work}
This book aims to provide a comprehensive yet accessible introduction to the principles, methods, and applications of matrix decomposition. Designed for readers with a foundational understanding of linear algebra, it bridges the gap between theoretical rigor and practical applications. The goals include:

\begin{enumerate}
\item Presenting core concepts: Introducing the mathematical foundations of matrix decomposition, including LU, Cholesky, QR, and SVD, along with more advanced methods like eigenvalue and Jordan decompositions.
\item Highlighting applications: Demonstrating the relevance of these techniques in various domains, such as optimization, machine learning, and signal processing.
\item Providing rigorous proofs: Ensuring that the presented methods are mathematically rigorous, with proofs and derivations to deepen understanding.

\end{enumerate}

\paragraph{Notation and preliminaries.}
In the remainder of this section, we  introduce and review fundamental concepts from linear algebra. We will also introduce additional important notions as necessary to ensure clarity.  
Throughout the text, our focus will be on real matrices. Unless otherwise specified, the eigenvalues of the matrices under discussion are assumed to be real as well.

Scalars are represented in non-bold font, potentially with subscripts (e.g., $a$, $\alpha$, $\alpha_i$).
Vectors are denoted using \textbf{boldface} lowercase letters, possibly with subscripts (e.g., $\bmu$, $\bx$, $\bx_n$, $\bz$), while matrices are represented by \textbf{boldface} uppercase letters, possibly with subscripts (e.g., $\bA$, $\bL_j$). 
The $i$-th element of a vector $\bz$  is written as $z_i$ in non-bold font. 
For a matrix $\bA$, the value in the $i$-th row and $j$-th column  is represented as  $a_{ij}$. 
Additionally, we also adopt \textbf{Matlab-style notation}; the submatrix of  $\bA$ from the $i$-th  to $j$-th rows and $k$-th  to $m$-th columns is denoted by $\bA_{i:j,k:m}=\bA[i:j,k:m]$. 
When the indices are not continuous, with ordered subindex sets $\sI$ and $\sJ$, $\bA[\sI, \sJ]$ indicates the submatrix of $\bA$ obtained by extracting the rows and columns indexed by $\sI$ and $\sJ$, respectively. Similarly, $\bA[:, \sJ]$ denotes the submatrix of $\bA$ obtained by extracting the columns of $\bA$ indexed by $\sJ$.

\index{Matlab-style notation}

All vectors are represented in column format rather than row format. A row vector is indicated by the transpose of a column vector, e.g., denoted by $\ba^\top$.
A specific column vector with values is delineated by the semicolon  symbol $``;"$, for example, $\bx=[1;2;3]$ is a column vector in $\real^3$. 
Similarly, a row vector with specific values is separated by commas, e.g., $\by=[1,2,3]$ is a row vector with three values. 
Furthermore, a column vector can be expressed as the transpose of a row vector, for instance, $\by=[1,2,3]^\top$ is a column vector.

The transpose of a matrix $\bA$ is denoted by $\bA^\top$, and its inverse is denoted by $\bA^{-1}$. 
The $p \times p$ identity matrix is denoted by $\bI_p$. 
A vector or matrix consisting entirely of zeros is denoted by the \textbf{boldface} zero, $\bzero$, with its size inferred from context. 
Specifically, $\bzero_p$ signifies a vector of all zeros with $p$ entries, and $\bzero_{p\times q}$ represents a matrix of all zeros with dimensions $p\times q$.



\index{Eigenvalue}
\index{Eigenvector}
\begin{definition}[Eigenvalue and eigenvector]
Given any vector space $\sF$ and any linear map $\bA: \sF \rightarrow \sF$ (or simply a real matrix $\bA\in\real^{n\times n}$), a scalar $\lambda \in \sK$ is called a \textit{(right) eigenvalue, or proper value, or characteristic value} of $\bA$, if there exists a nonzero vector $\bu \in \sF$ such that
\begin{equation*}
\bA \bu = \lambda \bu.
\end{equation*}
And $\bu$ is called a \textit{(right) eigenvector} of $\bA$ associated with $\lambda$.

On the other hand, $\kappa$ is referred to as a \textit{left eigenvalue} if there exists a nonzero vector $\bv\in \sF$ such that 
$$
\bv^\top\bA = \kappa \bv^\top.
$$
And $\bv$ is called a \textit{(left) eigenvector} of $\bA$ associated with $\kappa$.

When it is clear from the context, we will simply use the term ``eigenvalue/eigenvector" instead of ``right eigenvalue/eigenvector."

\end{definition}
In simple terms, an eigenvector $\bu$ of a matrix $\bA$ represents a direction that remains unchanged when transformed into the coordinate system defined by the columns of $\bA$ (see Section~\ref{section:coordinate-transformation} for more details on coordinate transformations).
In fact, real-valued matrices can have complex eigenvalues. However, all the eigenvalues of symmetric matrices are real (see Theorem~\ref{theorem:spectral_theorem}).

\index{Spectrum}
\index{Spectral radius}
\begin{definition}[Spectrum and spectral radius]\label{definition:spectrum}
The set of all eigenvalues of $\bA$ is called the \textit{spectrum} of $\bA$ and is denoted by $\Lambda(\bA)$. The largest magnitude of the eigenvalues is known as the \textit{spectral radius} $\rho(\bA)$:
$$
\rho(\bA) = \mathop{\max}_{\lambda\in \Lambda(\bA)}  \abs{\lambda}.
$$
\end{definition}

Moreover, the pair $(\lambda, \bu)$ mentioned above is commonly referred to as an \textit{eigenpair}. Intuitively, the  above definitions  indicate that multiplying the matrix $\bA$ by the vector $\bu$ yields a new vector that lies in the same direction as $\bu$, but scaled by a factor $\lambda$. 
For any eigenvector $\bu$, it can be scaled by a scalar $s$ such that $s\bu$ remains an eigenvector of $\bA$.  
This is why we refer to $\bu$ as an eigenvector of $\bA$ associated with the eigenvalue $\lambda$. 
To avoid any ambiguity, we usually assume that the eigenvector is normalized to have unit length, and its first entry is positive, since both $\bu$ and $-\bu$ are valid eigenvectors.

In linear algebra, it is a fundamental property that every vector space has a basis. Any vector in the space can be expressed as a linear combination of the basis vectors. Using this concept, we define the \textit{span} and \textit{dimension} of a subspace in terms of its basis.
\index{Subspace}
\begin{definition}[Subspace]
A nonempty subset $\mathcal{V}$ of $\real^n$ is called a \textit{subspace} if for all $\ba,\bb\in \mathcal{V}$ and all $x,y\in \real$, the linear combination  $x\ba+y\ba$ also belongs to $\mathcalV$.
\end{definition}

\index{Span}
\begin{definition}[Span]
If every vector $\bv$ in a subspace $\mathcal{V}$ can be expressed as a linear combination of the vectors $\{\ba_1, \ba_2, \ldots,$ $\ba_m\}$, then the set $\{\ba_1, \ba_2, \ldots, \ba_m\}$ is said to span $\mathcal{V}$.
\end{definition}

\index{Linearly independent}
In linear algebra, the concept of linear independence is fundamental when studying sets of vectors. Two equivalent definitions are provided below.
\begin{definition}[Linearly independent]
A set of vectors $\{\ba_1, \ba_2, \ldots, \ba_m\}$ is called \textit{linearly independent} if the equation $x_1\ba_1+x_2\ba_2+\ldots+x_m\ba_m=\bzero$ has only the trivial solution where all scalars $x_i=0$. 
An equivalent definition is that $\ba_1\neq \bzero$, and for every $k>1$, the vector $\ba_k$ does not belong to the span of the preceding vectors $\{\ba_1, \ba_2, \ldots, \ba_{k-1}\}$.
\end{definition}

\begin{exercise}
Show that the columns of the $m \times n$ matrix $\bA$ are linearly independent if and only if $f(\bx) = \bA\bx$ is a one-to-one function.
\end{exercise}

\index{Basis}
\index{Dimension}
\begin{definition}[Basis and dimension]
A set of vectors $\{\ba_1, \ba_2, \ldots, \ba_m\}$ is called a \textit{basis} of a subspace $\mathcal{V}$ if they are linearly independent and span $\mathcal{V}$. All bases of a given subspace contain the same number of vectors, and this common number of vectors in any basis is called the \textit{dimension} of the subspace $\mathcal{V}$. 

By convention, the subspace containing only the zero vector, $\{\bzero\}$, has dimension zero. Furthermore, every nonzero subspace has a basis consisting of mutually \textit{orthogonal} vectors (i.e., the vectors in the basis are mutually perpendicular).
\end{definition}

\index{Column space}
\index{Range}
\begin{definition}[Column space (range)]
For  an $m \times n$ real matrix $\bA$, the \textit{column space (or range)} of $\bA$ is defined as the set of all linear combinations of its columns:
\begin{equation*}
\cspace (\bA) = \{ \by\in \real^m: \exists\, \bx \in \real^n, \, \by = \bA \bx \}.
\end{equation*}
Similarly, the \textit{row space} of $\bA$ is the set of all linear combinations of its rows, which is equal to the column space of the transpose $\bA^\top$:
\begin{equation*}
\cspace (\bA^\top) = \{ \bx\in \real^n: \exists\, \by \in \real^m, \, \bx = \bA^\top \by \}.
\end{equation*}
\end{definition}

\index{Null space (nullspace, kernel)}
\begin{definition}[Null space (nullspace, kernel)]
For an $m \times n$ real matrix $\bA$, the \textit{null space} (also called the \textit{kernel or nullspace}) of  $\bA$ is the set of all vectors in $\real^n$ that satisfy:
\begin{equation*}
\nspace (\bA) = \{\by \in \real^n:  \, \bA \by = \bzero \}.
\end{equation*}
Similarly, the null space of $\bA^\top$ (i.e., the \textit{left null space} of $\bA$) is defined as 	
\begin{equation*}
\nspace (\bA^\top) = \{\bx \in \real^m:  \, \bA^\top \bx = \bzero \}.
\end{equation*}
\end{definition}

Both the column space of $\bA$ and the null space of $\bA^\top$ are subspaces of $\real^n$. Moreover, every vector in $\nspace(\bA^\top)$ is orthogonal to $\cspace(\bA)$, and vice versa; similarly, every vector in $\nspace(\bA)$ is also orthogonal to $\cspace(\bA^\top)$, and vice versa.

\index{Rank}
\begin{definition}[Rank]
The $rank$ of a matrix $\bA\in \real^{m\times n}$ is the dimension of the column space of $\bA$. That is, the rank of $\bA$ is equal to the maximum number of linearly independent columns of $\bA$, and is also the maximum number of linearly independent rows of $\bA$. The rank of  $\bA$  is equal to the rank of its transpose, $\bA^\top$. Additionally, $\bA$ is said to have full rank if its rank equals $\min\{m,n\}$. 
Specifically, given a vector $\bu \in \real^m$ and a vector $\bv \in \real^n$, then the $m\times n$ matrix $\bu\bv^\top$ is of rank 1. In short, the rank of a matrix is equal to:
\begin{itemize}
\item the number of linearly independent columns;
\item the number of linearly independent rows;
\item and remarkably, these two quantities are always equal (see Theorem~\ref{theorem:equal-dimension-rank}).
\end{itemize}
\end{definition}

\begin{exercise}[Rank of matrix addition]\label{exercise:rk_ad}
Let $\bA$ and $\bB$ be two matrices with ranks $a$ and $b$, respectively. Show that the rank of $\bA + \bB$ is at most $a + b$ and at least $\abs{a-b}$.
\end{exercise}

\begin{exercise}[Rank of matrix multiplication, a.k.a., Sylvester's inequality]\label{exercise:rk_prod}
Let $\bA\in\real^{m\times n}$ and $\bB\in\real^{n\times p}$ be two matrices with ranks $a$ and $b$, respectively. Show that the rank of $\bA \bB$ is at most $\min\{a,b\}$ and at least $a+b-n$.
\end{exercise}

\index{Orthogonal complement}
\begin{definition}[Orthogonal complement in general]
The orthogonal complement $\mathcalV^\perp$ of a subspace $\mathcalV$ consists of all vectors that are perpendicular to $\mathcalV$. Formally,
$$
\mathcalV^\perp = \{\bv : \bv^\top\bu=0, \ \forall\, \bu\in \mathcalV  \}.
$$
The two subspaces are \textit{disjoint} (i.e., their intersection is $\{\bzero\}$) and together span the entire space. The dimensions of $\mathcalV$ and $\mathcalV^\perp$ add up to the dimension of the full space. 
Furthermore, taking the orthogonal complement twice returns the original subspace:  $(\mathcalV^\perp)^\perp=\mathcalV$.
\end{definition}

For example, we can explicitly define the orthogonal complement of the column space as follows:
\begin{definition}[Orthogonal complement of column space]
For an $m \times n$ real matrix $\bA$, the orthogonal complement of its column space $\mathcalC(\bA)$, denoted by $\cspace^{\bot}(\bA)$, is the subspace:
\begin{equation*}
\begin{aligned}
\cspace^{\bot}(\bA) &= \{\by\in \real^m: \, \by^\top \bA \bx=\bzero, \ \forall\, \bx \in \real^n \} \\
&=\{\by\in \real^m: \, \by^\top \bv = \bzero, \ \forall\, \bv \in \cspace(\bA) \}.
\end{aligned}
\end{equation*}
\end{definition}

We now introduce the \textit{four fundamental subspaces} associated with any matrix $\bA\in \real^{m\times n}$ of rank $r$, as outlined in Theorem~\ref{theorem:fundamental-linear-algebra}. 
To establish this fundamental theorem of linear algebra, we first need to verify a key result: the equality of the row rank and column rank of a matrix. 

\index{Rank}
\begin{theorem}[Row rank equals column rank]\label{theorem:equal-dimension-rank}
The dimension of the column space of a matrix $\bA\in \real^{m\times n}$ is equal to the dimension of its
row space. In other words, the row rank and the column rank of a matrix $\bA$ are equal.
\end{theorem}

\begin{proof}[{of Theorem~\ref{theorem:equal-dimension-rank}}]
We begin by observing that the null space of $\bA$ is orthogonal  to the row space of $\bA$: $\nspace(\bA) \bot \cspace(\bA^\top)$ (where the row space of $\bA$ corresponds to the column space of $\bA^\top$). That is, vectors in the null space of $\bA$ are orthogonal to vectors in the row space of $\bA$. To see this, suppose $\bA$ has rows $\{\ba_1^\top, \ba_2^\top, \ldots, \ba_m^\top\}$ and $\bA=[\ba_1^\top; \ba_2^\top; \ldots; \ba_m^\top]$ is the row partition. For any vector $\bx\in \nspace(\bA)$, we have $\bA\bx = \bzero$, or equivalently, $[\ba_1^\top\bx; \ba_2^\top\bx; \ldots; \ba_m^\top\bx]=\bzero$. 
Since the row space of $\bA$ is spanned by $\{\ba_1^\top, \ba_2^\top, \ldots, \ba_m^\top\}$, it follows that $\bx$ is perpendicular to all vectors in $\cspace(\bA^\top)$, which means $\nspace(\bA) \bot \cspace(\bA^\top)$.

Next, suppose  the dimension of the row space of $\bA$ is $r$. \textcolor{mylightbluetext}{Let $\{\br_1, \br_2, \ldots, \br_r\}$ be a set of vectors in $\real^n$ and form a basis for the row space}. Then the $r$ vectors $\{\bA\br_1, \bA\br_2, \ldots, \bA\br_r\}$ lie in the column space of $\bA$. We claim that these vectors are linearly independent.
To verify this, suppose there exists a linear combination of the $r$ vectors: $x_1\bA\br_1 + x_2\bA\br_2+ \ldots+ x_r\bA\br_r=\bzero$, that is, $\bA(x_1\br_1 + x_2\br_2+ \ldots+ x_r\br_r)=\bzero$, and the vector $\bv=x_1\br_1 + x_2\br_2+ \ldots+ x_r\br_r$ belongs to the null space of $\bA$. But since $\{\br_1, \br_2, \ldots, \br_r\}$ is a basis for the row space of $\bA$, $\bv$ must also lie  in the row space of $\bA$. 
We have shown that vectors from the null space of $\bA$ is perpendicular to vectors from the row space of $\bA$; thus, it holds that $\bv^\top\bv=0$, which implies that  $x_1=x_2=\ldots=x_r=0$. Hence, \textcolor{mylightbluetext}{$\bA\br_1, \bA\br_2, \ldots, \bA\br_r$ lie in the column space of $\bA$, and they are linearly independent}.
Since these $r$ linearly independent vectors are in the column space of $\bA$, the column space must have dimension at least $r$.  This proves that \textbf{row rank of $\bA\leq $ column rank of $\bA$}. 

Applying the same reasoning to $\bA^\top$, we conclude that \textbf{column rank of $\bA\leq $ row rank of $\bA$}. 
Combining these results, we obtain the equality of the row rank and column rank of $\bA$. This completes the proof.
\end{proof}

Additional insights from this proof reveal that if $\{\br_1, \br_2, \ldots, \br_r\}$ forms a basis for the row space of $\bA\in\real^{m\times n}$, then \textcolor{black}{$\{\bA\br_1, \bA\br_2, \ldots, \bA\br_r\}$} constitutes a basis for the column space of $\bA$.
This result is formalized in the following lemma:

\begin{lemma}[Column basis from row basis]\label{lemma:column-basis-from-row-basis}
For any matrix $\bA\in \real^{m\times n}$, if $\{\br_1, \br_2, \ldots, \br_r\}$ is a set of vectors in $\real^n$ that forms a basis for the row space of $\bA$, then $\{\bA\br_1, \bA\br_2, \ldots, \bA\br_r\}$ forms a basis for the column space of $\bA$.
\end{lemma}

For any matrix $\bA\in \real^{m\times n}$, it can be easily verified that any vector in the row space of $\bA$ is orthogonal to any vector in the null space of $\bA$. Specifically, if $\bx_n \in \nspace(\bA)$, then $\bA\bx_n = \bzero$, which implies that $\bx_n$ is perpendicular to every row of $\bA$, thus supporting this assertion.

Similarly, any vector in the column space of $\bA$ is orthogonal to any vector in the null space of $\bA^\top$. Moreover, the column space of $\bA$ together with the null space of $\bA^\top$ span the entire space $\real^m$. This observation is a key part of the  fundamental theorem of linear algebra.

The fundamental theorem consists of two essential components: the dimensions of the subspaces and the orthogonality relationships between pairs of subspaces. 
The orthogonality relationships have already  been demonstrated above.
Additionally, when the row space has dimension $r$, the null space has dimension $n-r$. These relationships are rigorously established in the following theorem.

\index{Fundamental spaces}
\index{Fundamental theorem of linear algebra}
\begin{figure}[h!]
	\centering
	\includegraphics[width=0.9\textwidth]{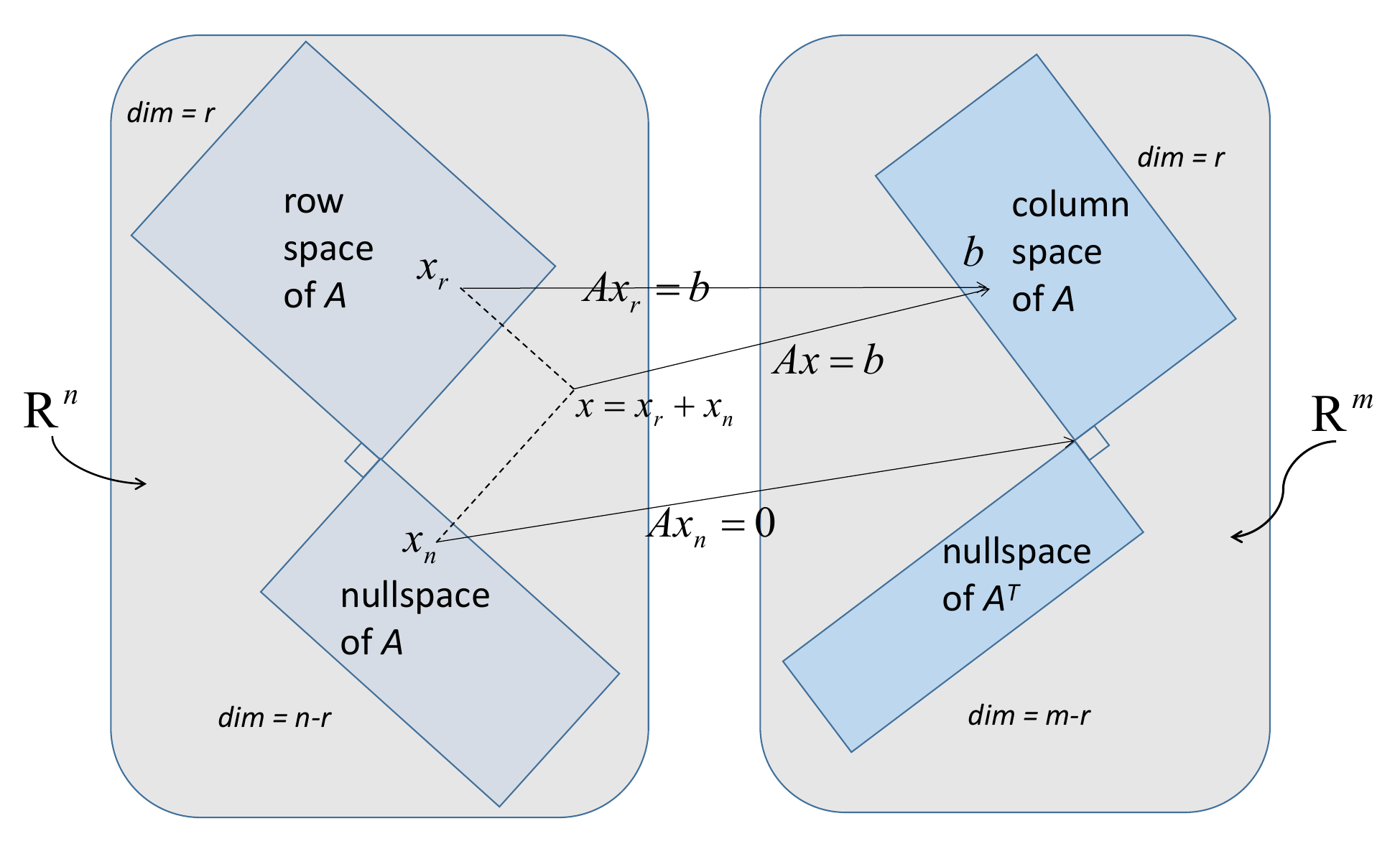}
	\caption{Two pairs of orthogonal subspaces in $\real^n$ and $\real^m$. $\dim(\cspace(\bA^\top)) + \dim(\nspace(\bA))=n$ and $\dim(\nspace(\bA^\top)) + \dim(\cspace(\bA))=m$. The null space component maps to zero as $\bA\bx_n = \bzero \in \real^m$. The row space component maps to the column space as $\bA\bx_r = \bA(\bx_r+\bx_n)=\bb \in \cspace(\bA)$.}
	\label{fig:lafundamental}
\end{figure}
\begin{theorem}[The fundamental theorem of linear algebra]\label{theorem:fundamental-linear-algebra}

\textit{Orthogonal Complement} and \textit{Rank-Nullity Theorem}: for any matrix $\bA\in \real^{m\times n}$, we have 
\begin{itemize}
\item The null space $\nspace(\bA)$ is orthogonal complement to the row space $\cspace(\bA^\top)$ in $\real^n$: $\dim(\nspace(\bA))+\dim(\cspace(\bA^\top))=n$;
	
\item The left null space $\nspace(\bA^\top)$ is orthogonal complement to the column space $\cspace(\bA)$ in $\real^m$: $\dim(\nspace(\bA^\top))+\dim(\cspace(\bA))=m$;
	
\item For a rank-$r$ matrix $\bA$, $\dim(\cspace(\bA^\top)) = \dim(\cspace(\bA)) = r$, that is, $\dim(\nspace(\bA)) = n-r$ and $\dim(\nspace(\bA^\top))=m-r$.
\end{itemize}
\end{theorem}

\begin{proof}[of Theorem~\ref{theorem:fundamental-linear-algebra}]
From the proof of Theorem~\ref{theorem:equal-dimension-rank}, let $\{\br_1, \br_2, \ldots, \br_r\}$ be a set of vectors in $\real^n$ that forms a basis for the row space; then \textcolor{black}{$\{\bA\br_1, \bA\br_2, \ldots, \bA\br_r\}$ is a basis for the column space of $\bA$}. Let $\bn_1, \bn_2, \ldots, \bn_k \in \real^n$ form a basis for the null space of $\bA$. Following again from the proof of Theorem~\ref{theorem:equal-dimension-rank}, $\nspace(\bA) \bot \cspace(\bA^\top)$, thus, $\br_1, \br_2, \ldots, \br_r$ are perpendicular to $\bn_1, \bn_2, \ldots, \bn_k$. Then, $\{\br_1, \br_2, \ldots, \br_r, \bn_1, \bn_2, \ldots, \bn_k\}$ is linearly independent in $\real^n$.

For any vector $\bx\in \real^n $, $\bA\bx$ lies in the column space of $\bA$, so it can be written as a linear combination of $\bA\br_1, \bA\br_2, \ldots, \bA\br_r$: $\bA\bx = \sum_{i=1}^{r}a_i\bA\br_i$. This implies that $\bA(\bx-\sum_{i=1}^{r}a_i\br_i) = \bzero$, and $\bx-\sum_{i=1}^{r}a_i\br_i$ is thus in $\nspace(\bA)$. Since $\{\bn_1, \bn_2, \ldots, \bn_k\}$ is a basis for the null space of $\bA$, $\bx-\sum_{i=1}^{r}a_i\br_i$ can be represented as a linear combination of $\bn_1, \bn_2, \ldots, \bn_k$: $\bx-\sum_{i=1}^{r}a_i\br_i = \sum_{j=1}^{k}b_j \bn_j$, i.e., $\bx=\sum_{i=1}^{r}a_i\br_i + \sum_{j=1}^{k}b_j \bn_j$. That is, any vector $\bx\in \real^n$ can be represented by $\{\br_1, \br_2, \ldots, \br_r, \bn_1, \bn_2, \ldots, \bn_k\}$, and the set forms a basis for $\real^n$. Thus, the dimensions satisfy: $r+k=n$, i.e., $\dim(\nspace(\bA))+\dim(\cspace(\bA^\top))=n$. Similarly, we can prove that $\dim(\nspace(\bA^\top))+\dim(\cspace(\bA))=m$.
\end{proof}

Figure~\ref{fig:lafundamental} illustrates  two pairs of such orthogonal subspaces and demonstrates  how $\bA$ maps $\bx$ into the column space. The dimensions of the row space of $\bA$ and the null space of $\bA$ sum to $n$. And the dimensions of the column space of $\bA$ and the null space of $\bA^\top$ add up to $m$. The null space component is mapped to zero, as $\bA\bx_{\bn} = \bzero \in \real^m$, which is the intersection of the column space of $\bA$ and the null space of $\bA^\top$. 
Conversely, the row space component is mapped  to the column space, as $\bA\bx_{\br} = \bA(\bx_{\br} + \bx_{\bn})=\bb\in \real^m$.

\index{Orthogonal matrix}
\begin{definition}[Orthogonal matrix]
A real square matrix $\bQ$ is called an \textit{orthogonal matrix} if its inverse equals its transpose, that is, $\bQ^{-1}=\bQ^\top$ and $\bQ\bQ^\top = \bQ^\top\bQ = \bI$. Equivalently, suppose $\bQ=[\bq_1, \bq_2, \ldots, \bq_n]$, where $\bq_i \in \real^n$ for all $i \in \{1, 2, \ldots, n\}$. Then, $\bq_i^\top \bq_j = \delta(i,j)$, where $\delta(i,j)$ is the Kronecker delta function. If $\bQ$ contains only $\gamma$ of these columns with $\gamma<n$, the relation $\bQ^\top\bQ = \bI_\gamma$ stills holds, where $\bI_\gamma$ is the $\gamma\times \gamma$ identity matrix.
However, in this case,  the equation $\bQ\bQ^\top=\bI$ no longer holds; and $\bQ$ is known as a \textit{semi-orthogonal matrix}.
An orthogonal matrix also preserves the length of any vector $\bx$, i.e., $\norm{\bQ\bx} = \norm{\bx}$ (see Definition~\ref{definition:vec_l2_norm}).
\end{definition}

\index{Permutation matrix}
\begin{definition}[Permutation matrix]\label{definition:permutation-matrix}
A permutation matrix $\bP$ is a square binary matrix with exactly one entry of 1 in each row and each column; and all other entries are 0.
\paragraph{Row perspective.} A permutation matrix $\bP$ can be viewed as having the rows of the identity matrix $\bI$ arranged in a specific order. 
This order determines the sequence of row permutations. To permute the rows of a matrix $\bA$, multiply $\bA$ on the left by  $\bP$, yielding  $\bP\bA$.

\paragraph{Column perspective.} 
Alternatively, $\bP$ can be viewed as having the columns of the identity matrix $\bI$ rearranged. This order determines the sequence of column permutations. To permute the columns of $\bA$, multiply $\bA$ on the right by $\bP$, yielding $\bA\bP$.
\end{definition}

The permutation matrix $\bP$ can also be efficiently represented using a vector  $\sJ \in \integer_{++}^n$ of indices, such that $\bP = \bI[:, \sJ]$, where $\bI$ is the $n\times n$ identity matrix. Notably, the sum of the elements in $\sJ$ equals $1+2+\ldots+n= \frac{n^2+n}{2}$.

\begin{example}[Permutation]
Let
$\bA=\begin{bmatrixscript}
	1 & 2&3\\
	4&5&6\\
	7&8&9
\end{bmatrixscript}
$
and 
$\bP=\begin{bmatrixscript}
0 &1& 0 \\
0 &0 &1\\
1 &0 &0 
\end{bmatrixscript}.
$
The row permutation and the column permutation are given by 
$
\bP\bA = \begin{bmatrixscript}
	4&5&6\\
	7&8&9\\
	1 & 2&3\\
\end{bmatrixscript}
$
and
$
\bA\bP = \begin{bmatrixscript}
	3 & 1 & 2 \\
	6 & 4 & 5\\
	9 & 7 & 8
\end{bmatrixscript},
$
where the order of the rows of $\bA$ appearing in $\bP\bA$ matches the order of the rows of $\bI$ in $\bP$, 
and the order of the columns of $\bA$ appearing in $\bA\bP$ matches the order of the columns of $\bI$ in $\bP$.
\end{example}

\index{Determinant}
Geometrically, the determinant of an $n \times n$ matrix  $\bA$, denoted by $\det(\bA)$, is the (signed) volume of the $n$-dimensional parallelepiped defined by its row (or column) vectors.
For example, an orthogonal matrix always forms a unit hypercube, and so the absolute value of its determinant is always 1.
The above definition is self-consistent because 
the volume defined by the row vectors and the volume defined by the column vectors of a square matrix can be mathematically shown to be the same.
We can recursively define the determinant of a matrix as follows:
\begin{definition}[Determinant: Laplace expansion by minors]\label{definition:determinant}
Let $\bA\in\real^{n\times n}$ be any square matrix, and let $\bA_{ij}\in\real^{(n-1)\times (n-1)}$ denote the submatrix of $\bA$ obtained by deleting the $i$-th row and $j$-th column. 
The \textit{determinant} of $\bA$ can be computed recursively using the following equations:
\begin{equation}\label{equation:def_det}
	\det(\bA)=\sum_{k=1}^{n} (-1)^{i+k} a_{ik}\det(\bA_{ik})=\sum_{k=1}^{n}(-1)^{k+j} a_{kj}\det(\bA_{kj}),
\end{equation}
where the first equation is the \textit{Laplace expansion by minors along row $i$}, and the second equation is the \textit{Laplace expansion by minors along column $j$}.
Equivalently, given a cardinality $r$, and consider an index set $\sJ\subseteq\{1,2,\ldots,n\}$ with cardinality $r$ ($\abs{\sJ}=r$) and its complementary set $\comple{\sJ}=\{1,2,\ldots,n\}\backslash \sJ$. Then we have:
$$
\begin{aligned}
	\det(\bA) 
	=\sum_{\sI} (-1)^{\gamma} \det(\bA[\sI,\sJ])\det(\bA[\comple{\sI}, \comple{\sJ}])
	=\sum_{\sI} (-1)^{\gamma} \det(\bA[\sJ,\sI])\det(\bA[\comple{\sJ}, \comple{\sI}]),
\end{aligned}
$$
where $\gamma=\sum_{i\in \sI} i +\sum_{j\in \sJ}j$, and the sum is taken over all the index sets $\sI\subseteq\{1,2,\ldots,n\}$ with cardinality $r$.
When $r=1$, this reduces to \eqref{equation:def_det}.
\end{definition}

\begin{remark}[Determinant: alternating sums and permutations]\label{remark:det_altsumper}
Let the function $p:\{1,2,\ldots,n\}\rightarrow \{1,2,\ldots,n\}$ be a one-to-one function  of permutations, i.e., $p(i)=i$ in the identity case.
Then, there are $n!$ distinct permutations of the set $\{1,2,\ldots,n\}$. For a given permutation, let $\text{sgn}(p)=1$ if the minimum number of transpositions to achieve this permutation is even; and $\text{sgn}(p)=-1$ otherwise.
Then, the determinant can be  equivalently defined as
$
	\det(\bA) = \sum_{p} \left( \text{sgn}(p) \prod_{i=1}^{n} a_{i p(i)} \right).
$
\end{remark}

A quantity closely related to the determinant is the adjugate of a matrix, which we now define.
\index{Adjugate}
\begin{definition}[Adjugate]\label{definition:adjugate}
Let $\bA\in\real^{n\times n}$ be any square matrix. Then, the \textit{adjugate} of $\bA$, denoted $\adjugate(\bA)$, is an $n\times n$ matrix whose $(i,j)$-th element is given by 
\begin{equation}\label{equation:adjug1}
\adjugate(\bA)_{ij} = (-1)^{i+j} \det\left(\bA\big[\comple{\{j\}}, \comple{\{i\}}\big]\right),
\end{equation}
where $\comple{\{i\}}$ is the complementary set of $\{1,2,\ldots,n\}$: $\comple{\{i\}}=\{1,2,\ldots,n\}\backslash i$.
Comparing this with the definition of determinants, we have
\begin{equation}\label{equation:adjug2}
\adjugate(\bA)\bA = \bA \adjugate(\bA) = \det(\bA)\bI.
\end{equation}
This shows that  $\adjugate(\bA)$ is nonsingular (resp., upper triangular, diagonal) if $\bA$ is nonsingular (resp., upper triangular, diagonal): 
\begin{equation}\label{equation:adjug3}
\adjugate(\bA)=\det(\bA)\bA^{-1}.
\end{equation}
If $\bA$ is singular, then every column of $\bA$ lies in the null space of $\adjugate(\bA)$: the dimension of the null space  of $\adjugate(\bA)$ is at least the rank of $\bA$.
\end{definition}
For example, $\adjugate\big(\scriptsize\begin{bmatrix}
a& b\\
c& d
\end{bmatrix}\normalsize\big)
=
\scriptsize
\begin{bmatrix}
d & -b \\
-c & a
\end{bmatrix}
$.
Additional properties of the adjugate of a matrix are discussed in Problem~\ref{prob:pro_adjug}, where we introduce the \textit{interpolative decomposition} of a matrix.

\index{Nonsingular matrix}
From an introductory linear algebra course, we observe the following equivalences regarding nonsingular matrices.
\begin{remark}[List of equivalence of nonsingularity for a matrix]
For a square matrix $\bA\in \real^{n\times n}$, the following claims are equivalent:
\begin{itemize}
\item $\bA$ is nonsingular;
\item $\bA$ is invertible, i.e., $\bA^{-1}$ exists;
\item $\bA\bx=\bb$ has a unique solution $\bx = \bA^{-1}\bb$;
\item $\bA\bx = \bzero$ has a unique, trivial solution: $\bx=\bzero$;
\item Columns of $\bA$ are linearly independent;
\item Rows of $\bA$ are linearly independent;
\item $\det(\bA) \neq 0$; 
\item $\dim(\nspace(\bA))=0$;
\item $\nspace(\bA) = \{\bzero\}$, i.e., the null space is trivial;
\item $\cspace(\bA)=\cspace(\bA^\top) = \real^n$, i.e., the column space or row space span the whole $\real^n$;
\item $\bA$ has full rank $r=n$;
\item The reduced row echelon form is $\bR=\bI$;
\item $\bA^\top\bA$ is symmetric positive definite;
\item $\bA$ has $n$ nonzero (positive) singular values;
\item All eigenvalues of $\bA$ are nonzero.
\end{itemize}
\end{remark}

\index{Singular matrix}
It is important to keep these equivalences in mind, as misunderstanding them can easily lead to confusion. On the other hand, the following remark outlines the corresponding set of equivalent conditions for singular matrices.
\begin{remark}[List of equivalence of singularity for a matrix]\label{remark:list_sing_equiv}
For a square matrix $\bA\in \real^{n\times n}$ with an eigenpair $(\lambda, \bu)$, the following claims are equivalent:
\begin{itemize}
\item $(\bA-\lambda\bI)$ is singular;
\item $(\bA-\lambda\bI)$ is not invertible;
\item $(\bA-\lambda\bI)\bx = \bzero$ has nonzero $\bx\neq \bzero$ solutions, and $\bx=\bu$ is one of such solutions;
\item $(\bA-\lambda\bI)$ has linearly dependent columns;
\item $\det(\bA-\lambda\bI) = 0$; 
\item $\dim(\nspace(\bA-\lambda\bI))>0$;
\item Null space of $(\bA-\lambda\bI)$ is nontrivial;
\item Columns of $(\bA-\lambda\bI)$ are linearly dependent;
\item Rows of $(\bA-\lambda\bI)$ are linearly dependent;
\item $(\bA-\lambda\bI)$ has rank $r<n$;
\item Dimension of column space = dimension of row space = $r<n$;
\item $(\bA-\lambda\bI)^\top(\bA-\lambda\bI)$ is symmetric semidefinite;
\item $(\bA-\lambda\bI)$ has $r<n$ nonzero (positive) singular values;
\item Zero is an eigenvalue of $(\bA-\lambda\bI)$.
\end{itemize}
\end{remark}

Norms provide a measure of the magnitude of a vector or matrix, which is useful in many applications, such as determining the length of a vector in Euclidean space or the size of a matrix in a multidimensional setting. 
Additionally, norms enable us to define distances between vectors or matrices. The distance between two vectors $\bu$ and $\bv$ can be computed using the norm of their difference $\norm{\bu-\bv}$. 
This is critical for tasks involving proximity measures, such as clustering algorithms in machine learning.

For a vector $\bx\in\real^n$, we define the vector $\ell_2$ norm as follows.
\index{Vector norm}
\index{Matrix norm}
\begin{definition}[Vector $\ell_2$ norm]\label{definition:vec_l2_norm}
	For a vector $\bx\in\real^n$, the \textit{$\ell_2$ vector norm} is defined as $\normtwo{\bx} = \sqrt{x_1^2+x_2^2+\ldots+x_n^2}$.
\end{definition}

For a matrix $\bA\in\real^{m\times n}$, we  define the (matrix) Frobenius norm as follows.
\begin{definition}[Matrix Frobenius norm\index{Frobenius norm}]\label{definition:frobernius-in-svd}
	The \textit{Frobenius norm} of a matrix $\bA\in \real^{m\times n}$ is defined as 
	$$
	\normf{\bA} = \sqrt{\sum_{i=1,j=1}^{m,n} (a_{ij})^2}=\sqrt{\trace(\bA\bA^\top)}=\sqrt{\trace(\bA^\top\bA)} = \sqrt{\sigma_1^2+\sigma_2^2+\ldots+\sigma_r^2}, 
	$$
where $\sigma_1, \sigma_2, \ldots, \sigma_r$ are the nonzero singular values of $\bA$  (see Section~\ref{section:SVD}).
The squared Frobenius norm of a matrix is often referred to as the \textit{energy} of the matrix in the machine learning community.
\end{definition}
The Frobenius norm can be interpreted as the $\ell_2$ norm applied to the vectorized form of the  matrix.
Additionally, the spectral norm of a matrix is defined as follows.
\begin{definition}[Matrix spectral norm]\label{definition:spectral_norm}
The \textit{spectral norm} of a matrix $\bA\in \real^{m\times n}$ is defined as 
$$
\normtwo{\bA} = \mathop{\max}_{\bx\neq\bzero} \frac{\normtwo{\bA\bx}}{\normtwo{\bx}}  =\mathop{\max}_{\bx\in \real^n: \normtwo{\bx}=1}  \normtwo{\bA\bx} ,
$$
which corresponds to the largest singular value of $\bA$, i.e., $\normtwo{\bA} = \sigma_{\max}(\bA)$.
The definition also implies  the inequality: $\normtwo{\bA\bx}\leq \sigma_{\max}(\bA)\normtwo{\bx}$ for any $\bx\in\real^n$.
\end{definition}

For simplicity, we will not always explicitly indicate the full subscript for both the vector $\ell_2$ norm and the matrix Frobenius norm when it is clear from the context which one we are referring to; that is, we may write $\norm{\bA}=\normf{\bA}$ and $\norm{\bx}=\normtwo{\bx}$.

\newpage
\part{Gaussian Elimination}

\chapter{LU Decomposition}

\section{LU Decomposition}
One of the most well-known and foundational matrix decompositions is the LU decomposition. The details are outlined in the following theorem, and the proof of its existence will be discussed in subsequent sections.~\footnote{Note that, in the subsequent text, decomposition-related results will be presented in blue boxes, while other claims and theorems will be in gray boxes. This convention will be consistently applied throughout the remainder of the book without further notification.}

\index{Decomposition: LU}
\begin{theoremHigh}[LU decomposition with permutation]\label{theorem:lu-factorization-with-permutation}
Let $\bA$ be a nonsingular $n\times n$ square matrix. Then, it can be decomposed as 
\begin{equation}
	\bA = \bP\bL\bU, \nonumber
\end{equation}
where $\bP$ is a permutation matrix, $\bL$ is a unit lower triangular matrix (i.e., a lower triangular matrix with all 1's on the diagonal), and $\bU$ is a \textit{nonsingular} upper triangular matrix. 
\end{theoremHigh}


In certain cases, the use of the permutation matrix is unnecessary. This decomposition depends on the (leading) principal minors. We provide a precise definition, which is crucial for the subsequent illustration.

\index{Principal minor}
\begin{definition}[Principal minors]\label{definition:principle-minors}
Let $\bA$ be an $n\times n$ square matrix. A $k \times k$ submatrix of $\bA$ obtained by deleting any $n-k$ columns and the same $n-k$ rows from $\bA$ is called a $k$-th order \textit{principal submatrix} of $\bA$. The determinant of a $k \times k$ principal submatrix is called a $k$-th order \textit{principal minor} of $\bA$.
\end{definition}

\begin{definition}[Leading principal minors\index{Leading principal minor}]\label{definition:leading-principle-minors}
Let $\bA$ be an $n\times n$ square matrix. A $k \times k$ submatrix of $\bA$ obtained by deleting the \textbf{last} $n-k$ columns and the \textbf{last} $n-k$ rows from $\bA$ is called the $k$-th order \textit{leading principal submatrix} of $\bA$; that is, the $k\times k$ submatrix taken from the top-left corner of $\bA$. The determinant of the $k \times k$ leading principal submatrix is called the $k$-th order \textit{leading principal minor} of $\bA$.
\end{definition}

If the leading principal minors of matrix $\bA$ satisfy mild conditions, the LU decomposition does not require a permutation matrix, which we now recall in the following theorem:
\begin{theoremHigh}[LU decomposition without permutation]\label{theorem:lu-factorization-without-permutation}
Let $\bA$ be an  $n\times n$ square matrix with nonzero leading principal minors, i.e., $\det(\bA_{1:k,1:k})\neq 0$, for all $k\in \{1,2,\ldots, n\}$. Then, $\bA$ can be decomposed as 
\begin{equation}
	\bA = \bL\bU, \nonumber
\end{equation}
where $\bL$ is a unit lower triangular matrix (i.e., a lower triangular matrix with all 1's on the diagonal), and $\bU$ is a \textit{nonsingular} upper triangular matrix. 
Specifically, this decomposition is \textbf{unique}; see Corollary~\ref{corollary:unique-lu-without-permutation}.
\end{theoremHigh}

In Theorem~\ref{theorem:lu-factorization-without-permutation}, we assume that the leading principal minors are nonzero, implying that the leading principal submatrices  and the matrix $\bA$ are nonsingular. 
In the previous theorem, we also assumed that  $\bA$ is nonsingular. However, it is important to note that an LU decomposition can still exist even when   $\bA$ is singular. 
As will be explained in the next section, if $\bA$ is singular, some of the pivots during Gaussian elimination will be zero, resulting in corresponding zero diagonal entries in the matrix $\bU$.

Even when  $\bA$ is nonsingular, its leading principal submatrices may still be singular.
Furthermore, if certain leading principal minors are zero, an LU decomposition may still exist, but it is no longer guaranteed to be unique under these conditions.

Additionally,  LU decomposition can be generalized to handle non-square or singular matrices---examples include the \textit{rank-revealing LU decomposition}. Interested readers are encouraged to consult \citet{pan2000existence, miranian2003strong, dopico2006multiple} for further details, or refer to Section~\ref{section:rank-reveal-lu-short} for a brief overview.

\index{Backward substitution}
\index{Gaussian elimination}
\section{Relation to Gaussian Elimination}\label{section:gaussian-elimination}
Solving the linear system equation $\bA\bx=\bb$ is a fundamental problem in linear algebra.
One widely used method for solving such systems is \textit{Gaussian elimination}, which simplifies a linear system by transforming it into an upper triangular form through a sequence of \textit{elementary row operations} (or \textit{elementary row transformations}).
This process unfolds over $n-1$ stages for a square matrix $\bA\in \real^{n\times n}$.  
As a result, the system becomes much easier to solve using \textit{backward substitution}.
The elementary operations involved are formally defined as follows.

\begin{definition}[Elementary transformations\index{Elementary transformation}]\label{definition:elemen_trans}
Given a square matrix $\bA$, the following three transformations are referred to as \textit{elementary row (resp., column) transformations}:
\begin{enumerate}
\item  Interchanging two rows (resp., columns) of $\bA$.
\item  Multiplying all elements of a single row (resp.,  column) of $\bA$ by a nonzero value.
\item  Adding a multiple of one row (resp., column) to another row (resp., column).
\end{enumerate}
\end{definition}
Specifically,  elementary row transformations of $\bA$ are represented by unit \textit{lower} triangular matrices that act on the left of $\bA$  (e.g., $\bE\bA$), while elementary column transformations are represented by unit \textit{upper} triangular matrices that act on the right of  $\bA$ (e.g., $\bA\bE$).

Gaussian elimination is based on the third type of elementary row transformation listed above. Suppose the upper triangular matrix obtained through Gaussian elimination is given by $\bU = \bE_{n-1}\bE_{n-2}\ldots\bE_1\bA$ (which corresponds to $n-1$ steps). 
And at the $k$-th stage ($k\leq n-1$), consider the $k$-th column of $\bE_{k-1}\bE_{k-2}\ldots\bE_1\bA$, denoted by $\bx\in \real^n$. 
Gaussian elimination aims to introduce zeros below the diagonal of $\bx$ using a transformation of the form
\begin{equation}\label{equation:elimination_mat}
\bE_k = \bI - \bz_k \be_k^\top,
\end{equation}
where $\be_k \in \real^n$ is the $k$-th standard basis vector, and $\bz_k\in \real^n$ is defined as 
$$
\bz_k = [0, \ldots, 0, z_{k+1}, \ldots, z_n]^\top, \qquad z_i= \frac{x_{i}}{x_{k}}, \quad \forall\, i \in \{k+1,\ldots, n\}.
$$
We observe that $\bE_k$ is a unit lower triangular matrix (with $1$'s on its diagonal), where only the entries below the diagonal in the $k$-th column are nonzero:
$$
\bE_k=
\begin{bmatrixfoot}
1 & \ldots & 0&  0 & \ldots & 0\\
\vdots & \ddots & \vdots & \vdots & \ddots & \vdots \\
0 & \ldots & 1  & 0 & \ldots & 0 \\
0 & \ldots & -z_{k+1}  & 1 & \ldots & 0\\
\vdots & \ddots & \vdots & \vdots & \ddots & \vdots \\
0 & \ldots & -z_n & 0 & \ldots & 1
\end{bmatrixfoot}.
$$
Multiplying on the left by $\bE_k$ will introduce zeros below the diagonal:
$$
\bE_k \bx =
\begin{bmatrixfoot}
1 & \ldots & 0&  0 & \ldots & 0\\
\vdots & \ddots & \vdots & \vdots & \ddots & \vdots \\
0 & \ldots & 1  & 0 & \ldots & 0 \\
0 & \ldots & -z_{k+1}  & 1 & \ldots & 0\\
\vdots & \ddots & \vdots & \vdots & \ddots & \vdots \\
0 & \ldots & -z_n & 0 & \ldots & 1
\end{bmatrixfoot}
\begin{bmatrixfoot}
x_1 \\
\vdots \\
x_k\\
x_{k+1} \\
\vdots \\
x_n
\end{bmatrixfoot}=
\begin{bmatrixfoot}
x_1 \\
\vdots \\
x_k\\
0 \\
\vdots \\
0
\end{bmatrixfoot}.
$$

As an example, we outline the Gaussian elimination steps for a $4\times 4$ matrix. For simplicity, we assume no row permutations. 
In the following matrix, $\boxtimes$ represents a value that may not be zero, and \textbf{boldface} indicates the value has just been changed. 
\paragraph{A Trivial Gaussian Elimination For a $4\times 4$ Matrix:}
\begin{equation}\label{equation:elmination-steps}
\footnotesize
\begin{sbmatrix}{\bA}
	\boxtimes & \boxtimes & \boxtimes & \boxtimes \\
	\boxtimes & \boxtimes & \boxtimes & \boxtimes \\
	\boxtimes & \boxtimes & \boxtimes & \boxtimes \\
	\boxtimes & \boxtimes & \boxtimes & \boxtimes
\end{sbmatrix}
\stackrel{\bE_1}{\longrightarrow}
\begin{sbmatrix}{\bE_1\bA}
	\boxtimes & \boxtimes & \boxtimes & \boxtimes \\
	\bm{0} & \textcolor{mylightbluetext}{\bm{\boxtimes}} & \bm{\boxtimes} & \bm{\boxtimes} \\
	\bm{0} & \bm{\boxtimes} & \bm{\boxtimes} & \bm{\boxtimes} \\
	\bm{0} & \bm{\boxtimes} & \bm{\boxtimes} & \bm{\boxtimes}
\end{sbmatrix}
\stackrel{\bE_2}{\longrightarrow}
\begin{sbmatrix}{\bE_2\bE_1\bA}
	\boxtimes & \boxtimes & \boxtimes & \boxtimes \\
	 0 & \textcolor{mylightbluetext}{\boxtimes} & \boxtimes & \boxtimes \\
	 0  & \bm{0} & \textcolor{mylightbluetext}{\bm{\boxtimes}} & \bm{\boxtimes} \\
      0 & \bm{0} & \bm{\boxtimes} & \bm{\boxtimes}
\end{sbmatrix}
\stackrel{\bE_3}{\longrightarrow}
\begin{sbmatrix}{\bE_3\bE_2\bE_1\bA}
	\boxtimes & \boxtimes & \boxtimes & \boxtimes \\
	0 &  \textcolor{mylightbluetext}{\boxtimes} & \boxtimes & \boxtimes \\
    0 & 0  & \textcolor{mylightbluetext}{\boxtimes} & \boxtimes \\
	0 & 0  & \bm{0} & \textcolor{mylightbluetext}{\bm{\boxtimes}}
\end{sbmatrix},
\end{equation}
where $\bE_1, \bE_2$, and $\bE_3$ are lower triangular matrices. Specifically, as discussed earlier, Gaussian transformation matrices $\bE_i$'s are unit lower triangular matrices with $1$'s on the diagonal. This can be explained that for the $k$-th transformation $\bE_k$, working on the matrix $\bE_{k-1}\ldots\bE_1\bA$, the transformation subtracts multiples of the $k$-th row from rows $\{k+1, k+2, \ldots, n\}$ in order to create zeros below the diagonal in the $k$-th column of the matrix, without using rows $\{1, 2, \ldots, k-1\}$.   

To make this more concrete, consider stage 1 of the example above. 
We multiply on the left by $\bE_1$, which subtracts suitable multiples of the first row from rows $2, 3$, and $4$, resulting in zeros in the first entry of each of these rows. Similar operations occur at steps 2 and 3.
By defining $\bL= \bE_1^{-1}\bE_2^{-1}\bE_3^{-1}$ and letting $\bU$ denote the matrix obtained after elimination,~\footnote{Unit lower triangular matrices have two important properties: their inverses are also unit lower triangular, and the product of such matrices results in another unit lower triangular matrix.} we obtain the decomposition $\bA= \bL\bU$. 
Thus, we have constructed an LU decomposition for the  $4\times 4$ matrix $\bA$. 

In the process of Gaussian elimination, we systematically eliminate entries below the diagonal to transform a matrix into an upper triangular form. A key element in guiding this elimination procedure is the first nonzero entry encountered in each row at every step. This special entry not only determines the feasibility of the elimination but also plays a crucial role in numerical stability. We now formally define this important concept.
\begin{definition}[Pivot\index{Pivot}]\label{definition:pivot}
The first nonzero entry in the row after each elimination step is referred to as a \textit{pivot}. For example, the \textcolor{mylightbluetext}{blue} crosses in Equation~\eqref{equation:elmination-steps} indicate the positions of the pivots.
\end{definition}

However,  the entry $a_{11}$ (the (1,1) element of the matrix $\bA$) may occasionally be zero.
In such cases, no such elimination matrix $\bE_1$ can successfully carry out the next elimination step. 
Therefore, we must swap the first and second rows using a permutation matrix $\bP_1$. This is known as  \textit{pivoting}, or simply \textit{permutation}.
\paragraph{Gaussian Elimination With a Permutation in the Beginning:}
$$
\footnotesize
\begin{aligned}
\begin{sbmatrix}{\bA}
	0 & \boxtimes & \boxtimes & \boxtimes \\
	\boxtimes & \boxtimes & \boxtimes & \boxtimes \\
	\boxtimes & \boxtimes & \boxtimes & \boxtimes \\
	\boxtimes & \boxtimes & \boxtimes & \boxtimes
\end{sbmatrix}
&\stackrel{\bP_1}{\longrightarrow}
\begin{sbmatrix}{\bP_1\bA}
	\bm{\boxtimes} & \bm{\boxtimes} & \bm{\boxtimes} & \bm{\boxtimes} \\
	\bm{0} & \bm{\boxtimes} & \bm{\boxtimes} & \bm{\boxtimes} \\
	\boxtimes & \boxtimes & \boxtimes & \boxtimes \\
	\boxtimes & \boxtimes & \boxtimes & \boxtimes
\end{sbmatrix}
\stackrel{\bE_1}{\longrightarrow}
\begin{sbmatrix}{\bE_1\bP_1\bA}
	\boxtimes & \boxtimes & \boxtimes & \boxtimes \\
	\bm{0} & \textcolor{mylightbluetext}{\bm{\boxtimes}} & \bm{\boxtimes} & \bm{\boxtimes} \\
	\bm{0} & \bm{\boxtimes} & \bm{\boxtimes} & \bm{\boxtimes} \\
	\bm{0} & \bm{\boxtimes} & \bm{\boxtimes} & \bm{\boxtimes}
\end{sbmatrix}
\stackrel{\bE_2}{\longrightarrow}
\begin{sbmatrix}{\bE_2\bE_1\bP_1\bA}
	\boxtimes & \boxtimes & \boxtimes & \boxtimes \\
	0 & \textcolor{mylightbluetext}{\boxtimes} & \boxtimes & \boxtimes \\
	0  & \bm{0} & \textcolor{mylightbluetext}{\bm{\boxtimes}} & \bm{\boxtimes} \\
	0 & \bm{0} & \bm{\boxtimes} & \bm{\boxtimes}
\end{sbmatrix}
\stackrel{\bE_3}{\longrightarrow}
\begin{sbmatrix}{\bE_3\bE_2\bE_1\bP_1\bA}
	\boxtimes & \boxtimes & \boxtimes & \boxtimes \\
	0 &  \textcolor{mylightbluetext}{\boxtimes} & \boxtimes & \boxtimes \\
	0 & 0  & \textcolor{mylightbluetext}{\boxtimes} & \boxtimes \\
	0 & 0  & \bm{0} & \textcolor{mylightbluetext}{\bm{\boxtimes}}
\end{sbmatrix}.
\end{aligned}
$$
By defining $\bL=\bE_1^{-1}\bE_2^{-1}\bE_3^{-1}$ and $\bP=\bP_1^{-1}$, the expression $\bA=\bP\bL\bU$ represents a complete  LU decomposition with permutation for the $4\times 4$ matrix $\bA$.

In certain cases, additional permutation matrices such as $\bP_2, \bP_3, \ldots$ may be required between the lower triangular transformations $\bE_i$'s. 
An example is provided below.
\paragraph{Gaussian Elimination With a Permutation in Between:}
$$
\footnotesize
\begin{sbmatrix}{\bA}
	\boxtimes & \boxtimes & \boxtimes & \boxtimes \\
	\boxtimes & \boxtimes & \boxtimes & \boxtimes \\
	\boxtimes & \boxtimes & \boxtimes & \boxtimes \\
	\boxtimes & \boxtimes & \boxtimes & \boxtimes \\
\end{sbmatrix}
\stackrel{\bE_1}{\longrightarrow}
\begin{sbmatrix}{\bE_1\bA}
	\boxtimes & \boxtimes & \boxtimes & \boxtimes \\
	\bm{0} & \bm{0} & \bm{\boxtimes} & \bm{\boxtimes} \\
	\bm{0} & \bm{\boxtimes} & \bm{\boxtimes} & \bm{\boxtimes} \\
	\bm{0} & \bm{\boxtimes} & \bm{\boxtimes} & \bm{\boxtimes}
\end{sbmatrix}
\stackrel{\bP_1}{\longrightarrow}
\begin{sbmatrix}{\bP_1\bE_1\bA}
	\boxtimes & \boxtimes & \boxtimes & \boxtimes \\
	\bm{0} & \textcolor{mylightbluetext}{\bm{\boxtimes}} & \bm{\boxtimes} & \bm{\boxtimes} \\
	\bm{0}  & \bm{0} & \textcolor{mylightbluetext}{\bm{\boxtimes}} & \bm{\boxtimes} \\
	0 & \boxtimes & \boxtimes & \boxtimes
\end{sbmatrix}
\stackrel{\bE_2}{\longrightarrow}
\begin{sbmatrix}{\bE_2\bP_1\bE_1\bA}
	\boxtimes & \boxtimes & \boxtimes & \boxtimes \\
	0 &  \textcolor{mylightbluetext}{\boxtimes} & \boxtimes & \boxtimes \\
	0 & 0  & \textcolor{mylightbluetext}{\boxtimes} & \boxtimes \\
	0 & \bm{0}  & \bm{0} & \textcolor{mylightbluetext}{\bm{\boxtimes}}
\end{sbmatrix}.
$$
In this scenario, we find that $\bU=\bE_2\bP_1\bE_1\bA$. In Section~\ref{section:lu-perm} or Section~\ref{section:partial-pivot-lu}, we will demonstrate that incorporating interleaved permutations still leads to the form $\bA=\bP\bL\bU$, where $\bP$ accounts for all permutations performed.

The provided examples can be easily extended to any $n\times n$ matrix, assuming there are no row permutations involved. 
For such matrices, we apply $n-1$ such lower triangular transformations.
The $k$-th transformation, $\bE_k$, introduces zeros below the diagonal in the $k$-th column of $\bA$ by subtracting multiples of the $k$-th row from rows $\{k+1, k+2, \ldots, n\}$.
Finally, by defining $\bL=\bE_1^{-1}\bE_2^{-1}\ldots \bE_{n-1}^{-1}$, we obtain the LU decomposition $\bA=\bL\bU$ without the need for permutations.

From the examples above involving elementary row operations in the Gaussian elimination process, we can draw the following conclusion about the row spaces after performing (elementary) row transformations.
\begin{proposition}[Row space after row operations]\label{proposition:rowspa_rowele}
Let  $\bA\in\real^{m\times n}$ be a matrix that undergoes a sequence of elementary row operations represented by  $\bE_1, \bE_2, \ldots, \bE_k$, and define $\bE = \bE_k\bE_{k-1}\ldots \bE_1$. 
Then, the row space of $\bB=\bE\bA$ is identical to the row space of $\bA$. 
\end{proposition}
\begin{proof}[of Proposition~\ref{proposition:rowspa_rowele}]
Since the rows of $\bB$ are linear combinations of the rows of $\bA$,  it follows that $\cspace(\bB^\top)\subseteq \cspace(\bA^\top)$. 
Moreover, since  the row transformations are invertible, $\bE= \bE_k\bE_{k-1}\ldots \bE_1$ is also invertible. Therefore, we can write: $\bA = {\bE}^{-1}\bB$. This  implies that the rows of $\bA$ are also  linear combinations of the rows of $\bB$: $\cspace(\bA^\top)\subseteq \cspace(\bB^\top)$. 
Combining the two results, we conclude that $\cspace(\bA^\top)= \cspace(\bB^\top)$.
\end{proof}

Note, however, that the column spaces of $\bA$ and $\bB$ may differ. Nonetheless, since the dimension of the row space equals the dimension of the column space (i.e., the rank of the matrix), the dimensions of the column spaces of $\bA$ and $\bB$ are the same.

\section{Existence of  LU Decomposition without Permutation}\label{section:exist-lu-without-perm}
Gaussian elimination, or Gaussian transformation, provides insight into the foundation of LU decomposition.
We now rigorously prove Theorem~\ref{theorem:lu-factorization-without-permutation}, which establishes the existence of LU decomposition without permutation, using mathematical induction.
\begin{proof}[{of Theorem~\ref{theorem:lu-factorization-without-permutation}: LU decomposition without permutation}]
We will prove by induction that every $n\times n$ square matrix $\bA$ with nonzero leading principal minors admits the LU decomposition of the form $\bA=\bL\bU$.
For the base case ($n=1$), the result is trivial: set $L=1$ and $U=A$ so that $A=LU$. 

Now assume that any $k\times k$ matrix $\bA_k$ with all  leading principal minors being nonzero has an LU decomposition without permutation. Our goal is to establish that any $(k+1)\times(k+1)$ matrix $\bA_{k+1}$ can also be expressed in this LU decomposition form without permutation.

For any $(k+1)\times(k+1)$ matrix $\bA_{k+1}$, suppose the $k$-th order leading principal submatrix of $\bA_{k+1}$ is $\bA_k$ with  size  $k\times k$. Then $\bA_k$ can be factored as $\bA_k =  \bL_k\bU_k$, where  $\bL_k$ is a unit lower triangular matrix and $\bU_k$ is a nonsingular upper triangular matrix, as per the assumption. Express $\bA_{k+1}$ as
$
\bA_{k+1} = \scriptsize\begin{bmatrix}
	\bA_k & \bb \\
	\bc^\top & d
\end{bmatrix}.
$
Then it admits the following factorization:
$$
\bA_{k+1} = \begin{bmatrix}
	\bA_k & \bb \\
	\bc^\top & d
\end{bmatrix}
=
\begin{bmatrix}
	\bL_k &\bzero \\
	\bx^\top  & 1 
\end{bmatrix}
\begin{bmatrix}
	\bU_k & \by\\
	\bzero & z
\end{bmatrix} = \bL_{k+1}\bU_{k+1},
$$
where $\bb = \bL_k\by$, $\bc^\top = \bx^\top\bU_k$, $d = \bx^\top\by + z$, $\bL_{k+1}=\scriptsize\begin{bmatrix}
	\bL_k &\bzero \\
	\bx^\top  & 1 
\end{bmatrix}$, and $\bU_{k+1}=\scriptsize\begin{bmatrix}
\bU_k & \by\\
\bzero & z
\end{bmatrix}$. 
From the assumption, $\bL_k$ and $\bU_k$ are nonsingular. Therefore, we have 
$$
\by = \bL_k^{-1}\bb, \qquad \bx^\top=\bc^\top\bU_k^{-1}, \qquad  z=d - \bx^\top\by.
$$
If, further, we could prove that $z$ is nonzero such that $\bU_{k+1}$ is nonsingular, we complete the proof.

Because all the leading principal minors of $\bA_{k+1}$ are nonzero, 
we have $\det(\bA_{k+1})=$
\footnote{By the fact that if matrix $\bM$ has a block formulation: $\bM=\begin{bmatrixscript}
		\bA & \bB \\
		\bC & \bD 
	\end{bmatrixscript}$, then $\det(\bM) = \det(\bA)\det(\bD-\bC\bA^{-1}\bB)$.}
$\det(\bA_k)\cdot$ $\det(d-\bc^\top\bA_k^{-1}\bb)  
=\det(\bA_k)\cdot(d-\bc^\top\bA_k^{-1}\bb) \neq 0$, since $d-\bc^\top\bA_k^{-1}\bb$ is a scalar.
As $\det(\bA_k)\neq 0$ from the assumption, we conclude that $d-\bc^\top\bA_k^{-1}\bb \neq 0$.  
By substituting $\bb = \bL_k\by$ and $\bc^\top = \bx^\top\bU_k$ into the formula, we have $d-\bx^\top\bU_k\bA_k^{-1}\bL_k\by =d-\bx^\top\bU_k(\bL_k\bU_k)^{-1}\bL_k\by =d-\bx^\top\by \neq 0$, which exactly matches the form of $z\neq 0$. Thus, we find $\bL_{k+1}$ with all the values on the diagonal being 1, and $\bU_{k+1}$ with all the values on the diagonal being nonzero, which means $\bL_{k+1}$ and $\bU_{k+1}$ are nonsingular. \footnote{A triangular matrix (upper or lower) is nonsingular if and only if all the entries on its main diagonal are nonzero.}
This completes the proof.
\end{proof}

We further show that the LU decomposition is unique when no permutation matrix is involved.
\begin{corollary}[Uniqueness of  LU decomposition without permutation]\label{corollary:unique-lu-without-permutation}
Let $\bA$ be an $n\times n$ square matrix with nonzero leading principal minors. Then,  the LU decomposition of $\bA$ is unique.
\end{corollary}
\begin{proof}[of Corollary~\ref{corollary:unique-lu-without-permutation}]
Suppose, for contradiction,  that the LU decomposition is not unique. Then, there exist two decompositions, $\bA=\bL_1\bU_1 = \bL_2\bU_2$, which implies $\bL_2^{-1}\bL_1=\bU_2\bU_1^{-1}$. The left-hand side of the equation is a unit lower triangular matrix, while the right-hand side is an upper triangular matrix. 
Consequently, both sides must be diagonal matrices. Since the inverse of a unit lower triangular matrix is also a unit lower triangular matrix, and the product of unit lower triangular matrices remains a unit lower triangular matrix, we deduce that $\bL_2^{-1}\bL_1 = \bI$.
The equality implies that both sides are identity matrices such that $\bL_1=\bL_2$ and $\bU_1=\bU_2$, resulting in a contradiction.  This completes the proof.
\end{proof}

In the proof of Theorem~\ref{theorem:lu-factorization-without-permutation}, we showed that the diagonal values of the upper triangular matrix are all nonzero if the leading principal minors of $\bA$ are all nonzero. Dividing each row of $\bU$ by its corresponding diagonal entry leads to an equivalent decomposition known as the \textit{LDU decomposition}.

\index{Decomposition: LDU}
\begin{corollaryHigh}[LDU decomposition]\label{corollary:ldu-decom}
Let $\bA$ be an $n\times n$ square matrix with nonzero leading principal minors, i.e., $\det(\bA_{1:k,1:k})\neq 0$, for all $k\in \{1,2,\ldots, n\}$. Then, $\bA$ can be \textbf{uniquely} decomposed as 
\begin{equation}
	\bA = \bL\bD\bU, \nonumber
\end{equation}
where $\bL$ is a unit lower triangular matrix, $\bU$ is a \textbf{unit} upper triangular matrix, and $\bD$ is a diagonal matrix. 
\end{corollaryHigh} 
The result follows directly from the LU decomposition of $\bA = \bL\bR$, where $\bL$ is unit lower triangular and $\bR$ is upper triangular. Define $\bD=\diag(r_{11}, r_{22}, \ldots, r_{nn})$, where $r_{ii}$'s are the diagonal entries of $\bR$. Then, $\bD^{-1}\bR = \bU$ is a unit upper triangular matrix. 
The uniqueness of the LDU decomposition follows from the uniqueness of the LU decomposition: since $\bL$ and $\bR$ are uniquely determined, so too are $\bD$ and $\bU$.

\index{Schur complement}
\section{Existence of  LU Decomposition with Permutation}\label{section:lu-perm}
In Theorem~\ref{theorem:lu-factorization-without-permutation}, we require that  $\bA$ has nonzero leading principal minors. 
However, this condition is not strictly necessary. Even if some leading principal minors are zero, a nonsingular matrix can still admit an LU decomposition, provided that row permutations are allowed. The proof still proceeds by induction. 

To formalize this, we first introduce the concept of the \textit{Schur complement}.
\begin{definition}[Schur complement]\label{definition:schur_comp}
Let $\bA \in \real^{n \times n}$ be a matrix, and suppose its (1,1) entry, $a_{11}$, is nonzero. 
Then,  the \textit{Schur complement} of $a_{11}$ in $\bA$ is $\bA_2 = \bA_{2:n,2:n} -\frac{1}{a_{11}} \bA_{2:n,1}\bA_{1,2:n}$.
\end{definition}

We now provide a rigorous proof for Theorem~\ref{theorem:lu-factorization-with-permutation}.
\begin{proof}[{of Theorem~\ref{theorem:lu-factorization-with-permutation}: LU decomposition with permutation}]
We start with the base case: any $1\times 1$ nonsingular matrix has a full LU decomposition $A=PLU$ by taking $P=1$, $L=1$, and $U=A$.
Assume that every $(n-1)\times (n-1)$ nonsingular matrix admits an LU decomposition with permutation. We now show this holds for any $n\times n$ nonsingular matrix $\bA$.

We will formulate the proof in the following order. \textcolor{black}{If $\bA$ is nonsingular, then its row-permuted matrix $\bB$ is also nonsingular}. And the \textit{Schur complement} of $b_{11}$ in $\bB$ is also nonsingular. Finally, we formulate the decomposition of $\bA$ by $\bB$ from this property.

At least one element in the first column of $\bA$ must be nonzero; otherwise, $\bA$ would be singular.  Therefore, we can apply a row permutation to move a nonzero element to the $(1,1)$ position.
Let $\bP_1$ be a permutation matrix that moves a nonzero entry to the (1,1) position. Define $\bB = \bP_1 \bA$ such that $b_{11} \neq 0$. Since both $\bA$ and $\bP_1$ are nonsingular, their product $\bB$ is also nonsingular.

\paragraph{Schur complement of $\bB$ is also nonsingular.}
The Schur complement of $b_{11}$ in $\bB$ is the $(n-1)\times (n-1)$ matrix given by
$
\widehat{\bB} = \bB_{2:n,2:n} -\frac{1}{b_{11}} \bB_{2:n,1} \bB_{1,2:n}.
$
To prove that $\widehat{\bB}$ is nonsingular,  assume there exists an $(n-1)$-vector $\bx$ such that
$
\widehat{\bB} \bx = \bzero.
$
Then, $\bx$ and $y=-\frac{1}{b_{11}}\bB_{1,2:n} \cdot \bx  $ satisfy
$$
\bB 
\begin{bmatrix}
y \\
\bx 
\end{bmatrix}
=
\begin{bmatrix}
	b_{11} & \bB_{1,2:n} \\
	\bB_{2:n,1} & \bB_{2:n,2:n}
\end{bmatrix}
\begin{bmatrix}
	y\\
	\bx 	 
\end{bmatrix}
=
\begin{bmatrix}
	0\\
	\bzero 	
\end{bmatrix}
.
$$
Since $\bB$ is nonsingular, $\bx$ and $y$ must be zero. Therefore, $\widehat{\bB} \bx = \bzero$ holds only if $\bx=\bzero$, which implies that the null space of $\widehat{\bB}$ has dimension 0. Hence, $\widehat{\bB}$ is nonsingular.

By the induction assumption, any $(n-1)\times(n-1)$ nonsingular matrix can be written in the full LU decomposition form:
$$
\widehat{\bB} = \bP_2\bL_2\bU_2.
$$
We then factor $\bA$ as
\begin{equation*}
\begin{aligned}
\bA &= \bP_1^\top  
\begin{bmatrix}
b_{11} & \bB_{1,2:n} \\
\bB_{2:n,1} & \bB_{2:n,2:n}
\end{bmatrix}
= \bP_1^\top  
\left[
\begin{matrix}
1 & 0 \\
0 & \bP_2
\end{matrix}
\right] 
\left[
\begin{matrix}
b_{11} & \bB_{1,2:n} \\
\bP_2^\top \bB_{2:n,1} &\bP_2^\top \bB_{2:n,2:n}
\end{matrix}
\right]\\
&= \bP_1^\top  
\left[
\begin{matrix}
1 & 0 \\
0 & \bP_2
\end{matrix}
\right] 
\left[
\begin{matrix}
b_{11} & \bB_{1,2:n} \\
\bP_2^\top \bB_{2:n,1} & \textcolor{mylightbluetext}{\bL_2\bU_2}+\bP_2^\top \textcolor{mylightbluetext}{\frac{1}{b_{11}} \bB_{2:n,1} \bB_{1,2:n}}
\end{matrix}
\right]\\
&= \bP_1^\top  
\left[
\begin{matrix}
1 & 0 \\
0 & \bP_2
\end{matrix}
\right] 
\left[
\begin{matrix}
1 & 0 \\
\frac{1}{b_{11}}\bP_2^\top \bB_{2:n,1} & \bL_2
\end{matrix}
\right] 
\left[
\begin{matrix}
b_{11} & \bB_{1,2:n} \\
\bzero & \bU_2
\end{matrix}
\right].\\
\end{aligned}
\end{equation*}
Thus, the full LU decomposition of $\bA=\bP\bL\bU$ is given by:
$$
\bP = \bP_1^\top  
\left[
\begin{matrix}
	1 & 0 \\
	0 & \bP_2
\end{matrix}
\right], \qquad
\bL=\left[
\begin{matrix}
	1 & 0 \\
	\frac{1}{b_{11}}\bP_2^\top \bB_{2:n,1} & \bL_2
\end{matrix}
\right], \qquad
\bU=
\left[
\begin{matrix}
	b_{11} & \bB_{1,2:n} \\
	\bzero & \bU_2
\end{matrix}
\right].
$$
This completes the proof.
\end{proof}

\section{Bandwidth Preserving in  LU Decomposition without Permutation}
We will demonstrate that the bandwidth of a matrix remains unchanged after LU decomposition, provided no row permutations are used.
The bandwidth of a matrix is defined as follows.
\begin{definition}[Matrix bandwidth\index{Matrix bandwidth}]\label{defin:matrix-bandwidth}
For any matrix $\bA\in \real^{n\times n}$, where the entry at position $(i,j)$ is denoted as $a_{ij}$, the matrix $\bA$ has  \textit{upper bandwidth $q$} if $a_{ij} =0$ whenever  $j>i+q$, and  \textit{lower bandwidth $p$} if $a_{ij}=0$ whenever  $i>j+p$. 
\end{definition}
An example of a $6\times 6$ matrix with  upper bandwidth  $2$ and lower bandwidth  $3$ is shown below:
$$
\begin{bmatrixfoot}
\boxtimes & \boxtimes & \boxtimes & 0& 0 & 0\\
\boxtimes & \boxtimes & \boxtimes & \boxtimes& 0 & 0\\
\boxtimes & \boxtimes & \boxtimes & \boxtimes& \boxtimes & 0\\
\boxtimes & \boxtimes & \boxtimes & \boxtimes& \boxtimes & \boxtimes\\
0 & \boxtimes & \boxtimes & \boxtimes& \boxtimes & \boxtimes\\
0 & 0 & \boxtimes & \boxtimes& \boxtimes & \boxtimes\\
\end{bmatrixfoot}.
$$

We now prove that the bandwidth of a matrix is preserved during LU decomposition when no row permutations are applied.
\begin{lemma}[Bandwidth preserving]\label{lemma:lu-bandwidth-presev}
Let $\bA\in \real^{n\times n}$ have  upper bandwidth  $q$ and  lower bandwidth $p$. If $\bA$ admits the LU decomposition $\bA=\bL\bU$, then  $\bL$ has a lower bandwidth of $p$, and  $\bU$ has an upper bandwidth of $q$.
\end{lemma}
\begin{proof}[of Lemma~\ref{lemma:lu-bandwidth-presev}]
The LU decomposition without permutation can be obtained as follows:
$$
\bA=
\begin{bmatrix}
	a_{11} & \bA_{1,2:n} \\
	\bA_{2:n,1} & \bA_{2:n,2:n}
\end{bmatrix}
=
\begin{bmatrix}
	1 & \bzero \\
\frac{1}{a_{11}} \bA_{2:n,1}  & \bI_{n-1}
\end{bmatrix}
\begin{bmatrix}
	a_{11} & \bA_{1,2:n}\\
	\bzero & \bS
\end{bmatrix}
 = \bL_1 \bU_1,
$$
where $\bS =\bA_{2:n,2:n} - \frac{1}{a_{11}}\bA_{2:n,1}\bA_{1,2:n}$ is the Schur complement of $a_{11}$ in $\bA$.
This process is referred to as the  $s$-decomposition of $\bA$.
The first column of $\bL_1$ and the first row of $\bU_1$ retain the required bandwidth properties (bandwidth $p$ and $q$, respectively). 
Furthermore, the Schur complement $\bS$ of $a_{11}$ has  upper bandwidth  $q-1$ and lower bandwidth  $p-1$, respectively. 
By applying induction on the $s$-decomposition to $\bS$, the result follows.
\end{proof}

\section{Block LU Decomposition}
Another form of LU decomposition involves factoring a matrix into block triangular matrices.
\begin{theoremHigh}[Block LU decomposition without permutation]\label{theorem:block-lu-factorization-without-permutation}
For any $n\times n$ square matrix $\bA$, if the first $m$ leading principal block submatrices are nonsingular, then $\bA$ can be uniquely factored as 
\begin{equation}
\bA = \bL\bU
=
\begin{bmatrix}
\bI  & & & \\
\bL_{21} & \bI & & \\
\vdots & & \ddots & \\
\bL_{m1} & \ldots & \bL_{m,m-1} & \bI 
\end{bmatrix}
\begin{bmatrix}
\bU_{11}  &\bU_{12} & \ldots & \bU_{1m}\\
  & \bU_{22} & &  \vdots \\
  & & \ddots & \bU_{m-1,m}\\
  & & & \bU_{mm}\\
\end{bmatrix}
, \nonumber
\end{equation}
where $\bL_{ij}$'s and $\bU_{ij}$'s are certain  block matrices. 
\end{theoremHigh}
It is important to note that that matrix $\bU$ in the above theorem is not necessarily upper triangular. For example:
$$
\bA = 
\footnotesize
\left[\begin{array}{cc;{2pt/2pt}cc}
0& 1 &   1 & 1\\
-1& 2 &   -1 & 2\\\hdashline[2pt/2pt]
2& 1 &   4 & 2\\	
1& 2 &   3 & 3\\
\end{array}\right]
=
\left[\begin{array}{cc;{2pt/2pt}cc}
	1& 0 &   0& 0\\
	0& 1 &   0 & 0\\\hdashline[2pt/2pt]
	5& -2 &   1 & 0\\	
	4& -1 &   0 & 1\\
\end{array}\right]
\left[\begin{array}{cc;{2pt/2pt}cc}
	0& 1 &   1 & 1\\
	-1& 2 &   -1 & 2\\\hdashline[2pt/2pt]
	0& 0 &   -3 & 1\\	
	0& 0 &   -2 & 1\\
\end{array}\right].
$$
The standard non-block LU decomposition fails for $\bA$ because the  $(1,1)$ entry is zero. However, the block LU decomposition still applies.

\index{Partial pivoting}
\section{Partial, Complete, and Rook Pivoting}\label{section:pivoting}
In the field of numerical linear algebra, solving systems of linear equations is a fundamental task that often relies on matrix factorization techniques such as LU decomposition.
However, the direct application of LU decomposition can lead to numerical instability, especially when dealing with matrices that have small diagonal entries in their upper triangular form $\bU$. To address this issue, \textit{pivoting} strategies are employed to improve the robustness and accuracy of the decomposition process.
This section explores the concept of pivoting, focusing specifically on \textit{partial pivoting, complete pivoting}, and \textit{rook pivoting}.

\subsection{Partial Pivoting}\label{section:partial-pivot-lu}
In practice, it is often advantageous to apply pivoting even when it is not strictly necessary. When solving a linear system using LU decomposition, as described in Algorithm~\ref{alg:linear-equation-by-LU}, selecting the pivot as the largest entry in the current column---especially  when the diagonal elements of $\bU$ are small---helps mitigate inaccuracies in the solutions. This technique, known as \textit{partial pivoting}, is widely adopted to enhance numerical stability.
For example, in the partial pivoting process applied to a $4\times 4$ matrix, the following transformation may occur:
\begin{equation}\label{equation:elmination-steps2}
\footnotesize
\begin{sbmatrix}{\bA}
\boxtimes & \boxtimes & \boxtimes & \boxtimes \\
\boxtimes & \boxtimes & \boxtimes & \boxtimes \\
\boxtimes & \boxtimes & \boxtimes & \boxtimes \\
\boxtimes & \boxtimes & \boxtimes & \boxtimes
\end{sbmatrix}
\stackrel{\bE_1}{\longrightarrow}
\begin{sbmatrix}{\bE_1\bA}
\boxtimes & \boxtimes & \boxtimes & \boxtimes \\
0 & \bm{2} & \bm{\boxtimes} & \bm{\boxtimes} \\
0 & \bm{5} & \bm{\boxtimes} & \bm{\boxtimes} \\
0 & \bm{7} & \bm{\boxtimes} & \bm{\boxtimes}
\end{sbmatrix}
\stackrel{\bP_1}{\longrightarrow}
\begin{sbmatrix}{\bP_1\bE_1\bA}
\boxtimes & \boxtimes & \boxtimes & \boxtimes \\
0 & \textcolor{mylightbluetext}{7} & \textcolor{mylightbluetext}{\boxtimes} & \textcolor{mylightbluetext}{\boxtimes} \\
0  & 5 & \boxtimes & \boxtimes \\
0 & \textcolor{mylightbluetext}{2} & \textcolor{mylightbluetext}{\boxtimes} & \textcolor{mylightbluetext}{\boxtimes}
\end{sbmatrix}
\stackrel{\bE_2}{\longrightarrow}
\begin{sbmatrix}{\bE_2\bP_1\bE_1\bA}
\boxtimes & \boxtimes & \boxtimes & \boxtimes \\
0 &  7 & \boxtimes & \boxtimes \\
0 & \bm{0}  & \textcolor{mylightbluetext}{\bm{\boxtimes}} & \bm{\boxtimes} \\
0 & \bm{0} & \bm{0} & \textcolor{mylightbluetext}{\bm{\boxtimes}}
\end{sbmatrix}.
\end{equation}
In this example, after applying the transformation $\bE_1$, the element 7 is chosen as the pivot. Although this pivot is not strictly necessary, it ensures that no multiplier exceeds an absolute value of 1 during Gaussian elimination.

\index{Floating point operations (flops)}
The general procedure for computing the  LU decomposition with partial pivoting for a matrix $\bA\in\real^{n\times n}$ is detailed in Algorithm~\ref{alg:lu-partial-pivot}.
The computational cost of this algorithm is approximately $\sim2/3(n^3)$ \textit{floating point operations (flops)}, along with $(n-1)+(n-2)+\ldots + 1 \sim \mathcalO(n^2)$  comparisons  due to the pivoting process \citep{lu2021numerical}. 
The result of this process is an upper triangular matrix $\bU$, given by
\begin{equation}\label{equation:partial_pivot_u1}
\bU = \bE_{n-1}\bP_{n-1} \ldots \bE_2\bP_2\bE_1\bP_1\bA.
\end{equation}

\begin{algorithm}[htp] 
	\caption{LU Decomposition with Partial Pivoting} 
	\label{alg:lu-partial-pivot} 
	\begin{algorithmic}[1] 
		\Require 
		Matrix $\bA$ of size $n\times n$;  
		\State Set $\bU = \bA$;
		\For{$k=1$ to $n-1$} \Comment{i.e., get the $k$-th column of $\bU$}
		\State Find a row permutation $\bP_k$ that swaps $u_{kk}$ with the largest element in $|\bU_{k:n,k}|$;
		\State $\bU =\bP_k\bU$;
		\State \algoalign{Compute the Gaussian transformation $\bE_k $ to zero out elements below the diagonal in the $k$-th column of $\bU$;}
		\State $\bU = \bE_k\bU$;
		\EndFor
		\State Output $\bU$;
	\end{algorithmic} 
\end{algorithm}
\paragraph{Computing the final $\bL$.} We still need to reconstruct the LU decomposition in the standard form:
$$
\bA = \bP\bL\bU,
$$
where $\bP=\bP_1 \bP_2\ldots\bP_{n-1}$ represents the overall permutation matrix, $\bU$ is the upper triangular matrix obtained directly from the algorithm, and $\bL$ is a unit lower triangular matrix with $\abs{l_{ij}}\leq 1$ for all $1 \leq i,j\leq n$. 
The submatrix $\bL_{k+1:n,k}$ is a permuted version of $\bE_k$'s multipliers. To see this, we notice that the permutation matrices used in the algorithm involve only  swaps of two rows. \textit{This implies each $\bP_k$ is symmetric and satisfies $\bP_k^2 = \bI$, for $k\in \{1,2,\ldots,n-1\}$}.
Let
$$
\bM_k = (\bP_{n-1} \ldots \bP_{k+1}) \bE_k (\bP_{k+1} \ldots \bP_{n-1}).
$$
Substituting this into Equation~\eqref{equation:partial_pivot_u1}, $\bU$ can be expressed as 
$
\bU = \bM_{n-1}\ldots \bM_2\bM_1 \bP^\top \bA.
$
To analyze $\bM_k$, recall that each $\bP_{k+1}$ is a permutation matrix with an identity submatrix in the upper-left $k\times k$ block. Thus,
$$
\begin{aligned}
\bM_k &= (\bP_{n-1} \ldots \bP_{k+1}) (\bI_n-\bz_k\be_k^\top) (\bP_{k+1} \ldots \bP_{n-1})\\
&=\bI_n - (\bP_{n-1} \ldots \bP_{k+1})(\bz_k\be_k^\top)(\bP_{k+1} \ldots \bP_{n-1}) \\
&=\bI_n - (\bP_{n-1} \ldots \bP_{k+1}\bz_k) (\be_k^\top\bP_{k+1} \ldots \bP_{n-1}) \\
&=\bI_n - (\bP_{n-1} \ldots \bP_{k+1}\bz_k)\be_k^\top.  \qquad &\text{(since $\be_k^\top\bP_{k+1} \ldots \bP_{n-1} = \be_k^\top$)}
\end{aligned}
$$
This shows that $\bM_k$ is unit lower triangular, with its $k$-th column being a permuted version of  $\bE_k$. Consequently, the final lower triangular matrix is:
$$
\bL = \bM_1^{-1}\bM_2^{-1} \ldots \bM_{n-1}^{-1}.
$$
Thus, we obtain the full LU decomposition $\bA =\bP\bL\bU$.

\begin{algorithm}[H] 
	\caption{LU Decomposition with Complete Pivoting} 
	\label{alg:lu-complete-pivot} 
	\begin{algorithmic}[1] 
		\Require 
		Matrix $\bA$ with size $n\times n$;  
		\State Set $\bU = \bA$;
		\For{$k=1$ to $n-1$} \Comment{the value $k$ is to get the $k$-th column of $\bU$}
		\State \algoalign{Find a row permutation matrix $\bP_k$ and a column permutation $\bQ_k$ that swaps $u_{kk}$ with the largest element in $|\bU_{k:n,k:n}|$, say $u_{ab} = \max{|\bU_{k:n,k:n}|}$;}
		\State $\bU =\bP_k\bU\bQ_k$;
		\State \algoalign{Determine the Gaussian transformation $\bE_k $ to introduce zeros below the diagonal in column $k$  of $\bU$;}
		\State $\bU = \bE_k\bU$;
		\EndFor
		\State Output $\bU$;
	\end{algorithmic} 
\end{algorithm}
\index{Complete pivoting}
\subsection{Complete Pivoting}\label{section:complete-pivoting}
In partial pivoting, zeros below the diagonal in the $k$-th column of $\bU$ are introduced by selecting the pivot as the largest absolute value in the subcolumn $\bU_{k:n,k}$. In contrast, \textit{complete pivoting} identifies the pivot as  the largest absolute entry in the current submatrix $\bU_{k:n,k:n}$, which is then moved to the  $(k,k)$ position in  $\bU$. This requires applying an additional  \textit{column permutation} $\bQ_k$ at each step. The final upper triangular matrix $\bU$ is expressed as 
$$
\bU = \bE_{n-1}\bP_{n-1}\ldots (\bE_2\bP_2(\bE_1\bP_1\bA\bQ_1)\bQ_2) \ldots \bQ_{n-1}.
$$
The complete pivoting procedure is detailed in Algorithm~\ref{alg:lu-complete-pivot}.

The algorithm requires $~2/3(n^3)$ flops, along with $(n^2+(n-1)^2+\ldots +1^2)\sim \mathcalO(n^3)$ comparisons due to the more extensive pivoting process. 
With $\bP=\bP_1 \bP_2\ldots\bP_{n-1}$, $\bQ = \bQ_1 \bQ_2\ldots\bQ_{n-1}$,  
$$
\bM_k = (\bP_{n-1} \ldots \bP_{k+1}) \bE_k (\bP_{k+1} \ldots \bP_{n-1}), \qquad \text{for all $k\in \{1,2,\ldots,n-1\}$},
$$
and 
$
\bL = \bM_1^{-1}\bM_2^{-1} \ldots \bM_{n-1}^{-1},
$
the final decomposition is  $\bA = \bP\bL\bU \bQ^\top$, or equivalently, $\bP^\top\bA\bQ = \bL\bU$.

\index{Pivoting}
\index{Rook pivoting}
\subsection{Rook Pivoting}
\textit{Rook pivoting} provides an alternative to partial and complete pivoting strategies. Instead of selecting the largest absolute value in $|\bU_{k:n,k:n}|$ at the $k$-th step, it identifies an element  that is \textit{maximal in both its row and column} within that submatrix. This method is non-unique; multiple elements may satisfy the criteria.
For instance, consider the following submatrix:
$$
\bU_{k:n,k:n} = 
\begin{bmatrix}
1 & 2 & 3 & 4 \\
2 & 3 & 7 & 3 \\
5 & 2 & 1 & 2 \\
2 & 1 & 2 & 1 \\
\end{bmatrix}.
$$
In this case, complete pivoting would select the element $7$. 
In contrast, rook pivoting could select any of the entries $5, 4,$ or $7$, since each of these values is the maximum in both its respective row and column.

\index{Rank-revealing}\index{Rank-revealing LU}
\section{Rank-Revealing LU Decomposition}\label{section:rank-reveal-lu-short}
In many  applications, applying Gaussian elimination with pivoting to a matrix $\bA$ of rank $r$ results in a factorization that reveals the rank structure in the following form:
$$
\bP\bA\bQ = 
\begin{bmatrix}
\bL_{11} & \bzero \\
\bL_{21}^\top & \bI 
\end{bmatrix}
\begin{bmatrix}
\bU_{11} & \bU_{12} \\
\bzero & \bzero 
\end{bmatrix},
$$
where $\bL_{11}\in \real^{r\times r}$ and $\bU_{11}\in \real^{r\times r}$ are nonsingular, $\bL_{21}, \bU_{21}\in \real^{r\times (n-r)}$, and $\bP$ and $\bQ$ are permutation matrices. Such a factorization can be obtained using Gaussian elimination with either rook pivoting or complete pivoting; see \citet{hwang1992rank, higham2002accuracy} for more details.

\section{Application: Linear System via  LU Decomposition}\label{section:lu-linear-sistem}
For a well-determined linear system $\bA\bx = \bb$, where $\bA$ is an $n\times n $  nonsingular matrix, directly computing $\bA^{-1}$ is computationally  inefficient. Instead, the system can be solved using LU decomposition. If $\bA$ admits an LU decomposition $\bA = \bP\bL\bU$, the solution  can be obtained using the following algorithm:
\begin{algorithm}[H] 
	\caption{Solving Linear Equations by LU Decomposition} 
	\label{alg:linear-equation-by-LU} 
	\begin{algorithmic}[1] 
		\Require  $\bA$ is a nonsingular $n\times n $ matrix; solve $\bA\bx=\bb$; 
		\State LU decomposition: factor $\bA$ as $\bA=\bP\bL\bU$; \Comment{(2/3)$n^3$ flops}
		\State Apply permutation: compute $\bw = \bP^\top\bb$; \Comment{0 flops }
		\State Solve $\bL\bv = \bw$ using forward substitution; \Comment{$n^2$ flops}
		\State Solve $\bU\bx= \bv$ using backward substitution; \Comment{$n^2$ flops}
	\end{algorithmic} 
\end{algorithm}

The LU decomposition requires a computational  complexity of  $(2/3)n^3$ flops \citep{lu2021numerical}.
Both the backward and forward substitution steps  require $n^2$ flops in total, which can be derived from the sum  $1+3+\ldots + (2n-1)=n^2$ flops. 
Therefore, the overall computational cost is approximately $(2/3)n^3 + 2n^2$ flops. For large values of $n$, the dominant cost comes from the LU decomposition step, which scales as  $(2/3)n^3$ flops. 
Additionally, in the case of a block LU decomposition, where $\bA=\bL\bU$, solving the systems $\bL\bv = \bw$ and $\bU\bx = \bv$  involves additional computational effort. This is because $\bU$ is generally not upper triangular, unlike in standard LU decomposition. 


\index{Nonsingular matrix}
\index{Inverse of a matrix}
\section{Application: Computing the Inverse of Nonsingular Matrices}\label{section:inverse-by-lu}
By Theorem~\ref{theorem:lu-factorization-with-permutation},  any nonsingular matrix $\bA\in \real^{n\times n}$ admits a full LU factorization  of the form $\bA=\bP\bL\bU$. The inverse of $\bA$ can be obtained by solving the matrix equation:
$
\bA\bX = \bI,
$
which involves solving $n$ linear systems of the form $\bA\bx_i = \be_i$ for all $i \in \{1, 2, \ldots, n\}$, where $\bx_i$ is the $i$-the column of $\bX$ and $\be_i$ represents the $i$-th column of $\bI$ (i.e., the $i$-th standard basis vector). 

\begin{theorem}[Inverse of nonsingular matrix by linear system]
	Computing the inverse of a nonsingular matrix $\bA \in \real^{n\times n}$ using $n$ linear systems requires $\sim (2/3)n^3 + n(2n^2)=(8/3)n^3$ flops, where $(2/3)n^3$ corresponds to the cost of performing the LU decomposition of $\bA$.
\end{theorem}
This result follows directly from Algorithm~\ref{alg:linear-equation-by-LU}.
However, computational efficiency can be improved by leveraging the triangular structure of $\bU$ and $\bL$.
Specifically, the inverse of $\bA$ can be expressed as $\bA^{-1} = \bU^{-1}\bL^{-1}\bP^{-1}=\bU^{-1}\bL^{-1}\bP^\top$. Using this approach, the total computational cost can be  reduced from $(8/3)n^3$ to $2n^3$ flops \citep{lu2021numerical}.

\section{Application: Computing the Determinant}
The LU decomposition also simplifies the computation of the determinant of a matrix. 
If $\bA=\bL\bU$, then $\det(\bA) = \det(\bL\bU) = \det(\bL)\det(\bU) = u_{11}u_{22}\ldots u_{nn}$, where $u_{ii}$ denotes the $i$-th diagonal element of $\bU$ (for $i\in \{1,2,\ldots,n\}$).~\footnote{The determinant of a lower triangular matrix (or an upper triangular matrix) is the product of its diagonal entries.}

Furthermore, for an LU decomposition with permutation, where $\bA=\bP\bL\bU$, the determinant of $\bA$ becomes $\det(\bA) = \det(\bP\bL\bU) = \det(\bP)u_{11}u_{22}\ldots u_{nn}$.
The determinant of a permutation matrix is either 1 or –1 because after
changing rows around (which changes the sign of the determinant~\footnote{The determinant changes sign when two rows are exchanged (sign reversal).}), a permutation matrix
becomes the identity matrix $\bI$, whose determinant is one.

\index{Row equivalent}
\begin{problemset}
\item Solve the following system of equations using row reduction:
$$
\begin{aligned}
	2x_1 + 3x_2 + 4x_3 &= 9,\\
	x_1 + 2x_2+ 3x_3 &= 5,\\
	3x_1 + 4x_2 + 5x_3 &= 7.
\end{aligned}
$$

\item Two matrices $\bA$ and $\bB$ are said to be \textit{row equivalent} (denoted by $\bA\stackrel{r}{\sim}\bB$) 
if $\bA$  can be transformed into $\bB$ using a sequence of elementary row operations.
\begin{itemize}
	\item Show that $\bA\stackrel{r}{\sim}\bB$ if and only if $\bA=\bP\bB$ for some nonsingular matrix $\bP$.
	\item Show that if $\bA\stackrel{r}{\sim}\bC$ and $\bB\stackrel{r}{\sim}\bC$, then $\bA\stackrel{r}{\sim}\bB$.
	\item Show that if $\bA\stackrel{r}{\sim}\bB$ and $\bB\stackrel{r}{\sim}\bC$, then $\bA\stackrel{r}{\sim}\bC$.
	\item Show that if $\bA\stackrel{r}{\sim}\bB$, then $\bB\stackrel{r}{\sim}\bA$.
	\item Show that $\bA\stackrel{r}{\sim}\bB$ if both $\bA$ and $\bB$ are nonsingular.
\end{itemize}

\item Let $\bA_1, \bA_2, \ldots,\bA_n$ be $n \times n$ matrices that are \textit{strictly upper triangular} (having zeros on the diagonal). Show that the product of $\bA_1, \bA_2, \ldots, \bA_n$ is the zero matrix.

\item Given two matrices $\bE$ and $\bF$ obtained from the identity matrix by adding multiples of row $i$ to rows $j$ and $k$ with $i\neq j$ and $i\neq k$, respectively, show that $\bE\bF=\bF\bE$.

\item Show that the LU decomposition of the matrix $\scriptsize\begin{bmatrix}
	0& 1\\
	1 & 0
\end{bmatrix}$
does not exist.

\item Suppose $\bL_1$ and $\bL_2$ are nonsingular lower triangular, and $\bU_1$ and $\bU_2$ are nonsingular upper triangular.
Prove that $\bL_1\bU_1=\bL_2\bU_2$ if and only if there exists an nonsingular diagonal matrix $\bD$ such that $\bL_1=\bL_2\bD$ and $\bU_1=\bD^{-1}\bU_2$.

\item We know that elementary row operations on a matrix can be represented by left-multiplying the matrix with a corresponding transformation matrix (Definition~\ref{definition:elemen_trans}).
Describe the transformation matrices for:
\begin{itemize}
\item Interchanging two rows,
\item Multiplying all elements of a row by a scalar,
\item Adding a scalar multiple of one row to another row.
\end{itemize}
Extend this discussion to the three elementary column transformations.

\item \label{problem:row_det_chg} Consider the three types of elementary row transformation defined in Definition~\ref{definition:elemen_trans}. 
Show the following effects on the determinant:
\begin{itemize}
\item  Type-1 (row interchange): Multiplies the determinant by $-1$.
\item Type-2 (row scaling): Multiplies the determinant by the scaling factor.
\end{itemize}

\item Let $\bP\in \real^{n\times n}$ be a permutation matrix. Discuss how the matrix $\bP$ can be converted to the identity matrix using at most $n$ elementary row transformations of a single type. Use this fact to express $\bA$ as the product of at most $n$ elementary matrix operators.

\item Suppose we reorder all the columns of an invertible matrix $\bA$  using a random permutation, and we already know $\bA^{-1}$, the inverse of the original matrix. Show how we can compute the inverse of this reordered matrix directly from $\bA^{-1}$ without having to perform a full inversion from scratch. Use elementary matrices in the explanation.

\item Prove or disprove each of the following statements by providing a counterexample if applicable:
\begin{enumerate}
\item The sequence in which two elementary row transformations are applied to a matrix does not influence the final outcome.
\item The sequence in which one elementary row transformation and one elementary column transformation are applied to a matrix does not influence the final outcome.
\end{enumerate}

\item \label{prob:compl1} \textbf{Complexity of vector inner product.}
	Given two vectors $\bv,\bw\in \real^{n}$, the  inner product of the two vectors $\bv^\top\bw$ is calculated as $\bv^\top\bw=v_1w_1+v_2w_2+\ldots v_nw_n$. Show   that the computational complexity of evaluating the inner product is $2n-1$ floating-point operations (flops).

\item \label{prob:compl2} \textbf{Complexity of matrix multiplication.} Given two matrices $\bA\in\real^{m\times n}$ and $\bB\in \real^{n\times k}$, show that the computational complexity of their product, $\bA\bB$, is $mk(2n-1)$ flops.

\item Discuss and provide algorithms used to compute the LU decomposition of a matrix. Use the results from Problems~\ref{prob:compl1} and \ref{prob:compl2} to determine the computational complexity of the decomposition.

\item \textbf{Matrix inversion lemma.}
Let $\bA$ be an invertible $n \times n$ matrix and let $\bB, \bC$ be $n \times k$ nonzero matrices for some small value of $k$. Show that the matrix $\bA + \bB\bC^\top$ is invertible if and only if the $k \times k$ matrix $(\bI + \bC^\top\bA^{-1}\bB)$ is invertible. Furthermore, show that the inverse is given by the following:
$$
(\bA + \bB\bC^\top)^{-1} = \bA^{-1} - \bA^{-1}\bB(\bI + \bC^\top\bA^{-1}\bB)^{-1}\bC^\top\bA^{-1}.
$$
This is also known as the \textit{Sherman--Morrison--Woodbury identity}. \textit{Hint: Use Schur complements}.

\item \textbf{Matrix inversion lemma.}
Let $\bP\in\real^{n\times n}$ be any matrix. Show that
\begin{equation}\label{equation:inv_sum_idd}
(\bI +\bP)^{-1} = \bI -  (\bI+\bP)^{-1}\bP= \bI -  \bP(\bI+\bP)^{-1}.
\end{equation}
\textit{Hint: Premultiply and postmultiply the  above identities with  appropriate matrices. }

\item \textbf{Push-through identity \citep{aggarwal2020linear}.}
Let $\bA$ and $\bB$ be two $m \times n$ matrices. Show the following result:
\begin{equation}
\bA^\top(\bI_m + \bB\bA^\top)^{-1} = (\bI_n + \bA^\top\bB)^{-1}\bA^\top.
\end{equation}
Use the above result to show the following for any $m \times n$ matrix $\bC$ and scalar $\lambda > 0$:
\begin{equation}
\bC^\top(\lambda \bI_m + \bC\bC^\top)^{-1} = (\lambda \bI_n + \bC^\top \bC)^{-1}\bC^\top.
\end{equation}
The push-through identity derives its name from the fact that we push in a matrix on the left and it comes out on the right. 

\item Show that the inverse of a symmetric matrix is symmetric using LU decomposition.

\item Consider the $3\times 3$ \textit{row addition transformation} $\bA=\begin{bmatrixscript}
	1 & c & 0\\
	0 & 1 & 0\\
	0 & 0 & 1
\end{bmatrixscript}$ with $c\neq 0$. When  multiplied with another matrix on the right, it adds a multiple of one row to another row.
Derive the inverse of $\bA$  by inverting a sum of matrices and using \eqref{equation:inv_sum_idd}.
\end{problemset}

\newpage
\chapter{Cholesky Decomposition}

\section{Cholesky Decomposition}


The property of positive definiteness or positive semidefiniteness is a significant characteristic of matrices. It not only provides insights into a matrix's fundamental nature but also underpins various mathematical and computational applications.
In this chapter, we introduce decomposition methods for two special types of matrices, examining their unique properties and applications. We begin with the widely recognized Cholesky decomposition, a powerful method  for revealing the positive definiteness of a matrix by factoring it into the product of a lower (or an upper) triangular matrix and its transpose. This decomposition facilitates numerical computations and is indispensable in optimization, statistical modeling, and other fields where ensuring positive definiteness is essential.

\index{Decomposition: Cholesky}
\begin{theoremHigh}[Cholesky decomposition]\label{theorem:cholesky-factor-exist}
Every \textit{positive definite} (PD) matrix $\bA\in \real^{n\times n}$ can be decomposed as 
$$
\bA = \bR^\top\bR,
$$
where $\bR \in \real^{n\times n}$ is an upper triangular matrix \textbf{with positive diagonal entries}. This factorization is known as the \textit{Cholesky decomposition}  of $\bA$, and  $\bR$ is referred to  as the \textit{Cholesky factor} or \textit{Cholesky triangle} of $\bA$.

Alternatively, $\bA$ can be expressed as $\bA=\bL\bL^\top$, where $\bL=\bR^\top$ is a lower triangular matrix \textit{with positive diagonal elements}.
Importantly, the Cholesky decomposition is \textbf{unique} (see Corollary~\ref{corollary:unique-cholesky-main}).
\end{theoremHigh}

The Cholesky decomposition derives its name from the French military officer and mathematician, \textit{Andr\'{e}-Louis Cholesky} (1875--1918), credited with its development during his surveying work. 
Similar to the LU decomposition, the Cholesky decomposition is primarily used to solve linear systems involving positive definite matrices. The approach for solving such systems parallels that of the LU decomposition, as discussed in Section~\ref{section:lu-linear-sistem}, and will not be repeated here.

\index{Gaussian process}
\index{Variational autoencoder}
\index{Generative process}
\paragraph{Applications: an overview.}
We will discuss additional applications of the Cholesky decomposition in Sections~\ref{section:app_cho_md_newton}$\sim$\ref{section:cho_lowrank}. Here, we provide a brief overview of its general use.
Given a covariance matrix $\bSigma$, by applying the Cholesky decomposition $\bSigma = \bL\bL^\top$, we can transform independent standard normal random variables $\rvz$ into multivariate normal random variables $\rvx$ with covariance matrix $\bSigma$ through the transformation $\rvx = \bL \rvz$.
Mathematically, this process can be described as follows:
\begin{itemize}
\item Let $\rvz = [\rz_1, \rz_2, \ldots, \rz_n]^\top$ be a vector of independent standard normal random variables, i.e., $\rz_i \sim \normal(0, 1)$ for all $i$.
\item Let $\bSigma$ denote  the corresponding $n \times n$ positive definite covariance matrix.
\item The Cholesky decomposition of $\bSigma$ gives us $\bSigma = \bL\bL^\top$, where $\bL$ is a lower triangular matrix with positive diagonal entries.
\item Then, the random vector $\rvx = \bL \rvz$ follows a multivariate normal distribution with mean vector $\bmu = \bzero$ and covariance matrix $\bSigma$, since
$
\Cov[\rvx] = \Cov[\bL \rvz] = \bL \Cov[\rvz] \bL^\top = \bSigma.
$
\end{itemize}
This transformation plays a crucial role in simulation processes across various domains.
In finance, \textit{Monte Carlo simulations} are widely used to model portfolios containing multiple assets. Since asset returns are often correlated, accurately capturing these dependencies is essential \citep{lu2022autoencoding}. Using the Cholesky decomposition, one can generate simulated paths of asset returns that reflect historical correlations embedded in the covariance matrix. Applications include \textit{Value-at-Risk (VaR)} estimation, stress testing, and pricing multi-asset derivatives \citep{turkay2003correlation}.
The method is also valuable in machine learning, particularly in sampling from \textit{Gaussian processes}, which are used in Gaussian process regression and Bayesian optimization \citep{williams2006gaussian, lu2021rigorous}.
Additionally, in probabilistic generative models such as \textit{variational autoencoders (VAEs)} or diffusion models, sampling from a multivariate normal distribution is a key step in the generation process \citep{kingma2019introduction, lu2023bayesian}.
In \textit{quantization} of large language or neural network models, the computational process of the Cholesky decomposition can be applied for efficient computation of quantization \citep{frantar2022gptq}.

In summary, the Cholesky decomposition offers a computationally efficient method for generating correlated random variables from uncorrelated ones. This makes it an essential tool in stochastic simulation and probabilistic modeling across diverse fields.

On the other hand, this decomposition has wide applications in optimization algorithms.
For example, the goal of a \textit{quadratic programming} problem is to minimize a quadratic function while satisfying a set of linear constraints. The standard form of a quadratic programming problem can be expressed as:
$$
\min_{\bx} \frac{1}{2} \bx^\top \bA \bx - \bb^\top \bx,
$$
where $\bx$ is the vector of decision variables, $\bA$ is a symmetric positive definite matrix, and $\bb$ is a constant vector.
When the matrix $\bA$ is symmetric and positive definite, Cholesky decomposition can be used to simplify the solving process. The specific steps follow by
replacing $\bA$ in the original objective function with its Cholesky decomposition $\bL\bL^\top$, resulting in the new objective function $\frac{1}{2} \bx^\top (\bL\bL^\top) \bx - \bb^\top \bx$.
Let $\by = \bL^\top \bx$, then the optimization problem becomes:
$$
\min_{\by} \frac{1}{2} \by^\top \by - (\bL^{-\top} \bb)^\top \by.
$$
This is a simple quadratic function in terms of $\by$, which is easy to solve.
By solving the quadratic function in terms of $\by$, we obtain the optimal solution $\by^*$.
This, in turn, yields the optimal solution $\bx^*$ of the original problem using backward substitution.

\section{Existence of Cholesky Decomposition via Recursive Calculation}\label{section:recursi_choles}

In this section, we demonstrate the existence of the Cholesky decomposition using recursive calculation. In Section~\ref{section:cholesky-by-qr-spectral}, we will provide an alternative proof of its existence using QR decomposition and spectral decomposition.
Before proving the existence of the Cholesky decomposition, we introduce the following definitions and lemmas.
\begin{definition}[Positive definite and positive semidefinite\index{Positive definite}\index{Positive semidefinite}]\label{definition:psd-pd-defini}
A matrix $\bA\in \real^{n\times n}$ is \textit{positive definite (PD)} if $\bx^\top\bA\bx>0$ for all nonzero $\bx\in \real^n$, denoted as $\bA\succ \bzero$.
And a matrix $\bA\in \real^{n\times n}$ is \textit{positive semidefinite (PSD)} if $\bx^\top\bA\bx \geq 0$ for all $\bx\in \real^n$, denoted as $\bA\succeq \bzero$.
\footnote{In discussions regarding positive definite or positive semidefinite matrices, it is essential to note that these matrices are symmetric. Therefore, the concept of a positive definite matrix holds significance only in the context of symmetric matrices.}
\end{definition}

One requirement for the existence of the Cholesky decomposition is the concept of positive definiteness. Several key properties of positive definite matrices are summarized below:
\begin{tcolorbox}[title={Positive Definite Matrix Property 1 of 5},colback=\mdframecolorTheorem]
 A matrix $\bA$ is positive definite if and only if all of its eigenvalues are positive. Similarly, $\bA$ is positive semidefinite if and only if all of its eigenvalues are nonnegative.
A detailed proof of this equivalence is presented in Section~\ref{section:equivalent-pd-psd}, based on the spectral theorem.	

\end{tcolorbox}

While not all components of a positive definite matrix need to be positive, it is guaranteed  that the diagonal components of such a matrix are positive:
\begin{tcolorbox}[title={Positive Definite Matrix Property 2 of 5},colback=\mdframecolorTheorem]
\begin{lemma}[Positive diagonals of positive definite matrices]\label{lemma:positive-in-pd}
The diagonal elements of a positive definite matrix $\bA$ are all \textit{positive}. Likewise, the diagonal elements of a positive semidefinite matrix $\bB$ are all \textit{nonnegative}.
\end{lemma}
\end{tcolorbox}
\begin{proof}[of Lemma~\ref{lemma:positive-in-pd}]
By definition, for a positive definite matrix $\bA$, we have  $\bx^\top\bA \bx >0$ for all nonzero vectors $\bx$. In particular, let $\bx=\be_i$, where $\be_i$ is the $i$-th standard basis vector with 1 in the $i$-th position and 0 elsewhere. Then:
$$
\be_i^\top\bA \be_i = a_{ii}>0, \qquad \forall\, i \in \{1, 2, \ldots, n\},
$$	
where $a_{ii}$ represents the $i$-th diagonal component. A similar argument applies to PSD matrices, where $a_{ii}\geq 0$. This completes the proof.
\end{proof}

Like the LU decomposition, the existence of the Cholesky decomposition also relies on properties of the Schur complement.
\begin{tcolorbox}[title={Positive Definite Matrix Property 3 of 5},colback=\mdframecolorTheorem]
\begin{lemma}[Schur complement of positive definite matrices\index{Schur complement}]\label{lemma:pd-of-schur}
For a positive definite matrix $\bA\in \real^{n\times n}$, the Schur complement of $a_{11}$ is given by $\bS_{n-1}=\bA_{2:n,2:n}-\frac{1}{a_{11}} \bA_{2:n,1}\bA_{2:n,1}^\top$. The Schur complement $\bS_{n-1}$ is also positive definite. 

\paragraph{A note on notation.} The subscript $n-1$   indicates that $\bS_{n-1}$ is an $(n-1)\times (n-1)$ matrix obtained from an $n\times n$ positive definite matrix. 
This notation will be used consistently in the following sections.
\end{lemma}	
\end{tcolorbox}

\begin{proof}[of Lemma~\ref{lemma:pd-of-schur}]
Let $\bv\in \real^{n-1}$ be  any nonzero vector. Construct a corresponding vector $\bx\in \real^n$ as 
$
\bx = 
\begin{bmatrixscript}
-\frac{1}{a_{11}} \bA_{2:n,1}^\top  \bv \\
\bv
\end{bmatrixscript},
$
which is nonzero. Now compute:
$$
\begin{aligned}
\bx^\top\bA\bx 
&= \left[-\frac{1}{a_{11}} \bv^\top \bA_{2:n,1}\qquad \bv^\top\right]
\begin{bmatrix}
a_{11} & \bA_{2:n,1}^\top \\
\bA_{2:n,1} & \bA_{2:n,2:n}
\end{bmatrix}
\begin{bmatrix}
-\frac{1}{a_{11}} \bA_{2:n,1}^\top  \bv \\
\bv
\end{bmatrix} \\
&= \left[-\frac{1}{a_{11}} \bv^\top \bA_{2:n,1}\qquad \bv^\top\right]
\begin{bmatrix}
0 \\
\bS_{n-1}\bv
\end{bmatrix} 
= \bv^\top\bS_{n-1}\bv.
\end{aligned}
$$
Since $\bA$ is positive definite, we have $\bx^\top\bA\bx = \bv^\top\bS_{n-1}\bv >0$ for all nonzero $\bv$. Thus,  $\bS_{n-1}$ is positive definite as well.
\end{proof}

This argument extends to PSD matrices as well: if $\bA$ is PSD, then its Schur complement $\bS_{n-1}$ is also PSD.

In the proof of Theorem~\ref{theorem:lu-factorization-with-permutation}, we showed that the Schur complement $\bS_{n-1}=\bA_{2:n,2:n}-\frac{1}{a_{11}} \bA_{2:n,1}\bA_{2:n,1}^\top$ is  nonsingular if $\bA$ is nonsngular and $a_{11}\neq 0$. Similarly, the Schur complement of $a_{nn}$ in $\bA$ is given by ${\bS}^\prime_{n-1} =\bA_{1:n-1,1:n-1} - \frac{1}{a_{nn}}\bA_{1:n-1,n} \bA_{1:n-1,n}^\top$, which is also positive definite if $\bA$ is positive definite.This property is critical in proving that the leading principal minors of a PD matrix are all positive; further details can be found in Section~\ref{appendix:leading-minors-pd}.

\index{Recursive algorithm}
Using these results, we now demonstrate the existence of the Cholesky decomposition via recursion.
\begin{proof}[{of Theorem~\ref{theorem:cholesky-factor-exist}: existence of Cholesky decomposition recursively}]
For any positive definite matrix $\bA$, note that $a_{11} > 0$ by Lemma~\ref{lemma:positive-in-pd}. We can express $\bA$ as:
$$
\setlength{\arraycolsep}{2pt}
\begin{aligned}
\bA &= 
\begin{bmatrix}
a_{11} & \bA_{2:n,1}^\top \\
\bA_{2:n,1} & \bA_{2:n,2:n}
\end{bmatrix} 
=\begin{bmatrix}
\sqrt{a_{11}} &\bzero\\
\frac{1}{\sqrt{a_{11}}} \bA_{2:n,1} &\bI 
\end{bmatrix}
\begin{bmatrix}
\sqrt{a_{11}} & \frac{1}{\sqrt{a_{11}}}\bA_{2:n,1}^\top \\
\bzero & \bA_{2:n,2:n}-\frac{1}{a_{11}} \bA_{2:n,1}\bA_{2:n,1}^\top
\end{bmatrix}\\
&=\begin{bmatrix}
\sqrt{a_{11}} &\bzero\\
\frac{1}{\sqrt{a_{11}}} \bA_{2:n,1} &\bI 
\end{bmatrix}
\begin{bmatrix}
1 & \bzero \\
\bzero & \bA_{2:n,2:n}-\frac{1}{a_{11}} \bA_{2:n,1}\bA_{2:n,1}^\top
\end{bmatrix}
\begin{bmatrix}
\sqrt{a_{11}} & \frac{1}{\sqrt{a_{11}}}\bA_{2:n,1}^\top \\
\bzero & \bI
\end{bmatrix}
=\bR_1^\top
\begin{bmatrix}
1 & \bzero \\
\bzero & \bS_{n-1}
\end{bmatrix}
\bR_1,
\end{aligned}
$$
where  
$\bR_1 = 
\scriptsize
\begin{bmatrix}
\sqrt{a_{11}} & \frac{1}{\sqrt{a_{11}}}\bA_{2:n,1}^\top \\
\bzero & \bI
\end{bmatrix}.
$
By Lemma~\ref{lemma:pd-of-schur}, $\bS_{n-1}$ is positive definite. Thus, we can factor it similarly:
$
\bS_{n-1}=
\widehat{\bR}_2^\top
\scriptsize
\begin{bmatrix}
1 & \bzero \\
\bzero & \bS_{n-2}
\end{bmatrix}
\normalsize
\widehat{\bR}_2
$,
where $\bS_{n-2}$ is also positive definite.
Substituting this back, we obtain:
$$
\setlength{\arraycolsep}{1.2pt}
\footnotesize
\begin{aligned}
\bA &= \bR_1^\top
\begin{bmatrix}
1 & \bzero \\
\bzero & \widehat{\bR}_2^\top
\begin{bmatrix}
1 & \bzero \\
\bzero & \bS_{n-2}
\end{bmatrix}
\widehat{\bR}_2.
\end{bmatrix}
\bR_1
=
\bR_1^\top
\begin{bmatrix}
1 &\bzero \\
\bzero &\widehat{\bR}_2^\top
\end{bmatrix}
\begin{bmatrix}
1 &\bzero \\
\bzero &\begin{bmatrix}
1 & \bzero \\
\bzero & \bS_{n-2}
\end{bmatrix}
\end{bmatrix}
\begin{bmatrix}
1 &\bzero \\
\bzero &\widehat{\bR}_2
\end{bmatrix}
\bR_1
=
\bR_1^\top \bR_2^\top
\begin{bmatrix}
1 &\bzero \\
\bzero &\begin{bmatrix}
1 & \bzero \\
\bzero & \bS_{n-2}
\end{bmatrix}
\end{bmatrix}
\bR_2 \bR_1.
\end{aligned}
$$
Repeating this process recursively, we eventually express $\bA$ as:
$$
\begin{aligned}
\bA &= \bR_1^\top\bR_2^\top\ldots \bR_n^\top \bR_n\ldots \bR_2\bR_1
= \bR^\top \bR,
\end{aligned}
$$
where $\bR_1, \bR_2, \ldots, \bR_n$ are upper triangular matrices with positive diagonal elements, and $\bR=\bR_1\bR_2\ldots\bR_n$ is also an upper triangular matrix with positive diagonal elements, from which the result follows.
\end{proof}

The above process can also be used to compute the Cholesky decomposition and analyze the computational complexity of the algorithm.

To go in the reverse direction, we can prove that the scatter matrix $\bR^\top\bR$ is positive definite under mild conditions.
\begin{lemma}[$\bR^\top\bR$ is PD]\label{lemma:r-to-pd}\index{Positive definite}
Given any upper triangular matrix $\bR$ with positive diagonal elements, the matrix 
$
\bA = \bR^\top\bR
$
is positive definite.
\end{lemma}
\begin{proof}[of Lemma~\ref{lemma:r-to-pd}]
Since  $\bR$ has positive diagonals, it has full column rank, and its null space is of dimension 0 by the fundamental theorem of linear algebra (Theorem~\ref{theorem:fundamental-linear-algebra}). 
Consequently, $\bR\bx \neq \bzero$ for any nonzero vector $\bx$. Therefore, $\bx^\top\bA\bx = \norm{\bR\bx}^2 >0$ for any nonzero vector $\bx$.
\end{proof}
This lemma extends to any $\bR$ with linearly independent columns.

\paragraph{A word on the two claims.} Combining  Theorem~\ref{theorem:cholesky-factor-exist} and Lemma~\ref{lemma:r-to-pd}, we conclude that a matrix $\bA$ is positive definite if and only if $\bA$ can be factored as $\bA=\bR^\top\bR$, where $\bR$ is an upper triangular matrix with positive diagonals.

\begin{algorithm}[H] 
	\caption{Cholesky Decomposition via Recursive Algorithm: Pseudo Code} 
	\label{alg:compute-choklesky11} 
	\begin{algorithmic}[1] 
		\Require 
		Positive definite matrix $\bA$ with size $n\times n$; 
		\For{$k=1$ to $n$} \Comment{compute the $k$-th row of $\bR$}
		\State $r_{kk} \leftarrow \sqrt{a_{kk}}$; \Comment{first element of $k$-th row}
		\State $\bR_{k,k+1:n} \leftarrow \frac{1}{r_{kk}} \bA_{k,k+1:n}$; \Comment{the rest elements of $k$-th row}
		\State $\bA_{k+1:n,k+1:n} \leftarrow \bA_{k+1:n,k+1:n} - \bR_{k,k+1:n}^\top\bR_{k,k+1:n}$; 
		\EndFor
		\State Output $\bA=\bR^\top\bR$.
	\end{algorithmic} 
\end{algorithm}

\paragraph{An alternative perspective of the recursive algorithm.}
The previous proof of the Cholesky decomposition can also be used to compute the decomposition itself; see Algorithm~\ref{alg:compute-choklesky11}.
Since $\bL= \bR^\top$ is lower triangular. The lower triangular factor $\bL$ can be computed as a product of a sequence  of lower triangular matrices. 
To see this, we have
$$
\bA=
\begin{aligned}
\begin{bmatrix}
a_{11} & \bA_{1,2:n} \\
\bA_{2:n,1} & \bA_{2:n,2:n}
\end{bmatrix} 
=
\begin{bmatrix}l_{11} & \bzero \\ 
\bL_{21} & \bL_{22} 
\end{bmatrix}
\begin{bmatrix}
l_{11} & \bL_{21}^\top \\ 
\bzero & \bL_{22}^\top 
\end{bmatrix} 
\end{aligned} 
= \bL\bL^\top.
$$
Then we still have 
$$
\begin{bmatrix}
a_{11} & \bA_{1,2:n} \\
\bA_{2:n,1} & \bA_{2:n,2:n}
\end{bmatrix}  
= 
\begin{bmatrix}
l_{11}^2 & l_{11}\bL_{21}^\top \\ 
l_{11}\bL_{21} & \bL_{21}\bL_{21}^\top+\bL_{22}\bL_{22}^\top 
\end{bmatrix} 
\implies
\begin{cases}
l_{11} &= \sqrt{a_{11}}; \\ 
\bL_{21} &= \frac{1}{l_{11}}\bA_{2:n,1}; \\ 
\bL_{22}\bL_{22}^\top &= \bA_{2:n,2:n} - \bL_{21}\bL_{21}^\top .
\end{cases}
$$
The second perspective involves constructing $n+1$ set of $n \times n$ matrices: $\bA^{(1)},\bA^{(2)}, \ldots, \bA^{(n+1)}$, where $\bA^{(1)}=\bA$, and we want to obtain $\bA^{(n+1)}=\bI$ via the relation:
\begin{equation}\label{equation:choes_recur_secon}
\bA^{(i)} = \bL^{(i)}\bA^{(i+1)}\bL^{(i)^\top}, \ \forall\, i \in\{1,2,\ldots,n\}.
\end{equation}
If these $\bL^{(i)}, \ \forall\, i$ are lower triangular, then we obtain the Cholesky decomposition by 
$$
\bA= 
(\bL^{(1)} \bL^{(2)}\ldots \bL^{(n)}) (\bL^{(1)} \bL^{(2)}\ldots \bL^{(n)})^\top
= \bL\bL^\top.
$$
This is indeed the case. To see this, we can construct 
$$
\bA^{(i)} =
\begin{bmatrix} 
\bI_{i-1} & 0 & \bzero \\ 
0 & a_{ii} & \bb_i^\top \\ 
\bzero & \bb_i & \bB^{(i)} 
\end{bmatrix} 
\qquad\text{and}\qquad
\bL^{(i)} = 
\begin{bmatrix} 
\bI_{i-1} & 0 & \bzero \\ 
0 & \sqrt{a_{ii}} & \bzero \\ 
\bzero & \frac{1}{\sqrt{a_{ii}}}\bb_i & \bI_{n-i} 
\end{bmatrix},
$$
satisfying  $\bA^{(i)} = \bL^{(i)}\bA^{(i+1)}(\bL^{(i)})^\top$:
$$
\begin{aligned} \bA^{(i+1)} 
&= \begin{bmatrix} 
\bI_{i-1} & 0 & \bzero \\ 
0 & 1 & \bzero \\ 
\bzero & \bzero & \bB^{(i)}-\frac{1}{a_{ii}}\bb_i\bb_i^\top 
\end{bmatrix} 
= 
\begin{bmatrix} 
\bI_i & 0 & \bzero \\ 
0 & a_{i+1, i+1} & \bb_{i+1}^\top \\ 
\bzero & \bb_{i+1} & \bB^{(i+1)} 
\end{bmatrix}.
\end{aligned}
$$
Therefore, $\bA$ can be decomposed as a set of lower triangular matrices in \eqref{equation:choes_recur_secon}.
Using the result in Exercise~\ref{exercise:choe_recur_sec} can show that the algorithm for this perspective is equivalent to Algorithm~\ref{alg:compute-choklesky11}.

\begin{exercise}\label{exercise:choe_recur_sec}
Verify that $\bL^{(i)}_{i:,i} = \bL_{i:,i}, i = 1,2,\ldots, n$.
\end{exercise}

\section{Sylvester's Criterion: Leading Principal Minors of PD Matrices}\label{appendix:leading-minors-pd}

In Lemma~\ref{lemma:pd-of-schur}, we proved that for any positive definite matrix $\bA\in \real^{n\times n}$, the Schur complement of $a_{11}$ is given by $\bS_{n-1}=\bA_{2:n,2:n}-\frac{1}{a_{11}} \bA_{2:n,1}\bA_{2:n,1}^\top$, which is also positive definite. 
Similarly, the Schur complement of $a_{nn}$,  $\bS_{n-1}^\prime = \bA_{1:n-1,1:n-1} -\frac{1}{a_{nn}} \bA_{1:n-1,n}\bA_{1:n-1,n}^\top$, is also positive definite.

We now claim that all leading principal minors (Definition~\ref{definition:leading-principle-minors}) of a positive definite matrix $\bA \in \real^{n \times n}$ are positive, a result known as \textit{Sylvester's criterion} \citep{swamy1973sylvester, gilbert1991positive}. Recall that these positive leading principal minors imply the existence of the LU decomposition for any positive definite matrix, as established in Theorem~\ref{theorem:lu-factorization-without-permutation}.

To prove Sylvester's criterion, we begin by establishing the following lemma:
\begin{tcolorbox}[title={Positive Definite Matrix Property 4 of 5},colback=\mdframecolorTheorem]
\begin{lemma}[Quadratic PD]\label{lemma:quadratic-pd}
Let $\bE$ be any invertible matrix. Then $\bA$ is positive definite if and only if $\bE^\top\bA\bE$ is also positive definite.
\end{lemma}
\end{tcolorbox}
\begin{proof}[of Lemma~\ref{lemma:quadratic-pd}]
If $\bA$ is positive definite, then for any nonzero vector $\bx$, $\bx^\top \bE^\top\bA\bE \bx = \by^\top\bA\by > 0$, since $\bE$ is invertible such that $\bE\bx$ is nonzero. \footnote{Since the null space of $\bE$ is of dimension 0 and the only solution for $\bE\bx=\bzero$ is the trivial solution $\bx=\bzero$.} Thus, $\bE^\top\bA\bE$ is PD.

Conversely, if $\bE^\top\bA\bE$ is positive definite, for any nonzero $\bx$, $\bx^\top \bE^\top\bA\bE \bx>0$. For any nonzero $\by$, there exists a nonzero $\bx$ such that $\by =\bE\bx$, since $\bE$ is invertible. Hence, $\bA$ is also PD.
\end{proof}

We now provide a rigorous proof of Sylvester's criterion.
\begin{tcolorbox}[title={Positive Definite Matrix Property 5 of 5},colback=\mdframecolorTheorem]
\begin{theorem}[Sylvester's criterion\index{Sylvester's criterion}]\label{lemma:sylvester-criterion}
A real symmetric matrix $\bA\in \real^{n\times n}$ is positive definite if and only if all of its leading principal minors  are positive. 

\end{theorem}
\end{tcolorbox}

\begin{proof}[of Theorem~\ref{lemma:sylvester-criterion}]
We  prove  the forward implication by induction.  Base case ($n = 1$): Since all the components on the diagonal of positive definite matrices are  positive (Lemma~\ref{lemma:positive-in-pd}), for a scalar matrix $\bA$, $\det(\bA) > 0$ if $\bA$ is positive definite.

Assume all leading principal minors of any $k \times k$ positive definite matrix are positive. For a $(k+1) \times (k+1)$ positive definite matrix $\bM$, expressed in block form as $\bM=\scriptsize\begin{bmatrix}
\bA & \bb\\
\bb^\top & d
\end{bmatrix}$, where $\bA$ is a $k\times k$ positive definite submatrix. Its Schur complement of $d$, $\bS_{k} = \bA - \frac{1}{d} \bb\bb^\top$, is also positive definite, and its determinant is positive by the inductive hypothesis. Therefore, $\det(\bM) = \det(d)\det( \bA - \frac{1}{d} \bb\bb^\top) $=
\footnote{By the fact that if matrix $\bM$ has a block formulation: $\bM=\tiny\begin{bmatrix}
\bA & \bB \\
\bC & \bD 
\end{bmatrix}$, then $\det(\bM) = \det(\bD)\det(\bA-\bB\bD^{-1}\bC)$.}  
$d\cdot \det( \bA - \frac{1}{d} \bb\bb^\top)>0$, establishing the result for $(k+1) \times (k+1)$ matrices.

Conversely, if  all the leading principal minors of $\bA\in \real^{n\times n}$ are positive, then all  leading principal submatrices are nonsingular. Denote the $(i,j)$-th entry of $\bA$ as $a_{ij}$. By assumption, $a_{11} > 0$.
To simplify $\bA$, subtract appropriate multiples of its first row from the rows below to zero out the entries in the first column beneath the diagonal element $a_{11}$. This operation can be expressed as:
$$
\bA = 
\begin{bmatrix}
a_{11} & a_{12} & \ldots &a_{1n}\\
a_{21} & a_{22} & \ldots &a_{2n}\\
\vdots & \vdots & \ddots &\vdots\\
a_{n1} & a_{n2} & \ldots &a_{nn}\\
\end{bmatrix}
\stackrel{\bE_1 \bA}{\longrightarrow}
\begin{bmatrix}
a_{11} & a_{12} & \ldots &a_{1n}\\
0 & a_{22} & \ldots &a_{2n}\\
\vdots & \vdots & \ddots &\vdots\\
0 & a_{n2} & \ldots &a_{nn}\\
\end{bmatrix}.
$$

Next, subtract appropriate multiples of the first column of $\bE_1 \bA$ from the other columns to zero out the entries in the first row to the right of the diagonal element $a_{11}$. Due to the symmetry of $\bA$, this operation can also be performed by multiplying $\bE_1 \bA$ on the right by $\bE_1^\top$. The result is:
$$
\bA = 
\begin{bmatrix}
a_{11} & a_{12} & \ldots &a_{1n}\\
a_{21} & a_{22} & \ldots &a_{2n}\\
\vdots & \vdots & \ddots &\vdots\\
a_{n1} & a_{n2} & \ldots &a_{nn}\\
\end{bmatrix}
\stackrel{\bE_1 \bA}{\longrightarrow}
\begin{bmatrix}
a_{11} & a_{12} & \ldots &a_{1n}\\
0 & a_{22} & \ldots &a_{2n}\\
\vdots & \vdots & \ddots &\vdots\\
0 & a_{n2} & \ldots &a_{nn}\\
\end{bmatrix}
\stackrel{\bE_1 \bA\bE_1^\top}{\longrightarrow}
\begin{bmatrix}
a_{11} & 0 & \ldots &0\\
0 & a_{22} & \ldots &a_{2n}\\
\vdots & \vdots & \ddots &\vdots\\
0 & a_{n2} & \ldots &a_{nn}\\
\end{bmatrix}.
$$
This operation preserves the principal minors of $\bA$. Consequently, the leading principal minors of $\bE_1 \bA \bE_1^\top$ are identical to those of $\bA$.

By repeating this process iteratively, we transform $\bA$ into a diagonal matrix of the form $\bE_n \ldots \bE_1 \bA \bE_1^\top \ldots \bE_n^\top$, where the diagonal entries match the diagonal entries of $\bA$ and are positive.
Let $\bE = \bE_n \ldots \bE_2 \bE_1$, which is an invertible matrix. Clearly, $\bE \bA \bE^\top$ is PD, which implies that $\bA$ is also PD, as per Lemma~\ref{lemma:quadratic-pd}.
\end{proof}

\section{Existence of Cholesky Decomposition via  LU  without Permutation}
By Theorem~\ref{lemma:sylvester-criterion} on Sylvester's criterion and  Theorem~\ref{theorem:lu-factorization-without-permutation} regarding the existence of an LU decomposition without permutation, a unique LU decomposition exists for a positive definite matrix $\bA$ of the form $\bA = \bL \bU_0$, where $\bL$ is a unit lower triangular matrix and $\bU_0$ is an upper triangular matrix.
It is also established that \textit{the signs of the pivots of a symmetric matrix are the same as the signs of the eigenvalues} \citep{strang1993introduction}: \index{Pivot}
$$
\text{number of positive pivots = number of positive eigenvalues. }
$$
The decomposition $\bA = \bL \bU_0$ can be expressed as follows:
$$
\begin{aligned}
\bA = \bL\bU_0 &=
\begin{bmatrix}
1 & 0 & \ldots & 0 \\
l_{21} & 1 & \ldots & 0\\
\vdots & \vdots & \ddots & \vdots\\
l_{n1} & l_{n2} & \ldots & 1
\end{bmatrix}
\begin{bmatrix}
u_{11} & u_{12} & \ldots & u_{1n} \\
0 & u_{22} & \ldots & u_{2n}\\
\vdots & \vdots & \ddots & \vdots\\
0 & 0 & \ldots & u_{nn}
\end{bmatrix}.\\
\end{aligned}
$$
Here, the diagonal entries of $\bU_0$ correspond to the pivots of $\bA$. Moreover, as all eigenvalues of PD matrices are positive (by Lemma~\ref{lemma:eigens-of-PD-psd}, a consequence of the spectral decomposition), it follows that the diagonal entries of $\bU_0$ are also positive.

Let us now arrange the diagonal entries of $\bU_0$ into a diagonal matrix $\bD$ such that $\bU_0 = \bD \bU$. Substituting this into the decomposition yields:
$$
\setlength{\arraycolsep}{2.5pt}
\begin{aligned}
\bA = \bL\bU_0 =
\begin{bmatrix}
1 & 0 & \ldots & 0 \\
l_{21} & 1 & \ldots & 0\\
\vdots & \vdots & \ddots & \vdots\\
l_{n1} & l_{n2} & \ldots & 1
\end{bmatrix}
\begin{bmatrix}
u_{11} & 0 & \ldots & 0 \\
0 & u_{22} & \ldots & 0\\
\vdots & \vdots & \ddots & \vdots\\
0 & 0 & \ldots & u_{nn}
\end{bmatrix}
\begin{bmatrix}
1 & u_{12}/u_{11} & \ldots & u_{1n}/u_{11} \\
0 & 1 & \ldots & u_{2n}/u_{22}\\
\vdots & \vdots & \ddots & \vdots\\
0 & 0 & \ldots & 1
\end{bmatrix}=\bL\bD\bU.
\end{aligned}
$$
This simplifies to $\bA = \bL \bD \bU$, where $\bU$ is a \textit{unit} upper triangular matrix.
By the uniqueness of the LU decomposition without permutation in Corollary~\ref{corollary:unique-lu-without-permutation} and the symmetry of $\bA$, we conclude that $\bU = \bL^\top$, and hence $\bA = \bL\bD\bL^\top$. 
Since the diagonal entries of $\bD$ are positive, we can define $\bR = \bD^{1/2} \bL^\top$, where $\bD^{1/2}=\diag(\sqrt{u_{11}}, \sqrt{u_{22}}, \ldots, \sqrt{u_{nn}})$.
Thus, we obtain $\bA = \bR^\top \bR$, which represents the Cholesky decomposition of $\bA$. The matrix $\bR$ is upper triangular with positive diagonal entries.

\index{Upper triangular}
\subsection{Diagonal Values of the Upper Triangular Matrix}\label{section:cholesky-diagonals}

Assume that $\bA$ is a positive definite matrix. We can express $\bA$ as a block matrix $\bA = \scriptsize\begin{bmatrix}
\bA_{k} & \bA_{12} \\
\bA_{21} & \bA_{22}
\end{bmatrix}$, where $\bA_{k}\in \real^{k\times k}$. The block LU decomposition of $\bA$ is given by 
$$
\begin{aligned}
\bA &= \begin{bmatrix}
\bA_{k} & \bA_{12} \\
\bA_{21} & \bA_{22}
\end{bmatrix}
=\bL\bU_0=
\begin{bmatrix}
\bL_{k} & \bzero \\
\bL_{21} & \bL_{22}
\end{bmatrix}
\begin{bmatrix}
\bU_{k} & \bU_{12} \\
\bzero & \bU_{22}
\end{bmatrix} 
=\begin{bmatrix}
\bL_{k}\bU_{k} & \bL_{k}\bU_{12} \\
\bL_{21}\bU_{k11}  & \bL_{21}\bU_{12}+\bL_{22}\bU_{22}
\end{bmatrix}.
\end{aligned}
$$
The $k$-th order leading principal minor of $\bA$ is defined as $\Delta_k = \det(\bA_{1:k, 1:k}) = \det(\bA_k)$ (Definition~\ref{definition:leading-principle-minors}). From the block LU decomposition, we have: 
$$
\Delta_k = \det(\bA_{k}) = \det(\bL_{k}\bU_{k} ) = \det(\bL_{k} )\det(\bU_{k}).
$$
Since $\bL_k$ is a unit lower triangular matrix, its determinant is 1.  
Moreover, by the fact that \textit{the determinant of a lower triangular matrix (or an upper triangular matrix) is equal to the product of the diagonal entries},  we obtain:
$$
\Delta_k = \det(\bU_{k})= u_{11} u_{22}\ldots u_{kk},
$$
i.e., the $k$-th order leading principal minor of $\bA$ is equal to the determinant of the $k\times k$ leading submatrix of $\bU_0$, which is also the product of the first $k$ diagonal entries of $\bD$ (from the decomposition $\bA = \bL\bD\bL^\top$). Let $\bD = \diag(d_1, d_2, \ldots, d_n)$. Then:
$$
\Delta_k = d_1 d_2\ldots d_k = \Delta_{k-1}d_k.
$$
The entries of $\bD$ can also be expressed in terms of the leading principal minors of $\bA$ as:
$$
\bD = \diag\left(\Delta_1, \frac{\Delta_2}{\Delta_1}, \ldots, \frac{\Delta_n}{\Delta_{n-1}}\right),
$$
where $\Delta_k$ denotes the $k$-th order leading principal minor of $\bA$, for all $k\in \{1,2,\ldots, n\}$. Consequently, the diagonal entries of $\bR$ (from the Cholesky decomposition $\bA = \bR^\top\bR$) are:
$$
\diag\left(\sqrt{\Delta_1}, \sqrt{\frac{\Delta_2}{\Delta_1}}, \ldots, \sqrt{\frac{\Delta_n}{\Delta_{n-1}}}\right).
$$

\subsection{Block Cholesky Decomposition}
Building on the previous discussion, let $\bA\in\real^{n\times n}$ be a PD matrix expressed in block form as $\bA = \scriptsize\begin{bmatrix}
\bA_k & \bA_{12} \\
\bA_{21} & \bA_{22}
\end{bmatrix}$, where $\bA_k\in \real^{k\times k}$. Its block LU decomposition is given by:
$$
\begin{aligned}
\bA &= \begin{bmatrix}
\bA_k & \bA_{12} \\
\bA_{21} & \bA_{22}
\end{bmatrix}
=\bL\bU_0=
\begin{bmatrix}
\bL_k & \bzero \\
\bL_{21} & \bL_{22}
\end{bmatrix}
\begin{bmatrix}
\bU_k & \bU_{12} \\
\bzero & \bU_{22}
\end{bmatrix}
=\begin{bmatrix}
\bL_k\bU_k & \bL_{k}\bU_{12} \\
\bL_{21}\bU_{k}  & \bL_{21}\bU_{12}+\bL_{22}\bU_{22}
\end{bmatrix}.\\
\end{aligned}
$$
The $k$-th order leading principal submatrix $\bA_k$ of $\bA$ also admits its own LU decomposition: $\bA_k = \bL_k\bU_k$.  This implies that the Cholesky decomposition of an $n \times n$ matrix $\bA$ contains $n-1$ smaller Cholesky decompositions for its leading principal submatrices: $\bA_k = \bR_k^\top\bR_k$, for all $k\in \{1,2,\ldots, n-1\}$. This is particularly true because any leading principal submatrix $\bA_k$ of a positive definite matrix $\bA$ is also positive definite. To see this, for a PD matrix $\bA_{k+1} \in \real^{(k+1) \times (k+1)}$, consider a vector $\bx_k \in \mathbb{R}^k$ extended by a zero element, $\bx_{k+1} = \scriptsize\begin{bmatrix}
	\bx_k\\
	0
\end{bmatrix}$. 
Then, 
$$
\bx_k^\top\bA_k\bx_k = \bx_{k+1}^\top\bA_{k+1}\bx_{k+1} >0.
$$
Thus, $\bA_k$ is positive definite. 
By recursively applying this argument starting from $\bA \in \mathbb{R}^{n \times n}$, we demonstrate that $\bA_{n-1}, \bA_{n-2}, \ldots, \bA_1$ are all positive definite. Consequently, each of these matrices admits a Cholesky decomposition.

\section{Existence of Cholesky Decomposition via Induction}
In the previous section, we demonstrated the existence of the Cholesky decomposition using the LU decomposition without permutation. Building on the proof of the LU decomposition presented in Section~\ref{section:exist-lu-without-perm}, we now demonstrate that the existence of the Cholesky decomposition can also be directly established using mathematical induction.
\begin{proof}[{of Theorem~\ref{theorem:cholesky-factor-exist}: existence of Cholesky decomposition by induction\index{Induction}}]
We will use induction to prove that every $n\times n$ positive definite matrix $\bA$ can be decomposed as $\bA=\bR^\top\bR$. 
The based case for $1\times 1$ matrices is straightforward; setting $R=\sqrt{A}$ renders $A=R^2$.

Assume that  any $k\times k$ positive definite matrix $\bA_k$ admits a Cholesky decomposition. We now show that any $(k+1)\times(k+1)$ PD matrix $\bA_{k+1}$ can also be factored as this Cholesky decomposition. Write  $\bA_{k+1}$ as a block matrix
$
\bA_{k+1} = 
\scriptsize
\begin{bmatrix}
\bA_k & \bb \\
\bb^\top & d
\end{bmatrix},
$
where $\bA_k$ is a $k\times k$ PD matrix. By the inductive hypothesis, $\bA_k $ admits a Cholesky decomposition: $\bA_k = \bR_k^\top\bR_k$. 
Construct the upper triangular matrix
$
\bR_{k+1}=
\scriptsize
\begin{bmatrix}
\bR_k & \br\\
0 & s
\end{bmatrix}.
$
Then,
$$
\bR_{k+1}^\top\bR_{k+1} = 
\begin{bmatrix}
\bR_k^\top\bR_k & \bR_k^\top \br\\
\br^\top \bR_k & \br^\top\br+s^2
\end{bmatrix}.
$$
Therefore, if we can prove $\bR_{k+1}^\top \bR_{k+1} = \bA_{k+1}$ is the Cholesky decomposition of $\bA_{k+1}$ (which requires the value $s$ to be positive), then we complete the proof. That is, we need to prove
$$
\begin{aligned}
\bb &= \bR_k^\top \br
\qquad \text{and}\qquad 
d = \br^\top\br+s^2.
\end{aligned}
$$
Since $\bR_k$ is nonsingular, we have a unique solution for $\br$ and $s$ that 
$$
\begin{aligned}
\br &= \bR_k^{-\top}\bb
\qquad \text{and}\qquad 
s = \sqrt{d - \br^\top\br} = \sqrt{d - \bb^\top\bA_k^{-1}\bb},
\end{aligned}
$$
where we assume $s$ is nonnegative. To ensure $s>0$, note that since $\bA_k$ is PD, from Sylvester's criterion and the fact that if matrix $\bM$ has a block formulation: $\bM=\scriptsize\begin{bmatrix}
\bA & \bB \\
\bC & \bD 
\end{bmatrix}$, then $\det(\bM) = \det(\bA)\det(\bD-\bC\bA^{-1}\bB)$, we have
$$
\det(\bA_{k+1}) = \det(\bA_k)\det(d- \bb^\top\bA_k^{-1}\bb) =  \det(\bA_k)(d- \bb^\top\bA_k^{-1}\bb)>0.
$$
Since $ \det(\bA_k)>0$, it follows that $(d- \bb^\top\bA_k^{-1}\bb)>0$. Thus, $s>0$, and this completes the proof.
\end{proof}

\section{Uniqueness of Cholesky Decomposition}

This uniqueness of the Cholesky decomposition follows directly from the uniqueness of the LU decomposition without permutation. Alternatively, a more detailed proof of this uniqueness is provided below.
\begin{corollary}[Uniqueness of Cholesky decomposition\index{Uniqueness}]\label{corollary:unique-cholesky-main}
The Cholesky decomposition $\bA=\bR^\top\bR$ of any positive definite matrix $\bA\in \real^{n\times n}$ is unique.
\end{corollary}
\begin{proof}[of Corollary~\ref{corollary:unique-cholesky-main}]
Suppose, for contraction, that the Cholesky decomposition is not unique. Then, there exist  two distinct decompositions such that $\bA=\bR_1^\top\bR_1 = \bR_2^\top\bR_2$. This implies
$
\bR_1\bR_2^{-1}= \bR_1^{-\top} \bR_2^\top.
$
From the fact that the inverse of an upper triangular matrix is also an upper triangular matrix, and the product of two upper triangular matrices is also an upper triangular matrix, \footnote{Similarly, the inverse of a lower triangular matrix is also a lower triangular matrix, and the product of two lower triangular matrices is also a lower triangular matrix.} we realize that the left-hand side of the previous equation is an upper triangular matrix, while the right-hand side  is a lower triangular matrix. 
For both sides to be equal, they must both be diagonal matrices, and $\bR_1^{-\top} \bR_2^\top= (\bR_1^{-\top} \bR_2^\top)^\top = \bR_2\bR_1^{-1}$.
Let $\bLambda = \bR_1\bR_2^{-1}= \bR_2\bR_1^{-1}$ be the diagonal matrix. We notice that each diagonal value of $\bLambda$ is the product of the corresponding diagonal values of $\bR_1$ and $\bR_2^{-1}$ (or $\bR_2$ and $\bR_1^{-1}$). Suppose
$$
\bR_1=\begin{bmatrix}
r_{11} & r_{12} & \ldots & r_{1n} \\
0 & r_{22} & \ldots & r_{2n}\\
\vdots & \vdots & \ddots & \vdots\\
0 & 0 & \ldots & r_{nn}
\end{bmatrix}
\qquad \text{and}\qquad 
\bR_2=
\begin{bmatrix}
s_{11} & s_{12} & \ldots & s_{1n} \\
0 & s_{22} & \ldots & s_{2n}\\
\vdots & \vdots & \ddots & \vdots\\
0 & 0 & \ldots & s_{nn}
\end{bmatrix}.
$$
We have
$$
\begin{aligned}
\bR_1\bR_2^{-1}=
\begin{bmatrix}
\frac{r_{11}}{s_{11}} & 0 & \ldots & 0 \\
0 & \frac{r_{22}}{s_{22}} & \ldots & 0\\
\vdots & \vdots & \ddots & \vdots\\
0 & 0 & \ldots & \frac{r_{nn}}{s_{nn}}
\end{bmatrix}
=
\begin{bmatrix}
\frac{s_{11}}{r_{11}} & 0 & \ldots & 0 \\
0 & \frac{s_{22}}{r_{22}} & \ldots & 0\\
\vdots & \vdots & \ddots & \vdots\\
0 & 0 & \ldots & \frac{s_{nn}}{r_{nn}}
\end{bmatrix}
=\bR_2\bR_1^{-1}.
\end{aligned}
$$ 
Given that both $\bR_1$ and $\bR_2$ have positive diagonals, it follows that $r_{11}=s_{11}, r_{22}=s_{22}, \ldots, r_{nn}=s_{nn}$. And $\bLambda = \bR_1\bR_2^{-1}= \bR_2\bR_1^{-1}  =\bI$.
In other words, $\bR_1=\bR_2$, which contradicts the assumption that the decomposition is not unique.
\end{proof}

As a consequence of this proof, if we do not require the diagonal entries of $\bR_1$ and $\bR_2$ to be positive, then for each diagonal entry, we could have $r_{ii} = \pm s_{ii}$. In that case, the factorization $\bA=\bR^\top\bR$ would not be unique.

\section{Computing  Cholesky Decomposition}
We presented a recursive algorithm for computing the Cholesky decomposition in Algorithm~\ref{alg:compute-choklesky11}.
It is also common to compute the Cholesky decomposition using element-level equations derived directly from  the matrix equation $\bA=\bR^\top\bR$. 
Observe that the  $(i,j)$-th entry of $\bA$ is $a_{ij} = \bR_{:,i}^\top \bR_{:,j} = \sum_{k=1}^{i} r_{ki}r_{kj}$ if $i<j$. This further implies, if $i<j$, we have
$$
\begin{aligned}
	a_{ij} &= \bR_{:,i}^\top \bR_{:,j} = \sum_{k=1}^{i} r_{ki}r_{kj} 
	= \sum_{k=1}^{i-1} r_{ki}r_{kj} + r_{ii}r_{ij}
	\implies
	r_{ij} = (a_{ij} - \sum_{k=1}^{i-1} r_{ki}r_{kj})/r_{ii},
	\quad 
	\text{if }i<j.
\end{aligned}
$$
On the other hand, if $i=j$, we have 
\begin{equation}\label{equation:cho_elem_eq2}
\begin{aligned}
	a_{jj} &= \sum_{k=1}^{j} r_{kj}^2=\sum_{k=1}^{j-1} r_{kj}^2 + r_{jj}^2
	&\implies
	r_{jj} = \sqrt{a_{jj} - \sum_{k=1}^{j-1} r_{kj}^2}.
\end{aligned}
\end{equation}
If we equate the elements of $\bR$ by taking a column at a time and start with $r_{11} = \sqrt{a_{11}}$, we arrive at the element-level formulation of the Cholesky decomposition, as described in Algorithm~\ref{alg:compute-choklesky-element-level}.

\begin{algorithm}[H] 
\caption{Cholesky Decomposition Element-Wise: $\bA=\bR^\top\bR$} 
\label{alg:compute-choklesky-element-level} 
\begin{algorithmic}[1] 
\Require 
Positive definite matrix $\bA$ with size $n\times n$; 
\For{$j=1$ to $n$} \Comment{Compute the $j$-th column of $\bR$}
\For{$i=1$ to $j-1$} 
\State $r_{ij} \leftarrow (a_{ij} - \sum_{k=1}^{i-1} r_{ki}r_{kj})/r_{ii}$, since $i<j$;
\EndFor
\State $r_{jj} \leftarrow \sqrt{a_{jj}- \sum_{k=1}^{j-1}r_{kj}^2}$;
\EndFor
\State Output $\bA=\bR^\top\bR$.
\end{algorithmic} 
\end{algorithm}

On the other hand, Algorithm~\ref{alg:compute-choklesky-element-level}  can be adapted to compute the Cholesky decomposition in the form $\bA=\bL\bD\bL^\top$, where $\bL$ is unit lower triangular and $\bD$ is diagonal, as outlined in Algorithm~\ref{alg:compute-choklesky-_ldl}, where Step 3 and Step 5 are derived from (since $l_{ii}=1, \,\forall\, i\in\{1,2,\ldots,n\}$):
$$
\begin{aligned}
a_{jj}&=\sum_{k=1}^{j-1}d_{kk} l_{jk}^2 + d_{jj};
\qquad \;
a_{ij}= d_{jj} l_{ij}+ \sum_{k=1}^{j-1} d_{kk} l_{ik}l_{jk}, \gap \text{if }i>j.
\end{aligned}
$$
\begin{exercise}
Derive the complexity of Algorithms~\ref{alg:compute-choklesky-element-level}  and \ref{alg:compute-choklesky-_ldl}.
\end{exercise}

\index{Condition number}
This alternative  form of the Cholesky decomposition is particularly  useful for estimating  the \textit{condition number}  of a PD matrix. 
In essence, the condition number of a function measures how sensitive its output is to small perturbations in the input; a smaller condition number indicates greater numerical stability.
For positive definite linear systems, the condition number is defined as the ratio of the largest eigenvalue to the smallest eigenvalue of the PD matrix.
The condition number of a positive definite matrix is lower bounded by the diagonal matrix obtained from its Cholesky decomposition:
\begin{equation}\label{equation:cond_pd_ineq}
	\cond(\bA) \geq \cond(\bD).
\end{equation}
This inequality  can be proven by showing that $\lambda_{\max}\geq d_{\max}$ and $\lambda_{\min}\leq d_{\min}$, where $\lambda_{\max}$ and $\lambda_{\min}$ denote  the largest and smallest eigenvalues of $\bA$, and $d_{\max}$ and $d_{\min}$ represent the largest and smallest diagonals of $\bD$.
Therefore, this form of the Cholesky decomposition can be used to improve the numerical behavior of \textit{Newton's method}; see \S~\ref{section:app_cho_md_newton}.

\begin{algorithm}[h] 
	\caption{Cholesky Decomposition Element-Wise: $\bA=\bL\bD\bL^\top$}  
	\label{alg:compute-choklesky-_ldl} 
	\begin{algorithmic}[1] 
		\Require 
		Positive definite matrix $\bA$ with size $n\times n$; 
		\For{$j=1$ to $n$} \Comment{Compute the $j$-th column of $\bL$}
		\State $l_{jj}\leftarrow1$;
		\State $d_{jj}\leftarrow a_{jj}-\sum_{k=1}^{j-1}d_{kk} l_{jk}^2$;
		\For{$i=j+1$ to $n$} 
		\State $c_{ij}\leftarrow a_{ij}-\sum_{k=1}^{j-1}d_{kk} l_{ik}l_{jk}$, since $i>j$;
		\State $l_{ij}\leftarrow \frac{c_{ij}}{d_{jj}}$;
		\EndFor
		\EndFor
		\State Output $\bA=\bL\bD\bL^\top$, where $\bD=\diag(d_{11}, d_{22},\ldots,d_{nn})$.
	\end{algorithmic} 
\end{algorithm}

\section{Final Remarks on Positive Definite Matrices}
In Section~\ref{section:equivalent-pd-psd}, we will prove that a matrix $\bA$ is PD if and only if $\bA$ can be factored as $\bA=\bP^\top\bP$, where $\bP$ is nonsingular. Furthermore, in Section~\ref{section:unique-posere-pd}, we will demonstrate that a PD matrix $\bA$ admits a unique factorization $\bA =\bB^2$, where $\bB$ is also PD. 
Both results are derived from the spectral decomposition of positive definite matrices.
To summarize, for a PD matrix $\bA$, we can obtain the following factorizations:
\begin{itemize}
\item $\bA=\bR^\top\bR$, where $\bR$ is an upper triangular matrix with positive diagonals, as established in Theorem~\ref{theorem:cholesky-factor-exist} via the Cholesky decomposition; 
\item $\bA = \bP^\top\bP$, where $\bP$ is nonsingular, as stated in  Theorem~\ref{lemma:nonsingular-factor-of-PD}; 
\item and $\bA = \bB^2$, where $\bB$ is PD, as given in Theorem~\ref{theorem:unique-factor-pd}. 
\end{itemize}
For a comprehensive overview, these factorizations of a positive definite matrix $\bA$ are summarized in Figure~\ref{fig:pd-summary}.

\begin{figure}[htbp]
\centering
\begin{widepage}
\centering
\resizebox{0.75\textwidth}{!}{%
\begin{tikzpicture}[>=latex]

\tikzstyle{state} = [draw, very thick, fill=white, rectangle, minimum height=3em, minimum width=6em, node distance=8em, font={\sffamily\bfseries}]
\tikzstyle{stateEdgePortion} = [black,thick];
\tikzstyle{stateEdge} = [stateEdgePortion,->];
\tikzstyle{stateEdge2} = [stateEdgePortion,<->];
\tikzstyle{edgeLabel} = [pos=0.5, text centered, font={\sffamily\small}];

\node[ellipse, name=pdmatrix, draw,font={\sffamily\bfseries},  node distance=7em, xshift=-9em, yshift=-1em,fill={colorals}]  {PD Matrix $\bA$};
\node[state, name=bsqure, below of=pdmatrix, xshift=0em, yshift=1em, fill={colorlu}] {$\bB^2$};
\node[state, name=lq, right of=bsqure, xshift=3em, fill={colorlu}] {$\bP^\top\bP$};
\node[state, name=rsqure, left of=bsqure, xshift=-3em, fill={colorlu}] {$\bR^\top\bR$};
\node[ellipse, name=utv, below of=pdmatrix,draw,  node distance=7em, xshift=0em, yshift=-4em,font={\tiny},fill={coloruppermiddle}]  {PD $\bB$};
\node[ellipse, name=upperr, left of=utv, draw, node distance=8em, xshift=-3em,font={\tiny},fill={coloruppermiddle}]  {\parbox{6em}{Upper \\Triangular $\bR$}};
\node[ellipse, name=nonp, right of=utv,draw,  node distance=8em, xshift=3em, font={\tiny},fill={coloruppermiddle}]  {Nonsingular $\bP$};

\coordinate (lq2inter3) at ($(pdmatrix.east -| lq.north) + (-0em,0em)$);
\draw (pdmatrix.east) edge[stateEdgePortion] (lq2inter3);
\draw (lq2inter3) edge[stateEdge] 
node[edgeLabel, text width=7.25em, yshift=0.8em]{\parbox{5em}{Spectral\\Decomposition}} (lq.north);

\coordinate (rqr2inter1) at ($(pdmatrix.west) + (0,0em)$);
\coordinate (rqr2inter3) at ($(rqr2inter1-| rsqure.north) + (-0em,0em)$);
\draw (rqr2inter1) edge[stateEdgePortion] (rqr2inter3);
\draw (rqr2inter3) edge[stateEdge] 
node[edgeLabel, text width=8em, yshift=0.8em]{\parbox{2em}{LU/\\Spectral/\\Recursive}} (rsqure.north);

\draw (pdmatrix.south)
edge[stateEdge] node[edgeLabel,yshift=0.5em]{\parbox{5em}{Spectral\\Decomposition} } 
(bsqure.north);

\draw (upperr.north)
edge[stateEdge] node[edgeLabel,yshift=0.5em]{} 
(rsqure.south);

\draw (utv.north)
edge[stateEdge] node[edgeLabel,yshift=0.5em]{} 
(bsqure.south);

\draw (nonp.north)
edge[stateEdge] node[edgeLabel,yshift=0.5em]{} 
(lq.south);

\end{tikzpicture}
}
\end{widepage}
\caption{Demonstration of different factorizations for a positive definite matrix $\bA$.}
\label{fig:pd-summary}
\end{figure}

\section{Decomposition for Semidefinite Matrices}
For positive semidefinite matrices, the Cholesky decomposition can still exist, though slight modifications are required.
\begin{theoremHigh}[Semidefinite decomposition\index{Positive semidefinite}]\label{theorem:semidefinite-factor-exist}
Every positive semidefinite matrix $\bA\in \real^{n\times n}$ can be factored as 
$$
\bA = \bR^\top\bR,
$$
where $\bR \in \real^{n\times n}$ is an upper triangular matrix. The diagonal elements of $\bR$ may be zero, and it is important to note that the factorization is generally \textbf{not unique}.
\end{theoremHigh}
In such a decomposition, the diagonal elements of $\bR$ may not accurately reflect the rank of $\bA$ \citep{higham2009cholesky}. 
\begin{example}[\citep{higham2009cholesky}]
Consider the matrix 
$$
\bA = 
\begin{bmatrix}
1 & -1 & 1 \\
-1 & 1 & -1 \\
1 & -1 & 2 
\end{bmatrix}.
$$
A semidefinite decomposition is given by 
$$
\bA = 
\begin{bmatrix}
1 & 0 & 0 \\
-1 & 0 & 0 \\
1 & 1 & 0 
\end{bmatrix}
\begin{bmatrix}
1 & -1 & 1 \\
0 & 0 & 1 \\
0 & 0 & 0
\end{bmatrix}
=\bR^\top\bR.
$$
The matrix $\bA$ has a rank of 2, while $\bR$  has only one nonzero diagonal element.
\end{example}

It is worth noting that all PD matrices are full-rank, a property that plays a key role in many earlier proofs. Sylvester's criterion (Theorem~\ref{lemma:sylvester-criterion}) confirms this by stating that all leading principal minors of PD matrices are positive.
Alternatively, one can argue that if a PD matrix $\bA$ were rank-deficient, the null space of $\bA$ would have a positive dimension. This would imply the existence of a vector $\bx$ in the null space such that $\bA\bx=\bzero$, contradicting the definition of positive definiteness.

However, this property does not hold for PSD matrices, whose null space may have a dimension greater than zero. To address this limitation, we introduce a more general, rank-revealing decomposition for semidefinite matrices.

\index{Rank-revealing}\index{Semidefinite rank-revealing}
\begin{theoremHigh}[Semidefinite rank-revealing decomposition\index{Rank-revealing}]\label{theorem:semidefinite-factor-rank-reveal}
Let $\bA\in \real^{n\times n}$ be a positive semidefinite matrix  of rank $r$. Then, it can be factored as 
$$
\bP^\top \bA\bP  = \bR^\top\bR, \qquad \mathrm{with} \qquad 
\bR = \begin{bmatrix}
\bR_{11} & \bR_{12}\\
\bzero &\bzero 
\end{bmatrix} \in \real^{n\times n},
$$
where $\bR_{11} \in \real^{r\times r}$ is an upper triangular matrix with positive diagonal elements, and $\bR_{12}\in \real^{r\times (n-r)}$. 
\end{theoremHigh}
\begin{proof}[of Theorem~\ref{theorem:semidefinite-factor-rank-reveal}]
The proof is constructive and follows a similar approach to the second perspective used for constructing the Cholesky decomposition, as described in \textcolor{black}{Section~\ref{section:recursi_choles}}.
The algorithm begins with  $ \bA^{(1)} = \bA  $ and generates a sequence of matrices defined as
$$
\bA^{(k)} = [a_{ij}^{(k)}] = 
\begin{bmatrixfoot}
\bI_{k-1} & \bzero \\
\bzero & \bB^{(k)}
\end{bmatrixfoot}, \quad k = 1, 2, \ldots,\quad\text{with } \bB^{(k)}\in\real^{(n-k+1)\times (n-k+1)}.
$$
At the beginning  of step $ k $, we select the largest diagonal element of $ \bA^{(k)} $,
$$
s_q^{(k)} = \max_{k \leq i \leq n} a_{ii}^{(k)},
$$
and interchange rows and columns $q$ and $k$ to bring this into pivot position; that is, $s_q^{(k)}$ appears in the ($k,k$) position of $\bP^{(k)\top}\bA^{(k)}\bP^{(k)}$, where the permutation matrix $\bP^{(k)}$ has the form 
$$
\bP^{(k)}= \begin{bmatrixfoot}
\bI_{k-1} & \bzero \\
\bzero & \widetildebP^{(k)}
\end{bmatrixfoot},
$$ 
and $\widetildebP^{(k)}\in\real^{(n-k+1)\times (n-k+1)}$ is a smaller permutation matrix. 
This pivot must be positive for $k < r$, because otherwise $\bB^{(k)} = \bzero$, which implies that $\rank(\bA) < r$. Next, the elements in the permuted $\bA^{(k)}$ are transformed according to the Cholesky Algorithm~\ref{alg:compute-choklesky11}:
$$
\begin{aligned}
r_{kk} &= \sqrt{a_{kk}^{(k)}}, \quad r_{kj} = a_{kj}^{(k)} / r_{kk}, \quad j = k + 1:n,\\
a_{ij}^{(k+1)} &= a_{ij}^{(k)} - r_{ki} r_{kj}, \quad i,j = k + 1:n.
\end{aligned}
$$
This process is equivalent to subtracting a symmetric rank-one matrix $\br_j \br_j^\top$ from $\bA^{(k)}$, where $\br_j = \be_j^\top \bR$ is the $j$-th row of $\bR$. The algorithm stops when $k = r + 1$. Then all the remaining diagonal elements are zero, which implies that $\bA^{(r+1)} = \begin{bmatrixscript}
\bI_r & \bzero \\
\bzero & \bzero 
\end{bmatrixscript}$.

\paragraph{Construction algorithm.} Below contains more constructive analysis.
Following the  second perspective for computing the Cholesky decomposition in \textcolor{black}{Section~\ref{section:recursi_choles}}, 
we can construct 
$$
\bP^{(k)\top}\bA^{(k)} \bP^{(k)}=
\begin{bmatrixfoot} 
\bI_{k-1} & 0 & \bzero \\ 
0 & a_{kk}^{(k)} & \bb_k^\top \\ 
\bzero & \bb_k & \bB^{(k)} 
\end{bmatrixfoot} 
\qquad\text{and}\qquad
\bL^{(k)} = 
\begin{bmatrixfoot} 
\bI_{k-1} & 0 & \bzero \\ 
0 & \sqrt{a_{kk}^{(k)}} & \bzero \\ 
\bzero & \frac{1}{\sqrt{a_{kk}^{(k)}}}\bb_k & \bI_{n-k} 
\end{bmatrixfoot},
$$
satisfying  $\bP^{(k)\top}\bA^{(k)}\bP^{(k)} = \bL^{(k)}\bA^{(k+1)}(\bL^{(k)})^\top$:
$$
\begin{aligned} \bA^{(k+1)} 
&= \begin{bmatrixfoot} 
\bI_{k-1} & 0 & \bzero \\ 
0 & 1 & \bzero \\ 
\bzero & \bzero & \bB^{(k)}-\frac{1}{a_{kk}^{(k)}}\bb_k\bb_k^\top 
\end{bmatrixfoot} 
= 
\begin{bmatrixfoot} 
\bI_k & 0 & \bzero \\ 
0 & a_{k+1, k+1}^{(k+1)} & \bb_{k+1}^\top \\ 
\bzero & \bb_{k+1} & \bB^{(k+1)} 
\end{bmatrixfoot}.
\end{aligned}
$$

However, we notice that these permutation matrices $\bP^{(1)}, \bP^{(2)}, \ldots, \bP^{(r)}$ are used to permute two columns; therefore, they are symmetric satisfying $\bP^{(k)}\cdot \bP^{(k)} = \bI$ for all $k$.
Let $\bP= \bP^{(1)}\bP^{(2)}\ldots\bP^{(r)}$. 
Since $(\bA^{(r+1)})^2=\bA^{(r+1)}$, $\bA^{(1)}=\bA$ can be expressed as 
\begin{align}
&\bP^\top\bA^{(1)}\bP = \bL\bL^\top;\\
&\bL= \left\{\bP^{(r)}\bP^{(r-1)}\ldots \bP^{(2)} \bP^{(1)}\right\}
\left\{\bP^{(1)}\bL^{(1)} \right\} \left\{\bP^{(2)}\bL^{(2)} \right\}
\ldots \left\{\bP^{(r)}\bL^{(r)} \right\} \bA^{(r+1)}. \label{equation:rrcho_low1}
\end{align}
To complete the proof, it suffice to show that $\bL$ is lower triangular with the rank-revealing property.
On the other hand, each lower triangular $\bL^{(k)}$ can be written as 
$$
\bL^{(k)} = 
\bI - \bl_k \be_k^\top
\quad \text{with}\quad \bl_k = [\bzero_{k-1}, l_{k}, l_{k+1}, \ldots, l_n]^\top,
$$
where $\be_k$ is the $k$-th standard unit basis, and $\bl_k$ is a vector containing $k-1$ zeros. Note that $1-l_k \equiv \sqrt{a_{kk}^{(k)}}$ in this notation.
For $k\in\{1,2,\ldots,r-1\}$, define 
$$
\begin{aligned}
\bM_k&= \bP^{(r)}\bP^{(r-1)}\ldots \bP^{(k+1)}\bL^{(k)}\bP^{(k+1)}\ldots \bP^{(r-1)}\bP^{(r)}\\
&=\bP^{(r)}\bP^{(r-1)}\ldots \bP^{(k+1)} (\bI - \bl_k\be_k^\top)\bP^{(k+1)}\ldots \bP^{(r-1)}\bP^{(r)}\\
&=\bI - (\bP^{(r)}\bP^{(r-1)}\ldots \bP^{(k+1)} \bl_k) (\be_k^\top\bP^{(k+1)}\ldots \bP^{(r-1)}\bP^{(r)})\\
&=\bI - (\bP^{(r)}\bP^{(r-1)}\ldots \bP^{(k+1)} \bl_k) \be_k^\top,
\end{aligned}
$$
where the last equality follows since  $\be_k^\top\bP^{(k+1)}\ldots \bP^{(r-1)}\bP^{(r)}=\be_k^\top$.
This implies $\bM_k$ is lower triangular with its $k$-th column representing a permuted version of $\bL^{(k)}$.
Therefore, it holds that 
$$
\begin{aligned}
&\bM_1 \bM_2\ldots \bM_{r-1} = 
\left\{\bP^{(r)}\bP^{(r-1)}\ldots \bP^{(2)}\right\} \left\{\bL^{(1)} \bP^{(2)}\right\} 
\left\{\bL^{(2)}\bP^{(3)}\right\} \ldots \left\{\bL^{(r-1)}\bP^{(r)}\right\};\\
&\bL \equiv \bM_1 \bM_2\ldots \bM_{r-1} \bL^{(r)}\bA^{(r+1)}.
\end{aligned}
$$
From the above analysis, $\bM_1 \bM_2\ldots \bM_{r-1}$ is lower triangular, and $\bL^{(r)}\bA^{(r+1)}$ has the form
$$
\bL^{(r)}\bA^{(r+1)}
=
\begin{bmatrixfoot}
\bM_{11} &\bzero \\
\bM_{21} &\bzero
\end{bmatrixfoot},
\quad \text{with lower triangular $\bM_{11}$}. 
$$
Therefore, $\bL = \bM_1 \bM_2\ldots \bM_{r-1} \bL^{(r)}\bA^{(r+1)}$ has the desired form
$$
\bL=
\begin{bmatrixfoot}
\bL_{11} &\bzero \\
\bL_{21} &\bzero
\end{bmatrixfoot},
\quad 
\text{with lower triangular $\bL_{11}$}. 
$$
This completes the proof.
\end{proof}

A more compact proof of this rank-revealing decomposition for semidefinite matrices will be presented in Section~\ref{section:semi-rank-reveal-proof}, relying on the spectral decomposition (Theorem~\ref{theorem:spectral_theorem}) and the column-pivoted QR decomposition (Theorem~\ref{theorem:rank-revealing-qr-general}).
Whereas, the  proof for the trivial semidefinite decomposition Theorem~\ref{theorem:semidefinite-factor-exist} can be derived directly from the spectral decomposition and the standard QR decomposition (Theorem~\ref{theorem:qr-decomposition}).

When the matrix $\bA$ is symmetric and indefinite, we can employ  a \textit{symmetric indefinite  decomposition} or \textit{Bunch--Kaufman decomposition} \citep{bunch1977some}. 
\begin{theoremHigh}[Bunch--Kaufman decomposition]\label{theorem:Bunch_Kaufman}
Let $\bA\in \real^{n\times n}$ be a symmetric (indefinite) matrix. Then, it can be factored as 
$$
\bP^\top \bA\bP  = \bL\bB\bL^\top, 
$$
where $\bP$ is a permutation matrix, $\bL$ is a unit lower triangular matrix, and $\bB$ is a block-diagonal matrix with each diagonal block of $\bB$ being either a  $1\times 1$ or a  $2\times 2$ matrix.
\end{theoremHigh}
This type of decomposition is sometimes referred to as an \textit{$\bL\bB\bL^\top$ decomposition}.
It is particularly useful in practical applications, such as solving linear systems and computing  eigenvalues of matrices, especially in cases where a direct Cholesky decomposition cannot be applied (e.g., when the matrix is not positive definite) \citep{dumas2018symmetric}.

\index{Least squares}
\section{Application: Rank-One and Rank-Two Update/Downdate}\label{section:cholesky-rank-one-update}

Updating linear systems after low-rank modifications of the system matrix is a common procedure in fields such as machine learning, statistics, and more \citep{lu2021rigorous}.
For example, when computing the least squares solution using Cholesky decomposition (see Section~\ref{section:application-ls-qr}),  we may want to add or remove one or more data points from the data matrix $\bA$ and the observed data vector $\by$ (that is, to add or delete a row in both $\bA$ and $\by$) in order to analyze the performance of the updated system.
However, it is well known that such updates can become numerically unstable in the presence of round-off errors \citep{seeger2004low}. If the system matrix is positive definite, a more numerically stable approach involves using a representation based on the Cholesky decomposition. In this section, we will provide a proof of the rank-one update/downdate using Cholesky decomposition.

\subsection{Rank-One Update}\index{Rank-one update}
A rank-one update $\bA^\prime$ of a matrix $\bA \in \real^{n\times n}$ by a vector $\bv$ is defined as follows:
\begin{equation*}
\begin{aligned}
\bA^\prime &= \bA + \bv \bv^\top;\\
\bR^{\prime\top}\bR^\prime &= \bR^\top\bR + \bv \bv^\top.
\end{aligned}
\end{equation*}
If we have already calculated the Cholesky factor $\bR$ of $\bA$, then the Cholesky factor $\bR^\prime$ of $\bA^\prime$ can be calculated efficiently.  
This avoids recomputing the decomposition from scratch, reducing the computational cost from $\mathcalO(n^3)$ to $\mathcalO(n^2)$.
Specifically,  $\bR^\prime$ is obtained via a \textit{rank-one Cholesky update}, leveraging the fact that  $\bA^\prime$ differs from $\bA$ only by a symmetric rank-one matrix.
To derive $\bR^\prime$, 
consider a set of orthogonal matrices $\bQ_n \bQ_{n-1}\ldots \bQ_1$ such that:
$$
\bQ_n \bQ_{n-1}\ldots \bQ_1 
\begin{bmatrix}
\bv^\top \\
\bR
\end{bmatrix}
=
\begin{bmatrix}
\bzero \\
\bR^\prime
\end{bmatrix}.
$$
The Cholesky factor $\bR^\prime$ can be determined by analyzing the above transformation.
Specifically, the left-hand side  of the equation, when multiplied by its transpose, yields:
$$
\begin{bmatrix}
\bv & \bR^\top
\end{bmatrix}
\bQ_1^\top \ldots \bQ_{n-1}^\top\bQ_n^\top
\bQ_n \bQ_{n-1}\ldots \bQ_1 
\begin{bmatrix}
\bv^\top \\
\bR
\end{bmatrix}
= \bR^\top\bR + \bv \bv^\top.
$$ 
Similarly, the right-hand side, when multiplied by its transpose, results in:
$$
\begin{bmatrix}
\bzero & \bR^{\prime\top}
\end{bmatrix}
\begin{bmatrix}
\bzero \\
\bR^\prime
\end{bmatrix}=\bR^{\prime\top}\bR^\prime,
$$
which agrees with the left-hand side equation. \textit{Givens rotations} are such orthogonal matrices that can transfer $\bR$ and $\bv$ into $\bR^\prime$. 

\begin{definition}[$n$-th Order Givens rotation]\label{definition:givens-rotation-in-qr}
An $n$-th order \textit{Givens rotation} is a matrix $\bG_{kl}$ of the following form:
\begin{equation}\label{equa:givens_eq1}
\bG_{kl}= \bI + (c-1)(\bdelta_k\bdelta_k^\top + \bdelta_l\bdelta_l^\top) + s(\bdelta_k\bdelta_l^\top -\bdelta_l\bdelta_k^\top ),
\end{equation}
where $\bdelta_k \in \real^n$ is the  $k$-th standard unit basis. 
The subscripts $k$ and $l$ indicate that the \textbf{rotation occurs  in the plane defined by the $k$-th and $l$-th dimensions}.
In other words, we have
$$
\bG_{kl}=
\begin{bmatrixfoot}
1 &          &   &  &   &   & &  & &\\
& \ddots  &  &  &  & && & &\\
&      & 1 &  & & &  && &\\
&      &  & c &  &  &  & s & &\\
&& &   & 1 & & && &\\
&& &   &   &\ddots &  && &\\
&& &  &   &  & 1&& &\\
&& & -s &  &  & &c& &\\
&& & &  &  & & &1 & \\
&& & &  &  & & & &\ddots
\end{bmatrixfoot}_{n\times n},
$$
where the $(k,k), (k,l), (l,k), (l,l)$ entries are $c, s, -s, c$ respectively, and $s = \cos \theta$ and $c=\cos \theta$ for some angle $\theta$.
Using this angle, the Given rotation in \eqref{equa:givens_eq1} can be more precisely denoted as 
\begin{equation}\label{equa:givens_eq2}
\bG_{kl} = \bG_{kl}(\theta).
\end{equation}
Specifically, one can also define the $n$-th order Givens rotation, where $(k,k), (k,l), (l,k),$ and $ (l,l)$ entries are $c, \textcolor{mylightbluetext}{-s, s}, $ and $c$, respectively. The ideas are the same.
\end{definition}
Some fundamental significance of Givens rotations, crucial for proving the existence of the QR decomposition, will be discussed shortly in Section~\ref{section:qr-givens}. 

It can be easily verified that the $n$-th order Givens rotation is  orthogonal, and its determinant is 1. For any vector $\bx =[x_1, x_2, \ldots, x_n]^\top \in \real^n$, 
the effect of applying the Givens rotation matrix $\bG_{kl}$ to $\bx$ is given by:
$$ 
\left\{
\begin{aligned}
&y_k = c \cdot x_k + s\cdot x_l;   \\
&y_l = -s\cdot x_k +c\cdot x_l;  \\
&y_j = x_j , &  (j\neq k,l) 
\end{aligned}
\right.
$$
In other words, a Givens rotation applied to $\bx$ rotates the components $x_k$ and $x_l$ of $\bx$ by an angle $\theta$, while leaving all other components unchanged.

Now suppose we have an $(n+1)$-th order Givens rotation indexed from $0$ to $n$:
$$
\bG_k = \bI + (c_k-1)(\bdelta_0\bdelta_0^\top + \bdelta_k\bdelta_k^\top) + s_k(\bdelta_0\bdelta_k^\top -\bdelta_k\bdelta_0^\top ),
$$
where $c_k = \cos \theta_k, s_k=\sin\theta_k$ for some angle $\theta_k$, $\bG_k \in \real^{(n+1)\times (n+1)}$, and $\bdelta_k\in \real^{n+1}$ is a zero vector except that its $(k+1)$-th entry is 1.

Taking out the $k$-th column of the following transformation 
$$
\begin{bmatrix}
\bv^\top \\
\bR
\end{bmatrix}
\rightarrow 
\begin{bmatrix}
\bzero \\
\bR^\prime
\end{bmatrix}.
$$
Let the $k$-th element of $\bv$ be $v_k$, and the $k$-th diagonal of $\bR$ be $r_{kk}$.
Since $\sqrt{v_k^2 + r_{kk}^2} \neq 0$, we can define $c_k = \frac{r_{kk}}{\sqrt{v_k^2 + r_{kk}^2}}$, $s_k=-\frac{v_k}{\sqrt{v_k^2 + r_{kk}^2}}$. Then, 
$$ 
\left\{
\begin{aligned}
&v_k \rightarrow c_kv_k+s_kr_{kk}=0;   \\
&r_{kk}\rightarrow -s_k v_k +c_kr_{kk}= \sqrt{v_k^2 + r_{kk}^2} = r^\prime_{kk} .  \\
\end{aligned}
\right.
$$
In other words, the Givens rotation $\bG_k$ will set  the $k$-th element of $\bv$ to zero and assign a nonzero value to $r_{kk}$.
This result is essential for performing a rank-one update.
A sequence of Givens rotations $\bG_n \bG_{n-1}\ldots \bG_1 $ transforms the augmented matrix as follows:
$$
\bG_n \bG_{n-1}\ldots \bG_1 
\begin{bmatrix}
\bv^\top \\
\bR
\end{bmatrix}
=
\begin{bmatrix}
\bzero \\
\bR^\prime
\end{bmatrix}.
$$
Each  rotation requires $6n$  floating-point operations (flops), yielding a total computational cost of $6n^2$ flops for  $n$ such rotations. 
This approach significantly reduces the complexity of calculating the Cholesky factor of $\bA^\prime$ from $\frac{1}{3} n^3$ to $6n^2$ flops, assuming the Cholesky factor of $\bA$ is already known \citep{lu2021numerical}. 
The algorithm is particularly useful in reducing the computational complexity of posterior calculations in Bayesian inference for \textit{Gaussian mixture models} \citep{lu2021bayes}.
At each stage, $k$ new samples are added or removed from an existing cluster, which corresponds to performing $k$ rank-one updates.

\index{Gaussian mixture models}
\subsection{Rank-One Downdate}
Now suppose that the Cholesky factor of $\bA$ has been computed, and  $\bA^\prime$ is a \textit{rank-one downdate} of $\bA$, defined as:
\begin{equation*}
\begin{aligned}
\bA^\prime &= \bA - \bv \bv^\top;\\
\bR^{\prime\top}\bR^\prime &= \bR^\top\bR - \bv \bv^\top. 
\end{aligned}
\end{equation*}
The algorithm for performing such a downdate follows a similar procedure:
\begin{equation}\label{equation:rank-one-downdate}
\bG_1 \bG_{2}\ldots \bG_n
\begin{bmatrix}
\bzero \\
\bR
\end{bmatrix}
=
\begin{bmatrix}
\bv^\top \\
\bR^\prime
\end{bmatrix}.
\end{equation}
Once again, each transformation, $
\bG_k = \bI + (c_k-1)(\bdelta_0\bdelta_0^\top + \bdelta_k\bdelta_k^\top) + s_k(\bdelta_0\bdelta_k^\top -\bdelta_k\bdelta_0^\top ),
$ can be constructed in the following way.
Taking out the $k$-th column of the following equation 
$$
\begin{bmatrix}
\bzero \\
\bR
\end{bmatrix}
\rightarrow 
\begin{bmatrix}
\bv^\top \\
\bR^\prime
\end{bmatrix}.
$$
We realize that $r_{kk} \neq 0$, and let $c_k=\frac{\sqrt{r_{kk}^2 - v_k^2}}{r_{kk}}$, $s_k = \frac{v_k}{r_{kk}}$. Then, 
$$ 
\left\{
\begin{aligned}
& 0 \rightarrow s_kr_{kk}=v_k;   \\
&r_{kk}\rightarrow c_k r_{kk}= \sqrt{r_{kk}^2-v_k^2  }=r^\prime_{kk} .  \\
\end{aligned}
\right.
$$
To ensure that $\bA^\prime$ remains positive definite, it is necessary that $r^2_{kk} > v_k^2$. 
If this condition is not satisfied, then $c_k$, as defined above, will not be real-valued, and the update cannot proceed.
As a verification step, one can check that  multiplying the left-hand side of \eqref{equation:rank-one-downdate} by its transpose yields:
$$
\begin{bmatrix}
\bzero & \bR^\top
\end{bmatrix}
\bG_n^\top \ldots \bG_{2}^\top\bG_1^\top
\bG_1 \bG_{2}\ldots \bG_n
\begin{bmatrix}
\bzero \\
\bR
\end{bmatrix} =\bR^\top\bR.
$$
Similarly, multiplying the right-hand side by its transpose gives:
$$
\begin{bmatrix}
\bv & \bR^{\prime\top} 
\end{bmatrix}
\begin{bmatrix}
\bv^\top \\
\bR^\prime
\end{bmatrix}=\bv\bv^\top + \bR^{\prime\top}\bR^\prime.
$$
This confirms that $\bR^{\prime\top}\bR^\prime = \bR^\top\bR - \bv \bv^\top$.

\section{Application: Indefinite Rank-Two Update}\index{Rank-two update}

Let $\bA = \bR^\top\bR$ be the Cholesky decomposition of $\bA$. \citet{goldfarb1976factorized, seeger2004low} introduced a stable method for performing  an indefinite rank-two update of the form 
$$
\bA^\prime = (\bI+\bv\bu^\top)\bA(\bI+\bu\bv^\top).
$$
Let
$$
\bigg\{ 
\begin{aligned}
\bz &= \bR^{-\top}\bv,   \\
\bw &= \bR\bu,
\end{aligned}
\qquad 
\implies
\qquad  
\bigg\{
\begin{aligned}
\bv &= \bR^{\top}\bz,   \\
\bu &= \bR^{-1}\bw.
\end{aligned}
$$
Now suppose that the LQ decomposition \footnote{This will be introduced in Theorem~\ref{theorem:lq-decomposition}.} of $\bI+\bz\bw^\top$ is given by $\bI+\bz\bw^\top =\bL\bQ$, where $\bL$ is lower triangular and $\bQ$ is orthogonal. Then, we can express $\bA^\prime$ as
$$
\begin{aligned}
\bA^\prime &= (\bI+\bv\bu^\top)\bA(\bI+\bu\bv^\top)
= (\bI+\bR^{\top}\bz   \bw^\top \bR^{-\top })\bA(\bI+\bR^{-1}\bw \bz^\top\bR)\\
&= \bR^\top (\bI+\bz\bw^\top)(\bI+\bw\bz^\top)\bR
= \bR^\top\bL \bQ\bQ^\top \bL^\top\bR 
= \bR^\top\bL \bL^\top\bR.
\end{aligned}
$$
Finally, let $\bR^\prime = \bR^\top\bL$, which is a lower triangular matrix. This establishes the Cholesky decomposition of $\bA^\prime$.

\index{Newton's method}
\index{Modified Newton's method}
\section{Application: Modified Newton's Method and Nearest Correlation}\label{section:app_cho_md_newton}
When optimizing or minimizing a function $f(\bx)$ over $\bx$, 
the standard Newton's method~\footnote{See, for example, \citet{lu2025practical}.} updates the estimate at the  $t$-th iteration as
$$
\bx^{(t+1)} \leftarrow \bx^{(t)} + \bd^{(t)},
$$
where $(\nabla^2 f(\bx^{(t)}) )\bd^{(t)} = -\nabla f(\bx^{(t)})$ determines the ``candidate" descent direction $\bd^{(t)}$.
The vector $\bd^{(t)}$ is  a descent direction only when the Hessian $(\nabla^2 f(\bx^{(t)}) )$ is PD, which is not always the case.

The modified Newton's method addresses this issue by approximating the Hessian with $\bH^{(t)} = \nabla^2 f(\bx^{(t)}) +\bE^{(t)}$, ensuring that $\bH^{(t)}$ is PD \citep{gill2019practical, lu2025practical}.
Given the Cholesky decomposition in the form $\nabla^2 f(\bx^{(t)})=\bL\bD\bL^\top=\bR^\top\bR$ (where $\bR=\bD^{1/2}\bL^\top$) and the condition number inequality $\cond(\nabla^2 f(\bx^{(t)}))\geq \cond(\bD)$ (see {Equation~\eqref{equation:cond_pd_ineq}}), the goal of the modified Newton's method can be approximately achieved by adjusting the diagonals of $\bD$.
To be more specific, when computing the Cholesky decomposition using Algorithm~\ref{alg:compute-choklesky-_ldl}, the modified Newton's method imposes bounds on the diagonal $d_{jj}$, given two parameters $\alpha$ and $\beta$, such that 
$$
d_{jj} \geq \alpha,  \gap  l_{ij}\sqrt{d_{jj}}\leq \beta, \,i=\{j+1,j+2,\ldots,n\}.
$$
The latter constraint serves  to upper-bound each row of $\bR$, since $\bR=\bD^{1/2}\bL^\top$. And this is equivalent to updating each $d_{jj}$ in {Algorithm~\ref{alg:compute-choklesky-_ldl}} by
$$
d_{jj} \leftarrow \max\left\{ \abs{c_{jj}},\, \beta,\, \mathop{\max}_{i>j}\abs{c_{ij}}\right\}.
$$

\index{Nearest correlation matrix problem}
\paragraph{Nearest correlation matrix problem.}
The modified Cholesky decomposition discussed above can also be applied to the \textit{nearest correlation matrix (NCM)} problem.
In statistical modeling, a correlation matrix is often used to represent the correlation coefficients between a set of two or more random variables. The $(i,j)$-th entry of such a matrix represents the correlation coefficient between the variables $\bx_i$ and $\bx_j$. Clearly, such a matrix must be symmetric, have ones along the diagonal, and be positive semidefinite.

In many practical applications, however, a matrix that is intended to represent correlations between variables may fail to be a valid correlation matrix---most commonly because it is not positive semidefinite. There are several reasons this might occur, but it is typically due to missing data being estimated or matrix entries being altered, either intentionally or out of necessity. One specific example where this issue arises is in financial stress testing, which often involves modifying the elements of a matrix that represents the correlations among various stocks \citep{higham2002computing, higham2016bounds, mcsweeney2017modified}.
In such cases, we often seek to find the nearest correlation matrix to the given one, which can then serve as the ``true" matrix for further computations. This problem has long been of interest, especially in the finance industry.

\section{Application: Obtain Orthonormal Basis}\label{section:cho_orthonbasis}

The Cholesky decomposition can be used to orthonormalize a basis set in an $n$-dimensional vector space.
Let $\bS\in\real^{n\times n}$ be a full-rank matrix with Cholesky decomposition $\bS^\top\bS=\bL\bL^\top$, where $\bL$ is a lower triangular matrix.
Consider the transformation $\bQ = \bS (\bL^{-1})^\top$. We verify that $\bQ$ is orthogonal by computing:
\begin{equation}
\bQ^\top \bQ = \bL^{-1} \bS^\top \bS (\bL^{-1})^\top = \bL^{-1} \bL \bL^\top (\bL^{-1})^\top = \bL^{-1} \bL (\bL^{-1} \bL)^\top = \bI.
\end{equation}
To implement this transformation in a program, we can transpose both sides:
$
\bQ^\top = \bL^{-1} \bS^\top.
$
Denoting the $i$-th row vectors of $\bQ$ and $\bS$ as $\bq_i$ and $\bs_i$, respectively, we obtain:
$$
\bq_i = \bL^{-1} \bs_i  
\quad \implies \quad 
\bL \bq_i = \bs_i,\quad i = 1, 2,\ldots, n.
$$
Since $\bL$ is lower triangular, each system can be efficiently solved using forward substitution. For simplicity, dropping the index $i$, consider the system $
\bL\bq=\bs$. The solution via forward substitution is given by the following recursion:
$$
q_1 = \frac{s_1}{l_{11}},
\qquad 
q_i = \frac{1}{l_{ii}} \left( s_i - \sum_{j=1}^{i-1} l_{ij} q_j \right), \quad i = 2,3, \ldots, n.
$$

\begin{figure}[h]
\centering                       
\vspace{-0.35cm}               
\subfigtopskip=2pt              
\subfigbottomskip=2pt           
\subfigcapskip=-15pt            
\includegraphics[width=0.9\textwidth]{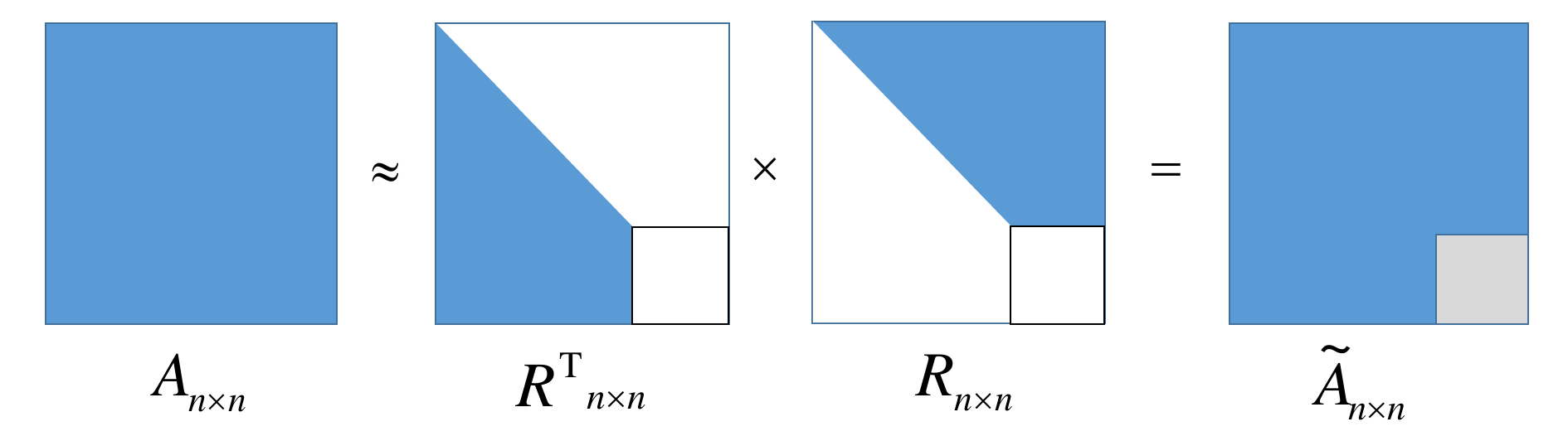}
\caption{
	Demonstration of low-rank approximation using pivoted Cholesky decomposition. White entries represent zeros, while \textcolor{darkgray}{gray} entries represent the approximated values.}
\label{fig:cholesky_lowrank}
\end{figure}
\index{Low-rank approximation}
\section{Application: Low-Rank Approximation}\label{section:cho_lowrank}
We will discuss low-rank approximation in more detail in Section~\ref{section:svd-low-rank-approxi} and Chapter~\ref{chapter:als}.
The Cholesky decomposition of a positive definite matrix can also be used to construct a low-rank approximation of that matrix.
Let $\bA=\bR^\top\bR$ be the Cholesky decomposition of a PD matrix $\bA\in\real^{n\times n}$.
We then observe that the $(i,j)$-th entry of $\bA$ is given by
\begin{equation}\label{equation:cho_elem_eq2_var}
a_{ij} = \sum_{k=1}^{\min{(i,j)}} r_{ki} r_{kj} , \quad \forall\, i,j.
\end{equation}
A low-rank approximation of $\bA$ can be obtained by truncating summation over $k$ in \eqref{equation:cho_elem_eq2_var} at some $k \leq r \ll n$. This is achieved by applying row and column permutations during each iteration of the Cholesky decomposition so that the largest diagonal element appears at the top of the currently considered submatrix \citep{golub2013matrix}. 
This approach is equivalent to performing a complete pivoting strategy (see Section~\ref{section:complete-pivoting}), since in a PD matrix, the largest absolute entry always lies on the diagonal (see Problem~\ref{prob:large_pd_diag}).

This method is implemented in the pivoted Cholesky algorithm described in Algorithm~\ref{alg:low_piv_cho}, where the iterations are truncated once the largest remaining diagonal element falls below a specified threshold $\delta$.
At the termination of the algorithm, $r$ represents the rank of the approximation, and the resulting rank-$r$ approximation of $\bA$ is given by
\begin{equation}
a_{ij} \approx \sum_{k=1}^{\min{(i,j,r)}} r_{ki} r_{kj}, \quad \forall\, i,j.
\end{equation}
This process is illustrated in Figure~\ref{fig:cholesky_lowrank}.

\begin{algorithm}[H]
\caption{Low-Rank Approximation via Pivoted Cholesky decomposition}
\label{alg:low_piv_cho}
\begin{algorithmic}[1]
\Require 
Positive definite matrix $\bA$ with size $n\times n$; 
\For{$j = 1$ to $n$} \Comment{Compute the $j$-th column of $\bR$}
\State $v = \argmax_{k \in \{j,\ldots,n\}} a_{kk}$;
\If{$a_{vv} < \delta$}
\State $r = j - 1$;
\State \textbf{break}
\EndIf
\State $\bA_{j,:} \leftrightarrow \bA_{v,:}$; \Comment{Swap $j$-th and $v$-th rows}
\State $\bA_{:,j} \leftrightarrow \bA_{:,v}$; \Comment{Swap $j$-th and $v$-th columns}
\For{$i = 1$ to $j-1$}
\State $r_{ij} = ( a_{ij} - \sum_{k=1}^{i-1} r_{ki} r_{kj} )/r_{ii}$, since $i<j$;
\EndFor
\State $r_{jj} \leftarrow \sqrt{a_{jj} - \sum_{k=1}^{j-1} r_{kj}^2}$;
\EndFor
\State Output $\bA=\bR^\top\bR$ and rank $r$.
\end{algorithmic}
\end{algorithm}


\begin{problemset}
	
\item \label{prob:large_pd_diag} Show that the largest element in a positive definite matrix lies on the diagonal. And a similar argument applies to positive semidefinite matrices.

\item Suppose that $\bA_1$ and $\bA_2$ are $n \times n$ positive semidefinite matrices of ranks $k_1$ and $k_2$, respectively, where $k_2 > k_1$. Prove that $\bA_1 - \bA_2$ cannot be positive semidefinite.

\item \label{prob:pd_cmequiv} \textbf{PD.} Let $\bA\in\real^{n\times n}$. Show that the following two statements are equivalent:
\begin{itemize}
	\item $\bx^*\bA\bx>0$ for all nonzero $\bx\in\complex^n$.
	\item $\bx^\top\bA\bx>0$ for all nonzero $\bx\in\real^n$.
\end{itemize}

\item \label{prob:tr_de_pd} \textbf{Trace, det of PD/PSD/ND matrices.} Let $\bA$ be positive definite (resp., positive semidefinite), show that $\trace(\bA), \det(\bA)$, and the principal minors of $\bA$ are all positive (resp., nonnegative). Moreover, $\trace(\bA)=0$ if and only if $\bA=\bzero$.
Let $\bB\in\real^{n\times n}$ be negative definite. Show that $\trace(\bB)$ is negative; $\det(\bB)$ is negative for odd $n$ and positive for even $n$.

\item Show that the following matrix is positive definite and compute  its Cholesky decomposition:
$
\bA = 
\begin{bmatrixscript}
	5 & -1 & 0 \\
	-1 & 4 & 2 \\
	0 & 2 & 8
\end{bmatrixscript}.
$

\item 
Given two positive semidefinite matrices $\bA,\bB\in\real^{n\times n}$, show that $\bA+\bB$ is also positive semidefinite.
\item 
Given two symmetric matrices $\bA\in\real^{n\times n}$ and $\bB\in\real^{m\times m}$. Prove that the following two claims are equivalent:
\begin{enumerate}
\item $\bA$ and $\bB$ are positive semidefinite.
\item $\scriptsize\begin{bmatrix}
\bA& \bzero \\
\bzero & \bB
\end{bmatrix}$ is positive semidefinite.
\end{enumerate}
\item Let $\bB\in\real^{n\times k}$ and  $\bA=\bB\bB^\top$. Show that 
$\bA$ is positive semidefinite; and $\bA$ is positive definite if and only if $\bB$ has full row rank.

\item Show that if $\bA$ is positive semidefinite, then $\bA^{-1}$ is positive definite (if exists).

\item  Prove that any positive definite matrix $\bA$ is nonsingular. \textit{Hint: Consider $\bA\bx=\bzero$ and analyze $\bx^\top\bA\bx=0$.}

\item \label{prob:sing_psd}
Let $\bA$ be positive semidefinite. Show that $\bx^\top\bA\bx=0$ if and only if $\bA\bx=\bzero$.
Furthermore, prove that a positive semidefinite $\bA$ is positive definite if and only if it is nonsingular.


\item \textbf{Quadratic form.} Consider the quadratic form $L(\bx) = \frac{1}{2} \bx^\top \bA \bx - \bb^\top \bx + c$, where $\bA\in\real^{d\times d}$, $ \bx\in \real^d$, and $c\in\real$. Suppose $\bA$ is positive semidefinite. 
Show that $L(\bx)$ is bounded below over $\real^d$ if and only if $\bb$ is in the column space of $\bA$.

\item \textbf{Quadratic form.} Consider the quadratic form $L(\bx) = \frac{1}{2} \bx^\top \bA \bx - \bb^\top \bx + c$. Show that $L(\bx)$ is \textit{coercive} if and only if $\bA$ is PD (A function $f(\bx):\real^n\rightarrow \real$ is called coercive if $\mathop{\lim}_{\norm{\bx}\rightarrow \infty} f(\bx)=\infty$.).

\item \label{prob:sam_quad} \textbf{Quadratic form.} Let $\bA\in\real^{n\times n}$ be a general square matrix (not necessarily symmetric). Show that $\bx^\top\bA\bx=\bx^\top[\frac{1}{2}(\bA+\bA^\top)]\bx$. The latter quadratic form is induced from a symmetric matrix.

\item \textbf{Symmetric form.} Define $P(\bA)=\frac{1}{2}(\bA+\bA^\top)$ for $\bA\in\real^{n\times n}$. Show that 
\begin{itemize}
	\item \textit{Null space.} $\nspace(\bA)\subset \nspace(P(\bA))$ and $\nspace(\bA^\top)\subset \nspace(P(\bA))$ such that $\rank(P(\bA))\leq \rank(\bA)$.
	\item When $\rank(P(\bA))= \rank(\bA)$, then $\bA$, $\bA^\top$, and $P(\bA)$ have the same null space.
\end{itemize}
\textit{Hint: Consider the quadratic form $\bx^\top\bA\bx$ and $\bx^\top\bA^\top\bx$, and  use Problem~\ref{prob:sing_psd}.}


\item  Let $\bA, \bB\in\real^{n\times n}$. Show that the matrix $\bA\bB - \bB\bA$ can never be positive
semidefinite unless it is the zero matrix. \textit{Hint: Use the fact that the trace of a symmetric matrix is equal to the sum of its eigenvalues.}

\item \label{prob:givens_rot1} \textbf{Givens rotation and rotary embedding.} 
Let $\bG(\theta) = \begin{bmatrixfoot}
	\cos(\theta) & \sin(\theta) \\
	-\sin(\theta) & \cos(\theta)
\end{bmatrixfoot}$ be a Givens rotation matrix, 
and let $\bv(\theta) = \begin{bmatrixfoot}
	\sin(\theta) \\ 
	\cos(\theta)
\end{bmatrixfoot}$.
Show that $\bG(\theta \delta) \bv(\theta t) = \bv(\theta (t+\delta))$, where $t$ can represent time or position.

\item \textbf{Givens rotation and rotary embedding.}  Using the notation from Problem~\ref{prob:givens_rot1}, plot the dot product $\bv(\theta t)^\top \bv(\theta (t+\delta))$ as a function of  $\delta$. What do you observe?

\item \textbf{Givens rotation and rotary embedding.}   Using the notation from Problem~\ref{prob:givens_rot1}, plot the dot product $\bv(\theta t)^\top \bW\bv(\theta (t+\delta))$ as a function of  $\delta$, where $\bW$ is an appropriately chosen random matrix. What behavior do you observe?

\item What is the difference between a Givens rotation  with entries $(k,k), (k,l), (l,k),$ and $ (l,l)$ set to $c, \textcolor{mylightbluetext}{-s, s}, $ and $c$, versus one where those entries are set to $c, \textcolor{mylightbluetext}{s, -s}, $ and $c$?

\item Verify that a Givens rotation is an orthogonal matrix, and its determinant is 1.

\end{problemset}

\part{Triangularization, Orthogonalization, and Gram--Schmidt Process}

\newpage
\chapter{QR Decomposition}

\index{Span}
\index{Subspace}
\index{Column space}
\index{Linearly independent}
\section{QR Decomposition}
In many applications,  the column space of a matrix $\bA=[\ba_1, \ba_2, \ldots, \ba_n] \in \real^{m\times n}$ is of particular interest. The sequence of subspaces spanned by the columns $\ba_1, \ba_2, \ldots$ of $\bA$ is given by
$$
\cspace([\ba_1])\,\,\,\, \subseteq\,\,\,\, \cspace([\ba_1, \ba_2]) \,\,\,\,\subseteq\,\,\,\, \cspace([\ba_1, \ba_2, \ba_3])\,\,\,\, \subseteq\,\,\,\, \ldots,
$$
where $\cspace([\ldots])$ denotes the subspace spanned by the vectors enclosed in the brackets.
The principle behind \textit{QR decomposition} is to construct an orthonormal basis set $\bq_1, \bq_2, \ldots$ that spans the same sequence of subspaces:
$$
\big\{\cspace([\bq_1])=\cspace([\ba_1])\big\}\subseteq 
\big\{\cspace([\bq_1, \bq_2])=\cspace([\ba_1, \ba_2])\big\}\subseteq
\big\{\cspace([\bq_1, \bq_2, \bq_3])=\cspace([\ba_1, \ba_2, \ba_3])\big\} 
\subseteq \ldots.
$$
Orthogonal basis sets have many useful properties, such as simplifying coordinate transformations, projections, and distance computations.
The QR decomposition, stated below, summarizes the result. A detailed discussion of its existence follows in subsequent sections.

\index{Decomposition: QR}
\begin{theoremHigh}[QR decomposition]\label{theorem:qr-decomposition}
Any $m\times n$ matrix $\bA=[\ba_1, \ba_2, \ldots, \ba_n]$ (whether its columns are linearly independent or not) with $m\geq n$ can be decomposed as 
$$
\bA = \bQ\bR,
$$
where 
\begin{enumerate}
\item  \textbf{Reduced}: $\bQ$ is an $m\times n$ matrix with orthonormal columns, and $\bR$ is an $n\times n$ upper triangular matrix, known as the \textit{reduced QR decomposition} or \textit{economy} QR decomposition;
\item \textbf{Full}: $\bQ$ is an $m\times m$ matrix with orthonormal columns, and $\bR$ is an $m\times n$ upper triangular matrix,  known as the \textit{full QR decomposition}. If the upper triangular matrix is further restricted to be square, the full QR decomposition can be expressed as:
	$$
	\bA = \bQ\begin{bmatrix}
		\bR_0\\
		\bzero
	\end{bmatrix},
	$$
	where $\bR_0$ is an $n\times n$ upper triangular matrix. 
\end{enumerate}
If $\bA$ has full rank, i.e., $\bA$ has linearly independent columns, $\bR$ also has linearly independent columns, and $\bR$ is nonsingular in the \textit{reduced} case. This implies the diagonals of $\bR$ are nonzero. Under the additional condition that the diagonal entries of $\bR$ are positive, the reduced QR decomposition is  \textbf{unique}. However, the full QR decomposition is typically not unique because the rightmost $(m-n)$ columns of $\bQ$ can be arranged in any order.
\end{theoremHigh}

Note that geometrically, the diagonal element $r_{ii}$ of the upper triangular matrix $\bR_0$ is the distance (w.r.t. the $\ell_2$ norm) between $\ba_i$ (the $i$-th column of $\bA$) and $\spn\{\ba_1,\ba_2, \ldots,\ba_{i-1}\}$, $i=2,3,\ldots,n$; see Section~\ref{section:project-onto-a-vector}.

Once the decomposition $\bA = \bQ\bR$ is known (for a square invertible $\bA$), inverting $\bA$ is easy:
$$
\bA^{-1} = \bR^{-1} \bQ^\top. 
$$
Since $\bR$ is upper triangular, computing $\bR^{-1}$ by backward substitution is much simpler and more stable than inverting $\bA$ directly (see Problem~\ref{prob:qr_inver}). Numerical software exploits this fact (often under the hood) to compute inverses or pseudo-inverses via QR factorizations.

The method for computing the QR decomposition was formally introduced by Erhard Schmidt in 1907  \citep{schmidt1907theorie}. However, Schmidt himself observed that similar mathematical expressions had already appeared in the earlier work of Gram in 1883 \citep{gram1883ueber}. 
Despite this historical overlap, contemporary literature generally distinguishes between the two formulations. The procedure based on Schmidt's derivation is commonly referred to as the classical Gram--Schmidt process, whereas the version derived from Gram's original approach is known as the modified Gram--Schmidt process. For a more in-depth comparison and analysis, refer to Section~\ref{section:qr-gram-compute}.

\section{Project a Vector Onto Another Vector and Onto a Plane}\label{section:project-onto-a-vector}
An important concept in deriving the QR decomposition of a matrix is the projection of a vector onto another vector or onto a subspace.
\paragraph{Project a vector onto another vector.}
Projecting a vector $\ba$ onto another vector $\bb$ involves finding the vector that is closest to $\ba$ along the line defined by $\bb$. 
The projected vector, denoted as $\widehat{\ba}$, is a scalar multiple of $\bb$: $\widehat{\ba} = \widehat{x} \bb$. By construction, $\ba - \widehat{\ba}$ is perpendicular to $\bb$, as illustrated in Figure~\ref{fig:project-line}. 
This orthogonality condition leads to the following result:
\begin{tcolorbox}[title={Project vector $\ba$ onto vector $\bb$}]
$\ba^\perp =\ba-\widehat{\ba}$ is perpendicular to $\bb$, so $(\ba-\widehat{x}\bb)^\top\bb=0$: $\widehat{x}$ = $\frac{\ba^\top\bb}{\bb^\top\bb}$ and $\widehat{\ba} = \frac{\ba^\top\bb}{\bb^\top\bb}\bb = \frac{\bb\bb^\top}{\bb^\top\bb}\ba$.
\end{tcolorbox}

\begin{figure}[h]
	\centering
	\vspace{-0.8cm}
	\subfigtopskip=2pt
	\subfigbottomskip=2pt
	\subfigcapskip=-5pt
	\subfigure[Project onto a line.]{\label{fig:project-line}
		\includegraphics[width=0.37\linewidth]{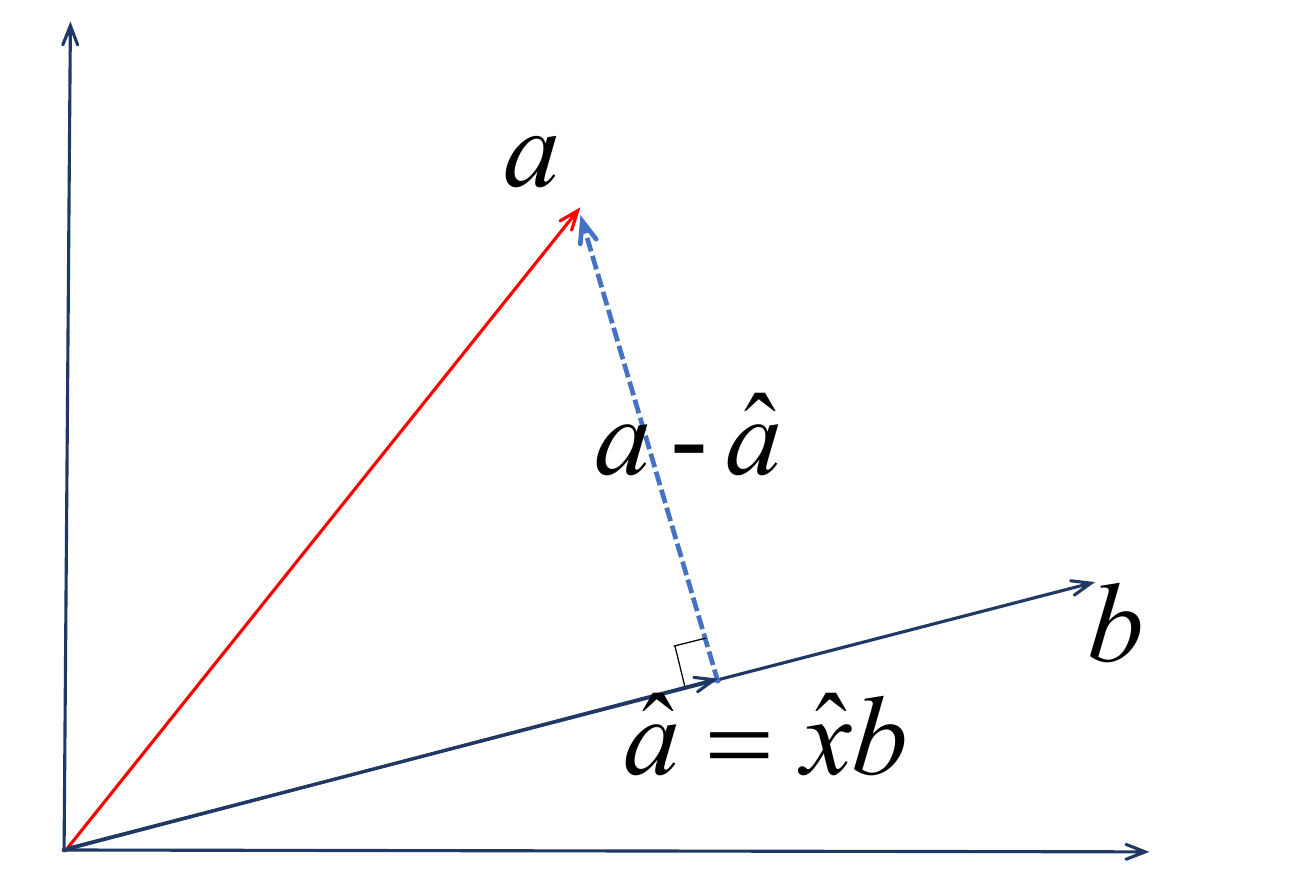}}
	\quad 
	\subfigure[Project onto a space.]{\label{fig:project-space}
		\includegraphics[width=0.37\linewidth]{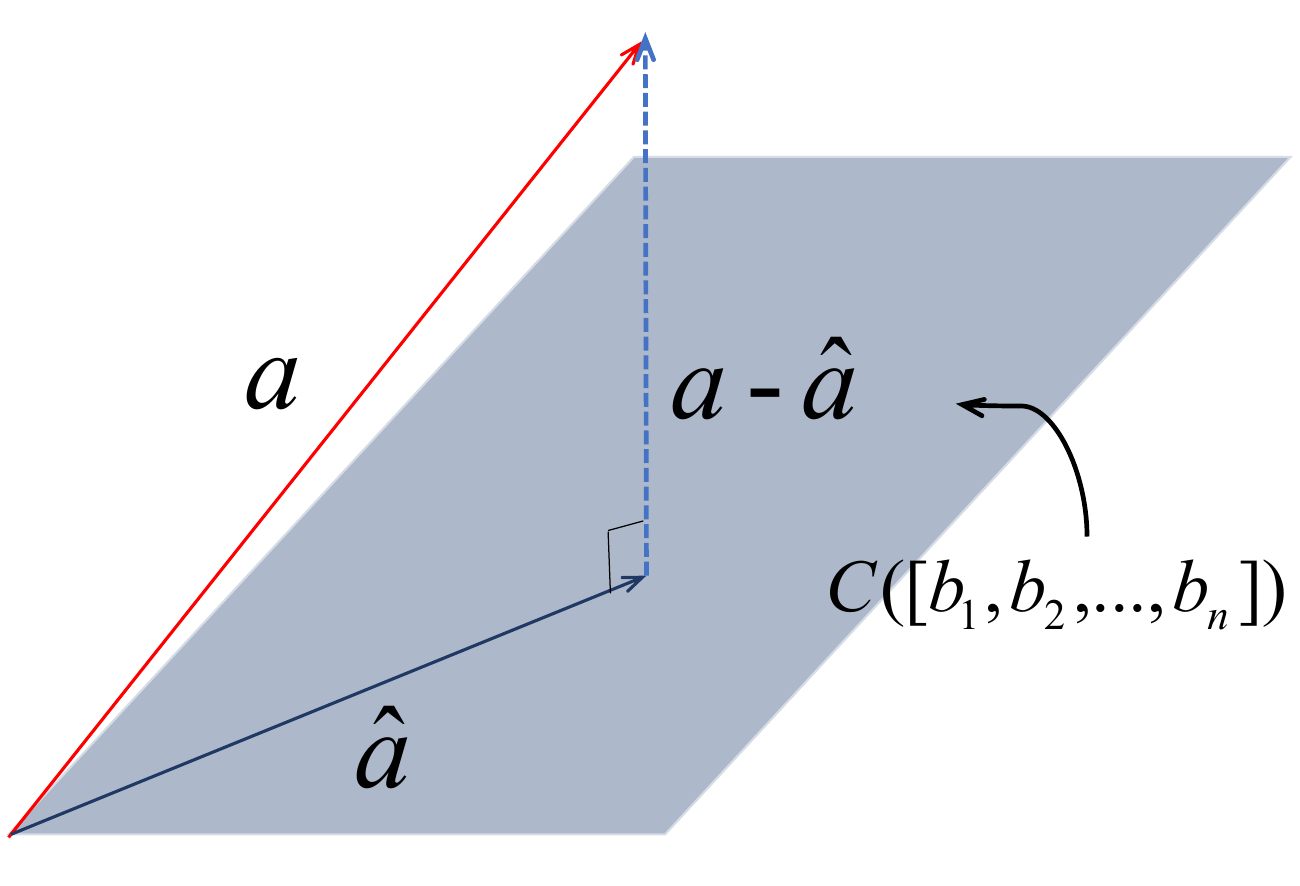}}
	\caption{Project a vector onto a line and a space.}
	\label{fig:projection-qr}
\end{figure}

\index{Normal equation}
\paragraph{Project a vector onto a plane.}
Similarly, the projection of a vector $\ba$ onto a space spanned by the vectors $\bb_1, \bb_2, \ldots, \bb_n$ involves finding the vector that is closest to $\ba$ within the column space of $\bB=[\bb_1, \bb_2, \ldots, \bb_n]$. The projected vector, denoted as $\widehat{\ba}$, is expressed as a linear combination of $\bb_1, \bb_2, \ldots, \bb_n$: $\widehat{\ba} = \widehat{x}_1\bb_1+ \widehat{x}_2\bb_2+\ldots+\widehat{x}_n\bb_n$. 
This process defines a least squares problem, which is solved using the normal equation: $\bB^\top\bB\widehat{\bx} = \bB^\top\ba$, where $\bB=[\bb_1, \bb_2, \ldots, \bb_n]$ and $\widehat{\bx}=[\widehat{x}_1, \widehat{x}_2, \ldots, \widehat{x}_n]$. 
Each individual projection of $\ba$ onto a single vector $\bb_i$ can be computed as:
$
\widehat{\ba}_i = \frac{\bb_i\bb_i^\top}{\bb_i^\top\bb_i}\ba,  \forall i \in \{1,2,\ldots, n\}.
$
The total projection is then obtained by summing all of these individual projections: $\widehat{\ba}=\sum_{i=1}^{n}\widehat{\ba}_i$. Consequently, the residual vector is orthogonal to the entire column space of $\bB$:
$
\ba^\perp = (\ba-\widehat{\ba}) \perp \cspace(\bB),
$ 
as shown in Figure~\ref{fig:project-space}.

\index{Gram--Schmidt}

\section{Existence of  QR Decomposition via  Gram--Schmidt Process}\label{section:gram-schmidt-process}

Given three linearly independent vectors $\ba_1, \ba_2,$ and $\ba_3$ that span  a space denoted by $\cspace{([\ba_1, \ba_2, \ba_3])}$---which corresponds to the column space of the matrix $[\ba_1, \ba_2, \ba_3]$---we aim to construct three orthogonal vectors $\{\bb_1, \bb_2, \bb_3\}$ such that $\cspace{([\bb_1, \bb_2, \bb_3])}$ = $\cspace{([\ba_1, \ba_2, \ba_3])}$. By normalizing these orthogonal vectors (dividing each by its norm), we obtain three mutually orthonormal vectors: $\bq_1 = \frac{\bb_1}{\norm{\bb_1}}$, $\bq_2 = \frac{\bb_2}{\norm{\bb_2}}$, and $\bq_2 = \frac{\bb_2}{\norm{\bb_2}}$.

To achieve this, we begin by setting $\bb_1 = \ba_1$. The second vector,  $\bb_2$,	 must be orthogonal to $\bb_1$. 
It is obtained by subtracting the projection of $\ba_2$ onto $\bb_1$:
\begin{equation}
\begin{aligned}
\bb_2 &= \ba_2- \frac{\bb_1 \bb_1^\top}{\bb_1^\top\bb_1} \ba_2 = \left(\bI- \frac{\bb_1 \bb_1^\top}{\bb_1^\top\bb_1} \right)\ba_2   \qquad &(\text{Projection view})\\
 &= \ba_2-  \underbrace{\frac{ \bb_1^\top \ba_2}{\bb_1^\top\bb_1} \bb_1}_{\widehat{\ba}_2}. \qquad &(\text{Combination view}) \nonumber
\end{aligned}
\end{equation}
The first equation shows that $\bb_2$ is computed by applying the matrix $\left(\bI- \frac{\bb_1 \bb_1^\top}{\bb_1^\top\bb_1} \right)$ to $\ba_2$, which projects $\ba_2$ onto the orthogonal complement of $\cspace{([\bb_1])}$. 
The second equality  expresses $\ba_2$ as a linear combination of its projection onto $\bb_1$ and a component orthogonal to $\bb_1$:  $\bb_2\perp\bb_1$.
This ensures that $\cspace([\bb_1, \bb_2]) =\cspace([\ba_1, \ba_2])$. Figure~\ref{fig:gram-schmidt1} illustrates the process, where  \textbf{the direction of $\bb_1$ is aligned with the $x$-axis of a  Cartesian coordinate system}. $\widehat{\ba}_2$ is the projection of $\ba_2$ onto the line defined by $\bb_1$. From the figure, it is clear that  $\bb_2 = \ba_2 - \widehat{\ba}_2$ is the component of $\ba_2$ orthogonal to $\bb_1$.

Similarly, the third vector, $\bb_3$, must be orthogonal to both  $\bb_1$ and $\bb_2$. It is constructed by subtracting the projections of $\ba_3$ onto the subspaces spanned by $\bb_1$ and $\bb_2$:
\begin{equation}\label{equation:gram-schdt-eq2}
\begin{aligned}
\bb_3 &= \ba_3- \frac{\bb_1 \bb_1^\top}{\bb_1^\top\bb_1} \ba_3 - \frac{\bb_2 \bb_2^\top}{\bb_2^\top\bb_2} \ba_3 = \left(\bI- \frac{\bb_1 \bb_1^\top}{\bb_1^\top\bb_1}  - \frac{\bb_2 \bb_2^\top}{\bb_2^\top\bb_2} \right)\ba_3   \qquad &(\text{Projection view})\\
&= \ba_3- \underbrace{\frac{ \bb_1^\top\ba_3}{\bb_1^\top\bb_1} \bb_1}_{\widehat{\ba}_3} - \underbrace{\frac{ \bb_2^\top\ba_3}{\bb_2^\top\bb_2}  \bb_2}_{\bar{\ba}_3}.    \qquad &(\text{Combination view})
\end{aligned}
\end{equation}
Once again, the first equation shows that the third vector $\bb_3$ is a multiplication of the matrix $\left(\bI- \frac{\bb_1 \bb_1^\top}{\bb_1^\top\bb_1}  - \frac{\bb_2 \bb_2^\top}{\bb_2^\top\bb_2} \right)$ and the vector $\ba_3$, i.e., projecting $\ba_3$ onto the orthogonal complement space of $\cspace{([\bb_1, \bb_2])}$. The second equality  expresses $\ba_3$ as a linear combination of $\bb_1, \bb_2, $ and $\bb_3$. We will see this property is essential in the idea of the QR decomposition.
Again, it can be shown that the space spanned by $\bb_1, \bb_2, \bb_3$ is identical to the space spanned by $\ba_1, \ba_2, \ba_3$. Figure~\ref{fig:gram-schmidt2} illustrates this step, where  \textbf{the direction of $\bb_2$ is aligned with the $y$-axis of the Cartesian coordinate system}. Here, $\widehat{\ba}_3$ is the projection of $\ba_3$ onto  $\bb_1$, while $\bar{\ba}_3$ is the projection of $\ba_3$ onto  $\bb_2$. 
The figure also shows that  the component of $\ba_3$ orthogonal to both $\bb_1$ and $\bb_2$ is $\bb_3=\ba_3-\widehat{\ba}_3-\bar{\ba}_3$.

Finally, each vector is normalized to produce the orthonormal set: $\bq_1 = \frac{\bb_1}{\norm{\bb_1}}$, $\bq_2 = \frac{\bb_2}{\norm{\bb_2}}$, and  $\bq_2 = \frac{\bb_2}{\norm{\bb_2}}$.

\begin{figure}[h]
	\centering
	\vspace{-0.35cm}
	\subfigtopskip=2pt
	\subfigbottomskip=2pt
	\subfigcapskip=-5pt
	\subfigure[Project $\ba_2$ onto the space perpendicular to $\bb_1$.]{\label{fig:gram-schmidt1}
		\includegraphics[width=0.37\linewidth]{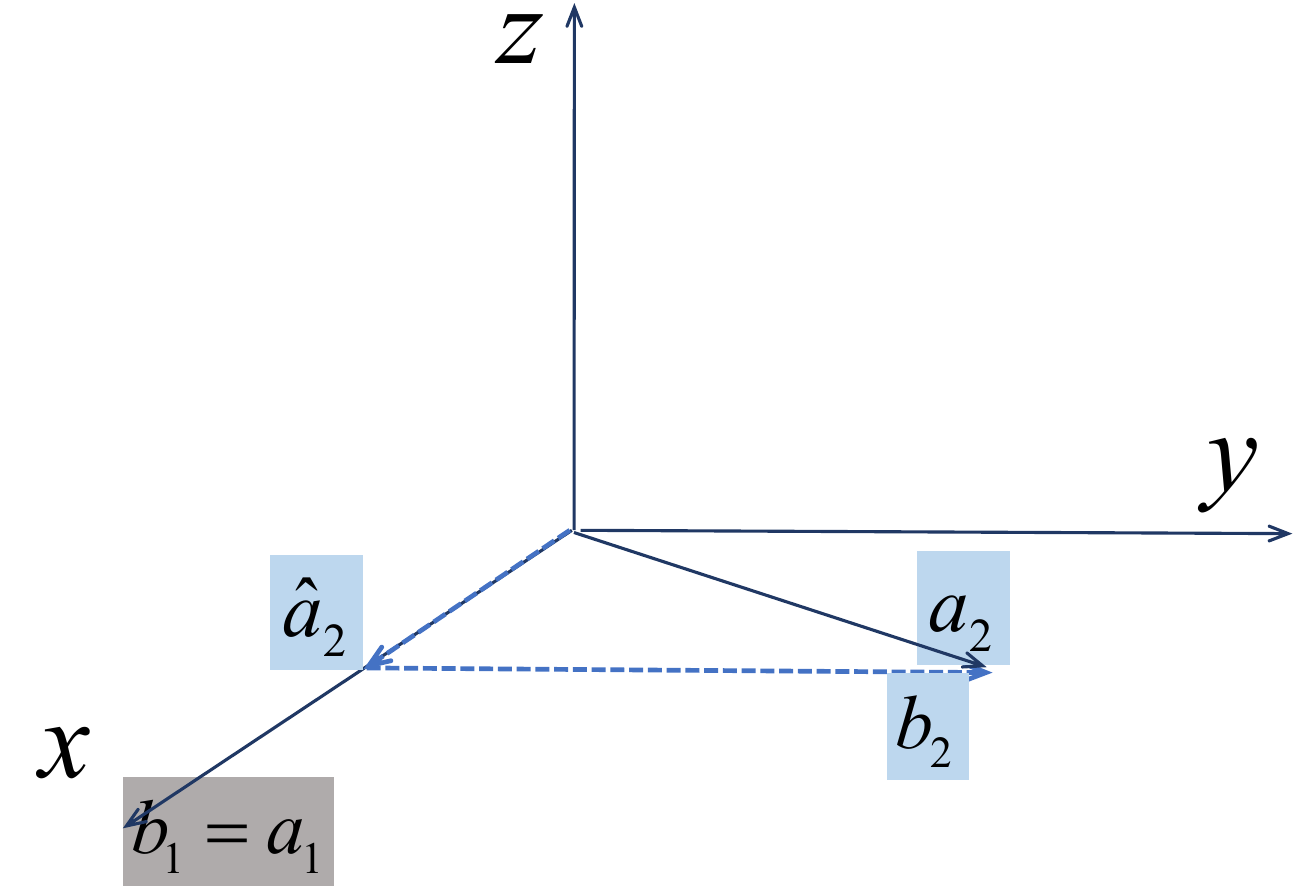}}
	\quad 
	\subfigure[Project $\ba_3$ onto the space perpendicular to $\bb_1, \bb_2$.]{\label{fig:gram-schmidt2}
		\includegraphics[width=0.37\linewidth]{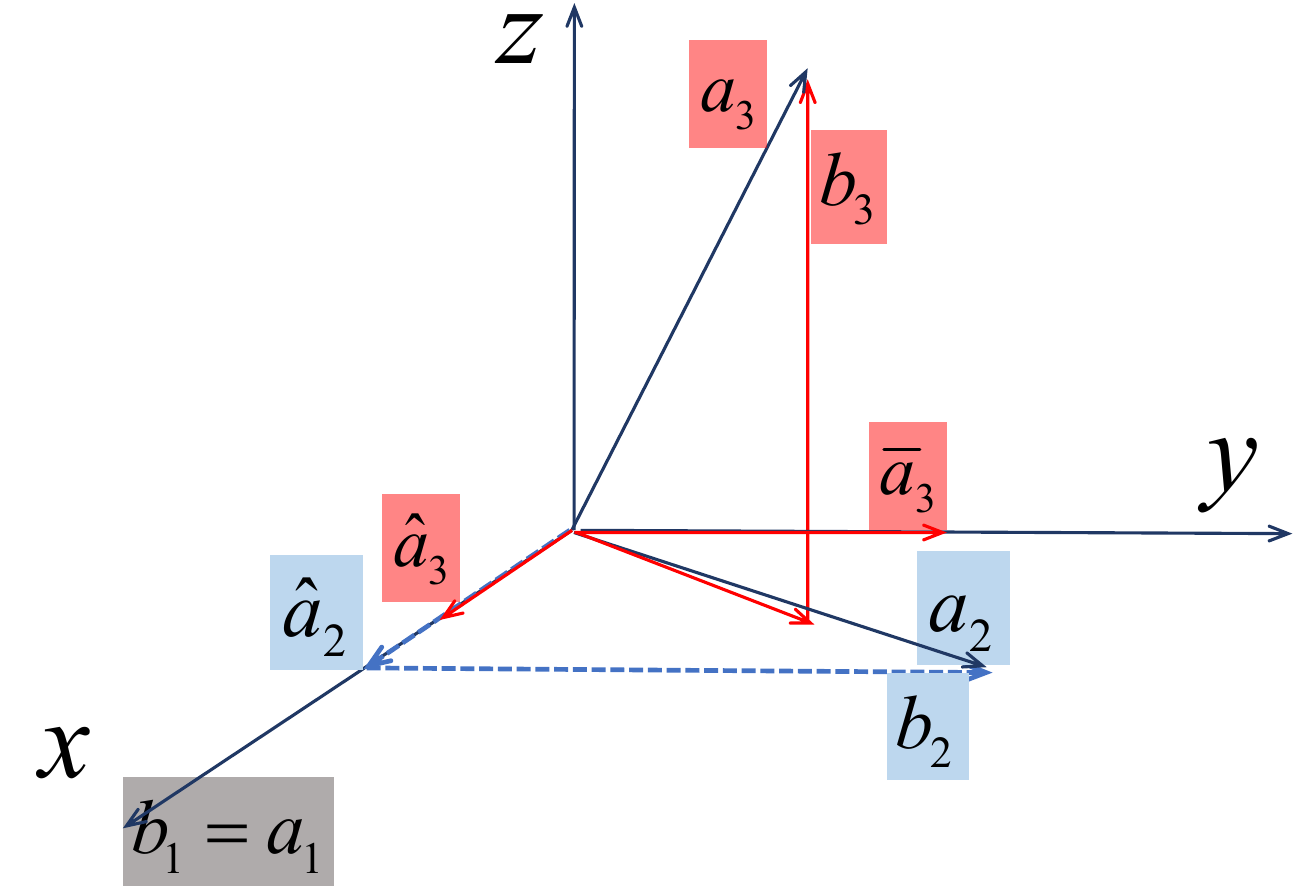}}
	\caption{The Gram--Schmidt process.}
	\label{fig:gram-schmidt-12}
\end{figure}

This process, called the \textit{(classical) Gram--Schmidt process}, generalizes to any set of linearly independent vectors \citep{gram1883ueber, schmidt1907theorie}. The method effectively triangularizes the matrix $\bA$, yielding its QR decomposition. 

As mentioned previously, the goal of the QR decomposition is to construct a sequence of orthonormal vectors $\bq_1, \bq_2, \ldots$ that span the same successive subspaces as the original vectors:
$$
\big\{\cspace([\bq_1])=\cspace([\ba_1])\big\} \subseteq
\big\{\cspace([\bq_1, \bq_2])=\cspace([\ba_1, \ba_2])\big\} \subseteq
\big\{\cspace([\bq_1, \bq_2, \bq_3])=\cspace([\ba_1, \ba_2, \ba_3])\big\} 
\subseteq \ldots.
$$
This implies that any vector $\ba_k$ lies in the space spanned by $\cspace([\bq_1, \bq_2, \ldots, \bq_k])$.~\footnote{And also, any vector $\bq_k$ lies in the space spanned by $\cspace([\ba_1, \ba_2, \ldots, \ba_k])$.} Once the orthonormal vectors are determined, the original matrix $\bA$ can be reconstructed as $\bA=\bQ\bR$, where $\bQ$ is orthogonal and $\bR$ is upper triangular.

While the Gram--Schmidt process is a classical approach to obtain the QR decomposition of a matrix, it is not the only one. Other algorithms, such as \textit{Householder reflections} and \textit{Givens rotations}, are also widely used and often preferred due to their superior numerical stability in the presence of rounding errors. Additionally, these alternative methods may process the columns of 
$\bA$ in a different order; see Sections~\ref{section:qr-via-householder} and \ref{section:qr-givens}.

\index{Orthogonal matrix}
\index{Orthogonal}
\index{Orthonormal}
\section{Orthogonal vs Orthonormal}\label{section:orthogonal-orthonormal-qr}
The vectors $\bq_1, \bq_2, \ldots, \bq_n\in \real^m$ are said to be \textit{mutually orthogonal} if their dot products satisfy $\bq_i^\top\bq_j=0$  whenever $i \neq j$. 
If each of these vectors is normalized to have unit length, they become \textit{mutually orthonormal}. These orthonormal vectors can be arranged as columns in a matrix  $\bQ$:
\begin{itemize}
\item When $m\neq n$: the matrix $\bQ$ is easy to work with because $\bQ^\top\bQ=\bI \in \real^{n\times n}$. Such a matrix $\bQ$ with $m> n$ is sometimes referred to as a \textit{semi-orthogonal} matrix.

\item When $m= n$: the matrix $\bQ$ is square, and the condition $\bQ^\top\bQ=\bI$ implies that $\bQ^\top=\bQ^{-1}$, meaning the transpose of $\bQ$ is its inverse. Then we also have $\bQ\bQ^\top=\bI$, i.e., $\bQ^\top$ is the \textit{two-sided inverse} of $\bQ$. In this case, $\bQ$ is called an  \textit{orthogonal matrix}.~\footnote{Although $\bQ$ has orthonormal columns, the term ``orthonormal matrix" is \textbf{not} used due to historical conventions.} 
\end{itemize}
Orthogonal matrices can be interpreted as transformations that change the basis of a vector space while preserving both angles (inner products) and lengths. Specifically:
\begin{itemize}
\item The length of a vector is also preserved:  $ \norm{\bQ\bu} = \norm{\bu}.$
\item The angle (inner product) between two vectors remains invariant: $ \bu^\top \bv = (\bQ\bu)^\top(\bQ\bv).$ 
\end{itemize}
In real-valued cases, multiplying a vector by an orthogonal matrix $\bQ$ results in a \textit{rotation} (if $\det(\bQ)=1$) or a \textit{reflection} (if $\det(\bQ)=-1$) in the vector space. Many decomposition algorithms produce two orthogonal matrices, leading to  two such transformations (rotations or reflections); see Chapters~\ref{chapter:ulv-urv-decomposition} and \ref{chapter:SVD_realc}.

\index{Classical Gram--Schmidt process}
\index{Modified Gram--Schmidt process}
\index{CGS}
\index{MGS}
\index{Numerical stability}
\section{Computing  Reduced QR Decomposition via CGS and MGS}\label{section:qr-gram-compute}
We express the reduced QR decomposition in the form $\bA = \bQ\bR$, where $\bQ\in \real^{m\times n}$ and $\bR\in \real^{n\times n}$, as follows:
\begin{equation}
\bA=\left[
\begin{matrix}
	\ba_1 & \ba_2 & \ldots & \ba_n
\end{matrix}
\right] 
=\left[
\begin{matrix}
	\bq_1 & \bq_2 & \ldots & \bq_n
\end{matrix}
\right] 
\footnotesize
\begin{bmatrix}
	r_{11} & r_{12}& \dots & r_{1n}\\
	       & r_{22}& \dots & r_{2n}\\
	       &       &    \ddots  & \vdots \\
	\multicolumn{2}{c}{\raisebox{1.3ex}[0pt]{\Huge0}} & & r_{nn} \nonumber
\end{bmatrix}.
\end{equation}
The orthogonal matrix $\bQ$, with orthonormal columns,  can be easily calculated using  the Gram--Schmidt process. 
To understand why the matrix $\bR$ is upper triangular, we explicitly write the corresponding equations:
\begin{equation*}
\begin{aligned}
\ba_1 & = r_{11}\bq_1 &= \sum_{i=1}^{1} r_{i1}\bq_1, \\
&\vdots& \vdots\\
\ba_k &= r_{1k}\bq_1 + r_{2k}\bq_2 + \ldots + r_{kk}\bq_k  &= \sum_{i=1}^{k} r_{ik} \bq_k,\\ 
&\vdots& \vdots.\\
\end{aligned}
\end{equation*}
This formulation aligns with the second equation in Equation~\eqref{equation:gram-schdt-eq2} and confirms the upper triangular structure of $\bR$. 
Extending the idea of Equation~\eqref{equation:gram-schdt-eq2} to the $k$-th term, we obtain:
\begin{equation}
\begin{aligned}
\ba_k &= \sum_{i=1}^{k-1}(\bq_i^\top\ba_k)\bq_i + \ba_k^\perp = \sum_{i=1}^{k-1}(\bq_i^\top\ba_k)\bq_i + \norm{\ba_k^\perp}\cdot \bq_k, 
\end{aligned}
\end{equation}
which implies that we can gradually orthonormalize $\bA$ to obtain an orthonormal set $\bQ=[\bq_1, \bq_2, \ldots, \bq_n]$ by 
\begin{equation}\label{equation:qr-gsp-equation}
\left\{
\begin{aligned}
	r_{ik} &= \bq_i^\top\ba_k, \,\,\,\,\forall i \in \{1,2,\ldots, k-1\};\\ 
	\ba_k^\perp&= \ba_k-\sum_{i=1}^{k-1}r_{ik}\bq_i;\\
	r_{kk} &= \norm{\ba_k^\perp};\\
	\bq_k &= \ba_k^\perp/r_{kk}.
\end{aligned}
\right.
\end{equation}
This again shows that the diagonal element $r_{ii}$ of the upper triangular matrix  is the distance (w.r.t. the $\ell_2$ norm) between $\ba_i$ (the $i$-th column of $\bA$) and $\spn\{\ba_1,\ba_2, \ldots,\ba_{i-1}\}$, $i=2,3,\ldots,n$.
The procedure is outlined in Algorithm~\ref{alg:reduced-qr}.
\begin{algorithm}[h] 
\caption{Reduced QR Decomposition via Gram--Schmidt Process} 
\label{alg:reduced-qr} 
\begin{algorithmic}[1] 
\Require Matrix $\bA\in\real^{m\times n}$ with linearly independent columns, where $m\geq n$; 
\For{$k=1$ to $n$} \Comment{compute the $k$-th column of $\bQ,\bR$}
\For{$i=1$ to $k-1$}
\State $r_{ik} \leftarrow \bq_i^\top\ba_k$; \Comment{entry ($i,k$) of $\bR$}
\EndFor  
\State $\ba_k^\perp\leftarrow \ba_k-\sum_{i=1}^{k-1}r_{ik}\bq_i$;
\State $r_{kk} \leftarrow \norm{\ba_k^\perp}$; \Comment{main diagonal of $\bR$}
\State $\bq_k \leftarrow \ba_k^\perp/r_{kk}$; 
\EndFor
\State Output $\bQ=[\bq_1, \ldots, \bq_n]$ and $\bR$ with entry $(i,k)$ being $r_{ik}$.

\end{algorithmic} 
\end{algorithm}

\index{Orthogonal projection}
\index{Projection matrix (projector)}
\paragraph{Orthogonal projection.}
From Equation~\eqref{equation:qr-gsp-equation}, particularly Steps 2 to  6 of Algorithm~\ref{alg:reduced-qr}, we observe  that the first two equalities imply that 
\begin{equation}\label{equation:qr-gsp-equation2}
\left.
\begin{aligned}
	r_{ik} &= \bq_i^\top\ba_k, \,\,\,\,\forall i \in \{1,2,\ldots, k-1\}\\ 
	\ba_k^\perp&= \ba_k-\sum_{i=1}^{k-1}r_{ik}\bq_i\\
\end{aligned}
\right\}
\rightarrow 
\ba_k^\perp= \ba_k- \bQ_{k-1}\bQ_{k-1}^\top \ba_k=(\bI-\bQ_{k-1}\bQ_{k-1}^\top )\ba_k,
\end{equation}
where $\bQ_{k-1}=[\bq_1,\bq_2,\ldots, \bq_{k-1}]$. This implies $\bq_k$ can be computed as:
$$
\bq_k = \frac{\ba_k^\perp}{\norm{\ba_k^\perp}} = \frac{(\bI-\bQ_{k-1}\bQ_{k-1}^\top )\ba_k}{\norm{(\bI-\bQ_{k-1}\bQ_{k-1}^\top )\ba_k}}.
$$
The matrix $(\bI-\bQ_{k-1}\bQ_{k-1}^\top )$ in the above expression is known as an \textit{orthogonal projection matrix} (symmetric and idempotent; see Problem~\ref{prob:orthogo_proj}) that  projects $\ba_k$ \textbf{along} the column space of $\bQ_{k-1}$, ensuring the projected vector is orthogonal to the column space of $\bQ_{k-1}$ \citep{lu2021numerical}. 
As a result, the vector $\ba_k^\perp$ or $\bq_k$ calculated in this manner will be orthogonal to  $\cspace(\bQ_{k-1})$, i.e., it lies in the null space of $\bQ_{k-1}^\top$: $\nspace(\bQ_{k-1}^\top)$, according to  the fundamental theorem of linear algebra (Theorem~\ref{theorem:fundamental-linear-algebra}). 

\index{Fundamental theorem}

Let $\bP_1=(\bI-\bQ_{k-1}\bQ_{k-1}^\top )$. We assert that $\bP_1=(\bI-\bQ_{k-1}\bQ_{k-1}^\top )$ is an orthogonal projection matrix, which  projects any vector $\bv$ onto the null space of $\bQ_{k-1}^\top$. 
Additionally, let $\bP_2=\bQ_{k-1}\bQ_{k-1}^\top$. Then $\bP_2$ is also an orthogonal projection matrix, such that $\bP_2\bv$  projects any vector $\bv$ onto the column space of $\bQ_{k-1}$.

Why can the matrices $\bP_1$ and $\bP_2$ effectively project  vectors onto the corresponding subspaces?  It can be  shown that the column space of $\bQ_{k-1}$ is equal to the column space of $\bQ_{k-1}\bQ_{k-1}^\top$: 
$
\cspace(\bQ_{k-1})=\cspace(\bQ_{k-1}\bQ_{k-1}^\top)=\cspace(\bP_2).
$
Hence, $\bP_2\bv$ represents  a linear combination of the columns of $\bP_2$, which lies in the column space of $\bP_2$ or the column space of $\bQ_{k-1}$. 

A \textit{projection matrix} $\bP$ is formally defined as an idempotent matrix satisfying $\bP^2=\bP$.
This property reflects the intuitive idea that projecting a vector twice is the same as projecting it once.
What distinguishes  $\bP_2=\bQ_{k-1}\bQ_{k-1}^\top $ is that the projection $\widehat{\bv}$ of any vector $\bv$ is orthogonal to $\bv-\widehat{\bv}$:
$$
(\widehat{\bv}=\bP_2\bv) \perp (\bv-\widehat{\bv}).
$$
This property is the defining characteristic of an \textit{orthogonal projection matrix}. 
In contrast, a projection that is not orthogonal is called an \textit{oblique projection matrix}. 
When $\bP_2$ is an orthogonal projection matrix, the matrix $\bP_1=\bI-\bP_2$ is also an orthogonal projection matrix, projecting any vector onto the space perpendicular to  $\cspace(\bQ_{k-1})$, i.e., $\nspace(\bQ_{k-1}^\top)$ (see Proposition~\ref{proposition:orthogonal-projection_tmp}). 
Thus, we conclude that there are two complementary orthogonal projections:
$$
\left\{
\begin{aligned}
	\bP_1: &\gap \text{project onto $\nspace(\bQ_{k-1}^\top)$, along the column space of $\bQ_{k-1}$;} \\
	\bP_2: &\gap \text{project onto $\cspace(\bQ_{k-1})$, onto the column space of $\bQ_{k-1}$} .
\end{aligned}
\right.
$$

\paragraph{Modified Gram--Schmidt process (MGS).}
An additional noteworthy result arises  when the columns of $\bQ_{k-1}$ are mutually orthonormal. 
In this case, we observe the following decomposition:
\begin{equation}\label{equation:qr-orthogonal-equality}
{\bP_1 = \bI - \bQ_{k-1}\bQ_{k-1}^\top = (\bI-\bq_1\bq_1^\top)(\bI-\bq_2\bq_2^\top)\ldots (\bI-\bq_{k-1}\bq_{k-1}^\top),}
\end{equation}
where $\bQ_{k-1}=[\bq_1,\bq_2,\ldots, \bq_{k-1}]$, and each term $(\bI-\bq_i\bq_i^\top)$ serves to project a vector onto the subspace orthogonal to $\bq_i$.
This finding is crucial for advancing towards a \textit{modified Gram--Schmidt process (MGS)}, where projections and subtractions are performed iteratively.
To avoid confusion, the original  Gram--Schmidt method is often referred to as the \textit{classical Gram--Schmidt process (CGS)}.

The primary distinction between CGS and MGS lies in how they perform projections and subtractions. 
In CGS, the same vector is projected onto all previously computed orthonormal vectors before performing the subtraction. Conversely, in MGS, projection and subtraction are interleaved.
To illustrate this difference, consider a three-column matrix $\bA=[\ba_1,\ba_2,\ba_3]$, as shown in Figure~\ref{fig:projection-mgs-demons-3d}, where each step is represented  using a different color. Below is a summary of the processes for computing $\bq_k$ from the $k$-th column $\ba_k$ of $\bA$, given the orthonormalized vectors  $\{\bq_1, \bq_2, \ldots, \bq_{k-1}\}$:
$$
\begin{aligned}
	\text{(CGS)}: &\, \text{obtain $\bq_k$ by normalizing $\ba_k^\perp=(\bI-\bQ_{k-1}\bQ_{k-1}^\top)\ba_k$;} \\
	\text{(MGS)}: &\, \text{obtain $\bq_k$ by normalizing  $\ba_k^\perp=\left\{(\bI-\bq_{k-1}\bq_{k-1}^\top)\ldots\left[(\bI-\bq_2\bq_2^\top)\left((\bI-\bq_1\bq_1^\top) \ba_k\right)\right]\right\}$}, 
\end{aligned}
$$
where the nested parentheses in MGS indicate  the order of operations---each projection is applied sequentially, and the result is updated immediately.
\begin{figure}[h]
\centering
\vspace{-0.35cm}
\subfigtopskip=2pt
\subfigbottomskip=2pt
\subfigcapskip=-5pt
\subfigure[CGS, step 1: \textcolor{mylightbluetext}{blue} vector; step 2: \textcolor{mydarkgreen}{green} vector; step 3: \textcolor{mydarkpurple}{purple} vector.]{\label{fig:projection-mgs-demons-cgs}
	\includegraphics[width=0.47\linewidth]{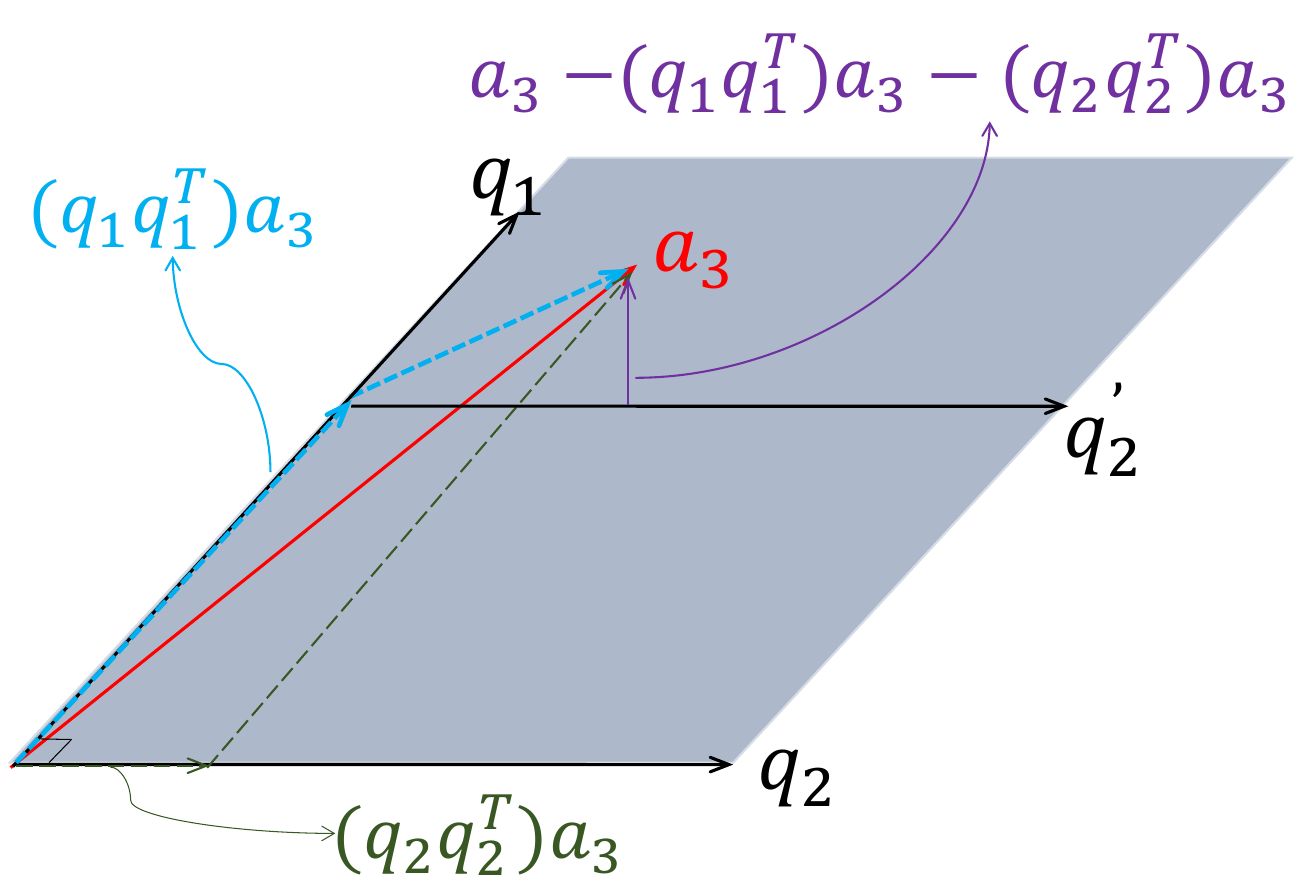}}
\quad 
\subfigure[MGS, step 1: \textcolor{mylightbluetext}{blue} vector; step 2: \textcolor{mydarkpurple}{purple} vector.]{\label{fig:projection-mgs-demons-mgs}
	\includegraphics[width=0.47\linewidth]{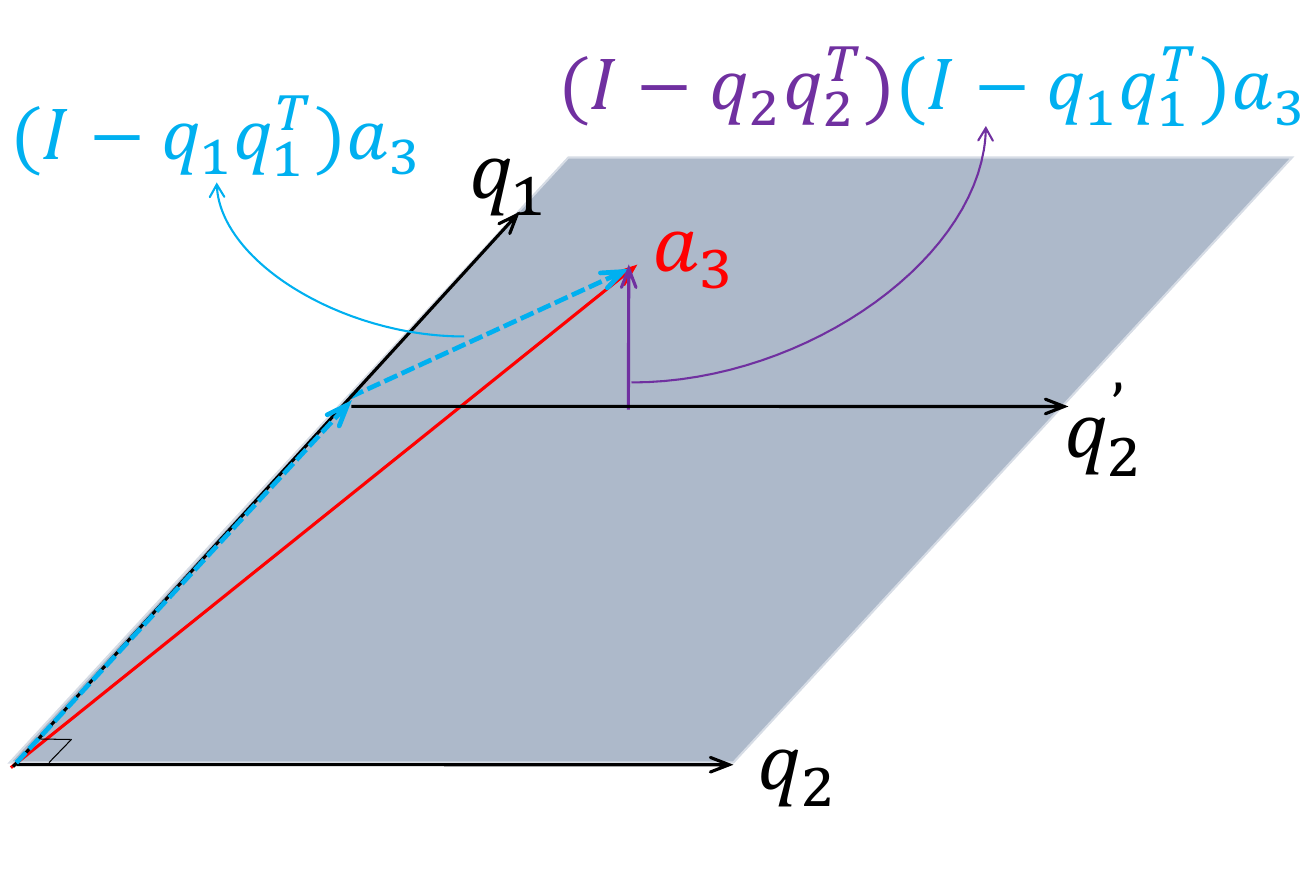}}
\caption{Comparison of CGS and MGS in three-dimensional space. Here, $\bq_2^\prime$ is parallel to $\bq_2$, so  projecting onto $\bq_2$ is equivalent to projecting onto $\bq_2^\prime$.}
\label{fig:projection-mgs-demons-3d}
\end{figure}

\paragraph{What's the difference?}
To illustrate, consider the three-column matrix  $\bA=[\ba_1, \ba_2, \ba_3]$. Suppose we have already computed $\{\bq_1, \bq_2\}$, where  $\spn\{\bq_1, \bq_2\}=\spn\{\ba_1, \ba_2\}$, our objective is to  compute $\bq_3$.

In the CGS algorithm, the orthogonalization of column $\ba_3$ against  $\{\bq_1, \bq_2\}$ is accomplished by simultaneously projecting the original column $\ba_3$ of $\bA$ onto $\bq_1$ and $\bq_2$, followed by subtracting these projections at once (Figure~\ref{fig:projection-mgs-demons-cgs}):
\begin{equation}\label{equation:cgs-3d-exmp}
\left\{
\begin{aligned}
\ba_3^\perp &= \ba_3 - (\bq_1^\top\ba_3)\bq_1 - (\bq_2^\top\ba_3)\bq_2= \ba_3 - (\bq_1\bq_1^\top)\ba_3 - \boxed{(\bq_2\bq_2^\top)\textcolor{mylightbluetext}{\ba_3}};\\
\bq_3 &=  \frac{\ba_3^\perp}{\norm{\ba_3^\perp}}.
\end{aligned}
\right.
\end{equation}

In contrast, the MGS algorithm subtracts the components along $\{\bq_1, \bq_2\}$ from $\ba_3$ sequentially as they are computed.
Therefore, the orthogonalization of column $\ba_3$ against $\{\bq_1, \bq_2\}$ is not performed by projecting the original column $\ba_3$ against $\{\bq_1, \bq_2\}$ as it is in the CGS, 
but rather by projecting onto updated intermediate vectors---those obtained after subtracting previously computed components.
This distinction is crucial because it leads to smaller error components of $\bq_3$ in $\spn\{\bq_1, \bq_2\}$ (a point discussed further in the following paragraphs).

More precisely, in the MGS algorithm, the orthogonalization of column $\ba_3$ against $\bq_1$ is accomplished by subtracting the component of $\bq_1$ from the vector $\ba_3$:
$$
\ba_3^{(1) }=  (\bI-\bq_1\bq_1^\top)\ba_3 = \ba_3 - (\bq_1\bq_1^\top)\ba_3,
$$
where $\ba_3^{(1) }$ represents the component of $\ba_3$ that is orthogonal to $\bq_1$. The subsequent step is then executed by 
\begin{equation}\label{equation:mgs-3d-exmp}
\begin{aligned}
\ba_3^{(2) }=  (\bI-\bq_2\bq_2^\top)\ba_3^{(1) }&=\ba_3^{(1) }-(\bq_2\bq_2^\top)\ba_3^{(1) }=\ba_3 - (\bq_1\bq_1^\top)\ba_3-\boxed{(\bq_2\bq_2^\top)\textcolor{mylightbluetext}{\ba_3^{(1) }}},
\end{aligned}
\end{equation}
where $\ba_3^{(2) }$ denotes the component of $\ba_3^{(1) }$ that is orthogonal  to $\bq_2$. The distinction from CGS (Equation~\eqref{equation:cgs-3d-exmp}) is highlighted in \textcolor{mylightbluetext}{blue} text.
Consequently, $\ba_3^{(2) }$ corresponds to the component of $\ba_3$ that is orthogonal to the entire subspace $\{\bq_1, \bq_2\}$, as shown in Figure~\ref{fig:projection-mgs-demons-mgs}. 

\index{Cancellation}
\paragraph{Main difference and catastrophic cancellation.} The key difference is that the vector $\ba_3$ can in general have large components in $\spn\{\bq_1, \bq_2\}$, in which case one starts with large
values and ends up with small values that result in large relative errors in them---a phenomenon known as \textit{catastrophic cancellation}. 
In contrast, in MGS, the intermediate vector $\ba_3^{(1) }$ is already orthogonal to $\bq_1$ and has only a small ``error" (residual) component in the direction of $\bq_1$. 
This significantly reduces the chance of large cancellations occurring in subsequent steps.
A comparison of the  \fbox{boxed} terms in Equations~\eqref{equation:cgs-3d-exmp} and \eqref{equation:mgs-3d-exmp} reveals that $(\bq_2\bq_2^\top)\ba_3^{(1) }$ in Equation~\eqref{equation:mgs-3d-exmp} is computed more accurately than $(\bq_2\bq_2^\top)\ba_3$ in CGS, as argued above.
Because of this reduced error in each projection step, 
the MGS method generally results in smaller orthogonalization errors at each stage compared to CGS. In fact, this difference can be quantified.
It can be shown that the final orthogonal matix $\bQ$ obtained using CGS satisfies the bound:
$$
\norm{\bI-\bQ\bQ^\top} \leq \mathcalO(\epsilon \kappa^2(\bA)),
$$
where $\kappa(\bA)$ is a value larger than 1 determined by $\bA$.
Whereas, in the MGS, the corresponding error satisfies
$$
\norm{\bI-\bQ\bQ^\top} \leq \mathcalO(\epsilon \kappa(\bA)).
$$
That is, the matrix $\bQ$ obtained via MGS is ``more orthogonal" than that obtained via CGS, making MGS a more numerically stable algorithm in practice.

\paragraph{More to go, preliminaries for Householder and Givens methods.} While MGS generally outperforms CGS in practice, it is not entirely immune to the \textit{catastrophic cancellation} issue. 
For example, in iteration $k$ of the MGS algorithm, if $\ba_k$ is nearly in the span of $\{\bq_1, \bq_2, \ldots, \bq_{k-1}\}$, then the resulting $\ba_k^\perp$ will have only a small component  perpendicular to $\spn\{\bq_1, \bq_2, \ldots, \bq_{k-1}\}$. This amplifies the ``error" component in $\spn\{\bq_1, \bq_2, \ldots, \bq_{k-1}\}$, leading to a less orthogonal $\bQ$. 
In such scenarios, a more robust approach involves finding a sequence of orthogonal matrices  $\{\bQ_1, \bQ_2, \ldots, \bQ_l\}$ such that the product $\bQ_l\ldots\bQ_2\bQ_1\bA$ becomes triangular. In this case, the resulting orthogonal matrix $\bQ=(\bQ_l\ldots\bQ_2\bQ_1)^\top$ will be ``more" orthogonal than those produced by either  CGS or  MGS. 
These more stable techniques will be explored in Section~\ref{section:qr-via-householder} and Section~\ref{section:qr-givens} using Householder reflectors and Givens rotations.

\section{Computing  Full QR Decomposition via  Gram--Schmidt Process}\label{section:silentcolu_qrdecomp}
A full QR decomposition of an $m\times n$ matrix with linearly independent columns involves extending the decomposition by appending additional $m-n$ orthonormal columns to $\bQ$, transforming it into an $m\times m$ orthogonal matrix. Simultaneously, rows of zeros are added to $\bR$, making it an $m\times n$ upper triangular matrix. The additional columns in $\bQ$ are referred to as \textit{silent columns}, while the additional rows in $\bR$ are called \textit{silent rows}. 
These do not affect the original decomposition but complete $\bQ$ to be a full orthogonal matrix.
Figure~\ref{fig:qr-comparison} illustrates the differences between the reduced and full QR decompositions, where silent columns in $\bQ$ are denoted in \textcolor{mydarkgray}{gray}, blank entries are zero, and \textcolor{mylightbluetext}{blue} entries indicate elements that are not necessarily zero.

\begin{figure}[h]
	\centering
	\vspace{-0.35cm}
	\subfigtopskip=2pt
	\subfigbottomskip=2pt
	\subfigcapskip=-5pt
	\subfigure[Reduced QR decomposition.]{\label{fig:gphalf}
		\includegraphics[width=0.47\linewidth]{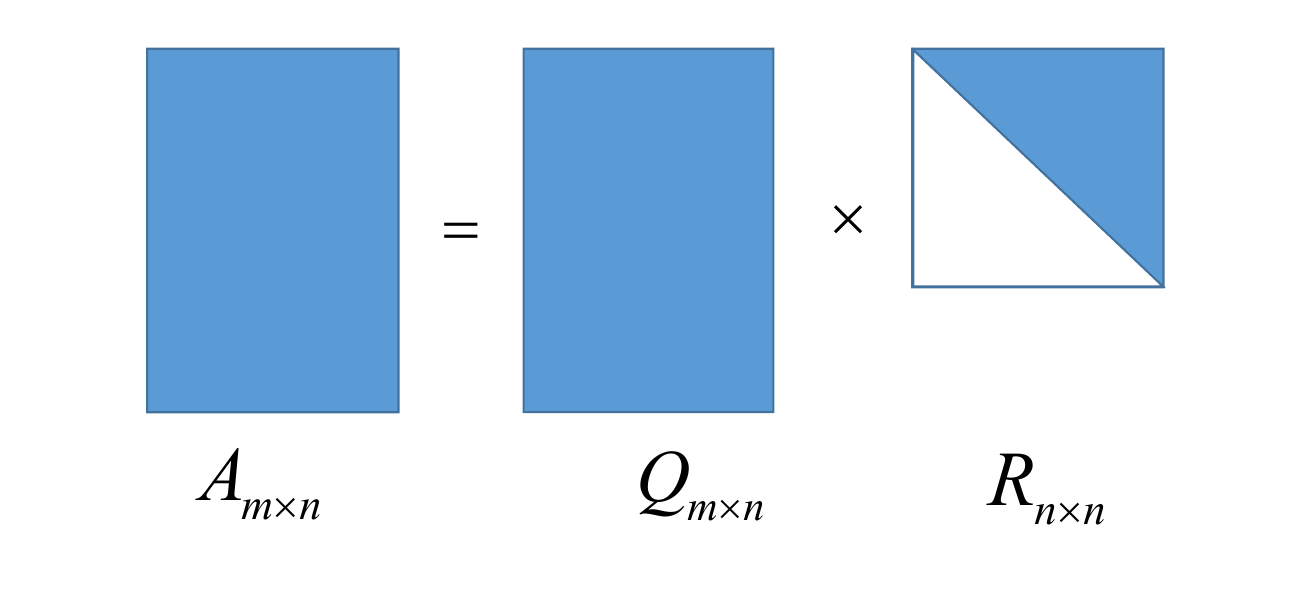}}
	\quad 
	\subfigure[Full QR decomposition.]{\label{fig:gpall}
		\includegraphics[width=0.47\linewidth]{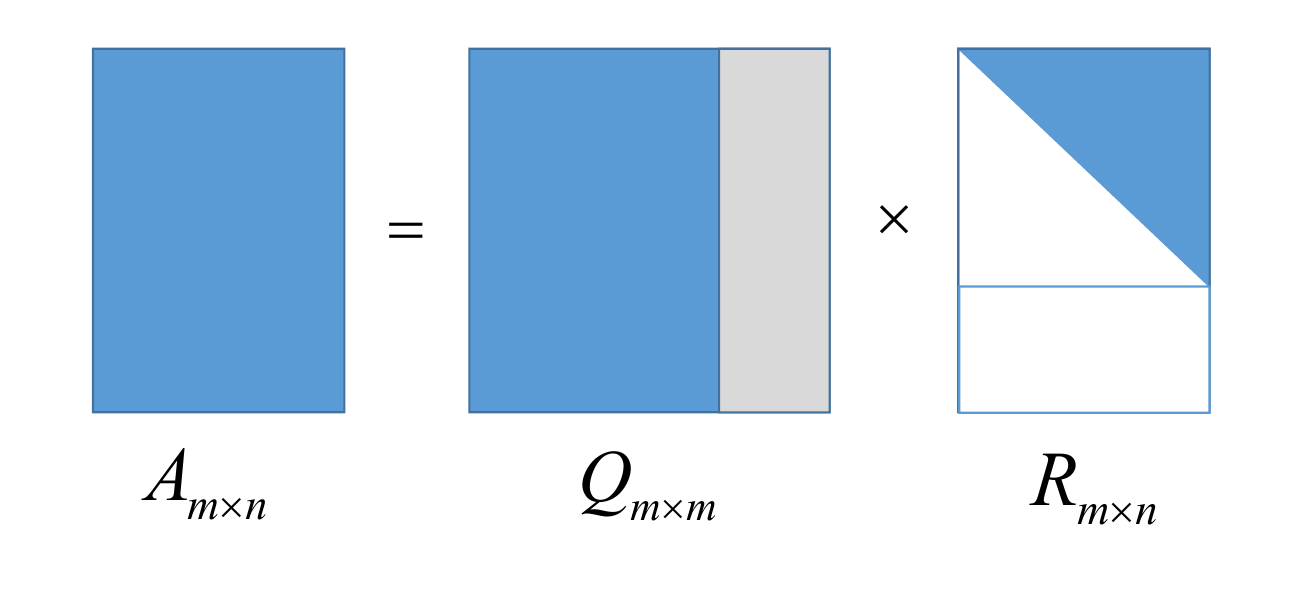}}
	\caption{Comparison between the reduced and full QR decompositions.}
	\label{fig:qr-comparison}
\end{figure}

\index{Independence check}
\section{Dependent Columns}\label{section:dependent-gram-schmidt-process}
Previously, we assumed that the matrix $\bA$ has linearly independent columns. However, this assumption is not always necessary.
Suppose that in step $k$ of the CGS or MGS process, the vector $\ba_k$ lies in the plane spanned by $\bq_1, \bq_2, \ldots, \bq_{k-1}$, which is equivalent to the space spanned by $\ba_1, \ba_2, \ldots, \ba_{k-1}$, i.e., the vectors $\ba_1, \ba_2, \ldots, \ba_k$ are dependent.
When this dependency occurs, the entry $r_{kk}$  becomes zero (see Equation~\eqref{equation:qr-gsp-equation} or Algorithm~\ref{alg:reduced-qr}), rendering $\bq_k$ undefined due to division by zero. 
To handle such scenarios, one can arbitrarily choose $\bq_k$ as any normalized vector orthogonal to the column space $\cspace([\bq_1, \bq_2, \ldots, \bq_{k-1}])$ and proceed with the Gram--Schmidt process.
For a matrix  $\bA$ with dependent columns, both reduced and full QR decomposition algorithms are still applicable. The procedure for step $k$ in the algorithm is redefined as follows:
$$
\bq_k=\left\{
\begin{aligned}
	&(\ba_k-\sum_{i=1}^{k-1}r_{ik}\bq_i)/r_{kk}, \qquad r_{ik}=\bq_i^\top\ba_k, r_{kk}=\norm{\ba_k-\sum_{i=1}^{k-1}r_{ik}\bq_i}, &\mathrm{if\,} r_{kk}\neq0, \\
	&\text{pick  one vector in }\cspace^{\bot}([\bq_1, \bq_2, \ldots, \bq_{k-1}]), \text{ and normalize},\qquad &\mathrm{if\,} r_{kk}=0.
\end{aligned}
\right.
$$

This idea can be further extended:  when $\bq_k$ does not exist, we simply skip the current step and add the silent columns at the end of the process. 
Consequently, the QR decomposition of a matrix with dependent columns is generally not unique. 

This framework also provides a practical method for determining linear independence. If $r_{kk}=0$ at any step in CGS or MGS, the vectors $\ba_1, \ba_2, \ldots, \ba_k$ are reported as linearly dependent.
At this point, the algorithm can be terminated for the purpose of detecting linear dependence.

\index{Permutation matrix}
\section{QR with Column Pivoting: Column-Pivoted QR (CPQR)}\label{section:cpqr}

If the columns of $\bA$ are linearly dependent, a \textit{column-pivoted QR (CPQR)} decomposition can be obtained as follows:
\begin{theoremHigh}[Column-pivoted QR decomposition\index{Column-pivoted QR (CPQR)}]\label{theorem:rank-revealing-qr-general}
Any $m\times n$ matrix $\bA=[\ba_1, \ba_2, \ldots, \ba_n]$ with $m\geq n$ and rank $r$ can be decomposed as 
$$
\bA\bP = \bQ
\begin{bmatrix}
\bR_{11} & \bR_{12} \\
\bzero   & \bzero 
\end{bmatrix},
$$
where $\bR_{11} \in \real^{r\times r}$ is upper triangular, $\bR_{12} \in \real^{r\times (n-r)}$, $\bQ\in \real^{m\times m}$ is an orthogonal matrix, and $\bP$ is a permutation matrix. This is  known as the \textit{full} CPQR decomposition. Similarly, the \textit{reduced} version is given by 
$$
\bA\bP = \bQ_r
\begin{bmatrix}
	\bR_{11} & \bR_{12} \\ 
\end{bmatrix},
$$
where $\bR_{11} \in \real^{r\times r}$ is upper triangular, $\bR_{12} \in \real^{r\times (n-r)}$, $\bQ_r\in \real^{m\times r}$ contains orthonormal columns, and $\bP$ is a permutation matrix.
\end{theoremHigh}
\index{CPQR}
\index{Column pivoting}
\subsection{A Simple CPQR via CGS}
The CPQR decomposition can be computed using the classical Gram--Schmidt process. 
In the context of QR decomposition for matrices with linearly dependent columns, if $r_{kk}=0$, this indicates that column $k$ of $\bA$ is linearly dependent on the previous $k-1$ columns.  
In such cases, a column permutation is performed, moving the dependent column to the end, after which the Gram--Schmidt process continues.
Here, $\bP$ represents the permutation matrix that reorders the dependent columns into the last $n-r$ positions. 
Suppose the first $r$ columns of $\bA\bP$ are $[\widehat{\ba}_1, \widehat{\ba}_2, \ldots, \widehat{\ba}_r]$. The span of these columns is equivalent to the span of $\bQ_r$ (in the reduced version) or the span of $\bQ_{:,1:r}$ (in the full version):
$$
\cspace([\widehat{\ba}_1, \widehat{\ba}_2, \ldots, \widehat{\ba}_r]) = \cspace(\bQ_r) = \cspace(\bQ_{:,1:r}).
$$
The matrix $\bR_{12}$ recovers the dependent $n-r$ columns from the column space of $\bQ_r$ or  $\bQ_{:,1:r}$. 
Figure~\ref{fig:qr-comparison-rank-reveal} compares the reduced and full CPQR decompositions, where silent columns in $\bQ$ are shown in \textcolor{mydarkgray}{gray}, blank entries represent zeros, and \textcolor{mylightbluetext}{blue}/\textcolor{orange}{orange} entries denote elements that are not necessarily zero.

\begin{figure}[H]
	\centering
	\vspace{-0.35cm}
	\subfigtopskip=2pt
	\subfigbottomskip=2pt
	\subfigcapskip=-5pt
	\subfigure[Reduced CPQR decomposition.]{\label{fig:gphalf-rank-reveal}
		\includegraphics[width=0.47\linewidth]{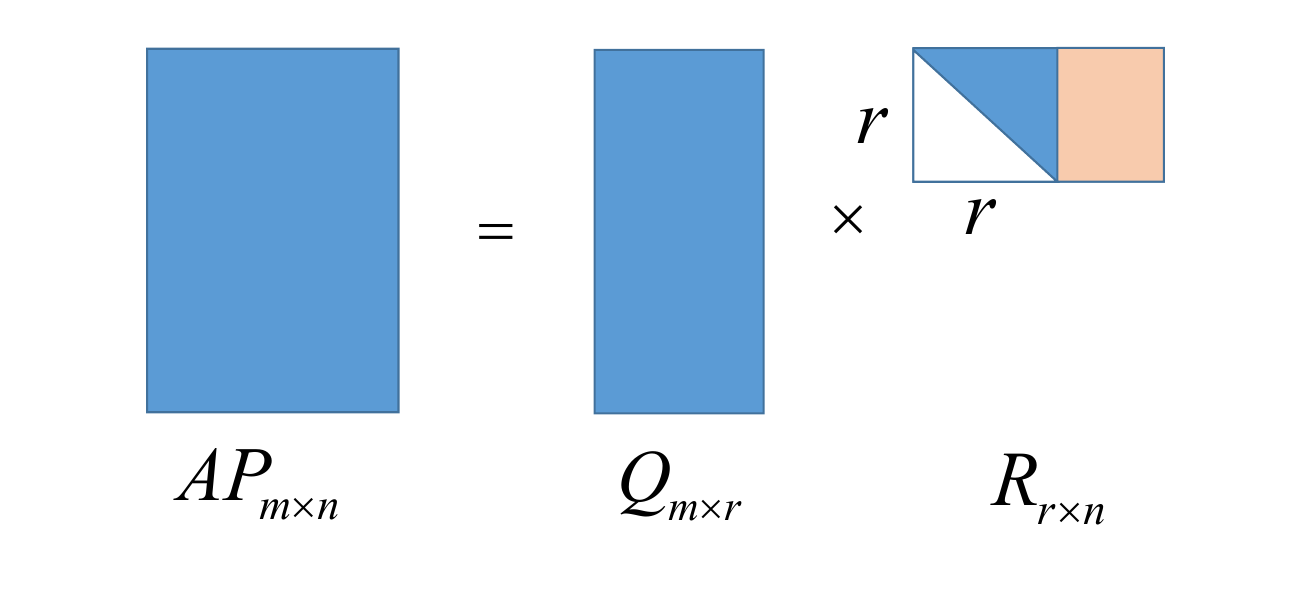}}
	\quad 
	\subfigure[Full CPQR decomposition.]{\label{fig:gpall-rank-reveal}
		\includegraphics[width=0.47\linewidth]{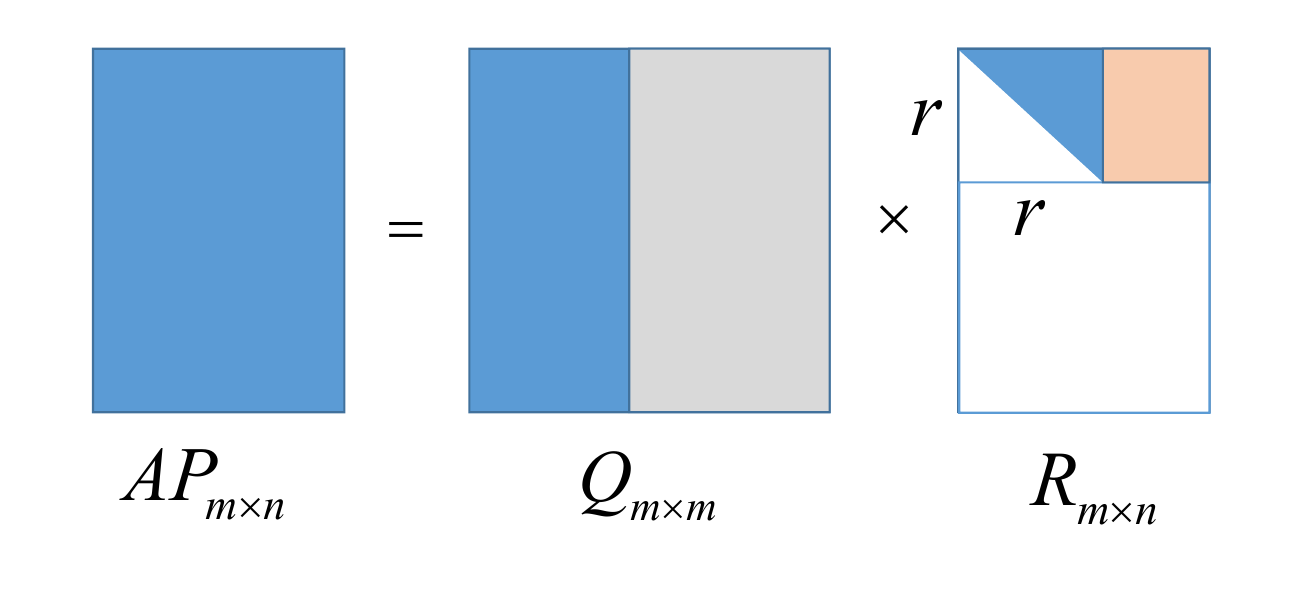}}
	\caption{Comparison between the reduced and full CPQR decompositions.}
	\label{fig:qr-comparison-rank-reveal}
\end{figure}

\subsection{A Practical CPQR via CGS}\label{section:practical-cpqr-cgs}
We observe that the simple CPQR algorithm pivots the first $r$ independent columns to become the first $r$ columns of $\bA\bP$. Let $\bA_1$ represent the first $r$ columns of $\bA\bP$, and $\bA_2$ represent the remaining columns. In the context of the full CPQR decomposition, we have:
$$
\bA\bP=
[\bA_1, \bA_2] = 
\bQ \begin{bmatrix}
	\bR_{11} & \bR_{12} \\
	\bzero   & \bzero 
\end{bmatrix}
=\left[
\bQ \begin{bmatrix}
	\bR_{11}  \\
	\bzero   
\end{bmatrix}
,
\bQ \begin{bmatrix}
	\bR_{12}  \\
	\bzero   
\end{bmatrix}
\right]
.
$$ 
It is evident that
$$
\norm{\bA_2} = \left\Vert\bQ \begin{bmatrix}
	\bR_{12}  \\
	\bzero   
\end{bmatrix}\right\Vert
=
\left\Vert\begin{bmatrix}
	\bR_{12}  \\
	\bzero   
\end{bmatrix}\right\Vert
=
\left\Vert \bR_{12}  \right\Vert,
$$
where the second equality follows from the invariance of the matrix norm under orthogonal transformations. Therefore, the norm of $\bR_{12}$ is directly determined  by the norm of $\bA_2$. 
For a well-conditioned CPQR, it is desirable for $\bR_{12}$ to have a small norm. A practical CPQR algorithm achieves this by first permuting the columns of $\bA$ so that they are ordered in decreasing vector norm:
$
\widetildebA = \bA\bP_0 = [\ba_{j_1}, \ba_{j_2}, \ldots, \ba_{j_n}],
$
where $\{j_1, j_2, \ldots, j_n\}$ is a permuted index set of $\{1,2,\ldots, n\}$, and 
$$
\norm{\ba_{j_1}}\geq  \norm{\ba_{j_2}}\geq  \ldots\geq  \norm{\ba_{j_n}}.
$$
The ``simple" reduced CPQR decomposition process is then applied to $\widetildebA$, resulting in $\widetildebA \bP_1= \bQ_r[\bR_{11}, \bR_{12}]$. The ``practical" reduced CPQR of $\bA$ is then recovered as
$$
\bA\underbrace{\bP_0\bP_1}_{\bP} =\bQ_r[\bR_{11}, \bR_{12}].
$$
The CPQR algorithm can be further enhanced by using the MGS process. This improved  approach has the additional advantage of stopping automatically when the factorization encounters a rank-deficient submatrix, thereby revealing the numerical rank of the matrix. This method is known as \textit{partial factorization}; see, for example, \citet{lu2021numerical} for more details.

\index{Column pivoting}
\index{Revealing rank-one deficiency}
\index{Rank-revealing}
\section{QR with Column Pivoting: Revealing Rank-One Deficiency}\label{section:rankone_defi}
Column-pivoted QR (CPQR) is one of several methods used to determine an appropriate column permutation when the matrix $\bA$ is rank-deficient. This process rearranges the first $r$ linearly independent  columns of $\bA$ to occupy the first $r$ columns of  $\bA\bP$. 
If $\bA$ is nearly rank-one deficient, the goal becomes identifying  a column permutation of $\bA$ that minimizes the pivotal element $r_{nn}$ in the resulting QR decomposition. This is commonly known as the \textit{revealing rank-one deficiency} problem for \textit{rank-revealing QR (RRQR)} decomposition.

The RRQR problem is particularly useful in the sense that it allows us to infer the numerical rank of a matrix without explicitly computing its singular value decomposition (SVD), which can be a significant advantage in terms of computational time and resources.
In least squares problems, where one seeks the best approximate solution to an overdetermined system of equations, the rank of the coefficient matrix plays a crucial role. An RRQR factorization can help identify the effective/numerical rank and thus the number of linearly independent equations, which is essential for solving such problems accurately.
On the other hand, in statistical modeling and machine learning, selecting a subset of regressors that best explains the variability in the response variable is a common task. RRQR factorization can assist in identifying the most relevant subset of variables by revealing the rank structure of the matrix formed by these variables, e.g., finding independent and significant alpha signals for quantitative strategies \citep{lu2022feature}.
\begin{theoremHigh}[Revealing rank-one deficiency \citep{chan1987rank}]\label{theorem:finding-good-qr-ordering}
Let $\bA\in \real^{m\times n}$ and let $\bv\in \real^n$ be a unit  vector (i.e., $\norm{\bv}=1$). There exists a permutation matrix $\bP$ such that the reduced QR decomposition
$$
\bA\bP = \bQ\bR
$$ 
satisfies $r_{nn} \leq \sqrt{n} \epsilon$, where $\epsilon = \norm{\bA\bv}$, and $r_{nn}$ is the $n$-th diagonal element of $\bR$. In this decomposition, $\bQ\in \real^{m\times n}$ and $\bR\in \real^{n\times n}$.
\end{theoremHigh}
\begin{proof}[of Theorem~\ref{theorem:finding-good-qr-ordering}]
Let $\bP\in \real^{n\times n}$ be a permutation matrix such that $\bw=\bP^\top\bv$, where 
$$
|w_n| = \max |v_i|,  \,\,\,\, \forall i \in \{1,2,\ldots,n\}.
$$
That is, we swap the entry with the largest magnitude to the last position, ensuring that the last component of $\bw$ equals the maximal component of $\bv$ in absolute value. 
Then we have $|w_n| \geq 1/\sqrt{n}$. Suppose the QR decomposition of $\bA\bP$ is $\bA\bP = \bQ\bR$. Then, 
$$
\epsilon = \norm{\bA\bv} = \norm{(\bQ^\top\bA\bP) (\bP^\top\bv)} = \norm{\bR\bw} = 
\norm{\begin{bmatrix}
		\vdots \\
		r_{nn} w_n
\end{bmatrix}}
\geq |r_{nn} w_n| \geq |r_{nn}|/\sqrt{n},
$$
where the second equality follows from the invariance of vector norms under orthogonal transformations, and  $\bP$ is an orthogonal matrix satisfying $\bP\bP^\top=\bI$. This concludes the proof.
\end{proof}
The following discussion makes use of the singular value decomposition (SVD), which will be introduced in Section~\ref{section:SVD}. You may skip this paragraph on a first reading.
Suppose the SVD of $\bA$ is given by $\bA = \sum_{i=1}^{n} \sigma_i \bu_i\bv_i^\top$, where $\sigma_i$'s are singular values satisfying $\sigma_1 \geq \sigma_2 \geq \ldots \geq \sigma_n$, i.e., $\sigma_n$ is the smallest singular value, and $\bu_i$'s and $\bv_i$'s are the corresponding left and right singular vectors, respectively. Then, if we let $\bv = \bv_n$ such that $\bA\bv_n = \sigma_n \bu_n$, \footnote{We will prove that the right singular vector of $\bA$ is equal to the right singular vector of $\bR$ if  $\bA$ admits the QR decomposition $\bA=\bQ\bR$ in Lemma~\ref{lemma:svd-for-qr}. The claim can also be applied to the singular values. So $\bv_n$ here is also a right singular vector of $\bR$.} we have
$$
\norm{\bA\bv} = \sigma_n. 
$$ 
By constructing a permutation matrix $\bP$ satisfying
$$
|\bP^\top \bv|_n = \max |v_i|,  \,\,\,\, \forall i \in \{1,2,\ldots,n\},
$$
we obtain a QR decomposition of $\bA\bP=\bQ\bR$ where the pivotal element $r_{nn}$ satisfies  $r_{nn} \leq \sqrt{n}\sigma_n$. If $\bA$ is rank-one deficient ($\sigma_n\approx 0$), then  $r_{nn}$ will also be close to zero, effectively revealing the matrix's near-rank deficiency.

\index{Revealing rank r deficiency}
\section{QR with Column Pivoting: Revealing Rank-r Deficiency*}\label{section:rank-r-qr}\index{Rank-revealing QR}
Building on the previous section, we now focus on computing the reduced QR decomposition of a matrix $\bA\in \real^{m\times n}$ that is approximately rank-$r$ deficient~\footnote{Note that rank $r$ here means  the matrix has a rank of $(\min\{m,n\}-r)$, not $r$.} with $r>1$. The goal now becomes finding a permutation matrix $\bP$ such that:
\begin{equation}\label{equation:rankr-reval-qr}
\bA\bP = 
\bQ\bR=
\bQ
\begin{bmatrix}
	\bL & \bM \\
	\bzero & \bN
\end{bmatrix},
\end{equation}
where $\bN \in \real^{r\times r}$, and $\norm{\bN}$ is small in some norm.
A recursive algorithm can be employed to achieve this. Suppose we have already isolated a small $k\times k$ block $\bN_k$. If we can isolate a small $(k+1)\times (k+1)$ block $\bN_{k+1}$, the permutation matrix can be determined recursively.
To reiterate, assume the existence of a permutation $\bP_k$ such that $\bN_k \in \real^{k\times k}$ has a small norm:
$$
\bA\bP_k = \bQ_k \bR_k=
\bQ_k 
\begin{bmatrix}
\bL_k & \bM_k \\
\bzero & \bN_k
\end{bmatrix}.
$$
Now, we aim to find a permutation $\bP_{k+1}$ such that $\bN_{k+1} \in \real^{(k+1)\times (k+1)}$ also has a small norm:
$$
\bA\bP_{k+1} = \bQ_{k+1} \bR_{k+1}=
\bQ_{k+1}
\begin{bmatrix}
	\bL_{k+1} & \bM_{k+1} \\
	\bzero & \bN_{k+1}
\end{bmatrix}.
$$
Using the algorithm described earlier, there exists an $(n-k)\times (n-k)$ permutation matrix $\widetilde{\bP}_{k+1}$ such that the matrix $\bL_k \in \real^{(n-k)\times (n-k)}$ has the QR decomposition $\bL_k \widetilde{\bP}_{k+1} = \widetilde{\bQ}_{k+1}\widetilde{\bL}_k$, where the entry $(n-k, n-k)$ of $\widetilde{\bL}_k$ is small. We then construct the following:
$$
\bP_{k+1} = \bP_k
\begin{bmatrix}
\widetilde{\bP}_{k+1} & \bzero \\
\bzero & \bI 
\end{bmatrix}
\qquad \text{and}\qquad 
\bQ_{k+1} = \bQ_k 
\begin{bmatrix}
\widetilde{\bQ}_{k+1} & \bzero \\
\bzero & \bI 
\end{bmatrix}.
$$
This leads to:
$$
{
\bA \bP_{k+1} = \bQ_{k+1}
\begin{bmatrix}
\widetilde{\bL}_k & \widetilde{\bQ}_{k+1}^\top  \bM_k \\
\bzero & \bN_k  
\end{bmatrix}}.
$$
Since the $(n-k, n-k)$-th entry of $\widetilde{\bL}_k$ is small, proving that  the last row of $\widetilde{\bQ}_{k+1}^\top  \bM_k$ is also small in norm will reveal the rank-$(k+1)$ deficiency of $\bA$ (see \citet{chan1987rank} for a formal proof).

\index{Householder transformation}
\index{Householder reflector}
\section{Existence of  QR Decomposition via  Householder Reflector}\label{section:qr-via-householder}

\textit{Householder matrices}, also known  as \textit{Householder reflectors}, are fundamental tools in numerical linear algebra. 
They are widely used in solving linear systems, estimating least squares solutions, and reducing matrices to Hessenberg form. This section illustrates how Householder reflectors can be utilized to prove the existence of the QR decomposition.

We begin by formally defining a Householder reflector and then examine  its key properties.
\begin{definition}[Householder reflector]\label{definition:householder-reflector}
Let $\bu \in \real^n$ be a vector of unit length (i.e., $\norm{\bu}=1$). The matrix $\bH = \bI - 2\bu\bu^\top$ is called a \textit{Householder reflector} or a \textit{Householder transformation}. This matrix is associated with the unit vector $\bu$, which is referred to as the \textit{Householder vector}. 
When a vector $\bx$ is multiplied by $\bH$, it is reflected across the hyperplane orthogonal to $\spn\{\bu\}$ (denoted as $\spn\{\bu\}^\perp$).

If $\norm{\bu} \neq 1$, the Householder reflector is defined as: $\bH = \bI - 2  \frac{\bu\bu^\top}{\bu^\top\bu} $.
\end{definition}
Derived from the definition of the Householder reflector, we obtain the following corollary, indicating that certain vectors remain unaltered when subjected to the Householder reflector.
\begin{corollary}[Unreflected by Householder]\label{corollary:unreflec_house}
Given a unit vector $\bu$,  the Householder reflector $\bH=\bI-2\bu\bu^\top$  leaves  any vector $\bv$ that is orthogonal to $\bu$ unchanged.  In other words, if  $\bu^\top\bv=0$, then $\bH\bv=\bv$. 
\end{corollary}
This result follows directly from substitution: $(\bI - 2\bu\bu^\top)\bv = \bv - 2\bu\bu^\top\bv=\bv$.

Let $\bu$ be a unit vector with $\norm{\bu}=1$, and let $\bv$ be a vector orthogonal to $\bu$. Then any vector $\bx$ in the plane can be decomposed into two components: $\bx = \bx_{\bv} + \bx_{\bu}$, where the first component $\bx_{\bu}$ is parallel to $\bu$ and the second one $\bx_{\bv}$ is orthogonal to $\bu$ (i.e., parallel to $\bv$). 
Using the projection formula from Section~\ref{section:project-onto-a-vector}, the component parallel to $\bx_{\bu}$ is $\bx_{\bu} = \frac{\bu\bu^\top}{\bu^\top\bu} \bx = \bu\bu^\top\bx$. 
Applying the Householder reflector $\bH=\bI - 2\bu\bu^\top$ to $\bx$, we get: $\bH\bx = (\bI - 2\bu\bu^\top)(\bx_{\bv} + \bx_{\bu}) = \bx_{\bv} -\bu\bu^\top \bx = \bx_{\bv} - \bx_{\bu}$. This demonstrates that the Householder reflector reflects  $\bx$ across the hyperplane $\spn\{\bu\}^\perp$.
In other words, the subspace perpendicular  to $\bu$ acts as a mirror, reflecting $\bx$. This transformation is illustrated in Figure~\ref{fig:householder}.

\begin{SCfigure}
\centering
\includegraphics[width=0.5\textwidth]{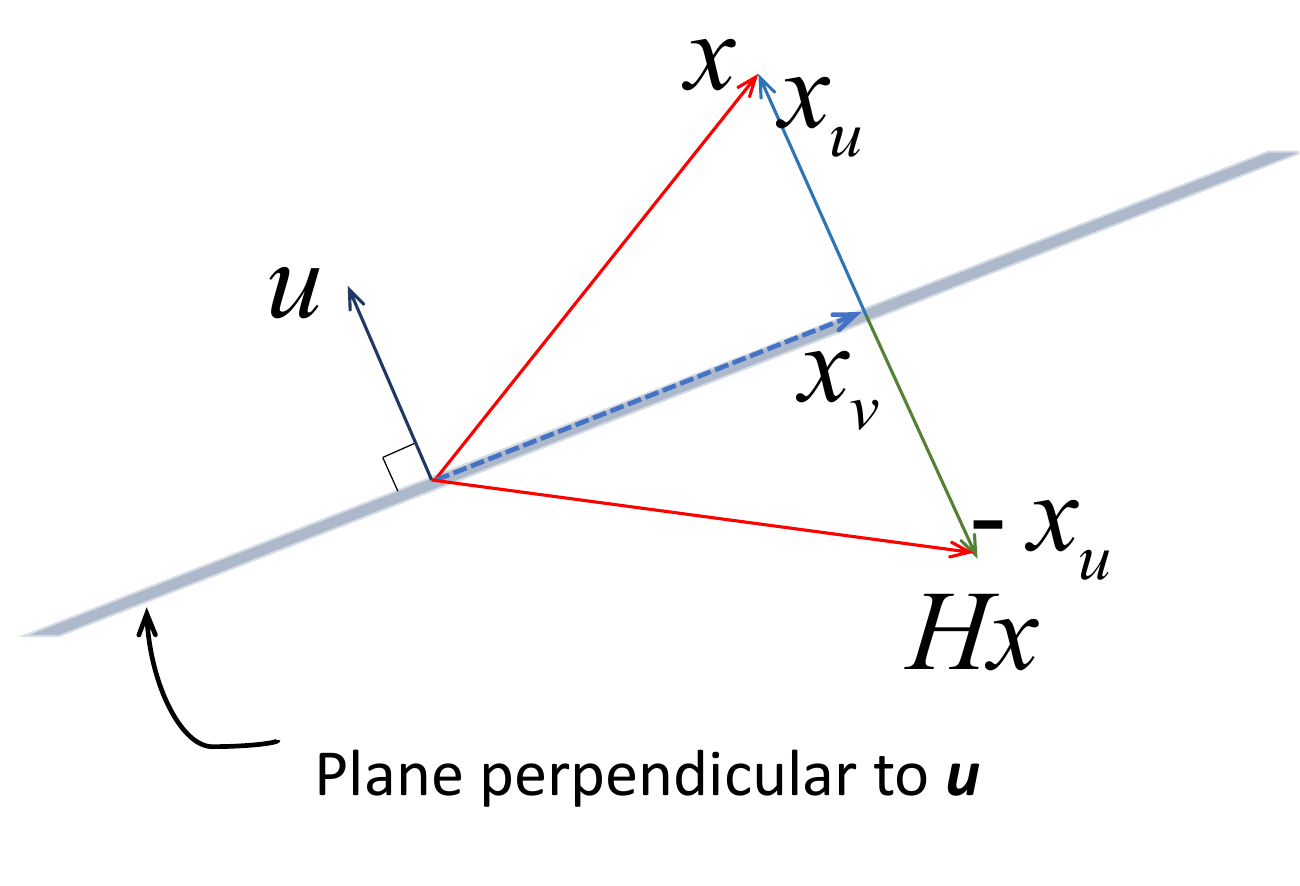}
\caption{Illustration  of the Householder reflector. The Householder reflector obtained by $\bH=\bI-2\bu\bu^\top$, where $\norm{\bu}=1$,  reflects a vector $\bx$ across the hyperplane orthogonal to $\bu$: $\bx=\bx_{\bv} + \bx_{\bu} \rightarrow \bx_{\bv} - \bx_{\bu}$.}
\label{fig:householder}
\end{SCfigure}

The previous explanation  explains how to determine the reflection of a vector using a Householder reflector. 
However, an additional question arises: Given two vectors that are reflections of each other via a Householder transformation, how can we determine the corresponding Householder reflector?
This question is especially relevant in the context of QR decomposition, where the objective is often to transform a column of a matrix into a desired form---typically one with zeros below the diagonal.
\begin{corollary}[Finding the Householder reflector]\label{corollary:householder-reflect-finding}
Suppose a vector $\bx$ is reflected to another vector $\by$ by a Householder reflector, and assume that  $\norm{\bx} = \norm{\by}$. Then, the corresponding Householder reflector can be computed as:
$$
\bH = \bI - 2 \bu\bu^\top, \text{ where } \bu = \frac{\bx-\by}{\norm{\bx-\by}}.
$$
\end{corollary}
\begin{proof}[of Corollary~\ref{corollary:householder-reflect-finding}]
Starting from the definition of the reflection,  we have:
$$
\begin{aligned}
\bH\bx &= \bx - 2 \bu\bu^\top\bx =\bx - 2\frac{(\bx-\by)(\bx^\top-\by^\top)}{(\bx-\by)^\top(\bx-\by)} \bx
= \bx - (\bx-\by) = \by.
\end{aligned}
$$
The condition $\norm{\bx} = \norm{\by}$ is necessary to ensure that this result holds true.
\end{proof}

Householder reflectors are particularly useful for zeroing out specific components of a vector.
For example, it is often desirable to eliminate all elements of a vector $\ba \in \real^n$  except for the $i$-th element. 
In such cases, the Householder vector can be defined as:
$$
\bu = \frac{\ba - r\be_i}{\norm{\ba - r\be_i}}, \qquad \text{where } r = \pm\norm{\ba},
$$
which is a valid Householder vector since $\norm{\ba} = \norm{r\be_i} = \abs{r}$. Specifically, if $r=\norm{\ba}$, then the Householder reflector $\bH = \bI - 2 \bu\bu^\top$ reflects $\ba$ to $\norm{\ba}\be_i$; conversely, if $r=-\norm{\ba}$, the reflector maps $\ba$  to $-\norm{\ba}\be_i$.

\begin{remark}[Householder properties]
A Householder reflector $\bH$ has the following properties:
\begin{itemize}
\item $\bH\bH = \bI$:  reflecting a vector twice yields the original vector.

\item Symmetry: $\bH = \bH^\top$.

\item Orthogonality: $\bH^\top\bH = \bH\bH^\top = \bI$, meaning $\bH$ is an orthogonal matrix.

\item $\bH\bu = -\bu$, if $\bH = \bI - 2 \bu\bu^\top$.
\end{itemize}
\end{remark}

\paragraph{QR using Householder.}
In the Gram--Schmidt process, we observe that the QR decomposition employs a triangular matrix to orthogonalize a given matrix $\bA$.
An alternative and computationally efficient approach involves constructing a sequence of orthogonal matrices that iteratively reduce $\bA$ to upper triangular form---this also results in a QR decomposition.
For example, consider an orthogonal matrix $\bQ_1$ that introduces zeros into all elements of the first column of $\bA$  except for the entry (1,1); similarly, another orthogonal matrix  $\bQ_2$ zeros out all entries of the second column below (2,2); and so forth. 
By applying such a sequence of transformations, we can obtain the QR decomposition of $\bA$.
This method uses reflections to align columns of the matrix with the standard basis vector $\be_1$, which has all entries equal to zero except for the first one.

To be more specific, let $\bA=[\ba_1, \ba_2, \ldots, \ba_n]\in \real^{m\times n}$ be the column partition of $\bA$. Define:
\begin{equation}\label{equation:qr-householder-to-chooose-r-numeraically}
r_1 = \norm{\ba_1},\qquad  \bu_1 = \frac{\ba_1 - r_1 \be_1}{\norm{\ba_1 - r_1 \be_1}}, \qquad \text{and}\qquad \bH_1 = \bI - 2\bu_1\bu_1^\top.
\end{equation}
Here, $\be_1 = [1;0;0;\ldots;0]\in \real^m$ denotes the first standard basis vector in $\real^m$.	Applying the reflector $\bH_1$ to $\bA$ gives:
\begin{equation}\label{equation:householder-qr-projection-step1}
\bH_1\bA = [\bH_1\ba_1, \bH_1\ba_2, \ldots, \bH_1\ba_n] = 
\begin{bmatrix}
r_1 & \bR_{1,2:n} \\
\bzero&  \bB_2
\end{bmatrix}.
\end{equation}
This operation reflects $\ba_1$ to $r_1\be_1$,  zeroing out all entries below the diagonal in the first column. 
Notably, we reflect $\ba_1$ to $\norm{\ba_1}\be_1$, where the two vectors have the same length (i.e., the transformation preserves the norm) rather than directly to $\be_1$  to ensure {numerical stability}; and this aligns with the conditions stated in Corollary~\ref{corollary:householder-reflect-finding}.

Next, we apply the same process to the submatrix $\bB_2$ from Equation~\eqref{equation:householder-qr-projection-step1}, aiming to zero out all elements below the (2,2) entry. This selective application ensures that previously introduced zeros in the first column are preserved.
Let $\bB_2 = [\bb_2, \bb_3, \ldots, \bb_n]$ be the column partition of $\bB_2$, and define 
$$
r_2 = \norm{\bb_2},\qquad \bu_2 = \frac{\bb_2 - r_2 \be_1}{\norm{\bb_2 - r_2 \be_1}}, \qquad  \qquad \widetilde{\bH}_2 = \bI - 2\bu_2\bu_2^\top, \qquad \text{and}\qquad  \bH_2 =
 \begin{bmatrix}
	1 & \bzero \\
	\bzero & \widetilde{\bH}_2
\end{bmatrix}.
$$
In this context, $\be_1$ now denotes the first unit basis in $\real^{m-1}$, and $\bH_2$ is orthogonal because $\widetilde{\bH}_2$ is orthogonal. Applying $\bH_2$ yields:
$$
\bH_2\bH_1\bA = [\bH_2\bH_1\ba_1, \bH_2\bH_1\ba_2, \ldots, \bH_2\bH_1\ba_n] = 
\begin{bmatrixfoot}
	r_1 & r_{12} & \bR_{1,3:n} \\
	0 & r_2 & \bR_{2,3:n} \\
	\bzero &  \bzero &\bC_3
\end{bmatrixfoot}.
$$

By repeating this process iteratively, we eventually transform $\bA$ into upper triangular form: $\bA = (\bH_n \bH_{n-1}\ldots\bH_1)^{-1} \bR = \bQ\bR$. Since each $\bH_i$ is symmetric and orthogonal, the inverse simplifies to:  $\bQ=(\bH_n \bH_{n-1}\ldots\bH_1)^{-1} = \bH_1 \bH_2\ldots\bH_n$.

For example, consider applying this method to a $5\times 4$ matrix. The transformation proceeds as follows, where $\boxtimes$ represents a value that is not necessarily zero, and \textbf{boldface} indicates the value has just been changed:
$$
\footnotesize
\begin{aligned}
\begin{sbmatrix}{\bA}
	\boxtimes & \boxtimes & \boxtimes & \boxtimes \\
	\boxtimes & \boxtimes & \boxtimes & \boxtimes \\
	\boxtimes & \boxtimes & \boxtimes & \boxtimes \\
	\boxtimes & \boxtimes & \boxtimes & \boxtimes \\
	\boxtimes & \boxtimes & \boxtimes & \boxtimes
\end{sbmatrix}
&\stackrel{\bH_1}{\rightarrow}
\begin{sbmatrix}{\bH_1\bA}
	\bm{\boxtimes} & \bm{\boxtimes} & \bm{\boxtimes} & \bm{\boxtimes} \\
	\bm{0} & \bm{\boxtimes} & \bm{\boxtimes} & \bm{\boxtimes} \\
	\bm{0} & \bm{\boxtimes} & \bm{\boxtimes} & \bm{\boxtimes} \\
	\bm{0} & \bm{\boxtimes} & \bm{\boxtimes} & \bm{\boxtimes} \\
	\bm{0} & \bm{\boxtimes} & \bm{\boxtimes} & \bm{\boxtimes}
\end{sbmatrix}
\stackrel{\bH_2}{\rightarrow}
\begin{sbmatrix}{\bH_2\bH_1\bA}
	\boxtimes & \boxtimes & \boxtimes & \boxtimes \\
	0 & \bm{\boxtimes} & \bm{\boxtimes} & \bm{\boxtimes} \\
	0 & \bm{0} & \bm{\boxtimes} & \bm{\boxtimes} \\
	0 & \bm{0} & \bm{\boxtimes} & \bm{\boxtimes} \\
	0 & \bm{0} & \bm{\boxtimes} & \bm{\boxtimes}
\end{sbmatrix}
\stackrel{\bH_3}{\rightarrow}
\begin{sbmatrix}{\bH_3\bH_2\bH_1\bA}
	\boxtimes & \boxtimes & \boxtimes & \boxtimes \\
	0 & \boxtimes & \boxtimes & \boxtimes \\
	0 & 0 & \bm{\boxtimes} & \bm{\boxtimes} \\
	0 & 0 & \bm{0} & \bm{\boxtimes} \\
	0 & 0 & \bm{0} & \bm{\boxtimes}
\end{sbmatrix}
\stackrel{\bH_4}{\rightarrow}
\begin{sbmatrix}{\bH_4\bH_3\bH_2\bH_1\bA}
	\boxtimes & \boxtimes & \boxtimes & \boxtimes \\
	0 & \boxtimes & \boxtimes & \boxtimes \\
	0 & 0 & \boxtimes & \boxtimes \\
	0 & 0 & 0 & \bm{\boxtimes} \\
	0 & 0 & 0 & \bm{0}
\end{sbmatrix}.
\end{aligned}
$$

The Householder algorithm is a powerful technique for transforming a matrix into upper triangular form using a sequence of orthogonal transformations.
In contrast to the Gram--Schmidt process (both  CGS  and MGS), which employs a triangular matrix to orthogonalize a given matrix, the Householder algorithm relies on orthogonal matrices to achieve triangularization. 
This key distinction can be summarized as follows:
\begin{itemize}
	\item Gram--Schmidt algorithm (triangular orthogonalization): Uses projections to orthogonalize vectors, resulting in a triangular matrix.
	\item Householder algorithm (orthogonal triangularization): Applies orthogonal transformations to triangularize the matrix.
\end{itemize}

Moreover, both the Householder algorithm and the Givens rotation method (to be discussed shortly) produce a \textit{full} QR decomposition by applying a sequence of orthogonal transformations. In contrast, the QR decomposition obtained via CGS or MGS typically results in a \textit{reduced} QR factorization. While it is possible to extend the reduced decomposition to full form by appending silent orthogonal columns or rows, this extension is not inherent to the CGS or MGS methods.

\index{Givens rotation}
\section{Existence of  QR Decomposition via Givens Rotation}\label{section:qr-givens}
In Definition~\ref{definition:givens-rotation-in-qr}, we introduced the concept of a Givens rotation, particularly in the context of finding the rank-one update or downdate of the Cholesky decomposition. Let us now examine the specific effects of Givens rotations through illustrative examples.
Consider the following $2\times 2$ orthogonal matrices:
$$
\bF = 
\begin{bmatrix}
-c & s\\
s & c
\end{bmatrix}, 
\qquad 
\bJ=
\begin{bmatrix}
c & -s \\
s & c
\end{bmatrix},
\qquad \text{and}\qquad 
\bG=
\begin{bmatrix}
	c & s \\
	-s & c
\end{bmatrix},
$$
where $s = \sin \theta$ and $c=\cos \theta$ for some angle $\theta$. The first matrix has determinant $\det(\bF)=-1$ and represents a special case of a Householder reflector in two dimensions. It can be expressed as  $\bF=\bI-2\bu\bu^\top$, where $\bu=\begin{bmatrix}
	\sqrt{\frac{1+c}{2}}, &\sqrt{\frac{1-c}{2}}
\end{bmatrix}^\top$ or $\bu=\begin{bmatrix}
-\sqrt{\frac{1+c}{2}}, &-\sqrt{\frac{1-c}{2}}
\end{bmatrix}^\top$. 
This matrix reflects vectors across a specific axis.
In contrast, the matrices $\bJ$ and $\bG$ have determinants $\det(\bJ)=\det(\bG)=1$ and perform rotations rather than reflections. Such matrices are referred to as  {\textit{Givens rotations}}.

\begin{figure}[H]
	\centering
	\vspace{-0.35cm}
	\subfigtopskip=2pt
	\subfigbottomskip=2pt
	\subfigcapskip=-5pt
	\subfigure[$\by = \bJ\bx$, counter-clockwise rotation.]{\label{fig:rotation1}
		\includegraphics[width=0.4\linewidth]{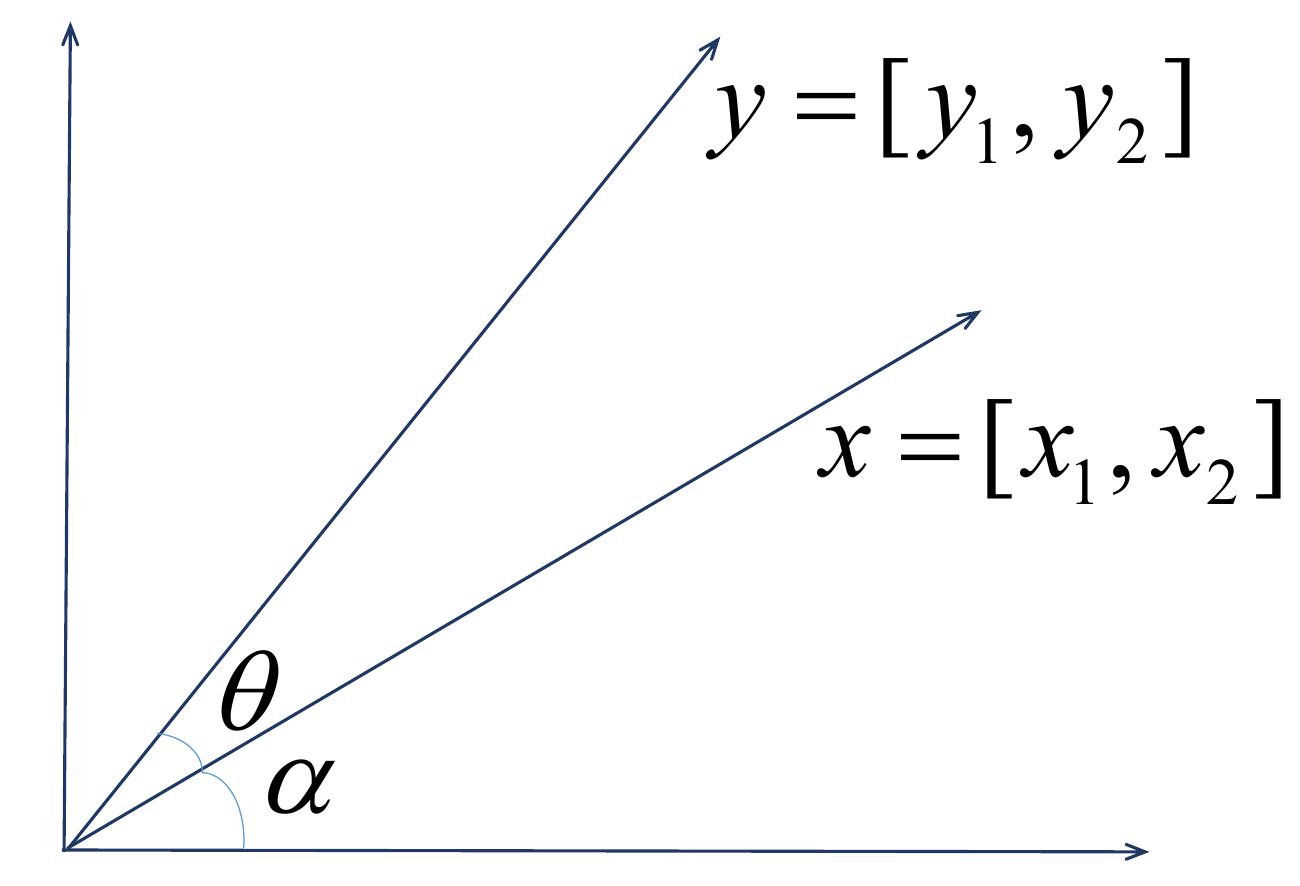}}
	\quad 
	\subfigure[$\by = \bG\bx$, clockwise rotation.]{\label{fig:rotation2}
		\includegraphics[width=0.4\linewidth]{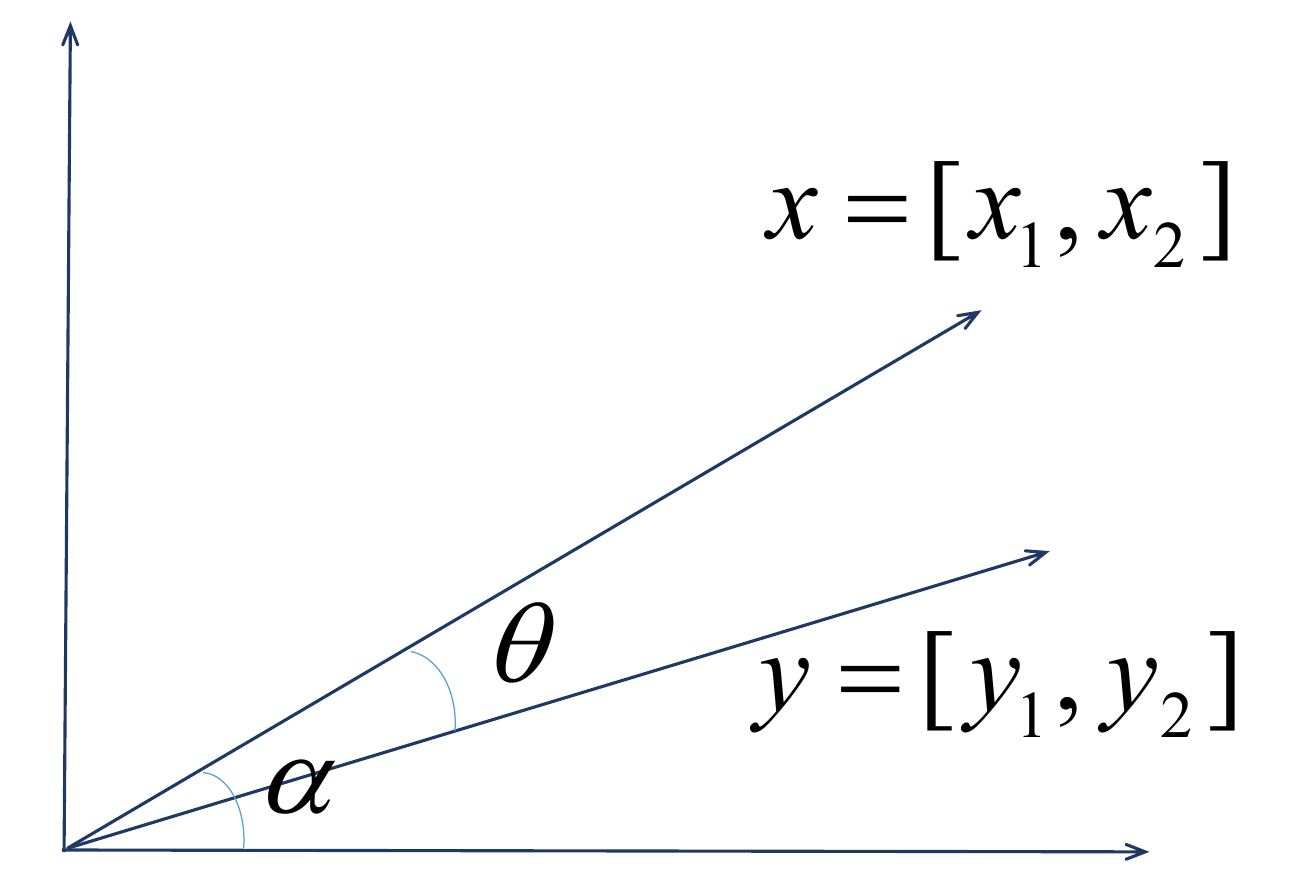}}
	\caption{Illustration  of two Givens rotations.}
	\label{fig:rotation}
\end{figure}

Figure~\ref{fig:rotation1} demonstrate the rotation of a vector $\bx$  under the action of the matrix $\bJ$, resulting in  $\by = \bJ\bx$ with
$
\left\{
\begin{aligned}
	&y_1 = c\cdot x_1 - s\cdot x_2;   \\
	&y_2 = s \cdot x_1 + c\cdot x_2.
\end{aligned}
\right.
$
We aim to verify that the angle between the  vectors $\bx$ and $\by$ is indeed $\theta$ (indicating a counter-clockwise rotation) under the Givens rotation $\bJ$.  
Firstly, we note that 
$$ 
\left\{
\begin{aligned}
	&\cos(\alpha) =\frac{x_1}{\sqrt{x_1^2+x_2^2}};   \\
	&\sin (\alpha) =\frac{x_2}{\sqrt{x_1^2+x_2^2}},
\end{aligned}
\right.
\qquad 
\text{and }\qquad 
\left\{
\begin{aligned}
	&\cos(\theta) =c;  \\
	&\sin (\theta) =s.
\end{aligned}
\right.
$$
This implies that $\cos(\theta+\alpha) = \cos(\theta)\cos(\alpha)-\sin(\theta)\sin(\alpha)$.
If we can show that $\cos(\theta+\alpha) = \cos(\theta)\cos(\alpha)-\sin(\theta)\sin(\alpha)$ is equal to $\frac{y_1}{\sqrt{y_1^2+y_2^2}}$, then we complete the proof.

For the former one, $\cos(\theta+\alpha) = \cos(\theta)\cos(\alpha)-\sin(\theta)\sin(\alpha)=\frac{c\cdot x_1 - s\cdot x_2}{\sqrt{x_1^2+x_2^2}}$. For the latter one, we can verify that $\sqrt{y_1^2+y_2^2}=\sqrt{x_1^2+x_2^2}$, and $\frac{y_1}{\sqrt{y_1^2+y_2^2}} = \frac{c\cdot x_1 - s\cdot x_2}{\sqrt{x_1^2+x_2^2}}$. This completes the proof. Similarly, we can also show that the angle between the vectors $\by=\bG\bx$ and $\bx$ is also $\theta$ in Figure~\ref{fig:rotation2}, and the rotation is clockwise.

It can be easily verified that the $n$-th order Givens rotation (Definition~\ref{definition:givens-rotation-in-qr}) is an orthogonal matrix with determinant 1. For any vector $\bx =[x_1, x_2, \ldots, x_n]^\top \in \real^n$, applying a Givens rotation yields $\by = \bG_{kl}\bx$ (the subscripts $k,l$ indicate the rotations occur \textbf{in plane $k$ and $l$}), where
$$ 
\left\{
\begin{aligned}
	&y_k = c \cdot x_k + s\cdot x_l,   \\
	&y_l = -s\cdot x_k +c\cdot x_l,  \\
	&y_j = x_j . &  (j\neq k,l) 
\end{aligned}
\right.
$$
In other words, a Givens rotation applied to $\bx$ rotates two components of $\bx$ by an angle $\theta$, while leaving all other components unchanged.
When $\sqrt{x_k^2 + x_l^2} \neq 0$,
let $c = \frac{x_k}{\sqrt{x_k^2 + x_l^2}}$ and $s=\frac{x_l}{\sqrt{x_k^2 + x_l^2}}$. Then we have 
$$ 
\left\{
\begin{aligned}
	&y_k = \sqrt{x_k^2 + x_l^2},   \\
	&y_l = 0,  \\
	&y_j = x_j . &  (j\neq k,l) 
\end{aligned}
\right.
$$
This result is critical for implementing the QR decomposition using Givens rotations.

\begin{corollary}[Basis transformation using forward Givens rotations]\label{corollary:basis-from-givens}
For any vector $\bx \in \real^n$, there exists a set of Givens rotations $\{\bG_{12}, \bG_{13}, \ldots, \bG_{1n}\}$ such that $\bG_{1n}\ldots \bG_{13}\bG_{12}\bx = \norm{\bx}\be_1$, where $\be_1\in \real^n$ is the first standard basis vector in $\real^n$.
\end{corollary}
\begin{proof}[of Corollary~\ref{corollary:basis-from-givens}]
From earlier results, we can construct a sequence of Givens rotations $\bG_{12}, \bG_{13},$ and $\bG_{14}$ such that 
$$
\begin{aligned}
\bG_{12}\bx &= \left[\sqrt{x_1^2 + x_2^2}, 0, x_3, \ldots, x_n \right]^\top;\\
\bG_{13}\bG_{12}\bx &= \left[\sqrt{x_1^2 + x_2^2+x_3^2}, 0, 0, x_4, \ldots, x_n \right]^\top;\\
\bG_{14}\bG_{13}\bG_{12}\bx &= \left[\sqrt{x_1^2 + x_2^2+x_3^2+x_4^2},0, 0, 0, x_5, \ldots, x_n \right]^\top.
\end{aligned}
$$
By continuing this process, we eventually obtain: $\bG_{1n}\ldots \bG_{13}\bG_{12} = \norm{\bx}\be_1$.
This completes the proof.
\end{proof}

\begin{remark}[Basis transformation using reverse Givens rotations]\label{remark:basis-from-givens2}
In Corollary~\ref{corollary:basis-from-givens}, the Givens rotations are applied in forward order, introducing zeros starting from the second entry  up to the $n$-th entry. 
However, in some cases, it may be desirable to apply the rotations in reverse order, introducing zeros from the $n$-th entry down to the second entry, such that $\bG_{12}\bG_{13}\ldots  \bG_{1n}\bx = \norm{\bx}\be_1$, where $\be_1\in \real^n$ is the first unit basis in $\real^n$.

The procedure is analogous to the forward case. Specifically, we construct Givens rotations $\bG_{1n},\bG_{1,(n-1)}, \bG_{1,(n-2)}$ as follows:
$$
\begin{aligned}
\bG_{1n}\bx &= \left[\sqrt{x_1^2 + x_n^2}, x_2, x_3, \ldots, x_{n-1}, 0 \right]^\top;\\
\bG_{1,(n-1)}\bG_{1n}\bx &= \left[\sqrt{x_1^2 + x_{n-1}^2+x_n^2}, x_2, x_3, \ldots, x_{n-2}, 0, 0  \right]^\top;\\
\bG_{1,(n-2)}\bG_{1,(n-1)}\bG_{1n}\bx &= \left[\sqrt{x_1^2 + x_{n-2}^2+x_{n-1}^2+x_n^2}, x_2, x_3,  \ldots,x_{n-3},0, 0, 0  \right]^\top.
\end{aligned}
$$
Continuing this process, we ultimately obtain: $\bG_{12}\bG_{13}\ldots  \bG_{1n}\bx = \norm{\bx}\be_1$.

Alternatively, there exists another sequence of rotations $\{\bG_{12}, \bG_{23}, \ldots,\bG_{(n-1),n}\}$ such that $\bG_{12} \bG_{23} \ldots \bG_{(n-1),n}\bx = \norm{\bx}\be_1$, where
$$
\begin{aligned}
\bG_{(n-1),n}\bx &= \left[x_1, x_2, \ldots, x_{n-2},\sqrt{x_{n-1}^2 + x_n^2}, 0 \right]^\top;\\
\bG_{(n-2),(n-1)}\bG_{(n-1),n}\bx &= \left[x_1, x_2, \ldots,x_{n-3}, \sqrt{x_{n-2}^2+x_{n-1}^2 + x_n^2}, 0, 0  \right]^\top;\\
\bG_{(n-3),(n-2)}\bG_{(n-2),(n-1)}\bG_{(n-1),n}\bx &=
\bigg[x_1,  x_2, \ldots ,x_{n-4}, \sqrt{x_{n-3}^2+x_{n-2}^2+x_{n-1}^2 + x_n^2},0, 0, 0  \bigg]^\top.
\end{aligned}
$$
By continuing this process, we ultimately obtain: $\bG_{12} \bG_{23} \ldots \bG_{(n-1),n}\bx = \norm{\bx}\be_1$.

This reverse application of Givens rotations will prove useful in the context of rank-one updates to the QR decomposition (Section~\ref{section:qr-rank-one-changes}).
\end{remark}

\paragraph{QR using Givens.}
From Corollary~\ref{corollary:basis-from-givens}, we know that we can introduce zeros by {rotating} the columns of a matrix to align with the basis vector $\be_1$.
Let $\bA=[\ba_1, \ba_2, \ldots, \ba_n] \in \real^{m\times n}$ be the column partition of $\bA$, and let 
\begin{equation}\label{equation:qr-rotate-to-chooose-r-numeraically}
 \bG_1 = \bG_{1m}\ldots \bG_{13}\bG_{12}.
\end{equation}
Then,
\begin{equation}\label{equation:rotate-qr-projection-step1}
\begin{aligned}
	\bG_1\bA &= [\bG_1\ba_1, \bG_1\ba_2, \ldots, \bG_1\ba_n] = 
\begin{bmatrix}
	\norm{\ba_1} & \bR_{1,2:n} \\
	\bzero&  \bB_2
\end{bmatrix},
\end{aligned}
\end{equation}
which rotates the first column $\ba_1$ to $\norm{\ba_1}\be_1$, introducing zeros below the diagonal in the first column.

Next, we apply this process to the submatrix $\bB_2$ from Equation~\eqref{equation:rotate-qr-projection-step1}, aiming to eliminate all entries below the (2,2) position.
Suppose $\bB_2 = [\bb_2, \bb_3, \ldots, \bb_n]$, and let 
$$
\bG_2 = \bG_{2m}\ldots\bG_{24}\bG_{23},
$$
where $\bG_{2n}, \ldots, \bG_{24}, \bG_{23}$ can be inferred from the context.
Applying both rotations yields:
$$
\begin{aligned}
\bG_2\bG_1\bA &= [\bG_2\bG_1\ba_1, \bG_2\bG_1\ba_2, \ldots, \bG_2\bG_1\ba_n] = 
\begin{bmatrixfoot}
	\norm{\ba_1} & r_{12} & \bR_{1,3:n} \\
	0 & \norm{\bb_2} & \bR_{2,3:n} \\
	\bzero &  \bzero &\bC_3
\end{bmatrixfoot}.
\end{aligned}
$$

This procedure can be repeated iteratively until the entire matrix $\bA$ is upper triangularized. The final result is: $\bA = (\bG_n \bG_{n-1}\ldots\bG_1)^{-1} \bR = \bQ\bR$. Since each matrix $\bG_i$ is orthogonal for $i\in\{1,2,\ldots,n\}$, we have $\bQ=(\bG_n \bG_{n-1}\ldots\bG_1)^{-1} = \bG_1^\top \bG_2^\top\ldots\bG_n^\top $, and 
\begin{equation}\label{equation:givens-q}
\begin{aligned}
	\bG_1^\top \bG_2^\top\ldots\bG_n^\top &=(\bG_n \ldots \bG_2 \bG_1)^\top \\
	&=\left\{(\bG_{nm} \ldots \bG_{n,(n+1)}) \ldots (\bG_{2m}\ldots \bG_{23})  ( \bG_{1m} \ldots \bG_{12} )\right\}^\top .
\end{aligned}
\end{equation}

In practice, the Givens rotation algorithm often outperforms the Householder method when the matrix $\bA$ already contains many zeros below the main diagonal.
Therefore, Givens rotations are particularly suited for rank-one changes in the QR decomposition, as these changes introduce only a small number of nonzero values (Section~\ref{section:qr-rank-one-changes}).
An example of a $5\times 4$ matrix is presented below, where $\boxtimes$ represents a value that is not necessarily zero, and \textbf{boldface} indicates the value has just been changed. 

\paragraph{Givens rotations in $\bG_1$.} For a $5\times 4$ example, we can express $\bG_1 = \bG_{15}\bG_{14}\bG_{13}\bG_{12}$. The process is shown below:
$$
\footnotesize
\begin{aligned}
\begin{sbmatrix}{\bA}
\boxtimes & \boxtimes & \boxtimes & \boxtimes \\
\boxtimes & \boxtimes & \boxtimes & \boxtimes \\
\boxtimes & \boxtimes & \boxtimes & \boxtimes \\
\boxtimes & \boxtimes & \boxtimes & \boxtimes \\
\boxtimes & \boxtimes & \boxtimes & \boxtimes
\end{sbmatrix}
&\stackrel{\bG_{12}}{\rightarrow}
\begin{sbmatrix}{\bG_{12}\bA}
\bm{\boxtimes} & \bm{\boxtimes} & \bm{\boxtimes} & \bm{\boxtimes} \\
\bm{0} & \bm{\boxtimes} & \bm{\boxtimes} & \bm{\boxtimes} \\
\boxtimes & \boxtimes & \boxtimes & \boxtimes \\
\boxtimes & \boxtimes & \boxtimes & \boxtimes \\
\boxtimes & \boxtimes & \boxtimes & \boxtimes
\end{sbmatrix}
\stackrel{\bG_{13}}{\rightarrow}
\begin{sbmatrix}{\bG_{13}\bG_{12}\bA}
\bm{\boxtimes} & \bm{\boxtimes} & \bm{\boxtimes} & \bm{\boxtimes} \\
0 & \boxtimes & \boxtimes &\boxtimes \\
\bm{0} & \bm{\boxtimes} & \bm{\boxtimes} & \bm{\boxtimes} \\
\boxtimes & \boxtimes & \boxtimes & \boxtimes \\
\boxtimes & \boxtimes & \boxtimes & \boxtimes
\end{sbmatrix}
\stackrel{\bG_{14}}{\rightarrow}
\begin{sbmatrix}{\bG_{14}\bG_{13}\bG_{12}\bA}
\bm{\boxtimes} & \bm{\boxtimes} & \bm{\boxtimes} & \bm{\boxtimes} \\
0 & \boxtimes & \boxtimes &\boxtimes \\
0 & \boxtimes & \boxtimes & \boxtimes \\
\bm{0} & \bm{\boxtimes} & \bm{\boxtimes} & \bm{\boxtimes} \\
\boxtimes & \boxtimes & \boxtimes & \boxtimes
\end{sbmatrix}
\stackrel{\bG_{15}}{\rightarrow}
\begin{sbmatrix}{\bG_{15}\bG_{14}\bG_{13}\bG_{12}\bA}
\bm{\boxtimes} & \bm{\boxtimes} & \bm{\boxtimes} & \bm{\boxtimes} \\
0 & \boxtimes & \boxtimes &\boxtimes \\
0 & \boxtimes & \boxtimes & \boxtimes \\
0 & \boxtimes & \boxtimes & \boxtimes \\
\bm{0} & \bm{\boxtimes} & \bm{\boxtimes} & \bm{\boxtimes} \\
\end{sbmatrix}.
\end{aligned}
$$

\paragraph{Givens rotation as a big picture.} 
When we consider $\bG_1, \bG_2, \bG_3, \bG_4$ as a single matrix, we have:
$$
\footnotesize
\begin{aligned}
\begin{sbmatrix}{\bA}
\boxtimes & \boxtimes & \boxtimes & \boxtimes \\
\boxtimes & \boxtimes & \boxtimes & \boxtimes \\
\boxtimes & \boxtimes & \boxtimes & \boxtimes \\
\boxtimes & \boxtimes & \boxtimes & \boxtimes \\
\boxtimes & \boxtimes & \boxtimes & \boxtimes
\end{sbmatrix}
&\stackrel{\bG_1}{\rightarrow}
\begin{sbmatrix}{\bG_1\bA}
\bm{\boxtimes} & \bm{\boxtimes} & \bm{\boxtimes} & \bm{\boxtimes} \\
\bm{0} & \bm{\boxtimes} & \bm{\boxtimes} & \bm{\boxtimes} \\
\bm{0} & \bm{\boxtimes} & \bm{\boxtimes} & \bm{\boxtimes} \\
\bm{0} & \bm{\boxtimes} & \bm{\boxtimes} & \bm{\boxtimes} \\
\bm{0} & \bm{\boxtimes} & \bm{\boxtimes} & \bm{\boxtimes}
\end{sbmatrix}
\stackrel{\bG_2}{\rightarrow}
\begin{sbmatrix}{\bG_2\bG_1\bA}
\boxtimes & \boxtimes & \boxtimes & \boxtimes \\
0 & \bm{\boxtimes} & \bm{\boxtimes} & \bm{\boxtimes} \\
0 & \bm{0} & \bm{\boxtimes} & \bm{\boxtimes} \\
0 & \bm{0} & \bm{\boxtimes} & \bm{\boxtimes} \\
0 & \bm{0} & \bm{\boxtimes} & \bm{\boxtimes}
\end{sbmatrix}
\stackrel{\bG_3}{\rightarrow}
\begin{sbmatrix}{\bG_3\bG_2\bG_1\bA}
\boxtimes & \boxtimes & \boxtimes & \boxtimes \\
0 & \boxtimes & \boxtimes & \boxtimes \\
0 & 0 & \bm{\boxtimes} & \bm{\boxtimes} \\
0 & 0 & \bm{0} & \bm{\boxtimes} \\
0 & 0 & \bm{0} & \bm{\boxtimes}
\end{sbmatrix}
\stackrel{\bG_4}{\rightarrow}
\begin{sbmatrix}{\bG_4\bG_3\bG_2\bG_1\bA}
\boxtimes & \boxtimes & \boxtimes & \boxtimes \\
0 & \boxtimes & \boxtimes & \boxtimes \\
0 & 0 & \boxtimes & \boxtimes \\
0 & 0 & 0 & \bm{\boxtimes} \\
0 & 0 & 0 & \bm{0}
\end{sbmatrix}.
\end{aligned}
$$

\index{Uniqueness}
\section{Uniqueness of  QR Decomposition}\label{section:nonunique-qr}
The results of QR decomposition can vary depending on the method used---such as the Gram–Schmidt process, the Householder algorithm, or the Givens algorithm.
Even within the Householder algorithm, different strategies exist for selecting the sign of  $r_1$ in Equation~\eqref{equation:qr-householder-to-chooose-r-numeraically}. As a result, the QR decomposition of a matrix is not necessarily unique.

However, the uniqueness of the \textit{reduced} QR decomposition for a full-column-rank matrix $\bA$ is guaranteed when the diagonal elements of $\bR$ are positive. 
Here, we provide a proof for the uniqueness of the {reduced} QR decomposition under the assumption that the diagonal elements of $\bR$ are positive. This proof also offers insight into the \textit{implicit Q theorem} used in Hessenberg decomposition (Section~\ref{section:hessenberg-decomposition}) and tridiagonal decomposition (Section~\ref{section:tridiagonal-decomposition}).
\begin{corollary}[Uniqueness of  reduced QR decomposition]\label{corollary:unique-qr}
Let  $\bA$ be an $m\times n$ matrix with full column rank $n$, where $m\geq n$. Then, the \textit{reduced} QR decomposition is \textbf{unique} if the main diagonal values of $\bR$ are positive.
\end{corollary}
\begin{proof}[of Corollary~\ref{corollary:unique-qr}]
Assume that the reduced QR decomposition is not unique. Then, it can be extended to a full QR decomposition, and we can find two such decompositions satisfying $\bA=\bQ_1\bR_1 = \bQ_2\bR_2$. This implies $\bR_1 = \bQ_1^{-1}\bQ_2\bR_2 = \bV \bR_2$, where $\bV= \bQ_1^{-1}\bQ_2$ is an orthogonal matrix. Expanding this equation gives:
$$
\begin{aligned}
\bR_1 &= 
\begin{bmatrixfoot}
	r_{11} & r_{12}& \dots & r_{1n}\\
	& r_{22}& \dots & r_{2n}\\
	&       &    \ddots  & \vdots \\
	\multicolumn{2}{c}{\raisebox{1.3ex}[0pt]{\Huge0}} & & r_{nn} \\
	\bzero & \bzero &\ldots & \bzero
\end{bmatrixfoot}=
\begin{bmatrix}
	v_{11}& v_{12} & \ldots & v_{1m}\\
	v_{21} & v_{22} & \ldots & v_{2m}\\
	\vdots & \vdots & \ddots & \vdots\\
	v_{m1} & v_{m2} & \ldots & v_{\textcolor{black}{mm}}
\end{bmatrix}
\begin{bmatrixfoot}
	s_{11} & s_{12}& \dots & s_{1n}\\
	& s_{22}& \dots & s_{2n}\\
	&       &    \ddots  & \vdots \\
	\multicolumn{2}{c}{\raisebox{1.3ex}[0pt]{\Huge0}} & & s_{nn} \\
		\bzero & \bzero &\ldots & \bzero
\end{bmatrixfoot}= \bV\bR_2,
\end{aligned}
$$
This implies 
$$
r_{11} = v_{11} s_{11}, \qquad v_{21}=v_{31}=v_{41}=\ldots=v_{m1}=0.
$$
Since $\bV$ is an orthogonal matrix and its columns are mutually orthonormal, with the first column having a norm of 1, it follows that $v_{11} = \pm 1$.
By assumption, $r_{ii}> 0$ and $s_{ii}> 0$ for $i\in \{1,2,\ldots,n\}$, meaning that $r_{11}> 0$ and $s_{11}> 0$, so $v_{11}$ must be positive 1. Since $\bV$ is an orthogonal matrix, we also have 
$$
v_{12}=v_{13}=v_{14}=\ldots=v_{1m}=0. 
$$
By applying this reasoning to the submatrices of $\bR_1, \bV, $ and $\bR_2$, we find that the upper-left submatrix of $\bV$ is the identity: $\bV[1:n,1:n]=\bI_n$, indicating $\bR_1=\bR_2$. This implies $\bQ_1[:,1:n]=\bQ_2[:,1:n]$ and leads to a contradiction. Therefore, the reduced QR decomposition is unique.
\end{proof}

\index{Row space}
\section{LQ, RRLQ, QL, RQ Decomposition}\label{section:lq_ql_rq}
We previously established the existence of the QR decomposition using the Gram--Schmidt process, which is applied to the column space of a matrix $\bA=[\ba_1, \ba_2, \ldots, \ba_n] \in \real^{m\times n}$. 
However, in many applications (see, for example, \citet{schilders2009solution} or Chapter~\ref{chapter:ulv-urv-decomposition}), there is also interest in the row space of a matrix 
$\bB=[\bb_1^\top; \bb_2^\top; \ldots;\bb_m^\top] \in \real^{m\times n}$, where $\bb_i$ denotes the $i$-th row of $\bB$. The successive spaces spanned by the rows $\bb_1, \bb_2, \ldots$ of $\bB$ are
$$
\cspace([\bb_1])\,\,\,\, \subseteq\,\,\,\, \cspace([\bb_1, \bb_2]) \,\,\,\,\subseteq\,\,\,\, \cspace([\bb_1, \bb_2, \bb_3])\,\,\,\, \subseteq\,\,\,\, \ldots.
$$
The QR decomposition has an analogous  counterpart that characterizes the orthogonal row space. 
If we perform the QR decomposition on $\bB^\top = \bQ_0\bR$, we obtain the LQ decomposition of the matrix $\bB = \bL \bQ$, where $\bQ = \bQ_0^\top$ and $\bL = \bR^\top$.
The LQ decomposition is useful in demonstrating the existence of the UTV decomposition in the following chapter.

\begin{theoremHigh}[LQ decomposition]\label{theorem:lq-decomposition}
Any $m\times n$ matrix $\bB$ (whether its rows are linearly independent or not) with $n\geq m$ can be decomposed as 
$$
\bB = \bL\bQ,
$$
where 
\begin{enumerate}
\item \textbf{Reduced}: $\bL$ is an $m\times m$ lower triangular matrix, and $\bQ$ is $m\times n$ with orthonormal rows, known as the \textit{reduced LQ decomposition};

\item \textbf{Full}: $\bL$ is an $m\times n$ lower triangular matrix, and $\bQ$ is $n\times n$ with orthonormal rows, known as the \textit{full LQ decomposition}. If we further restrict the lower triangular matrix to be square, the full LQ decomposition can be written as:
$$
\bB = \begin{bmatrix}
	\bL_0 & \bzero
\end{bmatrix}\bQ,
$$
where $\bL_0$ is an $m\times m$ square lower triangular matrix.
\end{enumerate}
\end{theoremHigh}

\paragraph{Row-pivoted LQ (RPLQ).\index{Row-pivoted}\index{RPLQ}} Additionally, similar to the column-pivoted QR discussed in Section~\ref{section:cpqr}, there exists a \textit{row-pivoted LQ (RPLQ)} decomposition  for a rank-$r$ matrix $\bB\in\real^{m\times n}$: 
$$
\left\{
\begin{aligned}
	\text{Reduced RPLQ: }&\qquad 
	\bP\bB &=& 
	\underbrace{\begin{bmatrix}
			\bL_{11} \\
			\bL_{21}
	\end{bmatrix}}_{m\times r}
	\underbrace{\bQ_r }_{r\times n};\\
	\text{Full RPLQ: }&\qquad 
	\bP\bB &=& 
	\underbrace{\begin{bmatrix}
			\bL_{11} & \bzero \\
			\bL_{21} & \bzero 
	\end{bmatrix}}_{m\times m}
	\underbrace{\bQ }_{m\times n},\\
\end{aligned}
\right.
$$
where $\bL_{11}\in \real^{r\times r}$ is lower triangular, $\bQ_r$ or $\bQ_{1:r,:}$ spans the same row space as $\bB$, and $\bP$ is a permutation matrix that interchanges independent rows into the upper-most rows.

\begin{exercise}[QL and RQ decomposition]
Let $\bA\in\real^{m\times n}$ with $m\geq n$. Show that there exists a permutation matrix $\bP$ such that $\bA\bP=\bQ\bL$, where $\bQ$ is orthogonal and $\bL$ is lower triangular.
Similarly, let $\bB\in\real^{m\times n}$ with $m\leq n$. Show that there exists a permutation matrix $\bP$ such that $\bP\bB=\bR\bQ$, where $\bQ$ is orthogonal and $\bR$ is upper triangular.
\end{exercise}

\section{Two-Sided Orthogonal Decomposition}
To this end, an immediate consequence of the CPQR and RPLQ decompositions is the \textit{two-sided orthogonal decomposition}, which simultaneously identifies orthonormal bases for both the row space and column space of a matrix.
\begin{theoremHigh}[Two-sided orthogonal decomposition]\label{theorem:two-sided-orthogonal}
Let $\bA\in \real^{n\times n}$ be  a square matrix of rank $r$. Suppose the full CPQR and RPLQ decompositions of $\bA$ are given by 
$$\bA\bP_1=\bQ_1
\begin{bmatrix}
	\bR_{11} & \bR_{12}\\
	\bzero & \bzero
\end{bmatrix}
\qquad \text{and}\qquad
\bP_2\bA=
\begin{bmatrix}
	\bL_{11} & \bzero \\
	\bL_{21} & \bzero 
\end{bmatrix}
\bQ_2,
$$ 
respectively. Then, we obtain:
$$
\bA\bP\bA = \bQ_1
\underbrace{\begin{bmatrix}
		\bR_{11}\bL_{11}+\bR_{12}\bL_{21} & \bzero \\
		\bzero & \bzero 
\end{bmatrix}}_{\text{rank $r$}}
\bQ_2,
$$
where the first $r$ columns of $\bQ_1$ span the same column space as $\bA$, the first $r$ rows of $\bQ_2$ span the same row space as $\bA$, and $\bP$ is a permutation matrix. This decomposition is known as  the \textit{two-sided orthogonal decomposition}.
\end{theoremHigh}

This decomposition exhibits a similarity with the singular value decomposition (SVD): $\bA=\bU\bSigma\bV^\top$, where the first $r$ columns of $\bU$ span the same column space as $\bA$, and the first $r$ columns of $\bV$ span the same row space as $\bA$ (as will be shown in Proposition~\ref{proposition:svd-four-orthonormal-Basis}). Thus, the two-sided orthogonal decomposition can be viewed as a computationally inexpensive alternative to the SVD for square matrices. The orthonormal bases from the two-sided orthogonal decomposition are collected in the following proposition.

\index{Fundamental spaces}
\index{Orthonormal basis}
\begin{proposition}[Four orthonormal basis]
Given the two-sided orthogonal decomposition of the matrix $\bA\in \real^{n\times n}$ with rank $r$: $\bA\bP\bA = \bU \bF\bV^\top$, where $\bU=[\bu_1, \bu_2, \ldots,\bu_n]$ and $\bV=[\bv_1, \bv_2, \ldots, \bv_n]$ are the column partitions of $\bU$ and $\bV$, respectively, the following properties hold:
\begin{itemize}
\item  $\{\bv_1, \bv_2, \ldots, \bv_r\} $ is an orthonormal basis of $\cspace(\bA^\top)$;

\item $\{\bv_{r+1},\bv_{r+2}, \ldots, \bv_n\}$ is an orthonormal basis of $\nspace(\bA)$;

\item $\{\bu_1,\bu_2, \ldots,\bu_r\}$ is an orthonormal basis of $\cspace(\bA)$;

\item $\{\bu_{r+1}, \bu_{r+2},\ldots,\bu_n\}$ is an orthonormal basis of $\nspace(\bA^\top)$. 
\end{itemize}
\end{proposition}

\index{Rank-one update}
\index{Rank-one change}
\section{Application: Rank-One Changes}\label{section:qr-rank-one-changes}
In Section~\ref{section:cholesky-rank-one-update}, we discussed the rank-one update and downdate of the Cholesky decomposition.
In the context of least squares problems, the Cholesky decomposition (along with its rank-one update/downdate) is applied to the Gram matrix of the data matrix $\bA\in\real^{m\times n}$: $\bA^\top\bA = \bR^\top\bR$; see Section~\ref{section:application-ls-qr}. Each row of the data matrix represents a data point, while each column corresponds to a feature variable.
Therefore, the rank-one update/downdate of the Choleaky decomposition is useful for efficiently adding or removing a small number of data points from the least squares problem.

Similarly, a rank-one change of a matrix $\bA$ in the QR decomposition is defined as follows:
$$
\begin{aligned}
	\bA^\prime &= \bA + \bu\bv^\top, \\
	\downarrow &\gap  \downarrow\\
	\bQ^\prime\bR^\prime &=\bQ\bR + \bu\bv^\top,
\end{aligned}
$$
where the downdate form can be recovered by setting $\bA^\prime = \bA - (-\bu)\bv^\top$. This shows that the update and downdate forms in the QR decomposition are essentially equivalent.
Since the least squares problem using the QR decomposition is applied directly to the data matrix  ($\bA=\bQ\bR$; see Section~\ref{section:application-ls-qr}),  rank-one changes to the QR decomposition cannot be be applied for adding or deleting a few data points from the least squares problem. However, it can be used for modifying the features in the system.
For example, when $\bu=\bone_m$ and $\bv=\bone_n$, the rank-one change effectively adds one to each feature across all data points.

To restate, the rank-one update/downdate problem involves determining the QR decomposition of $\bA^\prime$ efficiently, given the QR decomposition of $\bA$.
Let $\bw = \bQ^\top\bu$. Then we have
$$
\bA^\prime = \bQ(\bR + \bw\bv^\top).
$$
Using the second form from Remark~\ref{remark:basis-from-givens2}, which introduces zeros in reverse order, there exists a set of Givens rotations $\bG_{12} \bG_{23} \ldots \bG_{(n-1),n}$  such that
$$
\bG_{12} \bG_{23} \ldots \bG_{(n-1),n} \bw = \pm \norm{\bw} \be_1,
$$
where $\bG_{(k-1),k}$ denotes the Givens rotation in the plane corresponding to indices $k-1$ and $k$, and it zeros out the $k$-th entry of $\bw$. Applying these same rotations to $\bR$, we have
$$
\bG_{12} \bG_{23} \ldots \bG_{(n-1),n}\bR = \bH_0 ,
$$
where the Givens rotations in this \textit{reverse order}  (\textit{backward rotations}) are useful to transform the upper triangular $\bR$ into a ``simple" \textit{upper Hessenberg matrix}---a nearly triangular matrix (see Definition~\ref{definition:upper-hessenbert} that  will be introduced in the Hessenberg decomposition). 
In contrast, using forward rotations to transform $\bw$ into $\pm \norm{\bw}\be_1$, as demonstrated in Corollary~\ref{corollary:basis-from-givens}, does not yield an upper Hessenberg matrix. 
Instead, it typically produces a dense matrix. 
For example, considering $\bR\in \real^{4\times 4}$, backward rotations preserve many zeros, simplifying the structure, where $\boxtimes$ represents a value that is not necessarily zero, and \textbf{boldface} indicates the value has just been changed:
$$
\begin{aligned}
\text{\parbox{5.5em}{Backward\\(Right Way)}: }
\footnotesize
\begin{sbmatrix}{\bR}
\boxtimes & \boxtimes & \boxtimes & \boxtimes \\
0 & \boxtimes & \boxtimes & \boxtimes \\
0 & 0 & \boxtimes & \boxtimes \\
0 & 0 & 0 & \boxtimes \\
\end{sbmatrix}
\stackrel{\bG_{34}}{\rightarrow}
\footnotesize
\begin{sbmatrix}{\bG_{34}\bR}
\boxtimes & \boxtimes & \boxtimes & \boxtimes \\
0 & \boxtimes & \boxtimes & \boxtimes \\
0 & 0 & \bm{\boxtimes} & \bm{\boxtimes} \\
0 & 0 & \bm{\boxtimes} & \bm{\boxtimes} \\
\end{sbmatrix}
\stackrel{\bG_{23}}{\rightarrow}
\footnotesize
\begin{sbmatrix}{\bG_{23}\bG_{34}\bR}
\boxtimes & \boxtimes & \boxtimes & \boxtimes \\
0 & \bm{\boxtimes} & \bm{\boxtimes} & \bm{\boxtimes} \\
0 & \bm{\boxtimes} & \bm{\boxtimes} & \bm{\boxtimes} \\
0 & 0 & \boxtimes & \boxtimes \\
\end{sbmatrix}
\stackrel{\bG_{12}}{\rightarrow}
\footnotesize
\begin{sbmatrix}{\bG_{12}\bG_{23}\bG_{34}\bR}
\bm{\boxtimes} & \bm{\boxtimes} & \bm{\boxtimes} & \bm{\boxtimes} \\
\bm{\boxtimes} & \bm{\boxtimes} & \bm{\boxtimes} & \bm{\boxtimes} \\
0 & \boxtimes & \boxtimes & \boxtimes \\
0 & 0 & \boxtimes & \boxtimes \\
\end{sbmatrix}.
\end{aligned}
$$
While forward rotations eliminate these zeros, resulting in a \textbf{dense matrix}:
$$
\begin{aligned}
\text{\parbox{5.5em}{Forward\\(Wrong Way)}: }
\footnotesize
\begin{sbmatrix}{\bR}
\boxtimes & \boxtimes & \boxtimes & \boxtimes \\
0 & \boxtimes & \boxtimes & \boxtimes \\
0 & 0 & \boxtimes & \boxtimes \\
0 & 0 & 0 & \boxtimes \\
\end{sbmatrix}
&\stackrel{\bG_{12}}{\rightarrow}
\footnotesize
\begin{sbmatrix}{\bG_{12}\bR}
\bm{\boxtimes} & \bm{\boxtimes} & \bm{\boxtimes} & \bm{\boxtimes} \\
\bm{\boxtimes} & \bm{\boxtimes} & \bm{\boxtimes} & \bm{\boxtimes} \\
0 & 0 & \boxtimes & \boxtimes \\
0 & 0 & 0 & \boxtimes \\
\end{sbmatrix}
\stackrel{\bG_{23}}{\rightarrow}
\footnotesize
\begin{sbmatrix}{\bG_{23}\bG_{12}\bR}
\boxtimes & \boxtimes & \boxtimes & \boxtimes \\
\bm{\boxtimes} & \bm{\boxtimes} & \bm{\boxtimes} & \bm{\boxtimes} \\
\bm{\boxtimes} & \bm{\boxtimes} & \bm{\boxtimes} & \bm{\boxtimes} \\
0 & 0 & 0 & \boxtimes \\
\end{sbmatrix}
\stackrel{\bG_{34}}{\rightarrow}
\footnotesize
\begin{sbmatrix}{\bG_{34}\bG_{23}\bG_{12}\bR}
\boxtimes & \boxtimes & \boxtimes & \boxtimes \\
\boxtimes & \boxtimes & \boxtimes & \boxtimes \\
\bm{\boxtimes} & \bm{\boxtimes} & \bm{\boxtimes} & \bm{\boxtimes} \\
\bm{\boxtimes} & \bm{\boxtimes} & \bm{\boxtimes} & \bm{\boxtimes} \\
\end{sbmatrix}.
\end{aligned}
$$
In general, backward rotations yield:
$$
\bG_{12} \bG_{23} \ldots \bG_{(n-1),n} (\bR+\bw\bv^\top) = \bH_0  \pm \norm{\bw} \be_1 \bv^\top = \bH, 
$$
which is also upper Hessenberg. Subsequently, as in the triangularization process via Givens rotation in Section~\ref{section:qr-givens}, there exists a set of rotations $\bJ_{12}, \bJ_{23}, \ldots, \bJ_{(n-1),n}$ such that 
$$
\bJ_{(n-1),n} \ldots \bJ_{23}\bJ_{12}\bH = \bR^\prime
$$
is upper triangular. 
To illustrate this process, following the example of a $4\times 4$ matrix, the sequence of rotations progressively simplifies the matrix, preserving and introducing zeros step by step:
$$
\begin{aligned}
\underbrace{\bH_0  \pm \norm{\bw} \be_1 \bv^\top}_{\bH} =
\footnotesize
\begin{sbmatrix}{\bH}
\boxtimes & \boxtimes & \boxtimes & \boxtimes  \\
\boxtimes & \boxtimes & \boxtimes & \boxtimes  \\
0 & \boxtimes & \boxtimes & \boxtimes \\
0 & 0 & \boxtimes & \boxtimes \\
\end{sbmatrix}
&\stackrel{\bJ_{12}}{\rightarrow}
\footnotesize
\begin{sbmatrix}{\bJ_{12}\bH}
\bm{\boxtimes} & \bm{\boxtimes} & \bm{\boxtimes} & \bm{\boxtimes}  \\
\bm{0} & \bm{\boxtimes} & \bm{\boxtimes} & \bm{\boxtimes}  \\
0 & \boxtimes & \boxtimes & \boxtimes \\
0 & 0 & \boxtimes & \boxtimes \\
\end{sbmatrix}
\stackrel{\bJ_{23}}{\rightarrow}
\footnotesize
\begin{sbmatrix}{\bJ_{23}\bJ_{12}\bH}
\boxtimes & \boxtimes & \boxtimes & \boxtimes  \\
0 & \bm{\boxtimes} & \bm{\boxtimes} & \bm{\boxtimes}  \\
0 & \bm{0} & \bm{\boxtimes} & \bm{\boxtimes}  \\
0 & 0 & \boxtimes & \boxtimes \\
\end{sbmatrix}
\stackrel{\bJ_{34}}{\rightarrow}
\footnotesize
\begin{sbmatrix}{\bJ_{34}\bJ_{23}\bJ_{12}\bH}
\boxtimes & \boxtimes & \boxtimes & \boxtimes  \\
0 & \boxtimes & \boxtimes & \boxtimes  \\
0 & 0 & \bm{\boxtimes} & \bm{\boxtimes}  \\
0 & 0 & \bm{0} & \bm{\boxtimes}  \\
\end{sbmatrix}
.
\end{aligned}
$$
The QR decomposition of $\bA^\prime$ can then be expressed as:
$
\bA^\prime = \bQ^\prime \bR^\prime,
$
where 
\begin{equation}\label{equation:qr-rank-one-update}
\left\{
\begin{aligned}
\bR^\prime &=(\bJ_{(n-1),n} \ldots \bJ_{23}\bJ_{12}) (\bG_{12} \bG_{23} \ldots \bG_{(n-1),n}) (\bR+\bw\bv^\top);\\
\bQ^\prime &= \bQ\left\{(\bJ_{(n-1),n} \ldots \bJ_{23}\bJ_{12}) (\bG_{12} \bG_{23} \ldots \bG_{(n-1),n}) \right\}^\top .\\
\end{aligned}
\right.
\end{equation}

\section{Application: Appending or Deleting a Column}\label{section:append-column-qr}
In certain applications, such as an $F$-test for least squares via QR decomposition \citep{lu2021rigorous}, it is often necessary to either delete or append a column (i.e., a feature variable) from the observed matrix. 
The goal, once again, is to efficiently compute the QR decomposition of the modified matrix.
\paragraph{Deleting a column.}
Suppose the QR decomposition of a matrix $\bA\in \real^{m\times n}$ is given by $\bA=\bQ\bR$, where the columns  of $\bA$ are partitioned as  $\bA=[\ba_1,\ba_2,\ldots,\ba_n]$. If the $k$-th column of $\bA$ is removed, the resulting matrix is denoted by $\bA^\prime = [\ba_1,\ldots,\ba_{k-1},\ba_{k+1},\ldots,\ba_n] \in \real^{m\times (n-1)}$. Our goal is to efficiently compute the QR decomposition of $\bA^\prime$. 
The matrix $\bR$ can be expressed using the following block structure:
$$
\begin{aligned}
	\begin{blockarray}{ccccc}
		\begin{block}{c[ccc]c}
			&	\bR_{11} & \ba & \bR_{12} & k-1  \\
			\bR = 	&  \bzero    &	r_{kk} & \bb^\top  & 1 \\
			&	\bzero  &\bzero& \bR_{22} & m-k  \\
		\end{block}
		& k-1 & 1 &  n-k & \\
	\end{blockarray},\\
\end{aligned}
\quad\text{such that}\quad 
\bQ^\top \bA^\prime = 
\begin{bmatrix}
	\bR_{11} &\bR_{12} \\
	\bzero & \bb^\top \\
	\bzero & \bR_{22}
\end{bmatrix} = \bH
$$
is upper Hessenberg. 
An illustrative example is presented below, demonstrating the scenario of a $6\times 5$ matrix. Here, $k=3$, and the column corresponding to $k$ has been removed:
$$
\begin{aligned}
\footnotesize
\begin{sbmatrix}{\bR = \bQ^\top\bA}
\boxtimes & \boxtimes & \boxtimes & \boxtimes & \boxtimes \\
0 & \boxtimes & \boxtimes & \boxtimes & \boxtimes \\
0 & 0 & \boxtimes & \boxtimes& \boxtimes \\
0 & 0 & 0 & \boxtimes& \boxtimes \\
0 & 0 & 0 & 0& \boxtimes \\
0 & 0 & 0 & 0& 0
\end{sbmatrix}
&\longrightarrow
\footnotesize
\begin{sbmatrix}{\bH = \bQ^\top\bA^\prime}
\boxtimes & \boxtimes  & \boxtimes & \boxtimes \\
0 & \boxtimes  & \boxtimes & \boxtimes \\
0 & 0 & \boxtimes& \boxtimes \\
0 & 0  & \boxtimes& \boxtimes \\
0 & 0  & 0& \boxtimes \\
0 & 0  & 0& 0
\end{sbmatrix}.
\end{aligned}
$$ 

To transform $\bH$ into a triangular matrix, we apply a sequence of Givens rotations $\bG_{k,k+1}$, $\bG_{k+1,k+2}$, $\ldots$, $\bG_{n-1,n}$ to eliminate specific off-diagonal entries $h_{k+1,k}$, $h_{k+2,k+1}$, $\ldots$, $h_{n,n-1}$ of $\bH$.
The resulting triangular matrix $\bR^\prime$ is then computed as:
$$
\bR^\prime = \bG_{n-1,n}\ldots \bG_{k+1,k+2}\bG_{k,k+1}\bQ^\top \bA^\prime.
$$
The updated orthogonal matrix is given by:
\begin{equation}\label{equation:qr-delete-column-finalq}
	\bQ^\prime = (\bG_{n-1,n}\ldots \bG_{k+1,k+2}\bG_{k,k+1}\bQ^\top )^\top = \bQ \bG_{k,k+1}^\top  \bG_{k+1,k+2}^\top \ldots \bG_{n-1,n}^\top,
\end{equation}
such that $\bA^\prime = \bQ^\prime\bR^\prime$. 
The $6\times 5$ example is shown below, where $\boxtimes$ represents a value that is not necessarily zero, and \textbf{boldface} indicates the value has just been changed:
$$
\begin{aligned}
\footnotesize
\begin{sbmatrix}{\bR = \bQ^\top\bA}
	\boxtimes & \boxtimes & \boxtimes & \boxtimes & \boxtimes \\
	0 & \boxtimes & \boxtimes & \boxtimes & \boxtimes \\
	0 & 0 & \boxtimes & \boxtimes& \boxtimes \\
	0 & 0 & 0 & \boxtimes& \boxtimes \\
	0 & 0 & 0 & 0& \boxtimes \\
	0 & 0 & 0 & 0& 0
\end{sbmatrix}
&\stackrel{k=3}{\rightarrow}
\footnotesize
\begin{sbmatrix}{\bH = \bQ^\top\bA^\prime}
	\boxtimes & \boxtimes  & \boxtimes & \boxtimes \\
	0 & \boxtimes  & \boxtimes & \boxtimes \\
	0 & 0 & \boxtimes& \boxtimes \\
	0 & 0  & \boxtimes& \boxtimes \\
	0 & 0  & 0& \boxtimes \\
	0 & 0  & 0& 0
\end{sbmatrix}
\stackrel{\bG_{34}}{\rightarrow}
\footnotesize
\begin{sbmatrix}{\bG_{34}\bH }
	\boxtimes & \boxtimes  & \boxtimes & \boxtimes \\
	0 & \boxtimes  & \boxtimes & \boxtimes \\
	0 & 0 & \bm{\boxtimes}& \bm{\boxtimes} \\
	0 & 0  & \bm{0}& \bm{\boxtimes} \\
	0 & 0  & 0& \boxtimes \\
	0 & 0  & 0& 0
\end{sbmatrix}
\stackrel{\bG_{45}}{\rightarrow}
\footnotesize
\begin{sbmatrix}{\bG_{45}\bG_{34}\bH}
	\boxtimes & \boxtimes  & \boxtimes & \boxtimes \\
	0 & \boxtimes  & \boxtimes & \boxtimes \\
	0 & 0 & \boxtimes & \boxtimes \\
	0 & 0  &0 & \bm{\boxtimes} \\
	0 & 0  & 0& \bm{0} \\
	0 & 0  & 0& 0
\end{sbmatrix}.
\end{aligned}
$$

\paragraph{Appending a column.}
Similarly, consider the case where a vector $\bw$ is appended as the $(k+1)$-th column of $\bA$, resulting in the updated matrix $\widetildebA = [\ba_1,\ba_k,\bw,\ba_{k+1},\ldots,\ba_n]$. 
The goal becomes to efficiently compute the QR decomposition of $\widetildebA$.
Applying the orthogonal transformation $\bQ^\top$ to $\widetildebA$, we have
$$
\bQ^\top \widetildebA = [\bQ^\top\ba_1,\ldots, \bQ^\top\ba_k, \bQ^\top\bw, \bQ^\top\ba_{k+1}, \ldots,\bQ^\top\ba_n] = \widetilde{\bH}.
$$
Next, a sequence of Givens rotations $\bJ_{m-1,m}, \bJ_{m-2,m-1}, \ldots, \bJ_{k+1,k+2}$ can be applied to zero out the elements $\widetilde{h}_{m,k+1}$, $\widetilde{h}_{m-1,k+1}$, $\ldots$, $\widetilde{h}_{k+2,k+1}$ in $\widetilde{\bH}$, transforming it into an upper triangular matrix:
$$
\widetilde{\bR} =\bJ_{k+1,k+2}\ldots \bJ_{m-2,m-1}\bJ_{m-1,m} \bQ^\top \widetildebA.
$$
To illustrate, suppose $\widetilde{\bH}$ is a $6\times 5$ matrix, and $k=2$. Then the process is shown as follows:
$$
\footnotesize
\begin{aligned}
\begin{sbmatrix}{\widetilde{\bH}}
\boxtimes & \boxtimes & \boxtimes & \boxtimes & \boxtimes \\
0 & \boxtimes & \boxtimes & \boxtimes & \boxtimes \\
0 & 0 & \boxtimes & \boxtimes& \boxtimes \\
0 & 0 & \boxtimes & 0& \boxtimes \\
0 & 0 & \boxtimes & 0& 0 \\
0 & 0 & \boxtimes & 0& 0
\end{sbmatrix}
&\stackrel{\bJ_{56}}{\rightarrow}
\begin{sbmatrix}{\bJ_{56}\widetilde{\bH} \rightarrow \widetilde{h}_{63}=0}
\boxtimes & \boxtimes & \boxtimes & \boxtimes & \boxtimes \\
0 & \boxtimes & \boxtimes & \boxtimes & \boxtimes \\
0 & 0 & \boxtimes & \boxtimes& \boxtimes \\
0 & 0 & \boxtimes & 0& \boxtimes \\
0 & 0 & \bm{\boxtimes} & 0& 0 \\
0 & 0 & \bm{0} & 0& 0
\end{sbmatrix}
\stackrel{\bJ_{45}}{\rightarrow}
\begin{sbmatrix}{\bJ_{45}\bJ_{56}\widetilde{\bH}\rightarrow \widetilde{h}_{53}=0}
\boxtimes & \boxtimes & \boxtimes & \boxtimes & \boxtimes \\
0 & \boxtimes & \boxtimes & \boxtimes & \boxtimes \\
0 & 0 & \boxtimes & \boxtimes& \boxtimes \\
0 & 0 & \bm{\boxtimes} & 0& \bm{\boxtimes} \\
0 & 0 & \bm{0} & 0& \bm{\boxtimes}\\
0 & 0 & 0 & 0& 0
\end{sbmatrix}\stackrel{\bJ_{34}}{\rightarrow}
\begin{sbmatrix}{\bJ_{34}\bJ_{45}\bJ_{56}\widetilde{\bH}\rightarrow \widetilde{h}_{43}=0}
\boxtimes & \boxtimes & \boxtimes & \boxtimes & \boxtimes \\
0 & \boxtimes & \boxtimes & \boxtimes & \boxtimes \\
0 & 0 & \bm{\boxtimes}  & \bm{\boxtimes} & \bm{\boxtimes}  \\
0 & 0 & \bm{0}  & \bm{\boxtimes} & \bm{\boxtimes}  \\
0 & 0 & 0 & 0& \boxtimes\\
0 & 0 & 0 & 0& 0
\end{sbmatrix}.
\end{aligned}
$$ 
Finally, the updated orthogonal matrix is given by:
\begin{equation}\label{equation:qr-add-column-finalq}
\widetilde{\bQ} = (\bJ_{k+1,k+2}\ldots \bJ_{m-2,m-1}\bJ_{m-1,m} \bQ^\top )^\top = \bQ \bJ_{m-1,m}^\top  \bJ_{m-2,m-1}^\top \ldots \bJ_{k+1,k+2}^\top,
\end{equation}
such that $\widetildebA = \widetilde{\bQ}\widetilde{\bR}$.

\paragraph{Real world application.} This method is particularly valuable for efficient variable selection in least squares problems using QR decomposition. At each step, a column of the data matrix $\bA$ is removed, and an  $F$-test is performed to assess  the significance of the corresponding variable.  Variables that are statistically insignificant are removed, leading to a simpler and more interpretable model \citep{lu2021rigorous}.

\section{Application: Appending or Deleting a Row}\label{section:append-row-qr}
Analogously, in the context of least squares problems using the QR decomposition (see Section~\ref{section:application-ls-qr}), it may become necessary to append or delete a row (representing a data point) from the observed matrix.
This is often done to evaluate how the updated data affects system performance or to accommodate an online data setting, in which data arrives sequentially.
The objective, as before, is to efficiently compute the QR decomposition of the updated matrix.
\paragraph{Appending a row.}
Suppose the full QR decomposition of a matrix $\bA\in \real^{m\times n}$ is given by $\bA= \scriptsize \begin{bmatrix}
	\bA_1 \\
	\bA_2 
\end{bmatrix}=\bQ\bR$, where $\bA_1\in \real^{k\times n}$ and $\bA_2 \in \real^{(m-k)\times n}$. Now, if we append a row, the resulting matrix becomes  $\bA^\prime =\scriptsize \begin{bmatrix}
	\bA_1 \\
	\bw^\top\\
	\bA_2 
\end{bmatrix} \in \real^{(m+1)\times n}$. 
Our goal is to efficiently compute the full QR decomposition of $\bA^\prime$. 
To achieve this, we construct a permutation matrix:
$$
\bP=
\begin{bmatrix}
	\bzero & 1 & \bzero  \\
	\bI_k & \bzero & \bzero \\
	\bzero & \bzero & \bI_{m-k}
\end{bmatrix}
\longrightarrow
\bP 
\begin{bmatrix}
	\bA_1 \\
	\bw^\top\\
	\bA_2 
\end{bmatrix}
=
\begin{bmatrix}
	\bw^\top\\
	\bA_1 \\
	\bA_2 
\end{bmatrix}
\quad\implies\quad
\begin{bmatrix}
	1 & \bzero \\
	\bzero & \bQ^\top 
\end{bmatrix} 
\bP 
\bA^\prime 
=
\begin{bmatrix}
	\bw^\top \\
	\bR 
\end{bmatrix}=\bH,
$$
such that $\bH$ is upper Hessenberg. Similarly, a set of rotations $\bG_{12}, \bG_{23}, \ldots, \bG_{n,n+1}$ can be applied to introduce zeros in the elements $h_{21}$, $h_{32}$, $\ldots$, $h_{n+1,n}$ of $\bH$. The triangular matrix $\bR^\prime$ is given by 
$$
\bR^\prime = \bG_{n,n+1}\ldots \bG_{23}\bG_{12}\begin{bmatrix}
	1 & \bzero \\
	\bzero & \bQ^\top 
\end{bmatrix} 
\bP   \bA^\prime.
$$
The updated orthogonal matrix is then computed as
$$
\bQ^\prime = \left(\bG_{n,n+1}\ldots \bG_{23}\bG_{12}\begin{bmatrix}
	1 & \bzero \\
	\bzero & \bQ^\top 
\end{bmatrix} 
\bP \right)^\top 
= 
\bP^\top 
\begin{bmatrix}
	1 & \bzero \\
	\bzero & \bQ 
\end{bmatrix}  
\bG_{12}^\top  \bG_{23}^\top \ldots \bG_{n,n+1}^\top,
$$
such that $\bA^\prime = \bQ^\prime\bR^\prime$ gives the QR decomposition of the updated matrix $\bA^\prime$.

\paragraph{Deleting a row.} Suppose $\bA = \scriptsize\begin{bmatrix}
	\bA_1 \\
	\bw^\top\\
	\bA_2 
\end{bmatrix} \in \real^{m\times n}$, where $\bA_1\in\real^{k\times n}$, $\bA_2 \in \real^{(m-k-1)\times n}$, and the full QR decomposition is given by $\bA=\bQ\bR$, with $\bQ\in \real^{m\times m}$ being orthogonal and  $\bR\in \real^{m\times n}$ being upper triangular. We aim to compute the full QR decomposition of $\widetildebA = \scriptsize \begin{bmatrix}
	\bA_1 \\
	\bA_2 
\end{bmatrix}$ efficiently (assuming $m-1\geq n$). Similarly, to achieve this, we construct a permutation matrix $\bP$ as follows:
$$
\bP = 
\begin{bmatrix}
	\bzero & 1 & \bzero  \\
	\bI_k & \bzero & \bzero \\
	\bzero & \bzero & \bI_{m-k-1}
\end{bmatrix}
\implies
\bP\bA = 
\begin{bmatrix}
	\bzero & 1 & \bzero  \\
	\bI_k & \bzero & \bzero \\
	\bzero & \bzero & \bI_{m-k-1}
\end{bmatrix}
\begin{bmatrix}
	\bA_1 \\
	\bw^\top\\
	\bA_2 
\end{bmatrix}=
\begin{bmatrix}
	\bw^\top \\
	\bA_1\\
	\bA_2 
\end{bmatrix} = \bP\bQ\bR =\bM\bR ,
$$
where $\bM = \bP\bQ$ is an orthogonal matrix. Let $\bmm^\top$ denote the first row of $\bM$. A series of Givens rotations, $\bG_{m-1,m}, \bG_{m-2,m-1}, \ldots, \bG_{1,2}$, can be applied to zero out the elements $m_m, m_{m-1}, \ldots, m_2$ of $\bmm$, resulting in $\bG_{1,2}\ldots \bG_{m-2,m-1}\bG_{m-1,m}\bmm = \alpha \be_1$, where $\alpha = \pm 1$. Consequently,
$$
\bG_{1,2}\ldots \bG_{m-2,m-1}\bG_{m-1,m} \bR = 
	\begin{blockarray}{cc}
	\begin{block}{[c]c}
		\bv^\top & 1 \\
		\bR_1& m-1  \\
	\end{block}
\end{blockarray},
$$ 
which is upper Hessenberg with $\bR_1\in \real^{(m-1)\times n}$ being upper triangular. And 
$$
\bM   \bG_{m-1,m}^\top \bG_{m-2,m-1}^\top \ldots  \bG_{1,2}^\top  = 
\begin{bmatrix}
	\alpha & \bzero \\
	\bzero & \bQ_1 
\end{bmatrix},
$$
where $\bQ_1\in \real^{(m-1)\times (m-1)}$ is an orthogonal matrix. The bottom-left block of the above matrix is a zero vector because $\alpha=\pm 1$ and $\bM$ is orthogonal. To see this, let $\bG=\bG_{m-1,m}^\top \bG_{m-2,m-1}^\top $ $\ldots  \bG_{1,2}^\top $, with its first  column denoted as $\bg$. Writing $\bM$ as the row partition $\bM = [\bmm^\top; \bmm_2^\top; \bmm_3^\top; \ldots, \bmm_{m}^\top]$,  we have
$$
\begin{aligned}
	\bmm^\top\bg &= \pm 1 \qquad  \rightarrow \qquad  \bg = \pm \bmm, \\
	\bmm_i^\top \bmm &=0,  \qquad \forall i \in \{2,3,\ldots,m\}.
\end{aligned}
$$
Thus, we can write:
$$
\begin{aligned}
\bP\bA&=\bM\bR=(\bM   \bG_{m-1,m}^\top \bG_{m-2,m-1}^\top \ldots  \bG_{1,2}\top ) (\bG_{1,2}\ldots \bG_{m-2,m-1}\bG_{m-1,m} \bR ) \\
&=
\begin{bmatrix}
\alpha & \bzero \\
\bzero & \bQ_1 
\end{bmatrix}
\begin{bmatrix}
\bv^\top \\
\bR_1 
\end{bmatrix} = 
\begin{bmatrix}
\alpha \bv^\top \\
\bQ_1\bR_1 
\end{bmatrix}
=
\begin{bmatrix}
\bw^\top \\
\widetildebA
\end{bmatrix}
.
\end{aligned}
$$
This shows that $\bQ_1\bR_1$ is the full QR decomposition of $\widetildebA=\begin{bmatrixfoot}
\bA_1\\
\bA_2 
\end{bmatrixfoot}$.

\index{Newton's method}
\index{Gauss--Newton method}
\index{Nonlinear least squares}
\section{Application: Gauss--Newton and Levenberg--Marquardt Method}\label{section:qr_gausnew_lev}
The QR decomposition is helpful for solving the Gauss--Newton and Levenberg--Marquardt methods for nonlinear least squares problems.
In Section~\ref{section:application-ls-qr}, we will introduce the (linear) least squares problem for linear systems:
\begin{equation}\label{equation:ls_qr_triv} 
\min_{\bx} \frac{1}{2}\normtwo{\bA\bx-\bb}^2.
\end{equation} 
When the residual $\br(\bx)$ in Equation~\eqref{equation:ls_qr_triv} is nonlinear, we obtain the \textit{nonlinear least squares} problem~\footnote{More details can refer to, for example, \citet{madsen2004methods}.}:
$$
\bx^* = \mathop{\argmin}_{\bx} \left\{f(\bx) = \frac{1}{2} \normtwo{\br(\bx)}^2\right\}, 
\gap \br(\bx)\in\real^m, \,\bx\in\real^n, \, m\geq n.
$$
When $\br(\bx)=\bA\bx-\bb$, this reduces to the linear least squares problem given in \eqref{equation:ls_qr_triv}.
The gradient and Hessian of $f(\bx)$ are 
\begin{equation}\label{equation:gaus_new_jaco}
	\begin{aligned}
		\nabla f(\bx) &= \bJ(\bx)^\top \br(\bx)
		\quad\text{and} \quad
		\nabla^2 f(\bx) = \bJ(\bx)^\top\bJ(\bx) + \sum_{i=1}^{m} r_i(\bx)\nabla^2r_{i}(\bx),
	\end{aligned}
\end{equation}
where $\bJ(\bx)\in\real^{m\times n}$ is the Jacobian matrix (see Problem~\ref{problem:gaus_new_jco}).
The standard Newton's method (see Section~\ref{section:app_cho_md_newton}) is an iterative optimization algorithm. 
At the $t$-th iteration, the update is given by:
$$
\bx^{(t+1)} \leftarrow \bx^{(t)} + \bd^{(t)},
$$
where $(\nabla^2 f(\bx^{(t)}) )\bd^{(t)} = -\nabla f(\bx^{(t)})$ determines  the ``candidate" descent direction $\bd^{(t)}$.
For brevity, we omit the superscript $t$ and 
apply a linear Taylor's approximation:
$$
\nabla f(\bx + \bd) 
\approx
\nabla f(\bx ) +\nabla^2 f(\bx )^\top \bd.
$$
Therefore, Newton's method can be interpreted  as finding a direction $\bd$ such that $\nabla f(\bx + \bd) =\nabla f(\bx ) +\nabla^2 f(\bx )^\top \bd$ approaches $\bzero$ (i.e., a \textit{stationary point}).
To see this, taking the quadratic Taylor's approximation, we have 
\begin{equation}\label{equation:ng_secondapp}
	f(\bx + \bd) 
	= 
	f(\bx ) +\nabla f(\bx )^\top \bd + \frac{1}{2}\bd^\top \nabla^2f(\bx) \bd + o(\normtwo{\bd}^2).
\end{equation}
If $\bx$ is a stationary point, then $\nabla f(\bx )=\bzero$. Suppose further that the Hessian of $f(\bx)$ is positive definite: $\nabla^2f(\bx)\succ 0$; this implies  that the smallest eigenvalue $\lambda_{\min}$ of  $\nabla^2f(\bx)$ satisfies $\lambda_{\min}>0$ (see Section~\ref{section:recursi_choles}), and $\bd^\top \nabla^2f(\bx) \bd \geq \lambda \normtwo{\bd}^2$ for all $\lambda_{\min}>\lambda>0$. 
This in turn implies that the third term in \eqref{equation:ng_secondapp} dominates the fourth term. Therefore, $\bx$ is a \textit{local minimizer} (a minimum point within a neighborhood of $\bx$ with some radius $r$) when $\bx$ is a stationary point and $\nabla^2f(\bx)$ is positive definite (as long as $\norm{\bd}$ is small enough).

\paragraph{Gauss--Newton method.}
However, since  the Hessian $\nabla^2r_{i}(\bx)$ can be difficult to compute or intractable, the \textit{Gauss--Newton method}  approximates the Hessian $\nabla^2 f(\bx)$ using only $\bJ(\bx)^\top\bJ(\bx)$. 
This leads to the following equation for determining the ``candidate" descent direction
$$
\bJ(\bx^{(t)})^\top\bJ(\bx^{(t)})\bd^{(t)} = -\bJ(\bx^{(t)})^\top \br(\bx^{(t)}).
$$
The ``candidate" descent direction can also be equivalently obtained by solving the following optimization problem:
\begin{equation}\label{equation:gasnewton}
	\textbf{(Gauss--Newton):}\gap \bd^{(t)} = \mathop{\argmin}_{\bd} \normtwo{\bJ(\bx^{(t)})\bd+\br(\bx^{(t)})}^2,
\end{equation}
which is a linear least squares problem and can be solved using QR decomposition  (see Theorem~\ref{theorem:qr-for-ls}, when $\bJ(\bx^{(t)})$ has full rank).
Let $\bJ(\bx^{(t)})$ admit the reduced QR decomposition $\bJ(\bx^{(t)})=\bQ^{(t)}\bR^{(t)}$. Then the ``candidate" descent direction can be obtained by 
$$
\bd^{(t)} \leftarrow -(\bR^{(t)})^{-1}(\bQ^{(t)})^\top\br(\bx^{(t)}).
$$
This approach avoids the need to explicitly compute the inverse of $\bJ(\bx^{(t)})^\top\bJ(\bx^{(t)})$.
When $(\bd^{(t)})^\top\nabla f(\bx^{(t)})\leq 0$, the  direction $\bd^{(t)}$ is called a ``true" descent direction (as opposed to the ``candidate" descent direction we used previously).
We can verify that when $\bJ(\bx^{(t)})$ has full rank $n$ (since $m\geq n$), we have 
$$
(\bd^{(t)})^\top\nabla f(\bx^{(t)}) = (\bd^{(t)})^\top \bJ(\bx^{(t)})^\top \br(\bx^{(t)}) = -\normtwo{\bJ(\bx^{(t)})\bd^{(t)} }^2\leq 0.
$$
Therefore, the resulting direction $\bd^{(t)}$ is indeed a ``true" descent direction.

\index{Levenberg--Marquardt method}
\index{Trust region method}
\index{KKT condition}
\paragraph{Levenberg--Marquardt (LM) method.}
Additionally, the \textit{Levenberg--Marquardt method} also addresses the same problem in \eqref{equation:gasnewton}, but introduces an additional constraint $\normtwo{\bd}\leq \Delta^{(t)}$ \citep{levenberg1944method, marquardt1963algorithm, wright1985inexact}:
\begin{equation}\label{equation:lmmethod}
	\textbf{(LM-1):}\gap \bd^{(t)} = \mathop{\argmin}_{\bd} \normtwo{\bJ(\bx^{(t)})\bd+\br(\bx^{(t)})}^2, \gap \text{s.t.}\gap \normtwo{\bd}\leq \Delta^{(t)}.
\end{equation}
This is equivalently to, using Lagrange  multiplier, the following problem
\begin{equation}\label{equation:lmmethod_lag}
	\begin{aligned}
		\textbf{(LM-2):}\gap 
		\bd^{(t)} &= \mathop{\argmin}_{\bd} \normtwo{\bJ(\bx^{(t)})\bd+\br(\bx^{(t)})}^2 + \lambda\normtwo{\bd}^2\\
		&= \mathop{\argmin}_{\bd} \normtwo{\begin{bmatrix}
				\bJ(\bx^{(t)})\\
				\sqrt{\lambda}\bI 
			\end{bmatrix}\bd
			+
			\begin{bmatrix}
				\br(\bx^{(t)})\\
				\bzero 
			\end{bmatrix}
		}^2,
	\end{aligned}
\end{equation}
where $\lambda$ is a Lagrange multiplier associated with the trust-region radius $\Delta^{(t)}$. The second form above represents an updated least squares problem. Given the knowledge of the QR decomposition of $\bJ(\bx^{(t)}) =\bQ^{(t)}\bR^{(t)}$, the least squares problem can be solved using the \textit{update of least squares problems} (i.e., appending rows to the existing data matrix; see Section~\ref{section:append-row-qr}).

\index{Low-rank approximation}
\section{Application: Low-Rank Approximation}\label{section:qr_lowrank}

We will discuss low-rank approximation or dimensionality reduction in more detail in Section~\ref{section:svd-low-rank-approxi} and Chapter~\ref{chapter:als}.
The QR decomposition of a data matrix (which may be triangular) can also be used to construct a low-rank approximation of that matrix.
In this context, the goal is to approximate a large matrix $\bA \in \real^{m \times n}$ with a low-rank matrix $\widetildebA$ of rank $k \ll \min(m, n)$.
This is particularly useful when $\bA$ is too large to store or process directly, or when the data in $\bA$ approximately lies in a lower-dimensional subspace. In such cases, we may wish to compress the data, reduce noise, or accelerate downstream computations.

To achieve this, we will introduce the \textit{truncated SVD} in Section~\ref{section:svd-low-rank-approxi}. Given the SVD of 
$
\bA = \bU \bSigma \bV^\top
$,
we keep only the top $k$ singular values: $
\bA \approx \widetildebA_k = \bU_k \bSigma_k \bV_k^\top,
$
where $\bU_k \in \real^{m \times k}$, $\bSigma_k \in \real^{k \times k}$, and $\bV_k \in \real^{n \times k}$.
This gives the best rank-$k$ approximation to $\bA$ (in terms of Frobenius or spectral norm). But it's computationally  expensive, requiring $\mathcalO(mn^2)$ operations.

In such cases, we seek a faster method to compute an approximate basis for the column space of $\bA$. 
One efficient approach is to use a randomized range finder combined with QR decomposition:
\begin{itemize}
\item \textit{Generate a random test matrix.}  
Let
$
\bOmega \in \real^{n \times k}
$
be a random Gaussian matrix or structured random matrix (e.g., subsampled Hadamard; see, for example, \citet{mahoney2016lecture}).

\item  \textit{Form a sample matrix.} Compute $\bY = \bA \bOmega \in \real^{m \times k}$, which projects $\bA$ onto random directions.  
This means  each column of $\bY$ is a random linear combination of the columns of $\bA$. If the top $k$-dimensional column space dominates, then $\bY$ will ``capture" most of it.

\item  \textit{Compute a reduced QR decomposition of $\bY$.}  
Let
$
\bY = \bQ \bR $,
where
$\bQ \in \real^{m \times k}$ with orthonormal columns (i.e., $\bQ^\top \bQ = \bI_k$), and
$\bR \in \real^{k \times k}$.
Now, the columns of $\bQ$ form an orthonormal basis for an approximate column space of $\bA$. This step is sometimes called \textit{orthonormalization of the sample space.}

\item \textit{Project $\bA$ onto the subspace spanned by $\bQ$.}
Compute $\widetildebA = \bQ\bQ^\top \bA$, where $\bQ\bQ^\top$ is an orthogonal projector onto the $k$-dimensional subspace spanned by $\bQ$ (see Section~\ref{section:qr-gram-compute}). That is, $\widetildebA \in \real^{m \times n}$ is a rank-$k$ approximation to $\bA$ \citep{drineas2006fast}.

\item  Optionally, compute a small matrix $\bB \in \real^{k \times n}$:
$$
\bB = \bQ^\top \bA
\quad\implies\quad
\widetildebA = \bQ \bB.
$$
Now we've reduced the problem to a small matrix $\bB$, making further computations (e.g. SVD, regression, classification, clustering) more efficient.
\end{itemize}

The randomized QR algorithm has a computational cost of $\mathcalO(mnk)$, which is faster than the truncated SVD algorithm, whose cost is $\mathcalO(mn^2)$ for approximating the matrix.
This efficiency can be very beneficial in practice.
For example, suppose we have a large document-term matrix $\bA \in \real^{100000 \times 10000}$ from some natural language processing (NLP) tasks.
We can compute $\bQ \in \real^{100000 \times 200}$ that captures the dominant 200-dimensional structure using randomized QR, and then work with $\bQ^\top\bA \in \real^{200 \times 10000}$ instead---greatly reducing both time and memory requirements.

\index{Adjugate}
\begin{problemset}
\item \label{prob:orthogo_proj} \textbf{Orthogonal projection.} Prove that an orthogonal projection $\bH$ is an idempotent and symmetric matrix such that $\bH\bv \perp (\bv-\bH\bv)$ and $\bH\bv\in\cspace(\bH)$ for any vector $\bv\notin \cspace(\bH)$.

\item \textbf{Adjugate of orthogonal.} Let $\bQ\in\real^{n\times n}$ be orthogonal. Show that $\adjugate(\bQ) = \det(\bQ)\bQ^\top$ such that $\adjugate(\bQ)$ is also orthogonal (Definition~\ref{definition:adjugate}).

\item Prove that if $\bA$ is triangular and orthogonal, then $\bA$ must be diagonal.

\item Let $\bA\in\real^{n\times n}$ be skew-symmetric ($\bA^\top=-\bA$). Show that the matrix $(\bI-\bA)^{-1}(\bI+\bA)$ is orthogonal.

\item Let  $\bu$ and $\bv$ be two orthogonal unit vectors. Show that $\bu+\bv$ is orthogonal to $\bu-\bv$.

\item \textbf{Reflector.} Let  $\bu\in\real^n$ and $\bv\in\real^n$ be two orthogonal  vectors (not necessarily unit), where $\bu\in\mathcalV$ and $\bv\in\mathcalV^\perp$. Define $\ba=\bu+\bv$ and $\bb=\bu-\bv$. Show that there exists a unique Householder reflector $\bH\in\real^{n\times n}$ (Definition~\ref{definition:householder-reflector}) such that $\bH\ba = \bb$.
Moreover, if $\mathcalV=\{\bw\}^\perp$, show that $\bH=\bI-2\frac{\bw\bw^\top}{\bw^\top\bw}$.

\item Let $\bu=[-\sin(\theta), \cos(\theta)]^\top$ be a unit vector. Show that the Householder reflector determined by $\bu$ is $\bH=\begin{bmatrixscript}
\cos(2\theta) & \sin (2\theta) \\
\sin(2\theta) & -\cos(2\theta) 
\end{bmatrixscript}$.

\item \label{prob:semi_eq1} Let $\bQ,\bU\in\real^{m\times n}$ be two semi-orthogonal matrices with $m\geq n$. Show that $\bQ$ and $\bU$ have the same column space if and only if there exists an orthogonal matrix $\bP\in\real^{n\times n}$ such that $\bQ=\bU\bP$.

\item  \label{prob:semi_eq2} Let $\bQ,\bU\in\real^{n\times n}$ be orthogonal. Show that there exists an orthogonal matrix $\bP$ such that $\bQ=\bP\bU$.

\item  Let $\bQ,\bU\in\real^{m\times n}$ be two semi-orthogonal matrices with $m\geq n$. Show that there exists an orthogonal matrix $\bP\in\real^{m\times m}$ such that $\bQ=\bP\bU$. Compare this result with Problems~\ref{prob:semi_eq1} and \ref{prob:semi_eq2}. \textit{Hint: Complete the semi-orthogonal matrices into $m\times m$ orthogonal matrices.}

\item Let $\bA$ admit the QR decomposition $\bA=\bQ\bR$. Show that $\bA=\bQ\bR$ is normal ($\bA^\top\bA=\bA\bA^\top$) if and  only if $\bR\bQ$ is normal.

\item \label{problem:part_ortho} Consider the partition of an orthogonal matrix
$
\bQ = \begin{bmatrixscript}
\underset{p\times p}{\bA} &\underset{p\times q}{\bB} \\
\underset{q\times p}{\bC} & \underset{q\times q}{\bD}
\end{bmatrixscript}\in\real^{n\times n}.
$
Show that $\rank(\bB)=\rank(\bC)$ and $\rank(\bD)=n+\rank(\bA)-2p$.

\item Consider the rank of matrices:
\begin{itemize}
\item Suppose matrices $\bA$ and $\bB$ have full column ranks. Show that $\bA\bB$ has full column rank.
\item Suppose $\bA\bB$ has full column ranks. Show that $\bB$ also has full column rank, but $\bA$ may not necessarily have full column rank.
\item Discuss the rank of the upper triangular matrices  obtained from the QR decompositions of $\bA\bB$, $\bA$, and $\bB$ in various cases of the matrices involved.
\end{itemize}

\item In Theorem~\ref{theorem:qr-decomposition}, we stated that
$\bR$ is nonsingular in the reduced QR decomposition when 
$\bA$ has full column rank $n$.
Suppose $\bA$ does not have full column rank. Examine  the relationship between the rank of $\bA$ and the number of nonzero entries in $\bR$.

\item Use the Gram--Schmidt process, Householder transformations, Givens rotations to find an orthonormal basis for the space spanned by the vectors
$$
\begin{aligned}
\bv_1 &= [1,3,7,5]^\top, \gap 
\bv_2 &= [6,3,6,3]^\top,  \gap 
\bv_3 &= [5,2,7,4]^\top.
\end{aligned}
$$

\item \textbf{Distance between a vector and a hyperplane.} Given a nonzero vector $\bzero\neq \ba\in\real^n$ and a scalar $\beta$,   define  the hyperplane $H(\ba, \beta) = \{\bx\in\real^n:\ba^\top\bx+\beta=0\}$. For any $\by\in\real^n$, use the projection along a line (see Section~\ref{section:project-onto-a-vector}) to show that the distance between $\by$ and $H(\ba, \beta)$ is given by $
d(\by, H(\ba, \beta)) = \frac{\abs{\ba^\top\by + \beta}}{\normtwo{\ba}}.$
\textit{Hint: Choose two random points on the plane and first show that $\ba$ is orthogonal to the plane.}

\item Although we have used the fact that every orthogonal (or orthonormal) list of vectors is linearly independent throughout our discussions, provide a rigorous proof of this claim. \textit{Hint: Assume the vectors are linearly dependent and derive a contradiction.}

\item Let $\bA\in\real^{m\times n}$ be given with $m\geq n$. Provide an algorithm using Householder reflectors to compute an orthogonal matrix $\bQ\in\real^{m\times m}$ such that $\bA=\bQ\bL$, where  $\bL[1:n, 1:n]$ is  lower triangular and $\bL[n + 1:m, :] = \bzero$.

\item Let $\bA\in\real^{n\times n}$ with rank $r$. Show that $\bA$ is range-symmetric (i.e., $\cspace(\bA)=\cspace(\bA^\top)$) if and only if there exist a nonsingular matrix $\bS\in\real^{n\times n}$ and a nonsingular matrix $\bM\in\real^{r\times r}$ such that 
$
\bA=\bS\begin{bmatrixscript}
	\bM & \bzero\\
	\bzero & \bzero 
\end{bmatrixscript}
\bS^\top .
$
\textit{Hint: Consider the QR decomposition of $\bS=\bQ\bR$.}

\item Let $\bA\in\real^{n\times n}$ be skew-symmetric (i.e., $\bA^\top=-\bA$). Show that $\bI+\bA$ is nonsingular,  $\bB=(\bI-\bA)(\bI+\bA)^{-1}$ is orthogonal, $\bI+\bB = 2(\bI+\bA)^{-1}$, and $\det(\bB)=1$.

\item Prove that the following statements about a square matrix $\bQ\in\real^{n\times n}$ are equivalent:
\begin{itemize}
	\item $\bQ$ is orthogonal.
	\item $\bQ^\top$ is orthogonal.
	\item $\bQ$ is nonsingular and $\bQ^\top=\bQ^{-1}$.
	\item The rows of $\bQ$ are orthogonormal.
	\item The columns of $\bQ$ are orthonormal.
	\item For all $\bx\in\real^n$, it follows that $\normtwo{\bx}=\normtwo{\bQ\bx}$.
\end{itemize}

\item \label{prob:ortho_prese1} \textbf{Orthogonal preservation.} Let $\bQ\in\real^{n\times n}$ be orthogonal. Show that $\bx,\by\in\real^n$ are orthogonal if and only if $\bQ\bx$ and $\bQ\by$ are orthogonal.

\item \label{prob:ortho_presen} \textbf{Orthogonal preservation.} Let $\bQ\in\real^{n\times n}$ be orthogonal, and  let $\lambda$ be an  eigenvalue of $\bQ$. Show that $\lambda=\pm 1$, and $\bx\in\real^n$ is a (right) eigenvector of $\bQ$ associated with $\lambda$ if and only if $\bx$ is a left eigenvector
of $\bQ$ associated with $\lambda$.

\item \label{prob:qr_inver} \textbf{Inverses with QR decomposition.} Suppose you perform QR decomposition of an invertible $n \times n$ matrix as $\bA = \bQ\bR$. Show how you can use this decomposition relationship for finding the inverse of $\bA$ by solving $n$ different triangular systems of linear equations, each of which can be solved by back-substitution. Show how to compute the left or right inverse of a matrix with QR decomposition and back-substitution.

\item Use the results from Problems~\ref{prob:compl1} and \ref{prob:compl2} to determine the computational complexity of  QR decomposition using the CGS, MGS, Householder,and  Givens approaches.

\item \textbf{Elementary row interchanging as a rotation and a reflection.} Prove that an $n \times n$ elementary row interchange matrix can be expressed as the product of a $90^\circ$ Givens rotation (i.e., of the form $\bG(i, j, \theta) = \bI + (\cos(\theta) - 1)(\be_i \be_i^\top + \be_j \be_j^\top) + \sin(\theta)(\be_i \be_j^\top - \be_j \be_i^\top)$ with $\theta=90^\circ$) and a Householder reflector. \textit{Hint: We need to understand the properties and forms of these matrices.}

\index{Givens geometric decomposition}
\item  \textbf{Givens geometric decomposition.} Show that all $n \times n$ orthogonal matrices can be written as  a product of at most $\mathcalO(n^2)$ Givens rotations and at most a single {elementary reflection} matrix (obtained by negating one diagonal element of the identity matrix).

\index{Householder geometric decomposition}
\item  \textbf{Householder geometric decomposition.} Show that all $n \times n$ orthogonal matrices can be written as a product of at most $n$ Householder reflectors.

\item Demonstrate that a sequence of $k$ Householder transformations, whose corresponding unit vectors are mutually orthonormal, can be represented as $\bI-2\bQ\bQ^\top$, where $\bQ$ is an $n\times k$ semi-orthogonal matrix. Identify the $(n - k)$-dimensional plane across which this reflection occurs.

\item Consider the $4 \times 4$ Givens rotation matrix $\bG_{2,4}(90^\circ)$ (Definition~\ref{definition:givens-rotation-in-qr}). This matrix performs a $90^\circ$ clockwise rotation of a 4-dimensional  vector in the plane of the second and fourth dimensions (see Figure~\ref{fig:rotation}). Show how to obtain this matrix {as the product of two Householder reflectors}. \textit{Hint: Think geometrically.}

\item Consider two orthogonal matrices 
$
\bQ_1 = \begin{bmatrixscript}
	-1 &0\\
0& -1
\end{bmatrixscript}
$
and 
$
\bQ_2 = \begin{bmatrixscript}
	1 &0\\
0 & -1
\end{bmatrixscript}.
$
Are these matrices rotation or reflection matrices?

\item Use Householder reflectors or Givens rotations to compute the LQ, QL, and RQ decompositions discussed in Section~\ref{section:lq_ql_rq}.

\index{Nonlinear least squares}
\item \label{problem:gaus_new_jco} Prove Equation~\eqref{equation:gaus_new_jaco}, the gradient and Hessian of nonlinear least squares problems. \textit{Hint: Derive element-wise: 
$$
\frac{\partial f(\bx)}{\partial x_j} = \sum_{i=1}^{m} r_i(\bx) \frac{\partial r_i (\bx)}{\partial x_j}, 
\gap
\frac{\partial^2 f(\bx)}{\partial x_j \partial x_k} =\sum_{i=1}^{m} 
\left( 
\frac{\partial r_i(\bx)}{\partial x_j}\frac{\partial r_i(\bx)}{\partial x_k}
+
r_i(\bx)\frac{\partial^2 r_i(\bx)}{\partial x_j\partial x_k}
\right).
$$}

\end{problemset}

\chapter{UTV Decomposition: ULV and URV Decomposition}\label{chapter:ulv-urv-decomposition}

\section{UTV Decomposition}\label{section:ulv-urv-decomposition}
The UTV decomposition generalizes the QR factorization of a matrix $\bA$ into two orthogonal matrices, $\bU$ and $\bV$, and a (upper or lower) triangular matrix $\bT$, such that $\bA=\bU\bT\bV$.~\footnote{These decompositions belong to a class known as  \textit{double-sided orthogonal decomposition}. We will see  the UTV decomposition, complete orthogonal decomposition (Theorem~\ref{theorem:complete-orthogonal-decom}), and singular value decomposition  are all instances of this framework.} 
The triangular matrix $\bT$ supports rank estimation. 
The decomposition takes different forms depending on the triangular structure of $\bT$: if $\bT$ is lower triangular, it is called the  \textit{ULV decomposition}; if $\bT$  is upper triangular, it is referred to as the URV decomposition.  
The UTV decomposition framework resembles the singular value decomposition (SVD; see Section~\ref{section:SVD}) in structure and serves as a computationally efficient alternative to the SVD. Both methods can be applied to find the least squares solution for rank-deficient matrices (Theorem~\ref{theorem:qr-for-ls-urv}).

\index{Decomposition: UTV}
\index{Orthogonal}
\index{Upper triangular}
\begin{theoremHigh}[Full ULV decomposition]\label{theorem:ulv-decomposition}
Any $m\times n$ matrix $\bA$ with rank $r$ can be decomposed as 
$$
\bA = \bU \begin{bmatrix}
\bL & \bzero \\
\bzero & \bzero 
\end{bmatrix}\bV,
$$
where $\bU\in \real^{m\times m}$ and $\bV\in \real^{n\times n}$ are  orthogonal matrices, and $\bL\in \real^{r\times r}$ is a lower triangular matrix  of full rank.
\end{theoremHigh}

The existence of the ULV decomposition follows  from those of the QR and LQ decomposition.
\begin{proof}[of Theorem~\ref{theorem:ulv-decomposition}]
For any rank-$r$ matrix $\bA=[\ba_1, \ba_2, \ldots, \ba_n]$, a column permutation matrix $\bP$ (Definition~\ref{definition:permutation-matrix}) can be used to reorder the columns of $\bA$, placing its linearly independent columns in the first $r$ positions of  $\bA\bP$. 
Without loss of generality, let $\bb_1, \bb_2, \ldots, \bb_r$ denote the $r$ linearly independent columns of $\bA$. Then, 
$$
\bA\bP = [\bb_1, \bb_2, \ldots, \bb_r, \bb_{r+1}, \ldots, \bb_n].
$$
Define $\bZ = [\bb_1, \bb_2, \ldots, \bb_r] \in \real^{m\times r}$. Since each $\bb_i$ lies in the column space of $\bZ$, there exists a matrix $\bE\in \real^{r\times (n-r)}$ such that 
$$
[\bb_{r+1}, \bb_{r+2}, \ldots, \bb_n] = \bZ \bE.
$$
Consequently, 
$$
\bA\bP = [\bb_1, \bb_2, \ldots, \bb_r, \bb_{r+1}, \ldots, \bb_n] = \bZ
\begin{bmatrix}
\bI_r & \bE
\end{bmatrix},
$$
where $\bI_r$ is the $r\times r$ identity matrix. Additionally, the matrix $\bZ\in \real^{m\times r}$ has full column rank, so it admits the full QR decomposition: $\bZ = \bU\begin{bmatrixscript}
\bR \\
\bzero
\end{bmatrixscript}$, where $\bR\in \real^{r\times r}$ is an upper triangular matrix of full rank, and $\bU$ is an orthogonal matrix. Substituting this into the previous expression gives:
\begin{equation}\label{equation:ulv-smpl}
\bA\bP = \bZ
\begin{bmatrix}
\bI_r & \bE
\end{bmatrix}
=
\bU\begin{bmatrix}
\bR \\
\bzero
\end{bmatrix}
\begin{bmatrix}
\bI_r & \bE
\end{bmatrix}
=
\bU\begin{bmatrix}
\bR & \bR\bE \\
\bzero & \bzero 
\end{bmatrix}.
\end{equation}
Since $\bR$ has full rank, 
$\begin{bmatrix}
\bR & \bR\bE 
\end{bmatrix}$ also has full rank. Its full LQ decomposition is given by:
$\begin{bmatrix}
\bL & \bzero 
\end{bmatrix} \bV_0$, where $\bL\in \real^{r\times r}$ is a lower triangular matrix, and $\bV_0$ is an orthogonal matrix. Substituting  this into Equation~\eqref{equation:ulv-smpl}, we have 
$$
\bA = \bU \begin{bmatrix}
\bL & \bzero \\
\bzero & \bzero 
\end{bmatrix}\bV_0 \bP^{-1}.
$$
Finally, let $\bV =\bV_0 \bP^{-1}$, which is orthogonal since it is  a product of two orthogonal matrices. This completes the proof.
\end{proof}
An alternative proof of the ULV decomposition will be discussed in Theorem~\ref{theorem:complete-orthogonal-decom} using the rank-revealing QR decomposition and the standard QR decomposition.

Now, suppose the ULV decomposition of a matrix $\bA$ is given by
$
\bA = \bU 
\scriptsize\begin{bmatrix}
\bL & \bzero \\
\bzero & \bzero 
\end{bmatrix}
\normalsize\bV
$.
Let $\bU_0 = \bU_{:,1:r}$ and $\bV_0 = \bV_{1:r,:}$, where $\bU_0$ consists of  the first $r$ columns of $\bU$, and $\bV_0$ consists of  the first $r$ rows of $\bV$. Then, we can write $\bA = \bU_0 \bL\bV_0$. This form is called the \textit{reduced ULV decomposition}.
Similarly, the URV decomposition can be derived as follows:
\begin{theoremHigh}[URV decomposition]\label{theorem:urv-decomposition}
Any $m\times n$ matrix $\bA$ with rank $r$ can be decomposed as 
$$
\bA = \bU \begin{bmatrix}
\bR & \bzero \\
\bzero & \bzero 
\end{bmatrix}\bV,
$$
where $\bU\in \real^{m\times m}$ and $\bV\in \real^{n\times n}$ are two orthogonal matrices, and $\bR\in \real^{r\times r}$ is an upper triangular matrix  of full rank.
The \textit{reduced URV decomposition} can be obtained as $\bA=\bU_{:,1:r}\bR\bV_{1:r,:}$.
\end{theoremHigh}
The proof closely resembles that of the ULV decomposition, and is left as an exercise.
Collectively, the ULV and URV decompositions are referred to as the UTV decomposition framework \citep{hanson1969extensions, fierro1997low, golub2013matrix}.

\index{Basis}
\index{Row space}
\index{Column space}
\index{Orthonormal basis}
\paragraph{Range and null space.}
This decomposition framework, first introduced by \citet{hanson1969extensions}, provides explicit orthogonal bases for the range and null space of $\bA$, as well as a representation for the pseudo-inverse (see Problem~\ref{prob:utv_pseudo}). 
We will soon observe that the structures of ULV and URV decompositions closely resemble that of the singular value decomposition (SVD). 
All three decompositions factorize  the matrix $\bA$ into two orthogonal matrices.  
More specifically, both ULV and URV decompositions provide orthonormal bases for the four fundamental subspaces of $\bA$, as described in the fundamental theorem of linear algebra (Theorem~\ref{theorem:fundamental-linear-algebra}). 
For example, in the ULV decomposition, the first $r$ columns of $\bU$ form an orthonormal basis for the column space $\cspace(\bA)$, while the last $(m-r)$ columns of $\bU$ form an orthonormal basis for the left null space $\nspace(\bA^\top)$. Similarly, the first $r$ rows of $\bV$ form an orthonormal basis for the row space $\cspace(\bA^\top)$, while the last $(n-r)$ rows provide an orthonormal basis for the  null space $\nspace(\bA)$ (resembling  the two-sided orthogonal decomposition; Theorem~\ref{theorem:two-sided-orthogonal}):
\begin{equation}\label{equation:ortho_space_utv}
\begin{aligned}
\cspace(\bA) &= \spn\{\bu_1, \bu_2, \ldots, \bu_r\}, \qquad
\nspace(\bA) &=& \spn\{\bv_{r+1}, \bv_{r+2}, \ldots, \bv_n\}, \\
\cspace(\bA^\top) &= \spn\{ \bv_1, \bv_2, \ldots,\bv_r \} ,\qquad
\nspace(\bA^\top)&=&\spn\{\bu_{r+1}, \bu_{r+2}, \ldots, \bu_m\}.
\end{aligned}
\end{equation}
The SVD extends this framework by establishing direct relationships between the corresponding two pairs of orthonormal bases. It characterizes the linear transformations between the column space and row space, as well as between the left null space and (right) null space: $\bA\bv_i=\sigma_i\bu_i$ for all $i$.  These connections will be explored in greater detail in the chapter on the SVD.

\section{Complete Orthogonal Decomposition}
The UTV decomposition is closely related to the concept of the \textit{complete orthogonal decomposition}, which also involves factoring a matrix into two orthogonal matrices.
\begin{theoremHigh}[Complete orthogonal decomposition]\label{theorem:complete-orthogonal-decom}
Any $m\times n$ matrix $\bA$ with rank $r$ can be factored as 
$$
\bA = \bU \begin{bmatrix}
\bT & \bzero \\
\bzero & \bzero 
\end{bmatrix}\bV,
$$
where $\bU\in \real^{m\times m}$ and $\bV\in \real^{n\times n}$ are two orthogonal matrices, and $\bT\in \real^{r\times r}$ is a matrix of full rank $r$.
\end{theoremHigh}
\begin{proof}[of Theorem~\ref{theorem:complete-orthogonal-decom}]
Using the column-pivoted QR decomposition (Theorem~\ref{theorem:rank-revealing-qr-general}), the matrix $\bA$ can be decomposed as 
$
\bQ_1^\top \bA\bP = 
\scriptsize
\begin{bmatrix}
\bR_{11} & \bR_{12} \\
\bzero   & \bzero 
\end{bmatrix},
$
where $\bR_{11} \in \real^{r\times r}$ is upper triangular, $\bR_{12} \in \real^{r\times (n-r)}$, $\bQ_1\in \real^{m\times m}$ is an orthogonal matrix, and $\bP$ is a permutation matrix. Next, we construct a decomposition that satisfies:
\begin{equation}\label{equation:orthogonal-complete-qr-or-not}
\begin{bmatrix}
\bR_{11}^\top \\
\bR_{12}^\top
\end{bmatrix}
=
\bQ_2
\begin{bmatrix}
\bS \\
\bzero 
\end{bmatrix},
\end{equation}
where $\bQ_2$ is an orthogonal matrix, and $\bS$ is a rank-$r$ matrix. This decomposition is valid because the matrix $\scriptsize\begin{bmatrix}
\bR_{11}^\top \\
\bR_{12}^\top
\end{bmatrix} \in \real^{n\times r}$ has rank $r$ of which the columns stay in a subspace of $\real^{n}$. Nevertheless, the columns of $\bQ_2$ span the entire space  $\real^n$, where we can assume that the first $r$ columns of $\bQ_2$ span the same space as that of  $\scriptsize\begin{bmatrix}
\bR_{11}^\top \\
\bR_{12}^\top
\end{bmatrix}$. The matrix $\scriptsize\begin{bmatrix}
\bS \\
\bzero 
\end{bmatrix}$ serves to map $\bQ_2$ back to $\scriptsize\begin{bmatrix}
\bR_{11}^\top \\
\bR_{12}^\top
\end{bmatrix}$.
Finally, substituting this decomposition, we find:
$
\bQ_1^\top \bA\bP \bQ_2 = 
\scriptsize
\begin{bmatrix}
\bS^\top & \bzero \\
\bzero & \bzero 
\end{bmatrix}.
$
Setting $\bU = \bQ_1$, $\bV=\bQ_2^\top\bP^\top$, and $\bT = \bS^\top$, we complete the proof.
\end{proof}

Note that the complete orthogonal decomposition is quite  general. When Equation~\eqref{equation:orthogonal-complete-qr-or-not} is interpreted as the reduced QR decomposition of $\scriptsize\begin{bmatrix}
\bR_{11}^\top \\
\bR_{12}^\top
\end{bmatrix}$, the complete orthogonal decomposition simplifies to the ULV decomposition.

\index{CPQR}
\section{Computing the UTV Decomposition}
The CPQR decomposition introduced in Section~\ref{section:cpqr} can be applied to find the UTV decomposition of a matrix.
The CPQR factorization of a rank-deficient matrix $\bA \in \real^{m \times n}$ is given by
$$
\bA\bP = [\bQ_1, \bQ_2] \begin{bmatrix} \bR_{11} & \bR_{12} \\ \bzero & \bzero \end{bmatrix},
$$
where $\bR_{11} \in \real^{r \times r}$ is nonsingular ($r < n$). Here $\bQ_1$ and $\bQ_2$ give orthogonal bases for $\cspace(\bA)$ and $\nspace(\bA^\top)$, respectively. 
However, this factorization is less useful for applications that need a basis for $\nspace(\bA)$. 
To address this, the off-diagonal block $\bR_{12}$ then can be annihilated by postmultiplying $\bR$ with a sequence of Householder reflectors:
\begin{equation}\label{equation:cpqr_utv_pro1}
[\bR_{11} , \bR_{12}] \bH_r \ldots \bH_2 \bH_1 = [\widehat{\bR} , \bzero], \quad \bH_j = \bI - 2 \bu_j \bu_j^\top,
\end{equation}
$j = r, r-1, \ldots, 1$, where each vector  $\bu_j$ has nonzero entries only in positions $j, r+1, \ldots, n$. 
This process is equivalent to performing a QL factorization on the transpose of the triangular factor $\bR$:
\begin{equation}\label{equation:utv_r_fac_ql}
\begin{bmatrix} 
\bR_{11}^\top  & \bzero \\ 
\bR_{12}^\top  & \bzero 
\end{bmatrix} 
= \widehat{\bQ} 
\begin{bmatrix} 
\widehat{\bR}^\top \\ 
\bzero 
\end{bmatrix},
\end{equation}
where the Householder reflectors are applied from the left rather than from the right.
And  this requires $2r^2(n-r)$ flops (see Problem~\ref{prob:comple_qr_utvr}). 
As a result, we obtain a URV decomposition of the form:
\begin{equation}\label{equation:urv_cpqr}
	\bA\bP = \bQ \begin{bmatrix} \widehat{\bR} & \bzero \\ \bzero & \bzero \end{bmatrix} \bV^\top, \quad \bV = \bH_1\bH_2 \ldots \bH_r.
\end{equation}

For example, the first three steps for a matrix with $n=6$ and $r=4$ in the reduction are shown below:
$$
\begin{sbmatrix}{\bA}
\boxtimes & 0 & 0 & 0  \\ 
\boxtimes & \boxtimes & 0 & 0  \\ 
\boxtimes & \boxtimes & \boxtimes & 0  \\ 
\boxtimes & \boxtimes & \boxtimes & \boxtimes  \\ 
\boxtimes & \boxtimes & \boxtimes & \boxtimes  \\ 
\boxtimes & \boxtimes & \boxtimes & \boxtimes  
\end{sbmatrix}
\stackrel{\bH_4}{\rightarrow}
\begin{sbmatrix} {\bH_4\bA}
\boxtimes & 0 & 0 & 0  \\ 
\boxtimes & \boxtimes & 0 & 0  \\ 
\boxtimes & \boxtimes & \boxtimes & 0  \\ 
\bm{\boxtimes} & \bm{\boxtimes} & \bm{\boxtimes} & \bm{\boxtimes}  \\ 
\bm{\boxtimes} & \bm{\boxtimes} & \bm{\boxtimes} & \bm{0}  \\ 
\bm{\boxtimes} & \bm{\boxtimes} & \bm{\boxtimes} & \bm{0}  
\end{sbmatrix}
\stackrel{\bH_3}{\rightarrow}
\begin{sbmatrix} {\bH_3\bH_4\bA}
\boxtimes & 0 & 0 & 0  \\ 
\boxtimes & \boxtimes & 0 & 0  \\ 
\bm{\boxtimes} & \bm{\boxtimes} & \bm{\boxtimes} & 0  \\ 
\bm{\boxtimes} & \bm{\boxtimes} & \bm{\boxtimes} & \boxtimes  \\ 
\bm{\boxtimes} & \bm{\boxtimes} & \bm{0} & 0  \\ 
\bm{\boxtimes} & \bm{\boxtimes} & \bm{0} & 0  
\end{sbmatrix}  
\stackrel{\bH_2}{\rightarrow}
\begin{sbmatrix} {\bH_2\bH_3\bH_4\bA}
\boxtimes & 0 & 0 & 0  \\ 
\bm{\boxtimes} & \bm{\boxtimes} & 0 & 0  \\ 
\bm{\boxtimes} & \bm{\boxtimes} & {\boxtimes} & 0  \\ 
\bm{\boxtimes} & \bm{\boxtimes} & \boxtimes & \boxtimes  \\ 
\bm{\boxtimes} & \bm{\boxtimes} & {0} & 0  \\ 
\bm{\boxtimes} & \bm{\boxtimes} & {0} & 0  
\end{sbmatrix}  
\ldots.
$$
Note that the application of $\bH_3$  does not affect the last column, and the premultiplication of $\bH_2$ does not affect the last two columns, as explained by Corollary~\ref{corollary:unreflec_house}.

\begin{exercise}[ULV]
Find a way to compute the ULV decomposition of a matrix.
\end{exercise}

\index{Decomposition: RR UTV}
\section{Rank-Revealing UTV Decomposition and Other Issues}\label{section:rrutv_other}

\paragraph{Rank-revealing URV.}
For matrices $\bA \in \real^{m \times n}$ that are nearly rank-deficient with rank $r < n$, \citet{stewart2002updating} introduced the ranking-revealing URV decomposition. This decomposition takes the form
\begin{equation}\label{equation:rr_urv}
	\bA\bP = \bU \begin{bmatrix} \bR_{11} & \bR_{12} \\ \bzero & \bR_{22} \end{bmatrix} \bV^\top, \quad \bR_{11} \in \real^{r \times r},
\end{equation}
where $\bU = [\bU_1 , \bU_2] \in \real^{m\times m}$ and $\bV = [\bV_1, \bV_2]\in\real^{n\times n}$ are orthogonal matrices, and  $\bR_{11} \in \real^{r \times r}$ and $\bR_{22} \in \real^{(m-r) \times (n-r)}$ are upper triangular. 
If the singular values~\footnote{Once again, see Section~\ref{section:SVD} for more details.} of $\bA$ are ordered such that 
$$
\sigma_1 \geq \sigma_2 \geq \ldots \geq \sigma_r \gg \sigma_{r+1} \geq \ldots \geq \sigma_n,
$$
then the decomposition \eqref{equation:rr_urv} is said to be rank-revealing if it satisfies the following conditions:
$$
\sigma_r(\bR_{11}) \geq \sigma_r/c, \quad (\normf{\bR_{12}}^2 + \normf{\bR_{22}}^2)^{1/2} \leq c \sigma_{r+1},
$$
where $c$ is bounded by a low-degree polynomial in terms of $r$ and $n$. For $\bP = \bI$, it follows from \eqref{equation:rr_urv} that
$$
\normf{\bA \bV_2} = \normf{\begin{bmatrix} \bR_{12} \\ \bR_{22} \end{bmatrix}} \leq c \sigma_{r+1}.
$$
Thus, $\bV_2$ forms an orthogonal basis for the approximate null space of $\bA$. The URV decomposition is particularly useful in applications such as  subspace tracking in signal processing, where there is a need to compute an approximate null space and update this basis as rows are added or removed from $\bA$ \citep{bjorck2024numerical}.

The rank-revealing process begins with a pivoted QR decomposition (Theorem~\ref{theorem:finding-good-qr-ordering}) and identifies  a vector $\bv$ such that $\normtwo{\bR\bv}$ is small. 
Such a vector exists. For example, $\bv=\bv_n$, where $\bv_n$ is the  right singular vector of $\bR$ corresponding to the smallest singular value $\sigma_n$ and left singular vector $\bu_n$ such that $\bR\bv_n=\sigma_n\bu_n$ and $\normtwo{\bR\bv_n}=\sigma_n$. 
If $\bA$ or $\bR$ is rank-deficient, $\sigma_n$ is small; see Sections~\ref{section:rankone_defi} and \ref{section:rank-r-qr}.
Next, a sequence of Givens rotations $\bG_{12},  \bG_{23}, \ldots, \bG_{n-1,n}$ is determined such that
$$
\bG^\top \bv = \bG_{n-1,n}^\top \ldots \bG_{23}^\top\bG_{12}^\top \bv = \normtwo{\bv} \be_n.
$$
Then, an orthogonal matrix $\bU$ is computed such that $\bU^\top \bR \bG = \bU^\top \bR\bG_{12} \ldots \bG_{n-1,n}$ is upper triangular. 
When applying $\bG_{i-1,i}$, a nonzero element---known as a ``\textit{bulge}"---is introduced just below the diagonal of $\bR$. To restore the triangular form, a left rotation is used to ``\textit{chase the bulge}." These left rotations amount to the orthogonal matrix $\bU$.

An example is shown below for a $4\times 4$ upper triangular matrix $\bR$, where ${\boxtimes}$ denotes an upper triangular entry of $\bR$,  \textbf{boldface} indicates a value that has just been modified,  $\textcolor{brown}{\bm{\boxtimes} }$ denotes a bulge value, and $\textcolor{mylightbluetext}{\bzero}$ denotes the zero is introduced back during the process of chasing the bulge:
$$
\begin{aligned}
\begin{sbmatrix}{\bR}
	\boxtimes & \boxtimes & \boxtimes & \boxtimes \\
	0 & \boxtimes & \boxtimes & \boxtimes \\
	0 & 0 & \boxtimes & \boxtimes \\
	0 & 0 & 0 & \boxtimes 
\end{sbmatrix}
&\stackrel{\times \bG_{12}}{\rightarrow}
\begin{sbmatrix}{\bR\bG_{12}}
	\bm{\boxtimes} & \bm{\boxtimes} & \bm{\boxtimes} & \bm{\boxtimes} \\
	\textcolor{brown}{\bm{\boxtimes} } & \bm{\boxtimes} & \bm{\boxtimes} & \bm{\boxtimes} \\
	0 & 0 & \boxtimes & \boxtimes \\
	0 & 0 & 0 & \boxtimes 
\end{sbmatrix}
\stackrel{ \bU_{12}\times}{\rightarrow}
\begin{sbmatrix}{\bU_{12}\bR\bG_{12}}
	\bm{\boxtimes} & \bm{\boxtimes} & \boxtimes & \boxtimes \\
	\textcolor{mylightbluetext}{\bzero} & \bm{\boxtimes} & \boxtimes & \boxtimes \\
	\bzero & \bzero & \boxtimes & \boxtimes \\
	\bzero & \bzero & 0 & \boxtimes 
\end{sbmatrix}
\stackrel{ \times\bG_{23}}{\rightarrow}
\begin{sbmatrix}{\bU_{12}\bR\bG_{12}\bG_{23}}
	0 & \boxtimes & \boxtimes & \boxtimes \\
	\bzero &  \bm{\boxtimes} & \bm{\boxtimes} & \bm{\boxtimes} \\
	\bzero & \textcolor{brown}{\bm{\boxtimes} } & \bm{\boxtimes} & \bm{\boxtimes} \\
	0 & 0 & 0 & \boxtimes 
\end{sbmatrix}\\
&\stackrel{ \bU_{23}\times}{\rightarrow}
\begin{sbmatrix}{\bU_{23}\bU_{12}\bR\bG_{12}\bG_{23}}
	\boxtimes & \bm{\boxtimes} & \bm{\boxtimes} & \boxtimes \\
	0 & \bm{\boxtimes} & \bm{\boxtimes} & \boxtimes \\
0 & \textcolor{mylightbluetext}{\bzero} & \bm{\boxtimes} & \boxtimes \\
0 & \bzero & 0 & \boxtimes 
\end{sbmatrix}
\stackrel{ \times\bG_{34}}{\rightarrow}
\begin{sbmatrix}{\bU_{23}\bU_{12}\bR\ldots\bG_{34}}
\boxtimes & \boxtimes & \boxtimes & \boxtimes \\
0 & \boxtimes & \boxtimes & \boxtimes \\
\bzero & \bzero & \bm{\boxtimes} & \bm{\boxtimes} \\
\bzero & \bzero & \textcolor{brown}{\bm{\boxtimes} } & \bm{\boxtimes} 
\end{sbmatrix}
\stackrel{ \bU_{34}\times}{\rightarrow}
\begin{sbmatrix}{\bU_{34}\ldots\bR\ldots\bG_{34}}
\boxtimes & \boxtimes & \bm{\boxtimes} & \bm{\boxtimes} \\
0 & \boxtimes & \bm{\boxtimes} & \bm{\boxtimes} \\
0 & 0 & \bm{\boxtimes} & \bm{\boxtimes} \\
0 & 0 & \textcolor{mylightbluetext}{\bzero} & \bm{\boxtimes} 
\end{sbmatrix}
=
\widehat{\bR}.
\end{aligned}
$$
This process of transforming $\bR$ to $\widehat{\bR}$ requires $\mathcalO(n^2)$ multiplications. We now have
$$
\bU^\top \bR \bv = (\underbrace{\bU^\top \bR \bG}_{=\widehat{\bR}})(\bG^\top \bv) =\normtwo{\bv} \widehat{\bR} \widehat{\be}_n.
$$
Since $\bU$ is orthogonal, it follows that if $\normtwobig{\bR \bv} < \abs{r_{nn}}$, then $\normtwo{\widehat{\bR}\widehat{\be}_n} < \gamma/\normtwo{\bv}$ for some $\gamma$. This bounds the norm for the last column of the transformed matrix $\widehat{\bR}$. If $\abs{r_{n-1,n-1}}$ is small, this process can be continued on the leading principal submatrix of order $n-1$ of $\widehat{\bR}$.

\paragraph{Appending a row.}
Just as with the rank-one update of the Cholesky decomposition (Section~\ref{section:cholesky-rank-one-update}) and the addition of a row to a QR decomposition (Section~\ref{section:append-row-qr}), we are often interested in efficiently updating solutions to least squares problems when new data arrive, particularly in online or streaming data settings; see Section~\ref{section:application-ls-qr} for related applications.
In such cases, we may want to append a new row to the observed data matrix $\bA$ and compute the (rank-revealing) UTV decomposition along with its corresponding least squares solution (Theorem~\ref{theorem:qr-for-ls-urv}) in an efficient manner.
For simplicity in notation, we denote the rank-revealing URV decomposition in \eqref{equation:rr_urv} as
\begin{equation}\label{equation:rr_urv2}
\bA = \bU 
\begin{bmatrix} 
\bR & \bJ \\ 
\bzero & \bF
\end{bmatrix} \bV^\top, \quad \bR \in \real^{r \times r},
\end{equation}
where $\bU$ and $\bV$ are orthogonal, and $\bR \in \real^{r \times r}$ and $\bF \in \real^{(m-r) \times (n-r)}$ are upper triangular. Let $\sigma_1 \geq \sigma_2 \geq \ldots \geq \sigma_n$ be the singular values of $\bA$, and assume that for some $r < n$, we have $\sigma_r \gg \sigma_{r+1} \leq \delta$, where $\delta$ is a given tolerance. 
Then, the numerical $\delta$-rank of $\bA$ equals $r$ (see Definition~\ref{definition:effective-rank-in-svd}). Furthermore, if
$$
\sigma_r(\bR) \geq \frac{1}{c} \sigma_r, \quad (\normf{\bJ}^2 + \normf{\bF}^2)^{1/2} \leq c \sigma_{r+1}
$$
for some constant $c$, the decomposition \eqref{equation:rr_urv2} reveals  the rank and null space of $\bA$. The URV decomposition can be updated in $\mathcalO(n^2)$ operations when a new row $\ba^\top$ is added to $\bA$. 
To see this, we have 
\begin{equation}\label{eq:urv_update}
\begin{bmatrix} \bU^\top & \bzero \\ \bzero & 1 \end{bmatrix} 
\begin{bmatrix} \bA \\ \ba^\top \end{bmatrix} 
\bV = 
\begin{bmatrix} \bR & \bJ \\ \bzero & \bF \\ \bx^\top  &  \by^\top \end{bmatrix},
\end{equation}\
where $\ba^\top \bV = [\bx^\top, \by^\top]$ and $(\normf{\bJ}^2 + \normf{\bF}^2)^{1/2} = \nu \leq \delta$. In the simplest case the inequality
\begin{equation}\label{eq:simple_case}
\sqrt{\nu^2 + \normtwo{\by}^2} \leq \delta
\end{equation}
is satisfied. In this case,  it suffices to reduce the matrix in \eqref{eq:urv_update} to upper triangular form using a sequence of left Givens rotations. Note that the updated matrix $\bR$ cannot become effectively rank-deficient because its singular values cannot decrease.

If \eqref{eq:simple_case} is not satisfied, we first reduce $\by^\top$ in \eqref{eq:urv_update} so that it becomes proportional to $\be_1^\top$, while preserving the upper triangular structure of $\bF$. This can be achieved by a sequence of (interleaved) right and left Givens rotations.

An example is shown  below for a matrix $\bA$ with dimensions $m-r=3$ and $n-r=3$. 
Note that here the $j$'s represent entire \textbf{columns} of $\bJ$, $f$ denotes an element of $\bF$,  $y$ denotes an element of $\by$, and \textbf{boldface} indicates the value has just been changed. 
Additionally,  $\boxtimes$ denotes a nonzero value, known as a bulge, introduced by the right Givens rotations, and $\textcolor{mylightbluetext}{\bzero}$ denotes a zero value that is reintroduced (i.e., chasing the bulge).
\paragraph{Step 1: Interleaved left and right Givens rotations.} 
We first consider the right-most part of 
$\begin{bmatrixscript} \bU^\top & \bzero \\ \bzero & 1 \end{bmatrixscript} 
\begin{bmatrixscript} \bA \\ \ba^\top \end{bmatrixscript} 
\bV$, which is defined as
 $
\bB=
\begin{bmatrixscript}
	\bJ\\
	\bF\\
	\by^\top
\end{bmatrixscript}$:
$$
\begin{sbmatrix}{\bB}
j & j & j \\ 
f & f & f \\ 
0 & f & f \\ 
0 & 0 & f \\ 
y & y & y  \\
\end{sbmatrix} 
\stackrel{\times \bH_1}{\rightarrow}
\begin{sbmatrix} {\bB\bH_1}
j & \bj & \bj \\ 
f & \bff & \bff \\ 
0 & \bff & \bff \\ 
0 & \bm{\boxtimes} & \bff \\ 
y & \by & \bzero
\end{sbmatrix} 
\stackrel{\bG_1\times}{\rightarrow}
\begin{sbmatrix} {\bG_1\bB\bH_1}
j & j & j \\ 
f & f & f \\ 
\bzero & \bff & \bff \\ 
\bzero & \textcolor{mylightbluetext}{\bzero} & \bff \\ 
y & y & 0 
\end{sbmatrix} 
\stackrel{\times \bH_2}{\rightarrow}
\begin{sbmatrix}  {\bG_1\bB\bH_2\bH_2}
\bj & \bj & j \\ 
\bff & \bff & f \\ 
\bm{\boxtimes} & \bff & f \\ 
\bzero & \bzero & f \\ 
\bsigma & \bzero & 0 
\end{sbmatrix} 
\stackrel{\bG_2\times}{\rightarrow}
\begin{sbmatrix}  {\bG_2\bG_1\bB\bH_2\bH_2}
	j & j & j \\ 
	\bff & \bff & \bff \\ 
	\textcolor{mylightbluetext}{\bzero} & \bff & \bff \\ 
	0 & 0 & f \\ 
	\sigma & 0 & 0 
\end{sbmatrix} .
$$
In this part of the reduction, $\bR$ and $\bx^\top$ are not involved. 

\paragraph{Step 2: Triangularization using Givens rotations.} 
At this point, the matrix has the form:
$$
\begin{bmatrix} 
\bR & \widetildebj & \widetildebH \\ 
\bzero  & \widetildebf & \widetildebN \\  
\bx^\top & \sigma & \bzero 
\end{bmatrix}.
$$
This matrix is then reduced to upper triangular form using Givens rotations from the left, and the rank $r$ is increased by 1.
This procedure is similar to the QR decomposition using Givens rotations; see Section~\ref{section:qr-givens} for more details.
The entire update process requires $\mathcalO(n^2)$ floating-point operations.

\section{Application: Row Rank equals Column Rank via UTV}
The UTV framework enables the proof of a fundamental theorem in linear algebra: the equality of the row rank and column rank of a matrix; Theorem~\ref{theorem:equal-dimension-rank}.
It is worth noting that when using the UTV decomposition for this proof, a slight adjustment is needed in framing the existence claim.
For instance, in Theorem~\ref{theorem:ulv-decomposition}, the initial assumption about the matrix $\bA$ is that its rank is $r$. However, since having rank $r$ inherently implies the equality of row  and column ranks, a more precise assertion in this context would be to state that $\bA$ has \textbf{column rank} $r$ in Theorem~\ref{theorem:ulv-decomposition}. 
For further discussion, see \citet{lu2021column}.

\begin{proof}[{of Theorem~\ref{theorem:equal-dimension-rank},  second approach}]
Any $m\times n$ matrix $\bA$ with rank $r$ can be factored as 
$$
\bA = \bU_0 \begin{bmatrix}
\bL & \bzero \\
\bzero & \bzero 
\end{bmatrix}\bV_0,
$$
where $\bU_0\in \real^{m\times m}$ and $\bV_0\in \real^{n\times n}$ are two orthogonal matrices, and $\bL\in \real^{r\times r}$ is a lower triangular matrix \footnote{Instead of using the ULV decomposition, some texts  use elementary transformations $\bE_1$ and $\bE_2$, such that $
\bA = \bE_1 \begin{bmatrix}
\bI_r & \bzero \\
\bzero & \bzero 
\end{bmatrix}\bE_2,
$ to prove the result.}. Let $\bD = \scriptsize\begin{bmatrix}
\bL & \bzero \\
\bzero & \bzero 
\end{bmatrix}$. Clearly, the row rank and column rank of $\bD$ are equal. If we can show that the column rank and row rank of $\bA$ are equal to those of  $\bD$, the proof is complete.

Let $\bU = \bU_0^\top$ and $\bV=\bV_0^\top$. Then, $\bD = \bU\bA\bV$. The proof can be broken into two parts: first prove that the row rank and column rank of $\bA$ are equal to those of $\bU\bA$; then prove that the row rank and column rank of $\bU\bA $ are equal to those of $\bU\bA\bV$.

\paragraph{Row  and column ranks of $\bA$ equal those of $\bU\bA$.} Let $\bB = \bU\bA$. Partition $\bA$ and $\bB$ by columns: $\bA=[\ba_1,\ba_2,\ldots, \ba_n]$ and $\bB=[\bb_1,\bb_2,\ldots,\bb_n]$, where $\bb_i=\bU\ba_i$ for all $i$.  If a linear combination $x_1\ba_1+x_2\ba_2+\ldots+x_n\ba_n=\bzero$, then premultiplying by $\bU$ gives
$$
\bU(x_1\ba_1+x_2\ba_2+\ldots+x_n\ba_n) = x_1\bb_1+x_2\bb_2+\ldots+x_n\bb_n = \bzero.
$$
Thus, any independent subset of columns in $\bA$ corresponds to an independent subset of columns in $\bB$, implying:
$
\dim(\cspace(\bB)) \leq \dim(\cspace(\bA)).
$
Similarly, since $\bA = \bU^\top\bB$, we have
$
\dim(\cspace(\bA)) \leq \dim(\cspace(\bB)).
$
This proves 
$
\dim(\cspace(\bB)) = \dim(\cspace(\bA)).
$
Applying the same reasoning to $\bB^\top$ and $\bA^\top$ shows that
$
\dim(\cspace(\bB^\top)) = \dim(\cspace(\bA^\top)).
$
This implies the row rank and column rank of $\bA$ and $\bB=\bU\bA$ are the same.

\paragraph{Row  and column ranks of $\bU\bA$ equal those of $\bU\bA\bV$.}
Using the same reasoning as above, applied to  $\bU\bA$ and $\bU\bA\bV$, we conclude that their row and column ranks are equal.
This completes the proof.
\end{proof}

\begin{problemset}
\item Use the UTV decompositions of $\bA$ and $\bB$ to derive  the UTV decomposition of $\diag(\bA,\bB)$.
\item Prove rigorously that the four subspaces described in Equation~\eqref{equation:ortho_space_utv} can be spanned using the UTV decomposition.

\item Let $\bA=\bU\bB\bV$ be given, where $\bU$ and $\bV$ are orthogonal matrices. Prove that $\sum_{i,j}\abs{a_{ij}}^2=\sum_{i,j} \abs{b_{ij}}^2$.
\textit{Hint: Verify this by showing that $\trace(\bA^\top\bA)=\trace(\bB^\top\bB)$.}

\item Explain how to compute the UTV decomposition using either Householder reflectors or Givens rotations. Provide insights into the advantages and limitations of each approach.

\item Present a detailed and rigorous proof of the URV decomposition, as outlined in Theorem~\ref{theorem:urv-decomposition}.

\item Let $\bA,\bB\in\complex^{n\times n}$ be given, where $\bB$ is nonsingular. Prove that there exist unitary matrices $\bU,\bV\in\complex^{n\times n}$ such that $\bA=\bU\bT_A\bV$ and $\bB=\bU\bT_B\bV$, where $\bT_A$ and $\bT_B$ are upper triangular matrices. Furthermore, show that the main diagonal entries of $\bT_B^{-1}\bT_A$ are the eigenvalues $\bB^{-1}\bA$. 
When these eigenvalues are real, show that all the underlying matrices can be chosen to be real, and $\bU$ and $\bV$ are orthogonal.
\textit{Hint: Use the Schur decomposition for  $\bB^{-1}\bA=\bU\bT\bU^*$ (Theorem~\ref{theorem:schur-decomposition_complex}) and the QR decomposition for $\bB\bU$.}

\item \label{prob:comple_qr_utvr} Prove that the computational complexity of computing the triangular factor $\bR$ in \eqref{equation:utv_r_fac_ql}  requires $2r^2(n-r)$ flops.

\item \textbf{Read Section~\ref{section:SVD} first.} Let $\bA,\bB\in\real^{m\times n}$. Show that~\footnote{
Let $\bA,\bB\in\real^{m\times n}$. 
Then $\bA$ and $\bB$ are \textit{left equivalent (resp., orthogonally left equivalent)} if there exists a nonsingular (resp., orthogonal) $\bZ_1\in\real^{m\times m}$ such that $\bA=\bZ_1\bB$;  
$\bA$ and $\bB$ are \textit{right equivalent (resp., orthogonally right equivalent)} if there exists a nonsingular (resp., orthogonal) matrix $\bZ_2\in\real^{n\times n}$ such that $\bA=\bB\bZ_2$; 
$\bA$ and $\bB$ are \textit{biequivalent (resp., orthogonally biequivalent)} if there exist nonsingular (resp., orthogonal) matrices $\bZ_1\in\real^{m\times m}$ and  $\bZ_2\in\real^{n\times n}$ such that $\bA=\bZ_1\bB\bZ_2$.
}
\begin{itemize}
\item  $\bA$ and $\bB$ are orthogonally left equivalent if and only if $\bA^\top\bA=\bB^\top\bB$.
\item   $\bA$ and $\bB$ are orthogonally right equivalent if and only if $\bA\bA^\top =\bB\bB^\top$.
\item  $\bA$ and $\bB$  are orthogonally biequivalent if and only if $\bA$ and $\bB$ have the same singular values with the same multiplicity.
\end{itemize}

\item  Let $\bA,\bB\in\real^{m\times n}$. Show that 
\begin{itemize}
	\item  The matrices $\bA$ and $\bB$ are left equivalent if and only if $\nspace(\bA) = \nspace(\bB)$.
	\item  The matrices $\bA$ and $\bB$ are right equivalent if and only $\cspace(\bA) = \cspace(\bB)$.
	\item  The matrices $\bA$ and $\bB$ are biequivalent if and only if $\rank(\bA)=\rank(\bB)$.
\end{itemize}

\item \label{prob:utv_pseudo} 
Let $
\bA = \bU \begin{bmatrixscript}
	\bR & \bzero \\
	\bzero & \bzero 
\end{bmatrixscript}\bV
$ be a UTV decomposition of $\bA\in\real^{m\times n}$.  Show that the  pseudo-inverse of $\bA$ is
$
\bA^+ = \bV^\top \begin{bmatrixscript} \bR^{-1} & \bzero \\ \bzero & \bzero \end{bmatrixscript} \bU^\top.
$

\item Following the procedure outlined in Section~\ref{section:rrutv_other} for appending a row to a rank-revealing URV decomposition, write out the complete pseudo-code for the algorithm and prove that it requires $\mathcalO(n^2)$ flops.

\item Recover the URV and SVD decompositions using the complete orthogonal decomposition.

\end{problemset}

\part{Data Interpretation and Information Distillation}\label{part:data-interation}

\newpage
\chapter{CR Decomposition}\label{section:cr-decomposition}

\index{Linearly independent}
\section{CR Decomposition}
The CR decomposition of a matrix, introduced in \citet{strang2021every, stranglu}, offers valuable insights into the matrix’s rank and the relationships between its columns and rows. As is customary, we begin by presenting the result, deferring the discussion of its existence and derivation to later sections.

\index{Decomposition: CR}
\index{Pseudo-inverse}
\index{Rank}
\index{Data storage}
\begin{theoremHigh}[CR decomposition]\label{theorem:cr-decomposition}
Let $\bA \in \real^{m \times n}$ be a matrix of  rank $r$. Then it can be factored as 
$$
\underset{m\times n}{\bA} = \underset{m\times r}{\bC} \gapthree \underset{r\times n}{\bR},
$$
where $\bC$ consists of the first $r$ linearly independent columns of $\bA$, and $\bR$ is an $r\times n$ matrix that reconstructs all the columns of $\bA$ from $\bC$. Specifically, $\bR$ corresponds to the \textit{reduced row echelon form (RREF)} of $\bA$, with the zero rows removed.

The storage required for this decomposition changes from $mn$ floating-point numbers to $r(m+n)$ floating-point numbers, which may either reduce or, in some cases, increase memory usage depending on the matrix dimensions and rank.
\end{theoremHigh}

The CR decomposition offers several key advantages: It highlights the independent columns and rows of the matrix, which are crucial for understanding its rank and the structure of its column and row spaces.
For large matrices, the decomposition can approximate the matrix by retaining only the most significant columns and rows, reducing computational complexity.
It simplifies the solution of linear systems $\bA\bx = \bb $ by transforming the problem into a more manageable form using the matrices $ \bC $ and $ \bR $.
Since both $\bC$ and $\bR$ have full rank, the decomposition provides an efficient way to compute the (Moore-Penrose) pseudo-inverse: $ \bA^+ = \bR^+\bC^+ $, where $ \bR^+ $ and $ \bC^+ $ denote the pseudo-inverses of $ \bR $ and $ \bC $, respectively (see Problem~\ref{prob:cr_pseudo}).
The CR decomposition is also useful for analyzing the incidence matrix of a graph. It helps  in studying conserved quantities, such as current flow in electrical circuits, as described by Kirchhoff's Current Law \citep{strang2021three}.

\section{Existence of  CR Decomposition}

The CR decomposition of a matrix can be obtained through the following steps.
Given that the matrix $\bA$ has  rank  $r$, it contains $r$ linearly independent columns. These columns can be identified and collected  in the matrix $\bC$ as follows:
\begin{itemize}
\item If column 1 of $\bA$ is nonzero, include it as a column of $\bC$. 
\item If column 2 of $\bA$ is not a scalar multiple of column 1, include it as a column of $\bC$.
\item If column 3 of $\bA$ is not a linear combination of columns 1 and 2, include it as a column of $\bC$.
\item Continue this process until $r$ linearly independent columns have been selected. If $r$ is not known in advance, continue until all linearly independent columns have been identified. 
\end{itemize}

Once $r$ linearly independent columns are extracted from $\bA$, the CR decomposition can be constructed by interpreting matrix multiplication in terms of the column space.
The product of two matrices, $\bD\in \real^{m\times k}$ and $\bE\in \real^{k\times n}$, results in the matrix $\bA=\bD\bE$. This can be expressed as $\bA=\bD[\be_1, \be_2, \ldots, \be_n] = [\bD\be_1, \bD\be_2, \ldots, \bD\be_n]$. 
In this interpretation, each column of $\bA$ is a linear combination of the columns of $\bD$.

\begin{proof}[of Theorem~\ref{theorem:cr-decomposition}]
Since $\bA$ has rank $r$ and $\bC$ is constructed from $r$ linearly independent columns of $\bA$, the column space of $\bC$ is the same as that of  $\bA$.
Therefore, any other column $\ba_i$ of $\bA$ can be represented as a linear combination of the columns of $\bC$, i.e., there exists a vector $\br_i$ such that $\ba_i = \bC \br_i$, $\forall i\in \{1, 2, \ldots, n\}$. By arranging  these vectors $\br_i$'s as the columns of a matrix $\bR$, we obtain 
$$
\bA = [\ba_1, \ba_2, \ldots, \ba_n] = [\bC \br_1, \bC \br_2, \ldots, \bC \br_n]= \bC \bR.
$$
Thus, the decomposition $
\bA=\bC\bR$ is established, completing the proof.
\end{proof}

\section{Reduced Row Echelon Form (RREF)}\label{section:rref-cr}
In Section~\ref{section:gaussian-elimination} on Gaussian elimination, we introduced the elimination matrix (a lower triangular matrix; see \eqref{equation:elimination_mat}) and the permutation matrix to facilitate transforming  $\bA$ into an upper triangular form. 
Let us now revisit the Gaussian elimination process for a $4\times 4$ square matrix, where $\boxtimes$ denotes a value that is not necessarily zero, and \textbf{boldface} indicates the value has just been changed:
$$
\footnotesize
\begin{sbmatrix}{\bA}
\boxtimes & \boxtimes & \boxtimes & \boxtimes \\
\boxtimes & \boxtimes & \boxtimes & \boxtimes \\
\boxtimes & \boxtimes & \boxtimes & \boxtimes \\
\boxtimes & \boxtimes & \boxtimes & \boxtimes
\end{sbmatrix}
\stackrel{\bE_1}{\longrightarrow}
\begin{sbmatrix}{\bE_1\bA}
\boxtimes & \boxtimes & \boxtimes & \boxtimes \\
\bm{0} & \bm{0} & \bm{\boxtimes} & \bm{\boxtimes} \\
\bm{0} & \bm{\boxtimes} & \bm{\boxtimes} & \bm{\boxtimes} \\
\bm{0} & \bm{\boxtimes} & \bm{\boxtimes} & \bm{\boxtimes}
\end{sbmatrix}
\stackrel{\bP_1}{\longrightarrow}
\begin{sbmatrix}{\bP_1\bE_1\bA}
\boxtimes & \boxtimes & \boxtimes & \boxtimes \\
\bm{0} & \textcolor{mylightbluetext}{\bm{\boxtimes}} & \bm{\boxtimes} & \bm{\boxtimes} \\
\bm{0}  & \bm{0} & \textcolor{mylightbluetext}{\bm{\boxtimes}} & \bm{\boxtimes} \\
0 & \boxtimes & \boxtimes & \boxtimes
\end{sbmatrix}
\stackrel{\bE_2}{\longrightarrow}
\begin{sbmatrix}{\bE_2\bP_1\bE_1\bA}
\boxtimes & \boxtimes & \boxtimes & \boxtimes \\
0 &  \textcolor{mylightbluetext}{\boxtimes} & \boxtimes & \boxtimes \\
0 & 0  & \textcolor{mylightbluetext}{\boxtimes} & \boxtimes \\
0 & \bm{0}  & \bm{0} & \textcolor{mylightbluetext}{\bm{\boxtimes}}
\end{sbmatrix}.
$$

\index{Pivot}
Moreover, Gaussian elimination can also be applied to rectangular matrices. Below, we demonstrate the process for a $4\times 5$ matrix:
$$
\footnotesize
\begin{sbmatrix}{\bA}
\textcolor{mylightbluetext}{2} & \boxtimes & 10 & 9 & \boxtimes\\
\boxtimes & \boxtimes & \boxtimes & \boxtimes & \boxtimes\\
\boxtimes & \boxtimes & \boxtimes & \boxtimes & \boxtimes\\
\boxtimes & \boxtimes & \boxtimes & \boxtimes & \boxtimes\\
\end{sbmatrix}
\stackrel{\bE_1}{\longrightarrow}
\begin{sbmatrix}{\bE_1\bA}
\textcolor{mylightbluetext}{2} & \boxtimes & 10 & 9 & \boxtimes\\
\bm{0} & \bm{0} & \textcolor{mylightbluetext}{\bm{5}} & \bm{6} & \bm{\boxtimes}\\
\bm{0} & \bm{0} & \bm{2} & \bm{\boxtimes} & \bm{\boxtimes}\\
\bm{0} & \bm{0} & \bm{\boxtimes} & \bm{\boxtimes} & \bm{\boxtimes}\\
\end{sbmatrix}
\stackrel{\bE_2}{\longrightarrow}
\begin{sbmatrix}{\bE_2\bE_1\bA}
\textcolor{mylightbluetext}{2} & \boxtimes & 10 & 9 & \boxtimes\\
0 & 0 & \textcolor{mylightbluetext}{5} & 6 & \boxtimes\\
0 & 0 & \bm{0} & \textcolor{mylightbluetext}{\bm{3}} & \bm{\boxtimes}\\
0 & 0 & \bm{0} & \bm{0} & \bm{0}\\
\end{sbmatrix},
$$
\noindent where the numbers highlighted in \textcolor{mylightbluetext}{blue} are \textit{pivots}, as previously defined  (Definition~\ref{definition:pivot})~\footnote{In the context of Gaussian elimination, a pivot element is the first nonzero element in a row when performing row operations to transform a matrix into its row echelon form or reduced row echelon form.}. 
The resulting matrix is referred to as the \textit{row echelon form} of  $\bA$. 
In this example, the fourth row becomes a zero row. To continue, we perform additional row operations to ensure that all entries above the pivots are zero:
$$
\footnotesize
\begin{sbmatrix}{\bE_2\bE_1\bA}
\textcolor{mylightbluetext}{2} & \boxtimes & 10 & 9 & \boxtimes\\
0 & 0 & \textcolor{mylightbluetext}{5} & 6 & \boxtimes\\
0 & 0 & 0 & \textcolor{mylightbluetext}{3} & \boxtimes\\
0 & 0 & 0 & 0 & 0\\
\end{sbmatrix}
\stackrel{\bE_3}{\longrightarrow}
\begin{sbmatrix}{\bE_3\bE_2\bE_1\bA}
\textcolor{mylightbluetext}{2} & \boxtimes & \bm{0}  & \bm{-3}  & \bm{\boxtimes} \\
0 & 0 & \textcolor{mylightbluetext}{5} & 6 & \boxtimes\\
0 & 0 & 0 & \textcolor{mylightbluetext}{3} & \boxtimes\\
0 & 0 & 0 & 0 & 0\\
\end{sbmatrix}
\stackrel{\bE_4}{\longrightarrow}
\begin{sbmatrix}{\bE_4\bE_3\bE_2\bE_1\bA}
\textcolor{mylightbluetext}{2} & \boxtimes & 0 & \bm{0} & \bm{\boxtimes}\\
0 & 0 & \textcolor{mylightbluetext}{5} & \bm{0} & \bm{\boxtimes}\\
0 & 0 & 0 & \textcolor{mylightbluetext}{3} & \boxtimes\\
0 & 0 & 0 & 0 & 0\\
\end{sbmatrix},
$$
\noindent where $\bE_3$ subtracts twice the second row from the first row, while $\bE_4$ adds the third row to the first row and subtracts twice the third row from the second row. To achieve the full \textit{reduced row echelon form (RREF)}, we need to ensure that all pivots are set to 1:
\begin{equation}\label{equation:cr_exp1}
\footnotesize
\begin{sbmatrix}{\bE_4\bE_3\bE_2\bE_1\bA}
\textcolor{mylightbluetext}{2} & \boxtimes & 0 & 0 & \boxtimes\\
0 & 0 & \textcolor{mylightbluetext}{5} & 0 & \boxtimes\\
0 & 0 & 0 & \textcolor{mylightbluetext}{3} & \boxtimes\\
0 & 0 & 0 & 0 & 0\\
\end{sbmatrix}
\stackrel{\bE_5}{\longrightarrow}
\begin{sbmatrix}{\bE_5\bE_4\bE_3\bE_2\bE_1\bA}
\textcolor{mylightbluetext}{\bm{1}} & \bm{\boxtimes} & \bm{0} & \bm{0} & \bm{\boxtimes}\\
\bm{0} & \bm{0} & \textcolor{mylightbluetext}{\bm{1}} & \bm{0} & \bm{\boxtimes}\\
\bm{0} & \bm{0} & \bm{0} & \textcolor{mylightbluetext}{\bm{1}} & \bm{\boxtimes}\\
0 & 0 & 0 & 0 & 0\\
\end{sbmatrix},
\end{equation}
\noindent where $\bE_5$ scales the pivots so that they equal 1. 
Unlike the transformation matrices used in LU decomposition, which are generally lower triangular, the transformation matrices $\bE_1, \bE_2, \ldots, \bE_5$ may also include permutation matrices or other types of matrices. 
The resulting matrix is the \textit{reduced row echelon form (RREF)} of $\bA$, characterized by having pivots equal to 1 and zeros above the pivots.

For a general matrix $\bA\in\real^{m\times n}$, let its row echelon form (with zeros above the pivots) be denoted as $\bF$. To transform $\bF$ into the RREF of $\bA$, we apply a sequence of transformations $\bE_1, \bE_2, \ldots, \bE_r$, defined as: 
$$
\bE_i=
\bE_{i;\eta} =
\begin{bmatrixscript}
1&&&&&&\\
&\ddots&&&&&\\
&& 1&&&&\\
&&& \eta &&&\\
&&&& 1 &&\\
&&&&& \ddots &\\
&&&&&&1\\
\end{bmatrixscript}
=
\bI + (\eta-1)\be_i\be_i^\top,
\gap 
\text{with $1\leq i\leq r$},
$$
where $\be_i$ is the $i$-th unit basis vector, $r$ is the rank of $\bA$ (i.e., the number of nonzero rows in $\bF$), and $\eta$ is the inverse of the $i$-th pivot in $\bF$. Each $\bE_i$ is invertible, with $\bE_{i;\eta}^{-1} = \bE_{i;\eta^{-1}} = \bI + (\frac{1}{\eta}-1)\be_i\be_i^\top$.
The reduced row echelon form of $\bA$  can then be obtained by $(\bE_{r}\bE_{r-1}\ldots\bE_1\bF)$.

We formally define the reduced row echelon form (RREF) of a matrix as follows:
\begin{definition}[Reduced row echelon form, RREF]
Let $\bA\in\real^{m\times n}$. The \textit{row echelon form (REF)} of $\bA$ satisfies the following conditions:
\begin{enumerate}
\item[1.] The leading nonzero entry (called a pivot, as defined in Definition~\ref{definition:pivot}) of the $(i+1)$-th row appears  to the right of the leading nonzero entry of the $i$-th row.
\item[2.] All entries below a pivot in a given column are zeros.
\item[3.] Any row that contains only zeros is positioned at the bottom of the matrix.
\end{enumerate} 
Although Property 2 is a consequence of  Property 1, we include it here for emphasis.
If a matrix in row echelon form satisfies the following additional conditions, it is said to be 
in \textit{reduced row echelon form (RREF)}:
\begin{enumerate}
\item[4.] The leading nonzero entry (pivot) in each row is equal to 1.
\item[5.] All entries above each pivot are zeros.
\end{enumerate}
\end{definition}
\begin{exercise} Let $\bA\in\real^{m\times  n}$ be any matrix. Show that there exists a sequence of row transformations $\bE_1, \bE_2, \ldots, \bE_k$ such that $\bB = \bE_k\bE_{k-1}\ldots\bE_1\bA$ is in reduced row echelon form. \textit{Hint: Use induction}.
\end{exercise}

The pivots  (in its RREF) play a key role in estimating the rank of a matrix.
\begin{lemma}[Rank and pivots]\label{lemma:rank-is-pivots}
The rank of a matrix $\bA$ is equal to the number of pivots (in its reduced row echelown form).
Consequently, the rank of a matrix is the same as the rank of its RREF.
\end{lemma}

\begin{proof}[of Lemma~\ref{lemma:rank-is-pivots}]
According to Proposition~\ref{proposition:rowspa_rowele}, the row space of $\bA$ is identical to the row space  of its RREF. Since the rank of the RREF is defined as the number of its pivots, it follows that the rank of $\bA$ is also equal to the number of pivots.
\end{proof}

We now show that the RREF and the CR decomposition are closely related.
\begin{lemma}[RREF in CR\index{Reduced row echelon form}]\label{lemma:r-in-cr-decomposition}
The reduced row echelon form of the matrix $\bA$, excluding zero rows, corresponds to the matrix $\bR$ in the CR decomposition.
\end{lemma}
\begin{proof}[Informal proof of Lemma~\ref{lemma:r-in-cr-decomposition}]
Informally, using the example provided earlier in  \eqref{equation:cr_exp1}, we express the matrix $\bA$ as:
$$
\bE_5\bE_4\bE_3\bE_2\bE_1\bA = \bR_0 \quad \longrightarrow \quad  \bA = (\bE_5\bE_4\bE_3\bE_2\bE_1)^{-1}\bR_0.
$$
We observe that  columns 1, 3, and 4 of $\bR_0$ each contain a single nonzero entry, which is equal to 1. This observation allows us to construct a matrix $\bC$ (identical to the  ``column matrix" in the CR decomposition) whose first three columns are equal to columns 1, 3, and 4 of  $\bA$, i.e., $\bC=[\ba_1, \ba_3, \ba_4]$. Additionally, because the last row of $\bR_0$ consists entirely of zeros, the last row of $\bR_0$ can be safely disregarded in computations. 
Notably, this matrix $\bC$ is unique in its ability to reconstruct columns 1, 3, and 4 of $\bA$, as the pivots of $\bR_0$ are all equal to 1. 
Thus, we obtain the CR decomposition:
$
\bA=\bC\bR
$.
\end{proof}

\index{Uniqueness}
Nest, we present a rigorous proof of the uniqueness of the RREF of a matrix.
\begin{theorem}[Uniqueness of RREF]\label{theorem:unique_rref}
Let $\bA\in\real^{m\times n}$ be any matrix of rank $r$. 
Suppose $\bX$ and $\bY$ are two reduced row echelon forms of $\bA$, obtained by applying two sequences of elementary row operations $\bE_1, \bE_2, \ldots, \bE_p$ and $\bF_1, \bF_2, \ldots, \bF_q$, respectively, where
$$
\bX = {\bE_p\ldots \bE_2\bE_1}\bA =\bE\bA
\gap\text{and}\gap
\bY = {\bF_q\ldots \bF_2\bF_1}\bA =\bF\bA.
$$
Then, the two reduced row echelon forms are identical, i.e., $\bX=\bY$ and $\bE=\bF$.

\end{theorem}
\begin{proof}[of Theorem~\ref{theorem:unique_rref}]
Let $\bB = \bE\bF^{-1}=\bE_p\ldots\bE_2\bE_1\bF_1^{-1}\bF_2^{-1}\ldots \bF_q^{-1}$. Then we have $\bX = \bB\bY$ and $\bY = \bB^{-1}\bX$.
The  $i$-th column of $\bX$ and $\bY$ can be expressed as $\bx_i =\bX\be_i$ and $\by_i = \bY\be_i$, respectively, where $\be_i$ represents the $i$-th standard basis vector in $\real^n$.
\paragraph{Zero columns match.}If $\bx_i=\bzero$, then $\by_i =\bB^{-1}\bx_i =\bzero$. Similarly, if  $\by_i=\bzero$, then $\bx_i =\bB\by_i =\bzero$. 
Thus, the zero columns in $\bX$ and $\bY$ are aligned.
Without loss of generality, we assume that $\bX$ and $\bY$ contain no zero columns for the rest of the analysis.
\paragraph{First column.} Since we assume $\bX$ and $\bY$ do not contain zero columns, their first columns must be  $\bx_1=\by_1=\be_1$. This also implies the first column of $\bB$ is $\be_1$.
We refer to columns in $\bX$ or $\bY$ that contain pivots as \textit{pivot columns}, and those do not  as \textit{non-pivot columns}.

\index{Pivot columns}
\index{Non-pivot columns}

\paragraph{Non-pivot columns between the first and second pivot columns.}
Suppose the indices of the pivot columns in $\bX$ are $\{i_1, i_2, \ldots, i_r\}$, and the indices of the pivot columns in $\bY$ are $\{j_1, j_2, \ldots, j_r\}$. According to Lemma~\ref{lemma:rank-is-pivots}, there are $r$ pivot columns. 
And we have already shown that $i_1=j_1=1$. Then for $k\in\{2,3,\ldots, j_2-1\}$, we suppose $\by_k=\lambda \be_1$ for some nonzero $\lambda$. We have
$$
\bx_k = \bB\by_k = \bB\lambda \be_1=\lambda\bb_1  =\lambda\be_1
\quad\implies\quad \bx_k = \by_k, \gap \forall k\in\{2,3,\ldots, j_2-1\}.
$$
Conversely, suppose $\bx_k=\lambda \be_1$ for some nonzero $\lambda$ with $k\in\{2,3,\ldots, i_2-1\}$. In this case, we also find that $\bx_k=\by_k$. This implies the non-pivot columns $k=2, 3, \ldots, j_2-1$ of $\bX$ and $\bY$ are the same, and moreover $i_2=j_2$ (the indices of the second pivot columns in $\bX$ and $\bY$ are the same, and $\bx_{j_2} = \by_{j_2}=\be_2$).

To conclude, we have demonstrated that the first $j_2$ columns of  $\bX$ and $\bY$ are identical.
\paragraph{Non-pivot columns between the second and third pivot columns.} Since $\bx_{j_2} = \by_{j_2} = \be_2$, we have $\bx_{j_2}=\bB\by_{j_2} = \bB\be_2 = \bb_2=\be_2$, i.e., the second column of $\bB$ is $\be_2$. On the other hand, considering the non-pivot columns of $k\in\{j_2+1, j_2+2, \ldots, j_3-1\}$ in $\bY$, we assume $\by_k = \lambda_1\be_1+\lambda_2\be_2$. Then we have 
$$
\begin{aligned}
\bx_k = \bB\by_k =\bB(\lambda_1\be_1+&\lambda_2\be_2) =\lambda_1\bb_1+\lambda_2\bb_2 =\lambda_1\be_1+\lambda_2\be_2 \\
&\quad\implies\quad \bx_k=\by_k,\gap\forall  
k\in\{j_2+1, j_2+2, \ldots, j_3-1\}.
\end{aligned}
$$
Conversely, suppose $\bx_k = \lambda_1\be_1+\lambda_2\be_2$ with $k\in\{i_2+1, i_2+2, \ldots, i_3-1\}$, we also have $\bx_k=\by_k$. 
Consequently, the non-pivot columns $k=j_2+1, j_2+2, \ldots, j_3-1$ of $\bX$ and $\bY$ are the same, and also $j_3=i_3$ (the indices of the third pivot columns in $\bX$ and $\bY$ are the same, and $\bx_{j_3} = \by_{j_3}=\be_3$).

By repeating this argument for all pivot positions, we show that all corresponding columns of 
$\bX$ and $\bX$ are equal, completing the proof.
\end{proof}


\begin{exercise}[Determinant of RREF]
Show that the determinant of a matrix $\bA$ is nonzero if and only if its RREF is the identity matrix $\bI$.
\end{exercise}
Using the above result and the multiplicative property of determinants (i.e., $\det(\bA\bB)=\det(\bA)\det(\bB)$), the determinant of $\bA$ can be determined by tracking the elementary row operations performed during the process of transforming $\bA$ into its RREF. 

\index{Rank decomposition}
In summary, we begin by calculating the reduced row echelon form of matrix $\bA$, denoted as $rref(\bA)$. 
Then, in the CR decomposition, the matrix $\bC$  is formed by selecting from $\bA$ only those columns that correspond to pivot columns in $rref(\bA)$.  
Simultaneously, the factor $\bR$ is obtained by removing all zero rows from $rref(\bA)$. 
This process represents a special case of rank decomposition (Theorem~\ref{theorem:rank-decomposition}), but it is notable because it explicitly involves the RREF. Hence, we introduce it here due to its specific relevance.

An important property of $\bR$  is that  a subset of its  $r$ columns, each containing a pivot, together form an $r\times r$ identity matrix. It's worth reiterating that we can obtain this matrix $\bR$ simply by eliminating the zero rows from the RREF. As noted in \citet{strang2021every}, a notation for the RREF that retains the zero rows is denoted by $\bR_0$:
$$
\bR_0 = rref(\bA)=
\begin{bmatrix}
\bR \\
\bzero
\end{bmatrix}=
\begin{bmatrix}
\bI_r & \bF \\
\bzero & \bzero
\end{bmatrix}\bP,
$$
where the $n\times n$ permutation matrix $\bP$ arranges the columns of the $r\times r$ identity matrix $\bI_r$ into their correct positions, aligning them with the first $r$ linearly independent columns of the original matrix $\bA$.

\section{Rank Decomposition}

We previously noted that the CR decomposition is a special case of rank decomposition. 
We formally prove that such a decomposition exists for any matrix.

\index{Rank decomposition}
\begin{theoremHigh}[Rank decomposition]\label{theorem:rank-decomposition}
Let $\bA \in \real^{m \times n}$ be any matrix of rank $r$. Then $\bA$ can be factored into what is known as the \textit{rank decomposition} as follows:
$$
\underset{m\times n}{\bA }= \underset{m\times r}{\bD}\gapthree  \underset{r\times n}{\bF},
$$
where  $\bD \in \real^{m\times r}$ and $\bF \in \real^{r\times n}$ both have (full) rank $r$. 
\end{theoremHigh}
\begin{proof}[of Theorem~\ref{theorem:rank-decomposition}]
From the ULV decomposition in Theorem~\ref{theorem:ulv-decomposition}, we can express $\bA$ as 
$
\bA = \bU \begin{bmatrixscript}
\bL & \bzero \\
\bzero & \bzero 
\end{bmatrixscript}\bV.
$
Let $\bU_0 = \bU_{:,1:r}$ and $\bV_0 = \bV_{1:r,:}$, i.e., $\bU_0$ comprises  the first $r$ columns of $\bU$, and $\bV_0$ consists of  the first $r$ rows of $\bV$. 
Thus,  $\bA$  can also be written as:  $\bA = \bU_0 \bL\bV_0$, where $\bU_0 \in \real^{m\times r}$ and $\bV_0\in \real^{r\times n}$. This is also referred to as the reduced ULV decomposition. Let \{$\bD = \bU_0\bL$ and $\bF =\bV_0$\}  or \{$\bD = \bU_0$ and $\bF =\bL\bV_0$\}, we obtain a valid rank decomposition of $\bA$.
\end{proof}

The rank decomposition is \textbf{not unique}. In fact, using elementary row and column operations, we can also write:
$
\bA = 
\bE_1
\begin{bmatrixscript}
\bZ & \bzero \\
\bzero & \bzero 
\end{bmatrixscript}
\bE_2,
$
where $\bE_1 \in \real^{m\times m}$ and $\bE_2\in \real^{n\times n}$ are products of nonsingular elementary row and column operations, and $\bZ\in \real^{r\times r}$. There exist many possible choices for $\bE_1,\bE_2,$ and $\bZ$. 
When $\bZ=\bI_r$, where $r$ is the rank of  $\bA$, this decomposition is known as the \textit{Smith decomposition or Smith form} of $\bA$ \citep{bernstein2009matrix}.
By using similar constructions as in the proof above, we can derive alternative rank decompositions from other matrix factorizations, such as SVD, URV, CR, and CUR.
However, we can also establish a general relationship between different rank decompositions using the following corollary.
\index{Decomposition: Smith}
\index{Smith decomposition}
\begin{corollary}[Connection between rank decompositions]\label{corollary:connection-rank-decom}
Let $\bA=\bD_1\bF_1=\bD_2\bF_2\in\real^{m\times n}$ be  two rank decompositions of $\bA$. Then there exists a nonsingular matrix $\bP$ such that 
$$
\bD_1 = \bD_2\bP
\qquad
\text{and}
\qquad 
\bF_1 = \bP^{-1}\bF_2.
$$
More generally, given $\bA,\bB\in\real^{m\times n}$,  $\bA$ and $\bB$ are biequivalent~\footnote{$\bA$ and $\bB$ are \textit{biequivalent} if there exist nonsingular $\bZ_1\in\real^{m\times m}$ and  $\bZ_2\in\real^{n\times n}$ such that $\bA=\bZ_1\bB\bZ_2$.} if and only if $\bA$ and $\bB$ share the same Smith form.
\end{corollary}
\begin{proof}[of Corollary~\ref{corollary:connection-rank-decom}]
Given  $\bD_1\bF_1=\bD_2\bF_2$, postmultiplying by $\bF_1^\top$ yields $\bD_1\bF_1\bF_1^\top=\bD_2\bF_2\bF_1^\top$. Since $\rank(\bF_1\bF_1^\top)=\rank(\bF_1)=r$, $\bF_1\bF_1^\top$ is a square matrix with full rank, hence nonsingular. 
Therefore, we have  $\bD_1=\bD_2\bF_2\bF_1^\top(\bF_1\bF_1^\top)^{-1}$. Let $\bP=\bF_2\bF_1^\top(\bF_1\bF_1^\top)^{-1}$, we have $\bD_1=\bD_2\bP$ and $\bF_1 = \bP^{-1}\bF_2$. The second part of the corollary can be proven similarly.
\end{proof}

\index{Fundamental theorem}
\index{Pseudo-inverse}
\index{Idempotent}
\index{Basis}
\section{Application: Idempotent Matrix and Matrix Rank}
The CR decomposition or rank decomposition plays a pivotal role in proving several essential theorems in linear algebra.
For instance, it is instrumental in establishing the existence of the pseudo-inverse; it it helps determine a basis for the four fundamental subspaces in linear algebra \citep{lu2021numerical}.

Moreover, the CR factorization finds practical applications in data analysis and computational problem-solving. For instance, it proves valuable in solving  least squares problems by reducing the system to a minimal set of variables, thereby eliminating redundancy.

The CR decomposition is also a powerful tool for analyzing the rank characteristics of idempotent matrices. Its utility in orthogonal projections is further explored in the Appendix of \citet{lu2021numerical}.
\begin{proposition}[Rank and trace of an idempotent matrix\index{Trace}]\label{proposition:rank-of-symmetric-idempotent2_tmp}
Let $\bA$ be an idempotent matrix (i.e., $\bA^2 = \bA$). Then the rank of $\bA$ is equal to its trace.
\end{proposition}
\begin{proof}[of Proposition~\ref{proposition:rank-of-symmetric-idempotent2_tmp}]
Consider an  $n\times n$ idempotent matrix $\bA$ of rank $r$. By the CR decomposition, we can express $\bA$ as $\bA = \bC\bR$, where $\bC\in\real^{n\times r}$ and $\bR\in \real^{r\times n}$ are both of full rank $r$. Therefore,
$$
\begin{aligned}
\bA^2 = \bA
\quad\implies \quad
\bC\bR\bC\bR = \bC\bR
\quad\implies \quad 
\bR\bC\bR =\bR
\quad\implies \quad
\bR\bC =\bI_r,
\end{aligned}
$$ 
where $\bI_r$ denotes the $r\times r$ identity matrix. Consequently,
$
\trace(\bA) = \trace(\bC\bR) =\trace(\bR\bC) = \trace(\bI_r) = r, 
$
which corresponds to the rank of $\bA$. This equality holds due to the invariance of the trace under cyclic permutations.
\end{proof}

On the other hand,  we previously established a fundamental theorem in linear algebra using the UTV framework, which demonstrated that the row rank and column rank of any matrix are equal  (Theorem~\ref{theorem:equal-dimension-rank}).
The CR decomposition offers an alternative explanation of this result.
\begin{proof}[{of Theorem~\ref{theorem:equal-dimension-rank}, the third way}]
Consider the CR decomposition of  $\bA = \bC\bR$, where $\bR$ can be expressed as  $\bR = [\bI_r, \bF]\bP$, and $\bP$ is an $n\times n$ permutation matrix used to arrange the columns of the $r\times r$ identity matrix $\bI_r$ in their appropriate positions. 
It is straightforward to verify that the $r$ rows of $\bR$ are linearly independent due to the nonsingular submatrix  $\bI_r$. Therefore,  the row rank of $\bR$ is $r$. 

First, by the definition of the CR decomposition, the $r$ columns of $\bC$ are selected  from $r$ linearly independent columns of $\bA$, and the column rank of $\bA$ is $r$. Furthermore,
\begin{itemize}
\item Since $\bA=\bC\bR$, every row of $\bA$ can be represented as a  linear combinations of the rows of $\bR$. Hence, the row space of $\bA$ is contained within the row space of $\bR$: $\cspace(\bA^\top)\subseteq \cspace(\bR^\top)$.

\item From  $\bA=\bC\bR$, we also have $(\bC^\top\bC)^{-1}\bC^\top\bC\bR = (\bC^\top\bC)^{-1}\bC^\top\bA$, simplifying to $\bR = (\bC^\top\bC)^{-1}\bC^\top\bA$ (Because $\bC$ has full column rank $r$, $\bC^\top\bC$ is nonsingular). 
Consequently, the rows of $\bR$ are  linear combinations of the rows of $\bA$, meaning the row space of $\bR$ is contained within the row space of $\bA$: $\cspace(\bR^\top)\subseteq \cspace(\bA^\top)$.
\end{itemize}
By this ``sandwich" argument, the row spaces of  $\bA$ and $\bR$ are equal, and thus their row ranks are equal: $\cspace(\bA^\top)= \cspace(\bR^\top)$.

Since the column rank of $\bA$ is also $r$ by the definition of the CR decomposition, it follows that both the row rank and column rank of $\bA$ are equal to $r$.
\end{proof}

\begin{problemset}
\item Discuss what rank number $r$ in Theorem~\ref{theorem:cr-decomposition} and Theorem~\ref{theorem:rank-decomposition} can reduce storage requirements.
\item Determine the reduced row echelon form and the CR decomposition for the matrix 
$
\bA = 
\begin{bmatrixscript}
1 & 3 & 2 \\
3 & 7 & 6 \\
4 & 5 & 8
\end{bmatrixscript}.
$

\item Apply the RREF process to the matrix
$
\bA = \begin{bmatrixscript}
1 & 2 & 1 & 1 \\
1 & 4 & 2 & 3 \\
1 & 1 & 2 & -1 \\
-3 & -1 & 4 & 0
\end{bmatrixscript}.
$

\item \label{prob:cr_pseudo} Find the pseudo-inverse of a matrix $\bA$ using its CR decomposition.

\item \label{problem:aug_lin} Show that the solution of the linear system $\bA\bx=\bb$ remains unchanged if the same sequence of elementary row transformations  is applied to both $\bA$ and $\bb$. Consequently, the solution can be revealed by finding the RREF of the augmented matrix $[\bA,\bb]$.

\item Following Problem~\ref{problem:aug_lin}, show that the two linear systems $\bA_1\bx=\bb_1$ and $\bA_2\bx=\bb_2$ have the same set of solutions if and only if $[\bA_1,\bb_1]$ and $[\bA_2,\bb_2]$ have the same RREF.

\item Show that if a system of linear equations has two distinct solutions, then it must have infinitely many solutions.
\item Show that if a linear system $\bA\bx=\bb$ has more than one solution, 
then the corresponding homogeneous system $\bA\bx=\bzero$  also has nontrivial solutions.

\item A system of linear equations with fewer equations than unknowns is sometimes referred to as an \textit{underdetermined} system. Provide an example of an inconsistent underdetermined system of two equations in three unknowns (If there is at least one solution, the linear system is called \textit{consistent}; otherwise, it is called \textit{inconsistent}).

\item Suppose an underdetermined system is consistent. Explain why such a system must have an infinite number of solutions.

\item A system of linear equations with more equations than unknowns is sometimes referred to as an \textit{overdetermined} system. Discuss the conditions under which such a system can be consistent.

\item Two matrices are called \textit{row equivalent} if there is a sequence of elementary row operations that transforms one matrix into the other. Show that if matrices $\bA$ and $\bB$ are row equivalent, they have the same RREF.

\item Let $\bA\bx=\bb$ be a consistent system where $\bA\in\real^{m\times n}$. Show that $\bA$ has $m$ pivot columns. Furthermore, let  $m=n$; show that the RREF of $\bA$ is the identity matrix.

\item 
Let $ \bA=\begin{bmatrixscript}
a & b \\
c & d
\end{bmatrixscript}$ be any $ 2 \times 2 $  nonsingular matrix. 
Show that there exists an nonsingular matrix $\bS$ such that
$
\bS\bA = \begin{bmatrixscript}
1 & 0 \\
0 & ad - bc
\end{bmatrixscript},
$
where $\bS$ is the product of at most four elementary matrices of the form $ \bE_{i,j;\alpha}=\bI+\alpha\be_i\be_j^\top\in\real^{n\times n}$.

\item Let $ \bA\in\real^{n\times n} $ be nonsingular. Show that there is a matrix $\bS$ such that
$
\bS\bA = \begin{bmatrixscript}
\bI_{n-1} & 0 \\
0 & d
\end{bmatrixscript},
$
where $d = \det(\bA) $, and  $\bS $ is again a product of elementary matrices of the form $ \bE_{i,j;\alpha} =\bI+\alpha\be_i\be_j^\top\in\real^{n\times n}$.

\end{problemset}

\chapter{Skeleton/CUR and Interpolative Decomposition}

\index{Skeleton}
\section{Skeleton/CUR Decomposition}
The CR decomposition utilizes  actual columns of a matrix, whereas the \textit{skeleton decomposition} extends this concept by incorporating both  actual columns and rows.
\index{Decomposition: Skeleton}
\index{Linearly independent}
\index{Data storage}
\begin{theoremHigh}[Skeleton/CUR decomposition]\label{theorem:skeleton-decomposition}
Any rank-$r$ matrix $\bA \in \real^{m \times n}$ can be decomposed as 
$$
\underset{m\times n}{\bA }= 
\underset{m\times r}{\bC} \gapthree \underset{r\times r}{\bU^{-1} }\gapthree \underset{r\times n}{\bR},
$$
where $\bC$ contains some $r$ linearly independent columns of $\bA$, $\bR$ contains some $r$ linearly independent rows of $\bA$, and $\bU$ is the nonsingular submatrix formed by the intersection of these selected rows and columns.
\begin{itemize}
\item The storage requirement for this decomposition may be reduced (or potentially increased) compared to storing the full matrix, from  $mn$ floating-point numbers to $r(m+n)+r^2$ floating-point numbers.
\item Alternatively, if we only record the indices of the selected rows and columns, it requires $mr$ and $nr$ floating-point numbers for storing $\bC$ and $\bR$, respectively. Additionally,  $2r$ integers are required to store the positions of the selected columns in $\bC$ and rows in $\bR$ within $\bA$, which allows reconstruction of $\bU$ from $\bC$ and $\bR$.
\end{itemize}
\end{theoremHigh}

The skeleton decomposition is also  referred to as the \textit{CUR decomposition}, named after its component. 
Compared to the singular value decomposition (SVD), CUR offers significant  advantages in terms of \textit{reification} and \textit{interpretability}. While SVD relies on artificial singular vectors that may not accurately reflect physical realities, CUR uses  actual columns and rows from the original  matrix, which makes it more interpretable and better  aligned with the structure of the original data \citep{mahoney2009cur}. Moreover, CUR preserves sparsity in the underlying data, making it particularly suitable for applications involving sparse matrices.

On the other hand, like SVD, CUR is a versatile tool widely used across various domains for tasks like data compression, feature extraction, and data analysis. It provides a computationally efficient way to approximate matrices, making it well-suited for handling large-scale datasets \citep{mahoney2009cur, an2012large, lee2008cur+}. 
For example, CUR reduces the storage and computational requirements by selecting only a subset of the original matrix's rows and columns. This results in a low-rank approximation that retains the essential information. It is particularly useful in numerical linear algebra for tasks like solving linear systems, eigenvalue problems, and matrix inversion.
CUR is also employed for image compression and analysis. By approximating the original image matrix with a lower-dimensional representation, CUR reduces storage costs while preserving key visual features.
In machine learning, CUR is effective for dimensionality reduction, feature extraction, and data representation, which can enhance the efficiency of machine learning algorithms and reduce computational overhead.
In collaborative filtering, CUR approximates large user-item interaction matrices in recommendation systems, improving scalability and efficiency.
CUR decomposition can also be extended to higher-dimensional arrays (tensors) for applications in multi-linear algebra and data analysis, enabling the processing of complex datasets \citep{kishore2017literature}.

An illustration of CUR decomposition is shown in Figure~\ref{fig:skeleton}, where  \textcolor{mydarkyellow}{yellow} vectors denote the linearly independent columns of $\bA$, and \textcolor{mydarkgreen}{green} vectors denote the linearly independent rows of $\bA$. Specifically, given  index vectors $\sI$ and $\sJ$, both of size $r$, containing the indices of rows and columns selected from $\bA$ to form $\bR$ and $\bC$, respectively, the submatrix $\bU$ can be expressed as $\bU=\bA[\sI,\sJ]$ using  Matlab-style notation.

\begin{figure}[h]
\centering
\includegraphics[width=0.7\textwidth]{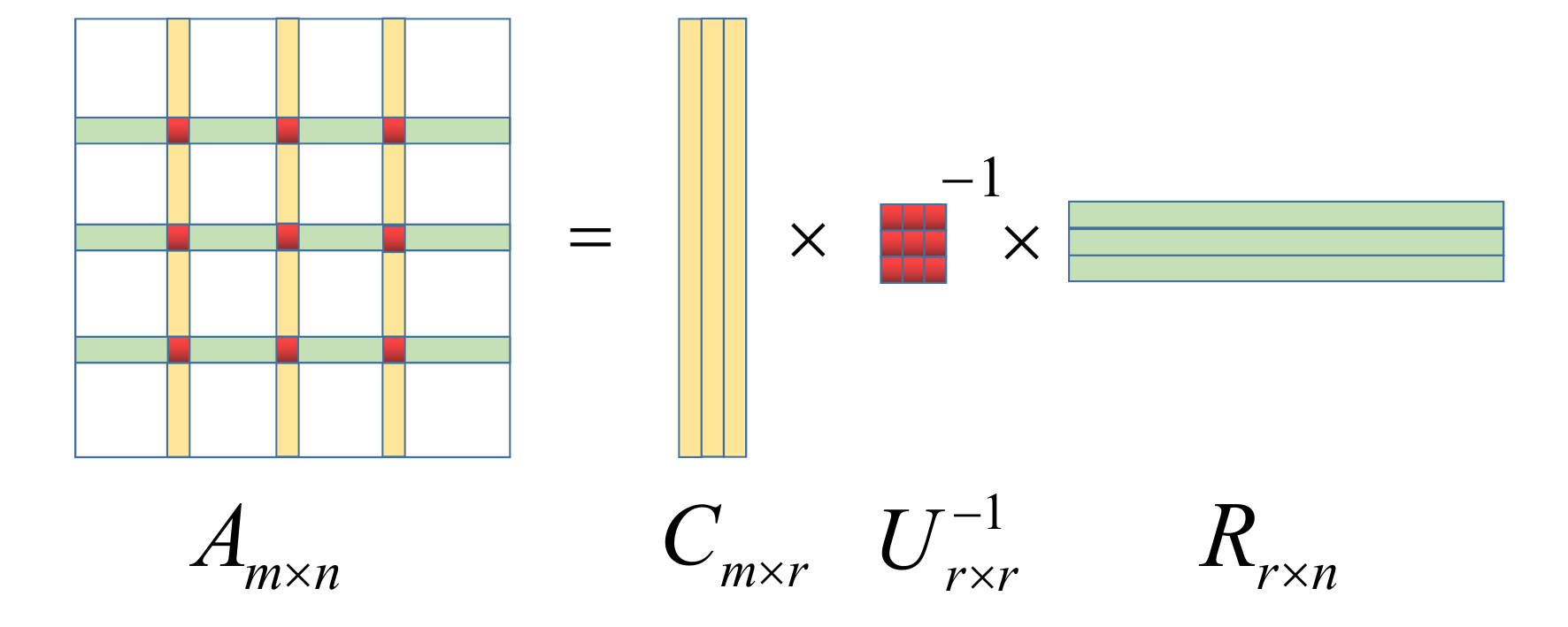}
\caption{Demonstration of the skeleton (CUR) decomposition of a matrix.}
\label{fig:skeleton}
\end{figure}

\index{Row space}
\index{Nonsingular matrix}
\section{Existence of the Skeleton Decomposition}
In Theorem~\ref{theorem:equal-dimension-rank}, we established that the row rank and  column rank of a matrix are equal. In other words, the dimension of the column space is equal to the dimension of the row space. This property is essential for the existence of the skeleton decomposition. 
We  now prove the existence of the skeleton decomposition. The proof is straightforward and relies on fundamental concepts from linear algebra.
\begin{proof}[of Theorem~\ref{theorem:skeleton-decomposition}]
The proof is based on demonstrating the existence of a nonsingular matrix $\bU$, which is central to the skeleton decomposition.

\paragraph{Existence of such a nonsingular matrix $\bU$.} Since the matrix $\bA$ is of rank $r$, we can select $r$ linearly independent columns from $\bA$. 
Let these columns be $\ba_{i1}, \ba_{i2}, \ldots, \ba_{ir}$ and form the $m\times r$ matrix $\bN=[\ba_{i1}, \ba_{i2}, \ldots, \ba_{ir}] \in \real^{m\times r}$. 
The column space of $\bN$ has  dimension $r$, which implies, by Theorem~\ref{theorem:equal-dimension-rank}, that its row space  also has dimension  $r$.
From the rows of $\bN$, we select $r$ linearly independent rows $\bn_{j1}^\top,\bn_{j2}^\top, \ldots, \bn_{jr}^\top $ to construct an $r\times r$ matrix $\bU = [\bn_{j1}^\top; \bn_{j2}^\top; \ldots; \bn_{jr}^\top]\in \real^{r\times r}$. 
Applying Theorem~\ref{theorem:equal-dimension-rank}  again,  the column space of $\bU$  also has  dimension  $r$, meaning $\bU$ has  $r$ linearly independent columns. Thus, $\bU$ is such a nonsingular matrix of size $r\times r$.

\paragraph{Main proof.}
Upon identifying a nonsingular $r\times r$ matrix $\bU$ within $\bA$, we proceed to establish the skeleton decomposition.
Let $\bU=\bA[\sI,\sJ]$, where $\sI$ and $\sJ$ are index vectors of size $r$ representing the selected rows and columns. Since $\bU$ is a nonsingular matrix, its columns  are linearly independent. Thus, the columns of the matrix $\bC$, formed by selecting the same  $r$ columns from $\bA$ are also linearly independent: $\bC=\bA[:,\sJ]$. Here, the matrix $\bC$ is equivalent to the previously constructed $\bN$. 

Because the rank of  $\bA$ is $r$, any column $\ba_i$ of $\bA$ can be expressed as a linear combination of the columns of $\bC$. Specifically, there exists a vector $\bx$ such that $\ba_i = \bC \bx$, for all $ i\in \{1, 2, \ldots, n\}$. Let $r$ rows (entries) of $\ba_i\in\real^n$ corresponding to the row entries of $\bU$ be $\br_i \in \real^r$ for all $i\in \{1, 2, \ldots, n\}$ (i.e., $\br_i$ contains $r$ entries of $\ba_i$). That is, select the $r$ entries of $\ba_i$'s corresponding to the entries of $\bU$ as follows:
$$
\bA = [\ba_1,\ba_2, \ldots, \ba_n]\in \real^{m\times n} \qquad \longrightarrow \qquad
\bA[\sI,:]=[\br_1, \br_2, \ldots, \br_n] \in \real^{r\times n}.
$$
Since $\ba_i = \bC\bx$, $\bU$ is a submatrix inside $\bC$, and $\br_i$ is a subvector inside $\ba_i$, we have $\br_i = \bU \bx$, which states that $\bx = \bU^{-1} \br_i$. Thus, for every $i\in\{1,2,\ldots,n\}$, we have $\ba_i = \bC \bU^{-1} \br_i$. Combining the $n$ columns of such $\br_i$ into $\bR=[\br_1, \br_2, \ldots, \br_n]$, we obtain
$$
\bA = [\ba_1, \ba_2, \ldots, \ba_n] = \bC \bU^{-1} \bR,
$$
from which the result follows.

In summary, the skeleton decomposition is constructed by identifying $r$ linearly independent columns of $\bA$ and placing them into $\bC\in \real^{m\times r}$. Subsequently, we extract an $r\times r$ nonsingular submatrix $\bU$ from $\bC$. The $r$ rows of $\bA$, corresponding to the entries of $\bU$, contribute to reconstruct the columns of $\bA$. This process is visually illustrated in Figure~\ref{fig:skeleton}.
\end{proof}

In the special case where $\bA$ is square and invertible,  the skeleton decomposition simplifies to $\bA=\bC\bU^{-1} \bR$, with $\bC=\bR=\bU=\bA$. Thus, the decomposition essentially reduces to $\bA = \bA\bA^{-1}\bA$.

\paragraph{CR decomposition vs skeleton decomposition.} The CR decomposition and skeleton decompositions share a similar structure and even comparable notation, with $\bA=\bC\bR$ for the CR decomposition and $\bA=\bC\bU^{-1}\bR$ for the skeleton decomposition. 

In both  decompositions, we have the flexibility to select  the \textbf{first} $r$ linearly independent columns to form the matrix $\bC$ (denoted the same way in both decompositions).
Consequently, the $\bC$ matrices in the CR  and skeleton decompositions are identical when the same columns are selected. However, the distinction lies in the interpretation of $\bR$: in the CR decomposition, it represents the reduced row echelon form without zero rows, while in the skeleton decomposition, it corresponds to $r$ linearly independent rows selected directly from $\bA$. 
This difference reflects a fundamental variation in how the two methods conceptualize $\bR$.


To summarize, the construction of the skeleton decomposition involves selecting $r$ linearly independent columns from $\bA$ to form the matrix $\bC\in\real^{m\times r}$. Subsequently,  we extract an $r\times r$ nonsingular submatrix $\bU$ from $\bC$. Finally, we identify the $r$ rows of $\bA$ that correspond to the entries of $\bU$ to form the row matrix $\bR\in\real^{r\times n}$.
This naturally leads to the following question: If matrix $\bA$ has rank $r$, matrix $\bC$ contains $r$ linearly independent columns of $\bA$, and matrix $\bR$ contains $r$ linearly independent rows of $\bA$, is the $r\times r$ ``intersection" of $\bC$ and $\bR$ necessarily invertible?~\footnote{We express our gratitude to Gilbert Strang for raising this  question.}

\begin{corollary}[Nonsingular intersection]\label{corollary:invertible-intersection}
If matrix $\bA \in \real^{m\times n}$ has rank $r$, matrix $\bC$ contains $r$ linearly independent columns of $\bA$, and matrix $\bR$ contains $r$ linearly independent rows of $\bA$, then the $r\times r$ ``intersection" matrix $\bU$ of $\bC$ and $\bR$ is invertible.
\end{corollary}
\begin{proof}[of Corollary~\ref{corollary:invertible-intersection}]
Let $\sI$ and $\sJ$ be the indices of the rows and columns selected from $\bA$ to form $\bR$ and $\bC$, respectively. Then, $\bR$ can be denoted as $\bR=\bA[\sI, :]$, $\bC$ can be represented as $\bC = \bA[:,\sJ]$, and $\bU$ can be denoted as $\bU=\bA[\sI,\sJ]$.

Since $\bC$ contains $r$ linearly independent columns of $\bA$, any column $\ba_i$ of $\bA$ can be represented as $\ba_i = \bC\bx_i = \bA[:,\sJ]\bx_i$ for all $i \in \{1,2,\ldots, n\}$. This implies that the $r$ entries of $\ba_i$ corresponding to the  indices in $\sI$ can be represented by the columns of $\bU$ such that $\ba_i[\sI] = \bU\bx_i \in \real^{r}$ for all $i \in \{1,2,\ldots, n\}$, i.e.,
$$
\ba_i = \bC\bx_i = \bA[:,\sJ]\bx_i \in \real^{m} \qquad \longrightarrow  \qquad
\ba_i[\sI] =\bA[\sI,\sJ]\bx_i= \bU\bx_i \in \real^{r}.
$$ 
Since $\bR$ contains $r$ linearly independent rows of $\bA$, the row rank and column rank of $\bR$ are equal to $r$. Combining the facts above, the $r$ columns of $\bR$ corresponding to the indices in $\sJ$ (i.e., the $r$ columns of $\bU$) are linearly independent. 

Finally, by Theorem~\ref{theorem:equal-dimension-rank}, the row space of $\bU$ also has dimension $r$. This implies that $\bU$ has $r$ linearly independent rows, making it invertible. 
\end{proof}

\section{Interpolative Decomposition (ID)}

A factorization closely related to the skeleton decomposition is the \textit{interpolative decomposition (ID) framework}. We begin by discussing the \textit{column interpolative decomposition}, which we will refer to simply as interpolative decomposition or ID when the context is clear.

The column interpolative decomposition (ID) factorizes a matrix into the product of two matrices: one consisting of selected columns from the original matrix, and the other containing a subset of columns that includes an identity matrix and entries whose magnitudes do not exceed 1. Formally, the details of the column ID are given in the following theorem.

\index{Decomposition: ID}
\index{Linearly independent}
\begin{theoremHigh}[Column interpolative decomposition]\label{theorem:interpolative-decomposition}
Any rank-$r$ matrix $\bA \in \real^{m \times n}$ can be decomposed as 
$$
\underset{m \times n}{\bA} = \underset{m\times r}{\bC} \gapthree  \underset{r\times n}{\bW},
$$
where $\bC\in \real^{m\times r}$ contains $r$ linearly independent columns of $\bA$, and $\bW\in \real^{r\times n}$ is the matrix used to reconstruct $\bA$. The factor $\bW$ contains an $r\times r$ identity submatrix (under a mild column permutation) and satisfies:
$$
\max |w_{ij}|\leq 1, \,\, \forall \, i\in \{1,2,\ldots,r\}, j\in \{1,2,\ldots,n\}.
$$

The storage requirements for this decomposition are reduced (or potentially increased) from $mn$ floating-point numbers to $mr$ and  $(n-r)r$ floating-point numbers for storing $\bC$ and $\bW$, respectively. 
Additionally, $r$ integers are needed to track  the position of each column of $\bC$ within $\bA$.
\end{theoremHigh}

\begin{figure}[h]
	\centering
	\includegraphics[width=0.7\textwidth]{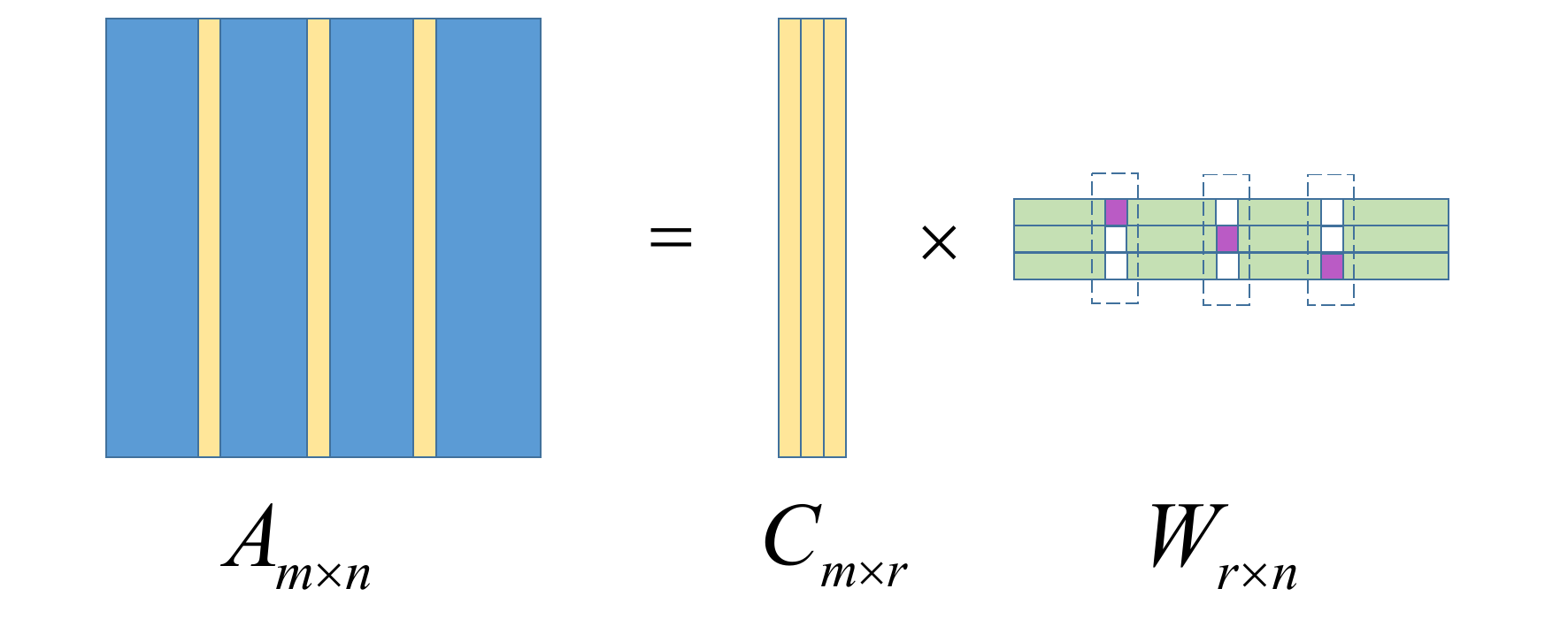}
	\caption{Demonstration of the column ID of a matrix, where the \textcolor{mydarkyellow}{yellow} vectors denote some linearly independent columns of $\bA$, white entries denote zero, and \textcolor{mydarkpurple}{purple} entries denote one.}
	\label{fig:column-id}
\end{figure}

Figure~\ref{fig:column-id} demonstrates a column ID of a matrix, where the \textcolor{mydarkyellow}{yellow} vectors represent some linearly independent columns of $\bA$, and the \textcolor{mydarkpurple}{purple} entries in $\bW$ form an $r\times r$ identity submatrix. The positions of these \textcolor{mydarkpurple}{purple} entries inside $\bW$  correspond to the positions of the \textcolor{mydarkyellow}{yellow} vectors in $\bA$. 
The column ID closely resembles the CR decomposition (Theorem~\ref{theorem:cr-decomposition}): both methods select $r$ linearly independent columns into the first factor, and the second factor contains an  $r\times r$ identity submatrix. 
However, in the CR decomposition, the first $r$ linearly independent columns are specifically chosen, and the identity submatrix corresponds to the pivot columns  (Definition~\ref{definition:pivot}). The second factor in the CR decomposition is derived from the reduced row echelon form (Lemma~\ref{lemma:r-in-cr-decomposition}).
As a result, column ID can be used in similar contexts as the CR decomposition.  
For example, it is useful in proving that the rank of a symmetric idempotent matrix equals its trace (Proposition~\ref{proposition:rank-of-symmetric-idempotent2_tmp}), and in demonstrating the fundamental linear algebra theorem that the column rank equals the row rank of a matrix (Theorem~\ref{theorem:equal-dimension-rank}).
Moreover,  column ID also serves as a special case of the rank decomposition (Theorem~\ref{theorem:rank-decomposition}) and is apparently not unique. The relationships among different column IDs are described in  Corollary~\ref{corollary:connection-rank-decom}.

\index{Rank decomposition}
\index{Matlab-style notation}
\paragraph{Notations for subsequent sections.} Using Matlab-style notation, let  $\sJ_s$ be an index vector of size $r$,  containing the indices of columns selected from $\bA$ to form the matrix $\bC$. Then, $\bC$ can be expressed as $\bC=\bA[:,\sJ_s]$.
The matrix $\bC$ contains the ``skeleton" columns of $\bA$, hence the subscript $s$ in $\sJ_s$. From the ``skeleton" index vector $\sJ_s$, the $r\times r$ identity submatrix inside $\bW$ can be recovered as
$$
\bW[:,\sJ_s] = \bI_r \in \real^{r\times r}.
$$ 
Let $\sJ_r$ denote the indices of the remaining columns of $\bA$, where 
$$
\sJ_s\cap \sJ_r=\varnothing \qquad \text{and}\qquad \sJ_s\cup \sJ_r = \{1,2,\ldots, n\}.
$$
The remaining $n-r$ columns of $\bW$ consist of an $r\times (n-r)$ \textit{expansion matrix}:
$$
\bE = \bW[:,\sJ_r] \in \real^{r\times (n-r)},
$$
where the entries of $\bE$ are called \textit{expansion coefficients}. 
Finally, let $\bP\in \real^{n\times n}$ be a (column) permutation matrix (Definition~\ref{definition:permutation-matrix}) defined as $\bP=\bI_n[:,(\sJ_s, \sJ_r)]$, such that
$$
\bA\bP = \bA[:,(\sJ_s, \sJ_r)] = \left[\bC, \bA[:,\sJ_r]\right],
$$
and 
\begin{equation}\label{equation:interpolatibve-w-ep}
\bW\bP = \bW[:,(\sJ_s, \sJ_r)] =\left[\bI_r, \bE \right] \leadto \bW = \left[\bI_r, \bE \right] \bP^\top.
\end{equation}

\section{Existence of the Column Interpolative Decomposition}\label{section:proof-column-id}
\index{Cramer's rule}
\paragraph{Cramer's rule.}
The proof of the existence of the column ID relies on  \textit{Cramer's rule}, which we will briefly review here; see Problem~\ref{prob:cramer_adj_1}$\sim$\ref{prob:cramer_adj_4} for more details. Consider a system of $n$ linear equations in $n$ unknowns, expressed in matrix form as:
$$
\bM \bx = \bl,
$$
where $\bM\in \real^{n\times n}$ is nonsingular, and $\bx,\bl \in \real^n$. Cramer's rule states that this system has a unique solution, where each unknown is given by:
$$
x_i = \frac{\det(\bM_i)}{\det(\bM)}, \qquad \text{for all}\; i\in \{1,2,\ldots, n\},
$$
where $\bM_i$ is the matrix formed by replacing the $i$-th column of $\bM$ with the column vector $\bl$. 
In a more general setting, consider the matrix equation:
$$
\bM\bX = \bL,
$$
where $\bM\in \real^{n\times n}$ is nonsingular, and $\bX,\bL\in \real^{n\times m}$. Let $\sI=[i_1, i_2, \ldots, i_k]$ and $\sJ=[j_1,j_2,\ldots, j_k]$ be two index vectors, where $1\leq i_1\leq i_2\leq \ldots\leq i_k\leq n$ and $1\leq j_1\leq j_2\leq \ldots\leq j_k\leq n$. Then, $\bX[\sI,\sJ]$ denotes a $k\times k$ submatrix of $\bX$. Let further $\bM_{\bL}(\sI,\sJ)$ be the $n\times n$ matrix formed by replacing the $(i_s)$-th column of $\bM$ with the $(j_s)$-th column of $\bL$ for all $s\in \{1,2,\ldots, k\}$. Then, we have:
$$
\det(\bX[\sI,\sJ]) = \frac{\det\left(\bM_{\bL}(\sI,\sJ)\right)}{\det(\bM)}.
$$ 
When $\sI$ and $\sJ$ are of size 1, this simplifies to:
\begin{equation}\label{equation:cramer-rule-general}
	x_{ij} = \frac{\det\left(\bM_{\bL}(i,j)\right)}{\det(\bM)}.
\end{equation}

With this background, we are now ready to prove the existence of the column ID.
\begin{proof}[of Theorem~\ref{theorem:interpolative-decomposition}]
We  mentioned that the proof relies on  Cramer's rule. To complete the proof, we will show that the entries of $\bW$ can be expressed using the formula in Equation~\eqref{equation:cramer-rule-general}, where the absolute value of the numerator is less than or equal to the denominator. 
Note that the denominator in Equation~\eqref{equation:cramer-rule-general} corresponds to the determinant of a square matrix. Here is the key idea.

\paragraph{Step 1: column ID for a full row rank matrix.}
First, consider a full row rank matrix $\bA$ (which implies $r=m$, $m\leq n$, and $\bA\in \real^{r\times n}$ such that the matrix $\bC\in \real^{r\times r}$ is a square matrix in the column ID $\bA=\bC\bW$ that we want to obtain). Determine the ``skeleton" index vector $\sJ_s$ by 
\begin{equation}\label{equation:interpolative-choose-js}
\boxed{	\sJ_s = \mathop{\arg\max}_{\sJ} \left\{\abs{\det(\bA[:,\sJ])}: \text{$\sJ$ is a subset of $\{1,2,\ldots, n\}$ with size $r=m$} \right\},}
\end{equation}
i.e., $\sJ_s$ is the index vector that is determined by maximizing the magnitude of the determinant of $\bA[:,\sJ_s]$. 
From earlier discussion, there exists a (column) permutation matrix $\bP$ such that:
$$
\bA\bP = 
\begin{bmatrix}
\bA[:,\sJ_s]&\bA[:,\sJ_r]
\end{bmatrix}.
$$
Since $\bC=\bA[:,\sJ_s]$ has full column rank $r=m$, it is then nonsingular. Rewriting  $\bA$:
$$
\begin{aligned}
\bA
&=\begin{bmatrix}
\bA[:,\sJ_s] &\bA[:,\sJ_r]
\end{bmatrix}\bP^\top
= 
\bA[:,\sJ_s]
\bigg[
\bI_r \gap \bA[:,\sJ_s]^{-1}\bA[:,\sJ_r]
\bigg]
\bP^\top\\
&= \bC 
\underbrace{\begin{bmatrix}
\bI_r & \bC^{-1}\bA[:,\sJ_r]
\end{bmatrix}
\bP^\top}_{\bW},
\end{aligned}
$$
where the matrix $\bW$ is given by 
$
\begin{bmatrix}
\bI_r & \bC^{-1}\bA[:,\sJ_r]
\end{bmatrix}\bP^\top
=
\begin{bmatrix}
\bI_r & \bE
\end{bmatrix}\bP^\top
$, from Equation~\eqref{equation:interpolatibve-w-ep}. To prove the claim that the magnitude of $\bW$ is at most 1,  it suffices to show that each entry of $\bE=\bC^{-1}\bA[:,\sJ_r]\in \real^{r\times (n-r)}$ satisfies $\abs{e_{kl}} \leq 1$ for all $k\in\{1,2,\ldots,r\}, l\in\{1,2,\ldots,n-r\}$.

Define the index vector $[j_1,j_2,\ldots, j_n]$ as a permutation of $[1,2,\ldots, n]$ such that 
$$
[j_1,j_2,\ldots, j_n] = [1,2,\ldots, n] \bP = [\sJ_s, \sJ_r].
$$
Thus, it follows from $\bC\bE=\bA[:,\sJ_r]$ that 
$$
\begin{aligned}
\underbrace{	[\ba_{j_1}, \ba_{j_2}, \ldots, \ba_{j_r}]}_{=\bC=\bA[:,\sJ_s]} \bE &=
\underbrace{[\ba_{j_{r+1}}, \ba_{j_{r+2}}, \ldots, \ba_{j_n}]}_{=\bA[:,\sJ_r]= \bB},
\end{aligned}
$$ 
where $\ba_i$ denotes the $i$-th column of $\bA$, and we let $\bB=\bA[:,\sJ_r]$.
Therefore, by Cramer's rule in Equation~\eqref{equation:cramer-rule-general}, we have 
\begin{equation}\label{equation:column-id-expansionmatrix}
e_{kl} = 
\frac{\det\left(\bC_{\bB}(k,l)\right)}
{\det\left(\bC\right)},
\end{equation}
where $e_{kl}$ is the entry ($k,l$) of $\bE$, and $\bC_{\bB}(k,l)$ is the $r\times r$ matrix formed by replacing the $k$-th column of $\bC$ with the $l$-th column of $\bB$. For example, 
$$
\begin{aligned}
e_{11} &= 
\frac{\det\left([\textcolor{mylightbluetext}{\ba_{j_{r+1}}}, \ba_{j_2}, \ldots, \ba_{j_r}]\right)}
{\det\left([\ba_{j_1}, \ba_{j_2}, \ldots, \ba_{j_r}]\right)},
\qquad 
&e_{12} &=
\frac{\det\left([\textcolor{mylightbluetext}{\ba_{j_{r+2}}}, \ba_{j_2},\ldots, \ba_{j_r}]\right)}
{\det\left([\ba_{j_1}, \ba_{j_2}, \ldots, \ba_{j_r}]\right)},\\
e_{21} &= 
\frac{\det\left([\ba_{j_1},\textcolor{mylightbluetext}{\ba_{j_{r+1}}}, \ldots, \ba_{j_r}]\right)}
{\det\left([\ba_{j_1}, \ba_{j_2}, \ldots, \ba_{j_r}]\right)},
\qquad 
&e_{22} &= 
\frac{\det\left([\ba_{j_1},\textcolor{mylightbluetext}{\ba_{j_{r+2}}}, \ldots, \ba_{j_r}]\right)}
{\det\left([\ba_{j_1}, \ba_{j_2}, \ldots, \ba_{j_r}]\right)}.
\end{aligned}
$$
Since $\sJ_s$  was chosen to maximize $\det(\bC)$ in Equation~\eqref{equation:interpolative-choose-js}, it follows that 
$$
\abs{e_{kl}}\leq 1, \qquad \text{for all}\gap k\in \{1,2,\ldots, r\}, l\in \{1,2,\ldots, n-r\}.
$$

\paragraph{Step 2: apply to general matrices.}
To summarize, for any matrix $\bF\in \real^{r\times n}$ with \textbf{full} rank $r\leq n$, the column ID exists such that $\bF=\bC_0\bW$, where the entries of  $\bW$ are bounded by 1 in absolute value.

For a general matrix $\bA\in \real^{m\times n}$ with rank $r\leq \{m,n\}$, the matrix admits a rank decomposition (Theorem~\ref{theorem:rank-decomposition}) of the form:
$$
\underset{m\times n}{\bA} = \underset{m\times r}{\bD}\gapthree \underset{r\times n}{\bF},
$$
where $\bD$ and $\bF$ have full column rank $r$ and full row rank $r$, respectively. 
Applying the column ID to $\bF=\bC_0\bW$, where $\bC_0=\bF[:,\sJ_s]$ contains $r$ linearly independent columns of $\bF$. We notice from $\bA=\bD\bF$ such that 
$$
\bA[:,\sJ_s]=\bD\bF[:,\sJ_s],
$$
i.e., the columns indexed by $\sJ_s$ of $(\bD\bF)$ can be obtained by $\bD\bF[:,\sJ_s]$, which in turn are the columns of $\bA$ indexed by $\sJ_s$. This makes
$$
\underbrace{\bA[:,\sJ_s]}_{\bC}= \underbrace{\bD\bF[:,\sJ_s]}_{\bD\bC_0},
$$
and 
$$
\bA = \bD\bF =\bD\bC_0\bW = \underbrace{\bD\bF[:,\sJ_s]}_{\bC}\bW=\bC\bW.
$$
This completes the proof.
\end{proof}
\index{CPQR}

The above proof provides  an intuitive way to compute the ``optimal" column ID of a matrix $\bA$. However, any algorithm guaranteed to achieve such an optimally conditioned factorization necessarily involves combinatorial complexity due to the need to search for the best column subset $\sJ_s$ that maximizes $\abs{\det(\bC)}$ \citep{martinsson2019randomized, lu2022bayesian, lu2022comparative}. To address this, randomized algorithms, along with approximations via column-pivoted QR (Section~\ref{section:cpqr}) and rank-revealing QR (Section~\ref{section:rank-r-qr}), are commonly employed to obtain a relatively well-conditioned column ID decomposition.
In these approaches, the matrix $\bW$ is designed to have a small norm rather than strictly ensure that all its entries are within the range $[-1,1]$. On the other hand,  Bayesian approaches can strictly constrain the entries of $\bW$ to lie within $[-1,1]$ \citep{lu2022bayesian, lu2022comparative}. However, these methods involve more advanced techniques and is beyond the scope of this discussion; and we will not elaborate on them here.

\begin{example}[Compute the column ID]\label{example:column-id-a}
Let
$$
\bA=
\begin{bmatrix}
56 & 41 & 30\\
32 & 23 & 18\\
80 & 59 & 42
\end{bmatrix}
$$
be a rank-2 matrix. We now demonstrate the process of computing a column ID of  $\bA$.
We begin by finding a rank decomposition of $\bA$:
$$
\bA = \bD\bF=
\begin{bmatrix}
1 & 0 \\
0 & 1 \\
2 &-1
\end{bmatrix}
\begin{bmatrix}
56 & 41 & 30 \\
32 & 23 & 18 
\end{bmatrix}.
$$
Since $\rank(\bA)=2$, the index vector $\sJ_s$ can take one of the following values: $[1,2], [0,2], [0,1]$, where the absolute determinants of $\bF[:,\sJ_s]$ are $48, 48$, and $ 24$, respectively. We proceed with $\sJ_s=[0,2]$, which yields:
$$
\begin{aligned}
\widetildebC &= \bF[:,\sJ_s]=
\begin{bmatrix}
56 & 30 \\
32 & 18 
\end{bmatrix},\qquad 
\bM = \bF[:,\sJ_r]=\begin{bmatrix}
41 \\
23
\end{bmatrix}.
\end{aligned}
$$
Thus, 
$$
\bF\bP = \bF[:(\sJ_s,\sJ_r)] = \bF[:,(0,2,1)]
\quad\implies\quad 
\bP = 
\begin{bmatrix}
1 &  & \\
& &1\\
& 1 & 
\end{bmatrix}.
$$
In this example, $\bE\in \real^{2\times 1}$:
$$
\begin{aligned}
e_{11} &=
\det\left(
\begin{bmatrix}
41 & 30 \\
23 & 18
\end{bmatrix}\right)\bigg/
\det\left(
\begin{bmatrix}
56 & 30 \\
32 & 18
\end{bmatrix}\right)=1;\\
e_{21} &=
\det\left(
\begin{bmatrix}
56 & 41 \\
32 & 23
\end{bmatrix}\right)\bigg/
\det\left(
\begin{bmatrix}
56 & 30 \\
32 & 18
\end{bmatrix}\right)=-\frac{1}{2}.
\end{aligned}
$$
This makes 
$$
\bE = 
\begin{bmatrix}
1\\-\frac{1}{2}
\end{bmatrix}
\quad \implies\quad  
\bW = [\bI_2, \bE]\bP^\top =
\begin{bmatrix}
1 & 1 & 0\\
0 & -\frac{1}{2} & 1
\end{bmatrix}.
$$
The final  selected columns and the resulting decomposition are:
$$
\bC = \bA[:,\sJ_s] = 
\begin{bmatrix}
56 & 30\\
32 & 18\\
80 & 42
\end{bmatrix}
\quad\implies\quad
\bA=\bC\bW =
\begin{bmatrix}
	56 & 30\\
	32 & 18\\
	80 & 42
\end{bmatrix}
\begin{bmatrix}
	1 & 1 & 0\\
	0 & -\frac{1}{2} & 1
\end{bmatrix}.
$$
As expected, the entries of $\bW$ have magnitudes no greater than 1.
\end{example}

To conclude this section, we discuss  the non-uniqueness of the column ID.
\begin{remark}[Non-uniqueness of the column ID]
The column ID is not unique, as illustrated in Example~\ref{example:column-id-a}. Specifically, both $\bF[:,(1,2)]$ and $\bF[:,(0,2)]$ yield   the  maximum absolute determinant. 
Either choice results in a valid column ID for $\bA$. Whilst, we only select one $\sJ_s$ from $[1,2], [0,2]$, and $ [0,1]$. 
Additionally, when selecting the index set $\sJ_s$, any permutation of it is also valid. For example,  $\sJ_s=[0,2]$ and $\sJ_s=[2,0]$ are both acceptable. 
This flexibility in selecting the column indices introduces non-uniqueness into the column ID.
\end{remark}

\section{Row ID and Two-Sided ID}
The decomposition described above is called the column interpolative decomposition, a name that is not arbitrary---it is closely related to other types of interpolative decompositions, as explained below:
\begin{theoremHigh}[The whole interpolative decomposition]\label{theorem:interpolative-decomposition-row}
Any rank-$r$ matrix $\bA \in \real^{m \times n}$ can be decomposed as 
$$
\begin{aligned}
\text{Column ID: }&\gap \underset{m \times n}{\bA} &=& \boxed{\underset{m\times r}{\bC}} \gap  \underset{r\times n}{\bW} ; \\
\text{Row ID: } &\gap &=&\underset{m\times r}{\bZ} \gap  \boxed{\underset{r\times n}{\bR}}; \\
\text{Two-Sided ID: } &\gap &=&\underset{m\times r}{\bZ} \gap \boxed{\underset{r\times r}{\bU}}  \gap  \underset{r\times n}{\bW}, \\
\end{aligned}
$$
where
\begin{itemize}
\item $\bC=\bA[:,\sJ_s]\in \real^{m\times r}$ contains  $r$ linearly independent columns of $\bA$, $\bW\in \real^{r\times n}$ is the matrix used to reconstruct $\bA$, which contains an $r\times r$ identity submatrix (under a mild column permutation): $\bW[:,\sJ_s]=\bI_r$;
\item $\bR=\bA[\sI_s,:]\in \real^{r\times n}$ contains  $r$ linearly independent rows of $\bR$, $\bZ\in \real^{m\times r}$ is the matrix used to reconstruct $\bA$, which contains an $r\times r$ identity submatrix (under a mild row permutation): $\bZ[\sI_s,:]=\bI_r$;
\item The entries in $\bW$ and $\bZ$ have values no larger than 1 in magnitude: $\max |w_{ij}|\leq 1$ and $\max |z_{ij}|\leq 1$;
\item $\bU=\bA[\sI_s,\sJ_s] \in \real^{r\times r}$ is the nonsingular submatrix at the intersection of $\bC$ and $\bR$;
\item The three matrices $\bC,\bR,$ and $\bU$ in the $\boxed{\text{boxed}}$ representations share the same notation and interpretation as in the skeleton decomposition (Theorem~\ref{theorem:skeleton-decomposition}). Specifically, $\bA=\bC\bU^{-1}\bR$ represents the skeleton decomposition.
\end{itemize}
\end{theoremHigh}
The proof of the row ID follows similarly from the column ID  by transposing. Assume the column ID of $\bA^\top$ is given as $\bA^\top=\bC_0\bW_0$, where $\bC_0$ contains $r$ linearly independent columns of $\bA^\top$ (i.e., $r$ linearly independent rows of $\bA$). Let $\bR=\bC_0^\top$ and $\bZ=\bW_0^\top$. Then, the row ID is obtained as $\bA=\bZ\bR$. 

From the skeleton decomposition, where $\bU$ is the intersection of $\bC$ and $\bR$, it follows that $\bA=\bC\bU^{-1}\bR$. Using the row ID, we get $\bC\bU^{-1}=\bZ$, which implies $\bC=\bZ\bU$. Substituting into the column ID yields $\bA=\bC\bW=\bZ\bU\bW$, thereby  proving the existence of the two-sided ID.

\index{Data storage}
\paragraph{Data storage.} For each ID, the storage requirements are summarized as follows:
\begin{itemize}
\item \textit{Column ID.} It requires $mr$ and $(n-r)r$ floating-point numbers to store $\bC$ and $\bW$, respectively, and $r$ integers to store the indices of the selected columns in $\bA$;
\item \textit{Row ID.} It requires $nr$ and $(m-r)r$ floating-point numbers to store $\bR$ and $\bZ$, respectively, and $r$ integers to store the indices of the selected rows in $\bA$;
\item \textit{Two-Sided ID.} It requires $(m-r)r$, $(n-r)r$, and $r^2$ floating-point numbers to store $\bZ,\bW$, and $\bU$, respectively. And an extra $2r$ integers are required to store the indices of the selected rows and columns in $\bA$.
\end{itemize}

\paragraph{Storage reduction for sparse matrices.} 
For sparse matrices, further storage  savings are possible. Consider the column ID: $\bA=\bC\bW$, where $\bC=\bA[:,\sJ_s]$, and a good spanning row index set $\sI_s$  of $\bC$ exists such that:
$$
\bA[\sI_s,:] = \bC[\sI_s,:]\bW.
$$ 
Since $\bC[\sI_s,:] = \bA[\sI_s,\sJ_s]\in \real^{r\times r}$ is nonsingular, we can compute:
$$
\bW = (\bA[\sI_s,\sJ_s])^{-1} \bA[\sI_s,:].
$$
Thus, $\bW$ does not need to be explicitly stored; only $\bA[\sI_s,:]$ and $(\bA[\sI_s,\sJ_s])^{-1}$ are required.
Alternatively, If the inverse of $\bA[\sI_s,\sJ_s]$ is computed dynamically, only $r$ integers for $\sJ_s$ are necessary, as $\bA[\sI_s,\sJ_s]$ can be reconstructed from $\bA[\sI_s,:]$. This approach is particularly efficient for sparse matrices, where the storage of $\bA[\sI_s,:]$ is economical.

\index{Pseudo-inverse}
\index{Low-rank approximation}
\section{Application: Low-Rank Approximation via Pseudoskeleton}\label{section:pseudoskeleton}
We will explore singular value decomposition (SVD) in detail in Section~\ref{section:SVD}. For now, we assume a basic understanding of SVD and demonstrate how it can be used to approximate skeleton decomposition. This section can be skipped during an initial reading.

Given a  matrix $\bA\in\real^{m\times n}$, our goal is to construct a rank-$\gamma$ approximation of $\bA$, where  $\gamma \leq \min(m,n)$, using skeleton decomposition. 
Specifically, we approximate $\bA$  as $\bA\approx \bC\bU^{-1}\bR$, where $\bC$ and $\bR$ are matrices containing $\gamma$ selected columns and rows, respectively, and $\bU$ is the submatrix formed by the intersection of these selected rows and columns. 
More precisely, if $\sI$ and $\sJ$ denote the indices of the selected rows and columns, then   $\bU=\bA[\sI,\sJ]$. Note that $\gamma$ does not necessarily equal the rank $r$ of $\bA$, thus forming a low-rank approximation.

Unlike standard skeleton decomposition, which selects $r$ linearly independent columns from $\bA$, we instead choose $k$ random columns (where $k > r$ or even $k = \min\{m, n\}$) to form $\bC$. The column indices $\sJ$ determine $\bC = \bA[:, \sJ] \in \mathbb{R}^{m \times k}$. 
Simultaneously, $k$ rows of $\bA$ are selected using the indices $\sI$, forming $\bR = \bA[\sI, :]$. 
These rows are chosen such that the intersection matrix $\bU = \bA[\sI, \sJ]$ has maximal volume; that is, $\det(\bU)$ is maximized. While the matrix $\bC$ is selected randomly, the choice of $\bR$ is deterministic. This leads to the decomposition:
$$
\bA = \bC_{m\times k}\bU_{k\times k}^{-1}\bR_{k\times n}.
$$
However, the inverse of  $\bU_{k\times k}$ can be numerically unstable due to the random selection of $\bC$.
To address this issue, we perform a full SVD on $\bU_{k\times k}$ (see Section~\ref{section:SVD} for the distinction between reduced and full SVD):
$$
\bU_{k\times k} = \bU_k\bSigma_k\bV_k^\top,
$$
where $\bU_k, \bV_k\in \real^{k\times k}$ are orthogonal matrices, and $\bSigma_k$ is a diagonal matrix containing $k$ singular values $\sigma_1 \geq \sigma_2 \geq \ldots \geq \sigma_k$, some of which may be zero.
We then select $\gamma$ singular values greater than a threshold $\epsilon$ and truncate $\bU_k$, $\bV_k$, and $\bSigma_k$ accordingly, approximating $\bU_{k \times k}$ by a rank-$\gamma$ matrix: $\bU_{k\times k} \approx \bU_\gamma\bSigma_\gamma\bV_\gamma^\top$, where $\bU_\gamma,\bV_\gamma\in \real^{k\times \gamma}$, and $\bSigma_\gamma \in \real^{\gamma\times \gamma}$. Therefore, the pseudo-inverse of $\bU_{k\times k}$ is
$$
\bU^+ = (\bU_\gamma\bSigma_\gamma\bV_\gamma^\top)^{+} =\bV_\gamma \bSigma_\gamma^{-1}\bU_\gamma^\top.
$$
Using this, we can approximate $\bA$ as a rank-$\gamma$ matrix:
\begin{equation}\label{equation:skeleton-low-rank}
	\begin{aligned}
		\bA 
		&\approx \bC\bV_\gamma \bSigma_\gamma^{-1}\bU_\gamma^\top \bR\\
		&=\bC_2 \bR_2,  \qquad (\text{let $\bC_2=\bC\bV_\gamma \bSigma_\gamma^{-1/2}$ and $\bR_2=\bSigma_\gamma^{-1/2}\bU_\gamma^\top \bR$})
	\end{aligned}
\end{equation}
where $\bC_2$ and $\bR_2$ are rank-$\gamma$ matrices. 
For guidance on selecting the threshold $\epsilon$, refer to \citet{goreinov1997pseudo} and \citet{kishore2017literature}. 
In the above approach, $\bC$ is chosen randomly, while $\bR$ is determined. Algorithms such as those in \citet{zhu2011randomised}, which select both $\bC$ and $\bR$ randomly, tend to yield more stable approximations.

Note that  data interpretation methods like CR and skeleton decomposition help identify subsets of columns or rows that retain essential information in the matrix. 
While Equation~\eqref{equation:skeleton-low-rank} may not immediately reflect this capability, we also notice that $\bC_2 = \bC(\bV_\gamma\bSigma_\gamma^{-1/2})$, and the columns of $\bV_\gamma\bSigma_\gamma^{-1/2}$ are mutually orthogonal. Thus, the columns of $\bC_2$ represent those of $\bC$ in a different coordinate system. 
Similarly, the rows of $\bR_2$ correspond to a transformed version of the rows of $\bR$;  see Section~\ref{section:coordinate-transformation}.

\section{Application: Feature Selection}
In quantitative finance, ID and its enhanced variant, \textit{intervened interpolative decomposition (IID)}, are useful tools for feature selection,  particularly in the context of developing and optimizing algorithmic trading strategies  \citep{lu2022bayesian, lu2022feature}.

Since ID can serve as a low-rank matrix approximation technique that reconstructs a full data matrix using only a subset of its columns, this allows it to extract key patterns or information from large volumes of market data while preserving sparsity and nonnegativity. IID builds on this idea by incorporating an importance score for each column (i.e., each feature), prioritizing those features considered most relevant or impactful to the model.

In quantitative investing, it's common to encounter thousands---or even millions---of alpha factors (signals that predict future stock returns). Standard ID algorithms can help identify a small set of representative factors from this large pool. However, these selected factors may not necessarily be the ones with the strongest predictive performance. In contrast, the IID approach selects factors that are both representative---meaning they can reconstruct other alphas with minimal error---and desirable, such as those with high \textit{RankIC scores}, which reflect strong predictive power.

Given the high dimensionality and potential multicollinearity among alpha factors in financial markets, including all available factors in a model can lead to overfitting and computational inefficiency. By applying ID or IID, practitioners can select a smaller but representative subset of alpha factors. This helps reduce overfitting while improving scalability and computational efficiency.

For example, experiments conducted by \citet{lu2022feature} used data from ten assets across various sectors in the Chinese market, including banking, public utilities, and ETFs. Applying the ID or IID method to historical data enabled the identification of significant alpha factors, which were then used to construct effective trading strategies.

In summary, ID and IID provide quantitative analysts with powerful tools to identify the most influential variables when working with complex financial datasets. These techniques contribute to the development of more robust and effective trading strategies. They are especially valuable in handling high-dimensional data, as they reduce computational costs, improve model interpretability, and enhance predictive accuracy.

\begin{problemset}

\item Find the CUR decomposition for the matrix 
$$
\bA = 
\begin{bmatrix}
1 & 3 & 2 \\
3 & 7 & 6 \\
4 & 5 & 8
\end{bmatrix}.
$$

\item Using the CUR decomposition, compute the pseudo-inverse of the matrix $\bA$.

\item \label{problem:nons_sub} Consider a matrix $\bA\in\real^{n\times n}$ in block form $\bA=\scriptsize\begin{bmatrix}
\bK & \bL\\
\bM & \bN
\end{bmatrix}$, where $\bK\in\real^{r\times r}$ and $\bN\in\real^{(n-r)\times (n-r)}$. Show that 
\begin{itemize}
\item If $\bK$ is nonsingular, then $[\bK, \bL]$ and $\scriptsize\begin{bmatrix}
	\bK\\
	\bM 
\end{bmatrix}$ has full ranks.
\item If $\rank(\bA) = \rank([\bK, \bL])=\rank(\scriptsize\begin{bmatrix}
	\bK\\
	\bM 
\end{bmatrix}\normalsize)$, then $\bK$ is nonsingular.
\end{itemize}

\item \label{problem:rank_pcin} Let $\bA\in\real^{n\times n}$ be  a symmetric or skew-symmetric matrix. Show that 
\begin{itemize}
\item $\rank(\bA[\sI, :])=\rank(\bA[:,\sI])$ for any index set $\sI\subseteq \{1,2,\ldots,n\}$.
\item The matrix $\bA$ is rank principal (i.e., it has a nonsingular $r\times r$ principal submatrix; Definition~\ref{definition:principle-minors}).
\end{itemize}

\index{Gram--Schmidt}
\item Discuss algorithms for computing the CUR decomposition using Gaussian elimination or the Gram--Schmidt process. Determine the computational complexity of these algorithms.

\item Investigate how different column and row selection strategies affect the accuracy of the CUR decomposition. Generate a random matrix $\bA\in\real^{100\times 100}$, and implement different selection strategies for columns and rows (e.g., random, based on column norms, leverage scores).
For each strategy, perform CUR decomposition/approximation with $r=10$. Evaluate the approximation error for each strategy. Discuss which strategy provides the best approximation.

\item Compare the CUR decomposition with other matrix decomposition methods like SVD and QR decomposition in terms of accuracy and computational efficiency.
Generate a random matrix $\bA\in\real^{100\times 100}$. Perform CUR decomposition, SVD, and QR decomposition on $\bA$. For each method, compute the approximation error using an appropriate norm. Measure the computational time for each method. Discuss the trade-offs between accuracy and computational efficiency for each method.


\item Suppose that you have an $n\times n$ matrix where  the absolute value of every entry is at most 1. Show that the absolute value of the determinant of this matrix is also at most $(n)^{n/2}$. Additionally, provide an example of a $2 \times 2$ matrix for which the determinant achieves this upper bound.

\item \label{prob:pro_adjug} \textbf{Adjugate.}
Let $\bA\in\real^{n\times n}$. Show that 
\begin{itemize}
\item $\adjugate(c\bA)=c^{n-1}\adjugate(\bA)\implies \adjugate(c\bI)=c^{n-1}\bI$.
\item $\det(\adjugate(\bA)) = (\det(\bA))^{n-1}$.
\item If $\bA$ is nonsingular, $\bA^{-1}=(\det(\bA))^{-1}\adjugate(\bA)$.
\item If $\bA$ is nonsingular, $\adjugate(\bA^{-1})=\bA/\det(\bA)$.
\item If $\bA$ is singular and $\rank(\bA)\leq n-2$, $\adjugate(\bA)=\bzero$.
\item If $\bA$ is singular and $\rank(\bA)= n-1$, $\rank(\adjugate(\bA))=1$.
\item If $\bA,\bB$ are nonsingular, $\adjugate(\bA\bB)=\adjugate(\bA)\adjugate(\bB)$. (This actually can be extended to all $\bA,\bB$ due to continuity.)
\item If $\bA$ is nonsingular, $\adjugate(\adjugate(\bA))=(\det(\bA))^{n-2}\bA$. (This actually can be extended to all $\bA$ due to continuity.)
\item If $\bA$ is nonsingular, $\adjugate(\bA^\top)=\adjugate(\bA)^\top$. (This actually can be extended to all $\bA$ due to continuity.) 
\item The adjugate is the transpose of the gradient of $\det(\bA)$: $\adjugate(\bA)_{ij}=\frac{\partial}{\partial a_{ji}}\det(\bA)$.
\end{itemize}

\index{Adjugate}
\index{Cramer's rule}
\item \label{prob:cramer_adj_1} \textbf{Cramer's rule.}	Consider the linear system $\bM\bx=\bl$, where $\bM\in\real^{n\times n}$, and $\bx,\bl\in\real^{n}$. Let $\bM_{\bl}(i)$ represent the matrix formed by replacing the $i$-th column of $\bM$ with $\bl$. Show that the $i$-th element of $\adjugate(\bM)\bl\in\real^n$ (where $\adjugate(\bM)$ is the adjugate of $\bM$; see Definition~\ref{definition:adjugate}) is 
\begin{equation}\label{equation:cramer_adj_1}
	\big(\adjugate(\bM)\bl\big)_i = \det(\bM_{\bl}(i)), \gap i\in\{1,2,\ldots,n\}.
\end{equation}
Now consider the linear system $\bM\bX=\bL$, where $\bM\in\real^{n\times n}$, and $\bX,\bL\in\real^{n\times m}$. Let $\bM_{\bL}(i,j)$ be the matrix formed by replacing the $i$-th column of $\bM$ with the $j$-th column $\bl_j$ of $\bL$. Show that the $(i,j)$-th element of $\adjugate(\bM)\bL\in\real^{n\times m}$ is 
\begin{equation}\label{equation:cramer_adj_2}
	\big(\adjugate(\bM)\bL\big)_{ij} = \det(\bM_{\bL}(i,j)), \gap i\in\{1,2,\ldots,n\}, j\in\{1,2,\ldots,m\}.
\end{equation}
\textit{Hint: Examine the definitions of determinant and adjugate (Definition~\ref{definition:determinant},~\ref{definition:adjugate}).}

\item \label{prob:cramer_adj_2} \textbf{Cramer's rule.} In the same setting as Problem~\ref{prob:cramer_adj_1}, represent  the vector  $\adjugate(\bM)\bl\in\real^n$ and  the matrix 
$\adjugate(\bM)\bL\in\real^{n\times m}$ as follows:
$$
\adjugate(\bM)\bl = \big[\det(\bM_{\bl}(i))\big]_{i=1}^n\in\real^n,
\gap 
\adjugate(\bM)\bL = \big[\det(\bM_{\bl}(i,j))\big]_{i,j=1}^{n,m}\in\real^{n\times m},
$$
i.e., the $i$-th element of the vector is $\det(\bM_{\bl}(i))$,  the  $(i,j)$-th element of the matrix is $\det(\bM_{\bL}(i,j))$. 
Show that
\begin{equation}\label{equation:cramer_adj_3}
	\begin{aligned}
		\bM \big[\det(\bM_{\bl}(i))\big]_{i=1}^n &= \bM \adjugate(\bM) \bl=\det(\bM)\bl;\\
		\bM \big[\det(\bM_{\bL}(i,j))\big]_{i,j=1}^{n,m} &= \bM \adjugate(\bM) \bL=\det(\bM)\bL.
	\end{aligned}
\end{equation}
\textit{Hint: Examine the definition of adjugate (Definition~\ref{definition:adjugate}) and Problem~\ref{prob:pro_adjug}.}

\item \label{prob:cramer_adj_3} \textbf{Cramer's rule.} In the same setting as Problem~\ref{prob:cramer_adj_1},  assume further that $\bM$ is nonsingular. Show that the $i$-th element of the solution $\bx$ is 
\begin{equation}\label{equation:cramer_adj_res1}
	x_i = \frac{\det(\bM_{\bl}(i))}{\det(\bM)}, \gap \forall i\in\{1,2,\ldots,n\}.
\end{equation}
Similarly, show that the $(i,j)$-th element of the solution $\bX$ is 
\begin{equation}\label{equation:cramer_adj_res2}
	x_{ij} = \frac{\det\left(\bM_{\bL}(i,j)\right)}{\det(\bM)}, \gap \forall i\in\{1,2,\ldots,n\}, j\in\{1,2,\ldots,m\}.
\end{equation}
That is,  Cramer's rule.

\item \label{prob:cramer_adj_4} \textbf{Cramer's rule: the simple way.}
In the same setting as Problem~\ref{prob:cramer_adj_1}, assume further that $\bM$ is nonsingular. We notice that 
\begin{equation}\label{equation:cramer_adj_4}
	\bM \bI_{\bl}(i) = \bM_{\bl}(i),  \gap \forall i\in\{1,2,\ldots,n\},
\end{equation}
where $\bI_{\bl}(i)$ represents the identity matrix with  the $i$-th column replaced by $\bl$. 
Taking the determinant yields:
\begin{equation}\label{equation:cramer_adj_5}
	\det(\bM) \det(\bI_{\bl}(i)) = \det(\bM_{\bl}(i)).
\end{equation}
Show that  $\det(\bI_{\bl}(i))=x_i$, thereby verifying the result in \eqref{equation:cramer_adj_res1}.

\index{Schur complement}
\item \label{prob:cramer_adj_5} \textbf{Determinant of inverses for subsets, Jacobi's equality.} Let $\bM\in\real^{n\times n}$, and let   $\sI,\sJ\subseteq\{1,2,\ldots,n\}$ be two index sets (their complementary sets are $\comple{\sI}$ and $\comple{\sJ}$, respectively). 
Show that 
\begin{equation}
	\det\big(\bM^{-1}[\comple{\sI},\comple{\sJ}]\big) = (-1)^{\gamma} \frac{\det(\bM[\sJ,\sI])}{\det(\bM)},
\end{equation}
where $\gamma=\sum_{i\in \sI} i +\sum_{j\in \sJ}j$ is the sum of indices.
When $\sI=\sJ$, this also indicates
\begin{equation}
	\det\big(\bM^{-1}[\comple{\sI},\comple{\sI}]\big) = \frac{\det(\bM[\sI,\sI])}{\det(\bM)},
\end{equation}
which is known as \textit{Jacobi's equality}.
\textit{Hint: Examine the definitions of determinant and adjugate (Definitions~\ref{definition:determinant} and \ref{definition:adjugate}). Alternatively, you may prove this using the Schur complement.}

\item Determine the column ID for the matrix 
$
\bA =
\scriptsize
\begin{bmatrix}
	1 & 3 & 2 \\
	3 & 7 & 6 \\
	4 & 5 & 8
\end{bmatrix}.
$

\item \label{problem:mtb_dis} \textbf{Matlab-style notation.} Consider a rectangular matrix $\bA\in\real^{m\times n}$ of rank $r$, which admits a rank decomposition $\bA=\bD\bF$. Let  $\sI,\sJ\subseteq \{1,2,\ldots,m\}$ and $\sK,\sL\subseteq\{1,2,\ldots,n\}$ be index sets with cardinality $\abs{\sI}=\abs{\sJ}=\abs{\sK}=\abs{\sL}=r$. Then, we have  $\bA[\sI,\sK]=\bD[\sI,:]\bF[:,\sK]$.
Show that 
\begin{itemize}
	\item $\bA[\sI,\sK]$ is nonsingular if and only if $\rank(\bD[\sI,:])=\rank(\bF[:,\sK])=r$.
	\item $\det(\bA[\sI,\sK]) \det(\bA[\sJ,\sL])=\det(\bA[\sI,\sL])\det(\bA[\sJ,\sK])$.
\end{itemize}

\item Discuss algorithms for computing the column ID of  a matrix using its column-pivoted QR decomposition. 
\end{problemset}

\part{Reduction to Hessenberg, Tridiagonal, and Bidiagonal Form}

\chapter{Hessenberg and Tridiagonal Decomposition}

\section*{Preliminary}
In real applications, we often aim to factor a matrix $\bA$ into two orthogonal matrices such that $\bA =\bQ\bLambda\bQ^\top$, where $\bLambda$ is either diagonal or upper triangular. Common examples include eigenanalysis via  Schur decomposition and principal component analysis (PCA) using spectral decomposition.  This type of factorization can be computed through a sequence of \textit{orthogonal similarity transformations}:
$$
\underbrace{\bQ_k^\top\ldots \bQ_2^\top \bQ_1^\top}_{\bQ^\top} \bA \underbrace{\bQ_1\bQ_2\ldots\bQ_k}_{\bQ},
$$
which converges to $\bLambda$. However, this transformation can be difficult to handle in practice; for instance, when using Householder reflectors.
Following the approach used in QR decomposition with Householder reflectors,  the sequence of orthogonal similarity transformations can be constructed using  Householder reflectors:
$$
\begin{aligned}
	\footnotesize
	\begin{sbmatrix}{\bA}
		\boxtimes & \boxtimes & \boxtimes & \boxtimes& \boxtimes \\
		\boxtimes & \boxtimes & \boxtimes & \boxtimes& \boxtimes \\
		\boxtimes & \boxtimes & \boxtimes & \boxtimes& \boxtimes \\
		\boxtimes & \boxtimes & \boxtimes & \boxtimes & \boxtimes\\
		\boxtimes & \boxtimes & \boxtimes & \boxtimes& \boxtimes
	\end{sbmatrix}
	&\stackrel{\bH_1\times }{\rightarrow}
	\footnotesize
	\begin{sbmatrix}{\bH_1\bA}
		\bm{\boxtimes} & \bm{\boxtimes} & \bm{\boxtimes} & \bm{\boxtimes}& \bm{\boxtimes} \\
		\bm{0} & \bm{\boxtimes} & \bm{\boxtimes} & \bm{\boxtimes}& \bm{\boxtimes} \\
		\bm{0} & \bm{\boxtimes} & \bm{\boxtimes} & \bm{\boxtimes}& \bm{\boxtimes} \\
		\bm{0} & \bm{\boxtimes} & \bm{\boxtimes} & \bm{\boxtimes} & \bm{\boxtimes}\\
		\bm{0} & \bm{\boxtimes} & \bm{\boxtimes} & \bm{\boxtimes}& \bm{\boxtimes}
	\end{sbmatrix}
	\stackrel{\times \bH_1^\top }{\rightarrow}
	\footnotesize
	\begin{sbmatrix}{\bH_1\bA\bH_1^\top}
		\bm{\boxtimes} & \bm{\boxtimes} & \bm{\boxtimes} & \bm{\boxtimes}& \bm{\boxtimes} \\
		\bm{\boxtimes} & \bm{\boxtimes} & \bm{\boxtimes} & \bm{\boxtimes}& \bm{\boxtimes} \\
		\bm{\boxtimes} & \bm{\boxtimes} & \bm{\boxtimes} & \bm{\boxtimes} & \bm{\boxtimes}\\
		\bm{\boxtimes} & \bm{\boxtimes} & \bm{\boxtimes} & \bm{\boxtimes} & \bm{\boxtimes}\\
		\bm{\boxtimes} & \bm{\boxtimes} & \bm{\boxtimes} & \bm{\boxtimes} & \bm{\boxtimes}
	\end{sbmatrix},
\end{aligned}
$$
where the left Householder ($\bH_1\times$) introduces zeros in the first column below the main diagonal (see Section~\ref{section:qr-via-householder}), but  unfortunately, the right Householder ($\times \bH_1^\top$) will undo  the zeros created  by the left Householder.

However, we can simplify the process by splitting the algorithm into two phases. In the first phase, we transform the matrix into a Hessenberg matrix (Definition~\ref{definition:upper-hessenbert}) or a tridiagonal matrix (Definition~\ref{definition:tridiagonal-hessenbert}). Then, in the second phase, we apply an iterative algorithm to convert the results from the first phase into the desired form. This leads to the following algorithmic structure:
$$
\begin{aligned}
	\footnotesize
	\begin{sbmatrix}{\bA}
		\boxtimes & \boxtimes & \boxtimes & \boxtimes& \boxtimes \\
		\boxtimes & \boxtimes & \boxtimes & \boxtimes& \boxtimes \\
		\boxtimes & \boxtimes & \boxtimes & \boxtimes& \boxtimes \\
		\boxtimes & \boxtimes & \boxtimes & \boxtimes & \boxtimes\\
		\boxtimes & \boxtimes & \boxtimes & \boxtimes& \boxtimes
	\end{sbmatrix}
	&\stackrel{\bH_1\times }{\rightarrow}
	\footnotesize
	\begin{sbmatrix}{\bH_1\bA}
		\boxtimes & \boxtimes & \boxtimes & \boxtimes& \boxtimes \\
		\bm{\boxtimes} & \bm{\boxtimes} & \bm{\boxtimes} & \bm{\boxtimes}& \bm{\boxtimes} \\
		\bm{0} & \bm{\boxtimes} & \bm{\boxtimes} & \bm{\boxtimes} & \bm{\boxtimes}\\
		\bm{0} & \bm{\boxtimes} & \bm{\boxtimes} & \bm{\boxtimes} & \bm{\boxtimes}\\
		\bm{0} & \bm{\boxtimes} & \bm{\boxtimes} & \bm{\boxtimes}& \bm{\boxtimes}
	\end{sbmatrix}
	\stackrel{\times \bH_1^\top }{\rightarrow}
	\footnotesize
	\begin{sbmatrix}{\bH_1\bA\bH_1^\top}
		\boxtimes & \bm{\boxtimes} & \bm{\boxtimes} & \bm{\boxtimes}& \bm{\boxtimes} \\
		\boxtimes & \bm{\boxtimes} & \bm{\boxtimes} & \bm{\boxtimes}& \bm{\boxtimes} \\
		0 & \bm{\boxtimes} & \bm{\boxtimes} & \bm{\boxtimes} & \bm{\boxtimes}\\
		0 & \bm{\boxtimes} & \bm{\boxtimes} & \bm{\boxtimes} & \bm{\boxtimes}\\
		0 & \bm{\boxtimes} & \bm{\boxtimes} & \bm{\boxtimes} & \bm{\boxtimes}
	\end{sbmatrix}\ldots.
\end{aligned}
$$
In this case, the left Householder does not affect the first row, and the right Householder does not disturb  the first column. A Phase 2 algorithm, typically iterative, for finding the triangular matrix is as follows:
$$
\begin{aligned}
	\footnotesize
	\begin{sbmatrix}{\bH_3\bH_2\bH_1\bA\bH_1^\top\bH_2^\top \bH_3^\top}
		\boxtimes & \boxtimes  & \boxtimes  & \boxtimes & \boxtimes  \\
		\boxtimes & \boxtimes  & \boxtimes  & \boxtimes & \boxtimes  \\
		0 & \boxtimes  & \boxtimes  & \boxtimes & \boxtimes  \\
		0 & 0 & \boxtimes  & \boxtimes & \boxtimes  \\
		0 & 0 & 0 & \boxtimes  & \boxtimes 
	\end{sbmatrix}
	\stackrel{\text{Phase 2} }{\longrightarrow}
	\footnotesize
	\begin{sbmatrix}{\bLambda}
		\bm{\boxtimes} & \bm{\boxtimes} & \bm{\boxtimes} & \bm{\boxtimes}& \bm{\boxtimes} \\
		\bm{0} & \bm{\boxtimes} & \bm{\boxtimes} & \bm{\boxtimes}& \bm{\boxtimes} \\
		0 & \bm{0}  & \bm{\boxtimes} & \bm{\boxtimes}& \bm{\boxtimes} \\
		0 & 0 & \bm{0}  & \bm{\boxtimes}& \bm{\boxtimes} \\
		0 & 0 & 0 & \bm{0}  & \bm{\boxtimes}
	\end{sbmatrix}
\end{aligned}
$$

As discussed above, to compute  spectral decomposition, Schur decomposition, or singular value decomposition (SVD), we often make a trade-off. In the first phase, we reduce the matrix to Hessenberg, tridiagonal, or bidiagonal form. 
The second stage then completes the decomposition using an iterative method  \citep{van2012families, van2014restructuring, trefethen1997numerical}.

\section{Hessenberg Decomposition}

The \textit{Hessenberg decomposition} is a technique used to transform a matrix into an upper Hessenberg form. This transformation simplifies the matrix structure, making it an effective first step in various algorithms, as it reduces computational complexity. Let's begin with a formal definition of upper Hessenberg matrices.

\index{Orthogonal}
\begin{definition}[Upper Hessenberg matrix\index{Hessenbert matrix}]\label{definition:upper-hessenbert}
An \textit{upper Hessenberg matrix} (simply called Hessenberg matrix when the context is clear)  is a square matrix in which all  entries below the subdiagonal  are zero. Similarly, a \textit{lower Hessenberg matrix} is a square matrix in which all the entries above the superdiagonal are zero.
The definition can be extended to rectangular matrices, where the structure is implied by the context.

Formally, for a matrix $\bH\in \real^{n\times n}$, with elements $h_{ij}$ for  $i,j\in \{1,2,\ldots, n\}$, $\bH$ is an upper Hessenberg matrix if $h_{ij}=0$ for all $i\geq j+2$.

Additionally, if $i$ is the smallest positive integer for which $h_{i+1, i}=0$ for $i\in \{1,2,\ldots, n-1\}$, then $\bH$ is called  \textbf{unreduced} if $i=n$.
\end{definition}

Consider a $5\times 5$ matrix. In an upper Hessenberg matrix, all elements below the first subdiagonal are zero:
$$
\footnotesize
\begin{sbmatrix}{possibly\,\, unreduced}
\boxtimes & \boxtimes & \boxtimes & \boxtimes & \boxtimes\\
\boxtimes & \boxtimes & \boxtimes & \boxtimes & \boxtimes\\
0 & \boxtimes & \boxtimes & \boxtimes & \boxtimes\\
0 & 0 & \boxtimes & \boxtimes & \boxtimes\\
0 & 0 & 0 & \boxtimes & \boxtimes
\end{sbmatrix}
\qquad 
\text{or}
\qquad 
\footnotesize
\begin{sbmatrix}{reduced}
	\boxtimes & \boxtimes & \boxtimes & \boxtimes & \boxtimes\\
	\boxtimes & \boxtimes & \boxtimes & \boxtimes & \boxtimes\\
	0 & \boxtimes & \boxtimes & \boxtimes & \boxtimes\\
	0 & 0 & \textcolor{mylightbluetext}{0} & \boxtimes & \boxtimes\\
	0 & 0 & 0 & \boxtimes & \boxtimes
\end{sbmatrix}.
$$
We now state the Hessenberg decomposition:
\begin{theoremHigh}[Hessenberg decomposition\index{Decomposition: Hessenberg}]\label{theorem:hessenberg-decom}
Any $n\times n$ square matrix $\bA$ can be decomposed as 
$$
\bA = \bQ\bH\bQ^\top, 
$$
where $\bH$ is an upper Hessenberg matrix, and $\bQ$ is an orthogonal matrix.
\end{theoremHigh}
For a lower Hessenberg decomposition, the transpose of $\bA$, $\bA^\top$, admits the decomposition $\bA^\top = \bQ\bH^\top\bQ^\top$ if $\bA$ admits the Hessenberg decomposition $\bA = \bQ\bH\bQ^\top$. The Hessenberg decomposition is conceptually similar to the QR decomposition in that both aim to reduce a matrix to a sparser form with zeros in the lower portion.

While the left orthogonal matrix $\bQ$ introduces zeros in $\bH$ (similar to  the QR decomposition), the right orthogonal matrix $\bQ^\top$ does not simplify the matrix further. Then why use Hessenberg decomposition instead of QR decomposition, which achieves zeros even in the lower subdiagonal? 
The answer lies in the intended application. Hessenberg decomposition serves as a preparatory step (phase one) for more advanced factorizations like singular value decomposition (SVD) or UTV decomposition. A more aggressive transformation (e.g., QR decomposition) would introduce zeros in the subdiagonal but disrupt zeros during subsequent transformations.

Furthermore, the form $\bA = \bQ\bH\bQ^\top$ is an  \textit{orthogonal similarity transformation} (Definition~\ref{definition:similar-matrices}), preserving key properties of $\bA$, such as its eigenvalues, rank, and trace (Proposition~\ref{proposition:eigenvalue-similar-matrices}). Thus, studying $\bH$ provides a simplified way to understand the behavior of $\bA$.

Moreover, let $\bA=\bQ\bH\bQ^\top\in\real^{n\times n}$ be given. In certain scenarios, we may need to solve the linear system $(\bA+\gamma\bI)\bx=\bb$ for different values of $\gamma\in\real$ and $\bb\in\real^n$.
The linear system can be equivalently expressed as $(\bH+\gamma\bI)\bQ^\top\bx=\bQ^\top\bb$. Since $\bH$ is upper Hessenberg, the system can be solved efficiently using methods like forward and backward substitution.

\index{Spectrum}
\section{(Orthogonal) Similarity Transformation}
As mentioned earlier, the Hessenberg decomposition introduced in this section, the tridiagonal decomposition in the next section, the Schur decomposition (Theorem~\ref{theorem:schur-decomposition}), and the spectral decomposition (Theorem~\ref{theorem:spectral_theorem}) all share a common structure:  they transform a matrix into another matrix that is similar to it. Below, we formally define similar matrices and similarity transformations.
\index{Similar matrices}\index{Similarity transformation}
\begin{definition}[Similar matrices and similarity transformation]\label{definition:similar-matrices}
Two matrices $\bA$ and $\bB$ are said to be \textit{similar matrices} if there exists a nonsingular matrix $\bP$ such that $\bB = \bP\bA\bP^{-1}$. 

In simpler terms, given any nonsingular matrix $\bP$, the matrices $\bA$ and $\bP\bA\bP^{-1}$ are similar. 
The transformation $\bP\bA\bP^{-1}$
is referred to as a \textit{similarity transformation}  of the matrix $\bA$.

Furthermore, if $\bP$ is an orthogonal matrix, the transformation $\bP\bA\bP^\top$ is also known as an \textit{orthogonal similarity transformation} of $\bA$.
Orthogonal similarity transformations are particularly significant because the condition number of the transformed matrix $\bP\bA\bP^\top$ is no worse than that of the original matrix $\bA$.~\footnote{Note that two matrices $\bA$ and $\bB$ are referred to as  \textit{congruent} if $\bB = \bS\bA\bS^\top$ for some nonsingular matrix $\bS$. In this sense, an orthogonal similarity transformation is both a similarity transformation and a congruence transformation.}
\end{definition}
The distinction  between  similarity transformations and orthogonal similarity transformations will be further clarified in the context of coordinate transformations (Section~\ref{section:coordinate-transformation}). We now proceed to establish some important properties of similar matrices, which will prove useful in later discussions.
\begin{proposition}[Eigenvalue, trace, and rank of similar matrices\index{Trace}]\label{proposition:eigenvalue-similar-matrices}
Any eigenvalue of $\bA$ is also an eigenvalue of $\bP\bA\bP^{-1}$, and vice versa. That is, $\Lambda(\bA) = \Lambda(\bB)$, where $\Lambda(\bX)$ denotes the spectrum of matrix $\bX$ (Definition~\ref{definition:spectrum}).

Moreover, the trace and rank of $\bA$ are equal to those of  $\bP\bA\bP^{-1}$ for any nonsingular matrix $\bP$.
\end{proposition}
\begin{proof}[of Proposition~\ref{proposition:eigenvalue-similar-matrices}]
Let  $(\lambda, \bx)$ be any eigenpair of $\bA$ so that $\bA\bx =\lambda \bx$. 
Then we have $\lambda \bP\bx = \bP\bA\bP^{-1} \bP\bx$ such that $\bP\bx$ is an eigenvector of $\bP\bA\bP^{-1}$ corresponding to $\lambda$.
Conversely, for any eigenpair $(\lambda, \bx)$ of $\bP\bA\bP^{-1}$, we have $\bP\bA\bP^{-1} \bx = \lambda \bx$. Then we have $\bA\bP^{-1} \bx = \lambda \bP^{-1}\bx$ such that $\bP^{-1}\bx$ is an eigenvector of $\bA$ corresponding to $\lambda$. 

Next, consider the trace. Using the cyclic invariance property of the trace, we have: $\trace(\bP\bA\bP^{-1}) = \trace(\bA\bP^{-1}\bP) = \trace(\bA)$.

For the rank, we proceed in two steps:
\paragraph{Rank claim 1: $\rank(\bZ\bA)=\rank(\bA)$ if $\bZ$ is nonsingular.}
For any vector $\bn$ in the null space of $\bA$ (i.e., $\bA\bn = \bzero$), we have $\bZ\bA\bn = \bzero$. Hence, $\bn$ is also in the null space of $\bZ\bA$. And this implies $\nspace(\bA)\subseteq \nspace(\bZ\bA)$.

Conversely, for any vector $\bmm$ in the null space of $\bZ\bA$ (i.e., $\bZ\bA\bmm = \bzero$), we have $\bA\bmm = \bZ^{-1} \bzero=\bzero$. That is, $\bmm$ is also in the null space of $\bA$. And this indicates $\nspace(\bZ\bA)\subseteq \nspace(\bA)$.

Combining both inclusions, we conclude:
$$
\nspace(\bA) = \nspace(\bZ\bA)\quad  \implies \quad \rank(\bZ\bA)=\rank(\bA).
$$

\paragraph{Rank claim 2: $\rank(\bA\bZ)=\rank(\bA)$ if $\bZ$ is nonsingular.}
Using the equality of row and column ranks (Theorem~\ref{theorem:equal-dimension-rank}), we have $\rank(\bA\bZ) = \rank(\bZ^\top\bA^\top)$. Since $\bZ^\top$ is nonsingular, applying claim 1 gives  $\rank(\bZ^\top\bA^\top) = \rank(\bA^\top) = \rank(\bA)$, where the last equality follows again from the fact that the row rank is equal to the column rank for any matrix. This results in $\rank(\bA\bZ)=\rank(\bA)$, as claimed.

Combining these results, and noting that both $\bP$ and $\bP^{-1}$ are nonsingular, we  have $\rank(\bP\bA\bP^{-1}) = \rank(\bA\bP^{-1}) = \rank(\bA)$. This completes the proof.
\end{proof}

\section{Existence of  Hessenberg Decomposition}
We will demonstrate  that any $n\times n$ matrix can be transformed  into Hessenberg form through a sequence of Householder transformations  applied alternately from the left and the right.  
These transformations are performed in an interleaved manner.
Previously, we used Householder reflectors to triangularize matrices by introducing zeros below the diagonal, as part of the QR decomposition process. A similar strategy can be employed to introduce zeros below the subdiagonal, enabling the transformation to Hessenberg form.
Before delving into the mathematical construction of this decomposition, we highlight the following remark, which will prove essential for deriving the decomposition.
\begin{remark}[Left and right multiplied by a matrix with block identity]\label{remark:left-right-identity}
Let $\bA\in \real^{n\times n}$ be a square matrix, and let 
$
\bB = \scriptsize\begin{bmatrix}
	\bI_k &\bzero \\
	\bzero & \bB_{n-k}
\end{bmatrix},
$
where $\bI_k$ is the a $k\times k$ identity matrix. Then, $\bB\bA$ does not alter the first $k$ rows of $\bA$, and $\bA\bB$ does not alter the first $k$ columns of $\bA$.
\end{remark}

\subsection*{\textbf{First Step: Introduce Zeros for the First Column}}	
Let $\bA=[\ba_1, \ba_2, \ldots, \ba_n]$ be the column partition of $\bA$, where each $\ba_i \in \real^{n}$. Suppose $\bar{\ba}_1, \bar{\ba}_2, \ldots, \bar{\ba}_n \in \real^{n-1}$ are the vectors obtained by removing the first component in $\ba_i$'s. Define 
$$
 r_1 = \norm{\bar{\ba}_1}, \qquad \bu_1 = \frac{\bar{\ba}_1 - r_1 \be_1}{\norm{\bar{\ba}_1 - r_1 \be_1}}, \qquad \text{and}\qquad \widetilde{\bH}_1 = \bI - 2\bu_1\bu_1^\top \in \real^{(n-1)\times (n-1)},
$$
where $\be_1$  is the first unit basis in $\real^{n-1}$, i.e., $\be_1=[1;0;0;\ldots;0]\in \real^{n-1}$. To introduce zeros below the subdiagonal and operate on the submatrix $\bA_{2:n,1:n}$, we append the Householder reflector into
$
\bH_1 = \scriptsize
\begin{bmatrix}
	1 &\bzero \\
	\bzero & \widetilde{\bH}_1
\end{bmatrix},
$
in which case, $\bH_1\bA$ will introduce zeros in the first column of $\bA$ below entry (2,1). The first row of $\bA$ remains unchanged, as noted in  Remark~\ref{remark:left-right-identity}. Furthermore, it is straightforward to verify that both $\bH_1$ and $\widetilde{\bH}_1$ are symmetric and orthogonal matrices. To obtain the form in Theorem~\ref{theorem:hessenberg-decom}, we multiply $\bH_1\bA$ on the right by $\bH_1^\top$,  resulting in $\bH_1\bA\bH_1^\top$. The multiplication  on the right will not affect the first column of $\bH_1\bA$, preserving the zeros introduced in that column.

An example of a $5\times 5$ matrix is shown as follows, where $\boxtimes$ represents a value that is not necessarily zero, and \textbf{boldface} indicates the value has just been changed:
$$
\begin{aligned}
\footnotesize
\begin{sbmatrix}{\bA}
\boxtimes & \boxtimes & \boxtimes & \boxtimes & \boxtimes \\
\boxtimes & \boxtimes & \boxtimes & \boxtimes & \boxtimes\\
\boxtimes & \boxtimes & \boxtimes & \boxtimes & \boxtimes\\
\boxtimes & \boxtimes & \boxtimes & \boxtimes & \boxtimes\\
\boxtimes & \boxtimes & \boxtimes & \boxtimes & \boxtimes
\end{sbmatrix}
\stackrel{\bH_1\times}{\rightarrow}
&\footnotesize\begin{sbmatrix}{\bH_1\bA}
\boxtimes & \boxtimes & \boxtimes & \boxtimes & \boxtimes \\
\bm{\boxtimes} & \bm{\boxtimes} & \bm{\boxtimes} & \bm{\boxtimes} & \bm{\boxtimes}\\
\bm{0} & \bm{\boxtimes} & \bm{\boxtimes} & \bm{\boxtimes} & \bm{\boxtimes}\\
\bm{0} & \bm{\boxtimes} & \bm{\boxtimes} & \bm{\boxtimes} & \bm{\boxtimes}\\
\bm{0} & \bm{\boxtimes} & \bm{\boxtimes} & \bm{\boxtimes} & \bm{\boxtimes}
\end{sbmatrix}
\stackrel{\times\bH_1^\top}{\rightarrow}
\footnotesize
\begin{sbmatrix}{\bH_1\bA\bH_1^\top}
\boxtimes & \bm{\boxtimes} & \bm{\boxtimes} & \bm{\boxtimes} & \bm{\boxtimes} \\
\boxtimes & \bm{\boxtimes} & \bm{\boxtimes} & \bm{\boxtimes} & \bm{\boxtimes}\\
0 & \bm{\boxtimes} & \bm{\boxtimes} & \bm{\boxtimes} & \bm{\boxtimes}\\
0 & \bm{\boxtimes} & \bm{\boxtimes} & \bm{\boxtimes} & \bm{\boxtimes}\\
0 & \bm{\boxtimes} & \bm{\boxtimes} & \bm{\boxtimes} & \bm{\boxtimes}
\end{sbmatrix}\\
\end{aligned}
$$

\subsection*{\textbf{Second Step: Introduce Zeros for the Second Column}}	
Let $\bB = \bH_1\bA\bH_1^\top$, where the entries in the first column below entry (2,1) are all zeros. The goal  now is to introduce zeros in the second column below entry (3,2). 
Define $\bB_2 = \bB_{2:n,2:n}=[\bb_1, \bb_2, \ldots, \bb_{n-1}]$. Let again $\bar{\bb}_1, \bar{\bb}_2, \ldots, \bar{\bb}_{n-1} \in \real^{n-2}$ be the vectors obtained by removing the first component from each  $\bb_i$. 
We can again construct a Householder reflector:
\begin{equation}\label{equation:householder-qr-lengthr}
r_1 = \norm{\bar{\bb}_1},  \qquad \bu_2 = \frac{\bar{\bb}_1 - r_1 \be_1}{\norm{\bar{\bb}_1 - r_1 \be_1}}, \qquad \text{and}\qquad \widetilde{\bH}_2 = \bI - 2\bu_2\bu_2^\top\in \real^{(n-2)\times (n-2)},
\end{equation}
where $\be_1$ is now the first unit basis in $\real^{n-2}$. To introduce zeros below the subdiagonal and operate on the submatrix $\bB_{3:n,1:n}$, we extend the Householder reflector into
$
\bH_2 = \scriptsize\begin{bmatrix}
	\bI_2 &\bzero \\
	\bzero & \widetilde{\bH}_2
\end{bmatrix},
$
where $\bI_2$ is the $2\times 2$ identity matrix. We can see that the product $\bH_2\bH_1\bA\bH_1^\top$ does not alter the first two rows of $\bH_1\bA\bH_1^\top$; and since the Householder transformation cannot reflect a zero vector, the zeros in the first column are preserved. 
Again, applying $\bH_2^\top$ to the right of $\bH_2\bH_1\bA\bH_1^\top$ will not change the first two columns, thus preserving the previously introduced zeros.

Following the example of a $5\times 5$ matrix, the second step is shown as follows:
$$
\begin{aligned}
	\footnotesize
	\begin{sbmatrix}{\bH_1\bA\bH_1^\top}
		\boxtimes & \boxtimes & \boxtimes & \boxtimes &  \boxtimes  \\
		\boxtimes & \boxtimes & \boxtimes & \boxtimes &  \boxtimes  \\
		0 & \boxtimes & \boxtimes & \boxtimes &  \boxtimes  \\
		0 & \boxtimes & \boxtimes & \boxtimes &  \boxtimes  \\
		0 & \boxtimes & \boxtimes & \boxtimes &  \boxtimes 
	\end{sbmatrix}
	\stackrel{\bH_2\times}{\rightarrow}
	\footnotesize
	\begin{sbmatrix}{\bH_2\bH_1\bA\bH_1^\top}
		\boxtimes & \boxtimes & \boxtimes & \boxtimes & \boxtimes \\
		\boxtimes & \boxtimes & \boxtimes & \boxtimes & \boxtimes \\
		0 & \bm{\boxtimes} & \bm{\boxtimes} & \bm{\boxtimes} & \bm{\boxtimes}\\
		0 & \bm{0} & \bm{\boxtimes} & \bm{\boxtimes} & \bm{\boxtimes}\\
		0 & \bm{0} & \bm{\boxtimes} & \bm{\boxtimes} & \bm{\boxtimes}
	\end{sbmatrix}
	\stackrel{\times\bH_2^\top}{\rightarrow}
	\footnotesize
	\begin{sbmatrix}{\bH_2\bH_1\bA\bH_1^\top\bH_2^\top}
		\boxtimes & \boxtimes & \bm{\boxtimes} & \bm{\boxtimes} & \bm{\boxtimes} \\
		\boxtimes & \boxtimes & \bm{\boxtimes} & \bm{\boxtimes} & \bm{\boxtimes} \\
		0 & \boxtimes & \bm{\boxtimes} & \bm{\boxtimes} & \bm{\boxtimes}\\
		0 & 0 & \bm{\boxtimes} & \bm{\boxtimes} & \bm{\boxtimes}\\
		0 & 0 & \bm{\boxtimes} & \bm{\boxtimes} & \bm{\boxtimes}
	\end{sbmatrix}.
\end{aligned}
$$

This process continues iteratively, and a total of $n-2$ such steps are required. 
In the end, the matrix will be transformed into Hessenberg form:
$$
\bH = \bH_{n-2} \bH_{n-3}\ldots\bH_1 \bA\bH_1^\top\bH_2^\top\ldots\bH_{n-2}^\top.
$$
Since each $\bH_i$ is symmetric and orthogonal, this simplifies to:
$$
\bH =\bH_{n-2} \bH_{n-3}\ldots\bH_1 \bA\bH_1\bH_2\ldots\bH_{n-2}.
$$
Note that only $n-2$ stages are required, rather than $n-1$ or $n$. This can be verified using the full example  for a $5\times 5$ matrix:
$$
\begin{aligned}
	\footnotesize
\begin{sbmatrix}{\bA}
	\boxtimes & \boxtimes & \boxtimes & \boxtimes & \boxtimes \\
	\boxtimes & \boxtimes & \boxtimes & \boxtimes & \boxtimes\\
	\boxtimes & \boxtimes & \boxtimes & \boxtimes & \boxtimes\\
	\boxtimes & \boxtimes & \boxtimes & \boxtimes & \boxtimes\\
	\boxtimes & \boxtimes & \boxtimes & \boxtimes & \boxtimes
\end{sbmatrix}
\stackrel{\bH_1\times}{\rightarrow}
&\footnotesize\begin{sbmatrix}{\bH_1\bA}
	\boxtimes & \boxtimes & \boxtimes & \boxtimes & \boxtimes \\
	\bm{\boxtimes} & \bm{\boxtimes} & \bm{\boxtimes} & \bm{\boxtimes} & \bm{\boxtimes}\\
	\bm{0} & \bm{\boxtimes} & \bm{\boxtimes} & \bm{\boxtimes} & \bm{\boxtimes}\\
	\bm{0} & \bm{\boxtimes} & \bm{\boxtimes} & \bm{\boxtimes} & \bm{\boxtimes}\\
	\bm{0} & \bm{\boxtimes} & \bm{\boxtimes} & \bm{\boxtimes} & \bm{\boxtimes}
\end{sbmatrix}
\stackrel{\times\bH_1^\top}{\rightarrow}
\footnotesize
\begin{sbmatrix}{\bH_1\bA\bH_1^\top}
	\boxtimes & \bm{\boxtimes} & \bm{\boxtimes} & \bm{\boxtimes} & \bm{\boxtimes} \\
	\boxtimes & \bm{\boxtimes} & \bm{\boxtimes} & \bm{\boxtimes} & \bm{\boxtimes}\\
	0 & \bm{\boxtimes} & \bm{\boxtimes} & \bm{\boxtimes} & \bm{\boxtimes}\\
	0 & \bm{\boxtimes} & \bm{\boxtimes} & \bm{\boxtimes} & \bm{\boxtimes}\\
	0 & \bm{\boxtimes} & \bm{\boxtimes} & \bm{\boxtimes} & \bm{\boxtimes}
\end{sbmatrix}\\
 \stackrel{\bH_2\times}{\rightarrow}
\footnotesize
&\footnotesize\begin{sbmatrix}{\bH_2\bH_1\bA\bH_1^\top}
	\boxtimes & \boxtimes & \boxtimes & \boxtimes & \boxtimes \\
	\boxtimes & \boxtimes & \boxtimes & \boxtimes & \boxtimes \\
	0 & \bm{\boxtimes} & \bm{\boxtimes} & \bm{\boxtimes} & \bm{\boxtimes}\\
	0 & \bm{0} & \bm{\boxtimes} & \bm{\boxtimes} & \bm{\boxtimes}\\
	0 & \bm{0} & \bm{\boxtimes} & \bm{\boxtimes} & \bm{\boxtimes}
\end{sbmatrix}
\stackrel{\times\bH_2^\top}{\rightarrow}
\footnotesize
\begin{sbmatrix}{\bH_2\bH_1\bA\bH_1^\top\bH_2^\top}
	\boxtimes & \boxtimes & \bm{\boxtimes} & \bm{\boxtimes} & \bm{\boxtimes} \\
	\boxtimes & \boxtimes & \bm{\boxtimes} & \bm{\boxtimes} & \bm{\boxtimes} \\
	0 & \boxtimes & \bm{\boxtimes} & \bm{\boxtimes} & \bm{\boxtimes}\\
	0 & 0 & \bm{\boxtimes} & \bm{\boxtimes} & \bm{\boxtimes}\\
	0 & 0 & \bm{\boxtimes} & \bm{\boxtimes} & \bm{\boxtimes}
\end{sbmatrix}\\
\stackrel{\bH_3\times}{\rightarrow}
\footnotesize
&\footnotesize\begin{sbmatrix}{\bH_3\bH_2\bH_1\bA\bH_1^\top\bH_2^\top}
	\boxtimes & \boxtimes & \boxtimes & \boxtimes & \boxtimes \\
	\boxtimes & \boxtimes & \boxtimes & \boxtimes & \boxtimes \\
	0 & \boxtimes & \boxtimes & \boxtimes & \boxtimes \\
	0 & 0 & \bm{\boxtimes} & \bm{\boxtimes} & \bm{\boxtimes}\\
	0 & 0 & \bm{0} & \bm{\boxtimes} & \bm{\boxtimes}
\end{sbmatrix}
\stackrel{\times\bH_3^\top}{\rightarrow}
\footnotesize
\begin{sbmatrix}{\bH_3\bH_2\bH_1\bA\bH_1^\top\bH_2^\top\bH_3^\top}
	\boxtimes & \boxtimes & \boxtimes & \bm{\boxtimes} & \bm{\boxtimes} \\
	\boxtimes & \boxtimes & \boxtimes & \bm{\boxtimes} & \bm{\boxtimes} \\
	0 & \boxtimes & \boxtimes & \bm{\boxtimes} & \bm{\boxtimes}\\
	0 & 0 & \boxtimes & \bm{\boxtimes} & \bm{\boxtimes}\\
	0 & 0 & 0 & \bm{\boxtimes} & \bm{\boxtimes}
\end{sbmatrix}.
\end{aligned}
$$

\section{Properties of  Hessenberg Decomposition}\label{section:hessenberg-decomposition}
The Hessenberg decomposition is not unique, as there are multiple ways to construct the Householder reflectors (e.g.,  Equation~\eqref{equation:householder-qr-lengthr}). However, under mild conditions, different decompositions exhibit a similar structure.

\begin{theorem}[Implicit Q theorem for Hessenberg decomposition\index{Implicit Q theorem}]\label{theorem:implicit-q-hessenberg}
Let $\bA \in \real^{n \times n}$ be a matrix with two Hessenberg decompositions, $\bA = \bU\bH\bU^\top = \bV\bG\bV^\top$, where $\bU = [\bu_1, \bu_2, \ldots, \bu_n]$ and $\bV = [\bv_1, \bv_2, \ldots, \bv_n]$ are the column partitions of $\bU$ and $\bV$, respectively. Assume $k$ is the smallest positive integer such that $h_{k+1,k} = 0$, where $h_{ij}$ denotes the $(i,j)$ entry of $\bH$. Then:
\begin{itemize}
\item If $\bu_1=\bv_1$, then $\bu_i = \pm \bv_i$ and $|h_{i,i-1}| = |g_{i,i-1}|$ for $i\in \{2,3,\ldots,k\}$. 
\item When $k=n$, the Hessenberg matrix $\bH$ is called \textit{unreduced}. Otherwise, if $k<n$, then $g_{k+1,k}=0$.
\end{itemize}
\end{theorem}
\begin{proof}[of Theorem~\ref{theorem:implicit-q-hessenberg}]
Define the orthogonal matrix $\bZ=\bV^\top\bU$. We have:
$$
\left. 
\begin{aligned}
	\bG\bZ &= \bV^\top\bA\bV \bV^\top\bU = \bV^\top\bA\bU \\
	\bZ\bH &= \bV^\top\bU \bU^\top\bA\bU = \bV^\top\bA\bU 
\end{aligned}
\right\}
\quad\implies\quad  
\bG\bZ = \bZ\bH.
$$
For the $(i-1)$-th column, we have
$
\bG\bz_{i-1} = \bZ\bh_{i-1},
$
where $\bz_{i-1}$ and $\bh_{i-1}$ are the $(i-1)$-th columns of $\bZ$ and $\bH$, respectively. Since $h_{l,i-1}=0$ for $l\geq i+1$ (as per the definition of upper Hessenberg matrices), $\bZ\bh_{i-1}$ can be represented as 
$$
\bZ\bh_{i-1} = \sum_{j=1}^{i} h_{j,i-1} \bz_j = h_{i,i-1}\bz_i + \sum_{j=1}^{i-1} h_{j,i-1} \bz_j.
$$
Combining results, we have
$
h_{i,i-1}\bz_i  = \bG\bz_{i-1} - \sum_{j=1}^{i-1} h_{j,i-1} \bz_j.
$
A moment of reflexion reveals that $[\bz_1, \bz_2, \ldots,\bz_k]$ is upper triangular. Since $\bZ$ is orthogonal, it must be diagonal, and each value on the diagonal is in $\{-1, 1\}$ for $i\in \{2,\ldots, k\}$. Then $\bz_1=\be_1$ and $\bz_i = \pm \be_i$ for $i\in \{2,\ldots, k\}$. Additionally,  $\bz_i =\bV^\top\bu_i$ and $h_{i,i-1}=\bz_i^\top (\bG\bz_{i-1} - \sum_{j=1}^{i-1} h_{j,i-1} \bz_j)=\bz_i^\top \bG\bz_{i-1}$. Therefore, for $i\in \{2,\ldots,k\}$, $\bz_i^\top \bG\bz_{i-1}$ is just $\pm g_{i,i-1}$. It follows that 
$
\begin{aligned}
	|h_{i,i-1}| &= |g_{i,i-1}| \,\, \text{and}\,\,   \bu_i = \pm\bv_i,\,\, \forall i\in \{2,3,\ldots,k\}.
\end{aligned}
$ 
This proves the first part. For the second part, if $k<n$, 
$$
\begin{aligned}
	g_{k+1,k} &= \be_{k+1}^\top\bG\be_{k} = \pm \be_{k+1}^\top\underbrace{\bG\bZ}_{\bZ\bH} \be_{k} =  \pm \be_{k+1}^\top\underbrace{\bZ\bH \be_{k}}_{\text{$k$-th column of $\bZ\bH$}} \\
	&= \pm \be_{k+1}^\top \bZ\bh_k = \pm \be_{k+1}^\top \sum_{j=1}^{k+1} h_{jk}\bz_j 
	=\pm \be_{k+1}^\top \sum_{j=1}^{\textcolor{mylightbluetext}{k}} h_{jk}\bz_j=0,
\end{aligned}
$$
where the penultimate equality is derived from the assumption that $h_{k+1,k}=0$. This completes the proof.
\end{proof}
From the above theorem, we observe that if two Hessenberg decompositions of a matrix $\bA$ are both unreduced and share the same first column in their respective orthogonal matrices, the corresponding Hessenberg matrices $\bH$ and $\bG$ are similar matrices such that $\bH = \bD\bG\bD^{-1}$, where $\bD=\diag(\pm 1, \pm 1, \ldots, \pm 1)$. \textit{Moreover, and most importantly, if we impose the condition that the elements on the lower subdiagonal of the Hessenberg matrix $\bH$ are positive (if possible), then the Hessenberg decomposition $\bA=\bQ\bH\bQ^\top$ is uniquely determined by $\bA$ and the first column of $\bQ$.} 
This property is analogous to the uniqueness of the QR decomposition (as established in Corollary~\ref{corollary:unique-qr}) and is crucial for simplifying the \textit{QR algorithm}, which is widely used for computing the singular value decomposition or eigenvalues of a matrix \citep{golub2013matrix, lu2021numerical}.

\index{QR algorithm}

The next concept we introduce is that of a Krylov matrix, defined as follows:
\begin{definition}[Krylov matrix\index{Krylov matrix}]\label{definition:krylov-matrix}
	Given a matrix $\bA\in \real^{n\times n}$, a vector $\bq\in \real^n$, and a scalar $k$, the \textit{Krylov matrix} is defined as:
	$$
	\bK(\bA, \bq, k) = 
	\begin{bmatrix}
		\bq, & \bA\bq, & \ldots, & \bA^{k-1}\bq 
	\end{bmatrix}
	\in \real^{n\times k}.
	$$
\end{definition}

\begin{theorem}[Unreduced Hessenberg]\label{theorem:implicit-q-hessenberg-v2}
Suppose there exists an orthogonal matrix $\bQ$ such that a matrix $\bA\in \real^{n\times n}$ can be factored as $\bA = \bQ\bH\bQ^\top$. Then, $\bQ^\top\bA\bQ=\bH$ is an unreduced upper Hessenberg matrix if and only if $\bR=\bQ^\top \bK(\bA, \bq_1, n)$ is nonsingular and upper triangular, where $\bq_1$ is the first column of $\bQ$.  

 If $\bR$ is singular and $k$ is the smallest index such that $r_{kk}=0$, then $k$ is also the smallest index satisfying $h_{k,k-1}=0$.
\end{theorem}
\begin{proof}[of Theorem~\ref{theorem:implicit-q-hessenberg-v2}] 
Assume $\bH$ is an unreduced upper Hessenberg matrix. Write out the following matrix 
$$
\bR = \bQ^\top \bK(\bA, \bq_1, n) = [\be_1, \bH\be_1, \ldots, \bH^{n-1}\be_1],
$$
where, obviously, $\bR$ is upper triangular with $r_{11}=1$. Observe that $r_{ii} = h_{21}h_{32}\ldots h_{i,i-1}$ for $i\in \{2,3,\ldots, n\}$. When $\bH$ is unreduced, $\bR$ is nonsingular as well. 

Conversely, assume $\bR$ is upper triangular and nonsingular. We observe the recurrence $\br_{k+1} = \bH\br_{k}$, which implies that the $(k+2:n)$-th rows of $\bH[:,1:k]$ are zero and $h_{k+1,k}\neq 0$ for $k\in \{1,2,\ldots, n-1\}$. Thus, $\bH$ is unreduced.

If $\bR$ is singular and $k$ is the smallest index satisfying $r_{kk}=0$, then 
$$
\left. 
\begin{aligned}
r_{k-1,k-1}&=h_{21}h_{32}\ldots h_{k-1,k-2}&\neq 0 \\
r_{kk}&=h_{21}h_{32}\ldots h_{k-1,k-2} h_{k,k-1}&= 0
\end{aligned}
\right\}
\leadto 
h_{k,k-1} =0,
$$
from which the result follows.
\end{proof}

\section{Hessenberg-Triangular Decomposition}\label{section:ht_form}
A factorization that is closely related to the Hessenberg decomposition is called the \textit{Hessenberg-triangular decomposition} for a pair of matrices. 
Given a matrix pair $(\bA, \bB)$, where $\bA, \bB \in \real^{n \times n}$, a preprocessing step of the \textit{QZ decomposition or generalized Schur decomposition} \citep{moler1973algorithm} for solving the regular \textit{generalized eigenvalue problem} $(\bA - \lambda \bB)\bx = \bzero$ involves  computing orthogonal matrices $\bQ, \bZ \in \real^{n \times n}$ such that $\bQ^\top \bA \bZ$ is upper Hessenberg while $\bQ^\top \bB \bZ$ is upper triangular. This so-called \textit{Hessenberg-triangular (HT) form} of the matrix pair $(\bA, \bB)$  significantly reduces the computational cost during the iterative part of the QZ algorithm, which in turn plays a crucial role in the computation of quadratic eigenvalue problems \citep{zhang2017matrix}.

The reduction to HT form begins by computing a QR decomposition $\bB = \bQ_0 \bB_0$, where $\bQ_0$ is orthogonal and $\bB_0$ is upper triangular. The matrices $\bA$ and $\bB$ are then overwritten by $\bQ_0^\top \bA$ and $\bQ_0^\top \bB = \bB_0$, respectively. 
Thus, for the rest of this section, we assume that the matrix $\bB$ in the pair $(\bA, \bB)$ is already in upper triangular form. 
In the HT algorithm, the matrix $\bA$ is then reduced to Hessenberg form by applying a sequence of Givens rotations.
The goal is to reduce $\bA$ to Hessenberg form while maintaining the triangular form of $\bB$. 
This is achieved by premultiplying $\bA$ with Householder reflections or Givens rotations to annihilate elements below the first subdiagonal, and postmultiplying $\bB$ with a different set of Householder reflections or Givens rotations to preserve its triangular form (we use Givens rotations in Algorithm~\ref{alg:hess_tri_form}).

\begin{algorithm}[H]
\caption{Moler and Stewart's HT reduction \citep{moler1973algorithm}}
\label{alg:hess_tri_form} 
\begin{algorithmic}[1] 
	\Require A general matrix $\bA \in \real^{n \times n}$ and an upper triangular matrix $\bB \in \real^{n \times n}$;
	\Ensure Orthogonal  $\bQ, \bZ \in \real^{n \times n}$ such that $(\bH,\bT) = (\bQ^\top \bA \bZ, \bQ^\top \bB \bZ)$ is in HT form;
	\State \textbf{Remark:} $\bL_{i-1,i}, \bR_{i,i-1} \in \real^{n \times n}$ denote  Givens rotations (Section~\ref{section:qr-givens}) acting on rows/columns $i-1$ and $i$. 
	\State Initially set $\bQ \leftarrow \bI_n$, $\bZ \leftarrow \bI_n$, $\bH\leftarrow \bA$, and $\bT\leftarrow \bB$;
	\For{$j = 1, 2, \ldots, n-2$} \Comment{Introduce zeros in the $j$-th column of $\bA$}
	\For{$i = n, n-1, \ldots, j+2$}
	\State Construct $\bL_{i-1,i}$ such that the $(i,j)$-th entry of $\bL_{i-1,i}^\top \bH$ is zero.
	\State Update $\bH \leftarrow \bL_{i-1,i}^\top \bH$, $\bT \leftarrow \bL_{i-1,i}^\top \bT$, $\bQ \leftarrow \bQ \bL_{i-1,i}$.
	\State Construct $\bR_{i,i-1}$ such that the fill-in $(i,i-1)$ entry of $\bT \bR_{i,i-1}$ is zero.
	\State Update $\bH \leftarrow \bH \bR_{i,i-1}$, $\bT \leftarrow \bT \bR_{i,i-1}$, $\bZ \leftarrow \bZ \bR_{i,i-1}$.
	\EndFor
	\EndFor
	\State Output $(\bH,\bT) = (\bQ^\top \bA \bZ, \bQ^\top \bB \bZ)$.
\end{algorithmic}
\end{algorithm}

An example of a $7\times 7$ matrix is shown as follows at $i=5$ and $j=2$, where $\boxtimes$ represents a value that is not necessarily zero, and \textbf{boldface} indicates the value has just been changed. The \textcolor{mylightbluetext}{blue} elements are introduced to zero from a nonzero value; while the \textcolor{brown}{brown} elements are modified to nonzero from a zero value:
$$
\setlength{\arraycolsep}{2pt}
\begin{sbmatrix}{\bL_{4,5}^\top\bH}
	\boxtimes & \boxtimes & \boxtimes & \boxtimes & \boxtimes & \boxtimes & \boxtimes \\
	\boxtimes & \boxtimes & \boxtimes & \boxtimes & \boxtimes & \boxtimes & \boxtimes \\
	0 & \boxtimes & \boxtimes & \boxtimes & \boxtimes & \boxtimes & \boxtimes \\
	0 & \bm{\boxtimes} & \bm{\boxtimes} & \bm{\boxtimes} & \bm{\boxtimes} & \bm{\boxtimes} & \bm{\boxtimes} \\
	0 & \textcolor{mylightbluetext}{\bm{0}} & \bm{\boxtimes} & \bm{\boxtimes} & \bm{\boxtimes} & \bm{\boxtimes} & \bm{\boxtimes} \\
	0 & 0 & \boxtimes & \boxtimes & \boxtimes & \boxtimes & \boxtimes \\
	0 & 0 & \boxtimes & \boxtimes & \boxtimes & \boxtimes & \boxtimes \\
\end{sbmatrix}
\begin{sbmatrix}{\bL_{4,5}^\top\bT}
	\boxtimes & \boxtimes & \boxtimes & \boxtimes & \boxtimes & \boxtimes & \boxtimes \\
	0 & \boxtimes & \boxtimes & \boxtimes & \boxtimes & \boxtimes & \boxtimes \\
	0 & 0 & \boxtimes & \boxtimes & \boxtimes & \boxtimes & \boxtimes \\
	0 & 0 & 0 & \bm{\boxtimes} & \bm{\boxtimes} & \bm{\boxtimes} & \bm{\boxtimes} \\
	0 & 0 & 0 & \textcolor{brown}{\bm{\boxtimes}} & \bm{\boxtimes} & \bm{\boxtimes} & \bm{\boxtimes} \\
	0 & 0 & 0 & 0 & 0 & \boxtimes & \boxtimes \\
	0 & 0 & 0 & 0 & 0 & 0 & \boxtimes \\
\end{sbmatrix}
\rightarrow
\begin{sbmatrix}{\bL_{4,5}^\top\bH\bR_{5,4}}
	\boxtimes & \boxtimes & \boxtimes & \bm{\boxtimes} & \bm{\boxtimes} & {\boxtimes} & \boxtimes \\
	\boxtimes & \boxtimes & \boxtimes & \bm{\boxtimes} & \bm{\boxtimes} & {\boxtimes} & \boxtimes \\
	\boxtimes & \boxtimes & \boxtimes & \bm{\boxtimes} & \bm{\boxtimes} & {\boxtimes} & \boxtimes \\
	0 & \boxtimes & \boxtimes & \bm{\boxtimes} & \bm{\boxtimes} & {\boxtimes} & \boxtimes \\
	0 & 0 & \boxtimes & \bm{\boxtimes} & \bm{\boxtimes} & {\boxtimes} & \boxtimes \\
	0 & 0 & \boxtimes & \bm{\boxtimes} & \bm{\boxtimes} & {\boxtimes} & \boxtimes \\
	0 & 0 & \boxtimes & \bm{\boxtimes} & \bm{\boxtimes} & {\boxtimes} & \boxtimes \\
\end{sbmatrix}
\begin{sbmatrix}{\bL_{4,5}^\top\bT\bR_{5,4}}
	\boxtimes & \boxtimes & \boxtimes & \bm{\boxtimes} & \bm{\boxtimes} & \boxtimes & \boxtimes \\
	0 & \boxtimes & \boxtimes & \bm{\boxtimes} & \bm{\boxtimes} & \boxtimes & \boxtimes \\
	0 & 0 & \boxtimes & \bm{\boxtimes} & \bm{\boxtimes} & \boxtimes & \boxtimes \\
	0 & 0 & 0 & \bm{\boxtimes} & \bm{\boxtimes} & \boxtimes & \boxtimes \\
	0 & 0 & 0 & \textcolor{mylightbluetext}{\bm{0}} & \bm{\boxtimes} & \boxtimes & \boxtimes \\
	0 & 0 & 0 & 0 & 0 & \boxtimes & \boxtimes \\
	0 & 0 & 0 & 0 & 0 & 0 & \boxtimes \\
\end{sbmatrix}.
$$


\section{Tridiagonal Decomposition: Hessenberg in Symmetric Matrices}
Similar to the Hessenberg decomposition, the \textit{tridiagonal decomposition} simplifies matrices and serves as a preliminary step for other algorithms (e.g., diagonalization of a matrix), reducing their computational complexity. We begin by formally defining tridiagonal matrices.

\index{Orthogonal}
\begin{definition}[Tridiagonal matrix\index{Tridiagonal matrix}]\label{definition:tridiagonal-hessenbert}
A \textit{tridiagonal matrix} is a square matrix, where all the entries below the subdiagonal and the entries above the superdiagonal are zero. In other words, a tridiagonal matrix is a special type of \textit{band matrix}. 

The concept of a tridiagonal matrix can also extend to rectangular matrices, with the form inferred from context.

Formally, consider a matrix $\bT\in \real^{n\times n}$ with  entries $t_{ij}$ for  $i,j\in \{1,2,\ldots, n\}$. The matrix $\bT$ is tridiagonal if  $t_{ij}=0$ for all $i\geq j+2$ and $i \leq j-2$.

Additionally, let $i$ denote the smallest positive integer such that $h_{i+1, i}=0$ for $i\in \{1,2,\ldots, n-1\}$. The matrix $\bT$ is termed \textbf{unreduced} if $i=n$.
\end{definition}

For example, the following $5\times 5$ matrix is a tridiagonal matrix:
$$
\footnotesize
\begin{sbmatrix}{possibly\,\, unreduced}
\boxtimes & \boxtimes & 0 & 0 & 0\\
\boxtimes & \boxtimes & \boxtimes & 0 & 0\\
0 & \boxtimes & \boxtimes & \boxtimes & 0\\
0 & 0 & \boxtimes & \boxtimes & \boxtimes\\
0 & 0 & 0 & \boxtimes & \boxtimes
\end{sbmatrix}
\qquad 
\text{or}
\qquad  
\footnotesize
\begin{sbmatrix}{reduced}
\boxtimes & \boxtimes & 0 & 0 & 0\\
\boxtimes & \boxtimes & \boxtimes & 0 & 0\\
0 & \boxtimes & \boxtimes & \boxtimes & 0\\
0 & 0 & \textcolor{mylightbluetext}{0} & \boxtimes & \boxtimes\\
0 & 0 & 0 & \boxtimes & \boxtimes
\end{sbmatrix}.
$$
Clearly, a tridiagonal matrix is a special case of an upper Hessenberg matrix. This allows us to formulate the tridiagonal decomposition as follows:
\begin{theoremHigh}[Tridiagonal decomposition\index{Decomposition: Tridiagonal}]\label{theorem:tridiagonal-decom}
Any $n\times n$ symmetric matrix $\bA$ can be decomposed as 
$$
\bA = \bQ\bT\bQ^\top, 
$$
where $\bT$ is a \textit{symmetric} tridiagonal matrix, and $\bQ$ is an orthogonal matrix.
\end{theoremHigh}
The existence of the tridiagonal decomposition follows directly from applying the Hessenberg decomposition to the symmetric matrix $\bA$.

\section{Properties of  Tridiagonal Decomposition}\label{section:tridiagonal-decomposition}
Like the Hessenberg decomposition, the tridiagonal decomposition is generally not unique. 
However, a similar implicit Q theorem can be stated.
\begin{theorem}[Implicit Q theorem for tridiagonal\index{Implicit Q theorem}]\label{theorem:implicit-q-tridiagonal}
Let $\bA\in \real^{n\times n}$ be a symmetric matrix with two tridiagonal decompositions: $\bA=\bU\bT\bU^\top=\bV\bG\bV^\top$, where $\bU=[\bu_1, \bu_2, \ldots, \bu_n]$ and $\bV=[\bv_1, \bv_2, \ldots, \bv_n]$ are the column partitions of $\bU$ and $\bV$, respectively. Suppose further that $k$ is the smallest positive integer such that $t_{k+1,k}=0$, where $t_{ij}$ is the entry $(i,j)$ of $\bT$. Then: 
\begin{itemize}
\item If $\bu_1=\bv_1$, then $\bu_i = \pm \bv_i$ and $|t_{i,i-1}| = |g_{i,i-1}|$ for $i\in \{2,3,\ldots,k\}$. 
\item When $k=n$, the tridiagonal matrix $\bT$ is called unreduced. However, if $k<n$, then $g_{k+1,k}=0$.
\end{itemize}
\end{theorem}
From the above theorem, we see that constraining the elements of the subdiagonal of $\bT$ to be positive (if possible) ensures that the tridiagonal decomposition $\bA=\bQ\bT\bQ^\top$ is uniquely determined by $\bA$ and the first column of $\bQ$. This is again analogous to the uniqueness of the QR decomposition (see Corollary~\ref{corollary:unique-qr}).

Similarly, a reduced tridiagonal decomposition can be derived using the Krylov matrix (Definition~\ref{definition:krylov-matrix}).
\index{Krylov matrix}
\begin{theorem}[Unreduced tridiagonal]\label{theorem:implicit-q-tridiagonal-v2}
Suppose there exists an orthogonal matrix $\bQ$ such that $\bA\in \real^{n\times n}$ can be factored as $\bA = \bQ\bT\bQ^\top$. Then, $\bQ^\top\bA\bQ=\bT$ is an unreduced tridiagonal matrix if and only if $\bR=\bQ^\top \bK(\bA, \bq_1, n)$ is nonsingular and upper triangular, where $\bq_1$ is the first column of $\bQ$.  

If $\bR$ is singular and $k$ is the smallest index satisfying $r_{kk}=0$, then $k$ is also the smallest index such that $t_{k,k-1}=0$.
\end{theorem}

\begin{problemset}
\item Show that if $\lambda$ is a nonzero eigenvalue of $\bA\bB$, then it is also a nonzero eigenvalue of $\bB\bA$. Explain why this reasoning does not hold when $\lambda=0$. 

\item Show that if either $\bA$ or $\bB$ is invertible, then the matrices $\bA\bB$ and $\bB\bA$ are similar.

\item Let $\bA,\bB\in\real^{n\times n}$ be similar matrices. Show that $\adjugate(\bA)$ and $\adjugate(\bB)$ are also similar.
\item Let $\bA$ be given and $\bP$ be nonsingular. Show that if $\bP\bA\bP^{-1}$ is upper triangular, then the diagonal entries of $\bP\bA\bP^{-1}$ are the eigenvalues of $\bA$.
\item \textbf{Power property of similar matrices.} Let $\bB=\bP\bA\bP^{-1}$. Show that $\bB^k = \bP\bA^k\bP^{-1}$ for $k=1,2,\ldots$; that is, $\bB^k$ and $\bA^k$ are similar if $\bB$ and $\bA$ are similar. If one of $\bA$ and $\bB$ is nonsingular, show that $\bB^{-1} = \bP\bA^{-1}\bP^{-1}$ also holds.
\item In the main section, we transform the given matrix $\bA\in\real^{n\times n}$ into its orthogonal similarity transformation. Use Gaussian elimination matrices (see \eqref{equation:elimination_mat}) to transform into its similarity transformation. Discuss the complexity of your algorithm.	

\item Show  that if $\bA=\bE\bC\bE^{-1}$ and $\bB=\bF\bC\bF^{-1}$, then $\bA$ and $\bB$ are similar matrices.

\item Show that the matrices $\scriptsize\begin{bmatrix}
	4 & 1\\
	-1 & 0
\end{bmatrix}$
and 
$\scriptsize\begin{bmatrix}
	1 & 1\\
	0 & 3
\end{bmatrix}$
are similar.

\item \label{prob:hess_poly1} \textbf{Polynomial.} Let $\bA$ and $\bB$ be similar, and consider a polynomial $p(\bC)=\gamma_n\bC^n+\gamma_{n-1}\bC^{n-1}+\ldots+\gamma_0$. Show that 
$p(\bA)$ and $p(\bB)$ are also similar.
\item \label{prob:hess_poly2} \textbf{Polynomial.}  Let $\bA$ and a nonsingular $\bP$ be given, and consider a polynomial $p(\bC)=\gamma_n\bC^n+\gamma_{n-1}\bC^{n-1}+\ldots+\gamma_0$. Show that $p(\bP\bA\bP^{-1}) = \bP p(\bA)\bP^{-1}$.

\index{Similarity transformation}
\item \textbf{Similarity transformation.} Let $\bA\in\real^{n\times n}$ and let $\bP\in\real^{n\times n}$ be nonsingular. Show that $\det(\bP^{-1}\bA\bP - \lambda \bI)=\det(\bA-\lambda\bI)$. This again demonstrates that the eigenvalues remain unchanged under similarity transformations.

\item  Let $\bH\in\real^{n\times n}$ be an unreduced upper Hessenberg matrix. Show that $\rank(\bH-\lambda\bI)\geq n-1$ for any $\lambda\in\real$.
\item Let $\bH\in\real^{n\times n}$ be an unreduced upper Hessenberg matrix. Show that its geometric multiplicity is 1 for any eigenvalue (Definition~\ref{definition:eigen_multipli}). 

\index{Matrix bandwidth}
\item Let $\bA\in\real^{n\times n}$ be given with a lower bandwidth of $p$ (Definition~\ref{defin:matrix-bandwidth}). Provide an algorithm that computes the Hessenberg decomposition of $\bA$ using Householder reflectors or Givens rotations.

\item \textbf{Hessenberg LU.} Let $\bH\in\real^{n\times n}$ be upper Hessenberg. Show that there exists a set of Gaussian elimination matrices $\bE_1, \bE_2, \ldots, \bE_{n-1}$ with entries bounded by unity (see Equation~\eqref{equation:elimination_mat}) and a set of permutation matrices $\bP_1, \bP_2, \bP_{n-1}$ such that $\bE_{n-1}\bP_{n-1}\ldots \bE_2\bP_2\bE_1\bP_1\bH$ is upper triangular. Discuss the complexity of your algorithm.

\item \textbf{Hessenberg QR.} Let $\bH\in\real^{n\times n}$ be upper Hessenberg. Provide an algorithm that computes the QR decomposition of $\bH$ using Givens rotations with a complexity of $\mathcalO(n^2)$ flops.

\item Let $\bH \in \real^{n \times n}$ be upper Hessenberg with an eigenpair  $(\lambda, \bv)$. Provide an algorithm that computes an orthogonal matrix $\bQ$ such that
$ \bQ^\top\bH \bQ = \scriptsize\begin{bmatrix} \lambda & \bu^\top \\ \bzero & \bH_1 \end{bmatrix}, $
where $\bH_1 \in \real^{(n-1) \times (n-1)}$ is also upper Hessenberg. \textit{Hint: Consider $\bQ$ as a product of Givens rotations.}

\item  (Read Chapter~\ref{chapter:spectral-decomposition} first) Consider a $4\times 4$ Hessenberg matrix:
$$
\bH =
\begin{bmatrix}
	b_1 & c_1 & d_1 & e_1 \\
	a_1 & b_2 & c_2 & d_2 \\
	0   & a_2 & b_3 & c_3 \\
	0   & 0   & a_3 & b_4
\end{bmatrix}.
$$
Show that
\begin{itemize}
	\item If $a_1, a_2, a_3$ are all nonzero, and any eigenvalue $\lambda$ of $\bH$ is a real number, then the geometric multiplicity (Definition~\ref{definition:eigen_multipli}) of $\lambda$ must be equal to one.
	\item If $\bH$ is similar to a symmetric matrix $\bA$, and the algebraic multiplicity (Definition~\ref{definition:eigen_multipli}) of some eigenvalue $\lambda$ of $\bA$ is greater than 1, then at least one of  $a_1, a_2, a_3$  must be zero.
\end{itemize}
\index{Algebraic multiplicity}
\index{Geometric multiplicity}

\item Consult \citet{kagstrom2008blocked, bujanovic2018householder} and derive the complexity of the Hessenberg-triangular decomposition.

\item Let $\bA\in\real^{n\times n}$. Show that $\bA$ is idempotent (i.e., $\bA^2=\bA$) if and only if there exists an orthogonal matrix $\bB\in\real^{n\times n}$ such that $\bA$ and $\bB$ are similar.

\item Show that if $\bA\in\real^{n\times n}$ is similar to an orthogonal matrix, then $\bA^{-1}$ is similar to $\bA^\top$.

\item  Show that all Householder reflection matrices are similar.

\item Let $\bA\in\real^{m\times n}$ be a matrix with full column rank. The matrix $\bH=\bA(\bA^\top\bA)^{-1}\bA^\top$ is known as a projection matrix. Show that all projection matrices $\bH$ obtained by varying $\bA$ (but for particular values of $m$ and $n$) are similar.  \textit{Hint: Use the QR decomposition of $\bA$.}

\item Show that all Givens matrices with the same rotation angle $\theta$ are similar.



\item \citep{golub2013matrix} Let $\bA=\bS+\sigma\bu\bu^\top\in\real^{n\times n}$, where $\bS\in\real^{n\times n}$ is skew-symmetric (satisfying $\bA^\top=-\bA$), $\bu\in\real^n$, and $\sigma\in\real$. Show that there exists an orthogonal matrix $\bQ$ such that $\bQ^\top\bA\bQ = \bT+\sigma\be_1\be_1^\top$, where $\bT$ is tridiagonal and skew-symmetric.

\item Let $\bH\in\real^{n\times n}$ be upper Hessenberg. Provide an algorithm that computes the decomposition $\bH\bR=\bR\bT$, where $\bR$ is unit upper triangular, and $\bT$ is tridiagonal.
\item Based on the proofs of Theorems~\ref{theorem:implicit-q-hessenberg} and~\ref{theorem:implicit-q-hessenberg-v2}, prove Theorems~\ref{theorem:implicit-q-tridiagonal} and \ref{theorem:implicit-q-tridiagonal-v2}.

\item Let $\gamma_0, \gamma_1, \ldots, \gamma_n>0$. Show that the following $n\times n$ tridiagonal matrix is positive definite:
$$
\scriptsize
\begin{bmatrix}
	\gamma_0+\gamma_1 & -\gamma_1 & 0 & \ldots & 0\\
	-\gamma_1 & \gamma_1+\gamma_2 & -\gamma_2 & \ldots & 0\\
	0 & -\gamma_2 & \gamma_2+\gamma_3 & \ldots & 0\\
	\vdots & \vdots & \vdots & \ddots & \vdots\\
	0 & 0 & 0 & \ldots & \gamma_{n-1}+\gamma_n
\end{bmatrix}.
$$
\textit{Hint: Consider the leading principal minors.}

\index{Toeplitz matrix}
\item \citep{higham2002accuracy} Let  $\bT_n(a,b,c)\in\real^{n\times n}$ be a tridiagonal matrix defined as:
$$
\textbf{(toeplitz tridiagonal matrix)}:
\qquad 
\bT_n(a,b,c)=
\scriptsize
\begin{bmatrix}
	b & c & 0 & \ldots & 0\\
	a & b & c & \ldots & 0\\
	0 & a & b & \ldots & 0\\
	\vdots & \vdots & \vdots & \ddots & \vdots\\
	0 & 0 & 0 & \ldots & b
\end{bmatrix}.
$$
Show that the eigenvalues of $\bT_n(a,b,c)$ are $b+2\sqrt{ac}\cos(\frac{k\pi}{n+1})$ for $k\in\{1,2,\ldots,n\}$.

\item \citep{noschese2013tridiagonal} Show that the matrix $\bT_n(a,b,c)$ is normal ($\bT_n^\top\bT_n = \bT_n\bT_n^\top$) if and only if $\abs{a}=\abs{c}$.

\item  Let $\bT\in\real^{n\times n}$ be an unreduced tridiagonal matrix. Show that $\rank(\bT-\lambda\bI)\geq n-1$ for any $\lambda\in\real$.
\item Let $\bT\in\real^{n\times n}$ be an unreduced tridiagonal matrix. Show that its geometric multiplicity (Definition~\ref{definition:eigen_multipli}) is 1 for any eigenvalue. 

\item Let $\bA\in\real^{n\times n}$ be tridiagonal. Show that if $a_{i,i+1}a_{i+1,i} > 0$ for all $i\in\{1,2,\ldots, n - 1\}$, then $\bA$ has $n$ distinct real eigenvalues. Moreover, show that if $a_{i,i+1}a_{i+1,i} \geq 0$ for all  $i\in\{1,2,\ldots, n - 1\}$, then all  eigenvalues of $\bA$ are real. \textit{Hint: Use the Jordan decomposition discussed in Chapter~\ref{chapter:eig_jordan}.}
\end{problemset}

\chapter{Bidiagonal Decomposition}
\section{Bidiagonal Decomposition}\label{section:bidiagonal-decompo}

For a non-square symmetric matrix, reducing it to tridiagonal form is not straightforward. However, we can take an alternative approach by considering a decomposition that involves two distinct orthogonal matrices. To begin, we formally define upper bidiagonal matrices:
\index{Bidiagonal matrix}
\index{Upper bidiagonal matrix}
\begin{definition}[Upper bidiagonal matrix]\label{definition:bidiagonal-matrix}	
An \textit{upper bidiagonal matrix} or simply \textit{bidiagonal matrix} is a square matrix characterized by a banded structure, containing nonzero entries only along the \textit{main diagonal} and the \textit{superdiagonal} (i.e., the diagonal directly above the main diagonal). In this case, the matrix contains exactly two diagonals with nonzero entries.

If the nonzero entries instead appear on the diagonal directly below the main diagonal (i.e., the \textit{subdiagonal}), the matrix is referred to as a  \textit{lower bidiagonal matrix}.

This definition can be naturally extended to rectangular matrices, where the bidiagonal structure can be implied based on the context.
\end{definition}

As an example, consider a $7\times 5$ upper bidiagonal matrix. In such a matrix, all entries below the main diagonal and above the superdiagonal are zero:
$$
\footnotesize
\begin{bmatrix}
\boxtimes & \boxtimes & 0 & 0 & 0\\
0 & \boxtimes & \boxtimes & 0 & 0\\
0 & 0 & \boxtimes & \boxtimes & 0\\
0 & 0 & 0 & \boxtimes & \boxtimes\\
0 & 0 & 0 & 0 & \boxtimes\\
0 & 0 & 0 & 0 & 0\\
0 & 0 & 0 & 0 & 0
\end{bmatrix}.
$$

We now state the following result regarding bidiagonal decomposition:
\begin{theoremHigh}[Bidiagonal decomposition\index{Decomposition: Bidiagonal}]\label{theorem:Golub-Kahan-Bidiagonalization-decom}
Any $m\times n$ matrix $\bA$ can be decomposed as 
$$
\bA = \bU\bB\bV^\top, 
$$
where $\bB$ is an upper bidiagonal matrix, and $\bU\in\real^{m\times m}$ and $\bV\in\real^{n\times n}$ are orthogonal matrices.
\end{theoremHigh}
The process of bidiagonalization shares structural similarities with the singular value decomposition (SVD). The key difference lies in the form of $\bB$, which, in the bidiagonal decomposition, contains nonzero entries specifically on the superdiagonal.  This distinction plays an important role in the numerical computation of the singular value decomposition \citep{golub2013matrix, lu2021numerical}.

\section{Existence of  Bidiagonal Decomposition: Three Approaches}\label{section:exist_bidia_gk}

In earlier discussions, we employed Householder reflectors to triangularize matrices, achieving the QR decomposition by introducing zeros below the main diagonal and the Hessenberg decomposition by introducing zeros below the subdiagonal. A similar strategy can be applied to compute the bidiagonal decomposition.

\subsection*{\textbf{First Step 1.1: Introduce Zeros for the First Column}}	
Let $\bA=[\ba_1, \ba_2, \ldots, \ba_n]$ be the column partition of $\bA$, where  each $\ba_i \in \real^{m}$. 
We  construct the Householder reflector as follows:
$$
r_1 = \norm{\ba_1}, \qquad \bu_1 = \frac{\ba_1 - r_1 \be_1}{\norm{\ba_1 - r_1 \be_1}} ,\qquad \text{and}\qquad \bH_1 = \bI - 2\bu_1\bu_1^\top \in \textcolor{mylightbluetext}{\real^{m\times m}},
$$
where $\be_1$ here is the first standard basis vector in $\textcolor{mylightbluetext}{\real^{m}}$, i.e., $\be_1=[1;0;0;\ldots;0]\in \textcolor{mylightbluetext}{\real^{m}}$.
The matrix $\bH_1$ is  symmetric and orthogonal  (from the definition of Householder reflectors).
Applying $\bH_1$ to $\bA$ introduces zeros in the first column of $\bA$ below the $(1,1)$ entry, effectively reflecting $\ba_1$ to $r_1 \be_1$.

For example, consider a $7 \times 5$ matrix $\bA$. The transformation is illustrated below, where $\boxtimes$ represents a potentially  nonzero value, and \textbf{boldface} indicates entries  modified by the transformation:
$$
\begin{aligned}
	\footnotesize
\begin{sbmatrix}{\bA}
\boxtimes & \boxtimes & \boxtimes & \boxtimes & \boxtimes \\
\boxtimes & \boxtimes & \boxtimes & \boxtimes & \boxtimes\\
\boxtimes & \boxtimes & \boxtimes & \boxtimes & \boxtimes\\
\boxtimes & \boxtimes & \boxtimes & \boxtimes & \boxtimes\\
\boxtimes & \boxtimes & \boxtimes & \boxtimes & \boxtimes\\
\boxtimes & \boxtimes & \boxtimes & \boxtimes & \boxtimes\\
\boxtimes & \boxtimes & \boxtimes & \boxtimes & \boxtimes
\end{sbmatrix}
\stackrel{\bH_1\times}{\rightarrow}
&\footnotesize\begin{sbmatrix}{\bH_1\bA}
\bm{\boxtimes} & \bm{\boxtimes} & \bm{\boxtimes} & \bm{\boxtimes} & \bm{\boxtimes}\\
\bm{0} & \bm{\boxtimes} & \bm{\boxtimes} & \bm{\boxtimes} & \bm{\boxtimes}\\
\bm{0} & \bm{\boxtimes} & \bm{\boxtimes} & \bm{\boxtimes} & \bm{\boxtimes}\\
\bm{0} & \bm{\boxtimes} & \bm{\boxtimes} & \bm{\boxtimes} & \bm{\boxtimes}\\
\bm{0} & \bm{\boxtimes} & \bm{\boxtimes} & \bm{\boxtimes} & \bm{\boxtimes}\\
\bm{0} & \bm{\boxtimes} & \bm{\boxtimes} & \bm{\boxtimes} & \bm{\boxtimes}\\
\bm{0} & \bm{\boxtimes} & \bm{\boxtimes} & \bm{\boxtimes} & \bm{\boxtimes}\\
\end{sbmatrix}.
\end{aligned}
$$
At this stage, the process is similar to the steps used in the QR decomposition using Householder reflectors, as described in Section~\ref{section:qr-via-householder}. To proceed, introducing zeros above the superdiagonal in  $\bH_1\bA$ is equivalent to introducing zeros below the subdiagonal of $(\bH_1\bA)^\top$.

\subsection*{\textbf{First Step 1.2: Introduce Zeros for the First Row}}	
Now, consider the \textit{transpose} of $\bH_1\bA$, denoted as $(\bH_1\bA)^\top = \bA^\top\bH_1^\top \in \real^{n \times m}$. The column partition is given by $\bA^\top\bH_1^\top = [\bz_1, \bz_2, \ldots, \bz_m]$, where each $\bz_i \in \real^n$. 
Let $\bar{\bz}_1, \bar{\bz}_2, \ldots, \bar{\bz}_m \in \real^{n-1}$ represent the vectors obtained by removing the first component of each $\bz_i$. 
We can construct the Householder reflector as follows:
$$
r_1 = \norm{\bar{\bz}_1}, \qquad \bv_1 = \frac{\bar{\bz}_1 - r_1 \be_1}{\norm{\bar{\bz}_1 - r_1 \be_1}},  \qquad \text{and}\qquad \widetilde{\bL}_1 = \bI - 2\bv_1\bv_1^\top \in\textcolor{mylightbluetext}{\real^{(n-1)\times (n-1)}},
$$
where $\be_1$ now denotes the first standard basis vector in $\textcolor{mylightbluetext}{\real^{n-1}}$. To introduce zeros below the subdiagonal and operate on the submatrix $(\bA^\top\bH_1^\top)_{2:n,1:m}$, we extend  the Householder reflector into
$
\bL_1 = \scriptsize\begin{bmatrix}
1 &\bzero \\
\bzero & \widetilde{\bL}_1
\end{bmatrix},
$
where both $\bL_1$ and $\widetilde{\bL}_1$ are orthogonal  and  symmetric (by the definition of  Householder reflectors).
In this case, multiplying $(\bA^\top\bH_1^\top)$ by $\bL_1$ on the left introduces zeros in the first column of $(\bA^\top\bH_1^\top)$ below entry (2,1), i.e., reflect $\bar{\bz}_1$ to $r_1\be_1$. The first row of $(\bA^\top\bH_1^\top)$ remains unchanged, as noted in Remark~\ref{remark:left-right-identity}, ensuring that the zeros introduced in step (1.1) are preserved.

Returning to the original (untransposed) matrix $\bH_1\bA$, multiplying on the right by $\bL_1^\top$ introduces zeros in the first row to the right of entry $(1,2)$.
To illustrate, using the same $7\times 5$ matrix, the transformation is shown below:
$$
\begin{aligned}
	\footnotesize
\begin{sbmatrix}{\bA^\top\bH_1^\top}
\boxtimes & 0 & 0 & 0 & 0 & 0 & 0 \\
\boxtimes & \boxtimes & \boxtimes & \boxtimes & \boxtimes & \boxtimes & \boxtimes\\
\boxtimes & \boxtimes & \boxtimes & \boxtimes & \boxtimes & \boxtimes & \boxtimes\\
\boxtimes & \boxtimes & \boxtimes & \boxtimes & \boxtimes & \boxtimes & \boxtimes\\
\boxtimes & \boxtimes & \boxtimes & \boxtimes & \boxtimes & \boxtimes & \boxtimes\\
\end{sbmatrix}
\stackrel{\bL_1\times}{\rightarrow}
\footnotesize
\begin{sbmatrix}{\bL_1 \bA^\top\bH_1^\top}
\boxtimes & 0 & 0 & 0 & 0 & 0 & 0 \\
\bm{\boxtimes} & \bm{\boxtimes} & \bm{\boxtimes} & \bm{\boxtimes} & \bm{\boxtimes} & \bm{\boxtimes} & \bm{\boxtimes}\\
\bm{0} & \bm{\boxtimes} & \bm{\boxtimes} & \bm{\boxtimes} & \bm{\boxtimes} & \bm{\boxtimes} & \bm{\boxtimes}\\
\bm{0} & \bm{\boxtimes} & \bm{\boxtimes} & \bm{\boxtimes} & \bm{\boxtimes} & \bm{\boxtimes} & \bm{\boxtimes}\\
\bm{0} & \bm{\boxtimes} & \bm{\boxtimes} & \bm{\boxtimes} & \bm{\boxtimes} & \bm{\boxtimes} & \bm{\boxtimes}\\
\end{sbmatrix}
\stackrel{(\cdot)^\top}{\rightarrow}
\footnotesize
\begin{sbmatrix}{\bH_1\bA\bL_1^\top }
\boxtimes & \bm{\boxtimes} & \bm{0} & \bm{0} & \bm{0} \\
0 & \bm{\boxtimes} & \bm{\boxtimes} & \bm{\boxtimes} & \bm{\boxtimes}\\
0 & \bm{\boxtimes} & \bm{\boxtimes} & \bm{\boxtimes} & \bm{\boxtimes}\\
0 & \bm{\boxtimes} & \bm{\boxtimes} & \bm{\boxtimes} & \bm{\boxtimes}\\
0 & \bm{\boxtimes} & \bm{\boxtimes} & \bm{\boxtimes} & \bm{\boxtimes}\\
0 & \bm{\boxtimes} & \bm{\boxtimes} & \bm{\boxtimes} & \bm{\boxtimes}\\
0 & \bm{\boxtimes} & \bm{\boxtimes} & \bm{\boxtimes} & \bm{\boxtimes}\\
\end{sbmatrix}.
\end{aligned}
$$

In summary, $\bH_1\bA\bL_1^\top$ achieves the first step by introducing zeros to both the first column and the first row of $\bA$.
The second step is similar, which introduces zeros to both the second column and the second row of $\bA$.
\subsection*{\textbf{Second Step 2.1: Introduce Zeros for the Second Column}}	
Let $\bB = \bH_1\bA\bL_1^\top$, where all entries  in the first column  below entry (1,1) and all entries in the first row to the right of (1,2) are zero.
The goal of this step is to introduce zeros in the second column below the entry at $(2,2)$.
Let $\bB_2 = \bB_{2:m,2:n}=[\bb_1, \bb_2, \ldots, \bb_{n-1}] \in \real^{(m-1)\times (n-1)}$. 
We can again construct a Householder reflector as follows:
$$
r_1 = \norm{\bb_1},\qquad	\bu_2 = \frac{\bb_1 - r_1 \be_1}{\norm{\bb_1 - r_1 \be_1}}, \qquad  \text{and}\qquad \widetilde{\bH}_2 = \bI - 2\bu_2\bu_2^\top\in \textcolor{mylightbluetext}{\real^{(m-1)\times (m-1)}},
$$
where $\be_1$ now is the first standard basis vector for $\textcolor{mylightbluetext}{\real^{m-1}}$. To introduce zeros below the main diagonal and operate on the submatrix $\bB_{2:m,2:n}$, we append the Householder reflector into
$
\bH_2 = \scriptsize\begin{bmatrix}
1 &\bzero \\
\bzero & \widetilde{\bH}_2
\end{bmatrix}.
$
This transformation ensures that $\bH_2(\bH_1\bA\bL_1^\top)$ does not affect the first row of $(\bH_1\bA\bL_1^\top)$, as shown in Remark~\ref{remark:left-right-identity}. Additionally, because a Householder reflector preserves zero vectors, the zeros in the first column remain unchanged.

Continuing from the previous example, the process applied to the $7\times 5$ matrix is illustrated below:
$$
\begin{aligned}
\footnotesize
\begin{sbmatrix}{\bH_1\bA\bL_1^\top }
\boxtimes & \boxtimes & 0 & 0 & 0 \\
0 & \boxtimes & \boxtimes & \boxtimes & \boxtimes\\
0 & \boxtimes & \boxtimes & \boxtimes & \boxtimes\\
0 & \boxtimes & \boxtimes & \boxtimes & \boxtimes\\
0 & \boxtimes & \boxtimes & \boxtimes & \boxtimes\\
0 & \boxtimes & \boxtimes & \boxtimes & \boxtimes\\
0 & \boxtimes & \boxtimes & \boxtimes & \boxtimes\\
\end{sbmatrix}
\stackrel{\bH_2\times }{\rightarrow}
\footnotesize
\begin{sbmatrix}{\bH_2\bH_1\bA\bL_1^\top }
\boxtimes & \boxtimes & 0& 0 & 0 \\
0 & \bm{\boxtimes} & \bm{\boxtimes} & \bm{\boxtimes} & \bm{\boxtimes}\\
0 & \bm{0} & \bm{\boxtimes} & \bm{\boxtimes} & \bm{\boxtimes}\\
0 & \bm{0} & \bm{\boxtimes} & \bm{\boxtimes} & \bm{\boxtimes}\\
0 & \bm{0} & \bm{\boxtimes} & \bm{\boxtimes} & \bm{\boxtimes}\\
0 & \bm{0} & \bm{\boxtimes} & \bm{\boxtimes} & \bm{\boxtimes}\\
0 & \bm{0} & \bm{\boxtimes} & \bm{\boxtimes} & \bm{\boxtimes}\\
\end{sbmatrix}.
\end{aligned}
$$

\subsection*{\textbf{Second Step 2.2: Introduce Zeros for the Second Row}}	
Following step (1.2), consider the \textit{transpose} of $\bH_2\bH_1\bA\bL_1^\top$, given by $\bL_1\bA^\top\bH_1^\top\bH_2^\top \in \real^{n\times m}$. Assume the column partition of $\bL_1\bA^\top\bH_1^\top\bH_2^\top$ is  $\bL_1\bA^\top\bH_1^\top\bH_2^\top = [\bx_1, \bx_2, \ldots, \bx_m]$, where each $\bx_i \in \real^n$. 
Let  $\bar{\bx}_1, \bar{\bx}_2, \ldots, \bar{\bx}_m \in \real^{n-2}$ denote the vectors obtained by removing the first two components of each $\bx_i$. 
We construct the next Householder reflector:
$$
r_1 = \norm{\bar{\bx}_1},\qquad \bv_2 = \frac{\bar{\bx}_1 - r_1 \be_1}{\norm{\bar{\bx}_1 - r_1 \be_1}}, \qquad \text{and}\qquad \widetilde{\bL}_2 = \bI - 2\bv_2\bv_2^\top \in \textcolor{mylightbluetext}{\real^{(n-2)\times (n-2)}},
$$
where $\be_1$ now is the first standard basis vector for $\textcolor{mylightbluetext}{\real^{n-2}}$. To introduce zeros below the subdiagonal and operate on the submatrix $(\bL_1\bA^\top\bH_1\bH_2)_{3:n,1:m}$, we append the Householder reflector into
$
\bL_2 = \scriptsize\begin{bmatrix}
\bI_2 &\bzero \\
\bzero & \widetilde{\bL}_2
\end{bmatrix},
$
where $\bI_2$ is the $2\times 2$ identity matrix. 
By definition, $\bL_2$ and $\widetilde{\bL}_2$ are both symmetric and orthogonal matrices.
In this case, $\bL_2(\bL_1\bA^\top\bH_1^\top\bH_2^\top)$ will introduce zeros in the second column of $(\bL_1\bA^\top\bH_1^\top\bH_2^\top)$ below entry (3,2). The first two rows of $(\bL_1\bA^\top\bH_1^\top\bH_2^\top)$ remain unaffected and kept unchanged, as noted in Remark~\ref{remark:left-right-identity}. {Furthermore, its first column will be kept unchanged as well}. 

Returning to the   \textit{untransposed} matrix $\bH_2\bH_1\bA\bL_1^\top$, multiplying on the right by $\bL_2^\top$   introduces zeros in the second row to the right of entry (2,3). As before, the transformation for the $7\times 5$ matrix is illustrated below:
$$
\begin{aligned}
\footnotesize
\begin{sbmatrix}{\bL_1\bA^\top\bH_1^\top\bH_2^\top }
\boxtimes & 0 & 0 & 0 & 0 & 0 & 0 \\
\boxtimes & \boxtimes & 0 & 0 & 0 & 0 & 0\\
0 & \boxtimes & \boxtimes & \boxtimes & \boxtimes & \boxtimes & \boxtimes\\
0 & \boxtimes & \boxtimes & \boxtimes & \boxtimes & \boxtimes & \boxtimes\\
0 & \boxtimes & \boxtimes & \boxtimes & \boxtimes & \boxtimes & \boxtimes\\
\end{sbmatrix}
\stackrel{\bL_2\times }{\rightarrow}
\footnotesize
\begin{sbmatrix}{\bL_2\bL_1\bA^\top\bH_1^\top\bH_2^\top }
\boxtimes & 0 & 0 & 0 & 0 & 0 & 0 \\
\boxtimes & \boxtimes & 0 & 0 & 0 & 0 & 0\\
0 & \bm{\boxtimes} & \bm{\boxtimes} & \bm{\boxtimes} & \bm{\boxtimes} & \bm{\boxtimes} & \bm{\boxtimes}\\
0 & \bm{0} & \bm{\boxtimes} & \bm{\boxtimes} & \bm{\boxtimes} & \bm{\boxtimes} & \bm{\boxtimes}\\
0 & \bm{0} & \bm{\boxtimes} & \bm{\boxtimes} & \bm{\boxtimes} & \bm{\boxtimes} & \bm{\boxtimes}\\
\end{sbmatrix}
\stackrel{(\cdot)^\top }{\rightarrow}
\footnotesize
\begin{sbmatrix}{\bH_2\bH_1\bA\bL_1^\top\bL_2^\top}
\boxtimes & \boxtimes & 0 & 0 & 0\\
0 & \boxtimes & \bm{\boxtimes} & \bm{0} & \bm{0} \\
0 & 0 & \bm{\boxtimes} & \bm{\boxtimes} & \bm{\boxtimes}\\
0 & 0 & \bm{\boxtimes} & \bm{\boxtimes} & \bm{\boxtimes}\\
0 & 0 & \bm{\boxtimes} & \bm{\boxtimes} & \bm{\boxtimes}\\
0 & 0 & \bm{\boxtimes} & \bm{\boxtimes} & \bm{\boxtimes}\\
0 & 0 & \bm{\boxtimes} & \bm{\boxtimes} & \bm{\boxtimes}
\end{sbmatrix}.
\end{aligned}
$$

Thus, $\bH_2(\bH_1\bA\bL_1^\top)\bL_2^\top$ completes the second step by introducing zeros into the second column and row of $\bA$.

\index{Golub--Kahan process}
This process can be continued iteratively. It is important to observe that there are $n$ left  reflectors, denoted as $\bH_i$, and $n-2$ right  reflectors, denoted as  $\bL_i$ (suppose $m>n$ for simplicity). 
This alternating application of left and right reflectors is commonly referred to as the \textit{Golub--Kahan bidiagonalization} \citep{golub1965calculating}. Ultimately, this procedure yields the following bidiagonalized form:
$$
\bB = \bH_{n} \bH_{n-1}\ldots\bH_1 \bA\bL_1^\top\bL_2^\top\ldots\bL_{n-2}^\top.
$$
Since all  $\bH_i$'s and $\bL_i$'s are symmetric and orthogonal by definition, this can also be expressed as:
$$
\bB =\bH_{n} \bH_{n-1}\ldots\bH_1 \bA\bL_1\bL_2\ldots\bL_{n-2}.
$$
The complete procedure for the $7\times 5$ matrix is shown as follows:
$$
\begin{aligned}
	\footnotesize
\begin{sbmatrix}{\bA}
\boxtimes & \boxtimes & \boxtimes & \boxtimes & \boxtimes \\
\boxtimes & \boxtimes & \boxtimes & \boxtimes & \boxtimes\\
\boxtimes & \boxtimes & \boxtimes & \boxtimes & \boxtimes\\
\boxtimes & \boxtimes & \boxtimes & \boxtimes & \boxtimes\\
\boxtimes & \boxtimes & \boxtimes & \boxtimes & \boxtimes\\
\boxtimes & \boxtimes & \boxtimes & \boxtimes & \boxtimes\\
\boxtimes & \boxtimes & \boxtimes & \boxtimes & \boxtimes
\end{sbmatrix}
\stackrel{\bH_1\times}{\rightarrow}
&\footnotesize\begin{sbmatrix}{\bH_1\bA}
\bm{\boxtimes} & \bm{\boxtimes} & \bm{\boxtimes} & \bm{\boxtimes} & \bm{\boxtimes}\\
\bm{0} & \bm{\boxtimes} & \bm{\boxtimes} & \bm{\boxtimes} & \bm{\boxtimes}\\
\bm{0} & \bm{\boxtimes} & \bm{\boxtimes} & \bm{\boxtimes} & \bm{\boxtimes}\\
\bm{0} & \bm{\boxtimes} & \bm{\boxtimes} & \bm{\boxtimes} & \bm{\boxtimes}\\
\bm{0} & \bm{\boxtimes} & \bm{\boxtimes} & \bm{\boxtimes} & \bm{\boxtimes}\\
\bm{0} & \bm{\boxtimes} & \bm{\boxtimes} & \bm{\boxtimes} & \bm{\boxtimes}\\
\bm{0} & \bm{\boxtimes} & \bm{\boxtimes} & \bm{\boxtimes} & \bm{\boxtimes}\\
\end{sbmatrix}
\stackrel{\times\bL_1^\top}{\rightarrow}
\footnotesize
\begin{sbmatrix}{\bH_1\bA\bL_1^\top}
\boxtimes & \bm{\boxtimes} & \bm{0} & \bm{0} & \bm{0} \\
0 & \bm{\boxtimes} & \bm{\boxtimes} & \bm{\boxtimes} & \bm{\boxtimes}\\
0 & \bm{\boxtimes} & \bm{\boxtimes} & \bm{\boxtimes} & \bm{\boxtimes}\\
0 & \bm{\boxtimes} & \bm{\boxtimes} & \bm{\boxtimes} & \bm{\boxtimes}\\
0 & \bm{\boxtimes} & \bm{\boxtimes} & \bm{\boxtimes} & \bm{\boxtimes}\\
0 & \bm{\boxtimes} & \bm{\boxtimes} & \bm{\boxtimes} & \bm{\boxtimes}\\
0 & \bm{\boxtimes} & \bm{\boxtimes} & \bm{\boxtimes} & \bm{\boxtimes}\\
\end{sbmatrix}\\
\end{aligned}
$$

$$
\begin{aligned}
\qquad \qquad \qquad \qquad \gap \stackrel{\bH_2\times}{\rightarrow}
&\footnotesize\begin{sbmatrix}{\bH_2\bH_1\bA\bL_1^\top}
\boxtimes & \boxtimes & \boxtimes & \boxtimes & \boxtimes \\
0 & \bm{\boxtimes} & \bm{\boxtimes} & \bm{\boxtimes} & \bm{\boxtimes}\\
0 & \bm{0} & \bm{\boxtimes} & \bm{\boxtimes} & \bm{\boxtimes}\\
0 & \bm{0} & \bm{\boxtimes} & \bm{\boxtimes} & \bm{\boxtimes}\\
0 & \bm{0} & \bm{\boxtimes} & \bm{\boxtimes} & \bm{\boxtimes}\\
0 & \bm{0} & \bm{\boxtimes} & \bm{\boxtimes} & \bm{\boxtimes}\\
0 & \bm{0} & \bm{\boxtimes} & \bm{\boxtimes} & \bm{\boxtimes}
\end{sbmatrix}
\stackrel{\times\bL_2^\top}{\rightarrow}
\footnotesize
\begin{sbmatrix}{\bH_2\bH_1\bA\bL_1^\top\bL_2^\top}
\boxtimes & \boxtimes & 0 & 0 & 0\\
0 & \boxtimes & \bm{\boxtimes} & \bm{0} & \bm{0} \\
0 & 0 & \bm{\boxtimes} & \bm{\boxtimes} & \bm{\boxtimes}\\
0 & 0 & \bm{\boxtimes} & \bm{\boxtimes} & \bm{\boxtimes}\\
0 & 0 & \bm{\boxtimes} & \bm{\boxtimes} & \bm{\boxtimes}\\
0 & 0 & \bm{\boxtimes} & \bm{\boxtimes} & \bm{\boxtimes}\\
0 & 0 & \bm{\boxtimes} & \bm{\boxtimes} & \bm{\boxtimes}
\end{sbmatrix}\\
\end{aligned}
$$
$$
\begin{aligned}
\qquad \qquad \qquad \qquad \qquad  \stackrel{\bH_3\times}{\rightarrow}
&\footnotesize\begin{sbmatrix}{\bH_3\bH_2\bH_1\bA\bL_1^\top\bL_2^\top}
\boxtimes & \boxtimes & \boxtimes & \boxtimes & \boxtimes \\
0 & \boxtimes & \boxtimes & \boxtimes & \boxtimes \\
0 & 0 & \bm{\boxtimes} & \bm{\boxtimes} & \bm{\boxtimes}\\
0 & 0 & \bm{0} & \bm{\boxtimes} & \bm{\boxtimes}\\
0 & 0 & \bm{0} & \bm{\boxtimes} & \bm{\boxtimes}\\
0 & 0 & \bm{0} & \bm{\boxtimes} & \bm{\boxtimes}\\
0 & 0 & \bm{0} & \bm{\boxtimes} & \bm{\boxtimes}
\end{sbmatrix}
\stackrel{\times\bL_3^\top}{\rightarrow}
\footnotesize
\begin{sbmatrix}{\bH_3\bH_2\bH_1\bA\bL_1^\top\bL_2^\top\bL_3^\top}
\boxtimes & \boxtimes & 0 & 0 & 0\\
\boxtimes & \boxtimes & \boxtimes & 0 & 0\\
0 & \boxtimes & \boxtimes & \bm{\boxtimes} & \bm{0}\\
0 & 0 & \boxtimes & \bm{\boxtimes} & \bm{\boxtimes}\\
0 & 0 & 0 & \bm{\boxtimes} & \bm{\boxtimes}\\
0 & 0 & 0 & \bm{\boxtimes} & \bm{\boxtimes}\\
0 & 0 & 0 & \bm{\boxtimes} & \bm{\boxtimes}
\end{sbmatrix}\\
\end{aligned}
$$
$$
\begin{aligned}
	\qquad \qquad \qquad \qquad \qquad \qquad 
\stackrel{\bH_4\times}{\rightarrow}
&\footnotesize\begin{sbmatrix}{\bH_4\bH_3\bH_2\bH_1\bA\bL_1^\top\bL_2\bL_3^\top}
\boxtimes & \boxtimes & 0 & 0 & 0 \\
0 & \boxtimes & \boxtimes & 0 & 0 \\
0 & 0 & \boxtimes& \boxtimes & 0\\
0 & 0 & 0 & \bm{\boxtimes} & \bm{\boxtimes}\\
0 & 0 & 0 & \bm{0} & \bm{\boxtimes}\\
0 & 0 & 0 & \bm{0} & \bm{\boxtimes}\\
0 & 0 & 0 & \bm{0} & \bm{\boxtimes}
\end{sbmatrix}
\stackrel{\bH_5\times}{\rightarrow}
\footnotesize
\begin{sbmatrix}{\bH_5\bH_4\bH_3\bH_2\bH_1\bA\bL_1^\top\bL_2\bL_3^\top}
\boxtimes & \boxtimes & 0 & 0 & 0 \\
0 & \boxtimes & \boxtimes & 0 & 0 \\
0 & 0 & \boxtimes& \boxtimes & 0\\
0 & 0 & 0 & \boxtimes & \boxtimes\\
0 & 0 & 0 & 0 & \bm{\boxtimes}\\
0 & 0 & 0 & 0 & \bm{0}\\
0 & 0 & 0 & 0 & \bm{0}
\end{sbmatrix}.
\end{aligned}
$$


In our implementation, each right Householder reflector $\bL_i$ follows immediately after its corresponding left reflector $\bH_i$. A common mistake is to apply all the left reflectors first, followed by all the right reflectors, which essentially combines a QR decomposition with a Hessenberg decomposition. However, this method is problematic because applying the right reflector $\bL_1$ after all left reflectors would undo the zeros introduced by the latter. To preserve the structure, the left and right reflectors must be applied in an interleaved fashion to maintain and reinforce the zero patterns.

Although the Golub--Kahan bidiagonalization is effective, it is not the most computationally efficient approach for calculating a bidiagonal decomposition.
For an $m\times n$ matrix with $m>n$, the method requires $\sim 4mn^2-\frac{4}{3}n^3$ flops to compute a bidiagonal decomposition. Furthermore, if the explicit computation of the orthogonal matrices $\bU$ and $\bV$ is also required, an additional $\sim 4m^2n-2mn^2 + 2n^3$ flops are needed \citep{lu2021numerical}.

\paragraph{LHC Bidiagonalization.}
Nevertheless, when $m\gg n$, we can extract a square triangular matrix through QR decomposition and then apply the Golub--Kahan bidiagonalization to the resulting $n\times n$ square triangular matrix. This procedure, known as the \textit{Lawson-Hanson-Chan (LHC) bidiagonalization} \citep{lawson1995solving, chan1982improved},  is illustrated in Figure~\ref{fig:lhc-bidiagonal}.\index{LHC bidiagonalization}
\begin{figure}[h]
\centering
\includegraphics[width=0.9\textwidth]{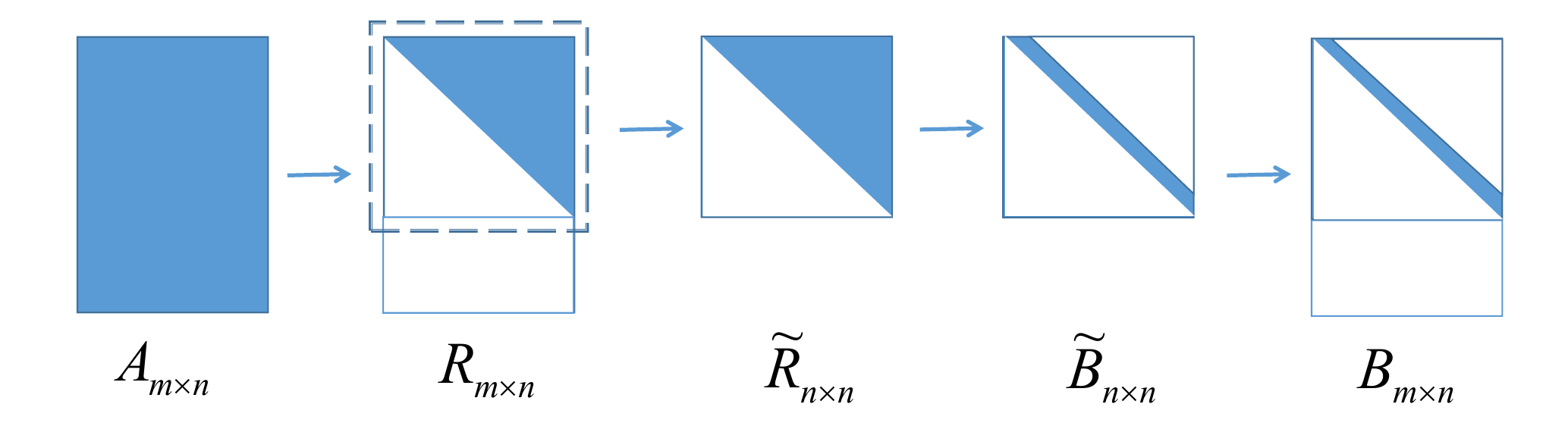}
\caption{Illustration of the LHC bidiagonalization process for  a matrix.}
\label{fig:lhc-bidiagonal}
\end{figure}
The LHC bidiagonalization begins with the full QR decomposition of $\bA$, expressed as $\bA = \bQ\bR$, where $\bQ\in\real^{m\times m}$ is orthogonal and $\bR\in\real^{m\times n}$ is upper triangular. Next, the Golub--Kahan process is applied to the square $n\times n$ triangular submatrix $\widetilde{\bR}$ within $\bR$, resulting in $\widetilde{\bR} = \widetilde{\bU} \widetilde{\bB} \bV^\top$. The matrices $\widetilde{\bU}$ and $\widetilde{\bB} $ are then appended to form 
$$
\bU_0 = 
\begin{bmatrix}
	\widetilde{\bU} & \bzero \\
	\bzero & \bI_{m-n}
\end{bmatrix} \in\real^{m\times m}
\qquad \text{and}\qquad 
\bB = \begin{bmatrix}
	\widetilde{\bB} \\ 
	\bzero_{(m-n)\times n}	
\end{bmatrix}
\in\real^{m\times n},
$$
which gives $\bR=\bU_0\bB \bV^\top$ and $\bA = \bQ\bU_0\bB \bV^\top$. Let $\bU=\bQ\bU_0$, we obtain the desired bidiagonal decomposition of $\bA$. 
The computational cost of the QR decomposition is $\sim 2mn^2-\frac{2}{3}n^3$ flops, while the Golub--Kahan process applied to the $n\times n$ submatrix $\widetilde{\bR}$ requires  $\sim \frac{8}{3}n^3$ \citep{lu2021numerical}. Therefore, the total computational cost for obtaining the bidiagonal matrix $\bB$ through the LHC bidiagonalization is approximately
$$
\text{LHC bidiagonalization:   } \sim 2mn^2 + 2n^3 \text{  flops}.
$$
The LHC process creates zeros and then destroys them again in the lower triangle of the upper $n\times n $ square of $\bR$. 
However, the zeros in the lower $(m-n)\times n$ rectangular submatrix of $\bR$ remain unaffected. Consequently, when $m-n$  is sufficiently large (i.e., $m\gg n$), this approach achieves a net computational gain. 
In fact, simple analysis shows that the LHC bidiagonalization becomes more efficient than the standard Golub--Kahan method when $m>\frac{5}{3}n$.

\begin{figure}[h]
	\centering
	\vspace{-0.35cm}
	\includegraphics[width=0.9\textwidth]{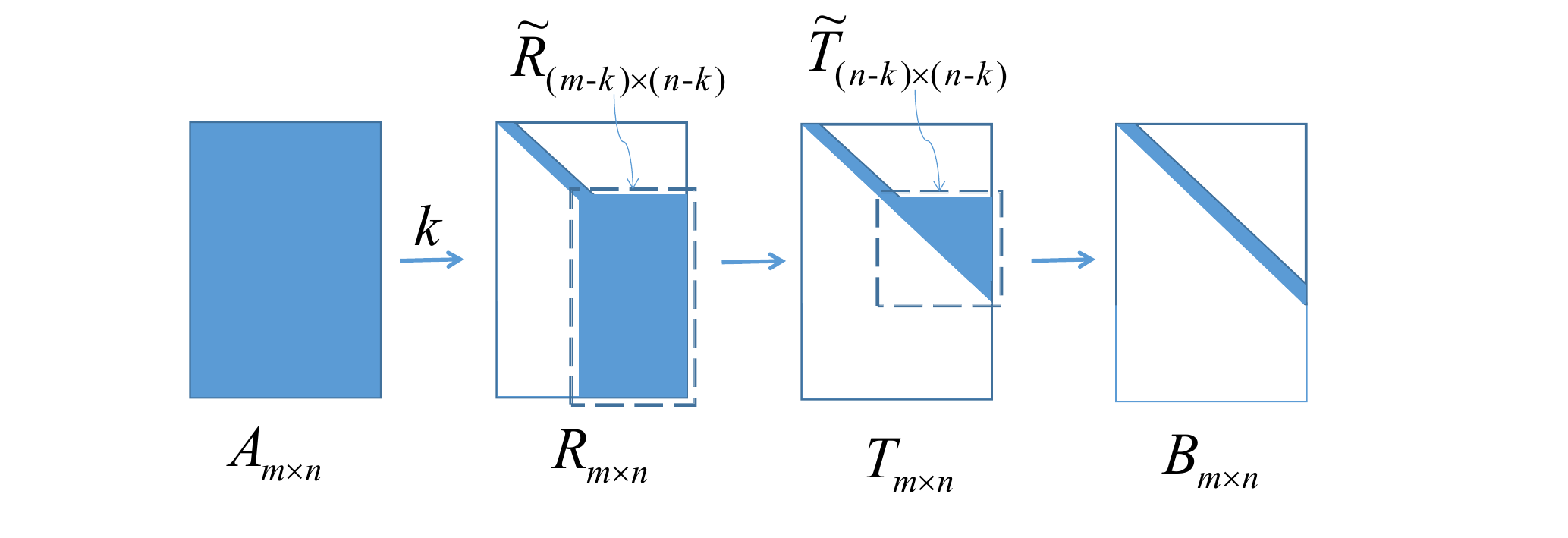}
	\caption{Illustration of the Three-Step bidiagonalization process for a matrix.}
	\label{fig:lhc-bidiagonal2}
\end{figure}
\index{Golub--Kahan process}
\index{LHC bidiagonalization}
\paragraph{Three-Step Bidiagonalization.}
While the LHC method is advantageous when $m>\frac{5}{3}n$, an alternative approach is to apply the QR decomposition at an intermediate stage rather than at the beginning \citep{trefethen1997numerical}. 
This modified process---known as the \textit{Three-Step bidiagonalization} and illustrated in Figure~\ref{fig:lhc-bidiagonal2}---begins with the application of the first $k$ steps of left and right Householder reflectors, as in the Golub--Kahan process, while leaving the bottom-right $(m-k)\times(n-k)$ submatrix ``unreflected."
The LHC procedure is then applied to this submatrix to produce the final bidiagonal decomposition. This adjustment reduces computational complexity in cases where $n<m<2n$.

The computational costs of the three bidiagonalization methods are summarized as follows:
$$ 
\left\{
\begin{aligned}
&\text{Golub--Kahan: }  \sim 4mn^2-\frac{4}{3}n^3 \,\, \text{ flops},   \\
&\text{LHC: } \sim 2mn^2 + 2n^3  \,\, \text{ flops},  \\
&\text{Three-Step: } \sim 2mn^2 + 2m^2n -\frac{2}{3}m^3 -\frac{2}{3}n^3 \,\, \text{ flops} .
\end{aligned}
\right.
$$
When $m>2n$, the LHC method is preferable; when $n<m<2n$, the Three-Step method offers marginal improvements, as shown in Figure~\ref{fig:bidiagonal-loss-compare}, which plots the operation counts of the three methods as a function of $\frac{m}{n}$. Note that the above complexity estimates do not include the cost of computing the orthogonal matrices $\bU$ and $\bV$. These additional costs are omitted here for simplicity.

\begin{SCfigure}
\centering
\includegraphics[width=0.5\textwidth]{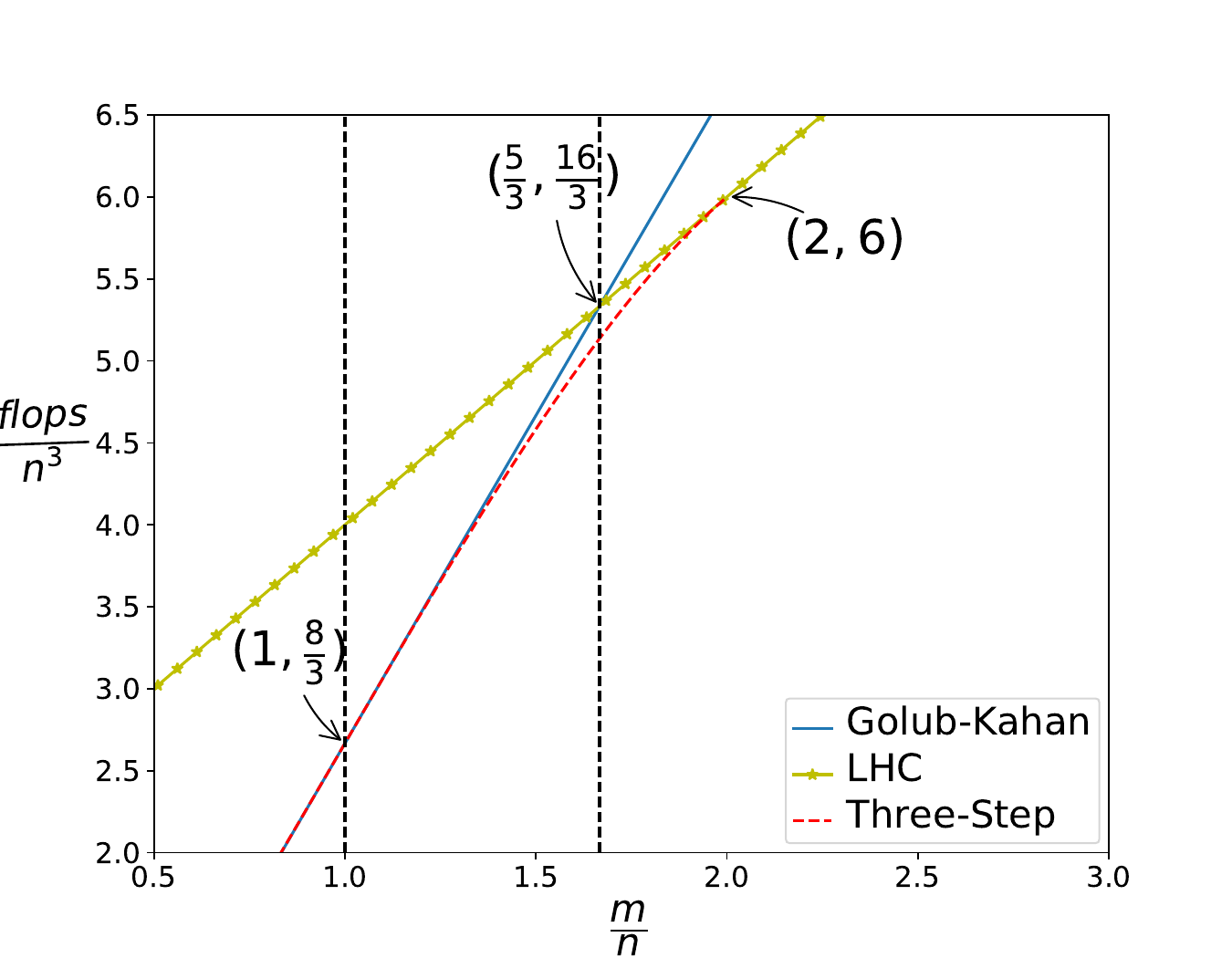}
\caption{Comparison of the computational complexities of the three bidiagonalization methods. When $m>2n$, the LHC method is preferred; when $n<m<2n$, the Three-Step method is slightly more efficient.}
\label{fig:bidiagonal-loss-compare}
\end{SCfigure}

\section{Connection to Tridiagonal Decomposition}

To see the connection to tridiagonal decomposition,
we begin by illustrating the relationship between tridiagonal and bidiagonal decompositions using the following lemma, which explains how to construct a tridiagonal matrix from a bidiagonal one.
\begin{lemma}[Construct tridiagonal from bidiagonal]\label{lemma:construct-triangular-from-bidia}
Let  $\bB\in \real^{n\times n}$ be an upper bidiagonal matrix. Then, $\bT_1=\bB^\top\bB$ and $\bT_2=\bB\bB^\top$ are \textit{symmetric} tridiagonal matrices.
\end{lemma}
This lemma reveals a key property: if  $\bA=\bU\bB\bV^\top$ is the bidiagonal decomposition of $\bA$, then the symmetric matrix $\bA\bA^\top$ admits  a tridiagonal decomposition:
$$
\bA\bA^\top=\bU\bB\bV^\top \bV\bB^\top\bU^\top = \bU\bB\bB^\top\bU^\top.
$$
Similarly, the symmetric matrix $\bA^\top\bA$ also admits a tridiagonal decomposition:
$$
\bA^\top\bA=\bV\bB^\top\bU^\top \bU\bB\bV^\top=\bV\bB^\top\bB\bV^\top.
$$ 
As a final result in this section, we present a theorem that provides the tridiagonal decomposition of a symmetric matrix with nonnegative eigenvalues.
\begin{theoremHigh}[Tridiagonal decomposition for nonnegative eigenvalues]\label{theorem:tri-nonnegative-eigen}
Let   $\bA$ be an $n\times n$ symmetric matrix with nonnegative eigenvalues. Then, there exists a matrix $\bZ$ such that 
$$
\bA=\bZ\bZ^\top.
$$
Furthermore, the tridiagonal decomposition of $\bA$ reduces to finding the bidiagonal decomposition of $\bZ =\bU\bB\bV^\top$ , such that the tridiagonal decomposition of $\bA$ is given by
$$
\bA = \bZ\bZ^\top = \bU\bB\bB^\top\bU^\top.
$$
\end{theoremHigh}
\begin{proof}[of Theorem~\ref{theorem:tri-nonnegative-eigen}]
The eigenvectors of a symmetric matrix can be chosen to be orthogonal (Lemma~\ref{lemma:orthogonal-eigenvectors}), allowing $\bA$ to be decomposed as $\bA=\bQ\bLambda\bQ^\top$ (spectral theorem~\ref{theorem:spectral_theorem}), where $\bLambda$ is a diagonal matrix containing the eigenvalues of $\bA$. When the eigenvalues are nonnegative, $\bLambda$ can be factored as $\bLambda=\bLambda^{1/2} \bLambda^{1/2}$. Setting $\bZ = \bQ\bLambda^{1/2}$, we obtain $\bA=\bZ\bZ^\top$. Combining these results leads to the desired conclusion.
\end{proof}

\begin{problemset}
\item Prove Lemma~\ref{lemma:construct-triangular-from-bidia}.
\item \label{prob:bidia_mln} We discussed the bidiagonalization for a matrix $\bA\in\real^{m\times n}$ with $m\geq n$ in the main section. Provide an algorithm to compute the bidiagonalization when $m<n$, and analyze  its computational complexity. Alternatively, discuss the algorithm for computing $\bA=\bU\bB\bV^\top$ with orthogonal $\bU,\bV$ and lower bidiagonal $\bB$ when $m\geq n$.
\item Prove in detail that the LHC bidiagonalization method is more efficient when $m>\frac{5}{3}n$ compared to the Golub--Kahan bidiagonalization.

\item Prove in detail that the Three-Step bidiagonalization method is more efficient when $n\leq m<2n$ compared to the Golub--Kahan and LHC bidiagonalization methods.

\item (Read Section~\ref{section:SVD} first) Let  $\bA\in\real^{n\times n}$ be upper bidiagonal with  a repeated singular value. Show that $\bA$ must have a zero on its diagonal or superdiagonal.

\item \textbf{Singular values of bidiagonal (read Section~\ref{section:SVD} first \citep{bernstein2008matrix, mathias2014singular}).}
Let $\bA\in\real^{n\times n}$ be upper bidiagonal with the main diagonal values $\{a_1, a_2, \ldots, a_n\}$ and the superdiagonal values $\{b_1,b_2, \ldots, b_{n-1}\}$, and let $\bB\in\real^{n\times n}$ be bidiagonal. Show that
\begin{enumerate}
	\item The singular values of $\bA$ are distinct.
	\item If $\abs{\bB} = \abs{\bA}$, where $\abs{\cdot}$ denotes the element-wise absolute value of a matrix, then $ \bA $ and $ \bB $ have the same singular values.
	\item If $ \abs{\bA} \preceq \abs{\bB} $ (i.e., $\abs{\bB}-\abs{\bA}$ is PSD)  and $ \abs{\bA} \neq \abs{\bB} $, then $ \sigma_{\max}(\bA) < \sigma_{\max}(\bB) $.
	\item If $ \abs{\bI \hadaprod \bA} \preceq \abs{\bI \hadaprod \bB} $ and $ \abs{\bI \hadaprod \bA} \neq \abs{\bI \hadaprod \bB} $, where $\hadaprod$ denotes the Hadamard product, then $ \sigma_{\min}(\bA) < \sigma_{\min}(\bB) $.
	\item If $ \abs{\bI_{\text{up}} \hadaprod \bA} \preceq \abs{\bI_{\text{up}} \hadaprod \bB} $, and $ \abs{\bI_{\text{up}} \hadaprod \bA} \neq \abs{\bI_{\text{up}} \hadaprod \bB} $, where $ \bI_{\text{up}} $ denotes the matrix with all entries on the superdiagonal equal to 1 and all other entries equal to 0, then $ \sigma_{\min}(\bB) < \sigma_{\min}(\bA) $.
\end{enumerate}

\item Explore the process of bidiagonalization using Givens rotations. What happens if the matrix is upper triangular or tridiagonal?

\item Let  $\bA\in\real^{n\times n}$ be upper bidiagonal with $a_{nn} = 0$. Show how to construct orthogonal matrices $\bU$ 
and $\bV$ (as products of Givens rotations) such  that $\bU^\top\bA\bV$ is upper bidiagonal with the $n$-th column being zero.

\item \label{prob:bid_vstruc} Show that the matrix $\bV$ in the bidiagonal decomposition (Theorem~\ref{theorem:Golub-Kahan-Bidiagonalization-decom}) has the  structure 
$
\bV = \scriptsize
\begin{bmatrix}
	1 & \bzero \\
	\bzero & \bQ 
\end{bmatrix},
$
where $\bQ\in\real^{(n-1)\times (n-1)}$ is orthogonal.
\end{problemset}

\part{Eigenvalue Problem}

\chapter{Eigenvalue, Jordan, and Schur Decomposition}\label{chapter:eig_jordan}

\section{Eigenvalue Decomposition}
\index{Decomposition: EVD}
\begin{theoremHigh}[Eigenvalue decomposition]\label{theorem:eigenvalue-decomposition}
Any square matrix $\bA\in \real^{n\times n}$ with linearly independent eigenvectors can be decomposed as 
$$
\bA = \bX\bLambda\bX^{-1},
$$
where $\bX$ contains the eigenvectors of $\bA$ as its columns, and $\bLambda$ is a diagonal matrix $\diag(\lambda_1, \lambda_2,$ $\ldots, \lambda_n)$, with $\lambda_1, \lambda_2, \ldots, \lambda_n$ representing	 the  eigenvalues of $\bA$.
\end{theoremHigh}

This decomposition is known as the \textit{eigenvalue decomposition (EVD)}, or sometimes as \textit{diagonalizing} the matrix   $\bA$.
If all the eigenvalues of $\bA$ are distinct, then its eigenvectors are guaranteed to be linearly independent, and hence $\bA$ can be diagonalized.
Note that without $n$ linearly independent eigenvectors, diagonalization is not possible. In Section~\ref{section:otherform-spectral}, we will explore further conditions under which a matrix has linearly independent eigenvectors.

\begin{proof}[of Theorem~\ref{theorem:eigenvalue-decomposition}]
Let $\bX=[\bx_1, \bx_2, \ldots, \bx_n]$ be the linearly independent eigenvectors of $\bA$. Clearly, we have
$$
\bA\bx_1=\lambda_1\bx_1,\qquad \bA\bx_2=\lambda_2\bx_2, \qquad \ldots, \qquad\bA\bx_n=\lambda_n\bx_n.
$$
Stacking these equations in matrix form yields:
$$
\bA\bX = [\bA\bx_1, \bA\bx_2, \ldots, \bA\bx_n] = [\lambda_1\bx_1, \lambda_2\bx_2, \ldots, \lambda_n\bx_n] = \bX\bLambda.
$$
Since the eigenvectors are assumed to be linearly independent,  the matrix $\bX$ has full rank and is invertible. Therefore, we obtain
$
\bA = \bX\bLambda \bX^{-1}.
$
This completes the proof.
\end{proof}

\index{Geometric multiplicity}
\index{Algebraic multiplicity}
In the spectral decomposition chapter (Chapter~\ref{chapter:spectral-decomposition}), we will discuss similar forms of eigenvalue decomposition, where the matrix $\bA$ is required to be symmetric, and $\bX$ is not only nonsingular but also orthogonal. Alternatively, $\bA$ may be a \textit{simple matrix}, meaning that the algebraic and geometric multiplicities of its eigenvalues are equal. In this case, $\bX$ will be a trivial nonsingular matrix. The decomposition also has a geometric interpretation, which we will explore in Section~\ref{section:coordinate-transformation}.

A matrix decomposition of the form $\bA = \bX\bLambda\bX^{-1}$ has a notable property, allowing for efficient computation of the $m$-th power of $\bA$.
\begin{remark}[$m$-th power]\label{remark:power-eigenvalue-decom}
The $m$-th power of $\bA$ is given by $\bA^m = \bX\bLambda^m\bX^{-1}$ if  $\bA$ can be factored as $\bA=\bX\bLambda\bX^{-1}$.
Computing $\bLambda^m$ is easy because we can apply this operation individually to each diagonal element.
Moreover, if $\bA=\bX\bLambda\bX^{-1}$, then the eigenvalues of $\bA^m$ are precisely the $m$-th powers of the eigenvalues of $\bA$.
\end{remark}

We observe that for the eigenvalue decomposition to exist, the matrix $\bA$ must have a complete set of linearly independent eigenvectors. 
This condition is naturally satisfied under certain circumstances.
\begin{lemma}[Different eigenvalues]\label{lemma:diff-eigenvec-decompo}
If the eigenvalues $\lambda_1, \lambda_2, \ldots, \lambda_n$ of a matrix $\bA\in \real^{n\times n}$ are all distinct, then the corresponding eigenvectors are linearly independent. In other words, any square matrix with distinct eigenvalues can be diagonalized. 
\end{lemma}
\begin{proof}[of Lemma~\ref{lemma:diff-eigenvec-decompo}]
Assume that the eigenvalues $\lambda_1, \lambda_2, \ldots, \lambda_n$ are distinct, but that the eigenvectors $\bx_1,\bx_2, \ldots, \bx_n$ are linearly dependent. Without loss of generality, assume there exists a nonzero vector  $\bc = [c_1,c_2,\ldots,c_{n-1}]^\top$ such that:
$
\bx_n = \sum_{i=1}^{n-1} c_i\bx_{i}. 
$
Then we have 
$$
\begin{aligned}
\bA \bx_n &= \bA \left(\sum_{i=1}^{n-1} c_i\bx_{i}\right) 
=c_1\lambda_1 \bx_1 + c_2\lambda_2 \bx_2 + \ldots + c_{n-1}\lambda_{n-1}\bx_{n-1}.
\end{aligned}
$$
and 
$$
\begin{aligned}
\bA \bx_n &= \lambda_n\bx_n
=\lambda_n (c_1\bx_1 +c_2\bx_2+\ldots +c_{n-1} \bx_{n-1}).
\end{aligned}
$$
Equating these two expressions gives:
$
\sum_{i=1}^{n-1} (\lambda_n - \lambda_i)c_i \bx_i = \bzero .
$
This leads to a contradiction since $\lambda_n \neq \lambda_i$ for all $i\in \{1,2,\ldots,n-1\}$, thus proving that the eigenvectors are linearly independent. 
\end{proof}

There are also several  limitations to the eigenvalue decomposition, which will be  addressed in the following chapters: 
\begin{itemize} 
\item The eigenvectors in $\bX$ are generally not orthogonal, and there may not be enough eigenvectors (i.e., some eigenvalues are repeated). 
\item To compute the eigenvalues and eigenvectors, $\bA\bx = \lambda\bx$, $\bA$ must be square. Rectangular matrices cannot be diagonalized using the eigenvalue decomposition. 
\end{itemize} 	

\section{Jordan Decomposition}
In eigenvalue decomposition, we assume that the matrix $\bA$ has $n$ linearly independent eigenvectors. However, this assumption does not hold for all square matrices. To address this limitation, we introduce a generalized form of eigenvalue decomposition, known as the \textit{Jordan decomposition} or \textit{Jordan canonical form}, named after  \textit{Camille Jordan} \citep{jordan1870traite}.

To describe the Jordan decomposition, we begin by defining \textit{Jordan blocks} and the \textit{Jordan form}.
\begin{definition}[Jordan block]
An $m\times m$ upper triangular matrix $B(\lambda, m)$ is called a \textit{Jordan block} if all its $m$ diagonal elements are equal to $\lambda$, and all superdigonal elements are 1. Mathematically,
$$
B(\lambda, m)=
\scriptsize
\begin{bmatrix}
	\lambda & 1        & 0 & \ldots & 0 & 0 & 0 \\
	0       & \lambda  & 1 & \ldots & 0 & 0 & 0 \\
	0       & 0  & \lambda & \ldots & 0 & 0 & 0 \\
	\vdots  & \vdots   & \vdots & \vdots & \vdots & \vdots\\
	0 & 0 & 0 & \ldots & \lambda & 1& 0\\
	0 & 0 & 0 & \ldots & 0 &\lambda & 1\\
	0 & 0 & 0 & \ldots & 0 & 0 & \lambda
\end{bmatrix}_{m\times m}.
$$
\end{definition}
\begin{definition}[Jordan form\index{Jordan block}]
Given an $n\times n$ matrix $\bA$, a Jordan form $\bJ$ of $\bA$ is a block diagonal matrix of the form:
$$
\bJ=\diag(B(\lambda_1, m_1), B(\lambda_2, m_2), \ldots B(\lambda_k, m_k)),
$$
where $\lambda_1, \lambda_2, \ldots, \lambda_k$ are eigenvalues of $\bA$ (with possible repetitions), and $m_1+m_2+\ldots+m_k=n$.
\end{definition}

\index{Similarity transformation}
Although not all matrices can be decomposed using eigenvalue decomposition, they can be factored using Jordan decomposition. A non-diagonalizable matrix $\bA$ with multiple eigenvalues can be reduced to its Jordan canonical form through a similarity transformation.
\begin{theoremHigh}[Jordan decomposition]\label{theorem:jordan-decomposition}
Any square matrix $\bA\in \real^{n\times n}$ can be decomposed as 
$$
\bA = \bX\bJ\bX^{-1},
$$
where $\bX$ is a nonsingular matrix containing the \textit{generalized eigenvectors} of $\bA$ as its columns, and $\bJ$ is a Jordan form matrix represented as $\diag(\bJ_1, \bJ_2, \ldots, \bJ_k)$. Each block $\bJ_i \in\real^{m_i\times m_i}$ is defined as:
$$
\bJ_i = 
\scriptsize
\begin{bmatrix}
\lambda_i & 1        & 0 & \ldots & 0 & 0 & 0 \\
0       & \lambda_i  & 1 & \ldots & 0 & 0 & 0 \\
0       & 0  & \lambda_i & \ldots & 0 & 0 & 0 \\
\vdots  & \vdots   & \vdots & \vdots & \vdots & \vdots\\
0 & 0 & 0 & \ldots & \lambda_i & 1& 0\\
0 & 0 & 0 & \ldots & 0 &\lambda_i & 1\\
0 & 0 & 0 & \ldots & 0 & 0 & \lambda_i
\end{bmatrix}_{m_i\times m_i},
$$
where $\lambda_i$ is an  eigenvalue of $\bA$, and $m_1+m_2+\ldots +m_k = n$. These blocks $\bJ_i$ are referred to as Jordan blocks.
Furthermore, the nonsingular matrix $\bX$ is called the \textit{matrix of generalized eigenvectors} of $\bA$.
\end{theoremHigh}
For example, a Jordan form $\bJ$ can take the following structure:
$$
\begin{aligned}
\bJ&=\diag(B(\lambda_1, m_1),  \ldots, B(\lambda_k, m_k))
=\footnotesize
\begin{bmatrix}
\begin{bmatrix}
\lambda_1 & 1        & 0 \\
0       & \lambda_1  & 1 \\
0       & 0  & \lambda_1
\end{bmatrix} &  & & &\\  
& \begin{bmatrix}
\lambda_2
\end{bmatrix} & & &\\  
&  &\begin{bmatrix}
\lambda_3 & 1 \\
0 & \lambda_3
\end{bmatrix} & &\\  
&  &   & \ddots & &\\  
&  &   &  & &\begin{bmatrix}
\lambda_k & 1 \\
0 & \lambda_k
\end{bmatrix}\\  
\end{bmatrix}.
\end{aligned}
$$
Note that zeros can appear on the superdiagonal of $\bJ$, and the first column is always a vector containing only eigenvalues of $\bA$ in each block. 
Although Jordan decomposition is theoretically significant, it is rarely used in practice due to its extreme sensitivity to perturbations. Even small random changes to a matrix can render it diagonalizable \citep{van2020advanced}. 
As a result, no major mathematical software libraries or tools provide direct support for computing the Jordan decomposition. Additionally, its proof spans dozens of pages and is beyond the scope of this discussion. Interested readers are encouraged to explore the references for further details
\citep{gohberg1996simple, hales1999jordan, lu2021numerical}.

\section{Schur Decomposition}
\index{Decomposition: Schur}
\index{Upper triangular}
\index{Similarity transformation}
The eigenvalue decomposition is a special case of the \textit{Schur decomposition}.
The latter generalizes the eigenvalue decomposition to all square matrices, even those that are not diagonalizable.
It uses an orthogonal similarity transformation to transform an arbitrary square matrix into an upper triangular matrix. 
This transformation allows many properties of the original matrix to be analyzed using the simpler structure of the upper triangular form.
\begin{theoremHigh}[Schur decomposition]\label{theorem:schur-decomposition}
Any \textbf{real} square matrix $\bA\in \real^{n\times n}$ with \textbf{real} eigenvalues can be decomposed as 
$$
\bA = \bQ\bU\bQ^\top,
$$
where $\bQ$ is a (real) orthogonal matrix, and $\bU$ is a (real) upper triangular matrix. In other words, any real square matrix $\bA$ with real eigenvalues can be \textit{triangularized}.
\end{theoremHigh}

The first columns of $\bA\bQ$ and $\bQ\bU$ are given by $\bA\bq_1$ and $u_{11}\bq_1$, respectively.
Consequently,  $u_{11}$corresponds to an eigenvalue of $\bA$, while  $\bq_1$ serves as its associated eigenvector. 
However, the remaining columns of $\bQ$  are not necessarily eigenvectors of $\bA$. 

\paragraph{Schur decomposition for symmetric matrices.} For a symmetric matrix $\bA=\bA^\top$, the relation $\bQ\bU\bQ^\top = \bQ\bU^\top\bQ^\top$ holds.  In this case, $\bU$ must be diagonal, and this diagonal matrix contains the eigenvalues of $\bA$. Furthermore, all columns of $\bQ$ are eigenvectors of $\bA$.
Thus, we conclude that all symmetric matrices are diagonalizable, even in the presence of repeated eigenvalues; see Chapter~\ref{chapter:spectral-decomposition} for more discussions.

\index{Determinant}
To validate Theorem~\ref{theorem:schur-decomposition}, we rely on the following lemmas.
\begin{lemma}[Determinant intermezzo]\label{lemma:determinant-intermezzo}
The determinant of a matrix (Definition~\ref{definition:determinant}) satisfies the following properties:
\begin{itemize}
\item The determinant of the product of two matrices is given by $\det(\bA\bB)=\det (\bA)\det(\bB)$;

\item The determinant of the transpose of a matrix is the same as the determinant of the original matrix: $\det(\bA^\top) = \det(\bA)$;

\item If matrix $\bA$ has an eigenvalue $\lambda$, then $\det(\bA-\lambda\bI) =0$;

\item The determinant of an identity matrix is $1$;

\item For an orthogonal matrix $\bQ$, the determinant satisfies:
$$
\det(\bQ) = \det(\bQ^\top) = \pm 1, \qquad \text{since  } \det(\bQ^\top)\det(\bQ)=\det(\bQ^\top\bQ)=\det(\bI)=1;
$$

\item For any square matrix $\bA$ and an orthogonal matrix $\bQ$, the determinant relation holds:
$$
\det(\bA) = \det(\bQ^\top) \det(\bA)\det(\bQ) =\det(\bQ^\top\bA\bQ);
$$

\item For a square matrix $\bA \in \real^{n \times n}$, the determinant of $-\bA$ is given by $\det(-\bA) = (-1)^n \det(\bA)$.
\end{itemize}
\end{lemma}

\begin{lemma}[Submatrix with same eigenvalue]\label{lemma:submatrix-same-eigenvalue}
Let  $\bA_{k+1}\in \real^{(k+1)\times (k+1)}$ be a square matrix with real eigenvalues $\lambda_1, \lambda_2, \ldots, \lambda_{k+1}$. Then, we can construct a $k\times k$ matrix $\bA_{k}$ with eigenvalues $\lambda_2, \lambda_3, \ldots, \lambda_{k+1}$ as follows: 
$$
\bA_{k} = 
\begin{bmatrix}
-\bp_2^\top- \\
-\bp_3^\top- \\
\vdots \\
-\bp_{k+1}^\top-
\end{bmatrix}
\bA_{k+1}
\begin{bmatrix}
\bp_2 & \bp_3 &\ldots &\bp_{k+1}
\end{bmatrix},
$$
where $\bp_1$ is a unit-norm eigenvector of $\bA_{k+1}$  corresponding to the eigenvalue $\lambda_1$, and $\bp_2, \bp_3, \ldots, \bp_{k+1}$ denote any mutually orthonormal vectors orthogonal to $\bp_1$, i.e., $\bp_1\in  \cspace^\perp([\bp_2, \bp_3, \ldots, \bp_{k+1}])$. 
\end{lemma}
\begin{proof}[of Lemma~\ref{lemma:submatrix-same-eigenvalue}]
Let $\bP_{k+1} = [\bp_1, \bp_2, \ldots, \bp_{k+1}]$. 
It follows that $\bP_{k+1}^\top\bP_{k+1}=\bI$, and 
$
\bP_{k+1}^\top \bA_{k+1} \bP_{k+1} =
\begin{bmatrixscript}
\lambda_1 & \bzero \\
\bzero  & \bA_{k}
\end{bmatrixscript}. 
$
For any eigenvalue $\lambda \in \{\lambda_2, \lambda_3, \ldots, \lambda_{k+1}\}$, by Lemma~\ref{lemma:determinant-intermezzo}, we have
$$
\begin{aligned}
\det(\bA_{k+1} -\lambda\bI) &= \det(\bP_{k+1}^\top (\bA_{k+1}-\lambda\bI)  \bP_{k+1}) 
=\det(\bP_{k+1}^\top \bA_{k+1}\bP_{k+1} - \lambda\bP_{k+1}^\top\bP_{k+1}) \\
&= \det\left(
\begin{bmatrix}
\lambda_1-\lambda & \bzero \\
\bzero & \bA_k - \lambda\bI
\end{bmatrix}
\right)
=(\lambda_1-\lambda)\det(\bA_k-\lambda\bI).
\end{aligned}
$$
Since $\lambda$ is an eigenvalue of $\bA$ and $\lambda \neq \lambda_1$, it follows that  $\det(\bA_{k+1} -\lambda\bI) = (\lambda_1-\lambda)\det(\bA_{k}-\lambda\bI)=0$, which implies that  $\lambda$ is also an eigenvalue of $\bA_{k}$.
\end{proof}

We now establish the existence of the Schur decomposition using an inductive proof.
\begin{proof}[{of Theorem~\ref{theorem:schur-decomposition}: Existence of Schur decomposition}]
We begin by noting that the theorem is trivial when $n=1$, as we can simply set $Q=1$ and $U=A$. Now, suppose the theorem holds true for $n=k$ for some $k> 1$. To complete the proof, we must show that the theorem also holds for $n=k+1$.
Assume for $n=k$, the theorem is valid, i.e., any matrix $\bA_k\in\real^{k\times k}$ can be expressed as $\bA_k =\bQ_k \bU_k \bQ_k^\top$, where $\bQ_k$ is orthogonal, and $\bU_k$ is upper triangular. 

For $n=k+1$, let $\bA_{k+1}$ be a matrix with eigenvalues $\lambda_1, \lambda_2, \ldots, \lambda_{k+1}$. Using Lemma~\ref{lemma:submatrix-same-eigenvalue}, construct an orthogonal matrix  $\bP_{k+1} = [\bp_1, \bp_2, \ldots, \bp_{k+1}]$, where $\bp_1$ is a unit-norm eigenvector of $\bA_{k+1}$ corresponding to the eigenvalue $\lambda_1$, and $\bp_2, \ldots, \bp_{k+1}$ are mutually orthonormal vectors orthogonal to $\bp_1$. 
Since we assume  the theorem is true for $n=k$, we can find a matrix $\bA_{k}\in\real^{k\times k}$ with eigenvalues $\lambda_2, \lambda_3, \ldots, \lambda_{k+1}$, satisfying $\bA_k =\bQ_k \bU_k \bQ_k^\top$. 
By Lemma~\ref{lemma:submatrix-same-eigenvalue}, the following properties hold:
$$
\bP_{k+1}^\top \bA_{k+1} \bP_{k+1} = 
\begin{bmatrix}
\lambda_1 &\bzero \\
\bzero & \bA_k
\end{bmatrix} 
\qquad
\implies  
\qquad 
\bA_{k+1} \bP_{k+1} =
\bP_{k+1} 
\begin{bmatrix}
\lambda_1 &\bzero \\
\bzero & \bA_k
\end{bmatrix}.
$$
Let 
$
\bQ_{k+1} = \bP_{k+1}
\scriptsize
\begin{bmatrix}
1 &\bzero \\
\bzero & \bQ_k
\end{bmatrix}.
$
Then, it follows that
$$
\begin{aligned}
\bA_{k+1} \bQ_{k+1} &= 
\bA_{k+1}
\bP_{k+1}
\begin{bmatrix}
1 &\bzero \\
\bzero & \bQ_k
\end{bmatrix}
=
\bP_{k+1} 
\begin{bmatrix}
\lambda_1 &\bzero \\
\bzero & \bA_k
\end{bmatrix}
\begin{bmatrix}
1 &\bzero \\
\bzero & \bQ_k
\end{bmatrix} 
=
\bP_{k+1}
\begin{bmatrix}
\lambda_1 & \bzero \\
\bzero & \bA_k\bQ_k
\end{bmatrix}\\
&=
\bP_{k+1}
\begin{bmatrix}
\lambda_1 & \bzero \\
\bzero & \bQ_k \bU_k   
\end{bmatrix}  
=\bP_{k+1}
\begin{bmatrix}
1 &\bzero \\
\bzero & \bQ_k
\end{bmatrix}
\begin{bmatrix}
\lambda_1 &\bzero \\
\bzero & \bU_k
\end{bmatrix}
=\bQ_{k+1}\bU_{k+1},
\end{aligned}
$$
where we let $\bU_{k+1} = \scriptsize\begin{bmatrix}
\lambda_1 &\bzero \\
\bzero & \bU_k
\end{bmatrix}$.
Therefore, $\bA_{k+1} = \bQ_{k+1}\bU_{k+1}\bQ_{k+1}^\top$, where $\bU_{k+1}$ is an upper triangular matrix, and $\bQ_{k+1}$ is an orthogonal matrix since $\bP_{k+1}$ and 
$\scriptsize\begin{bmatrix}
1 &\bzero \\
\bzero & \bQ_k
\end{bmatrix}$ are both orthogonal matrices.
This completes the inductive step and proves the existence of the Schur decomposition.
\end{proof}

\section{Other Forms of  Schur Decomposition}\label{section:other-form-schur-decom}
In the proof of the Schur decomposition, the upper triangular matrix $\bU_{k+1}$ is constructed by appending the eigenvalue $\lambda_1$ to $\bU_k$. This ensures that the diagonal elements consistently represent the eigenvalues of the underlying matrix. Consequently, the upper triangular matrix can be decomposed into two distinct components.
\begin{corollaryHigh}[Form 2 of Schur decomposition]
Any real matrix $\bA\in \real^{n\times n}$ with real eigenvalues can be decomposed as 
$$
\bQ^\top\bA\bQ = \bLambda +\bT\qquad \text{or} \qquad \bA = \bQ(\bLambda +\bT)\bQ^\top,
$$
where $\bQ$ is an orthogonal matrix, $\bLambda=\diag(\lambda_1, \lambda_2, \ldots, \lambda_n)$ is a diagonal matrix containing the eigenvalues of $\bA$, and $\bT$ is a \textit{strictly upper triangular} matrix (with zeros on the diagonal).
\end{corollaryHigh}
A strictly upper triangular matrix is an upper triangular matrix in which all diagonal and lower-triangular entries are zero. 
Another way to understand this decomposition is by noting that $\bA$ and $\bU$ (where $\bU = \bQ^\top\bA\bQ$) are similar matrices, and therefore share the same eigenvalues (Proposition~\ref{proposition:eigenvalue-similar-matrices}).
Moreover, the eigenvalues of any upper triangular matrices are located on its diagonal. 
To see this, consider any upper triangular matrix $\bR \in \real^{n\times n}$, where the diagonal values are $r_{ii}$ for all $i\in \{1,2,\ldots,n\}$. We have
$$
\bR \be_i = r_{ii}\be_i,
$$
where $\be_i$ is the $i$-th standard basis vector in $\real^n$.
Thus, we can decompose $\bU$ into the sum of $\bLambda$ and $\bT$. 

\begin{remark}[$m$-th power]\label{remark:mth_schur}
The above observation also implies that the eigenvalues of the $m$-th power $\bA^m$ are simply the  $m$-th powers of the eigenvalues of $\bA$.
\end{remark}

A final observation about the second form of the Schur decomposition is  as follows. From the equation $\bA\bQ = \bQ(\bLambda +\bT)$, it follows that 
$
\bA \bq_k = \lambda_k\bq_k + \sum_{i=1}^{k-1}t_{ik}\bq_i,
$
where $t_{ik}$ is the ($i,k$)-th entry of $\bT$. The form is quite similar to the eigenvalue decomposition. However, instead of being eigenvectors, the columns of $\bQ$ form an orthonormal basis that is interrelated.

In the main result of Theorem~\ref{theorem:schur-decomposition}, we focus on real matrices with real eigenvalues. However, this restriction may not always be practical in various applications. 
A more general version is presented in the following theorem.
This decomposition is attributed to \textit{Issai Schur} (1875{\textendash}1941), a Russian mathematician who spent most of his professional  life in Germany.
\index{Decomposition: Complex Shur}
\begin{theoremHigh}[Complex Schur decomposition]\label{theorem:schur-decomposition_complex}
Let  $\bA\in \complex^{n\times n}$ be any \textbf{complex} square matrix. Then, it  can be decomposed as 
$$
\bA = \bU\bT\bU^\ast,
$$
where $\bU\in\complex^{n\times n}$ is a unitary matrix, and $\bT\in\complex^{n\times n}$ is an upper triangular matrix (not necessarily real). 
\end{theoremHigh}
\begin{proof}
See \citet{lu2021numerical}.
\end{proof}

\section{Application: Computing Fibonacci Numbers}
Eigenvalue decomposition offers a powerful method for computing Fibonacci numbers \citep{strang1993introduction}. The Fibonacci sequence is defined recursively: each term   $F_{k+2}$ is the sum of the two preceding terms,  $F_{k+1}+F_{k}$. 
The sequence begins as $0, 1, 1, 2, 3, 5, 8, \ldots$. 
A natural question arises: What is the value of $F_{100}$?
Eigenvalue decomposition allows us to derive a general formula for the Fibonacci sequence.
\index{Fibonacci number}
\index{General formula of a sequence}

Let $\bu_{k}=\scriptsize\begin{bmatrix}
	F_{k+1}\\
	F_k
\end{bmatrix}$. 
By the definition of the Fibonacci sequence, we have $\bu_{k+1}=\scriptsize\begin{bmatrix}
	F_{k+2}\\
	F_{k+1}
\end{bmatrix}=
\scriptsize\begin{bmatrix}
	1&1\\
	1&0
\end{bmatrix}
\bu_k
$.
Define $\bA=\scriptsize\begin{bmatrix}
	1&1\\
	1&0
\end{bmatrix}$. 
It follows that $\bu_{100} = \bA^{100}\bu_0$, where $\bu_0=\scriptsize
\begin{bmatrix}
	1\\
	0
\end{bmatrix}
$.

The eigenvalues of $\bA$  are found by solving  $\det(\bA-\lambda\bI)=0$, where $\lambda$ is an eigenvalue of $\bA$ (Remark~\ref{remark:list_sing_equiv}). Solving the characteristic equation $\det(\bA-\lambda\bI) = \lambda^2-\lambda+1=0$, we obtain the eigenvalues and their corresponding eigenvectors:
$$
(\lambda_1, \bx_1) = \left(\frac{1+\sqrt{5}}{2},  
\,\,\,\,
\begin{bmatrix}
	\lambda_1\\
	1
\end{bmatrix}\right)
\quad\text{and}\quad
(\lambda_2, \bx_2)
= 
\left(\frac{1-\sqrt{5}}{2}, 
\,\,\,\,
\begin{bmatrix}
	\lambda_2\\
	1
\end{bmatrix}\right).
$$
As per Remark~\ref{remark:power-eigenvalue-decom},
we can express $\bA^{100} = \bX\bLambda^{100}\bX^{-1} = \bX
\scriptsize
\begin{bmatrix}
	\lambda_1^{100}&0\\
	0&\lambda_2^{100}
\end{bmatrix}\bX^{-1}$, where $\bX^{-1}$ can be easily calculated as $\bX^{-1} =
\scriptsize
\begin{bmatrix}
	\frac{1}{\lambda_1-\lambda_2} & \frac{-\lambda_2}{\lambda_1-\lambda_2} \\
	-\frac{1}{\lambda_1-\lambda_2} & \frac{\lambda_1}{\lambda_1-\lambda_2}
\end{bmatrix} 
=
\scriptsize
\begin{bmatrix}
	\frac{\sqrt{5}}{5} & \frac{5-\sqrt{5}}{10} \\
	-\frac{\sqrt{5}}{5} & \frac{5+\sqrt{5}}{10}
\end{bmatrix}$. We notice that $\bu_{100} = \bA^{100}\bu_0$ corresponds to  the first column of $\bA^{100}$, which can be represented as:
$$
\bu_{100} = 
\begin{bmatrix}
	F_{101}\\
	F_{100}
\end{bmatrix}=
\begin{bmatrix}
	\frac{\lambda_1^{101}-\lambda_2^{101}}{\lambda_1-\lambda_2}\\
	\frac{\lambda_1^{100}-\lambda_2^{100}}{\lambda_1-\lambda_2}
\end{bmatrix}.
$$
Upon a simple check of the calculation,  we have $F_{100}=3.542248481792631e+20$. Or more generally, we can express $\bu_{K}$ as follows:
$$
\bu_{K} = 
\begin{bmatrix}
	F_{K+1}\\
	F_{K}
\end{bmatrix}=
\begin{bmatrix}
	\frac{\lambda_1^{K+1}-\lambda_2^{K+1}}{\lambda_1-\lambda_2}\\
	\frac{\lambda_1^{K}-\lambda_2^{K}}{\lambda_1-\lambda_2}
\end{bmatrix},
$$
where the general form of $F_K$ is given by $F_K=\frac{\lambda_1^{K}-\lambda_2^{K}}{\lambda_1-\lambda_2}$.

\index{Nonsingular matrix}
\index{Matrix polynomial}
\section{Application: Matrix Polynomials}
We previously demonstrated in Problems~\ref{prob:hess_poly1} and \ref{prob:hess_poly2} that $f(\bP\bA\bP^{-1}) =\bP f(\bA) \bP^{-1}$ if $f(\bC) = \gamma_m\bC^m+\gamma_{m-1}\bC^{m-1}+\ldots+\gamma_0$ is a polynomial.
Let 
$
\bA = \bX\bJ\bX^{-1}
$
be the Jordan decomposition of $\bA\in\real^{n\times n}$,
where $\bX\in\real^{n\times n}$ is a nonsingular matrix containing the \textit{generalized eigenvectors} of $\bA$ as its columns, and $\bJ\in\real^{n\times n}$ is a Jordan form matrix $\diag(\bJ_{m_1}(\lambda_1), \bJ_{m_2}(\lambda_2), \ldots, \bJ_{m_k}(\lambda_k)) =\diag(\bJ_1, \bJ_{2}, \ldots, \bJ_{k})$, where $\bJ_{m_i}(\lambda_i)\in\real^{m_i\times m_i}$ and $\sum_{i=1}^{k}m_i  = n$.
Using this decomposition, we have:
\begin{equation}
	f(\bA) = \bX f(\bJ) \bX^{-1}
	=
	\bX \diag(f(\bJ_1), f(\bJ_2), \ldots, f(\bJ_k)) \bX^{-1},
\end{equation}
where 
\begin{equation}
	{\normalsize f(\bJ_i)} = \footnotesize
	\begin{bmatrix}
		f(\lambda_i) & f'(\lambda_i) & \frac{1}{2!} f''(\lambda_i) &  \ldots & \frac{1}{(m_i-1)!} f^{(m_i-1)}(\lambda_i)\\
		0  & f(\lambda_i) & f'(\lambda_i) &  \ldots & \frac{1}{(m_i-2)!} f^{(m_i-2)}(\lambda_i)\\
		0  & 0 & \ddots &  \ddots & \vdots\\
		0  & 0 & 0 &  f(\lambda_i) &f'(\lambda_i)\\
		0  & 0 & 0 &  \ldots & f(\lambda_i)\\
	\end{bmatrix},
\end{equation}
and $f^{(k)}(x)$ denotes the $k$-th  derivative of $f(x)$.
This representation allows us to extend the concept of matrix functions to many common functions that can also be expressed as power series \citep{zhang2017matrix}. Below are several important examples:

\paragraph{Powers of a matrix.}
$\bA^m=  \bX\bJ^m\bX^{-1} = \bX f(\bJ)\bX^{-1}$, where $f(x) = x^m$ for all $m=1,2,\ldots$.

\paragraph{Matrix logarithm.} Let $f(x) = \ln(1 + x)$. Then,
\begin{equation}
	\ln(\bI + \bA) = \sum_{i=1}^{\infty} \frac{(-1)^{i-1}}{i} \bA^i = \bX \left( \sum_{i=1}^{\infty} \frac{(-1)^{i-1}}{i} \bA^i \right) \bX^{-1} = \bX f(\bJ)\bX^{-1}.
\end{equation}

\paragraph{Sine and cosine functions.} Let  $f_1(x) = \sin(x)$ and $f_2(x) = \cos (x)$. Then,
\begin{align}
	\sin(\bA) &= \sum_{i=0}^{\infty} \frac{(-1)^i}{(2i+1)!} \bA^{2i+1} = \bX \left( \sum_{i=0}^{\infty} \frac{(-1)^i}{(2i+1)!} \bJ^{2i+1} \right) \bX^{-1} = \bX f_1(\bJ)\bX^{-1};\\
	\cos (\bA)&= \sum_{i=0}^{\infty} \frac{(-1)^i}{(2i)!} \bA^{2i} = \bX \left( \sum_{i=0}^{\infty} \frac{(-1)^i}{(2i)!} \bJ^{2i} \right) \bX^{-1} = \bX f_2(\bJ)\bX^{-1}.
\end{align}

\paragraph{Matrix exponentials.} Let $f_1(x) = e^x$ and $f_2(x) = e^{-x}$. Then,
\begin{align}
	e^{\bA} &= \sum_{i=0}^{\infty} \frac{1}{i!} \bA^i = \bX \left( \sum_{i=0}^{\infty} \frac{1}{i!} \bJ^i \right) \bX^{-1} = \bX f_1(\bJ)\bX^{-1}; \label{equation:max_expo}\\
	e^{-\bA} &= \sum_{i=0}^{\infty} \frac{1}{i!} (-1)^i \bA^i = \bX \left( \sum_{i=0}^{\infty} \frac{1}{i!} (-1)^i \bJ^i \right) \bX^{-1} = \bX f_2(\bJ)\bX^{-1}.
\end{align}

\paragraph{Matrix exponential functions.} Let $f_1(x) = e^{xt}$ and $f_2(x) = e^{-xt}$. Then,
\begin{align}
	e^{\bA t} &= \sum_{i=0}^{\infty} \frac{1}{i!} \bA^i t^i = \bX \left( \sum_{i=0}^{\infty} \frac{1}{i!} \bJ^i t^i \right) \bX^{-1} = \bX f_1(\bJ)\bX^{-1};\label{equation:max_expo3}\\
	e^{-\bA t} &= \sum_{i=0}^{\infty} \frac{1}{i!} (-1)^i \bA^i t^i = \bX \left( \sum_{i=0}^{\infty} \frac{1}{i!} (-1)^i \bJ^i t^i \right) \bX^{-1} = \bX f_2(\bJ)\bX^{-1}.
\end{align}

\section{Applications and Properties of Schur Decomposition}
A few results can be easily proved using the Schur decomposition. For example, to prove the existence of the spectral decomposition (Theorem~\ref{theorem:spectral_theorem}), to prove the trace of a matrix is equal to the sum of eigenvalues, to prove the existence of the block-diagonalization, and to prove the Schur inequality \citep{lu2021numerical}.
In this section, we present additional results derived from the Schur decomposition.

\index{Cayley--Hamilton theorem}
\subsection*{Cayley--Hamilton Theorem} 
We now provide a rigorous proof of the \textit{Cayley--Hamilton theorem}.
\begin{theorem}[Cayley--Hamilton Theorem]\label{theorem:cayley_hami}
A matrix satisfies its own characteristic equation. That is, given a matrix $\bA\in\real^{n\times n}$, it holds that $p_{\bA}(\lambda)=\det(\lambda\bI-\bA)=\prod_{i=1}^{n}(\lambda-\lambda_i)$ and $p_{\bA}(\lambda)=0$ if $\lambda$ is an eigenvalue of $\bA$.
Then, $\bA$ also satisfies this characteristic equation: $p_{\bA}(\bA)=\prod_{i=1}^{n}(\bA-\lambda_i\bI)=\bzero$.
\end{theorem}
\begin{proof}[of Theorem~\ref{theorem:cayley_hami}]
Suppose $\bA$ admits the Schur decomposition $\bA=\bQ\bU\bQ^\top$. Then, 
$$
p_{\bA}(\bA)=\prod_{i=1}^{n}(\bQ\bU\bQ^\top-\lambda_i\bI) =\bQ \cdot p_{\bA}(\bU) \cdot\bQ^\top.
$$
Therefore, it suffices to show that $p_{\bA}(\bU)=\prod_{i=1}^{n}(\bU-\lambda_i\bI)=\bzero$.
We observe that the upper left 2-by-2 block of $(\bU-\lambda_1\bI)(\bU-\lambda_2\bI)$ is zero. This again invokes the upper left 3-by-3 block $(\bU-\lambda_1\bI)(\bU-\lambda_2\bI)(\bU-\lambda_3\bI)$ to be zero. Continuing this process, the result follows.
\end{proof}

\index{Characteristic polynomial}
\subsection*{Computation of Inverses}
We have shown in Remark~\ref{remark:power-eigenvalue-decom} that the eigenvalue decomposition can help identify the $m$-th power of a matrix easily. 
The Cayley--Hamilton theorem can be used to express the $m$-th power of a square matrix $\bA\in\real^{n\times n}$ as a linear combination of $\bI, \bA, \bA^2, \ldots, \bA^{m-1}$, i.e., as a linear combination of lower power values.
Let the characteristic polynomial of $\bA$ be given by  $p_{\bA}(\lambda)=\det(\lambda\bI-\bA ) =\lambda^n + \gamma_{n-1} \lambda^{n-1} + \ldots + \gamma_1 \lambda  + \gamma_0$. Then, 
\begin{equation}\label{equation:calhm_eq1}
	\bA^n =- \gamma_{n-1} \bA^{n-1} - \ldots - \gamma_1 \bA  - \gamma_0\bI.
\end{equation}
This also implies that
\begin{equation}\label{equation:calhm_eq2}
\bI = -\frac{1}{\gamma_0} (\bA^{n-1}+\gamma_{n-1}\bA^{n-2}+\ldots+\gamma_1)\bA.
\end{equation}
If $\bA$ is nonsingular, then multiplying \eqref{equation:calhm_eq2} by $\bA^{-1}$ yields
$$
\bA^{-1} = -\left(\frac{1}{\gamma_0}\bA^{n-1} +\frac{\gamma_{n-1}}{\gamma_0} \bA^{n-2} + \ldots + \frac{\gamma_1}{\gamma_0} \bI\right).
$$
That is, the inverse of an invertible $n\times n$ matrix $\bA$ can be expressed as a polynomial of $\bA$ of degree at most $(n - 1)$.

\index{Matrix inverse}
\index{Sylvester's theorem}
\subsection*{Sylvester's Theorem}
Given $\bA,\bX\in\real^{n\times n}$, matrices $\bA$ and $\bX$ are said to \textit{commute} if $\bA\bX=\bX\bA$. 
More generally, consider the  equation $\bA\bX=\bX\bB$, where $\bA\in\real^{n\times n}$, $\bB\in\real^{m\times m}$, and $\bX\in\real^{n\times m}$. 
The Cayley--Hamilton theorem indicates (see Problem~\ref{problem:poly_cay}):
\begin{equation}\label{equation:poly_cay_gen}
	p(\bA)\bX=\bX p(\bB), 
	\gap 
	\text{for any polynomial $p(\lambda)$.}
\end{equation}
This relationship leads to Sylvester's Theorem.

\begin{theorem}[Sylvester's theorem]\label{theorem:sylvesters_theorem}
	Let $\bA\in\real^{n\times n}$ and $\bB\in\real^{m\times m}$. 
	\begin{itemize}
		\item If $\Lambda(\bA)\cup \Lambda(\bB)=\varnothing$ (i.e., the intersection of the spectrum sets is empty),  the equation $\bA\bX-\bX\bB=\bzero$ is satisfied only when $\bX=\bzero\in\real^{n\times m}$.
		\item More generally,  \textit{Sylvester's equation} $\bA\bX-\bX\bB=\bC$ has a unique solution $\bX\in\real^{n\times m}$ for each $\bC\in\real^{n\times m}$ if and only if $\Lambda(\bA)\cup \Lambda(\bB)=\varnothing$.~\footnote{If $\bA$ and $\bB$ are complex, then there is a unique complex solution $\bX$ for each $\bC\in\complex^{n\times m}$.}
	\end{itemize}
\end{theorem}
\begin{proof}[of Theorem~\ref{theorem:sylvesters_theorem}]
	The second part is a direct result of the first part; so we only prove the first part. 
	For the first part, it suffices to show that $p_{\bB}(\bA)\bX=\bX p_{\bB}(\bB)=\bzero$ due to \eqref{equation:poly_cay_gen}.
	Suppose $\bB$ has eigenvalues $\lambda_1, \lambda_2, \ldots,\lambda_n$ and admits the characteristic polynomial $p_{\bB}(\lambda)=\prod_{i=1}^{n}(\lambda-\lambda_i)$ and $p_{\bB}(\bA)=\prod_{i=1}^{n}(\bA-\lambda_i\bI)$.
	If  $\Lambda(\bA)\cup \Lambda(\bB)=\varnothing$, then each component $(\bA-\lambda_i\bI)$ is nonsingular, and $p_{\bB}(\bA)$ is nonsingular.
	Therefore, $p_{\bB}(\bA)\bX=\bzero$ if and only if $\bX=\bzero$.
	Conversely, if $p_{\bB}(\bA)\bX=\bzero$ has a nontrivial solution, then at least one component $(\bA-\lambda_i\bI)$ must be singular. Thus, $\Lambda(\bA)\cup \Lambda(\bB)\neq \varnothing$.
\end{proof}

The existence of the Schur decomposition reveals the eigenvalues of $\bB^{-1}\bA$ (when $\bB$ is nonsingular) from the upper triangular matrices.
\begin{corollary}[Eigenvalues from Schur]\label{corollary:eig_utv}
Suppose $\bA,\bB\in\real^{n\times n}$ admit  decompositions $\bA=\bQ\bT_A\bV^\top$ and $\bB=\bQ\bT_B\bV^\top$, respectively, where $\bQ, \bV$ are orthogonal and $\bT_A, \bT_B$ are upper triangular.
Then, the diagonal elements of $\bT_B^{-1}\bT_A$ are the eigenvalues of $\bB^{-1}\bA$ (we assumes all the eigenvalues are real).
\end{corollary}
\begin{proof}[of Corollary~\ref{corollary:eig_utv}]
The proof relies on the Schur decomposition (Theorem~\ref{theorem:schur-decomposition}), and we assume all the eigenvalues discussed are real for simplicity.
Suppose $\bB^{-1}\bA$ admits  a Schur decomposition $\bB^{-1}\bA=\bV\bU\bV^\top$ ($\bV$ is orthogonal, $\bU$ is upper triangular)~\footnote{In the corollary, if we don't assume real eigenvalues, then $\bU$ can be upper quasi-triangular. And $\bT_A$ shown below is also upper quasi-triangular.}, and $\bB\bV$ admits a QR decomposition $\bB\bV=\bQ\bT_B\implies \bB=\bQ\bT_B\bV^\top$ ($\bQ$ is orthogonal, $\bT_B$ is upper triangular).
Then, $\bA=\bB\bV\bU\bV^\top=\bQ\underbrace{(\bT_B\bU)}_{=\bT_A}\bV^\top$, where $\bT_A= \bT_B\bU$ is upper triangular.
This completes the proof.
\end{proof}

\begin{exercise}
Discuss the connection between the decompositions in Corollary~\ref{corollary:eig_utv} and the UTV decomposition (Section~\ref{section:ulv-urv-decomposition}).
\end{exercise}

\begin{problemset}
\item Show that if a matrix $\bA$ satisfies $\bA^2 = 4\bI$, then all eigenvalues of $\bA$ are 2 and $-2$.

\item Given a matrix $\bA\in\real^{n\times n}$ where all entries are equal to 1, find the $n$ eigenvalues of $\bA$.
\item  Let $\bA\in\real^{n\times n}$ be an idempotent matrix (i.e., $\bA^2=\bA$). Show that the matrices $\bB\bA$ and $\bA\bB\bA$ share the same eigenvalues.

\item Consider a Householder transformation matrix $\bH = \bI - 2\bu\bu^\top\in\real^{n\times n}$, where $\norm{\bu}=1$. Show that $\bu$ is an eigenvector $\bH$ and determine its corresponding eigenvalue. Provide a geometric interpretation of the eigenvalues of $\bH$.
Suppose further  that $\bv^\top\bu=0$, where $\bv$ is a nonzero vector. Show that $\bv$ is also an eigenvector of $\bH$ and find its corresponding eigenvalue.

\item Let $\lambda$ be an eigenvalue of $\bA\in\real^{n\times n}$. Show that $\lambda-\mu$ is an eigenvalue of $\bA-\mu\bI$.

\item Let $\lambda$ be an eigenvalue of a nonsingular matrix $\bA\in\real^{n\times n}$. Show that $\lambda^{-1}$ is an eigenvalue of $\bA^{-1}$.

\item Derive the general formula for $\bu_{K} = \bA^{K}\bu_0$, where $\bA$ is a general $2\times 2$ matrix.

\item \label{prob:speci_eign} Consider the  matrix $\bA=\bone\bone^\top\in\real^{n\times n}$, where all entries are equal to 1, and $\bone\in\real^n$ is the vector of all ones. Find $n$ linearly independent eigenvectors of $\bA$, and determine the corresponding eigenvalues. \textit{Hint: Consider $\bx_i=\bone-n\be_i$ and $\bone$.} 

\item  What are the eigenvalues of the matrix $\bA=\scriptsize\begin{bmatrix}
5 & -1 & -1\\
-1 & 5 & -1\\
-1& -1 & 5
\end{bmatrix}$?

\item \label{prob:geneig1} \textbf{Generalized eigenproblem.} Many scientific packages address the generalized eigenproblem  $\bA\bx=\lambda\bB\bx$, where $\bB$ is nonsingular. If $\bA$ is symmetric and $\bB$ is PD with the Cholesky decomposition $\bB=\bR^\top\bR$, show that the eigenvalue $\lambda$ is a (standard) eigenvalue of $\bC=(\bR^{-1})^\top \bA\bR^{-1}$, corresponding to the eigenvector $\bR\bx$.

\item \label{prob:geneig2} \textbf{Generalized eigenproblem \citep{teukolsky1992numerical}.} Suppose $\bA\lambda^2 +\bB\lambda +\bC=\bzero$. Show that $\lambda$  can be solved by a standard eigenproblem. \textit{Hint: Let $\by=\lambda\bx$, and consider the matrix $\footnotesize\begin{bmatrix}
\bzero & \bI\\
-\bA^{-1}\bC & -\bA^{-1}\bB 
\end{bmatrix}$.}

\item \textbf{Matrix exponentials.} Given the definition of matrix exponentials in \eqref{equation:max_expo}, let  $\bA$ and $\bB$ commute, i.e., $\bA\bB=\bB\bA$. Show that $e^{\bA+\bB} = e^{\bA}\cdot e^{\bB}$.

\item \label{prob:mtexp2} \textbf{Matrix exponentials.} Let $f(t):\real\rightarrow \real^n$ be a function satisfying $f'(t)=\bA f(t)$ with $f(0)=\bx_0\in\real^n$ and $\bA\in\real^{n\times n}$. Show that the unique solution is $f(t)=e^{\bA t}\bx_0$.

\item \textbf{Matrix exponentials.} Consider the matrix exponential function in \eqref{equation:max_expo3}. Let $\bX^{-1}\bA\bX=\bJ=\diag(\bJ_1,\bJ_2,\ldots,\bJ_k)$ be the Jordan form of $\bA$, where $\bJ_i\in\real^{m_i\times m_i}$ contains the eigenvalue $\lambda_i$ along the diagonal. Show that 
$$
e^{\bJ_i t} =
e^{\lambda_i\cdot t}
\begin{bmatrix}
1 & t  & \frac{t^2}{2!} & \ldots & \frac{t^{m_i-1}}{(m_i-1)!} \\
0 &  1 & t & \ldots & \frac{t^{m_i-2}}{(m_i-2)!} \\
\vdots & \vdots & \ddots &\vdots  & \vdots \\
0 &  0 & 0 & \ddots & t \\
0 &  0 & 0 & \ldots & 1 \\
\end{bmatrix} .
$$  
Rewrite the matrix equation in Problem~\ref{prob:mtexp2} in the form  $g'(t)=\bJ g(t)$ and determine the explicit form of $g(t)$.
\textit{Hint: Decompose every Jordan block as a sum of a diagonal matrix and a nilpotent matrix.~\footnote{A matrix $\bA\in \real^{n\times n}$ is \textit{nilpotent} if there exists a $k$ such that $\bA^k = \bzero$.}}

\item We have presented  several important results regarding the determinant of a matrix in Lemma~\ref{lemma:determinant-intermezzo}. Given $\bA\in\real^{n\times n}$, show that
\begin{itemize}
\item $\det(c\bA)=c^n\det(\bA)$; (\textit{Hint: Use induction.})
\item $\det(\bA^{-1}) = 1/\det(\bA)$;
\item $\det(\bA^m) = \det(\bA)^m$;
\item $\det(\bI+\bu\bv^\top) = 1+\bu^\top\bv$.
\end{itemize}

\item Given $\bA\in\real^{n\times n}$, for $n=2$, show that 
\begin{itemize}
\item $\det(\bI+\bA) = 1+\det(\bA) + \tr(\bA)$.
\end{itemize}
For $n=3$, show that
\begin{itemize}
\item $\det(\bI+\bA) = 1+ \det(\bA) +\tr(\bA) + \frac{1}{2} \tr(\bA)^2 - \frac{1}{2}\tr(\bA^2)$.
\end{itemize}

\item Given $\bA\in\real^{n\times n}$, and let $\bB$ be the matrix obtained by interchanging two rows of $\bA$. Prove that $\det(\bB) = -\det(\bA)$. \textit{Hint: Use induction.}

\item Given $\bA\in\real^{n\times n}$, and let $\bB$ be the matrix obtained by multiplying a row of $\bA$ by a nonnegative scalar $\gamma$. Prove that $\det(\bB) = \gamma\det(\bA)$.  

\item Prove Theorem~\ref{theorem:schur-decomposition_complex} rigorously.

\item \label{problem:poly_cay} Given any polynomial $p(\lambda)$, show that $p(\bA)\bX=\bX p(\bB)$ if $\bA\bX=\bX\bB$.

\item Given any polynomial $p(\lambda)$, show that $\bA\bB p(\bA\bB)=\bA p(\bB\bA)\bB$ if $\bA\in\real^{m\times n}$ and $\bB\in\real^{n\times m}$.

\item \label{prob:diag_uppt} \textbf{Diagonalization of upper triangular matrices.} Let  $\bU\in\real^{n\times n}$ be an upper triangular matrix whose $(i,j)$-th entry is denoted by $u_{ij}$, and let $\bD_t=\diag(t, t^2, \ldots,t^n)$ be a diagonal matrix. Show that the similarity transformation on $\bU$ takes the following form:
$$
\bD_t\bU\bD_t^{-1}
=
\begin{bmatrix}
	u_{11} & t^{-1} u_{12} & t^{-2} u_{13} & \ldots & t^{-n+1}u_{1n}\\
	0   &  u_{22} & t^{-1} u_{23} & \ldots & t^{-n+2}u_{2n}\\
	0   &  0  &  u_{33} & \ldots & t^{-n+3}u_{3n}\\
	0   &  \vdots & \vdots & \ddots & \vdots\\
	0   &  0 & 0 & \ldots & u_{nn}\\
\end{bmatrix}.
$$
Thus, when $t$ is sufficiently large, the off-diagonal values can be made arbitrarily  small.

\item Show that $\bA\in\real^{n\times n}$ is nilpotent  if and only if $\trace(\bA^k)=0$ for all $k\in\{1,2,\ldots,n\}$. \textit{Hint: Use the $m$-th power eigenvalues, Remark~\ref{remark:mth_schur}.}

\item \textbf{Rank-one perturbation of Schur decomposition.} Suppose $\bA\in\complex^{n\times n}$ has eigenvalues $\lambda_1, \lambda_2,\ldots,\lambda_n\in\complex$, where $\bA\bx=\lambda_1\bx$. 
Show that, for any vector $\bv\in\complex^n$, the eigenvalues of $\bA+\bx\bv^*$ are $\lambda_1+\bv^*\bx, \lambda_2, \lambda_3, \ldots,\lambda_n$. Show that the Schur decomposition of $\bA+\bx\bv^*$ can be obtained efficiently if the Schur decomposition of $\bA$ is known.
\end{problemset}

\chapter{Spectral Decomposition (Theorem)}\label{chapter:spectral-decomposition}

\section{Spectral Decomposition (Theorem)}\label{section:spectral-decomposition}

The \textit{spectral theorem}, also known as the \textit{spectral decomposition} for symmetric matrices, states that symmetric matrices have   real eigenvalues and  can be diagonalized using  a (real) orthonormal basis \footnote{Note that  for \textit{Hermitian matrices}, the spectral decomposition states that they also have real eigenvalues and  can be diagonalized using a complex orthonormal basis.}. In the following theorem, we will present the main result and defer detailed discussions.

\index{Decomposition: Spectral}
\index{Orthogonal}
\begin{theoremHigh}[Spectral Decomposition]\label{theorem:spectral_theorem}
A real matrix $\bA \in \real^{n\times n}$ is symmetric if and only if there exists an orthogonal matrix $\bQ$ and a diagonal matrix $\bLambda$ such that
\begin{equation*}
\bA = \bQ \bLambda \bQ^\top,
\end{equation*}
where the columns of $\bQ = [\bq_1, \bq_2, \ldots, \bq_n]$ are eigenvectors of $\bA$ and are mutually orthonormal, and the entries of $\bLambda=\diag(\lambda_1, \lambda_2, \ldots, \lambda_n)$ are the corresponding eigenvalues of $\bA$, which are real.  Specifically, the following properties hold:
\begin{enumerate}
\item A symmetric matrix has only \textbf{real eigenvalues}.

\item The eigenvectors are orthogonal and can be chosen to be \textbf{orthonormal} by normalization.

\item The rank of $\bA$ is equal to the number of nonzero eigenvalues.

\item If the eigenvalues are distinct, the eigenvectors are linearly independent.
\end{enumerate}
\end{theoremHigh}

In the eigenvalue decomposition (Theorem~\ref{theorem:eigenvalue-decomposition}), we require the matrix $\bA$ to be square and its eigenvectors to be linearly independent. In contrast, the spectral theorem applies to any symmetric matrix, and the eigenvectors are chosen to be orthonormal.
On the other hand, analogous to  eigenvalue decomposition, the spectral decomposition enables efficient computation of the $m$-th power of a matrix $\bA$:
	If $\bA$ admits a spectral decomposition $\bA=\bQ\bLambda\bQ^\top$, then the $m$-th power of $\bA$ can be computed as  $\bA^m = \bQ\bLambda^m\bQ^\top$.

In Proposition~\ref{proposition:eigenvalue-similar-matrices}, we proved that similar matrices have the same eigenvalues. From the spectral decomposition, we observe that $\bA$ and $\bLambda$ are similar matrices and thus share the same eigenvalues. For any diagonal matrix, the eigenvalues are simply the entries on the main diagonal. \footnote{In fact, in the previous section, we showed that the diagonal entries of triangular matrices are their eigenvalues.} To verify this, observe that
$$
\bLambda \be_i = \lambda_i \be_i,
$$
where $\be_i$ is the $i$-th standard basis vector. Therefore, the diagonal matrix $\bLambda$ contains the eigenvalues of $\bA$.

\section{Existence of  Spectral Decomposition}\label{section:existence-of-spectral}

We will prove the theorem in several steps. 
We begin by showing that all eigenvalues of a symmetric matrix are real.
\begin{tcolorbox}[title={Symmetric Matrix Property 1 of 4},colback=\mdframecolorTheorem]
\begin{lemma}[Real eigenvalues]\label{lemma:real-eigenvalues-spectral}
All eigenvalues of a symmetric matrix are real.
\end{lemma}
\end{tcolorbox}
\begin{proof}[of Lemma~\ref{lemma:real-eigenvalues-spectral}]
Suppose $\lambda$ is a complex eigenvalue of a symmetric matrix $\bA$, expressed as  $\lambda=a+ib$, where $a$ and $b$ are real numbers. Its complex conjugate is $\bar{\lambda}=a-ib$. Similarly, for the corresponding complex eigenvector $\bx = \bc+i\bd$, its complex conjugate is $\bar{\bx}=\bc-i\bd$, where $\bc$ and $\bd$ are real vectors. The following properties hold:
$$
\bA \bx = \lambda \bx\qquad   \underrightarrow{\text{ leads to }}\qquad  \bA \bar{\bx} = \bar{\lambda} \bar{\bx}\qquad   \underrightarrow{\text{ transpose to }}\qquad  \bar{\bx}^\top \bA =\bar{\lambda} \bar{\bx}^\top.
$$
Taking the dot product of the first equation with $\bar{\bx}$ and the last equation with $\bx$, we get:
$$
\bar{\bx}^\top \bA \bx = \lambda \bar{\bx}^\top \bx \qquad \text{and } \qquad \bar{\bx}^\top \bA \bx = \bar{\lambda}\bar{\bx}^\top \bx.
$$
Equating these, we find $\lambda\bar{\bx}^\top \bx = \bar{\lambda} \bar{\bx}^\top\bx$. Since $\bar{\bx}^\top\bx = (\bc-i\bd)^\top(\bc+i\bd) = \bc^\top\bc+\bd^\top\bd$, which is a real number. 
Therefore, the imaginary part of $\lambda$ must be zero, implying $\lambda$ is real.
\end{proof}

\begin{tcolorbox}[title={Symmetric Matrix Property 2 of 4},colback=\mdframecolorTheorem]
\begin{lemma}[Orthogonal eigenvectors]\label{lemma:orthogonal-eigenvectors}
The eigenvectors  corresponding to distinct eigenvalues of any symmetric matrix are orthogonal. Therefore, these eigenvectors can be normalized to form an orthonormal set because $\bA\bx = \lambda \bx \underrightarrow{\text{ leads to } } \bA\frac{\bx}{\norm{\bx}} = \lambda \frac{\bx}{\norm{\bx}}$, which preserves the eigenvalue $\lambda$.
\end{lemma}
\end{tcolorbox}
\begin{proof}[of Lemma~\ref{lemma:orthogonal-eigenvectors}]
Let eigenvalues $\lambda_1$ and $\lambda_2$ correspond to eigenvectors $\bx_1$ and $\bx_2$, respectively, such that $\bA\bx_1=\lambda \bx_1$ and $\bA\bx_2 = \lambda_2\bx_2$. The following equalities hold:
$$
\bA\bx_1=\lambda_1 \bx_1 \leadto \bx_1^\top \bA =\lambda_1 \bx_1^\top \leadto \bx_1^\top \bA \bx_2 =\lambda_1 \bx_1^\top\bx_2,
$$
and 
$$
\bA\bx_2 = \lambda_2\bx_2 \leadto  \bx_1^\top\bA\bx_2 = \lambda_2\bx_1^\top\bx_2.
$$
Equating these, we get $\lambda_1 \bx_1^\top\bx_2=\lambda_2\bx_1^\top\bx_2$. Since eigenvalues $\lambda_1\neq \lambda_2$, the eigenvectors are orthogonal.
\end{proof}

In Lemma~\ref{lemma:orthogonal-eigenvectors}, we proved that  eigenvectors corresponding to distinct eigenvalues of a symmetric matrix are orthogonal. More generally, we now prove a stronger result: eigenvectors corresponding to distinct eigenvalues of any square matrix are linearly independent.
\begin{theorem}[Independent eigenvector theorem]\label{theorem:independent-eigenvector-theorem}
If a matrix $\bA\in \real^{n\times n}$ has $k$ distinct eigenvalues, then any set of $k$ corresponding (nonzero) eigenvectors are linearly independent.
\end{theorem}
\begin{proof}[of Theorem~\ref{theorem:independent-eigenvector-theorem}]
We  prove the theorem by induction. First, we show that any two eigenvectors corresponding to distinct eigenvalues are linearly independent. Let eigenvectors $\bv_1$ and $\bv_2$ correspond to distinct eigenvalues $\lambda_1$ and $\lambda_2$, respectively. 
Assume, for contradiction, that $\bv_1$ and $\bv_2$ are linearly dependent. Then there exists a nonzero vector $\bx=[x_1,x_2] \neq \bzero $ such that
\begin{equation}\label{equation:independent-eigenvector-eq1}
x_1\bv_1+x_2\bv_2=\bzero.
\end{equation}
Multiplying  \eqref{equation:independent-eigenvector-eq1} on the left by $\bA$ gives:
\begin{equation}\label{equation:independent-eigenvector-eq2}
x_1 \lambda_1\bv_1 + x_2\lambda_2\bv_2 = \bzero.
\end{equation}
Similarly, multiplying \eqref{equation:independent-eigenvector-eq1} by $\lambda_2$ yields:
\begin{equation}\label{equation:independent-eigenvector-eq3}
x_1\lambda_2\bv_1 + x_2\lambda_2\bv_2 = \bzero.
\end{equation}
Subtracting these two equations yields:
$$
x_1(\lambda_2-\lambda_1)\bv_1 = \bzero.
$$
Since $\lambda_2\neq \lambda_1$ and $\bv_1\neq \bzero$, we must have $x_1=0$. Since $\bv_2\neq \bzero$, substituting $x_1=0$ back into  \eqref{equation:independent-eigenvector-eq1}  shows that $x_2=0$, leading to a contradiction. Thus, $\bv_1$ and $\bv_2$ are linearly independent.

Suppose that any set of $j<k$ eigenvectors corresponding to distinct eigenvalues is linearly independent. 
We aim to show that any $j+1$ eigenvectors are also linearly independent. 
Let $\bv_1, \bv_2, \ldots, \bv_j$ be linearly independent eigenvectors corresponding to distinct eigenvalues $\lambda_1,\lambda_2,\ldots,\lambda_j$.
Assume, for contradiction, that an additional eigenvector $\bv_{j+1}$, corresponding to a different eigenvalue $\lambda_{j+1}$, is linearly dependent on $\bv_1, \bv_2, \ldots, \bv_j$. Then there exist scalars $x_1,x_2,\ldots, x_{j}$, not all zero, such that:
\begin{equation}\label{equation:independent-eigenvector-zero}
\bv_{j+1}=	x_1\bv_1+x_2\bv_2+\ldots+x_j\bv_j .
\end{equation}
Multiplying \eqref{equation:independent-eigenvector-zero} on the left by $\bA$ gives: 
\begin{equation}\label{equation:independent-eigenvector-zero2}
\lambda_{j+1} \bv_{j+1} = x_1\lambda_1\bv_1+x_2\lambda_2\bv_2+\ldots+x_j \lambda_j\bv_j .
\end{equation}
Similarly, multiplying \eqref{equation:independent-eigenvector-zero}  by $\lambda_{j+1}$ yields:
\begin{equation}\label{equation:independent-eigenvector-zero3}
	\lambda_{j+1} \bv_{j+1} = x_1\lambda_{j+1}\bv_1+x_2\lambda_{j+1}\bv_2+\ldots+x_j \lambda_{j+1}\bv_j .
\end{equation}
Subtracting the two equations gives:
$$
x_1(\lambda_{j+1}-\lambda_1)\bv_1+x_2(\lambda_{j+1}-\lambda_2)\bv_2+\ldots+x_j (\lambda_{j+1}-\lambda_j)\bv_j = \bzero. 
$$
From the assumption, $\lambda_{j+1} \neq \lambda_i$ for all $i\in \{1,2,\ldots,j\}$, and $\bv_i\neq \bzero$ for all $i\in \{1,2,\ldots,j\}$. We must have $x_1=x_2=\ldots=x_j=0$, which leads to a contradiction. Thus, the eigenvectors $\bv_1,\bv_2,\ldots,\bv_j,\bv_{j+1}$ are linearly independent. By induction, any $k$ eigenvectors corresponding to $k$ distinct eigenvalues are linearly independent.
\end{proof}

An immediate consequence of the above theorem is the following:
\begin{corollary}[Independent eigenvector theorem, CNT.]\label{theorem:independent-eigenvector-theorem-basis}
If a matrix $\bA\in \real^{n\times n}$ has $n$ distinct eigenvalues, then any set of $n$ corresponding eigenvectors form a basis for $\real^n$.
\end{corollary}

\begin{tcolorbox}[title={Symmetric Matrix Property 3 of 4},colback=\mdframecolorTheorem]
\begin{lemma}[Orthonormal eigenvectors for duplicate eigenvalue]\label{lemma:eigen-multiplicity}
Let $\bA\in\real^{n\times n}$ be symmetric. 
If $\bA$ has a repeated eigenvalue $\lambda_i$ with multiplicity \footnote{Multiplicity is rigorously defined in Definition~\ref{definition:eigen_multipli}.} $k\geq 2$, then there exist $k$ orthonormal eigenvectors corresponding to $\lambda_i$.
\end{lemma}
\end{tcolorbox}
\begin{proof}[of Lemma~\ref{lemma:eigen-multiplicity}]
First, note that there exists at least one eigenvector $\bx_{i1}$ corresponding to $\lambda_i$. 
For this eigenvector $\bx_{i1}$, we can always find  $n-1$ additional orthonormal vectors $\by_2, \by_3, \ldots, \by_n$, such that the set $\{\bx_{i1}, \by_2, \by_3, \ldots, \by_n\}$ forms an orthonormal basis for $\real^n$. 
Define the matrices $\bY_1$ and $\bP_1$ as follows:
$$
\bY_1=[\by_2, \by_3, \ldots, \by_n] \qquad \text{and} \qquad \bP_1=[\bx_{i1}, \bY_1].
$$
Since $\bA$ is symmetric, we then have
$
\bP_1^\top\bA\bP_1 =\scriptsize \begin{bmatrix}
\lambda_i &\bzero \\
\bzero & \bY_1^\top \bA\bY_1
\end{bmatrix}.
$
Since $\bP_1$ is nonsingular and orthogonal, it follows that $\bA$ and $\bP_1^\top\bA\bP_1$ are similar matrices such that they share the same eigenvalues  (see Proposition~\ref{proposition:eigenvalue-similar-matrices}), whence we have
$$
\det(\bP_1^\top\bA\bP_1 - \lambda\bI_n) =
\footnote{By the fact that if matrix $\bM$ has a block formulation: $\bM=\scriptsize\begin{bmatrix}
\bA & \bB \\
\bC & \bD 
\end{bmatrix}$, then $\det(\bM) = \det(\bA)\det(\bD-\bC\bA^{-1}\bB)$.
}~
(\lambda_i - \lambda )\det(\bY_1^\top \bA\bY_1 - \lambda\bI_{n-1}).
$$
If $\lambda_i$ has multiplicity  $k\geq 2$,  then the term $(\lambda_i-\lambda)$ appears  $k$ times in the polynomial from the determinant $\det(\bP_1^\top\bA\bP_1 - \lambda\bI_n)$, i.e., the term appears $k-1$ times in the polynomial from $\det(\bY_1^\top \bA\bY_1 - \lambda\bI_{n-1})$. 
Thus, $\det(\bY_1^\top \bA\bY_1 - \lambda_i\bI_{n-1})=0$, and $\lambda_i$ is an eigenvalue of $\bY_1^\top \bA\bY_1$ with multiplicity $k-1$.

Define $\bB=\bY_1^\top \bA\bY_1$. Since $\det(\bB-\lambda_i\bI_{n-1})=0$, the null space of $\bB-\lambda_i\bI_{n-1}$ is nontrivial. Suppose $(\bB-\lambda_i\bI_{n-1})\bn = \bzero$, i.e., $\bB\bn=\lambda_i\bn$, where  $\bn$ is an eigenvector of $\bB$. 

From $
\bP_1^\top\bA\bP_1 = \scriptsize\begin{bmatrix}
\lambda_i &\bzero \\
\bzero & \bB
\end{bmatrix},
$
we have $
\bA\bP_1 
\scriptsize
\begin{bmatrix}
z \\
\bn 
\end{bmatrix} 
\normalsize
= 
\bP_1
\scriptsize
\begin{bmatrix}
\lambda_i &\bzero \\
\bzero & \bB
\end{bmatrix}
\scriptsize
\begin{bmatrix}
z \\
\bn 
\end{bmatrix}$, where $z$ is any scalar. From the left side of this equation, we have 
\begin{equation}\label{equation:spectral-pro4-right}
\begin{aligned}
\bA\bP_1 
\begin{bmatrix}
z \\
\bn 
\end{bmatrix} 
&=
\begin{bmatrix}
\lambda_i\bx_{i1}, \bA\bY_1
\end{bmatrix}
\begin{bmatrix}
z \\
\bn 
\end{bmatrix} 
=\lambda_iz\bx_{i1} + \bA\bY_1\bn.
\end{aligned}
\end{equation}
From the right side of the equation, we have 
\begin{equation}\label{equation:spectral-pro4-left}
\begin{aligned}
\bP_1
\begin{bmatrix}
\lambda_i &\bzero \\
\bzero & \bB
\end{bmatrix}
\begin{bmatrix}
z \\
\bn 
\end{bmatrix}
&=
\begin{bmatrix}
\bx_{i1} & \bY_1
\end{bmatrix}
\begin{bmatrix}
\lambda_i &\bzero \\
\bzero & \bB
\end{bmatrix}
\begin{bmatrix}
z \\
\bn 
\end{bmatrix}
=
\begin{bmatrix}
\lambda_i\bx_{i1} & \bY_1\bB 
\end{bmatrix}
\begin{bmatrix}
z \\
\bn 
\end{bmatrix}\\
&= \lambda_i z \bx_{i1} + \bY_1\bB \bn 
=\lambda_i z \bx_{i1} + \lambda_i \bY_1 \bn,\\
\end{aligned}
\end{equation}
where the last equality follows from $\bB \bn=\lambda_i\bn$.
Combining Equations~\eqref{equation:spectral-pro4-left} and \eqref{equation:spectral-pro4-right}, we obtain 
$$
\bA\bY_1\bn = \lambda_i\bY_1 \bn,
$$
which means $\bY_1\bn$ is an eigenvector of $\bA$ corresponding to the eigenvalue $\lambda_i$ (the same eigenvalue corresponding to $\bx_{i1}$). Since $\bY_1\bn$ is a linear combination of $\by_2, \by_3, \ldots, \by_n$, which are orthonormal to $\bx_{i1}$, it can be chosen to be orthonormal to  $\bx_{i1}$ by scaling $\bn$.

To summarize, if there exists an  eigenvector, $\bx_{i1}$, corresponding to the eigenvalue $\lambda_i$, with  multiplicity  $k\geq 2$, we can construct a second eigenvector by selecting a vector from the null space of $(\bB-\lambda_i\bI_{n-1})$, as outlined above. Assume that we have constructed this second eigenvector, $\bx_{i2}$, which is orthonormal to $\bx_{i1}$.  
With these eigenvectors, $\bx_{i1}$ and $\bx_{i2}$, we can always find  $n-2$ additional orthonormal vectors $\by_3, \by_4, \ldots, \by_n$, such  that the set $\{\bx_{i1},\bx_{i2}, \by_3, \by_4, \ldots, \by_n\}$ forms an orthonormal basis for $\real^n$. 
Arrange these vectors $\by_3, \by_4, \ldots, \by_n$ into matrix $\bY_2$ and $\{\bx_{i1},\bx_{i2},  \by_3, \by_4, \ldots, \by_n\}$ into matrix $\bP_2$:
$$
\bY_2=[\by_3, \by_4, \ldots, \by_n] \qquad \text{and} \qquad \bP_2=[\bx_{i1}, \bx_{i2},\bY_1].
$$
Since $\bA$ is symmetric, we then have
$$
\bP_2^\top\bA\bP_2 = 
\begin{bmatrix}
\lambda_i & 0 &\bzero \\
0& \lambda_i &\bzero \\
\bzero & \bzero & \bY_2^\top \bA\bY_2
\end{bmatrix}
=
\begin{bmatrix}
\lambda_i & 0 &\bzero \\
0& \lambda_i &\bzero \\
\bzero & \bzero & \bC
\end{bmatrix},
$$
where $\bC=\bY_2^\top \bA\bY_2$ such that $\det(\bP_2^\top\bA\bP_2 - \lambda\bI_n) = (\lambda_i-\lambda)^2 \det(\bC - \lambda\bI_{n-2})$. If the multiplicity of $\lambda_i$ is $k\geq 3$, then $\det(\bC - \lambda_i\bI_{n-2})=0$, implying the null space of $\bC - \lambda_i\bI_{n-2}$ is non-empty. From the null space of $\bC - \lambda_i\bI_{n-2}$, we can select a vector $\bn$ such that $\bC\bn = \lambda_i \bn$. Using this vector, we construct $\scriptsize\begin{bmatrix}
z_1 \\
z_2\\
\bn
\end{bmatrix}\in \real^n $, where $z_1$ and $z_2$ are any scalar values, such that 
$$
\bA\bP_2\begin{bmatrix}
z_1 \\
z_2\\
\bn
\end{bmatrix} = \bP_2 
\begin{bmatrix}
\lambda_i & 0 &\bzero \\
0& \lambda_i &\bzero \\
\bzero & \bzero & \bC
\end{bmatrix}
\begin{bmatrix}
z_1 \\
z_2\\
\bn
\end{bmatrix}.
$$
Similarly, from the left side of the above equation, we will get $\lambda_iz_1\bx_{i1} +\lambda_iz_2\bx_{i2}+\bA\bY_2\bn$. From the right side of the above equation, we will get $\lambda_iz_1\bx_{i1} +\lambda_i z_2\bx_{i2}+\lambda_i\bY_2\bn$. As a result, 
$$
\bA\bY_2\bn = \lambda_i\bY_2\bn,
$$
where $\bY_2\bn$ is an eigenvector of $\bA$, orthogonal to $\bx_{i1}$ and $\bx_{i2}$. This eigenvector can also be normalized to ensure orthonormality with the first two eigenvectors.

By iterating this process, we can construct $k$  orthonormal eigenvectors corresponding to the eigenvalue $\lambda_i$. 

Finally, the dimension of the null space of $\bP_1^\top\bA\bP_1 -\lambda_i\bI_n$ equals the multiplicity $k$. This confirms that if $\lambda_i$ has multiplicity  $k$, there cannot be more than $k$ orthonormal eigenvectors corresponding to $\lambda_i$. Otherwise, we would obtain more than $n$ orthogonal eigenvectors in $\real^n$, which leads to a contradiction.
\end{proof}

The existence of the spectral decomposition can be straightforwardly established using the lemmas presented above. 
Alternatively, we can prove its existence by applying the Schur decomposition.
\begin{proof}[{of Theorem~\ref{theorem:spectral_theorem}: Existence of spectral decomposition, alternative proof}]
According to Theorem~\ref{theorem:schur-decomposition}, a symmetric matrix $\bA=\bA^\top$ can be decomposed as  $\bQ\bU\bQ^\top = \bQ\bU^\top\bQ^\top$, where $\bQ$ is orthogonal and $\bU$ is upper triangular. 
This implies that $\bU$ must be a diagonal matrix whose diagonal entries are the eigenvalues of $\bA$. 
Consequently, the columns of $\bQ$ are the corresponding eigenvectors of $\bA$. 
Therefore, we conclude that every symmetric matrix can be orthogonally diagonalized, even when eigenvalues are repeated.
\end{proof}

For a more concise proof, \citet{lu2021numerical} presents an approach utilizing the Gram--Schmidt process combined with mathematical induction.

For a symmetric matrix $\bA^\top \bA$, the rank remains the same as that of $\bA$, a property we will utilize in proving the singular value decomposition in the next chapter.
However, in general, the rank of a product of two matrices does not exceed the rank of either matrix; see also Exercises~\ref{exercise:rk_ad} and \ref{exercise:rk_prod}.
\begin{lemma}[Rank of $\bA\bB$]\label{lemma:rankAB}
Given matrices $\bA\in \real^{m\times n}$ and $\bB\in \real^{n\times k}$, the rank of their product $\bA\bB\in \real^{m\times k}$ satisfies $\rank$($\bA\bB$)$\leq\min\{\rank(\bA), \rank(\bB)\}$.
\end{lemma}
\begin{proof}[of Lemma~\ref{lemma:rankAB}]
Considering the matrix product $\bA\bB$: 
\begin{itemize}
\item  Each row of $\bA\bB$ is a linear combination of the rows of $\bB$, implying that the row space of $\bA\bB$ is contained within that of $\bB$. Therefore, $\rank$($\bA\bB$)$\leq$$\rank$($\bB$).

\item Similarly, each column of $\bA\bB$ is a linear combination of the columns of $\bA$, so the column space of $\bA\bB$ is contained within that of $\bA$. Hence, $\rank$($\bA\bB$)$\leq$$\rank$($\bA$).
\end{itemize}
Combining these observations, we conclude that $\rank$($\bA\bB$)$\leq\min\{\rank(\bA), \rank(\bB)\}$.
\end{proof}

\begin{tcolorbox}[title={Symmetric Matrix Property 4 of 4},colback=\mdframecolorTheorem]
\begin{lemma}[Rank of symmetric matrices]\label{lemma:rank-of-symmetric}
If $\bA$ is an $n\times n$ real symmetric matrix, then $\rank(\bA)$ =
the total number of nonzero eigenvalues of $\bA$. 
Furthermore, the column space $\cspace(\bA)$ is the linear subspace spanned by the eigenvectors of $\bA$ corresponding to its nonzero eigenvalues.
\end{lemma}
\end{tcolorbox}
\begin{proof}[of Lemma~\ref{lemma:rank-of-symmetric}]
For any symmetric matrix $\bA$, it can be expressed in its spectral form as $\bA = \bQ \bLambda\bQ^\top$, where $\bQ$ is an orthogonal matrix and $\bLambda$ is a diagonal matrix containing the eigenvalues of $\bA$. 
Using Lemma~\ref{lemma:rankAB}, we proceed as follows:
\begin{itemize}
\item From $\bA = \bQ \bLambda\bQ^\top$, we have $\rank(\bA) \leq \rank(\bQ \bLambda) \leq \rank(\bLambda)$.
\item From $\bLambda = \bQ^\top\bA\bQ$, we have $\rank(\bLambda) \leq \rank(\bQ^\top\bA) \leq \rank(\bA)$.
\end{itemize}
This implies $\rank(\bA) = \rank(\bLambda)$, which is equal to the total number of nonzero eigenvalues of $\bA$.
\end{proof}

\index{Uniqueness}
\section{Uniqueness of Spectral Decomposition}\label{section:uniqueness-spectral-decomposition}
It's important to note that spectral decomposition of a matrix is generally not unique. This is primarily due to the presence of repeated  eigenvalues.
When two or more eigenvalues $\lambda_i$ and $\lambda_j$ (for $1\leq i,j\leq n$) are identical, swapping their corresponding eigenvectors in the orthogonal matrix $\bQ$ results in a different decomposition that is still mathematically valid and equivalent.

However, the \textit{eigenspaces} associated with each eigenvalue---specifically, the null spaces $\nspace(\bA - \lambda_i\bI)$ for each eigenvalue $\lambda_i$---remain fixed. This means that while the choice of eigenvectors within each eigenspace can vary, leading to different decompositions, the decomposition in terms of eigenspaces is unique. 
In other words, any orthonormal basis for these eigenspaces can be used without affecting the overall spectral decomposition.

\index{Characteristic polynomial}
\section{Other Forms, Connecting Eigenvalue Decomposition*}\label{section:otherform-spectral}
In this section, we examine various types of spectral decomposition under different conditions. To support this discussion, we begin with a formal definition of the characteristic polynomial of a square matrix.
\begin{definition}[Characteristic polynomial\index{Characteristic polynomial}]
For any square matrix $\bA \in \real^{n\times n}$, its \textit{characteristic polynomial} is defined as:
$$
\begin{aligned}
\det(\lambda\bI-\bA ) &=\lambda^n + \gamma_{n-1} \lambda^{n-1} + \ldots + \gamma_1 \lambda  + \gamma_0
=(\lambda-\lambda_1)^{k_1} (\lambda-\lambda_2)^{k_2} \ldots (\lambda-\lambda_m)^{k_m},
\end{aligned}
$$
where $\lambda_1, \lambda_2, \ldots, \lambda_m$ are the distinct roots of $\det( \lambda\bI-\bA)=0$, which are also the eigenvalues of $\bA$. The sum of the multiplicities satisfies $k_1+k_2+\ldots +k_m=n$, indicating that $\det(\lambda\bI-\bA)$ is a polynomial of degree $n$ for any matrix $\bA\in \real^{n\times n}$ (see proof of Lemma~\ref{lemma:eigen-multiplicity}). 
The equation $\det(\lambda\bI - \bA)=0$ is referred to as  the \textit{characteristic equation} of $\bA$.
\end{definition}

The characteristic polynomial is essential in defining two key concepts: algebraic multiplicity and geometric multiplicity.
\begin{definition}[Algebraic multiplicity and geometric multiplicity\index{Algebraic multiplicity}\index{Geometric multiplicity}]\label{definition:eigen_multipli}
Given the characteristic polynomial of  a matrix $\bA\in \real^{n\times n}$:
$$
\begin{aligned}
\det(\lambda\bI-\bA ) =(\lambda-\lambda_1)^{k_1} (\lambda-\lambda_2)^{k_2} \ldots (\lambda-\lambda_m)^{k_m},
\end{aligned}
$$
the integer $k_i$ is called the \textit{algebraic multiplicity} of the eigenvalue $\lambda_i$, i.e., it equals the multiplicity of the corresponding root in the characteristic polynomial.

The \textit{eigenspace associated with the eigenvalue $\lambda_i$} is  the null space of $(\bA - \lambda_i\bI)$, denoted by $\nspace(\bA - \lambda_i\bI)$.
And the dimension of the eigenspace associated with $\lambda_i$, $\nspace(\bA - \lambda_i\bI)$, is known as the \textit{geometric multiplicity} of $\lambda_i$.

For brevity, we denote the algebraic multiplicity of $\lambda_i$ by $alg(\lambda_i)$ and its geometric multiplicity by $geo(\lambda_i)$.
\end{definition}

\index{Multiplicity}
\begin{remark}[Geometric multiplicity]\label{remark:geometric-mul-meaning}
For a matrix $\bA$ and its eigenspace $\nspace(\bA-\lambda_i\bI)$  corresponding to an eigenvalue	$\lambda_i$,  the dimension of the eigenspace reflects the number of linearly independent eigenvectors of $\bA$ associated with $\lambda_i$. This means that while there are infinitely many eigenvectors associated with each eigenvalue $\lambda_i$, they form a subspace that can be described using a finite set of basis vectors. In other words, the geometric multiplicity indicates the maximum number of linearly independent eigenvectors available for $\lambda_i$.
\end{remark}

By definition, the sum of the algebraic multiplicities of all eigenvalues equals $n$, whereas the sum of the geometric multiplicities can be strictly smaller.

\begin{corollary}[Multiplicity in similar matrices\index{Similar matrices}]\label{corollary:multipli-similar-matrix}
Similar matrices share the same algebraic and geometric multiplicities for their eigenvalues.
\end{corollary}
\begin{proof}[of Corollary~\ref{corollary:multipli-similar-matrix}]
From Proposition~\ref{proposition:eigenvalue-similar-matrices}, we know that similar matrices have identical eigenvalues, which implies they also share the same algebraic multiplicities.

Consider two similar matrices $\bA$ and $\bB= \bP\bA\bP^{-1}$, where $\bP$ is nonsingular. Suppose the geometric multiplicity of an eigenvalue $\lambda$ of  $\bA$ is $k$. This means there exist $k$ linearly independent eigenvectors  $\bv_1, \bv_2, \ldots, \bv_k$ forming a basis for the eigenspace $\nspace(\bA-\lambda\bI)$ such that $\bA\bv_i = \lambda \bv_i$ for each $i\in \{1, 2, \ldots, k\}$. Then, $\bw_i = \bP\bv_i$'s are the eigenvectors of $\bB$ associated with  $\lambda$. 
Since $\bP$ is nonsingular, these $\bw_i$'s are  also linearly independent. Thus, the dimension of the eigenspace $\nspace(\bB-\lambda\bI)$ is at least $k$, implying $\dim(\nspace(\bA-\lambda\bI)) \leq \dim(\nspace(\bB-\lambda\bI)) $. 

Conversely, if we start with a set of $k$ linearly independent eigenvectors $\bw_1, \bw_2, \ldots, \bw_k$ for $\bB$ corresponding to $\lambda$, then the vectors $\bv_i = \bP^{-1}\bw_i$ for all $i \in \{1, 2, \ldots, k\}$ are  eigenvectors of $\bA$ associated with $\lambda$. This gives us $\dim(\nspace(\bB-\lambda\bI)) \leq \dim(\nspace(\bA-\lambda\bI)) $.

By combining both inequalities, we conclude that $\dim(\nspace(\bA-\lambda\bI)) = \dim(\nspace(\bB-\lambda\bI)) $, establishing the equality of geometric multiplicities for similar matrices.
\end{proof}

\begin{lemma}[Bounded geometric multiplicity]\label{lemma:bounded-geometri}
For any matrix $\bA\in \real^{n\times n}$ and its eigenvalue $\lambda_i$, the geometric multiplicity is bounded by the algebraic multiplicity:
$$
geo(\lambda_i) \leq alg(\lambda_i).
$$
\end{lemma}
\begin{proof}[of Lemma~\ref{lemma:bounded-geometri}]
Suppose $\bP_1 = [\bv_1, \bv_2, \ldots, \bv_k]$ contains a set of  linearly independent eigenvectors of $\bA$ associated with $\lambda_i$. That is, the $k$ vectors form a basis for the eigenspace $\nspace(\bA-\lambda_i\bI)$, and the geometric multiplicity associated with $\lambda_i$ is $k$. 
Extend $\bP_1$ to a full basis 
$
\bP = [\bP_1, \bP_2] = [\bv_1, \bv_2, \ldots, \bv_k, \bv_{k+1}, \ldots, \bv_n],
$
where $\bP$ is nonsingular. Then we have $\bA\bP = [\lambda_i\bP_1, \bA\bP_2]$.

Now construct a matrix $\bB = \scriptsize\begin{bmatrix}
\lambda_i \bI_k & \bC \\
\bzero  & \bD
\end{bmatrix}$, where $\bA\bP_2 = \bP_1\bC + \bP_2\bD$. 
Then, $\bP^{-1}\bA\bP = \bB$, and therefore $\bA$ and $\bB$ are similar matrices. 
Such matrices $\bC$ and $\bD$ always exist because the vectors $\bv_i$ are linearly independent vectors spanning the entire space $\real^n$, and any column of $\bA\bP_2$ belongs to the column space of $\bP=[\bP_1,\bP_2]$.
Therefore, 
$$
\begin{aligned}
\det(\bA-\lambda\bI) &= \det(\bP^{-1})\det(\bA-\lambda\bI)\det(\bP)  
= \det(\bP^{-1}(\bA-\lambda\bI)\bP) 
=  \det(\bB-\lambda\bI) \\
&= \det\left(\begin{bmatrix}
	(\lambda_i-\lambda) \bI_k  & \bC \\
	\bzero  & \bD - \lambda \bI
\end{bmatrix}\right)
= (\lambda_i-\lambda)^k \det(\bD-\lambda\bI).
\end{aligned}
$$ 
This shows that the algebraic multiplicity of $\lambda_i$ is at least $k$, which is the geometric multiplicity. Therefore,
$
geo(\lambda_i) \leq alg(\lambda_i).
$
And we complete the proof.
\end{proof}

Building on the proof of Lemma~\ref{lemma:eigen-multiplicity}, it becomes evident that for symmetric matrices, the algebraic and geometric multiplicities of all eigenvalues are equal. Such matrices are called \textit{simple matrices}.
\begin{definition}[Simple matrix]
A square matrix is called \textit{simple} if, for each of its eigenvalues, the algebraic multiplicity equals the geometric multiplicity.
\end{definition}

\begin{definition}[Diagonalizable]\label{definition:diagonalizable}
A square matrix $\bA$ is said to be \textit{diagonalizable} if there exists a nonsingular matrix $\bP$ and a diagonal matrix $\bD$ such that $\bA = \bP\bD\bP^{-1}$.
\end{definition}

Diagonal matrices have a particularly simple structure, which makes computations such as determinants and inverses more straightforward. 
The eigenvalue decomposition (Theorem~\ref{theorem:eigenvalue-decomposition}) and the spectral decomposition (Theorem~\ref{theorem:spectral_theorem}) are examples of diagonalization techniques applicable to specific classes of matrices.
\begin{lemma}[Simple matrices are diagonalizable]\label{lemma:simple-diagonalizable}
A matrix is simple if and only if it is diagonalizable.
\end{lemma}
\begin{proof}[of Lemma~\ref{lemma:simple-diagonalizable}]
Suppose that $\bA\in \real^{n\times n}$ is a simple matrix, meaning that the algebraic and geometric multiplicities for each eigenvalue are equal. 
For a specific eigenvalue $\lambda_i$, let $\{\bv_1^i, \bv_2^i, \ldots, \bv_{k_i}^i\}$ be a basis for the eigenspace $\nspace(\bA - \lambda_i\bI)$. 
In other words, $\{\bv_1^i, \bv_2^i, \ldots, \bv_{k_i}^i\}$ is a set of linearly independent eigenvectors of $\bA$ associated with $\lambda_i$, where ${k_i}$ is the algebraic or geometric multiplicity of $\lambda_i$: $alg(\lambda_i)=geo(\lambda_i)=k_i$. Suppose there are $m$ distinct eigenvalues. 
Since $k_1+k_2+\ldots +k_m = n$, the set of eigenvectors consists of the union of $n$ vectors. 
Consider a  linear combination of these eigenvectors:
\begin{equation}\label{equation:proof-simple-diagonalize}
	\bz = \sum_{j=1}^{k_1} x_j^1 \bv_j^1+ \sum_{j=1}^{k_2} x_j^2 \bv_j^2 + \ldots+ \sum_{j=1}^{k_m} x_j^m \bv_j^m = \bzero.
\end{equation}
Let $\bw^i = \sum_{j=1}^{k_i} x_j^i \bv_j^i$. Then $\bw^i$ is either an eigenvector associated with $\lambda_i$ or the zero vector. Therefore, $\bz = \sum_{i=1}^{m} \bw^i$ is a sum of either zero vectors or  eigenvectors associated with different eigenvalues of $\bA$. Since eigenvectors associated with different eigenvalues are linearly independent. We must have $\bw^i =\bzero$ for all $i\in \{1, 2, \ldots, m\}$. That is,
$$
\bw^i = \sum_{j=1}^{k_i} x_j^i \bv_j^i = \bzero, \qquad \text{for all $i\in \{1, 2, \ldots, m\}$}.
$$
Since we assume the eigenvectors $\bv_j^i$'s associated with $\lambda_i$ are linearly independent, we must have $x_j^i=0$ for all $i \in \{1,2,\ldots, m\}, j\in \{1,2,\ldots,k_i\}$. Thus, the $n$ vectors are linearly independent:
$$
\{\bv_1^1, \bv_2^1, \ldots, \bv_{k_i}^1\},\{\bv_1^2, \bv_2^2, \ldots, \bv_{k_i}^2\},\ldots,\{\bv_1^m, \bv_2^m, \ldots, \bv_{k_i}^m\}.
$$
According to the eigenvalue decomposition presented in Theorem~\ref{theorem:eigenvalue-decomposition},  $\bA$ is diagonalizable.

Conversely, suppose $\bA$ is diagonalizable. That is, there exists a nonsingular matrix $\bP$ and a diagonal matrix $\bD$ such that $\bA =\bP\bD\bP^{-1} $. 
Then $\bA$ and $\bD$ are similar matrices, and therefore they have the same eigenvalues (Proposition~\ref{proposition:eigenvalue-similar-matrices}), the same algebraic multiplicities, and the same geometric multiplicities (Corollary~\ref{corollary:multipli-similar-matrix}). It can be readily verified that a diagonal matrix has equal algebraic  and geometric multiplicities. Therefore, $\bA$ is a simple matrix.
\end{proof}

From Theorem~\ref{theorem:independent-eigenvector-theorem}, which states that any eigenvectors corresponding to different eigenvalues are linearly independent, and Remark~\ref{remark:geometric-mul-meaning}, which explains that the geometric multiplicity is the dimension of the corresponding eigenspace, 
we can conclude the following: if, for a matrix $\bA\in \real^{n\times n}$, the geometric multiplicity is equal to the algebraic multiplicity (for all eigenvalues), the eigenspaces can span the entire space $\real^n$. 
Hence, the above lemma equivalently claims that if the eigenspaces span the entire space $\real^n$, then $\bA$ can be diagonalized.

\begin{corollary}
A square matrix $\bA$ is considered simple if it has a complete set of linearly independent eigenvectors. Alternatively, any symmetric matrix $\bA$ also qualifies as a simple matrix by definition.
\end{corollary}
The proof of this corollary follows directly from the eigenvalue decomposition given in Theorem~\ref{theorem:eigenvalue-decomposition} and the spectral theorem presented in Theorem~\ref{theorem:spectral_theorem}.

We now present an alternative expression for the spectral decomposition:
\begin{theoremHigh}[Spectral decomposition: the second form]\label{theorem:spectral_theorem_secondForm}
For a \textbf{simple matrix} $\bA \in \real^{n\times n}$,  it can be expressed as a weighted sum of idempotent matrices:
$$
\bA = \sum_{i=1}^{n} \lambda_i \bA_i,
$$
where each $\lambda_i$, for  $i\in \{1,2,\ldots, n\}$, represents an eigenvalue of $\bA$ (potentially repeated). 
The idempotent matrices $\bA_i$ satisfy the following properties:
\begin{enumerate}
\item \textit{Idempotent.} $\bA_i^2 = \bA_i$ for all $i\in \{1,2,\ldots, n\}$;

\item \textit{Orthogonal.} $\bA_i\bA_j = \bzero$ for all $i \neq j$;

\item \textit{Additivity.} $\sum_{i=1}^{n} \bA_i = \bI_n$;

\item \textit{Rank-Additivity.} $\rank(\bA_1) + \rank(\bA_2) + \ldots + \rank(\bA_n) = n$.
\end{enumerate}

\end{theoremHigh}

\begin{proof}[of Theorem~\ref{theorem:spectral_theorem_secondForm}]
Given that $\bA$ is a simple matrix, according to Lemma~\ref{lemma:simple-diagonalizable}, there exists a nonsingular matrix $\bP$ and a diagonal matrix $\bLambda$ such that $\bA=\bP\bLambda\bP^{-1}$, where $\bLambda=\diag(\lambda_1, \lambda_2, \ldots, \lambda_n)$,  $\lambda_i$'s are eigenvalues of $\bA$, and the columns of $\bP$ consist of the corresponding eigenvectors. Let 
$$
\bP = \begin{bmatrix}
\bv_1 & \bv_2&\ldots & \bv_n
\end{bmatrix}
\qquad
\text{and }
\qquad
\bP^{-1} = 
\begin{bmatrix}
\bw_1^\top;
\bw_2^\top ;
\ldots ;
\bw_n^\top
\end{bmatrix}
$$
denote the column and row partitions of $\bP$ and $\bP^{-1}$, respectively. Then, we can rewrite $\bA$ as
$$
\bA= \bP\bLambda\bP^{-1} = 
\begin{bmatrix}
	\bv_1 & \bv_2&\ldots & \bv_n
\end{bmatrix}
\bLambda
\begin{bmatrix}
	\bw_1^\top \\
	\bw_2^\top \\
	\vdots \\
	\bw_n^\top
\end{bmatrix}=
\sum_{i=1}^{n}\lambda_i \bv_i\bw_i^\top.
$$
By defining $\bA_i = \bv_i\bw_i^\top$, we obtain $\bA = \sum_{i=1}^{n} \lambda_i \bA_i$.
It follows from $\bP^{-1}\bP = \bI$  that 
$$ 
\left\{
\begin{aligned}
	&\bw_i^\top\bv_j = 1 ,& \mathrm{\,\,if\,\,} i = j;  \\
	&\bw_i^\top\bv_j = 0 ,& \mathrm{\,\,if\,\,} i \neq j. 
\end{aligned}
\right.
$$
Thus, 
$$ 
\bA_i\bA_j =\bv_i\bw_i^\top\bv_j\bw_j^\top = \left\{
\begin{aligned}
	&\bv_i\bw_i^\top = \bA_i ,& \mathrm{\,\,if\,\,} i = j;  \\
	& \bzero ,& \mathrm{\,\,if\,\,} i \neq j. 
\end{aligned}
\right.
$$
This confirms both the idempotency and orthogonality of the matrices $\bA_i$. Moreover, we have  $\sum_{i=1}^{n}\bA_i = \bP\bP^{-1}=\bI$, which verifies their additivity. 
Finally, the rank-additivity property holds trivially because  $\rank(\bA_i)=1$ for all $i\in \{1,2,\ldots, n\}$.
\end{proof}
This form of the decomposition is closely related to \textit{Cochran's theorem} and is widely used in the distribution theory of linear models \citep{lu2021numerical,lu2021rigorous}.
\index{Cochran's theorem}

Going further, suppose we have $k$ distinct eigenvalues.
Then we have the following result.
\begin{theoremHigh}[Spectral decomposition: the third form]\label{Corollary:spectral_theorem_3Form}
For  a \textbf{simple matrix} $\bA \in \real^{n\times n}$ \textbf{with $k$ distinct eigenvalues}, it can be expressed as a weighted sum of a set of idempotent matrices:
$$
\bA = \sum_{i=1}^{\textcolor{mylightbluetext}{k}} \lambda_i \bA_i,
$$
where each $\lambda_i$, for  $i\in \{1,2,\ldots, \textcolor{mylightbluetext}{k}\}$, represents one of the distinct eigenvalues of $\bA$. 
The idempotent matrices $\bA_i$ satisfy the following properties:
\begin{enumerate}
\item \textit{Idempotent.} $\bA_i^2 = \bA_i$ for all $i\in \{1,2,\ldots, \textcolor{mylightbluetext}{k}\}$;

\item \textit{Orthogonal.} $\bA_i\bA_j = \bzero$ for all $i \neq j$;

\item \textit{Additivity.} $\sum_{i=1}^{\textcolor{mylightbluetext}{k}} \bA_i = \bI_n$;

\item \textit{Rank-Additivity.} $\rank(\bA_1) + \rank(\bA_2) + \ldots + \rank(\bA_{\textcolor{mylightbluetext}{k}}) = n$.
\end{enumerate}

\end{theoremHigh}
\begin{proof}[of Theorem~\ref{Corollary:spectral_theorem_3Form}]
Building on Theorem~\ref{theorem:spectral_theorem_secondForm}, we can express $\bA$ as $\bA =\sum_{j=1}^{n} \beta_j \bB_j$, where $\beta_j's$ are the eigenvalues and $\bB_j's$ are the corresponding idempotent matrices from the second form of the spectral decomposition. 
Assume without loss of generality that the eigenvalues are ordered such that $\beta_1 \leq \beta_2 \leq \ldots \leq\beta_n$, allowing for duplicates. 
Let $\{\lambda_1, \lambda_2, \ldots, \lambda_k\}$ denote the set of $k$ distinct eigenvalues, and let $\bA_i$ represent the sum of the $\bB_j$ matrices associated with $\lambda_i$. 
Suppose the multiplicity of $\lambda_i$ is $m_i$, and the set of $\bB_j$ matrices associated with $\lambda_i$ can be denoted by $\{\bB_{1}^i, \bB_{2}^i, \ldots, \bB_{m_i}^i\}$. Thus, $\bA_i$ can be defined as  $\bA_i = \sum_{j=1}^{m_i} \bB_{j}^i$. Consequently, we have  $\bA = \sum_{i=1}^{k} \lambda_i \bA_i$.

\paragraph{Idempotency.} $\bA_i^2 = (\bB_1^i + \bB_2^i+\ldots \bB_{m_i}^i)(\bB_1^i + \bB_2^i+\ldots \bB_{m_i}^i)= \bB_1^i + \bB_2^i+\ldots \bB_{m_i}^i = \bA_i$  due to the idempotency and orthogonality of the $\bB_j^i$ matrices.

\paragraph{Ortogonality.} $\bA_i\bA_j = (\bB_1^i + \bB_2^i+\ldots \bB_{m_i}^i)(\bB_1^j + \bB_2^j+\ldots \bB_{m_j}^j)=\bzero$  due to the orthogonality of the $\bB_j^i$ matrices.

\paragraph{Additivity.} It is evident that  $\sum_{i=1}^{k} \bA_i = \bI_n$.

\paragraph{Rank-Additivity.} $\rank(\bA_i ) = \rank(\sum_{j=1}^{m_i} \bB_{j}^i) = m_i$ such that $\rank(\bA_1) + \rank(\bA_2) + \ldots + \rank(\bA_{k}) = m_1+m_2+\ldots+m_k=n$.
\end{proof}

The reverse implication of the above theorem also holds true.
\begin{theoremHigh}[Spectral decomposition: backward implication]\label{Corollary:spectral_theorem_4Form}
Let $\bA \in \real^{n\times n}$ be a matrix with $k$ distinct eigenvalues. If $\bA$ can be decomposed as a linear combination of a set of idempotent matrices
$$
\bA = \sum_{i=1}^{k} \lambda_i \bA_i,
$$
where each  $\lambda_i$, for  $i\in \{1,2,\ldots, k\}$, represents one of  the distinct eigenvalues of $\bA$, and 
the matrices $\bA_i$  satisfy the four conditions outlined in Theorem~\ref{Corollary:spectral_theorem_3Form}, then $\bA$ is a simple matrix.
\end{theoremHigh}
\begin{proof}[of Theorem~\ref{Corollary:spectral_theorem_4Form}]
Assume  that  $\rank(\bA_i) = r_i$ for all $i \in \{1,2,\ldots, k\}$. By the ULV decomposition given in Theorem~\ref{theorem:ulv-decomposition}, each $\bA_i$ can be decomposed as 
$
\bA_i = \bU_i 
\scriptsize
\begin{bmatrix}
\bL_i & \bzero \\
\bzero & \bzero 
\end{bmatrix}\bV_i,
$
where $\bL_i \in \real^{r_i \times r_i}$ is lower triangular, and $\bU_i \in \real^{n \times n}$  and $\bV_i\in \real^{n \times n}$ are orthogonal. Define 
$$
\bX_i = 
\bU_i \begin{bmatrix}
	\bL_i  \\
	\bzero  
\end{bmatrix}
\qquad 
\text{and}
\qquad 
\bV_i = 
\begin{bmatrix}
\bY_i \\
\bZ_i
\end{bmatrix},
$$
where $\bX_i$ is of size $\real^{n\times r_i}$, and $\bY_i \in \real^{r_i \times n}$ consists of the  first $r_i$ rows of $\bV_i$. Consequently, we have
$
\bA_i = \bX_i \bY_i.
$
This can be seen as a \textit{reduced} ULV decomposition of $\bA_i$. 
Concatenating  the $\bX_i$'s and $\bY_i$'s into matrices $\bX$ and $\bY$:
$$
\bX = [\bX_1, \bX_2, \ldots, \bX_k],
\qquad\text{and}\qquad
\bY = 
\begin{bmatrix}
\bY_1;
\bY_2;
\ldots;
\bY_k
\end{bmatrix},
$$
where $\bX\in \real^{n\times n}$ and $\bY\in \real^{n\times n}$ (by rank-additivity). Using block matrix multiplication and leveraging the additivity property of the $\bA_i$'s, we have 
$
\bX\bY = \sum_{i=1}^{k} \bX_i\bY_i = \sum_{i=1}^{k} \bA_i = \bI.
$
Therefore, $\bY$ is the inverse of $\bX$, and we also have
$$
\bY\bX = 
\begin{bmatrix}
	\bY_1\\
	\bY_2\\
	\vdots \\
	\bY_k
\end{bmatrix}
[\bX_1, \bX_2, \ldots, \bX_k]
=
\begin{bmatrix}
\bY_1\bX_1 & \bY_1\bX_2 & \ldots & \bY_1\bX_k\\
\bY_2\bX_1 & \bY_2\bX_2 & \ldots & \bY_2\bX_k\\
\vdots & \vdots & \ddots & \vdots\\
\bY_k\bX_1 & \bY_k\bX_2 & \ldots & \bY_k\bX_k\\
\end{bmatrix}
=\bI,
$$
such that 
$$ 
\bY_i\bX_j = \left\{
\begin{aligned}
	&\bI_{r_i}  ,& \mathrm{\,\,if\,\,} i = j;  \\
	& \bzero ,& \mathrm{\,\,if\,\,} i \neq j. 
\end{aligned}
\right.
$$
This implies 
$$ 
\bA_i\bX_j = \left\{
\begin{aligned}
	&\bX_i  ,& \mathrm{\,\,if\,\,} i = j;  \\
	& \bzero ,& \mathrm{\,\,if\,\,} i \neq j, 
\end{aligned}
\right.
\qquad 
\text{and}
\qquad
\bA \bX_i = \lambda_i\bX_i.
$$
Finally, we conclude that
$$
\begin{aligned}
\bA\bX &= \bA[\bX_1, \bX_2, \ldots, \bX_k] = [\lambda_1\bX_1, \lambda_2\bX_2, \ldots, \lambda_k\bX_k] = \bX\bLambda,
\end{aligned}
$$
where 
$\bLambda = \diag(\lambda_1 \bI_{r_1}, \lambda_2 \bI_{r_2}, \ldots,\lambda_k \bI_{r_k})
$
is a diagonal matrix. This implies $\bA$ can be diagonalized, and by Lemma~\ref{lemma:simple-diagonalizable}, $\bA$ is indeed a simple matrix.
\end{proof}

Combining Theorem~\ref{Corollary:spectral_theorem_3Form} and Theorem~\ref{Corollary:spectral_theorem_4Form}, we can claim that a matrix $\bA \in \real^{n\times n}$ is a simple matrix with $k$ distinct eigenvalues if and only if it can be decomposed as a sum of a set of idempotent matrices
$$
\bA = \sum_{i=1}^{k} \lambda_i \bA_i,
$$
where each $\lambda_i$, for  $i\in \{1,2,\ldots, k\}$, represents one of the distinct eigenvalues of $\bA$, and the matrices $\bA_i$  satisfy the four conditions outlined in Theorem~\ref{Corollary:spectral_theorem_3Form}.

\section{Skew-Symmetric Matrix and its Properties*}
We presented the spectral decomposition of symmetric matrices. Another significant class of matrices related to symmetry is known as \textit{skew-symmetric matrices}.
\begin{definition}[Skew-symmetric matrix\index{Skew-symmetric matrix}]
A matrix $\bA\in \real^{n\times n}$ is called a \textit{skew-symmetric matrix} if it satisfies the condition 
$\bA^\top = -\bA$.
Under this definition,  the diagonal entries $a_{ii}$ for all $i \in \{1,2,\ldots, n\}$ must satisfy the equation $a_{ii} = -a_{ii}$, which implies that all diagonal entries are  zero.
\end{definition}

Previously, in Lemma~\ref{lemma:real-eigenvalues-spectral}, we established that the eigenvalues of symmetric matrices are real. Similarly, it can be shown that all eigenvalues of skew-symmetric matrices are either purely imaginary or zero.
\begin{lemma}[Imaginary eigenvalues]\label{lemma:real-eigenvalues-spectral-skew}
The eigenvalues of any skew-symmetric matrix are either purely imaginary or zero. 
\end{lemma}
\begin{proof}[of Lemma~\ref{lemma:real-eigenvalues-spectral-skew}]
Suppose the eigenvalue $\lambda$ of the skew-symmetric matrix $\bA$ is a complex number $\lambda=a+ib$, where $a$ and $b$ are real numbers. Its complex conjugate is $\bar{\lambda}=a-ib$. 
Similarly, for the corresponding complex eigenvector $\bx = \bc+i\bd$, its complex conjugate is $\bar{\bx}=\bc-i\bd$, where $\bc$ and $\bd$ are real vectors. The eigenvalue equation and its conjugate can be written as:
$$
\bA \bx = \lambda \bx\qquad   \underrightarrow{\text{ leads to }}\qquad  \bA \bar{\bx} = \bar{\lambda} \bar{\bx}\qquad   \underrightarrow{\text{ transpose to }}\qquad  \bar{\bx}^\top \bA^\top =\bar{\lambda} \bar{\bx}^\top.
$$
Taking the dot product of the first equation with $\bar{\bx}$ and the last equation with $\bx$:
$$
\bar{\bx}^\top \bA \bx = \lambda \bar{\bx}^\top \bx \qquad \text{and } \qquad \bar{\bx}^\top \bA^\top \bx = \bar{\lambda}\bar{\bx}^\top \bx.
$$
Then we have the equality $-\lambda\bar{\bx}^\top \bx = \bar{\lambda} \bar{\bx}^\top\bx$ (since $\bA^\top=-\bA$). 
Since $\bar{\bx}^\top\bx = (\bc-i\bd)^\top(\bc+i\bd) = \bc^\top\bc+\bd^\top\bd$ is a real number, the real part of $\lambda$ must be zero, indicating $\lambda$ is either purely imaginary or zero.
\end{proof}

\begin{lemma}[Odd skew-symmetric determinant]\label{lemma:skew-symmetric-determinant}
For any skew-symmetric matrix $\bA\in \real^{n\times n}$, if $n$ is odd, then $\det(\bA)=0$.
\end{lemma}
\begin{proof}[of Lemma~\ref{lemma:skew-symmetric-determinant}]
When $n$ is odd, we have 
$$
\det(\bA) = \det(\bA^\top) = \det(-\bA) = (-1)^n \det(\bA) = -\det(\bA).
$$
This implies $\det(\bA)=0$.
\end{proof}

\index{Permutation matrix}
\begin{theoremHigh}[Block-diagonalization of skew-symmetric matrices]\label{theorem:skew-block-diagonalization_theorem}
A real skew-symmetric matrix $\bA \in \real^{n\times n}$ can be decomposed as
\begin{equation*}
\bA = \bZ \bD \bZ^\top,
\end{equation*}
where $\bZ$ is an $n\times n$ nonsingular matrix, and $\bD$ is a block-diagonal matrix of the following form
$$
\bD = 
\diag\left(\begin{bmatrix}
0 & 1 \\
-1 & 0
\end{bmatrix}, 
\ldots, 
\begin{bmatrix}
0 & 1 \\
-1 & 0
\end{bmatrix}, 
0, \ldots, 0\right).
$$
\end{theoremHigh}
\begin{proof}[of Theorem~\ref{theorem:skew-block-diagonalization_theorem}]
The proof follows from a recursive construction. 
As usual, we  denote the entry ($i,j$) of a matrix $\bA$ by $a_{ij}$. 

\paragraph{Case 1).} Suppose the first row of $\bA$ is nonzero. Note that $\bE\bA\bE^\top$ is skew-symmetric  for any matrix $\bE$ if $\bA$ is skew-symmetric. 
Therefore, both the diagonals of $\bA$ and $\bE\bA\bE^\top$ zero. The upper-left $2\times 2$ submatrix of $\bE\bA\bE^\top$ takes the following form
$$
(\bE\bA\bE^\top)_{1:2,1:2}=
\begin{bmatrix}
0 & x \\
-x & 0
\end{bmatrix}.
$$
Since we suppose the first row of $\bA$ is nonzero, there exists a permutation matrix $\bP$ (Definition~\ref{definition:permutation-matrix}), such that we will exchange the nonzero value, say $a$, in the first row to the second column of $\bP\bA\bP^\top$. The upper-left $2\times 2$ submatrix of $\bP\bA\bP^\top$ becomes
$$
(\bP\bA\bP^\top)_{1:2,1:2}=
\begin{bmatrix}
0 & a \\
-a & 0
\end{bmatrix}.
$$
Construct a nonsingular matrix $\bM = \scriptsize\begin{bmatrix}
1/a & \bzero \\
\bzero   & \bI_{n-1}
\end{bmatrix}$ such that the upper left $2\times 2$ submatrix of $\bM\bP\bA\bP^\top\bM^\top$ has the following form
$$
(\bM\bP\bA\bP^\top\bM^\top)_{1:2,1:2}=
\begin{bmatrix}
0 & 1 \\
-1 & 0
\end{bmatrix}.
$$
This completes the block-diagonalization of the upper-left $2\times 2$ block. 
Next, if there exists   a nonzero value, say $b$, in the first row of $(\bM\bP\bA\bP^\top\bM^\top)$ at position $(1,j)$ for some $j>2$, we can construct a nonsingular matrix $\bL = \bI - b\cdot\bE_{j2}$, where $\bE_{j2}$ is an all-zero matrix except that  the entry ($j,2$) is 1, such that $\bL(\bM\bP\bA\bP^\top\bM^\top)\bL^\top$  will set the entry with value $b$ to 0.

\begin{mdframed}[hidealllines=\mdframehideline,backgroundcolor=\mdframecolor,frametitle={A Trivial Example}]
For example, suppose $\bM\bP\bA\bP^\top\bM^\top$ is a $3\times 3$ matrix with the following value 
$$
\bM\bP\bA\bP^\top\bM^\top = 
\begin{bmatrix}
0 & 1 & b \\
-1 & 0 & \times \\
\times & \times & 0
\end{bmatrix}, 
\qquad \text{and}\qquad
\bL =\bI - b\cdot\bE_{j2}=
\begin{bmatrix}
1 & 0 & 0 \\
0 & 1 & 0 \\
0 &-b & 1
\end{bmatrix},
$$
where $j=3$ for this specific example. This results in 
$$
\bL\bM\bP\bA\bP^\top\bM^\top\bL^\top =
\begin{bmatrix}
1 & 0 & 0 \\
0 & 1 & 0 \\
0 &-b & 1
\end{bmatrix}
\begin{bmatrix}
0 & 1 & \textcolor{mylightbluetext}{b} \\
-1 & 0 & \times \\
\times & \times & 0
\end{bmatrix}
\begin{bmatrix}
1 & 0 & 0 \\
0 & 1 & -b \\
0 & 0 & 1
\end{bmatrix} 
= 
\begin{bmatrix}
0 & 1 & \textcolor{mylightbluetext}{0} \\
-1 & 0 & \times \\
\times & \times & 0
\end{bmatrix}.
$$
\end{mdframed}
Similarly, if the second row of $\bL\bM\bP\bA\bP^\top\bM^\top\bL^\top$ contains a nonzero value, say $c$, we can construct a nonsingular matrix $\bK = \bI+c\cdot \bE_{j1}$ such that $\bK(\bL\bM\bP\bA\bP^\top\bM^\top\bL^\top)\bK^\top$ will introduce a zero for the entry with value $c$.
\begin{mdframed}[hidealllines=\mdframehideline,backgroundcolor=\mdframecolor,frametitle={A Trivial Example}]
For example, suppose $\bL\bM\bP\bA\bP^\top\bM^\top\bL^\top$ is a $3\times 3$ matrix with the following value 
$$
\bL\bM\bP\bA\bP^\top\bM^\top\bL^\top = 
\begin{bmatrix}
	0 & 1 & 0 \\
	-1 & 0 & c \\
	\times & \times & 0
\end{bmatrix}, 
\qquad \text{and}\qquad
\bK =\bI + c\cdot\bE_{j1}=
\begin{bmatrix}
	1 & 0 & 0 \\
	0 & 1 & 0 \\
	c &0 & 1
\end{bmatrix},
$$
where $j=3$ for this specific example. This results in 
$$
\bK\bL\bM\bP\bA\bP^\top\bM^\top\bL^\top\bK^\top =
\begin{bmatrix}
	1 & 0 & 0 \\
	0 & 1 & 0 \\
	c & 0 & 1
\end{bmatrix}
\begin{bmatrix}
	0 & 1 & 0 \\
	-1 & 0 & \textcolor{mylightbluetext}{c} \\
	\times & \times & 0
\end{bmatrix}
\begin{bmatrix}
	1 & 0 & c \\
	0 & 1 & 0 \\
	0 & 0 & 1
\end{bmatrix} 
= 
\begin{bmatrix}
	0 & 1 & 0 \\
	-1 & 0 & \textcolor{mylightbluetext}{0} \\
	\times & \times & 0
\end{bmatrix}.
$$
Since we have shown that $\bK\bL\bM\bP\bA\bP^\top\bM^\top\bL^\top\bK^\top$ is also skew-symmetric, then it simplifies to 
$$
\bK\bL\bM\bP\bA\bP^\top\bM^\top\bL^\top\bK^\top= 
\begin{bmatrix}
0 & 1 & 0 \\
-1 & 0 & \textcolor{mylightbluetext}{0} \\
\textcolor{winestain}{0} & \textcolor{winestain}{0} & 0
\end{bmatrix},
$$
so we do not need to address the first two columns further.
\end{mdframed}
Apply this process iteratively to the bottom-right $(n-2)\times(n-2)$ submatrix can complete the block-diagonalization.
\paragraph{Case 2).} If the first row of $\bA$ is zero, we can use a permutation matrix to move the first row to the last row and then proceed with the process described in Case 1 to complete the proof.
\end{proof}

The block-diagonalization of skew-symmetric matrices, as discussed earlier, demonstrates that the rank of a skew-symmetric matrix is always even. Moreover, we can prove that the determinant of a skew-symmetric matrix of even order is nonnegative, as stated in the following lemma:
\begin{lemma}[Even skew-symmetric determinant]\label{lemma:skew-symmetric-determinant-even}
Let $\bA\in \real^{n\times n}$ be a skew-symmetric matrix. If $n$ is even, then $\det(\bA)\geq 0$.
\end{lemma}
\begin{proof}[of Lemma~\ref{lemma:skew-symmetric-determinant-even}]
Applying Theorem~\ref{theorem:skew-block-diagonalization_theorem}, $\bA$ can be  block-diagonalized as  $\bA = \bZ\bD\bZ^\top$, resulting in 
$
\det(\bA) = \det(\bZ\bD\bZ^\top) = \det(\bZ)^2 \det(\bD) \geq 0.
$
This completes the proof.
\end{proof}

\section{Applications in Optimization, Linear Algebra,  Machine Learning}
\index{Variable separation}
\subsection{Application: Variable Separation for Optimization}

Consider the quadratic function $ f(\bx) = \bx^\top \bA \bx + \bb^\top \bx + c $. Unless the symmetric matrix $ \bA $ is diagonal, the resulting function contains cross terms of the form $ x_i x_j $. 
These are known as \textit{interacting terms}, and they commonly appear in real-world quadratic functions.

It is worth noting that any multivariate quadratic function can be transformed into an \textit{additively separable} function (i.e., one without interacting terms) by applying a suitable linear transformation to the input variables.  
Additively separable functions are significantly easier to optimize, as the optimization problem can be decomposed into smaller, independent subproblems involving individual variables. 
For instance, a multivariate quadratic function can be rewritten as a simple sum of univariate quadratic functions, each of which is straightforward to minimize \citep{aggarwal2020linear}. We begin by formally defining the concept of separability:

\begin{definition}[Additively separable functions]
	A function $ F(x_1, x_2, \ldots, x_n) $ of $ n $ variables is said to be \textit{additively separable}  if it can be expressed in the following form for appropriately chosen univariate functions $ f_1(\cdot), f_2(\cdot), \ldots, f_n(\cdot) $:
	$$
	F(x_1, x_2, \ldots, x_n) = \sum_{i=1}^{n} f_i(x_i).
	$$
\end{definition}

Now consider the (symmetric) quadratic function defined on an  $ n $-dimensional vector $ \bx $:
$$
f(\bx) = \bx^\top \bA \bx + \bb^\top \bx + c.
$$
Since $ \bA $ is an $ n \times n $ symmetric matrix, it can be diagonalized as  $ \bA = \bQ \bLambda \bQ^\top $, and we can perform a variable transformation $ \bz = \bQ^\top \bx $. 
Substituting this transformation into the original function yields  a new function $ g(\bz) = f(\bQ\bz) $, which represents the same function expressed in a different basis. 
It can be shown that the transformed function becomes:
$$
g(\bz) = f(\bQ\bz) = \bz^\top \bLambda \bz + \bb^\top \bQ \bz + c.
$$
Because $\bLambda$ is a diagonal matrix, the function $g(\bz)$ becomes additively separable. This allows us to solve for $\bz$ using univariate optimization methods and then recover the original variable $\bx$ via $\bx=\bQ\bz$.

While this method simplifies the optimization process, a key drawback is that computing eigenvectors (as required for diagonalization) can be computationally expensive.
To mitigate this, one can generalize the approach by seeking a matrix $\bQ$ (not necessarily orthogonal) such that  $ \bA = \bQ \bLambda \bQ^\top $ for some diagonal matrix $\bLambda$. Note that this would not constitute a true diagonalization unless the columns of $\bQ$ are orthonormal and $ \bQ^\top = \bQ^{-1} $; see Definition~\ref{definition:diagonalizable} for more details.
Nevertheless, such a decomposition is sufficient for constructing a separable quadratic function.

\index{Orthogonal projection}
\index{Projection matrix (projector)}
\subsection{Application: Eigenvalue of Projection Matrices}\label{section:spec_app_eigproj}

In Section~\ref{section:application-ls-qr}, we will demonstrate how the QR, UTV, SVD decompositions can be applied to solve the least squares problem. Specifically, we consider the overdetermined system $\bA\bx = \bb$, where $\bA\in \real^{m\times n}$ is the data matrix, and $\bb\in \real^m$ is the observation vector, with $m\geq n$. Typically, $\bA$ is assumed to have full column rank, as real-world data is often sufficiently diverse to ensure linear independence, or the data can be made linearly independent after preprocessing.

Since $\bA$ has full column rank, $\bA^\top\bA$ is invertible, and $\rank(\bA^\top\bA) = \rank(\bA)$. 
Therefore, the least squares solution is given by $\bx_{LS} = (\bA^\top\bA)^{-1}\bA^\top\bb$, minimizing $\norm{\bA\bx-\bb}^2$.  The recovered observation vector is $\hat{\bb} = \bA\bx_{LS} = \bA(\bA^\top\bA)^{-1}\bA^\top\bb$. While the observed vector $\bb$ may not lie in the column space of $\bA$, the recovered vector $\hat{\bb}$ does.

We define the matrix $\bH = \bA(\bA^\top\bA)^{-1}\bA^\top$ as the (orthogonal) projection matrix, which projects $\bb$ onto the column space of $\bA$. This matrix is also known as the \textit{hat matrix} because it ``puts a hat" on $\bb$. It is straightforward to verify that $\bH$ is both symmetric ($\bH = \bH^\top$) and idempotent ($\bH^2 = \bH$).

\index{Idempotent}
\begin{remark}[Column space of projection matrices]
The hat matrix $\bH = \bA(\bA^\top\bA)^{-1}\bA^\top$ projects any vector in $\real^m$ onto the column space of $\bA$, i.e., $\bH\by \in \cspace(\bA)$. Notably, $\bH\by$ is a linear combination of the columns of $\bH$, which implies $\cspace(\bH) = \cspace(\bA)$.

More generally, for any projection matrix $\bH$ that projects vectors onto a subspace $\mathcalV$, it holds that $\cspace(\bH) = \mathcalV$. This property can be formally established using the singular value decomposition (Section~\ref{section:SVD}).
\end{remark}

We now show that any projection matrix has specific eigenvalues. 
\begin{proposition}[Eigenvalue of projection matrix]\label{proposition:eigen-of-projection-matrix}
The eigenvalues of a projection matrix are restricted to 0 and 1.
\end{proposition}
\begin{proof}[of Proposition~\ref{proposition:eigen-of-projection-matrix}]
Since $\bH$ is symmetric, it has a spectral decomposition $\bH = \bQ\bLambda\bQ^\top$. Using the idempotent property of $\bH$, we have:
$$
\begin{aligned}
(\bQ\bLambda\bQ^\top)^2 &= \bQ\bLambda\bQ^\top 
\,\,\implies\,\,
\bQ\bLambda^2\bQ^\top = \bQ\bLambda\bQ^\top
\,\,\implies\,\,
\bLambda^2 =\bLambda
\,\,\implies\,\,
\lambda_i^2 =\lambda_i,
\end{aligned}
$$
Thus, each eigenvalue satisfies $\lambda_i \in \{0, 1\}$.
\end{proof}

This property is significant in the analysis of distribution theory for linear models; see, for example, \citet{lu2021rigorous}.
Building on the eigenvalues of the projection matrix, we can also define the \textit{orthogonal complement projection matrix} $\bI-\bH$.
\begin{proposition}[Project onto $\mathcalV^\perp$]\label{proposition:orthogonal-projection_tmp}
	Let $\mathcalV$ be a subspace, and $\bH$ be the projection matrix onto $\mathcalV$. Then, $\bI-\bH$ serves as the projection matrix onto $\mathcalV^\perp$.
\end{proposition}

\begin{proof}[of Proposition~\ref{proposition:orthogonal-projection_tmp}]
First, $(\bI-\bH)$ is symmetric, $(\bI-\bH)^\top = \bI - \bH^\top = \bI-\bH$ since $\bH$ is symmetric. Furthermore, it follows that  
$$
(\bI-\bH)^2 = \bI^2 -\bI\bH -\bH\bI +\bH^2 = \bI-\bH,
$$
which shows that $\bI-\bH$ is idempotent.
Thus, $\bI-\bH$ qualifies as a projection matrix. Using the spectral theorem, write $\bH = \bQ\bLambda\bQ^\top$. Then, $\bI-\bH = \bQ\bQ^\top - \bQ\bLambda\bQ^\top = \bQ(\bI-\bLambda)\bQ^\top$. 
Consequently, the column space of $\bI-\bH$ is spanned by the eigenvectors of $\bH$ corresponding to the zero eigenvalues of $\bH$ (by Proposition~\ref{proposition:eigen-of-projection-matrix}), which aligns with $\mathcalV^\perp$.
\end{proof}

For a more detailed discussion of projection matrices and their applications, refer to \citet{lu2021numerical}. While these results are important, they extend beyond the primary focus of matrix decomposition techniques, and thus will not be repeated here.

\subsection{Application: An Alternative Definition of PD and PSD of Matrices}\label{section:equivalent-pd-psd}
In Definition~\ref{definition:psd-pd-defini}, positive definite (PD) and positive semidefinite (PSD) matrices are defined based on their quadratic forms. Here, we establish that a symmetric matrix is positive definite (resp., positive semidefinite) if and only if all its eigenvalues are positive (resp., nonnegative).

\begin{lemma}[Eigenvalues of PD and PSD matrices i.e., the eigenvalue characterization theorem\index{Positive definite}\index{Positive semidefinite}]\label{lemma:eigens-of-PD-psd}
A symmetric matrix $\bA\in \real^{n\times n}$ is positive definite (PD) if and only if all eigenvalues of $\bA$ are positive.
And a symmetric matrix $\bA\in \real^{n\times n}$ is positive semidefinite (PSD) if and only if all eigenvalues of $\bA$ are nonnegative.
\end{lemma}
\begin{proof}[of Lemma~\ref{lemma:eigens-of-PD-psd}] 
Suppose $\bA$ is PD. Then, for any eigenvalue $\lambda$ and its corresponding eigenvector $\bv$ of $\bA$, we have $\bA\bv = \lambda\bv$. Thus,
$
\bv^\top \bA\bv = \lambda\norm{\bv}^2 > 0.
$
This implies $\lambda>0$.

Conversely, suppose all eigenvalues of $\bA$ are positive, and consider the spectral decomposition of $\bA =\bQ\bLambda \bQ^\top$, where $\bQ$ is orthogonal and $\bLambda$ is diagonal. Let $\bx$ be any nonzero vector, and let $\by = \bQ^\top \bx$. We have:
$$
\bx^\top \bA \bx = \bx^\top (\bQ\bLambda \bQ^\top) \bx = (\bx^\top \bQ) \bLambda (\bQ^\top\bx) = \by^\top\bLambda\by = \sum_{i=1}^{n} \lambda_i y_i^2>0.
$$
Thus, $\bA$ is PD.
The proof for the PSD case follows similarly
\end{proof}

\begin{theoremHigh}[Nonsingular factor of PSD and PD matrices]\label{lemma:nonsingular-factor-of-PD}
A real symmetric matrix $\bA\in\real^{n\times n}$ is PSD if and only if $\bA$ can be factored as $\bA=\bP^\top\bP$, where $\bP\in\real^{n\times n}$; and it is PD if and only if $\bP$ is nonsingular. 
\end{theoremHigh}
\begin{proof}[of Theorem~\ref{lemma:nonsingular-factor-of-PD}]
Suppose $\bA$ is PSD. From its spectral decomposition $\bA = \bQ\bLambda\bQ^\top$, we can decompose $\bLambda=\bLambda^{1/2}\bLambda^{1/2}$ (since the eigenvalues of any PSD matrix are nonnegative). Let $\bP = \bLambda^{1/2}\bQ^\top$. Then, $\bA$ can be decomposed as $\bA=\bP^\top\bP$.

Conversely, suppose $\bA$ can be factored as $\bA=\bP^\top\bP$. Then, all eigenvalues of $\bA$ are nonnegative since for any eigenvalues $\lambda$ and its corresponding eigenvector $\bv$ of $\bA$, we have 
$$
\lambda = \frac{\bv^\top\bA\bv}{\bv^\top\bv} = \frac{\bv^\top\bP^\top\bP\bv}{\bv^\top\bv}=\frac{\norm{\bP\bv}^2}{\norm{\bv}^2} \geq 0.
$$
Therefore, $\bA$ is PSD by Lemma~\ref{lemma:eigens-of-PD-psd}.

Similarly, we can prove the second part for PD matrices, where the positive definiteness will result in the nonsingular $\bP$; and the nonsingularity of $\bP$ implies the positivity of the eigenvalues. 
\end{proof}

\subsection{Proof for Semidefinite Rank-Revealing Decomposition}\label{section:semi-rank-reveal-proof}

In this section, we provide an alternative proof for Theorem~\ref{theorem:semidefinite-factor-rank-reveal}, which establishes the existence of a rank-revealing decomposition for positive semidefinite matrices.\index{Rank-revealing}\index{Semidefinite rank-revealing}
\begin{proof}[of Theorem~\ref{theorem:semidefinite-factor-rank-reveal}]
The proof is based on two key results: the nonsingular factorization of PSD matrices (Theorem~\ref{lemma:nonsingular-factor-of-PD}) and the column-pivoted QR decomposition (Theorem~\ref{theorem:rank-revealing-qr-general}).
	
By Theorem~\ref{lemma:nonsingular-factor-of-PD}, any PSD matrix $\bA$ can be factored as $\bA = \bZ^\top \bZ$, where $\bZ = \bLambda^{1/2} \bQ^\top$, and $\bA = \bQ \bLambda \bQ^\top$ is the spectral decomposition of $\bA$.

By Lemma~\ref{lemma:rank-of-symmetric}, the rank of $\bA$ equals the number of its nonzero eigenvalues, which  corresponds to the positive eigenvalues for a PSD matrix. Consequently, only  $r$ diagonal elements of $\bLambda^{1/2}$ are nonzero, making $\bZ = \bLambda^{1/2} \bQ^\top$ a rank-$r$ matrix with $r$ linearly independent columns.
Applying the column-pivoted QR decomposition to $\bZ$, we obtain
$
\bZ\bP = \bQ
\scriptsize
\begin{bmatrix}
	\bR_{11} & \bR_{12} \\
	\bzero   & \bzero 
\end{bmatrix},
$
where $\bP$ is a permutation matrix, $\bR_{11}\in \real^{r\times r}$ is upper triangular with positive diagonals, and $\bR_{12}\in \real^{r\times (n-r)}$. Therefore,
$$
\bP^\top\bA\bP  = 
\bP^\top\bZ^\top\bZ\bP = 
\begin{bmatrix}
\bR_{11}^\top & \bzero \\
\bR_{12}^\top & \bzero 
\end{bmatrix}
\begin{bmatrix}
	\bR_{11} & \bR_{12} \\
	\bzero   & \bzero 
\end{bmatrix}.
$$
Let 
$
\bR = \scriptsize\begin{bmatrix}
	\bR_{11} & \bR_{12} \\
	\bzero   & \bzero 
\end{bmatrix}.
$
Thus, the rank-revealing decomposition of the PSD matrix $\bA$ is: $\bP^\top\bA\bP = \bR^\top\bR$.
\end{proof}

\index{CPQR}
\index{Pivoting}
\index{Complete pivoting}
This decomposition is obtained using complete pivoting, where at each step the algorithm selects the largest diagonal element in the active submatrix as the pivot. This strategy is conceptually similar to the partial pivoting technique discussed in Section~\ref{section:partial-pivot-lu}.

\subsection{Application: Cholesky  via  QR  and  Spectral Decompositions}\label{section:cholesky-by-qr-spectral}
In this section, we present an alternative proof for the existence of the Cholesky decomposition using the nonsingular factor of PD matrices.

\begin{proof}[of Theorem~\ref{theorem:cholesky-factor-exist}]
From Theorem~\ref{lemma:nonsingular-factor-of-PD}, the PD matrix $\bA$ can be factored as $\bA=\bP^\top\bP$, where $\bP$ is a nonsingular matrix. Applying the QR decomposition to $\bP$, we write $\bP = \bQ \bR$, which implies:
$$
\bA = \bP^\top\bP = \bR^\top\bQ^\top\bQ\bR = \bR^\top\bR.
$$
This result closely resembles the Cholesky decomposition, with the exception that $\bR$ is not explicitly required to have positive diagonal entries. However, by considering the CGS algorithm for computing the QR decomposition (discussed in Section~\ref{section:qr-gram-compute}), it can be observed that the diagonal entries of $\bR$ are nonnegative. Moreover, if $\bP$ is nonsingular, these diagonal entries are strictly positive.
\end{proof}
The proof above relies on the existence of both the QR decomposition and the spectral decomposition. Thus, in this context, the existence of the Cholesky decomposition can be demonstrated using these two fundamental decomposition methods.

\subsection{Application: Unique Power Decomposition of PD Matrices}\label{section:unique-posere-pd}
In this section, we present a \textit{unique power decomposition} for positive definite  matrices using their spectral decomposition.
\begin{theoremHigh}[Unique power decomposition of PD matrices]\label{theorem:unique-factor-pd}
Any $n\times n$ positive definite matrix $\bA$ can be \textbf{uniquely} decomposed as the square of another positive definite matrix  $\bB$, that is, $\bA =\bB^2$. 
\end{theoremHigh}
\begin{proof}[of Theorem~\ref{theorem:unique-factor-pd}]
We first prove the existence of a positive definite matrix $\bB$ satisfying $\bA = \bB^2$ and then demonstrate its uniqueness.
\paragraph{Existence.} Since $\bA$ symmetric and positive definite, its spectral decomposition is given by $\bA = \bQ\bLambda\bQ^\top$. 
By Lemma~\ref{lemma:eigens-of-PD-psd}, all eigenvalues of a PD matrix are strictly positive. Hence, the square root of $\bLambda$ exists, allowing us to define $\bB = \bQ\bLambda^{1/2}\bQ^\top$. It follows that $\bA = \bB^2$. Since $\bB$ is symmetric with positive eigenvalues, it is also positive definite.

\paragraph{Uniqueness.}Suppose the factorization is not unique. Then, there exist two positive definite matrices $\bB_1$ and $\bB_2$ such that
$$
\bA = \bB_1^2 = \bB_2^2,
$$
where both $\bB_1$ and $\bB_2$ are  PD. Their spectral decompositions are given by 
$$
\bB_1 = \bQ_1 \bLambda_1\bQ_1^\top \qquad \text{and} \qquad \bB_2 = \bQ_2 \bLambda_2\bQ_2^\top.
$$
We notice that $\bLambda_1^2$ and $\bLambda_2^2$ contain the eigenvalues of $\bA$, and both eigenvalues of $\bB_1$ and $\bB_2$ contained in $\bLambda_1$ and $\bLambda_2$ are positive (since $\bB_1$ and $\bB_2$ are both PD). Without loss of generality, we suppose $\bLambda_1=\bLambda_2=\bLambda^{1/2}$, and $\bLambda=\diag(\lambda_1,\lambda_2, \ldots, \lambda_n)$ such that $\lambda_1\geq \lambda_2 \geq \ldots \geq \lambda_n$. By $\bB_1^2 = \bB_2^2$, we have
$$
\bQ_1 \bLambda \bQ_1^\top = \bQ_2 \bLambda \bQ_2^\top  \leadto \bQ_2^\top\bQ_1 \bLambda = \bLambda \bQ_2^\top\bQ_1.
$$
Let $\bZ = \bQ_2^\top\bQ_1 $.  This implies $\bLambda$ and $\bZ$ commute, and $\bZ$ must be a block diagonal matrix whose partitioning conforms to the block structure of $\bLambda$ \citep{lu2021numerical}. This results in $\bLambda^{1/2} = \bZ\bLambda^{1/2}\bZ^\top$ and 
$$
\bB_2 = \bQ_2 \bLambda^{1/2}\bQ_2^\top = \bQ_2 \bQ_2^\top\bQ_1\bLambda^{1/2} \bQ_1^\top\bQ_2 \bQ_2^\top=\bB_1.
$$
This completes the proof.
\end{proof}
Similarly, we can prove  the unique decomposition of a PSD matrix $\bA$ such that $\bA = \bB^2$, where $\bB$ is PSD  \citep{koeber2006unique}.

\paragraph{Decompositions for PD matrices.} To summarize, a PD matrix $\bA$ can be decomposed in several ways: we can factor it into $\bA=\bR^\top\bR$, where $\bR$ is an upper triangular matrix with positive diagonals as shown in Theorem~\ref{theorem:cholesky-factor-exist}  by the Cholesky decomposition; $\bA = \bP^\top\bP$, where $\bP$ is nonsingular in Theorem~\ref{lemma:nonsingular-factor-of-PD}; and $\bA = \bB^2$, where $\bB$ is PD in Theorem~\ref{theorem:unique-factor-pd}.

\index{Kernel function}
\index{Scatter matrix}
\subsection{Application: Feature Engineering for Scatter Matrices}\label{section:feat_eng_scatt}

Consider an $n\times p$ data matrix $\bX$, where each row $\bx_i$ represents a data point. In machine learning, an $n\times n$ symmetric \textit{scatter matrix} or \textit{kernel matrix} $\bS$ can be defined among the $n$ data points as follows:
$$
s_{ij}  = k(\bx_i, \bx_j) =\phi(\bx_i)^\top\phi(\bx_j),
\quad \forall\, i,j,
$$
where $k(\bx,\by)=\phi(\bx)^\top\phi(\by)$ is called a \textit{kernel function}, and $\phi(\bx)$ is the associated \textit{basis function}. 
Thus, the scatter matrix $\bS$ contains all pairwise kernel evaluations between data points.
It  can be easily shown that $\bS$ must be symmetric and positive semidefinite \citep{lu2021rigorous}.
\begin{exercise}[Properties of scatter matrices]
	Let  $\bX$ be an $n\times p$ data matrix, where each row $\bx_i$ denotes a data point. And let $\phi(\cdot):\real^p\rightarrow \real^k$ be a basis function. Show that 
	\begin{itemize}
		\item $\bS$ is a symmetric matrix, i.e., $k(\bx_i, \bx_j) = k(\bx_j, \bx_i)$.
		\item $\bS$ is positive semidefinite.
	\end{itemize}
\end{exercise}

\index{Gaussian kernel}
\index{Linear kernel}
\index{Polynomial kernel}
At first glance, it may appear that $k(\bx, \bx')$ can be any arbitrary function of $\bx$ and $\bx'$. 
However, the requirement for the scatter matrix to be positive semidefinite constrains the form of valid kernel functions. 
This constraint ensures that every valid kernel corresponds to an implicit inner product in some (possibly infinite-dimensional) feature space.
The following are examples of widely used kernel functions:
\begin{enumerate}
	\item  \textit{Linear kernel.} $k(\bx, \bx') = \bx^\top\bx'$.
	\item  \textit{Polynomial kernel.} $k(\bx, \bx') = (\eta + \gamma \bx^\top\bx')^Q$ with $\gamma>0, \eta \geq 0$.
	\item  \textit{Gaussian kernel.} $k(\bx, \bx') = \exp(-\gamma \norm{\bx-\bx'}^2)$. 
	We now show that the Gaussian kernel corresponds to an infinite-dimensional feature mapping. Without loss of generality, let $\gamma = 1$. Then,
	\begin{equation}
		\begin{aligned}
			k(\bx, \bx') &= \exp\{- \norm{\bx-\bx'}^2\}
			=\exp\{-\bx^\top\bx\}\exp\{-\bx'^\top\bx'\}\exp\{2\bx^\top\bx'\} \\
			&\underset{\mathrm{expansion}}{\overset{\mathrm{Taylor}}{=}}
			\exp\{-\bx^\top\bx\}\exp\{-\bx'^\top\bx'\}\exp\left\{\sum_{i=0}^\infty \frac{(2\bx^\top\bx')^i}{i!}\right\} \\
			&=\sum_{i=0}^\infty\left(\exp\{-\bx^\top\bx\}\exp\{-\bx'^\top\bx'\} \sqrt{\frac{2^i}{i!}}\sqrt{\frac{2^i}{i!}} (\bx)^{i}\cdot (\bx')^i  \right) \\
			&=\sum_{i=0}^\infty\left({\color{winestain}\exp\{-\bx^\top\bx\}\sqrt{\frac{2^i}{i!}} (\bx)^{i}} \cdot {\color{mylightbluetext}\exp\{-\bx'^\top\bx'\} \sqrt{\frac{2^i}{i!}}  (\bx')^i}  \right) \\
			&={\color{winestain}\bphi(\bx)^\top} {\color{mylightbluetext}\bphi(\bx')},\nonumber
		\end{aligned}
	\end{equation}
	where $\bphi(\bx) = \sum_{i=0}^\infty  \exp\{-\bx^\top\bx\}\sqrt{\frac{2^i}{i!}} (\bx)^{i} $. 
	This shows that the Gaussian kernel maps inputs from a finite-dimensional space to an infinite-dimensional space. A similar derivation holds for general $\gamma > 0$.
\end{enumerate}

Given the data matrix $\bX$ and the basis function $\phi$, it is easy to transform the data matrix $\bX$ into its corresponding scatter matrix $\bS$.
However, recovering the original data $\bX$ from the scatter matrix $\bS$ is more complex. The recovery process cannot be unique due to the invariance of dot products under rotation and reflection.
For example, consider a $p \times p$ orthogonal matrix $\bU$, which acts as a rotation/reflection matrix. Then, the rotated/reflected version of $\bX$ is
$
\widetildebX = \bX\bU .
$
Consequently, the scatter matrix $\widetildebS$ using $\widetildebX$ can be shown to be equal to $\bS$ as follows:
$$
\widetildebS = \widetildebX\widetildebX^\top = (\bX\bU)(\bX\bU)^\top = \bX {(\bU\bU^\top)} \bX^\top = \bS .
$$

A \textit{symmetric factorization} of an $n \times n$ matrix is a factorization of $\bS$ into two $n \times k$ matrices of the form $\bS = \bW\bW^\top$. For exact factorization, the value of $k$ will be equal to the rank of the scatter matrix $\bS$. The $i$-th row of $\bW$ in any symmetric factorization $\bW\bW^\top$ of $\bS$ yields a valid set of features of the $i$-th data point.
The representation of $\bW$ is important because it enables the use of many machine learning algorithms---such as \textit{support vector machines (SVMs)} or \textit{logistic regression}---that operate on multidimensional data.

There are three common methods for performing symmetric factorization:
\begin{itemize}
\item Spectral decomposition: $\bS = \bQ\bLambda\bQ^\top$. 
Since the eigenvalues of a positive semidefinite matrix are nonnegative, we can represent the diagonal matrix as $\bLambda = \bSigma^2$:
$
\bS = \bQ\bSigma^2\bQ^\top = ({\bQ\bSigma})(\bQ\bSigma)^\top = \bW\bW^\top.
$
Spectral decomposition of the scatter matrix provides one of infinitely many possible representations that can be derived from factorizing $\bS$. Among these, it is also one of the most compact in terms of the number of nonzero columns. The compactness can be further improved by discarding eigenvectors corresponding to small eigenvalues.
\item Symmetric square-root matrix, which can also be extracted from the spectral decomposition as $\bS = \bQ\bSigma^2\bQ^\top = (\bQ\bSigma\bQ^\top)(\bQ\bSigma\bQ^\top)^\top = (\sqrt{\bS})^2$. In this case, we set $\bW$ to be $\bQ\bSigma\bQ^\top$. 
\item Cholesky factorization: $\bS = \bL\bL^\top$, and we set $\bW=\bL$.
\end{itemize}
In all cases, the $i$-th row of $\bW$ contains the \textit{embedded  representation} (also referred to as the \textit{hidden or latent representation}) of the $i$-th data point. 
Choosing any of these representations will not affect the predictions made by machine learning algorithms that rely on dot products (or Euclidean distances), since these quantities remain unchanged regardless of whether we use spectral decomposition, Cholesky factorization, or the square-root matrix. For example, see its application in large language model compression \citep{lu2025large}, and in generalized least squares models \citep{lu2021rigorous}.

\index{Kernel clustering}
\paragraph{Kernel clustering.}
The kernel representation of $\bW$ is crucial because it enables the use of various machine learning algorithms. Consider a scenario where we have an $n \times n$ {scatter matrix} $\bS$ for $n$ data points, and we aim to cluster these points into similar groups. 
Using the spectral decomposition as an example, the approach of explicit feature engineering involves diagonalizing the scatter matrix as follows:
\begin{enumerate}
\item Diagonalize $\bS = \bQ \bSigma^2 \bQ^\top$.
\item Extract the $n$-dimensional embeddings from the rows of $\bQ \bSigma$.
\item Remove any zero columns from $\bQ \bSigma$ to form $\bQ_0 \bSigma_0$.
\item Apply a clustering algorithm (e.g., Bayesian GMM, K-Means \citep{lu2021bayes}) on the rows of $\bQ_0 \bSigma_0$.
\end{enumerate}
In this process, the columns of $\bQ_0$ contain the nonzero eigenvectors, and the $n$ rows of $\bQ_0 \bSigma_0$ represent the latent features of the $n$ data points. 

\index{Clustering}
\paragraph{Kernel clustering for adjacency matrices.}

In addition to its application to scatter matrices, the concept of kernel clustering can also be applied to the \textit{adjacency matrix} of an \textit{undirected graph}.

A \textit{graph}, sometimes referred to as a \textit{network}, is a mathematical structure used to represent ``relationships" (i.e., \textit{edges} in the graph) among objects (i.e., \textit{vertices or nodes} in the graph). The objects can be of any type---such as web pages, individuals in a social network, or chemical elements---while the relationships depend on the specific application; examples include hyperlinks between web pages, friendships in social networks, or chemical bonds between molecules.

A graph is considered \textit{undirected} when its edges do not have a direction. For example, a Facebook friendship link  is undirected. In contrast, a graph is \textit{directed} when its edges do have direction, as in follower-followee relationships on Twitter.

An adjacency matrix $\bA$ is a square matrix used to represent the structure of an undirected graph. Each entry $a_{ij}$ of the matrix indicates whether there is an edge between node $i$ and node $j$: it takes the value 1 or a positive weight $w_{ij}$ if such an edge exists, and 0 otherwise. This definition extends naturally to directed graphs.
Note that for a directed graph with $q$ edges, the adjacency matrix will contain exactly $q$ nonzero entries. In contrast, for an undirected graph with $q$ edges, the adjacency matrix will have $2q$ nonzero entries due to symmetry.

Since the adjacency matrix of an undirected graph is symmetric (and possibly indefinite), kernel clustering methods based on spectral decomposition can be directly applied to it. This process, commonly known as \textit{spectral clustering}, allows us to extract similarity information between nodes.
However, because the adjacency matrix of a directed graph is asymmetric, these methods cannot be directly extended. We will explore this issue further in Section~\ref{section:spec_clust}.

\index{Outlier detection}
\index{Kernel outlier detection}
\paragraph{Kernel outlier detection.}
For outlier detection tasks, the scatter matrix can be utilized in the following way:
\begin{enumerate}
\item Diagonalize $\bS = \bQ \bSigma^2 \bQ^\top$.
\item Extract the $n$-dimensional embeddings from the rows of $\bQ \bSigma$.
\item Remove any zero columns from $\bQ \bSigma$ to form $\bQ_0 \bSigma_0$.
\item Calculate the outlier score for each row of $\bQ_0$ as the $\ell_2$ distance from the mean of all rows in $\bQ_0$.
\end{enumerate}
It is important to note that we use $\bQ_0$ instead of $\bQ_0 \bSigma_0$ to compute the outlier score for each point. This distinction is critical in outlier detection since outliers often manifest in the deviations along lower-order eigenvectors.
If we were to multiply by $\bSigma_0$, it would scale down these deviations, potentially making outliers harder to detect \citep{aggarwal2020linear}.

\begin{problemset}
\item True or False?
\begin{itemize}
	\item If a (square) matrix has all zero eigenvalues, then it must be the zero matrix.
	
	\item If a \textbf{symmetric} matrix has all zero eigenvalues, then it must be the zero matrix. 
\end{itemize}

\item Show that the determinant of any diagonalizable matrix equals the product of its eigenvalues.
\item Let $\bA$ be a square and diagonalizable matrix. Consider a situation in which we add $\lambda$ to each diagonal entry of $\bA$ to create $\bB$. Show that $\bB$ has the same eigenvectors as $\bA$, and its eigenvalues are related to $\bA$ by a difference of $\lambda$.

\item Show that the eigenvalues of a matrix $\bA \in \real^{n \times n}$ are the same as those of its transpose $\bA^\top$. \textit{Hint: Use the characteristic polynomial of $\bA$ to prove this.}

\item Let $\bA$ and $\bB$ be symmetric positive definite matrices. Show that the product $\bA\bB$ may not be symmetric, but its eigenvalues remain positive. \textit{Hint: Take the product of $\bB\bx$ and $\bA\bB\bx=\lambda\bx$.}

\item Given a symmetric positive definite matrix $\bA$, show that $\bB^\top\bA\bB$ is positive definite if $\bB$ contains linearly independent columns.

\item Let $\bA \in \real^{n \times n}$ be a symmetric positive definite matrix with eigenvalues ordered as   $\lambda_1\geq \lambda_2 \geq \ldots \geq \lambda_n$.
\begin{itemize}
\item Find the eigenvalues of $\lambda\bI-\bA$.
\item Prove that $\lambda\bI - \bA$ is positive semidefinite.
\item Show that $\lambda_1\bx^\top\bx\geq \bx^\top\bA\bx$ for all $\bx\in\real^n$.
\item Determine  the maximum value of $\bx^\top\bA\bx/\bx^\top\bx$.
\end{itemize}

\item Let $\bA \in \real^{n \times n}$ satisfy $\bA^2 - \bA = 2\bI$. Show that $\bA$ is diagonalizable.

\item Suppose that  $\bA$ is a diagonalizable matrix, i.e., $\bA$ can be expressed as $\bA = \bP \bLambda \bP^{-1}$. Show that the matrix $
\lim_{k \to \infty} \left(\bI + \frac{\bA}{k}\right)^k$  exists with finite entries.
\textit{Hint: Use the fact that $\lim_{k \to \infty} \left(1 + \frac{x}{k}\right)^k = e^x$.}~\footnote{The result also holds for any square matrix.}

\item What can you claim about $\bP\in\real^{m\times n}$ with $m\geq n$ in Theorem~\ref{lemma:nonsingular-factor-of-PD}?

\item Show that two normal matrices are similar if and only if they have the same characteristic polynomial.

\item \textbf{Symmetric idempotent.} Let $\bA=\bA^\top=\bA^2\in\real^{n\times n}$ with $\rank(\bA)=r$. Show that there exists an orthogonal matrix $\bQ$ such that $\bQ^\top\bA\bQ=\diag(\bI_r, \bzero)$.

\item \textbf{Skew-symmetric.} Let $\bA\in\real^{n\times n}$ be skew-symmetric. Show that $\trace(\bA)=0$. Additionally, if $\bB\in \real^{n\times n}$ is symmetric, show that $\trace(\bA\bB)=0$.

\item Show that when $\bA$ and $\bB$ are positive semidefintie, then  the condition $\trace(\bA\bB) = 0$ is equivalent to $\bA\bB = \bzero$. \textit{Hint: The trace is invariant under cyclic permutations, and write out the trace using the elements of matrices from the spectral  decomposition.}

\index{Fan's inequality}
\index{Hardy-Littlewood-P\'olya inequality}
\item \label{prob:fans_ineq} \textbf{Fan's inequality \citep{fan1949theorem, borwein2006convex}.} 
Let $\bA,\bB$ be real symmetric, and let $\blambda^{\downarrow}(\bA)$ be the vector containing the eigenvalues of $\bA$ in nonincreasing order. Show that  $\trace(\bA\bB) \leq \blambda^{\downarrow}(\bA)^{\top}\blambda^{\downarrow}(\bB)$.
The equality holds if and only if $\bA$ and $\bB$ admit spectral decompositions $\bA=\bQ\blambda^{\downarrow}(\bA)\bQ^\top$ and $\bB=\bQ\blambda^{\downarrow}(\bB)\bQ^\top$ (called \textit{simultaneous ordered spectral decomposition}). \footnote{Fan's inequality is a refinement of the Cauchy-Schwarz inequality for symmetric matrices.}

\item \label{prob:hardy_ineq} \textbf{Hardy-Littlewood-P\'olya inequality  \citep{borwein2006convex}.} Let $[\bx]^{\downarrow}$ denote the vector with the same components of $\bx$ permuted into nonincreasing order. Show that  $\bx^\top\by \leq [\bx]^{\downarrow \top}[\by]^{\downarrow}$. \textit{Hint: Apply  Fan's inequality to diagonal matrices.}

\index{Mirsky's theorem}
\item \label{prob:pert_eig_uniinvar} \textbf{Mirsky's theorem.} Let $\bA,\bB\in\real^{n\times n}$ be symmetric. Let further $\blambda^{\downarrow}(\bA)$ and $\blambda^{\downarrow}(\bB)$ be the vectors containing the eigenvalues of $\bA$ and $\bB$, respectively, in nonincreasing order. Show that $\norm{\diag(\blambda^{\downarrow}(\bA)) - \diag(\blambda^{\downarrow}(\bB))} \leq \norm{\bA-\bB}$ if the matrix norm $\norm{\cdot}$ is \textit{unitarily/orthogonally invariant} (i.e., $\norm{\bU\bA\bV} =\norm{\bA}$ for all orthogonal $\bU\in\real^{m\times m}$ and $\bV\in\real^{n\times n}$ and for all $\bA\in\real^{m \times n}$).

\item \label{problem:eig_det_house} Let $\bH=\bI - 2  \frac{\bu\bu^\top}{\bu^\top\bu}\in\real^{n\times n}$ be  a Householder reflector. Show  that the  eigenvalue $\lambda_1=1$ has multiplicity  $n-1$; and the  eigenvalue $\lambda_2=-1$ has multiplicity  $1$.
This implies $\det(\bH) = -1$.

\item \label{problem:eig_reverse} \textbf{Eigenvalues of reverse product.}  Let $\bA\in\real^{m\times n}$ and $\bB\in\real^{n\times m}$ with $m\leq n$. Show that the $n$ eigenvalues $\bB\bA$  are the eigenvalues of $\bA\bB$ together with $n-m$ zeros. \textit{Hint: Show that
	$\scriptsize
	\begin{bmatrix}
		\bA\bB&\bzero \\
		\bB& \bzero 
	\end{bmatrix}$
	and 
	$\scriptsize
	\begin{bmatrix}
		\bzero &\bzero \\
		\bB& \bA\bB
	\end{bmatrix}$
	are similar, and use Proposition~\ref{proposition:eigenvalue-similar-matrices} to discuss the eigenvalues of the two matrices.
}

\index{Rank decomposition}
\item \textbf{Eigenvalues of rank decomposition.} Consider the rank decomposition of $\bA=\bD\bF\in\real^{n\times n}$ with rank $r$ (Theorem~\ref{theorem:rank-decomposition}). Show that the eigenvalues of $\bA$ are the same as those of $\bF\bD$ together with $n-r$ zeros. 
How does this result change if the decomposition satisfies $\bD\in\real^{n\times k}$ and $\bF\in\real^{k\times n}$ with $k>r$?
\textit{Hint: Use Problem~\ref{problem:eig_reverse}.}

\item \textbf{Subspace in symmetric.} Let $\bA\in\real^{n\times n}$ be symmetric. Show that $\cspace(\bA)=\cspace(\bA^k)$ and $\nspace(\bA)=\nspace(\bA^k)$ for all integers $k\geq 2$.

\item Show that $\bA$ is symmetric $\iff$ 
$\scriptsize\begin{bmatrix}
	\bzero & \bA\\
	\bA& \bzero
\end{bmatrix}$ is symmetric 
$\iff$
$\scriptsize\begin{bmatrix}
	\bzero & \bA\\
	-\bA& \bzero
\end{bmatrix}$ 
is skew-symmetric.

\item We introduced kernel clustering and kernel outlier detection in the main text. Discuss how this process can be applied or adapted for use in classification or regression tasks.

\item \textbf{Sigmoid kernel.} Is the ``sigmoid kernel" a valid kernel function: $k(\bx,\bx')=\text{tanh}(\kappa \cdot \bx^\top\bx'-\sigma)$, where $\text{tanh}(x) = \frac{e^x-e^{-x}}{e^x+e^{-x}}$, and $\kappa$ and $\sigma$ are scalars?
\end{problemset}

\chapter{Singular Value Decomposition (SVD)}\label{chapter:SVD_realc}

\section{Singular Value Decomposition (SVD)}\label{section:SVD}

In eigenvalue decomposition, a matrix is typically factored into a diagonal matrix. However, this is not always possible. If the underlying matrix lacks linearly independent eigenvectors, diagonalization cannot be performed. The \textit{singular value decomposition (SVD)} overcomes this limitation. Instead of decomposing a matrix into an eigenvector matrix, SVD breaks it into two orthogonal matrices.
We present the result of the SVD in the following theorem and will discuss its existence in later sections.
\index{Decomposition: SVD}
\index{Orthogonal}
\index{Rank}
\begin{theoremHigh}[Reduced SVD for rectangular matrices]\label{theorem:reduced_svd_rectangular}
Given any real $m\times n$ matrix $\bA$ of rank $r$,  the  matrix $\bA$ can be decomposed as
$$
\bA = \bU \bSigma \bV^\top,
$$ 
where $\bSigma\in \real^{r\times r}$ is a diagonal matrix, $\bSigma=\diag(\sigma_1, \sigma_2 \ldots, \sigma_r)$ with $\sigma_1 \geq \sigma_2 \geq \ldots \geq \sigma_r >0$ and 
\begin{itemize}
\item The values $\sigma_i$ are the nonzero \textit{singular values} of $\bA$; in the meantime, they are also the (positive) square roots of the nonzero \textit{eigenvalues} of both ${\bA}^\top \bA$ and $ \bA {\bA}^\top$.

\item The columns of $\bU\in \real^{m\times r}$ contain the $r$ eigenvectors of $\bA\bA^\top$ corresponding to the $r$ nonzero eigenvalues of $\bA\bA^\top$. 

\item The columns of $\bV\in \real^{n\times r}$ contain the $r$ eigenvectors of $\bA^\top\bA$ corresponding to the $r$ nonzero eigenvalues of $\bA^\top\bA$. 

\item Additionally, the columns of $\bU$ and $\bV$ are called the \textit{left and right singular vectors} of $\bA$, respectively. 

\item Moreover, the columns of both $\bU$ and $\bV$ are mutually orthonormal (by spectral theorem~\ref{theorem:spectral_theorem}). 
\end{itemize}
In particular, we can express the matrix decomposition as a sum of outer products of vectors $\bA = \bU \bSigma \bV^\top = \sum_{i=1}^r \sigma_i \bu_i \bv_i^\top$, which represents a sum of $r$ rank-one matrices.
\end{theoremHigh}

If we append $m - r$ additional orthonormal columns to $\bU$, orthogonal to the $r$ eigenvectors of $\bA \bA^\top$ (similar to the silent columns in the QR decomposition; see Section~\ref{section:silentcolu_qrdecomp}), we obtain an orthogonal matrix $\bU \in \real^{m \times m}$.
The same principle applies to the columns of $\bV$, yielding the \textit{full SVD}. 
A comparison between the reduced and full SVD is shown in Figure~\ref{fig:svd-comparison}, where white entries represent zero, and \textcolor{mylightbluetext}{blue} entries are not necessarily zero.
\begin{figure}[h]
\centering
\vspace{-0.35cm}
\subfigtopskip=2pt
\subfigbottomskip=2pt
\subfigcapskip=-5pt
\subfigure[Reduced SVD decomposition.]{\label{fig:svdhalf}
\includegraphics[width=0.47\linewidth]{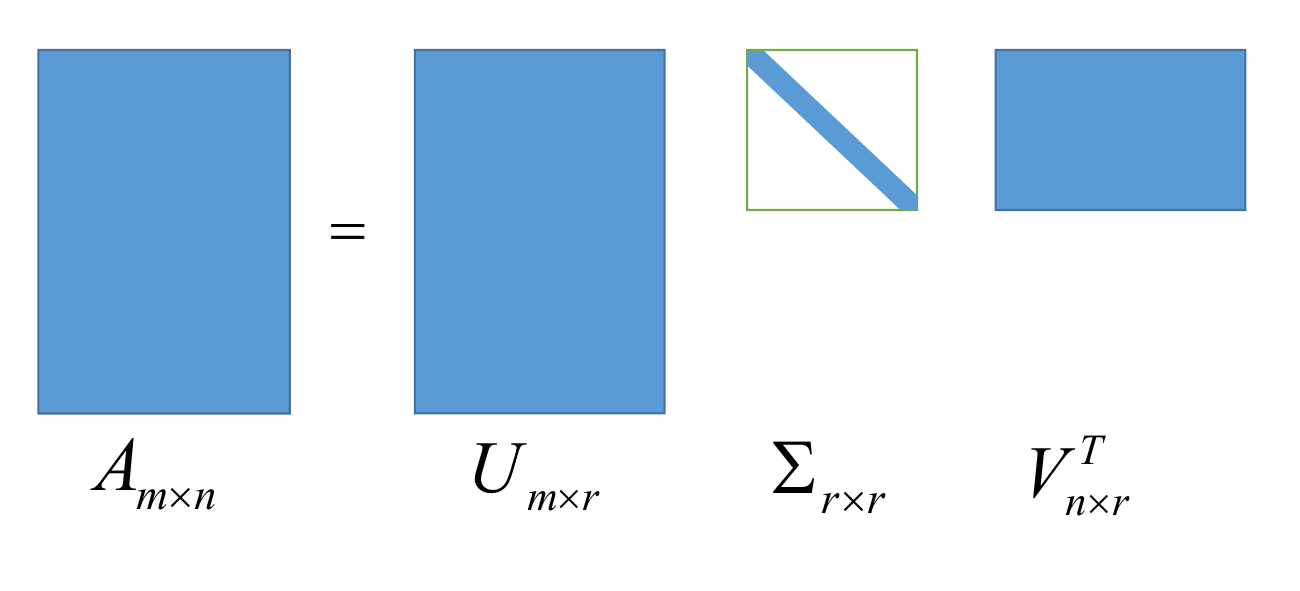}}
\quad 
\subfigure[Full SVD decomposition.]{\label{fig:svdall}
\includegraphics[width=0.47\linewidth]{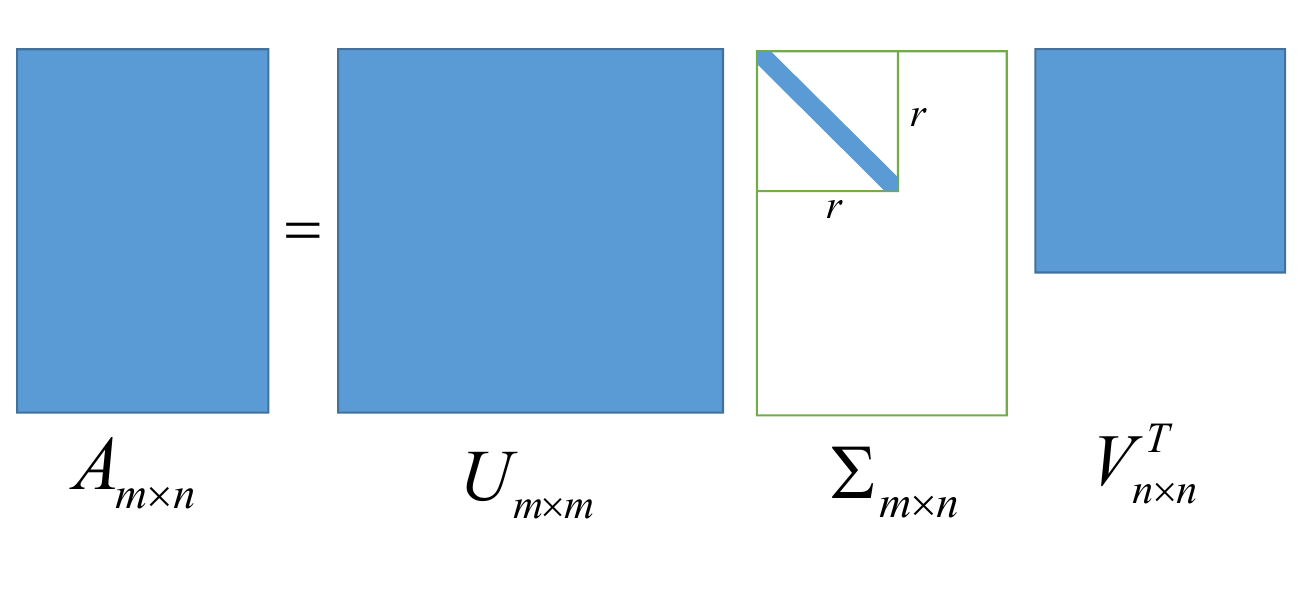}}
\caption{Comparison between the reduced and full SVD.}
\label{fig:svd-comparison}
\end{figure}

\section{Existence of the SVD}
To prove the existence of the SVD, we need to use the following lemmas. 
As previously mentioned, the singular values of a matrix $\bA$ are defined as the square roots of the eigenvalues of $\bA^\top \bA$. Since negative values do not have real square roots,  it is essential that the eigenvalues of  $\bA^\top \bA$ must be nonnegative.
\begin{lemma}[Nonnegative eigenvalues of $\bA^\top \bA$]\label{lemma:nonneg-eigen-ata}
For any matrix $\bA\in \real^{m\times n}$, the matrix  $\bA^\top \bA$ has nonnegative eigenvalues.
\end{lemma}
\begin{proof}[of Lemma~\ref{lemma:nonneg-eigen-ata}]
Let $\lambda$ be an eigenvalue of $\bA^\top \bA$ with the corresponding eigenvector $\bx$. We have
$
\bA^\top \bA \bx = \lambda \bx \implies \bx^\top \bA^\top \bA \bx = \lambda \bx^\top\bx. 
$
Since $\bx^\top \bA^\top \bA \bx  = \norm{\bA \bx}^2 \geq 0$ and $\bx^\top\bx > 0$, we have $\lambda \geq 0$.
\end{proof}

Since $\bA^\top\bA$ has nonnegative eigenvalues, we  can define the \textit{singular value} $\sigma\geq 0$ of $\bA$, such that $\sigma^2$ is the eigenvalue of $\bA^\top\bA$. In other words, \fbox{$\bA^\top\bA \bv = \sigma^2 \bv$}. This is a key condition for the existence of the SVD.

We also showed in Lemma~\ref{lemma:rankAB}  that $\rank$($\bA\bB$)$\leq$min$\{\rank$($\bA$), $\rank$($\bB$)\}.
However, the symmetric matrix $\bA^\top \bA$ is special in that its rank is equal to the rank of $\bA$. We now prove this result.
\begin{lemma}[Rank of $\bA^\top \bA$]\label{lemma:rank-of-ata}
The matrices $\bA^\top \bA$ and $\bA$ have same rank.
Extending this observation to $\bA^\top$, we can also prove that $\bA\bA^\top$ and $\bA$ share the same rank. 
\end{lemma}
\begin{proof}[of Lemma~\ref{lemma:rank-of-ata}]
Let $\bx\in \nspace(\bA)$. Then
$
\bA\bx  = \bzero \implies \bA^\top\bA \bx =\bzero, 
$
i.e., $\bx\in \nspace(\bA) \implies \bx \in \nspace(\bA^\top \bA)$. This shows that $\nspace(\bA) \subseteq \nspace(\bA^\top\bA)$. 

Next, let $\bx \in \nspace(\bA^\top\bA)$. We obtain
$
\bA^\top \bA\bx = \bzero\implies \bx^\top \bA^\top \bA\bx = 0\implies \norm{\bA\bx}^2 = 0 \implies \bA\bx=\bzero.
$
Thus, $\bx\in \nspace(\bA^\top \bA)$ implies $\bx\in \nspace(\bA)$, and therefore, $\nspace(\bA^\top\bA) \subseteq\nspace(\bA) $. 

By combining both inclusions, we conclude that:
$$\nspace(\bA) = \nspace(\bA^\top\bA) \qquad
\text{and} \qquad 
\dim(\nspace(\bA)) = \dim(\nspace(\bA^\top\bA)).
$$
By the fundamental theorem of linear algebra (Theorem~\ref{theorem:fundamental-linear-algebra}), it follows that $\bA^\top \bA$ and $\bA$ have the same rank.

Applying the observation to $\bA^\top$, we can also conclude that $\bA \bA^\top$ and $\bA$ have the same rank:
$
\rank(\bA) = \rank(\bA^\top \bA) = \rank(\bA\bA^\top).
$
\end{proof}

In the SVD, we claim that the matrix $\bA$ is a sum of $r$ rank-one matrices, where $r$ denotes the number of nonzero singular values. This count of nonzero singular values is, in fact, equal to the rank of the matrix.

\begin{lemma}[The number of nonzero singular values vs the rank]\label{lemma:rank-equal-singular}
The number of nonzero singular values of a matrix $\bA$ is equal to its rank.
\end{lemma}
\begin{proof}[of Lemma~\ref{lemma:rank-equal-singular}]
By Lemma~\ref{lemma:rank-of-symmetric}, the rank of any symmetric matrix (such as $\bA^\top \bA$) is equal to the number of nonzero eigenvalues (counting multiplicities). Thus, the number of nonzero singular values of $\bA$ equals the rank of $\bA^\top \bA$. By Lemma~\ref{lemma:rank-of-ata}, the number of nonzero singular values is therefore also equal to the rank of $\bA$.
\end{proof}

We are now ready to prove the existence of the SVD.
\begin{proof}[{of Theorem~\ref{theorem:reduced_svd_rectangular}: Existence of the reduced SVD}]
Since $\bA^\top \bA$ is a symmetric matrix, by the spectral theorem~\ref{theorem:spectral_theorem} and Lemma~\ref{lemma:nonneg-eigen-ata}, there exists a semi-orthogonal matrix $\bV\in\real^{n\times r}$ such that
$
{\bA^\top \bA = \bV \bSigma^2 \bV^\top},
$
where $\bSigma$ is a diagonal matrix containing the $r$ nonzero singular values of $\bA$, i.e., $\bSigma^2$ contains the corresponding nonzero eigenvalues of $\bA^\top \bA$.
Specifically, $\bSigma=\diag(\sigma_1, \sigma_2, \ldots, \sigma_r)$, and the set $\{\sigma_1^2, \sigma_2^2, \ldots, \sigma_r^2\}$ represents the nonzero eigenvalues of $\bA^\top \bA$, where $r=\rank(\bA)$. 
Now we proceed with the core of the proof. Starting from the equation \fbox{$\bA^\top\bA \bv_i = \sigma_i^2 \bv_i$}, $\forall\, i \in \{1, 2, \ldots, r\}$, i.e., the eigenvector $\bv_i$ of $\bA^\top\bA$ is corresponding to the eigenvalue $\sigma_i^2$:

1. Multiplying both sides by $\bv_i^\top$:
$$
\bv_i^\top\bA^\top\bA \bv_i = \sigma_i^2 \bv_i^\top \bv_i \leadto \norm{\bA\bv_i}^2 = \sigma_i^2 \leadto \norm{\bA\bv_i}=\sigma_i.
$$

2. Multiplying both sides by $\bA$:
$$
\bA\bA^\top\bA \bv_i = \sigma_i^2 \bA \bv_i \leadto \bA\bA^\top \frac{\bA \bv_i }{\sigma_i}= \sigma_i^2 \frac{\bA \bv_i }{\sigma_i} \leadto \bA\bA^\top \bu_i = \sigma_i^2 \bu_i,
$$
where we notice that this form can find the eigenvector of $\bA\bA^\top$ corresponding to $\sigma_i^2$, which is $\bA \bv_i$. Since the length of $\bA \bv_i$ is $\sigma_i$, we then define $\bu_i = \frac{\bA \bv_i }{\sigma_i}$ with a unit norm.

These vectors $\bu_i$ are mutually orthonormal because $(\bA\bv_i)^\top(\bA\bv_j)=\bv_i^\top\bA^\top\bA\bv_j=\sigma_j^2 \bv_i^\top\bv_j=0$  if $i\neq j$. 
Thus, we conclude that
$
{\bA \bA^\top = \bU \bSigma^2 \bU^\top},
$
where $\bU=[\bu_1, \bu_2, \ldots,\bu_r]$.
Since {$\bA\bv_i = \sigma_i\bu_i$}, we have 
\begin{equation}\label{equation:svd_eqprof1}
[\bA\bv_1, \bA\bv_2, \ldots, \bA\bv_r] = [ \sigma_1\bu_1,  \sigma_2\bu_2, \ldots,  \sigma_r\bu_r]\leadto
\bA\bV = \bU\bSigma.
\end{equation}
At this point, since $\bV \bV^\top \neq \bI$, we cannot directly obtain the reduced SVD.
However, by appending $\bV$ with additional orthogonal columns, we can construct an orthogonal matrix $\widetilde{\bV} = [\bV, \bV_2]$, and similarly append $\bU$ with orthonormal columns to form $\widetilde{\bU} = [\bU, \bU_2]$. This leads to the full SVD (since $\widetilde{\bV}\widetilde{\bV}^\top=\bI$):
$$
\bA \widetilde{\bV}= \widetilde{\bU} \widetilde{\bSigma} ,\gap \text{where}\gap 
\widetilde{\bSigma} 
= \begin{bmatrix}
	\bSigma & \bzero  \\
	\bzero &\bzero
\end{bmatrix}
\leadto
\bA = \widetilde{\bU} \widetilde{\bSigma}\widetilde{\bV}^\top.
$$
Finally, simplifying the product, we have
$
\bA =\bU \bSigma \bV^\top + \bU_2 \cdot \bzero \cdot \bV_2^\top = \bU \bSigma \bV^\top,
$
which is the reduced SVD. This completes the proof.
\end{proof}

The proof also shows that if $\bA=\bU\bSigma\bV^\top$ is the reduced SVD of $\bA$, it follows from \eqref{equation:svd_eqprof1} that $\bA\bV\bV^\top=\bA$. This implies that $\bV\bV^\top$ (where $\bV\in\real^{n\times r}$ is semi-orthogonal) is an (orthogonal) projection matrix that maps each row of $\bA$ onto itself (a projection matrix onto the row space of $\bA$).

\index{Row space}
\index{Subspace}
\index{Orthogonal projection}
\paragraph{SVD-related orthogonal projections.}
In the context of the SVD, several important orthogonal projections arise from the four fundamental subspaces. In simple terms, an orthogonal projection matrix has two key properties: it is symmetric and idempotent; see Sections~\ref{section:qr-gram-compute} and \ref{section:spec_app_eigproj}. Such a projection matrix projects any vector onto its column space.
Idempotency means that applying the projection twice is the same as applying it once. Symmetry has a geometric interpretation: the projection minimizes the distance between the original vector and its projection, where the projection lies within the column space of the projection matrix.
Now suppose $\bA = \bU \bSigma \bV^\top$ is the full SVD of $\bA$ with rank $r$. Consider the following column partitions:
\[
\begin{blockarray}{ccc}
	\begin{block}{c[cc]}
		\bU=&	\bU_r & \bU_m   \\
	\end{block}
	&m\times r & m\times (m-r)   \\
\end{blockarray}
,\qquad
\begin{blockarray}{ccc}
	\begin{block}{c[cc]}
		\bV=	&	\bV_r & \bV_n   \\
	\end{block}
	&n\times r & n\times (n-r)   \\
\end{blockarray},
\]
where $\bU_r$ and $\bV_r$ consist of the first $r$ columns of $\bU$ and $\bV$, respectively.
The four orthogonal projections can then be written as:
$$
\begin{aligned}
	\bV_r\bV_r^\top &= \text{projection onto $\cspace(\bA^\top)$},
	\quad
	&\bV_n\bV_n^\top &=\text{projection onto $\nspace(\bA)$},\\
	\bU_r\bU_r^\top &= \text{projection onto $\cspace(\bA)$},
	\quad
	&\bU_m\bU_m^\top &= \text{projection onto $\nspace(\bA^\top)$}.\\
\end{aligned}
$$
These projection matrices allow us to cleanly map vectors onto the four fundamental subspaces of $\bA$: its column space, row space, null space, and left null space.

\paragraph{Spectral decomposition of $\bA \bA^\top$.}
An additional result from the above proof is that the spectral decomposition of $\bA^\top \bA = \bV \bSigma^2 \bV^\top$ naturally leads to the spectral decomposition of $\bA \bA^\top = \bU \bSigma^2 \bU^\top$, with the same eigenvalues.

\begin{corollary}[Eigenvalues of $\bA^\top\bA$ and $\bA\bA^\top$]
The nonzero eigenvalues of $\bA^\top\bA$ and $\bA\bA^\top$ are identical.
\end{corollary}
We have shown in Lemma~\ref{lemma:nonneg-eigen-ata} that the eigenvalues of $\bA^\top \bA$ are nonnegative. Therefore, the eigenvalues of $\bA\bA^\top$ must also be nonnegative.
\begin{corollary}[Nonnegative eigenvalues of $\bA^\top\bA$ and $\bA\bA^\top$]
The eigenvalues of both  $\bA^\top\bA$ and $\bA\bA^\top$ are nonnegative.
\end{corollary}

Extending Lemma~\ref{lemma:rank-equal-singular}, 
the existence of the SVD is also crucial for defining the effective rank of a matrix.
\begin{definition}[Effective rank vs exact rank]\label{definition:effective-rank-in-svd}
The	\textit{effective rank},  also known as the \textit{numerical rank},  is defined as follows:
From Lemma~\ref{lemma:rank-equal-singular}, we know that the number of nonzero singular values of a matrix is equal to its rank. 
Let the $i$-th largest singular value of $\bA$ be denoted as $\sigma_i(\bA)$. If $\sigma_r(\bA) \gg \sigma_{r+1}(\bA) \approx 0$, then $r$ is called the numerical rank of $\bA$. In contrast, if $\sigma_i(\bA) > \sigma_{r+1}(\bA) = 0$, we say that $\bA$ has \textit{exact rank} $r$, as discussed in most of our previous examples.
\end{definition}

\section{Properties of the SVD}\label{section:property-svd}
\subsection{Four Subspaces in SVD}\label{section:four-space-svd}
For any matrix $\bA\in  \real^{m\times n}$, the following properties hold:
\begin{itemize}
\item The null space  $\nspace(\bA)$ is the orthogonal complement of the row space $\cspace(\bA^\top)$ in $\real^n$: $\dim(\nspace(\bA))+\dim(\cspace(\bA^\top))=n$.

\item The left null space  $\nspace(\bA^\top)$ is the orthogonal complement of the column space $\cspace(\bA)$ in $\real^m$: $\dim(\nspace(\bA^\top))+\dim(\cspace(\bA))=m$.
\end{itemize}
This result is known as the fundamental theorem of linear algebra, also referred to as the rank-nullity theorem (Theorem~\ref{theorem:fundamental-linear-algebra}).
Using the SVD, we can identify an orthonormal basis for each of these subspaces.
\begin{figure}[h!]
\centering
\vspace{-0.35cm}
\subfigtopskip=2pt
\subfigbottomskip=2pt
\subfigcapskip=-5pt
\includegraphics[width=0.93\textwidth]{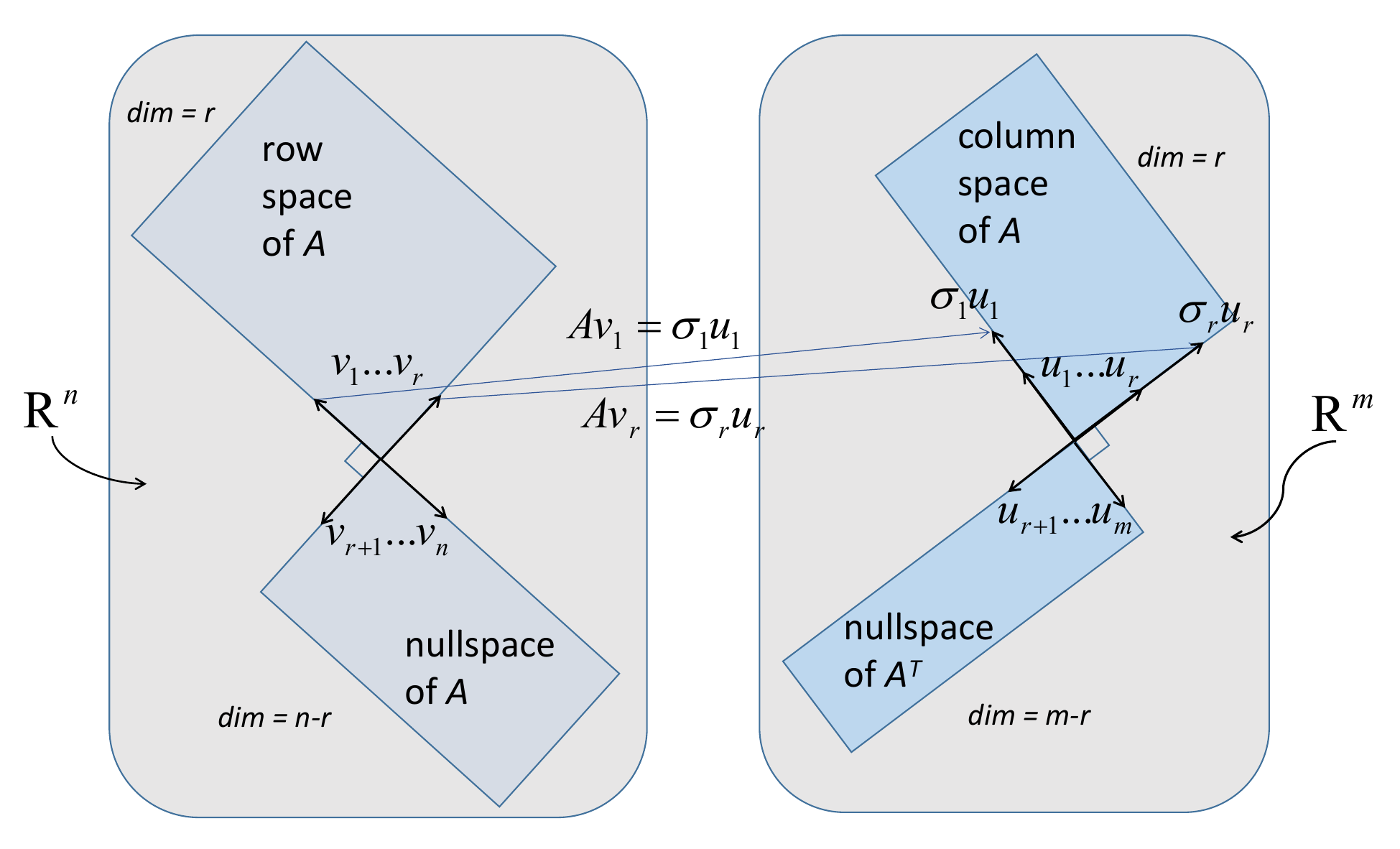}
\caption{Orthonormal bases that diagonalize $\bA$ using the SVD.}
\label{fig:lafundamental3-SVD}
\end{figure}

\index{Orthonormal basis}
\index{Basis}
\begin{proposition}[Four orthonormal bases]\label{proposition:svd-four-orthonormal-Basis}
Given the full SVD of a matrix $\bA = \bU \bSigma \bV^\top$, where $\bU=[\bu_1, \bu_2, \ldots,\bu_m]$ and $\bV=[\bv_1, \bv_2, \ldots, \bv_n]$ are the column partitions of $\bU$ and $\bV$, respectively, the following properties hold:
\begin{itemize}
\item $\{\bv_1, \bv_2, \ldots, \bv_r\} $ is an orthonormal basis of $\cspace(\bA^\top)$;

\item $\{\bv_{r+1},\bv_{r+2}, \ldots, \bv_n\}$ is an orthonormal basis of $\nspace(\bA)$;

\item $\{\bu_1,\bu_2, \ldots,\bu_r\}$ is an orthonormal basis of $\cspace(\bA)$;

\item $\{\bu_{r+1}, \bu_{r+2},\ldots,\bu_m\}$ is an orthonormal basis of $\nspace(\bA^\top)$. 
\end{itemize}
The relationship among these four subspaces is shown in Figure~\ref{fig:lafundamental3-SVD}, where $\bA$ maps the row basis $\bv_i$ to the column basis $\bu_i$  via the equation  $\sigma_i\bu_i=\bA\bv_i$ for all $i\in \{1, 2, \ldots, r\}$.
\end{proposition}
\begin{proof}[of Proposition~\ref{proposition:svd-four-orthonormal-Basis}]
From Lemma~\ref{lemma:rank-of-symmetric}, for the symmetric matrix $\bA^\top\bA$, the subspace $\cspace(\bA^\top\bA)$ is spanned by the eigenvectors, thus $\{\bv_1,\bv_2, \ldots, \bv_r\}$ forms an orthonormal basis for $\cspace(\bA^\top\bA)$.
We proceed as follows:
\begin{enumerate}
\item Since $\bA^\top\bA$ is symmetric,  the row space of $\bA^\top\bA$ is equal to its  column space. 

\item All rows of $\bA^\top\bA$ are  linear combinations of the rows of $\bA$, meaning the row space of $\bA^\top\bA$ $\subseteq$ the row space of $\bA$, i.e., $\cspace(\bA^\top\bA) \subseteq \cspace(\bA^\top)$. 

\item Since $\rank(\bA^\top\bA) = \rank(\bA)$ by Lemma~\ref{lemma:rank-of-ata}, we then have:

The row space of $\bA^\top\bA$ = the column space of $\bA^\top\bA$ =  the row space of $\bA$, i.e., $\cspace(\bA^\top\bA) = \cspace(\bA^\top)$. Thus, $\{\bv_1, \bv_2,\ldots, \bv_r\}$ is an orthonormal basis for $\cspace(\bA^\top)$. 
\end{enumerate}

Moreover, the space spanned by $\{\bv_{r+1}, \bv_{r+2},\ldots, \bv_n\}$ is the orthogonal complement to the space spanned by $\{\bv_1,\bv_2, \ldots, \bv_r\}$. 
Hence, $\{\bv_{r+1},\bv_{r+2}, \ldots, \bv_n\}$ forms an orthonormal basis for $\nspace(\bA)$. 

Applying a similar argument to $\bA \bA^\top$ proves the remaining claims. Alternatively, we can observe that $\{\bu_1, \bu_2, \ldots, \bu_r\}$ forms a basis for the column space of $\bA$ by Lemma~\ref{lemma:column-basis-from-row-basis}~\footnote{For any matrix $\bA$, if $\{\br_1, \br_2, \ldots, \br_r\}$ forms a basis for the row space, then $\{\bA \br_1, \bA \br_2, \ldots, \bA \br_r\}$ forms a basis for the column space of $\bA$.}, since $\bu_i = \frac{\bA \bv_i}{\sigma_i}$ for all $i \in \{1, 2, \ldots, r\}$. 
\end{proof}

\subsection{Relationship between Singular Values and Determinant}
Let $\bA\in \real^{n\times n}$ be a square matrix, and let its SVD be given by $\bA = \bU\bSigma\bV^\top$. It follows that 
$$
\abs{\det(\bA)} = \abs{\det(\bU\bSigma\bV^\top)} = \abs{\det(\bSigma)} = \sigma_1 \sigma_2\ldots \sigma_n.
$$
If all the singular values  are nonzero, then $\det(\bA)\neq 0$. That is, $\bA$ is \textbf{nonsingular}. 
If at least one singular value is zero, say $\sigma_i = 0$, then $\det(\bA) = 0$, implying that $\bA$ does not have full rank and is not invertible. 
In this case, $\bA$ is called \textbf{singular}. This explains why the values $\sigma_i$ are referred to as the \textit{singular values} of $\bA$.

\index{Orthogonally equivalence}
\subsection{Orthogonally Equivalence}
We have defined in Definition~\ref{definition:similar-matrices} that $\bA$ and $\bP\bA\bP^{-1}$ are similar matrices for any nonsingular matrix $\bP$. The concept of \textit{orthogonally equivalence} is defined in a similar way.

\begin{definition}[Orthogonally equivalent matrices]
Given two orthogonal matrices $\bU$ and $\bV$, the matrices $\bA$ and $\bU\bA\bV$ are called \textit{orthogonally equivalent matrices}. 
In the complex domain, when $\bU$ and $\bV$ are unitary matrices, the matrices are called  \textit{unitarily equivalent}.
\end{definition}
We now state the following property for orthogonally equivalent matrices:
\begin{lemma}[Orthogonally equivalent matrices]\label{lemma:orthogonal-equivalent-matrix}
If matrices $\bA$ and $\bB$ are orthogonally equivalent, then they have the same singular values.
\end{lemma}
\begin{proof}[of Lemma~\ref{lemma:orthogonal-equivalent-matrix}]
Since $\bA$ and $\bB$ are orthogonally equivalent, there exist orthogonal matrices $\bU$ and $\bV$ such that $\bB = \bU\bA\bV$. We then have 
$
\bB\bB^\top = (\bU\bA\bV)(\bV^\top\bA^\top\bU^\top) = \bU\bA\bA^\top\bU^\top.
$
This implies $\bB\bB^\top$ and $\bA\bA^\top$ are similar matrices. By Proposition~\ref{proposition:eigenvalue-similar-matrices}, the eigenvalues of similar matrices are the same, which implies that the singular values of $\bA$ and $\bB$ are the same.
\end{proof}

\subsection{SVD for QR}
\begin{lemma}[SVD for QR]\label{lemma:svd-for-qr}
Suppose the full QR decomposition of  a matrix $\bA\in \real^{m\times n}$, with $m\geq n$, is given by $\bA=\bQ\bR$, where $\bQ\in \real^{m\times m}$ is orthogonal and $\bR\in \real^{m\times n}$ is upper triangular. Then, $\bA$ and $\bR$ have the same singular values and right singular vectors.
\end{lemma}
\begin{proof}[of Lemma~\ref{lemma:svd-for-qr}]
We observe that $\bA^\top\bA = \bR^\top\bR$ such that $\bA^\top\bA$ and $\bR^\top\bR$ have the same eigenvalues and eigenvectors. Consequently, $\bA$ and $\bR$ have the same singular values and right singular vectors (i.e., the eigenvectors of $\bA^\top\bA$ or $\bR^\top\bR$).
\end{proof}

The above lemma implies that the SVD of a matrix can be derived from its QR decomposition. Suppose the QR decomposition of $\bA$ is given by $\bA=\bQ\bR$, and the SVD of $\bR$ is given by $\bR=\bU_0 \bSigma\bV^\top$. Then, the SVD of $\bA$ can be expressed as:
$
\bA =\underbrace{ \bQ\bU_0}_{\bU} \bSigma\bV^\top.
$

\section{Polar Decomposition}
A decomposition closely related to the SVD is the \textit{polar form} or \textit{polar decomposition} of a matrix. 
In the context of continuum mechanics, it is imperative to distinguish between stretching and rotation.
The polar decomposition factors any matrix into an orthogonal matrix (which  corresponds to a rotation or reflection) and a symmetric PSD matrix (which  corresponds to stretching or compression, see Section~\ref{section:coordinate-transformation}).

\begin{theoremHigh}[Polar decomposition]\label{theorem:polar-decomposition}
Let $\bA\in\real^{m\times n}$. Then $\bA$ can be factored as 
\begin{itemize}
\item \textbf{Case $m>n$: left polar decomposition.} $\bA=\bQ_l\bS_l$, where $\bS_l^2=\bA^\top\bA$ is PSD and is \textbf{uniquely} determined. 
The factor $\bQ_l$ has orthonormal columns, and it is \textbf{uniquely} determined if $\rank(\bA)=n$.
\item \textbf{Case $m<n$: right polar decomposition.} $\bA=\bS_r\bQ_r$, where $\bS_r^2=\bA\bA^\top$ is PSD and is \textbf{uniquely} determined.
The factor $\bQ_r$ has orthonormal rows, and it is \textbf{uniquely} determined if $\rank(\bA)=m$.
\item \textbf{Case $m=n$: left/right polar decomposition.} $\bA=\bQ\bS_l=\bS_r\bQ$, where $\bS_l^2=\bA^\top\bA$ and $\bS_r^2=\bA\bA^\top$ are PSD and are \textbf{uniquely} determined. The factor $\bQ$ is orthonoal, and it is the same for both the left and right polar decompositions. $\bQ$ is \textbf{uniquely} determined if $\bA$ is nonsingular (i.e., $\rank(\bA)=n$).
\end{itemize}
\noindent
Note in all cases, the PSD factors ($\bS_l$ or $\bS_r$) are uniquely determined, and become PD if $\bA$ has full rank  (full row  or column rank). The semi-orthogonal factors $\bQ_l, \bQ_r$, and $\bQ$ are uniquely determined only when $\bA$ has full rank.~\footnote{When $\bA$ is complex, then the orthogonal (resp., semi-orthogonal) matrices become unitary (resp., semi-unitary) matrices, and the PSD matrices become  complex Hermitian and PSD matrices.}
\end{theoremHigh}
\begin{proof}[of Theorem~\ref{theorem:polar-decomposition}]
Let the SVD of $\bA$ be $\bA=\bU\bSigma\bV^\top={(\bU\bV^\top)}{(\bV\bSigma\bV^\top)}=\bQ_l\bS_l$ such that $\bS_l^2=\bV\bSigma^2\bV^\top=\bA^\top\bA$. Since $\bA^\top\bA$ is PSD, $\bS_l$ is uniquely determined (Theorem~\ref{theorem:unique-factor-pd}). 
If further $\rank(\bA)=n$, i.e., $\bA$ has full (column) rank, $\bA^\top\bA$ is PD and $\bS_l$ has full rank (Theorem~\ref{theorem:unique-factor-pd}) such that $\bQ_l=\bA\bS_l^{-1}$, implying $\bQ_l$ is uniquely determined.

The second case can be similarly proved such that $\bA=\bU\bSigma\bV^\top=(\bU\bSigma\bU^\top)(\bU\bV^\top)=\bS_r\bQ_r$.
Since $\bA\bA^\top$ is PSD, $\bS_r$ is uniquely determined. If further $\rank(\bA)=m$, $\bA\bA^\top$ is PD and $\bS_r$ has full rank such that $\bQ_r=\bS_r^{-1}\bA$ is uniquely determined.

The third case is a combination of the previous two. This completes the proof.
\end{proof}
\begin{exercise}[Trace of PSD in polar decomposition]\label{exercise:trace_polar}
Show that the trace of the PSD matrices in the polar decomposition, $\trace(\bS_r)$ or $\trace(\bS_l)$, is equal to the sum of the singular values of $\bA$.
\end{exercise}
\begin{exercise}[Normal from polar]
Let $\bA\in\real^{n\times n}$ be nonsingular, and suppose it admits the polar decomposition $\bA=\bS_r\bQ$, where $\bS_r$ is PD and $\bQ$ is orthogonal. Show that $\bA$ is normal if and only if $\bS_r\bQ=\bQ\bS_r$.
\end{exercise}
\begin{exercise}
Let $\bA,\bB\in\real^{n\times n}$ be orthogonal, and let $\bA+\bB$ be nonsingular. Show that the orthogonal factor in the polar decomposition of $\bA+\bB$ is $\bA(\bA^\top\bB)^{1/2}$. 
\end{exercise}

\section{Coordinate Transformation in Matrix Decomposition}\label{section:coordinate-transformation}

\index{Matrix multiplication}
\index{Coordinate transformation}

Consider a vector $\bv\in \real^3$ with elements $\bv = [3 ,7, 2]^\top$. 
It is essential to clarify the significance of these values: In the Cartesian coordinate system, they represent a component of 3 along the $x$-axis, a component of 7 along the $y$-axis, and a component of 2 along the $z$-axis.
These scalar values are the \textit{coordinates} of $\bv$ with respect to the basis of the Cartesian system.
Matrix multiplication, on the other hand, gains significance when applied in high-dimensional spaces.

\paragraph{Coordinate defined by a nonsingular matrix.} Suppose we have a $3\times 3$ nonsingular matrix $\bB$, which  is invertible and possesses linearly independent columns. 
Consequently, the three columns of $\bB$ collectively form a basis for the $\real^{3}$ space. 
Taking a step further, the three columns of $\bB$ can serve as the basis for a \textcolor{black}{\textbf{new coordinate system}}, referred to as the \textcolor{black}{\textbf{$B$ coordinate system}}.

Returning to the Cartesian coordinate system, we also have a set of three vectors forming a basis, denoted by $\{\be_1, \be_2, \be_3\}$. 
If we arrange the three vectors as columns in a matrix, this matrix will be the identity matrix. 
Therefore, when we multiply a vector $\bv$ by the identity matrix, denoted by $\bI\bv$, we are essentially performing a coordinate transformation that leaves $\bv$ in the same coordinate system.
In other words, $\bI\bv = \bv$ means \textcolor{black}{\textbf{transferring $\bv$ from the Cartesian coordinate system into the Cartesian coordinate system}}, the same coordinate.

Similarly, when we multiply a vector $\bv$ by the matrix $\bB$, denoted by $\bB\bv$, \textcolor{black}{\textbf{we are transforming $\bv$ from the Cartesian coordinate system into the $B$ coordinate system}}. 
To illustrate this with a specific example, consider $\bv = [3, 7, 2]^\top$ and $\bB=[\bb_1, \bb_2, \bb_3]$. In this case, we have $\bu=\bB\bv = 3\bb_1+7\bb_2+2\bb_3$, i.e., vector $\bu$ contains 3 units of the first basis $\bb_1$ of $\bB$, 7 units of the second basis $\bb_2$ of $\bB$, and 2 units of the third basis $\bb_3$ of $\bB$. 
Now, if we wish to transform the vector $\bu$ from the $B$ coordinate system back to the Cartesian coordinate system, we can achieve this by multiplying $\bu$ by the inverse of $\bB$, denoted by $\bB^{-1}$. This operation results in $\bB^{-1}\bu = \bv$.

\index{Geometric interpretation}
\paragraph{Coordinate defined by an orthogonal matrix.} A $3\times 3$ orthogonal matrix $\bQ$ defines a ``better" coordinate system since its three columns, forming the basis, are mutually orthonormal  (same as those in the Cartesian coordinate system). 
The operation $\bQ\bv$ facilitates the transition of $\bv$ from the Cartesian coordinate system to the one defined by the orthogonal matrix. 
Since the basis vectors from the orthogonal matrix exhibit orthonormality, just like the three vectors $\be_1, \be_2, \be_3$ in the Cartesian coordinate system, the transformation induced by the orthogonal matrix involves rotating or reflecting the Cartesian system.
To revert to the Cartesian coordinate system, one can utilize $\bQ^{-1}=\bQ^\top$.

\begin{figure}[h]
\centering
\includegraphics[width=0.99\textwidth]{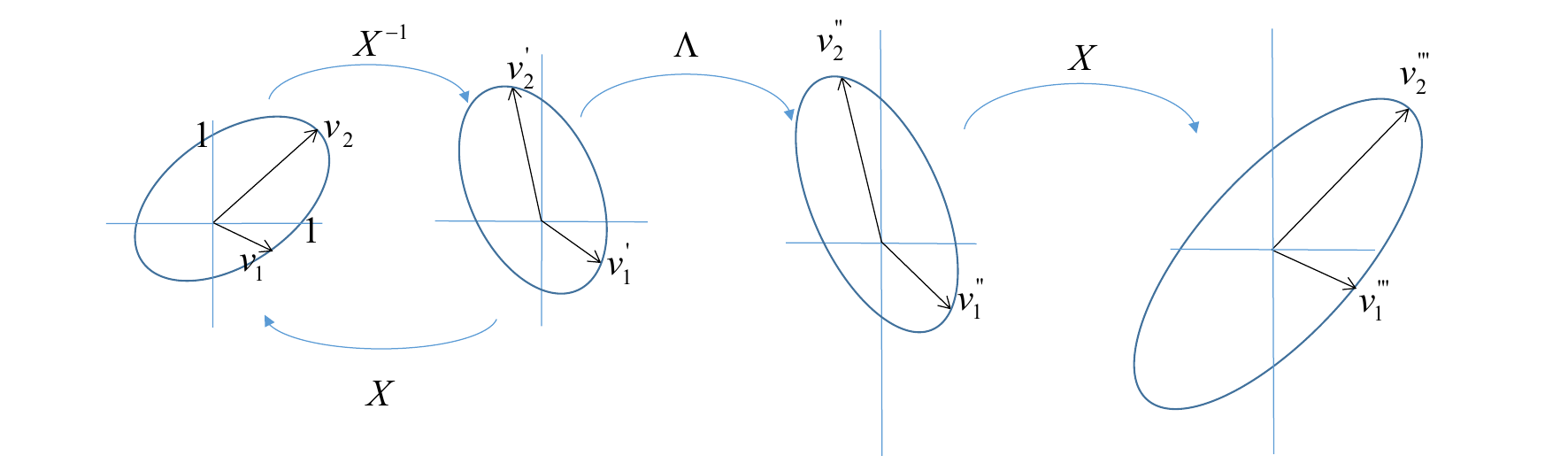}
\caption{Eigenvalue decomposition $\bA = \bX\bLambda\bX^{-1}$: $\bX^{-1}$ undergoes a transformation into a different coordinate system, followed by stretching with $\bLambda$, and then transforming back with $\bX$. 
$\bX^{-1}$ and $\bX$ are nonsingular, which will change the basis of the system, and the angle between the vectors $\bv_1$ and $\bv_2$ will \textbf{not} be preserved. In other words, the angle between $\bv_1$ and $\bv_2$ is \textbf{different} from the angle between $\bv_1^\prime$ and $\bv_2^\prime$. The lengths of $\bv_1$ and $\bv_2$ are also \textbf{not} preserved; that is, $\normtwo{\bv_1} \neq \normtwo{\bv_1^\prime}$ and $\normtwo{\bv_2} \neq \normtwo{\bv_2^\prime}$.}
\label{fig:eigen-rotate}
\end{figure}

\index{Anisotropic scaling}
\section*{Eigenvalue Decomposition}
A square matrix $\bA\in\real^{n\times n}$ with linearly independent eigenvectors can be factored as $\bA = \bX\bLambda\bX^{-1}$, where $\bX$ and $\bX^{-1}$ are nonsingular so that they define a system transformation inherently. 
The operation $\bA\bu = \bX\bLambda\bX^{-1}\bu$ firstly transfers $\bu$ into the coordinate system defined by $\bX^{-1}$, which we shall refer to as the \textit{eigen coordinate system}. 
Subsequently, the operation $\bLambda(\cdot)$  stretches each component of the vector in the eigen system by the length of the corresponding eigenvalue. 
Finally, $\bX$ facilitates the transformation of the resultant vector back to the Cartesian coordinate system. 
The overall result is an \textit{anisotropic} scaling in $n$ eigenvector directions.
A visual representation of the coordinate system transformation via eigenvalue decomposition is presented in Figure~\ref{fig:eigen-rotate}, where $\bv_1$ and $\bv_2$ are two linearly independent eigenvectors of $\bA$ such that they form a basis for $\real^2$.

\begin{figure}[h]
\centering
\includegraphics[width=0.99\textwidth]{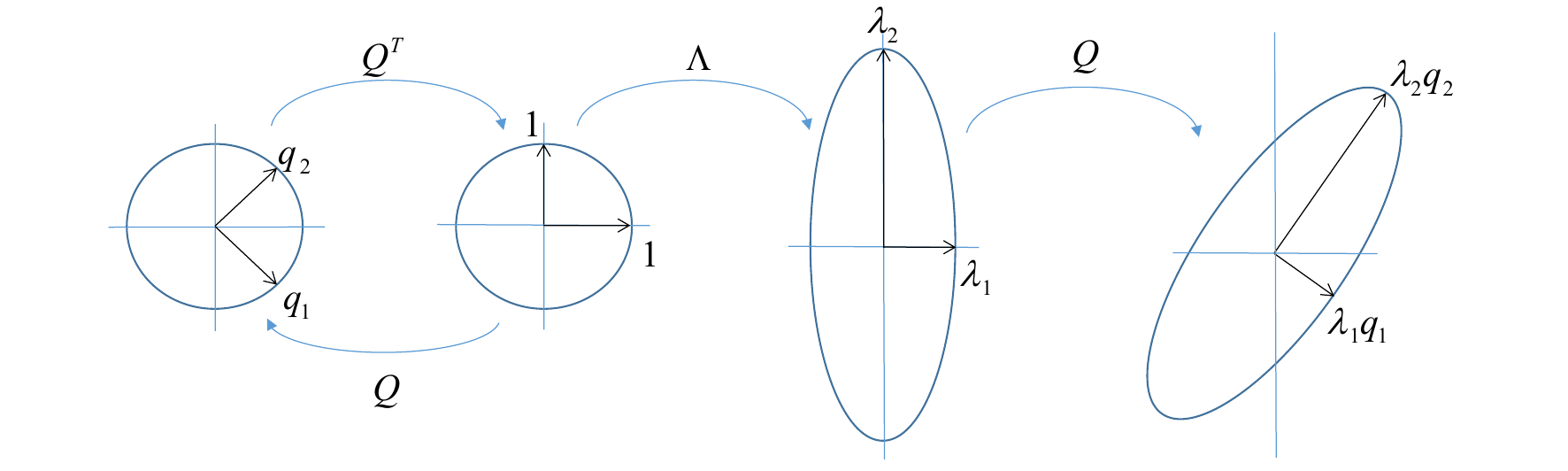}
\caption{Spectral decomposition $\bQ\bLambda \bQ^\top$: $\bQ^\top$ rotates or reflects, $\bLambda$ stretches the cycle to an ellipse, and $\bQ$ rotates or reflects back. Orthogonal matrices $\bQ^\top$ and $\bQ$ only change the basis of the system. However, they preserve both the angle between the vectors $\bq_1$ and $\bq_2$, and their lengths.}
\label{fig:spectral-rotate}
\end{figure}

\section*{Spectral Decomposition}
A symmetric matrix $\bA\in\real^{n\times n}$ can be decomposed as $\bA = \bQ\bLambda\bQ^\top$, where $\bQ$ and $\bQ^\top$ are orthogonal matrices so that they define a coordinate system transformation inherently as well. The operation $\bA\bu = \bQ\bLambda\bQ^\top\bu$ firstly rotates or reflects $\bu$ into the coordinate system defined by $\bQ^\top$, which we shall refer to as the \textit{spectral coordinate system}. 
The operation $\bLambda(\cdot)$ stretches each component of the vector in the spectral system by the length of the corresponding eigenvalue. 
Subsequently, $\bQ$ facilitates the rotation or reflection of the resultant vector back to the original coordinate system. 
Once again, the overall result is an {anisotropic} scaling in $n$ eigenvector directions.
A demonstration of how the spectral decomposition transforms between coordinate systems in the $\real^2$ space is shown in Figure~\ref{fig:spectral-rotate}, where $\bq_1$ and $\bq_2$ represent two linearly independent eigenvectors of $\bA$ such that they form a basis for $\real^2$. The coordinate transformation in the spectral decomposition is similar to that in the eigenvalue decomposition, with the distinction that in the spectral decomposition,  orthogonal vectors transformed by $\bQ^\top$ remain orthogonal. 
This is also a property of orthogonal matrices. That is, orthogonal matrices can be viewed as matrices, which change the basis of other matrices while preserving the angle (inner product) between  vectors:
$
\bu^\top \bv = (\bQ\bu)^\top(\bQ\bv).
$
The invariance of the angle between vectors also relies on the invariance of their lengths:
$
\norm{\bQ\bu}= \norm{\bu}.
$

\section*{SVD}
\begin{figure}[h]
\centering
\includegraphics[width=0.99\textwidth]{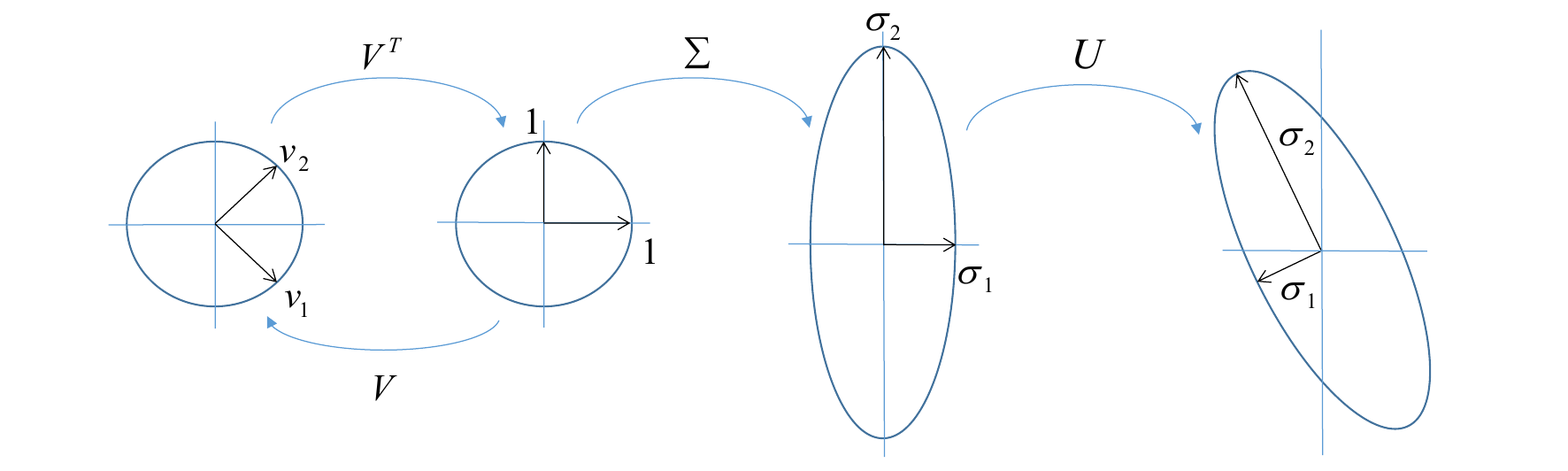}
\caption{SVD $\bA=\bU\bSigma\bV^\top$: $\bV^\top$ and $\bU$ rotate or reflect, $\bSigma$ stretches the circle to an ellipse. Orthogonal matrices $\bV^\top$ and $\bU$ only change the basis of the system. However, they preserve both the angle between the vectors $\bv_1$ and $\bv_2$, and their lengths.}
\label{fig:svd-rotate}
\end{figure}
\begin{figure}[h]
\centering
\includegraphics[width=0.99\textwidth]{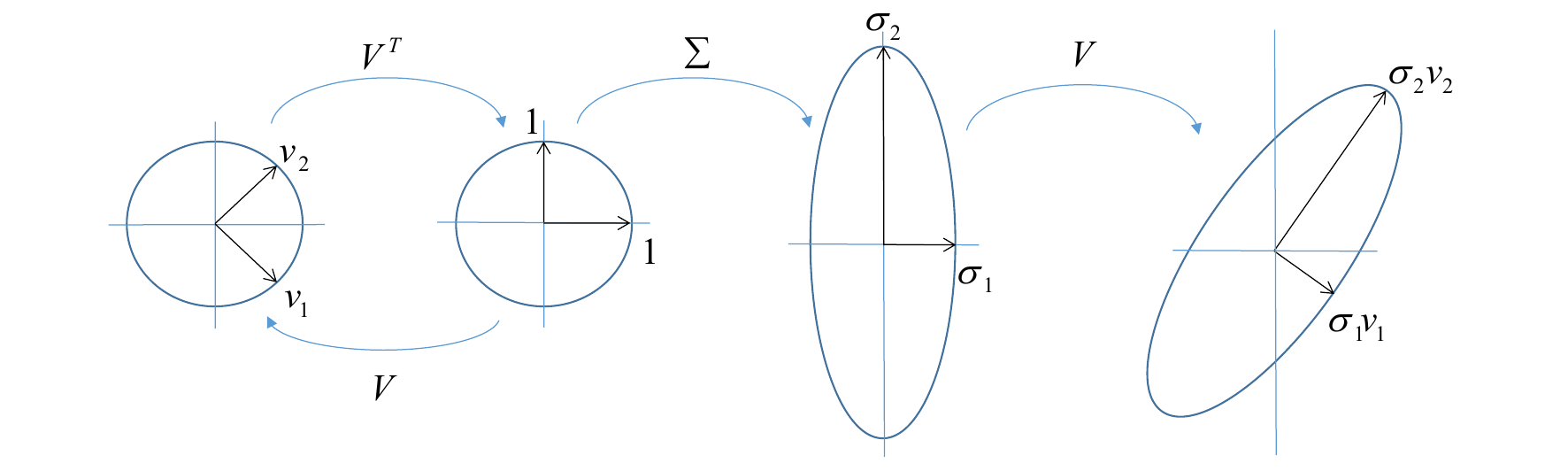}
\caption{$\bV\bSigma \bV^\top$ from SVD or polar decomposition: $\bV^\top$ rotates or reflects, $\bSigma$ stretches the cycle to an ellipse, and $\bV$ rotates or reflects back. Orthogonal matrices $\bV^\top$ and $\bV$ only change the basis of the system. However, they preserve both the angle between the vectors $\bv_1$ and $\bv_2$, and their lengths.}
\label{fig:polar-rotate}
\end{figure}
Any $m\times n$ matrix of rank $r$ can be factored as $\bA=\bU\bSigma\bV^\top$, which represents the SVD. The operation $\bA\bu=\bU\bSigma\bV^\top\bu$ then firstly rotates or reflects vector $\bu$ into the system defined by $\bV^\top$, which we refer to as the \textit{$V$ coordinate system}. $\bSigma$ stretches the first $r$ components of the resulting vector in the $V$ system by the lengths of the singular values. If $n\geq m$, then $\bSigma$ only keeps  $m-r$ additional components, which are scaled to zero, while removing the final $n-m$ components. If $m>n$, then $\bSigma$ scales $n-r$ components to zero and also adds  $m-n$  additional zero components. Finally, $\bU$ rotates or reflects the resulting vector into the \textit{$U$ coordinate system} defined by $\bU$. 
A visual demonstration of how the SVD transforms in a $2\times 2$ example is shown in Figure~\ref{fig:svd-rotate}. Further, Figure~\ref{fig:polar-rotate} demonstrates the transformation of $\bV\bSigma \bV^\top$ by a $2\times 2$ example. 
Similar to the spectral decomposition, orthogonal matrices $\bV^\top$ and $\bU$ only change the basis of the system but preserve the angle between vectors $\bv_1$ and $\bv_2$.

\begin{figure}[h]
\centering
\includegraphics[width=0.99\textwidth]{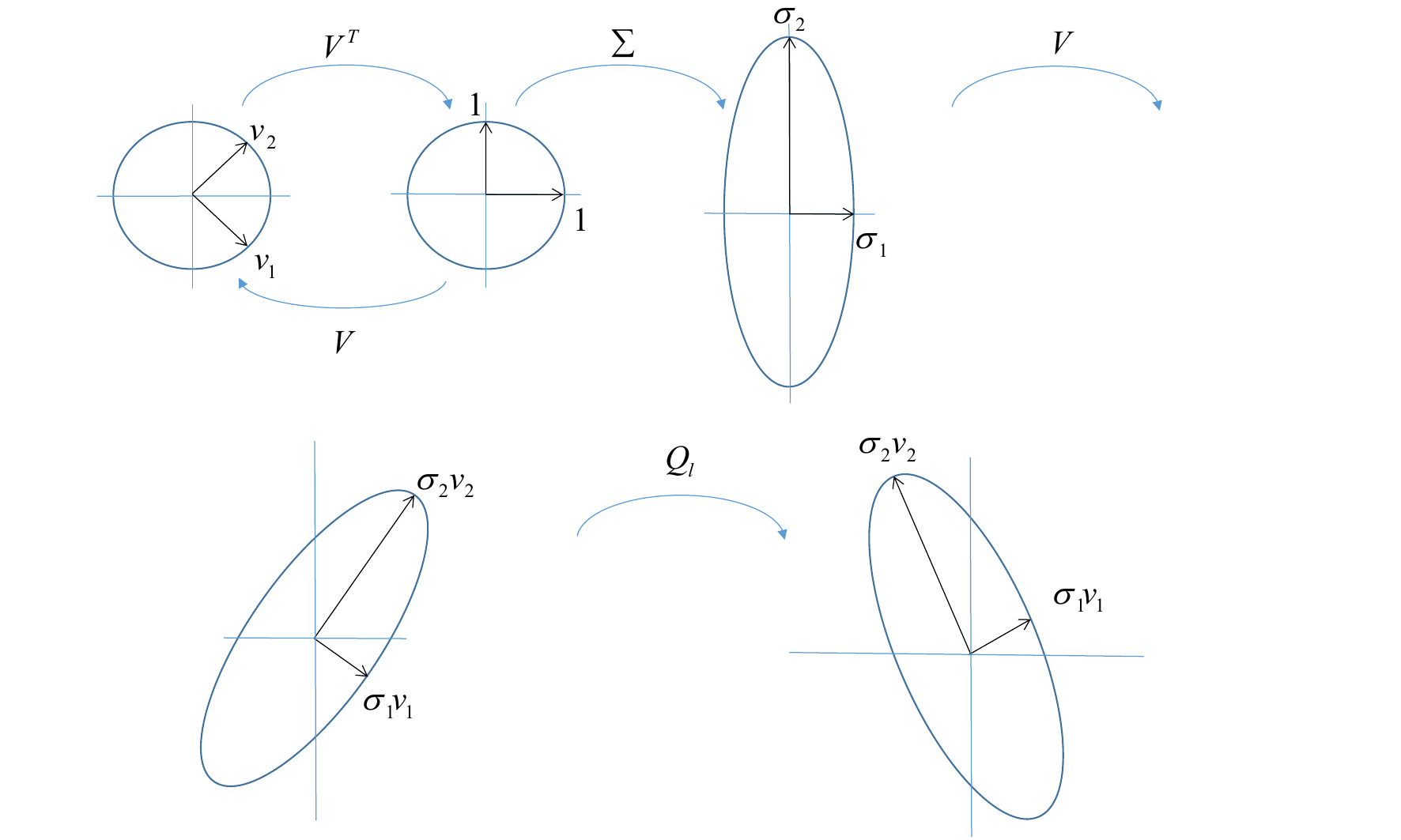}
\caption{Polar decomposition $\bA=\bQ_l\bS$: $\bV^\top$ rotates or reflects, $\bSigma$ stretches the cycle to an ellipse, and $\bV$ rotates or reflects back. Orthogonal matrices $\bV^\top$, $\bV$, and $\bQ_l$ only change the basis of the system. However, they preserve both the angle between the vectors $\bv_1$ and $\bv_2$, and their lengths.}
\label{fig:polar-rotate2}
\end{figure}

\section*{Polar Decomposition}
Any square matrix $\bA\in\real^{n\times n}$ can be factored as the left polar decomposition $\bA = (\bU\bV^\top)( \bV\bSigma \bV^\top) = \bQ_l\bS$. Similarly, the operation $\bA\bu = \bQ_l( \bV\bSigma \bV^\top)\bu$ transforms $\bu$ into the system defined by $\bV^\top$, and stretch each component by the lengths of the corresponding singular values. 
Subsequently, the resulting vector is transferred back into the Cartesian coordinate system by $\bV$. Finally, $\bQ_l$ will rotate or reflect the resulting vector from the Cartesian coordinate system into the $Q$ system defined by $\bQ_l$.
The right polar decomposition carries a similar interpretation. 
Similar to the spectral decomposition, orthogonal matrices $\bV^\top$, $\bV$, and $\bQ_l$ only change the basis of the system but preserve the angle between the vectors $\bv_1$ and $\bv_2$.

\section{Application: LS via  Cholesky, QR, UTV,  SVD, Bidiagonalization}\label{section:application-ls-qr}
In 1801, Gauss predicted the orbit of the steroid Ceres using the method of \textit{least squares (LS, or ordinary least squares)}. Since then, the principle of least squares has become  the standard procedure for the analysis of scientific data, which is also the genesis of a vast array of models in machine learning. 
We now show how to solve the least squares problems using various decompositional approaches \citep{lu2021rigorous}.

\index{Normal equation}
\subsection*{Least Squares via Choelesky Decomposition}
Let us consider the overdetermined system $\bA\bx = \bb$, where $\bA\in \real^{m\times n}$ is the data matrix, and $\bb\in \real^m$  is the observation vector, with $m\geq n$. 
In most real-world applications, $\bA$ typically has full column rank, either naturally or after preprocessing. 
The least squares  solution is given by $\bx_{LS} = (\bA^\top\bA)^{-1}\bA^\top\bb$, which minimizes $\norm{\bA\bx-\bb}^2$, where $\bA^\top\bA$ is invertible since $\bA$ has full column rank, and $\rank(\bA^\top\bA)=\rank(\bA)$.\index{Least squares}

The classical method for solving a linear least squares problem $\min_{\bx} \normtwo{\bA\bx - \bb}$ is to form and solve the symmetric \textit{normal equation} $\bA^\top \bA\bx = \bA^\top \bb$, which is derived from minimizing the objective function (the root of the objective function). If $\rank(\bA) = n$, then $\bx \neq \bzero$ implies that $\bA\bx \neq \bzero$. Hence,
\begin{equation}
	\bx^\top \bA^\top \bA\bx > 0, \quad \forall\, \bx \in \real^n, \quad \bx \neq \bzero,
\end{equation}
and $\bA^\top \bA$ is positive definite. Conversely, any symmetric positive definite matrix is nonsingular. If it were singular, there would be a vector $\bx\neq \bzero$ such that $\bA\bx = \bzero$, leading to $\bx^\top\bA^\top \bA\bx = 0$, which contradicts the positive definiteness.

\index{Normal equation}

Substituting the Cholesky factorization $ \bA^\top \bA = \bR^\top \bR $ into the normal equation yields $ \bR^\top \bR \bx = \balpha $, where $ \balpha = \bA^\top \bb $. Hence, the solution is obtained by solving two triangular systems:
\begin{equation}
	\bR^\top \bu = \balpha, \qquad \bR \bx = \bu.
\end{equation}
This method is easy to implement and often faster than other direct solution methods, e.g., using Gradient descent methods \citep{lu2021rigorous}. 

When solving a least squares problem, it is often preferable to work with the Cholesky factorization of the cross-product of the augmented matrix $[\bA ,\bb ]$:
\begin{equation}\label{equation:corss_prod_eq}
	\begin{bmatrix} 
		\bA^\top \\ 
		\bb^\top 
	\end{bmatrix} 
	\begin{bmatrix} 
		\bA & \bb 
	\end{bmatrix} 
	= 
	\begin{bmatrix} 
		\bA^\top \bA & \bA^\top \bb \\ 
		\bb^\top \bA & \bb^\top \bb
	\end{bmatrix}.
\end{equation}
If $\rank(\bA) = n$, then the Cholesky factor of the cross-product \eqref{equation:corss_prod_eq} takes the following form:
\begin{equation}
	\bS = \begin{bmatrix} 
		\bR & \bv \\ 
		\bzero & \rho 
	\end{bmatrix},
\end{equation}
which exists even when $\rho = 0$ (see Theorem~\ref{theorem:semidefinite-factor-exist}). Forming $\bS^\top \bS$ shows that
$$
\bA^\top \bA = \bR^\top \bR, \qquad \bR^\top \bv = \bA^\top \bb, \qquad \bb^\top \bb = \bv^\top \bv + \rho^2.
$$
Hence, $\bR$ is the Cholesky factor of $\bA^\top \bA$, and the least squares solution is obtained from $\bR\bx = \bv$. 
Since $\be = \bb - \bA\bx$ is orthogonal to $\bA\bx$ (by orthogonal projections; see Section~\ref{section:spec_app_eigproj}), we have
$$
\normtwo{\bA\bx}^2 = (\be + \bA\bx)^\top \bA\bx = \bb^\top \bA \bx = \bb^\top \bA \bR^{-1} \bR^{-\top} \bA^\top \bb = \bv^\top \bv,
$$
and therefore the residual term satisfies $\normtwo{\be}^2 = \rho^2 = \bb^\top \bb - \bv^\top \bv$ and $\normtwo{\bb - \bA\bx} = \rho$.

\index{Normal equation}
\paragraph{Cholesky QR factorization.}
On the other hand, suppose $\bA \in \real^{m\times n}$ have full column rank, and let $\bA^\top \bA = \bR^\top \bR$ be its Cholesky factorization. 
Define $\bQ_1 = \bA \bR^{-1} \in\real^{m\times n}$. Then, 
\begin{equation}\label{equation:qr_cho_dec}
	\bA = \bQ_1 \bR 
	\qquad\text{and}\qquad 
	\bQ_1^\top \bQ_1 = \bI_n
\end{equation}
is the \textit{Cholesky QR factorization} of $\bA$. The semi-orthogonal factor $\bQ_1$ can be obtained as the unique solution of the lower triangular matrix equation $\bR^\top \bQ_1^\top = \bA^\top$ using forward substitution. In this setting, the normal equation simplifies to $\bR^\top \bQ_1^\top \bQ_1 \bR\bx = \bR^\top \bR\bx = \bR^\top \bQ_1^\top \bb$ or
$ \bR\bx = \bQ_1^\top \bb. $

In real arithmetic, the computational cost of this Cholesky QR algorithm is $\sim 2mn^2 + n^3/3$ flops. 
More accurate methods for computing the QR factorization \eqref{equation:qr_cho_dec} directly from $\bA$ are described in Sections~\ref{section:qr-gram-compute}, \ref{section:silentcolu_qrdecomp}, and \citet{lu2021numerical}.

\subsection*{Least Squares via  Full QR Decomposition}

Since computing the inverse of a matrix can be computationally expensive, as an alternative, we can use the QR decomposition to find the least squares solution.
This approach is more efficient and numerically stable. The method is summarized in the following theorem:
\begin{theorem}[LS via QR for full column rank matrix]\label{theorem:qr-for-ls}
Let $\bA\in \real^{m\times n}$, with full rank and $m\geq n$, be the data matrix, and $\bb\in \real^m$ be the observation vector. 
And let $\bA = \bQ\bR$ be its full QR decomposition, where $\bQ\in\real^{m\times m}$ is  orthogonal, and $\bR\in \real^{m\times n}$ is  upper triangular, with  $m-n$  additional  rows of zeros appended at the bottom: $\bR = \scriptsize\begin{bmatrix}
	\bR_1 \\
	\bzero
\end{bmatrix}$ and $\bR_1 \in \real^{n\times n}$ is the square upper triangular part of $\bR$. Then, the LS solution to $\bA\bx=\bb$ is given by 
$$
\bx_{LS} = \bR_1^{-1}\bc, 
\quad \text{where} \quad 
\bQ^\top\bb = \begin{bmatrix}
\bc\\
\bd
\end{bmatrix}.
$$
\end{theorem}

\begin{proof}[of Theorem~\ref{theorem:qr-for-ls}]
Since $\bA=\bQ\bR$ is the full QR decomposition of $\bA$ and $m\geq n$, the last $m-n$ rows of $\bR$ are zero, as shown in Figure~\ref{fig:qr-comparison}. 
It follows that 
$$
\begin{aligned}
\normtwo{\bA\bx-\bb}^2
\stackrel{\dag}{=}\normtwo{\bQ^\top \bA \bx-\bQ^\top\bb}^2
&=\left\Vert\begin{bmatrix}
	\bR_1 \\
	\bzero
\end{bmatrix} \bx-\bQ^\top\bb\right\Vert^2
=\normtwo{\bR_1\bx - \bc}^2+\normtwo{\bd}^2,
\end{aligned}
$$ 
where the equality ($\dag$) follows from the invariance of norms under orthogonal transformations. Here, $\bc$ represents the first $n$ components of $\bQ^\top\bb$, and $\bd$ represents the last $m-n$ components. The least squares solution is obtained by solving the upper triangular system $\bR_1\bx = \bc$, which can be expressed as $\bx_{LS} = \bR_1^{-1}\bc$.
\end{proof}

In the least squares problem, the rows of $\bA$ represent data samples, and the number of columns, $n$, corresponds to the dimension of the variables.
In some applications, we may want to add or remove a data point (a row in $\bA$), or alternatively, add or remove a variable (a column in $\bA$); for instance, when performing variable selection using an $F$-test \citep{lu2021rigorous}.
By utilizing the update methods described in Sections~\ref{section:append-column-qr} and \ref{section:append-row-qr}, the QR decomposition of the modified matrix $\bA$ can be efficiently updated, rather than recomputed from scratch.
Therefore, QR decomposition is particular useful for \textit{online LS problems}, where the solution is updated sequentially as new data arrives, as well as for \textit{feature selection problems}, where variables are added or removed dynamically.

%
%

\subsection*{\textbf{Least Squares via ULV/URV for Rank-Deficient Matrices}}
In the previous sections, we introduced the least squares  method using the full QR decomposition and the Cholesky decomposition  for matrices with full column rank. 
However, in practice, many matrices are rank-deficient, meaning they do not have full rank. 
If $\bA$ does not have full column rank, $\bA^\top \bA$ is not invertible. To handle such cases, we can use the ULV/URV decomposition to find the least squares solution, as stated in the following theorem.
\index{Least squares}\index{Rank-deficient}
\begin{theorem}[LS via ULV/URV for rank-deficient matrix]\label{theorem:qr-for-ls-urv}
Let $\bA\in \real^{m\times n}$ be a matrix of rank $r$ and $m\geq n$. Suppose $\bA=\bU\bT\bV$ is its full ULV/URV decomposition, where $\bU\in\real^{m\times m}$ and $\bV\in \real^{n\times n}$ are orthogonal matrices, and
$$
\bT = \begin{bmatrix}
\bT_{11} & \bzero \\
\bzero & \bzero
\end{bmatrix},
$$
where $\bT_{11} \in \real^{r\times r}$ is either a lower or  upper triangular matrix.
Given $\bb\in \real^m$,  the LS solution with the minimal $\ell_2$ norm to $\bA\bx=\bb$ is given by 
$$
\bx_{LS} = \bV^\top 
\begin{bmatrix}
\bT_{11}^{-1}\bc\\
\bzero	
\end{bmatrix},
\quad \text{where}\quad 
\bU^\top\bb = \begin{bmatrix}
	\bc\\
	\bd
\end{bmatrix}.
$$
\end{theorem}

\begin{proof}[of Theorem~\ref{theorem:qr-for-ls-urv}]
Since $\bA=\bU\bT\bV$ is the full UTV decomposition of $\bA$ and $m\geq n$, we can write:
$$
\begin{aligned}
\norm{\bA\bx-\bb}^2 
&=\norm{\bU^\top \bA \bx-\bU^\top\bb}^2 
=\norm{\bU^\top\bU\bT\bV \bx-\bU^\top\bb}^2\\
&=\norm{\bT\bV \bx-\bU^\top\bb}^2
=\norm{\bT_{11}\be - \bc}^2+\norm{\bd}^2,   
\end{aligned}
$$ 
where $\bc$ is the first $r$ components of $\bU^\top\bb$, $\bd$ is the last $m-r$ components of $\bU^\top\bb$, $\be$ is the first $r$ components of $\bV\bx$, and $\bff$ is the last $n-r$ components of $\bV\bx$:
$$
\bU^\top\bb 
= \begin{bmatrix}
\bc \\
\bd 
\end{bmatrix}, 
\qquad 
\bV\bx 
= \begin{bmatrix}
\be \\
\bff
\end{bmatrix}.
$$
The least squares solution is obtained by performing backward/forward substitution of the upper/lower triangular system $\bT_{11}\be = \bc$, i.e., $\be = \bT_{11}^{-1}\bc$. To ensure that the solution $\bx$ has the minimal $\ell_2$ norm, $\bff$ must be zero. Thus,
$
\bx_{LS} = \bV^\top 
\scriptsize
\begin{bmatrix}
\bT_{11}^{-1}\bc\\
\bzero	
\end{bmatrix}.
$
This completes the proof.
\end{proof}

\paragraph{Note on the minimal $\ell_2$ norm LS solution.} For the least squares problem, the set of all minimizers
$$
\mathcal{X} = \{\bx\in \real^n: \norm{\bA\bx-\bb} =\min \}
$$
is a convex set. 
If $\bx_1, \bx_2 \in \mathcal{X}$ and $\lambda \in [0,1]$, then 
$$
\norm{\bA(\lambda\bx_1 + (1-\lambda)\bx_2) -\bb} \leq \lambda\norm{\bA\bx_1-\bb} +(1-\lambda)\norm{\bA\bx_2-\bb} = \mathop{\min}_{\bx\in \real^n} \norm{\bA\bx-\bb}. 
$$
Thus, $\lambda\bx_1 + (1-\lambda)\bx_2 \in \mathcal{X}$. In the proof above, if we do not set $\bff=\bzero$, we can still find other least squares solutions. However, the minimal $\ell_2$ norm least squares solution is unique. For the full-rank case, as discussed in the previous sections, the least squares solution is always unique and  must have the minimal $\ell_2$ norm \citep{foster2003solving, golub2013matrix, lu2021rigorous}.

\subsection*{\textbf{Least Squares via SVD for Rank-Deficient Matrices}}
\index{Least squares}

In addition to the UTV decomposition, the singular value decomposition (SVD) can also be used to solve rank-deficient least squares problems.
\begin{theorem}[LS via SVD for rank-deficient matrix]\label{theorem:svd-deficient-rank}
Let $\bA\in \real^{m\times n}$ be a  matrix of rank $r$ and $m\geq n$. 
Suppose $\bA=\bU\bSigma\bV^\top$ is its full SVD, where $\bU=[\bu_1, \bu_2, \ldots, \bu_m]\in\real^{m\times m}$ and $\bV=[\bv_1, \bv_2, \ldots, \bv_n]\in \real^{n\times n}$  are orthogonal. 
Given $\bb\in \real^m$, the LS solution with the minimal $\ell_2$ norm to $\bA\bx=\bb$ is given by 
\begin{equation}\label{equation:svd-ls-solution}
\bx_{LS} = \sum_{i=1}^{r}\frac{\bu_i^\top \bb}{\sigma_i}\bv_i = \bV\bSigma^+\bU^\top \bb, 
\end{equation}
where the upper-left side of $\bSigma^+ \in \real^{n\times m}$ is a diagonal matrix, structured as: $\bSigma^+ =\scriptsize \begin{bmatrix}
	\bSigma_1^+ & \bzero \\
	\bzero & \bzero
\end{bmatrix}$ with $\bSigma_1^+=\diag(\frac{1}{\sigma_1}, \frac{1}{\sigma_2}, \ldots, \frac{1}{\sigma_r})$.
\end{theorem}

\begin{proof}[of Theorem~\ref{theorem:svd-deficient-rank}]
To minimize the squared error, we begin by expanding:
$$
\begin{aligned}
\norm{\bA\bx-\bb}^2 
& \stackrel{\dag}{=}\norm{\bU^\top \bA \bx-\bU^\top\bb}^2 
=\norm{\bU^\top \bA \bV\bV^\top \bx-\bU^\top\bb}^2 \\
&\stackrel{+}{=}\norm{\bSigma\balpha - \bU^\top\bb}^2  
\stackrel{*}{=}\sum_{i=1}^{r}(\sigma_i\alpha_i - \bu_i^\top\bb)^2 +\sum_{i=r+1}^{m}(\bu_i^\top \bb)^2,
\end{aligned}
$$
where the equality ($\dag$) follows from the invariance of the norm under orthogonal transformations, the equality (+) follows by setting  $\balpha=\bV^\top \bx$, and the equality ($*$) follows because $\sigma_{r+1}=\sigma_{r+2}= \ldots= \sigma_m=0$.
Since $\bx$ only appears in $\balpha$, we minimize the expression by setting $\alpha_i = {\bu_i^\top\bb}/{\sigma_i}$ for all $i\in \{1, 2, \ldots, r\}$. Any value assigned to $\alpha_{r+1}, \alpha_{r+2}, \ldots, \alpha_{n}$ will not affect the error term. 
From the regularization point of view (or to obtain the solution with the smallest $\ell_2$ norm), we can set them to  0. This yields the LS solution via SVD:
$$
\bx_{LS} = \sum_{i=1}^{r}\frac{\bu_i^\top \bb}{\sigma_i}\bv_i=\bV\bSigma^+\bU^\top \bb = \bA^+\bb,
$$
where $\bA^+=\bV\bSigma^+\bU^\top\in \real^{n\times m}$ is known as the \textit{pseudo-inverse} of $\bA$.\index{Pseudo-inverse}
\end{proof}

\subsection*{Bidiagonal Least Squares and LGK Bidiagonalization}
We consider an  overdetermined linear system $\bC\bx=\bb$, where $\bC\in\real^{m\times (n-1)}$ and $m\geq  (n-1)$. 
\footnote{Note we set the matrix dimensions to $m\times (n-1)$ since we consider the bidiagonal decomposition of an $m\times n$ augmented matrix. Generally, we can also consider $\bC\in\real^{m\times n}$ with $m\geq n$.}
We are interested in the bidiagonalization (Theorem~\ref{theorem:Golub-Kahan-Bidiagonalization-decom}) of the augmented matrix $\bA = [\bb, \bC]\in\real^{m\times n}$, which admits the following bidiagonal decomposition:
\begin{equation}\label{equation:lgk_bidi_ls}
	\begin{aligned}
		\bA = &\bU\bB\bV^\top 
		=
		\bU\bB
		\begin{bmatrix}
			1 & \bzero \\
			\bzero & \bQ
		\end{bmatrix}^\top\\ 
		&\implies  
		\bB=
		\bU^\top [\bb, \bC] 
		\begin{bmatrix}
			1 & \bzero \\
			\bzero & \bQ
		\end{bmatrix}
		=
		[\bU^\top\bb, \bU^\top\bC\bQ]
		=
		\begin{bmatrix}
			b_{11}\be_1& \bB_2\\
			\bzero & \bzero
		\end{bmatrix},
	\end{aligned}
\end{equation}
where $b_{11}$ represents the (1,1) entry of  $\bB$, $\bB_2 = \bB[1:n,2:n]\in\real^{n\times (n-1)}$, and $\bQ=\bV[2:n,2:n]\in\real^{(n-1)\times (n-1)}$ is orthogonal (see Problem~\ref{prob:bid_vstruc}).
We then have 
$$
\begin{aligned}
	\normtwo{\bb-\bC\bx} =
	\normtwo{
		[\bb, \bC]
		\begin{bmatrix}
			1 \\
			-\bx 
		\end{bmatrix}
	}
	=
	\normtwo{
		\bU^\top
		[\bb, \bC]\bV\bV^\top
		\begin{bmatrix}
			1 \\
			-\bx 
		\end{bmatrix}
	}.
\end{aligned}
$$
Let $\bd = \bQ^\top\bx$. Then,
$$
\normtwo{\bb-\bC\bx}
=
\normtwo{\bB \bV^\top
	\begin{bmatrix}
		1 \\
		-\bx 
	\end{bmatrix}
}
=
\normtwo{b_{11} \be_1- \bB_2\bd}.
$$
Therefore, the least squares problem of $\normtwo{\bb-\bC\bx}$ then can be equivalently recovered by finding the least squares solution of $\normtwo{b_{11} \be_1- \bB_2\bd}$ in terms of the variable $\bd$.

\subsection*{LGK Bidiagonalization}
We now express $\bB_2 = \bB[1:n,2:n]$ as follows:
$$
\bB_2 = 
\footnotesize
\begin{bmatrix}
	b_{12} &  &  \ldots &   \\
	b_{22} & b_{23} & \ldots  &  \\
	&\ddots &\ddots & \vdots \\
	&&b_{n-1,n-1}& b_{n-1,n} \\
	&&& b_{nn}
\end{bmatrix}
\implies 
\normalsize
\bB_2^\top = 
\footnotesize
\begin{bmatrix}
	\footnotesize
	b_{12} & b_{22} &  \ldots &   &  \\
	0 & b_{23} & b_{33}  &   &  \\
	& &\ddots &b_{n-1,n-1} &  \\
	&& & b_{n-1,n} & b_{nn} \\
\end{bmatrix}
\in\real^{(n-1)\times n}.
$$
\paragraph{First step.}
From Equation~\eqref{equation:lgk_bidi_ls}, we find that 
${b_{11} = \normtwo{\bb} = \normtwo{\ba_1}}$ if $b_{11}$ is nonnegative, where $\bb=\ba_1$ is the first column of $\bA$. 
Additionally, we have:
$$
\bC\bQ = 
\bU
\begin{bmatrix}
	\bB_2\\
	\bzero 
\end{bmatrix}
\implies
\bC^\top 
\begin{bmatrix}
	\bu_1 & \bu_2&\ldots &\bu_{n}
\end{bmatrix}
=
\bQ\bB_2^\top, 
\quad \text{where}\quad \bU=[\bu_1,\bu_2,\ldots,\bu_m].
$$
Let $\bQ=[\bq_1,\bq_2,\ldots,\bq_{n-1}]$ be the column partition of $\bQ$, and let $\bq_0=\bzero$. We then have 
$$
\left\{
\begin{aligned}
	\bC^\top\bu_i &= b_{ii}\bq_{i-1} + b_{i,i+1}\bq_i \,\,\implies\,\, b_{i,i+1}\bq_i = \bC^\top\bu_i - b_{ii}\bq_{i-1}, \gap \forall\, i\in\{1,2,\ldots, n-1\};\\
	\bC^\top\bu_n &= b_{nn}\bq_{n-1}.
\end{aligned}
\right.
$$
If $b_{ii}$ and $\bq_{i-1}$ are known, $b_{i,i+1}$ can be determined as the norm of the right-hand side equation in the above equality: 
\begin{align}
b_{i,i+1}&=\pm\normtwo{\bC^\top\bu_i - b_{ii}\bq_{i-1}}, \gap \forall\, i\in\{1,2,\ldots, n-1\};\label{equation:lgk_bidi_ls_ste11} \\ 
\bq_i &= \frac{\bC^\top\bu_i - b_{ii}\bq_{i-1}}{b_{i,i+1}}, \gap \text{if }b_{i,i+1}\neq 0, \,\,\forall\, i\in\{1,2\ldots, n-1\}; \label{equation:lgk_bidi_ls_ste12}\\
b_{ii}&=\pm\normtwo{\bC^\top\bu_i - b_{i,i+1}\bq_i}, \gap \forall\, i\in\{2,3,\ldots, n-1\}.\label{equation:lgk_bidi_ls_ste13}
\end{align}

\paragraph{Second step.}
Similarly, from  Equation~\eqref{equation:lgk_bidi_ls}, we have ${\bu_1 = \bb/b_{11}\equiv\ba_1/b_{11}}$ and 
$
\bC\bQ =
\footnotesize 
\bU
\begin{bmatrixscript}
	\bB_2\\
	\bzero 
\end{bmatrixscript}.
$
This leads to:
\begin{equation}\label{equation:lgk_bidi_ls_eq2}
\begin{aligned}
\bC \bq_i = b_{i,i+1}\bu_i + b_{i+1,i+1}\bu_{i+1}, \gap \forall\, i \in\{1,2,\ldots,n-1\}\\
\implies 
{\bu_{i+1} = \frac{\bC \bq_i -b_{i,i+1}\bu_i}{b_{i+1,i+1}}} , \gap \text{if }b_{i+1,i+1}\neq 0,\,\, \forall\, i \in\{1,2,\ldots,n-1\}.
\end{aligned}
\end{equation}

The two steps described above form a recursive algorithm for computing the bidiagonal decomposition of the matrix  $\bA$, and is known as the \textit{LGK bidiagonalization}. 
The derivation above is valid  when $m> n$. A similar approach can be applied when $n\geq m$. 
Simple calculations can show the complexity is $\sim 4mn^2$ flops to obtain all $\bB, \bU$, and $\bV$, which is more efficient than the standard Golub--Kahan bidiagonalization; see Section~\ref{section:exist_bidia_gk}.

\begin{algorithm}[H] 
\caption{LGK Bidiagonal Decomposition} 
\label{alg:lgk_bidiagonal} 
\begin{algorithmic}[1] 
\Require Matrix $\bA$ with size $m\times n $ and $m\geq n$; 
\State Initially set $b_{11}  \leftarrow  \normtwo{\ba_1}$, $\bu_1 \leftarrow \ba_1/b_{11}$, $\bq_0\leftarrow\bzero$; 
\For{$i=1$ to $n-1$} 
\State $b_{i,i+1}\leftarrow \pm\normtwo{\bC^\top\bu_i - b_{ii}\bq_{i-1}}$ by Equation~\eqref{equation:lgk_bidi_ls_ste11}; 
\State $\bq_i \leftarrow \frac{\bC^\top\bu_i - b_{ii}\bq_{i-1}}{b_{i,i+1}}$ by 
Equation~\eqref{equation:lgk_bidi_ls_ste12}; 
\State $\bu_{i+1} \leftarrow \frac{\bC \bq_i -b_{i,i+1}\bu_i}{b_{i+1,i+1}}$ by 
Equation~\eqref{equation:lgk_bidi_ls_eq2}; 
\State $b_{jj}\leftarrow\pm\normtwo{\bC^\top\bu_j - b_{j,j+1}\bq_j}$ by 
Equation~\eqref{equation:lgk_bidi_ls_ste13}, where $j=i+1$; 
\EndFor
\State Output $\bB, \bU$, and $\bV$.
\end{algorithmic} 
\end{algorithm}

The algorithm breaks down if any $b_{i,i+1}$ or $b_{jj}$ is equal to zero. However, in the context of solving least squares problems, these cases can be handled with special treatment; see \citet{bjorck2004acta}.
Another issue arises is that, in floating-point arithmetic, the columns in $\bU$ and $\bV$ can lose orthogonality as the recursion proceeds (similar to the loss of orthogonality seen in the CGS and MGS methods for computing the QR decomposition; see Section~\ref{section:qr-gram-compute}).

\index{Approximate least squares}
\paragraph{Approximate least squares.}
We further explore the approximation of the least squares problem $\mathop{\min}_{\bx}\normtwo{\bC\bx-\bb}$. Denote $\bQ_k=[\bq_1,\bq_2,\ldots,\bq_k]$, $\bU_k=[\bu_1,\bu_2,\ldots,\bu_k]$, and $\bU=[\bU_{k+1}, \bU_\perp]$. 
Additionally, let $\bB_k$ be the upper-left $k\times (k-1)$ submatrix of $\bB_2$. 
Once again, referring to Equation~\eqref{equation:lgk_bidi_ls}, we have:
$$
\bC\bQ_k = \bU_{k+1}\bB_{k+1}.
$$
Note that the variable $\bx$ lies in $\real^{n-1}$, and the vectors $\{\bq_1,\bq_2,\ldots,\bq_k\}$ are mutually orthonormal in $\real^{n-1}$.
Approximately, we can estimate $\bx$ using a linear combination of the $k$ vectors, i.e., there exists a vector $\by$ such that $\bx\approx\bQ_k\by$.
Assume we want to find the optimal approximate solution within the subspace spanned by the $k$ vectors $\{\bq_1,\bq_2,\ldots,\bq_k\}$, i.e., solving the following problem in terms of $\by$:
\begin{equation}\label{equation:reduced_rank_bidiagonal}
	\mathop{\min}_{\by} \normtwo{\bC\bQ_k\by - \bb},
\end{equation}
where $\by\in\real^k$ (it can be shown that $\by=\bQ_k^\top\bx \in\real^k$).
Based on the preceding discussion, the optimization problem is equivalent to:
$$
\begin{aligned}
	\mathop{\min}_{\by}
	\normtwo{\bU_{k+1}\bB_{k+1}\by - \bb}
	&=
	\mathop{\min}_{\by}
	\normtwo{\bU^\top(\bU_{k+1}\bB_{k+1}\by - \bb)}\\
	&=
	\mathop{\min}_{\by}
	\normtwo{\begin{bmatrix}
			\bB_{k+1}\by \\
			\bzero 
		\end{bmatrix}
		-
		\begin{bmatrix}
			b_{11}\be_1\\
			\bzero 
		\end{bmatrix}
	}
	=
	\mathop{\min}_{\by}
	\normtwo{\bB_{k+1}\by
		-
		b_{11}\be_1
	}.
\end{aligned}
$$
Thus, the approximate least squares problem becomes 
$
\mathop{\min}_{\by}
\normtwo{
	\bB_{k+1}\by
	-
	b_{11}\be_1
},
$
where $b_{11} = \normtwo{\bb}$.
Due to the bidiagonal structure, the problem can be solved in $\sim n$ flops \citep{elden2007matrix}.

\paragraph{Reduced-rank model.}
The problem in Equation~\eqref{equation:reduced_rank_bidiagonal} is known as  the least squares problem associated with the \textit{reduced-rank model}. Instead of considering the full model 
$\mathop{\min}_{\by} \normtwo{\bC\bx - \bb}$, we introduce an approximate orthogonal basis of low dimension in $\real^{n-1}$ where the solution $\bx$ lies (i.e.,  $\{\bq_1,\bq_2,\ldots, \bq_k\}$). 
This approach helps reduce the ill-conditioning of the original problem and makes the solution less sensitive to perturbations in the data \citep{elden2007matrix}.

\index{Unbiased estimator}
\index{Consistent estimator}
\section{Application: PCA via  Spectral Decomposition and  SVD}
An important application of SVD is its use in \textit{principal component analysis (PCA)}. 
PCA is widely employed to identify patterns in data and to analyze the variance-covariance structure of the data. It serves two primary purposes:
\begin{enumerate}
	\item \textit{Data reduction}. Reducing the dimensionality of the data by selecting a smaller number of \textit{principal components}.
	\item \textit{Interpretation}. Uncovering relationships within the data that were previously unobserved.
\end{enumerate}

\index{Principal component analysis}
Given a data set of $n$ observations $\{\bx_1,\bx_2,\ldots,\bx_n\}$, where each $\bx_i\in \real^p$ for all $i\in \{1,2,\ldots,n\}$, the goal is to project the data into a lower-dimensional space of dimension $m$ ($m<p$). 
To do this, we first compute the sample mean vector and the sample covariance matrix:
$$
\overline{\bx} = \frac{1}{n}\sum_{i=1}^{n}\bx_i
\qquad 
\text{and}
\qquad 
\bS = \frac{1}{n-1}\sum_{i=1}^{n} (\bx_i - \overline{\bx})(\bx_i-\overline{\bx})^\top,
$$
where the $n-1$ term in $\bS$ ensures it is an unbiased and consistent estimator of the covariance matrix \citep{lu2021rigorous}. 
Alternatively, the covariance matrix can also be defined as $\bS = \frac{1}{\textcolor{mylightbluetext}{n}}\sum_{i=1}^{n} (\bx_i - \overline{\bx})(\bx_i-\overline{\bx})^\top$, which is still a consistent estimator of the covariance matrix \footnote{Consistency: An estimator $\theta_n $ of $\theta$ constructed on the basis of a sample of size $n$ is said to be consistent if $\theta_n\stackrel{p}{\rightarrow} \theta$   as $n \rightarrow \infty $.}.

Each data point $\bx_i$ is then projected onto a scalar value using a vector $\bu_1$ (see discussion below), such that the projection is given by $\bu_1^\top\bx_i$. The mean of the projected data is obtained by $\Exp[\bu_1^\top\bx_i] = \bu_1^\top \overline{\bx}$, and the variance of the projected data is given by 
$$
\begin{aligned}
\Cov[\bu_1^\top\bx_i] &= \frac{1}{n-1} \sum_{i=1}^{n}( \bu_1^\top \bx_i - \bu_1^\top\overline{\bx})^2=
\frac{1}{n-1} \sum_{i=1}^{n}\bu_1^\top  ( \bx_i -\overline{\bx})( \bx_i -\overline{\bx})^\top\bu_1=\bu_1^\top\bS\bu_1.
\end{aligned}
$$
To retain as much information as possible in the projection, we maximize the projected variance $\bu_1^\top\bS\bu_1$  with respect to $\bu_1$. 
To prevent $\bu_1$ from scaling indefinitely, a constraint is imposed: $\norm{\bu_1}^2=\bu_1^\top\bu_1=1$. 
Using the method of Lagrange multipliers (see, for example, \citet{bishop2006pattern, boyd2004convex}), the optimization problem becomes:
$$
\bu_1^\top\bS\bu_1 + \lambda_1 (1 - \bu_1^\top\bu_1).
$$
Solving this yields the equation:
$$
\bS\bu_1 = \lambda_1\bu_1 \leadto \bu_1^\top\bS\bu_1 = \lambda_1.
$$
This shows that $\bu_1$ is an eigenvector of $\bS$  corresponding to the eigenvalue $\lambda_1$. The direction of maximum variance, $\bu_1$, corresponds to the largest eigenvalue of $\bS$. The eigenvector $\bu_1$ is referred to as the \textit{first principal axis}.

The subsequent principal axes are defined by the remaining eigenvectors of $\bS$, arranged in descending order of their eigenvalues. By selecting the top $m$ principal components, the dimensionality of the data can be effectively reduced. This process is known as the \textit{maximum-variance formulation} of PCA \citep{hotelling1933analysis, bishop2006pattern, shlens2014tutorial}.
Alternative perspectives on the maximum-variance formulation, such as from data reconstruction, data projection, and autoencoders, are discussed in \citet{lu2021numerical}.
Another approach, known as the \textit{minimum-error formulation} of PCA, is discussed in \citet{pearson1901liii, bishop2006pattern}.

\paragraph{PCA via the spectral decomposition.}
Now, let's assume that the data are already centered, meaning the sample mean vector $\overline{\bx}$ is the zero vector. 
Alternatively, we can centralize the data by setting $\bx_i := \bx_i-\overline{\bx}$, which involves subtracting the mean from each data point. 
Let the data matrix $\bX \in \real^{n\times p}$ contain the centered data, with each row representing one observation. 
The covariance matrix is symmetric, and its spectral decomposition is given by 
\begin{equation}\label{equation:pca-equ1}
	\bS = \frac{\bX^\top\bX}{n-1} = \bU\bLambda\bU^\top,
\end{equation}
where $\bU$ is an orthogonal matrix of eigenvectors (the columns of $\bU$ are the eigenvectors of $\bS$), and $\bLambda=\diag(\lambda_1, \lambda_2,\ldots, \lambda_p)$ is a diagonal matrix containing the corresponding eigenvalues (ordered such that $\lambda_1 \geq \lambda_2 \geq \ldots \geq \lambda_p$).
As discussed above, the eigenvectors are called the \textit{principal axes} of the data, and they \textit{decorrelate} the  covariance matrix. Projections of the original data onto the principal axes are called the \textit{principal components}. 
Specifically, the $i$-th principal component is given by the $i$-th column of $\bX\bU$. 
If our objective is to reduce the dimension from $p$ to $m$, we simply select the first $m$ columns of $\bX\bU$, i.e., $\bX\bU[:,1:m]$.

\paragraph{PCA via  SVD.}
If the SVD of $\bX$ is given by $\bX = \bP\bSigma\bQ^\top$, then the covariance matrix can be expressed as 
\begin{equation}\label{equation:pca-equ2}
	\bS= \frac{\bX^\top\bX}{n-1} = \bQ \frac{\bSigma^2}{n-1}\bQ^\top,
\end{equation}
where $\bQ\in \real^{p\times p}$ is an orthogonal matrix containing the right singular vectors of $\bX$, and the upper-left part of $\bSigma$ is a diagonal matrix containing the singular values $\diag(\sigma_1,\sigma_2,\ldots)$, ordered such that $\sigma_1\geq \sigma_2\geq \ldots$. The number of singular values is equal to $\min\{n,p\}$, which will not be larger than $p$, and some of these values may be zero.

By comparing Equation~\eqref{equation:pca-equ2} with Equation~\eqref{equation:pca-equ1}, we can see that Equation~\eqref{equation:pca-equ2} also represents a spectral decomposition of $\bS$.
This is because both the eigenvalues in $\bLambda$ and the singular values in $\bSigma$ are ordered in descending order, and the spectral decomposition in terms of the eigenspaces is unique (as discussed in Section~\ref{section:uniqueness-spectral-decomposition}).

In other words, the right singular vectors $\bQ$ can also serve as the principal axes, which decorrelate the covariance matrix. 
The singular values are related to the eigenvalues of the covariance matrix through the relationship: $\lambda_i = \frac{\sigma_i^2 }{n-1}$ for each $i$. To reduce the dimensionality of the data from $p$ to $m$, we  select the largest $m$ singular values and their corresponding right singular vectors. 
This process is related to the truncated SVD (TSVD), where: $\bX_m = \sum_{i=1}^{m}\sigma_i \bp_i\bq_i^\top$, where $\bp_i$'s and $\bq_i$'s are the columns of $\bP$ and $\bQ$, respectively.

\index{Truncated SVD}

\paragraph{A byproduct of PCA via  SVD for high-dimensional data.} For a principal axis $\bu_i$ of $\bS = \frac{\bX^\top\bX}{n-1}$, we have 
$
\frac{\bX^\top\bX}{n-1} \bu_i = \lambda_i \bu_i.
$
Multiplying both sides by $\bX$ on the left, we obtain:
$$
\frac{\bX\bX^\top}{n-1} (\bX\bu_i) = \lambda_i (\bX\bu_i),
$$
which implies that $\lambda_i$ is also an eigenvalue of $\frac{\bX\bX^\top}{n-1} \in \real^{n\times n}$, and the corresponding eigenvector is $\bX\bu_i$. This relationship is also discussed in the proof of Theorem~\ref{theorem:reduced_svd_rectangular}, which establishes the existence of the SVD. 
When the number of features $p$ is much larger than the number of samples $n$ (i.e., $p \gg n$), instead of finding the eigenvectors of $\bS=\frac{\bX^\top\bX}{n-1}$, i.e., the principal axes of $\bS=\frac{\bX^\top\bX}{n-1}$, we can find the eigenvectors of $\frac{\bX\bX^\top}{n-1}$. This reduces the computational complexity from $\mathcalO(p^3)$ to $\mathcalO(n^3)$, which is more efficient when $p \gg n$. 

Now, returning to the principal axes of $\bS=\frac{\bX^\top\bX}{n-1}$, suppose  the eigenvector of $\frac{\bX\bX^\top}{n-1}$ is $\bv_i$, corresponding to a nonzero eigenvalue $\lambda_i$:
$
\frac{\bX\bX^\top}{n-1} \bv_i = \lambda_i \bv_i.
$
Multiplying both sides by $\bX^\top$, we obtain 
$$
\frac{\bX^\top\bX}{n-1} (\bX^\top\bv_i) = \bS(\bX^\top\bv_i)   = \lambda_i (\bX^\top\bv_i),
$$
which shows that the eigenvector $\bu_i$ of $\bS$ is proportional to $\bX^\top\bv_i$, where $\bv_i$ is the eigenvector of $\frac{\bX\bX^\top}{n-1}$ corresponding to the same eigenvalue $\lambda_i$. Note that a further normalization step is required to ensure that $\norm{\bu_i}=1$.
Thus, when $p\gg n$,  we can efficiently compute the principal axes using the  spectral decomposition of $\frac{\bX\bX^\top}{n-1}$, instead of directly computing the eigenvectors of $\bS$.

\index{Data whitening}
\paragraph{Data whitening.}
PCA is commonly used for feature preprocessing in machine learning. It first reduces the dimensionality of the data and then normalizes the newly transformed features so that the variance along each direction in the transformed space is equal. Let $\bU_m$ be the $p \times m$ matrix containing the top-$m$ eigenvectors obtained from PCA. The first step is to transform the mean-centered data matrix  $\bX$ into an $m$-dimensional representation using $\bU_m$, as follows:
$$
\widetildebX = \bX\bU_m.
$$
The next step involves scaling each column of $\widetildebX$ by its standard deviation. This process transforms the original data distribution into one that is approximately spherical in shape. This technique is known as \textit{whitening}.

Whitened data often leads to better performance in gradient-based optimization algorithms \citep{lu2025practical}. 
This is because large differences in variance across features can cause the loss function to have varying curvature in different directions, which slows down convergence. By normalizing the variance, whitening reduces ill-conditioning of the loss function, allowing gradient descent to converge faster. Additionally, it prevents certain features from dominating the learning process due to their scale.

Whitening is especially valuable in unsupervised learning tasks such as outlier detection, where no labels are available to guide the relative importance of different directions in the data. In such cases, ensuring that all directions are treated equally becomes even more critical.
An illustration of how PCA whitens an ellipsoidal data distribution is shown in Figure~\ref{fig:whitening}, where the resulting distribution becomes approximately spherical.

\begin{figure}[h]
	\centering
	\includegraphics[width=0.9\textwidth]{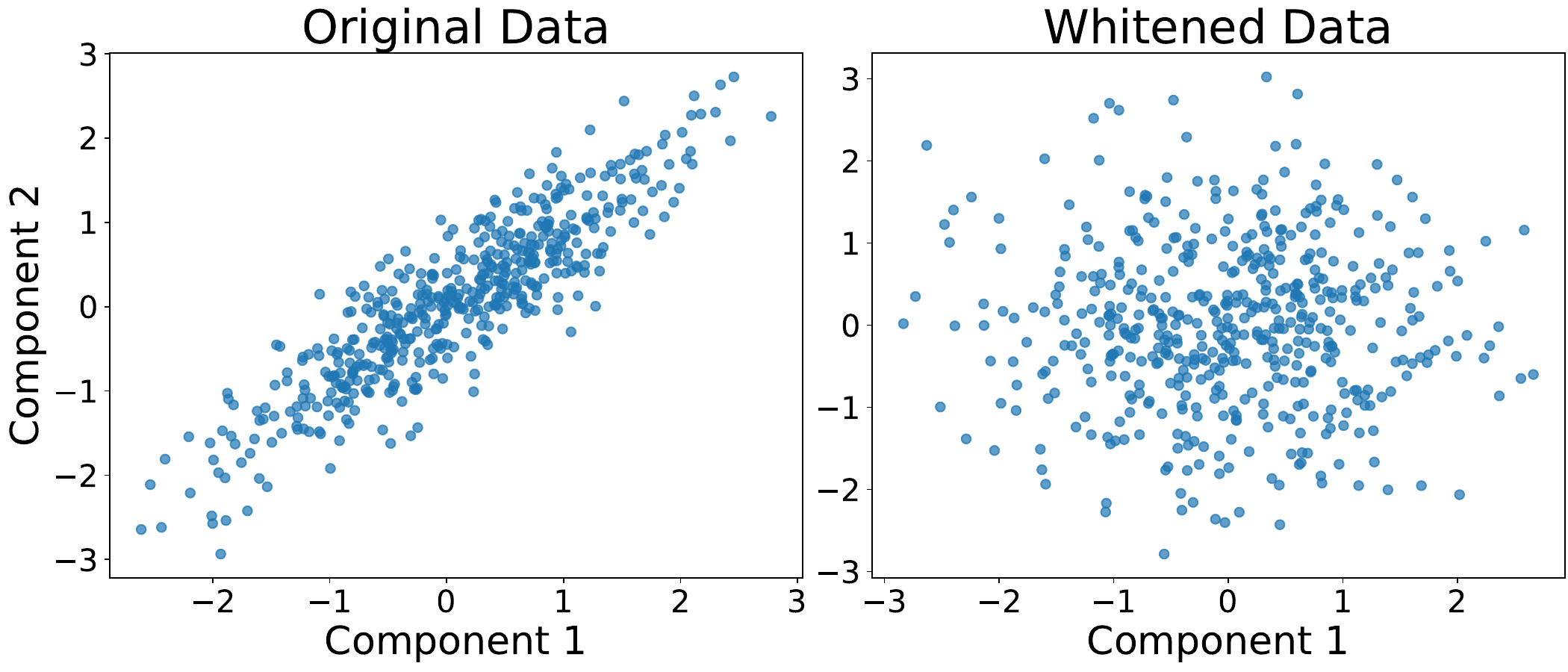} 
	\caption{An example of whitening an ellipsoidal data distribution using principal component analysis.}
	\label{fig:whitening}
\end{figure}

\section{Application: Low-Rank Approximation}\label{section:svd-low-rank-approxi}
In the context of low-rank approximation, two types of problems arise due to the interplay between rank and approximation error: \textit{fixed-precision approximation problem} and \textit{fixed-rank
approximation problem}. In the fixed-precision approximation problem, given a matrix $\bA$ and a  tolerance $\epsilon$, the goal is to find a matrix $\bB$
with rank $r = r(\epsilon)$ such that $\norm{\bA-\bB} \leq \epsilon$ in an appropriate matrix norm. On the contrary, in the fixed-rank approximation problem, one seeks  a matrix $\bB$ with a fixed rank $k$ that minimizes the error $\norm{\bA-\bB}$. This section focuses on the latter. 

To approximate a matrix $\bA\in \real^{m\times n}$ of rank $r$ with  a rank-$k$ matrix $\bB$ ($k<r$), the approximation can be evaluated using the spectral norm (Definition~\ref{definition:spectral_norm}):\index{Low-rank approximation}
\begin{equation}
\bB = \mathop{\arg\min}_{\rank(\bB)=k} \ \ \norm{\bA - \bB}_2.
\end{equation}
Then, we can recover the optimal rank-$k$ approximation by the following theorem.
\begin{theorem}[Eckart--Young--Mirsky theorem w.r.t. spectral norm\index{Eckart--Young--Mirsky theorem}]\label{theorem:young-theorem-spectral}
Given a matrix $\bA\in \real^{m\times n}$, $1\leq k\leq \rank(\bA)=r$, and let $\bA_k$ be the \textit{truncated SVD (TSVD)} of $\bA$ with the largest $k$ singular terms, i.e., $\bA_k = \sum_{i=1}^{k} \sigma_i\bu_i\bv_i^\top$ from the  SVD of $\bA=\sum_{i=1}^{r} \sigma_i\bu_i\bv_i^\top$ by zeroing out the $r-k$ trailing singular values of $\bA$. Then, $\bA_k$ is the optimal rank-$k$ approximation to $\bA$ in terms of the spectral norm.~\footnote{Note that $\bA_k$ can be stored using $(m+n)k + k$ entries, as opposed to $mn$ entries.} 
\end{theorem}

\begin{proof}[of Theorem~\ref{theorem:young-theorem-spectral}]
For any matrix $\bB$ with $\rank(\bB)=k$, we aim to show that $\norm{\bA-\bB}_2 \geq \norm{\bA-\bA_k}_2$.

Since $\rank(\bB)=k$,  $\dim(\nspace(\bB))=n-k$. Thus, any set of $k+1$ basis vectors in $\real^n$ intersects $\nspace(\bB)$. 
From Proposition~\ref{proposition:svd-four-orthonormal-Basis}, the singular vectors $\{\bv_1,\bv_2, \ldots, \bv_r\}$ form an orthonormal basis for $\cspace(\bA^\top)\subset \real^n$; so we can choose the first $k+1$ singular vectors $\bv_i$ as a basis for a $(k+1)$-dimensional subspace  of $\real^n$. Let $\bV_{k+1} = [\bv_1, \bv_2, \ldots, \bv_{k+1}]$. Then there exists a unit vector $\bx$ such that
$$
\bx \in \nspace(\bB) \cap \cspace(\bV_{k+1}),\qquad \text{s.t.}\,\,\,\, \norm{\bx}_2=1.
$$
That is, the vector $\bx$ can be expressed as $\bx = \sum_{i=1}^{k+1} a_i \bv_i$ with $\norm{\sum_{i=1}^{k+1} a_i \bv_i}_2 = \sum_{i=1}^{k+1}a_i^2=1$, and we also have $\bB\bx=\bzero$.
Using these properties, the following chain of inequalities holds:
$$
\begin{aligned}
\norm{\bA-\bB}_2^2 &\geq \norm{(\bA-\bB)\bx}_2^2\big/ \norm{\bx}_2^2
\stackrel{\dag}{=} \norm{\bA\bx}_2^2 
\stackrel{+}{=}\sum_{i=1}^{k+1} \sigma_i^2 (\bv_i^\top \bx)^2 \\
&\stackrel{*}{\geq} \sigma_{k+1}^2\sum_{i=1}^{k+1}  (\bv_i^\top \bx)^2 
\stackrel{\ddag}{\geq} \sigma_{k+1}^2\sum_{i=1}^{k+1} a_i^2 
= \sigma_{k+1}^2,
\end{aligned}
$$
where the first inequality follows from the defintion of the spectral norm, 
the equality ($\dag$) follows from the fact that $\bx$ lies in null space of $\bB$, 
the equality ($+$) follows from the fact that $\bx$ is orthogonal to $\bv_{k+2}, \ldots, \bv_r$, 
the inequality ($*$) follows from $\sigma_{k+1}\leq \sigma_{k}\leq\ldots\leq \sigma_{1}$, and the inequality ($\ddag$) follows from $\bv_i^\top \bx = a_i$.
On the other hand, it is evident that $\norm{\bA-\bA_k}_2^2 = \norm{\sum_{i=k+1}^{r}\sigma_i\bu_i\bv_i^\top}_2^2=\sigma_{k+1}^2$. Thus, $\norm{\bA-\bA_k}_2 \leq \norm{\bA-\bB}_2$, which completes the proof. 
\end{proof}

\index{Matrix norm}
\index{Frobenius norm}
Moreover, it can also be shown that  $\bA_k$ is  the optimal rank-$k$ approximation to $\bA$ in terms of the Frobenius norm (Definition~\ref{definition:frobernius-in-svd}). The minimal error is given by the Euclidean norm of the singular values that have been zeroed out in the process: $\norm{\bA-\bA_k}_F =\sqrt{\sigma_{k+1}^2 +\sigma_{k+2}^2+\ldots +\sigma_{r}^2}$.

\begin{exercise}
Given the definition of the Frobenius norm, show that the truncated SVD $\bA_k = \sum_{i=1}^{k} \sigma_i\bu_i\bv_i^\top$ is also the optimal rank-$k$ approximation to $\bA$ in terms of the Frobenius norm.
\end{exercise}

\index{Latent semantic analysis}
\paragraph{Application in the text domain.}
Low-rank approximation is an important technique with significant applications in  text analysis.
In this context, each document is represented as a row in a matrix, where the number of columns corresponds to the number of unique words (features). The value of each entry in the matrix represents the frequency of a specific word in the corresponding document.
It is worth noting that such matrices are typically very sparse, making them a standard use case for SVD. For example, a word-frequency matrix $\bA$ might have dimensions  $m = 10^6$ documents by  $n = 10^5$ words.
In such cases, truncated SVD often provides excellent approximations of the original matrix using a rank as small as 
$k \approx 400$ \citep{aggarwal2020linear}, which results in a substantial reduction in dimensionality. This application of SVD in text analysis is commonly referred to as \textit{latent semantic analysis}, due to its ability to uncover hidden (latent) topics represented by the rank-1 components of the SVD.

\begin{problemset}
\item Show that $(\bA\bA^\top)^{1/2}\bA = \bA(\bA^\top\bA)^{1/2}$.
	
\item Given a matrix $\bA\in\real^{m\times n}$, show that the trace of $\bA^\top\bA$ is equal to the sum of the squares of all its entries, i.e., $\trace(\bA^\top\bA) = \sum_{i,j=1}^{m,n}a_{ij}^2$.

\item  \textbf{Matrix factorization via spanning subspaces.} Let $\bA\in\real^{m\times n}$ be a matrix of rank $r$. Suppose the columns of $\bB\in\real^{m\times r}$ span the column space of  $\bA$, and the columns of $\bC\in\real^{n\times r}$ span the row space of  $\bA$. Show that the matrix $\bA$ can be factored as $\bA=\bB\bE\bC^\top$, where $\bE$ is an $r$ by $r$ nonsingular matrix.
\item Consider a square matrix $\bA\in\real^{n\times n}$ of rank $r$, and  the $(2n)\times (2n)$ symmetric matrix
$
\bB=
\begin{bmatrixfoot}
\bzero & \bA \\
\bA^\top & \bzero 
\end{bmatrixfoot}.
$
If $\bA$ admits the full SVD $\bA=\bU\bSigma\bV^\top$, where $\bSigma=\diag(\sigma_1,\sigma_2,\ldots,\sigma_n)$:
\begin{itemize}
\item Show that $\sigma_k$ is an eigenvalue of $\bB$ corresponding to the eigenvector $\begin{bmatrixfoot}
\bv_k \\
\bu_k
\end{bmatrixfoot}$ for any $k\in\{1,2,\ldots,n\}$,
and that $-\sigma_k$ is an eigenvalue of $\bB$ corresponding to the eigenvector $\begin{bmatrixfoot}
\bv_k \\
-\bu_k
\end{bmatrixfoot}$  for any $k\in\{1,2,\ldots,n\}$.
\item Show that the $2n$ eigenvectors are pairwise orthogonal.
\end{itemize}

\item Consider a rectangular matrix $\bA\in\real^{m\times n}$ of rank $r$, and  the $(m+n)\times (m+n)$ symmetric matrix
$
\bB=
\begin{bmatrixfoot}
\bzero & \bA \\
\bA^\top & \bzero 
\end{bmatrixfoot}.
$
If $\bA$ admits the full SVD $\bA=\bU\bSigma\bV^\top$, where $\bSigma=\diag(\sigma_1,\sigma_2,\ldots,\sigma_n)$:
\begin{itemize}
\item Show that $\sigma_k$ is an eigenvalue of $\bB$ corresponding to the eigenvector $\begin{bmatrixfoot}
\bv_k \\
\bu_k
\end{bmatrixfoot}$ for any $k\in\{1,2,\ldots,r\}$,
and that $-\sigma_k$ is an eigenvalue of $\bB$ corresponding to the eigenvector $\begin{bmatrixfoot}
\bv_k \\
-\bu_k
\end{bmatrixfoot}$  for any $k\in\{1,2,\ldots,r\}$.
\item Show that the remaining $m+n-2r$ eigenvectors of $\bB$ are corresponding to the eigenvalue 0.
\item Show that the $m+n$ eigenvectors  are pairwise orthogonal.
\end{itemize}

\item Given two nonzero vectors $\bu,\bv\in\real^n$, and let $\bA=\bu\bv^\top$. Show that the nonzero singular value of $\bA$ is $\norm{\bu}\cdot \norm{\bv}$.

\item For a square matrix $\bA\in\real^{n\times n}$ with singular values $\sigma_1\geq \sigma_2\geq\ldots \geq \sigma_n$, show that $\sigma_1^3, \sigma_2^3,\ldots , \sigma_n^3$ are the singular values of $\bA\bA^\top\bA$.

\item Let $\bA\in\real^{m\times n}$ be a rectangular matrix, and let $\bB\in\real^{\widehat{m}\times \widehat{n}}$ be a submatrix of $\bA$, where $\widehat{m}\leq m$ and $\widehat{n}\leq n$. Show that the largest singular value of $\bB$ is less than or equal to the largest singular value of $\bA$. 

\item For a positive definite matrix $\bA\in\real^{n\times n}$, show that the singular values and the eigenvalues of $\bA$ are the same.

\item Given a  matrix $\bA\in\real^{n\times n}$ and a positive definite matrix $\bB\in\real^{n\times n}$, prove that the singular values of $\bB\bA$ are the same as those of $\bA$. Discuss the relationship between the left and right singular vectors of $\bB\bA$ and $\bA$.

\item We have shown in Lemma~\ref{lemma:orthogonal-equivalent-matrix} that orthogonally equivalent matrices share the same singular values. 
Prove the reverse implication: if two matrices have the same singular values, then they are orthogonally equivalent.

\item In this chapter, we focus on the SVD of real matrices, expressed as $\bA=\bU\bSigma\bV^\top$. Show that if $\bA$ is real, then the matrices $\bU$ and $\bV$ are also real.

\item Given a Householder transformation matrix $\bH = \bI - 2\bu\bu^\top\in\real^{n\times n}$, where $\norm{\bu}=1$, determine the eigenvalues, determinant, and singular values of $\bH$.

\item Given the nonzero singular values $\sigma_1, \sigma_2, \ldots, \sigma_r$ of $\bA$, discuss the singular values of $\bA^\top$, $\gamma\bA$ with $\gamma>0$, and $\bA^{-1}$ (if $\bA$ is nonsingular).

\item Given a square and real matrix $\bA\in\real^{n\times n}$, show that $\bA=\bzero$ if and only if $\bA$ has only zero eigenvalues.

\item Given a square matrix $\bA\in\real^{n\times n}$, show that $\bA^\top\bA$ and $\bA\bA^\top$ are similar (Definition~\ref{definition:similar-matrices}). \textit{Hint: Proceeding with the SVD of $\bA$}.

\item Show that all eigenvalues of a square matrix are less than or equal to its largest singular value $\sigma_1$.

\item Suppose $\bx$ is an eigenvector of $\bA^\top\bA$ corresponding to a nonzero eigenvalue. Discuss the corresponding eigenvector of $\bA\bA^\top$. \textit{Hint: Premultiply by $\bA$}.

\item Given the SVD of a nonsingular square matrix $\bA=\bU\bSigma\bV^\top\in\real^{n\times n}$, determine the singular values of $\bA^\top\bA$.

\item Find the optimal rank-one approximation (in terms of the spectral norm) for the matrix:
$\bA=\begin{bmatrixfoot}
	\cos\theta & -\sin\theta \\
	\sin\theta & \cos\theta
\end{bmatrixfoot}.
$

\item \textbf{Skew-symmetric.} Given a skew-symmetric and tridiagonal matrix $\bS\in\real^{n\times n}$, show that it can be decomposed as:
$
\bP^\top\bS\bP 
=
\begin{bmatrixscript}
\bzero & \bB^\top \\
\bB & \bzero
\end{bmatrixscript},
$
where $\bB\in\real^{m\times m}$, $n=2m$, and $\bP$ is a permutation matrix.
Given further the SVD of $\bB=\bU\bSigma\bV^\top$, find the eigenvalues and eigenvectors of $\bS$.

\item Discuss the uniqueness of the polar decomposition for the  matrix:
$
\bA = \scriptsize\begin{bmatrix}
	1 & 0 \\
	0 & 0\\
\end{bmatrix}.
$

\item Let $\bA$ be a  negative semidefinite matrix. Show that the singular value decomposition of $
\bA$ is of the form $\bA = \bU \bSigma\bV^\top$, where $\bU = -\bV$.

\item \textbf{Block diagonal structure of padded SVD.} Let $\bB$ be a $p \times p$ matrix obtained by padding the $m \times n$ matrix $\bA$ with either zero rows or zero columns, where $p = \max\{m,n\}$. Depending on whether $m$ is greater than $n$ or vice versa, show that the SVD $\bB = \bU \bSigma \bV^\top$ takes one of the following forms:
$$
\begin{aligned}
	\text{When $n < m$}: \qquad &\bB = [\bA, \bzero] = \bU \begin{bmatrix} \bSigma_1 & \bzero \\ \bzero & \bzero \end{bmatrix} \begin{bmatrix} \bV_1 & \bzero \\ \bzero & \bV_2 \end{bmatrix}^\top;\\
	\text{When $m < n$}: \qquad &\bB = \begin{bmatrix} \bA \\ \bzero \end{bmatrix} = \begin{bmatrix} \bU_1 & \bzero \\ \bzero & \bU_2 \end{bmatrix} \begin{bmatrix} \bSigma_1 & \bzero \\ \bzero & \bzero \end{bmatrix} \bV^\top.
\end{aligned}
$$
Here, the matrices $\bU$, $\bV$, and $\bSigma$ are all square matrices of size $p \times p$. 
The matrix $\bV_1$ is of size $n \times n$, and $\bU_1$ is of size $m \times m$. The matrices $\bV_2$ and $\bU_2$ are of sizes $(p-n) \times (p-n)$ and $(p-m) \times (p-m)$, respectively. The matrix $\bSigma_1$ is of size $\min\{m,n\} \times \min\{m,n\}$.

\item
\textbf{Two-way to three-way SVD.} Let $\bA=\bP\bQ^\top\in\real^{m\times n}$ be a decomposition of matrix $\bA$, where the columns of $\bP\in\real^{m\times k}$ and $\bQ\in\real^{n\times k}$ are orthogonal (not necessarily orthonormal, i.e., having unit norms), and $k\leq \min\{m,n\}$. Provide a way to obtain the SVD of $\bA$.

\item \textbf{Push-through identity.} Use SVD to prove the push-through identity:
\begin{equation}
	\bC^\top(\lambda \bI_m + \bC\bC^\top)^{-1} = (\lambda \bI_n + \bC^\top \bC)^{-1}\bC^\top,
\end{equation}
where $\lambda>0$, and $\bC\in\real^{m\times n}$.

\item \textbf{Shared SVD from identical scatter matrices.} 
Consider two data matrices $\bA_1$ and $\bA_2$ that have identical scatter matrices $\bA_1^\top \bA_1 = \bA_2^\top \bA_2$, but are otherwise distinct. Show that both $\bA_1$ and $\bA_2$ can be decomposed using a partially shared singular value decomposition, such that $\bA_1 = \bU_1 \bSigma \bV^\top$ and $\bA_2 = \bU_2 \bSigma \bV^\top$. Use this fact to show that $\bA_2 = \bQ_{12} \bA_1$, where $\bQ_{12}$ is an orthogonal matrix.

\item \textbf{Frobenius norm.}
Let $\bA,\bB\in\real^{m\times n}$. Show that the squared Frobenius norm of $\bA-\bB$ is 
$$
\normf{\bA-\bB}^2 = \normf{\bA} + \normf{\bB} - 2\trace(\bA^\top\bB).
$$

\item Is there any coordinate transformation involved  in the QR or LQ decomposition?

\item Given the SVD of the matrix 
$$
\bA =
\begin{bmatrix}
	2 & 2 \\
	-1 & 1 
\end{bmatrix}
=
\begin{bmatrix}
	1 & 0 \\
	0 & 1 
\end{bmatrix}
\begin{bmatrix}
	2\sqrt{2} & 0 \\
	0 & \sqrt{2}
\end{bmatrix}
\begin{bmatrix}
	1/\sqrt{2} & 1/\sqrt{2} \\
	-1/\sqrt{2} & 1/\sqrt{2}
\end{bmatrix}
=\bU\bSigma\bV^\top,
$$
illustrate the coordinate transformation of this decomposition in a two-dimensional figure.

\item \citep{horn2012matrix} Let $\bQ\in\real^{n\times n}$ be an orthogonal matrix. Show that $\bQ$ can be decomposed as $\bQ=\bU_1\bU_2\ldots\bU_N\bD$, where $\bD=\diag(1,1,\ldots,1,\det(\bQ))$, each $\bU_i$ represents a plane rotation (Definition~\ref{definition:givens-rotation-in-qr}), and $N=n(n-1)/2$. \textit{Hint: Use the result in Problem~\ref{problem:part_ortho}.}
\end{problemset}

\newpage 
\part{Special Topics}

\newpage
\chapter{Alternating Least Squares (ALS)}\label{chapter:als}

It is evident that any given matrix can be factorized in infinitely many ways. However, certain types of factorizations are particularly valuable because of the specific properties they offer. Two main types of such properties are commonly sought in matrix decompositions:
\begin{itemize}
\item \textit{Linear algebra properties with exact decomposition.} 
In this type of decomposition, the goal is to break down a matrix into components that possess particular linear algebraic or geometric characteristics, such as orthogonality or triangular form. These properties make the decomposition useful for various tasks in linear algebra, including the construction of (orthogonal) bases, as discussed in previous chapters. So far, we have studied several decompositions that fall into this category, including LU decomposition, CR decomposition, QR decomposition, and singular value decomposition (SVD).

\item \textit{Optimization and compression properties with approximate decomposition.} 
This type focuses on approximating a large matrix by factoring it into smaller matrices. A well-known example is truncated SVD. Consider a matrix $\bA\in\real^{m\times n}$ that is approximated by a rank-$k$ matrix using the following factorization:
\begin{equation}
\bA \approx \bU_k \bSigma_k \bV_k^\top,
\end{equation}
where $\bU_k$ is an $m \times k$ semi-orthogonal matrix, $\bSigma_k$ is a $k \times k$ diagonal matrix with nonnegative entries, and $\bV_k$ is an $n \times k$ semi-orthogonal matrix. The total number of entries across all three matrices is $(m+n+k)k$ or $(m+n+1)k$ if counting only nonzero entries, which is often significantly smaller than the $mn$ entries in the original matrix when  $m$ and $n$ are both large. 
\end{itemize}

As discussed previously, singular value decomposition  is unique in that it offers advantages from both perspectives: it provides strong linear algebra properties when used exactly and useful compression properties when truncated; see Theorem~\ref{theorem:young-theorem-spectral}. The value $k$ is referred to as the {rank} of the approximation.
The optimization-based view of matrix factorization, where we approximate $\bA \approx \bW\bZ$, is especially valuable in machine learning. This approach involves defining $\bA$, $\bW$, and $\bZ$ in different ways depending on the application. 
Below are two important examples:
\begin{enumerate}
\item A \textit{rating} is a numerical score that a user assigns to an item, such as a movie~\footnote{For example, see the top 250 movies rated by different websites: \url{https://www.imdb.com/list/ls027618268/}.}. \textit{Recommender systems} collect these ratings to predict how users might rate items they haven't yet rated. When $\bA$ represents a \textit{user-item rating matrix}---where rows correspond to items, columns to users, and entries contain the observed ratings---the matrix factorization $\bA \approx \bW\bZ$ is performed using only the known ratings. In this case, the rows of $\bW$ represent latent features of items, and the columns of $\bZ$ represent latent features of users. The product $\bW\bZ$ reconstructs the full rating matrix, including predictions for missing entries.

\item When $\bA$ is a \textit{term-document matrix}, representing the frequency of words (rows of $\bA$) in documents (columns of $\bA$), the rows of $\bW$ provide latent representations of words, and the columns of $\bZ$ provide latent representations of documents (see Chapter~\ref{chapter:nmf}).
\end{enumerate}
From an optimization perspective, additional constraints can be imposed on the matrices involved in the factorization---such as requiring their entries to be nonnegative (as discussed in Chapter~\ref{chapter:nmf}). These constraints often enhance the usefulness of the decomposition in practical applications.

In this chapter, we will focus on the first application mentioned above (recommender systems). In the next chapter, we will discuss the second (topic modeling via term-document matrices using \textit{nonnegative matrix factorization (NMF)}).

\index{Least squares}
\index{Linear models}
\index{Regression analysis}
\section{Preliminary: Least Squares Approximations}\label{section:pre_ls}
The linear model is a fundamental technique in regression analysis, relying on the least squares approximation, which aims to minimize the sum of squared errors (refer to Section~\ref{section:application-ls-qr}). This method naturally emerges when trying to identify the regression function that minimizes the corresponding expected squared error.
Over the past several decades, linear models have found extensive applications across diverse domains, including decision-making \citep{dawes1974linear}, time series analysis \citep{christensen1991linear, lu2017machine}, quantitative finance \citep{menchero2011barra}, and various other fields such as production science, social science, and soil science  \citep{fox1997applied, lane2002generalized, schaeffer2004application, mrode2014linear}.

To be more concrete, consider an overdetermined system represented by $\bb = \bA\bx $, where $\bA\in \real^{m\times n}$ represents the \textit{input data matrix} (also known as the \textit{predictor variables}), $\bb\in \real^m$ is the \textit{observation vector} (or \textit{target/response vector}), and the  number of samples $m$ exceeds  the number of predictors $n$. 
The vector $\bx$ represents the \textit{weights} (or \textit{coefficients}) of the linear model.
Typically, it is assumed that $\bA$ has full column rank, as real-world data is often uncorrelated or can be preprocessed to meet this condition.
In practical scenarios, a \textit{bias term} (a.k.a., an \textit{intercept}) is added to the first column of $\bA$. This adjustment enables the least squares method to solve equations of the form:
\begin{equation}\label{equation:ls-bias}
	\widetildebA \widetildebx = 
[\bm{1} ,\bA ] 
\begin{bmatrix}
	x_0\\
	\bx
\end{bmatrix}
 = \bb .
\end{equation}

However, it is common for the equation $\bb = \bA\bx$ to have  no exact solution (the system is \textit{inconsistent}) because it is overdetermined---that is, there are more equations than unknowns.
Define the column space of $\bA$ as $\{\bA\bgamma: \,\, \forall \bgamma \in \real^n\}$, denoted by $\cspace(\bA)$.
In essence, when we say $\bb = \bA\bx$ has no solution, it implies that $\bb$ lies outside the column space of $\bA$. 
In other words, the error $\be = \bb -\bA\bx$ cannot be reduced to zero. 
The objective then becomes minimizing the error, which is typically measured using the mean squared error.
The resulting solution $\bx_{LS}$, which minimizes $\normtwo{\bb-\bA\bx_{LS}}^2$, is referred to   as the \textit{least squares solution}. The least squares method is a cornerstone of mathematical sciences, and a wealth of resources are dedicated to its study and application, including works by   \citet{trefethen1997numerical, strang2019linear, strang2021every,  lu2021rigorous}.

\paragraph{Least squares by calculus.}
When $\normtwo{\bb-\bA\bx}^2$ is differentiable and the parameter space of $\bx$ spans the entire space $\real^n$ (i.e., an unconstrained optimization problem)
\footnote{In this context, the \textit{domain} of the optimization problem $\mathop{\min}_{\bx} \normtwo{\bb-\bA\bx}^2$ is the entire space  $\real^n$.}, 
the least squares estimate corresponds to the root of the gradient of $\normtwo{\bb-\bA\bx}^2$. 
This leads us to the following lemma.~\footnote{Variants of the least squares problem are explored in Problems~\ref{problem:rls}$\sim$\ref{problem:twls2}.}

\index{Fermat's theorem}
\begin{lemma}[Least squares by calculus]\label{lemma:ols}
Let  $\bA \in \real^{m\times n}$ be a  fixed data matrix with full rank  and $m\geq n$ (i.e., its columns  are linearly independent)~\footnote{Relaxations of this condition using the pseudo-inverse are discussed in Problems~\ref{prob:als_pseudo1}$\sim$\ref{prob:als_pseudon}.}. 
For the overdetermined system $\bb = \bA\bx$, the least squares solution, obtained by   setting the partial derivatives in every direction of $\normtwo{\bb-\bA\bx}^2$ to  zero (i.e., the gradient vanishes), is given by $\bx_{LS} = (\bA^\top\bA)^{-1}\bA^\top\bb$ \footnote{This is known as the \textit{first-order optimality condition} for local optima points. Note that the proof of the first-order optimality condition for multivariate functions strongly relies on the first-order optimality conditions for univariate functions, which is also known as  \textit{Fermat's theorem}. See Problem~\ref{problem:fist_opt}.}. The value, $\bx_{LS} = (\bA^\top\bA)^{-1}\bA^\top\bb$, is commonly referred to as the \textit{ordinary least squares (OLS)} estimate or simply the \textit{least squares (LS)} estimate of $\bx$.
\end{lemma}

To prove the lemma above, we must show that $\bA^\top\bA$ is invertible. Given that $\bA$ has full rank and $m\geq n$, the matrix $\bA^\top\bA \in \real^{n\times n}$ is invertible if it has  rank  $n$, which matches the rank of $\bA$. This claim is verified in Lemma~\ref{lemma:rank-of-ata}.
\begin{proof}[of Lemma \ref{lemma:ols}]
Using calculus, a function $f(\bx)$ attains a minimum  at  $\bx_{LS}$ when its gradient $\nabla f(\bx)=\bzero$. The gradient of $\normtwo{\bb-\bA\bx}^2$ is given by $2\bA^\top\bA\bx -2\bA^\top\bb$. $\bA^\top\bA$ is invertible since we assume $\bA$ is fixed and has full rank with $m\geq n$ (Lemma~\ref{lemma:rank-of-ata}). 
Consequently, the OLS solution for $\bx$ is $\bx_{LS} = (\bA^\top\bA)^{-1}\bA^\top\bb$, which completes the proof.
\end{proof}

\index{Normal equation}
\begin{definition}[Normal equation]\label{definition:normal-equation-als}
The condition for the gradient of $\normtwo{\bb-\bA\bx}^2$ to be zero can be expressed as $\bA^\top\bA \bx = \bA^\top\bb$. This is called the \textit{normal equation}. 
Under the assumption that $\bA$ has full rank with $m\geq n$, the matrix  $\bA^\top\bA$ is invertible,  leading to the solution $\bx_{LS} = (\bA^\top\bA)^{-1}\bA^\top\bb$.
\end{definition}

\index{Convex function}
\begin{figure}[h!]
\centering  
\vspace{-0.35cm} 
\subfigtopskip=2pt 
\subfigbottomskip=2pt 
\subfigcapskip=-5pt 
\subfigure[A convex function.]{\label{fig:convex-1}
\includegraphics[width=0.26\linewidth]{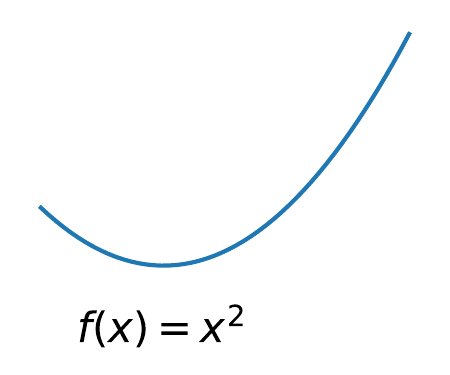}}
\subfigure[A concave function.]{\label{fig:convex-2}
\includegraphics[width=0.26\linewidth]{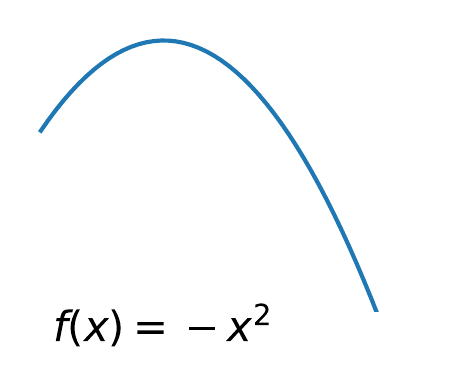}}
\subfigure[A random function.]{\label{fig:convex-3}
\includegraphics[width=0.26\linewidth]{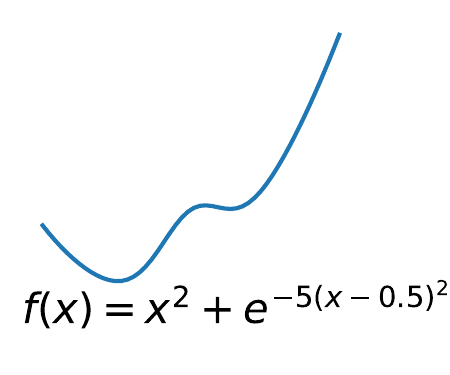}}
\caption{Three types of functions.}
\label{fig:convex-concave-none}
\end{figure}
However, it is not immediately evident whether the least squares estimate derived in Lemma~\ref{lemma:ols} corresponds to a minimum, maximum, or neither. An example illustrating this ambiguity is shown in Figure~\ref{fig:convex-concave-none}.
What we can assert with confidence is the existence of at least one root for the gradient of the function  $f(\bx)=\normtwo{\bb-\bA\bx}^2$. 
This root represents a necessary condition for a minimum point, but not necessarily a sufficient one. 
The following remark provides further clarification on this matter.
\begin{remark}[Verification of least squares solution]
Why does a zero gradient imply the least mean squared error? 
The usual explanation relies on convex analysis, which we will explore shortly. However, here we directly confirm that the OLS solution minimizes the mean squared error.
For any $\bx \neq \bx_{LS}$, we have 
\begin{equation}
\begin{aligned}
\normtwo{\bb - \bA\bx}^2 &= \normtwo{\bb - \bA\bx_{LS} + \bA\bx_{LS} - \bA\bx}^2
= \normtwo{\bb-\bA\bx_{LS} + \bA (\bx_{LS} - \bx)}^2 \\
&=\normtwo{\bb-\bA\bx_{LS}}^2 + \normtwo{\bA(\bx_{LS} - \bx)}^2 + 2\big(\bA(\bx_{LS} - \bx)\big)^\top(\bb-\bA\bx_{LS}) \\ 
&=\normtwo{\bb-\bA\bx_{LS}}^2 + \normtwo{\bA(\bx_{LS} - \bx)}^2 + 2(\bx_{LS} - \bx)^\top(\bA^\top\bb - \bA^\top\bA\bx_{LS}), \nonumber
\end{aligned} 
\end{equation}
where the third term is zero because of the normal equation, and $\normtwo{\bA(\bx_{LS} - \bx)}^2 \geq 0$. Therefore,
$
\normtwo{\bb - \bA\bx}^2 \geq \normtwo{\bb-\bA\bx_{LS}}^2. 
$
Thus, we have demonstrated that the OLS estimate corresponds to a minimum, not a maximum or a saddle point~\footnote{
A \textit{saddle point} is a point at which the gradient vanishes (a \textit{stationary point}), but the objective function increases in some directions and decreases in others.}.
As a matter of fact, this condition from the least squares estimate is also known as the \textit{sufficiency of stationarity under convexity}. When $\bx$ is defined over the entire space $\real^n$, this condition is also known as the \textit{necessity of stationarity under convexity}.
\end{remark}
\index{Saddle point}

Another natural question is: Why does the normal equation appear to ``magically" provide solutions for $\bx$?
A simple analogy can help illustrate this concept. The equation $x^2=-1$ has no real solution. 
However, $x\cdot x^2 = x\cdot (-1)$ does have a real solution $\hat{x} = 0$, in which case, $\hat{x}$ minimizes the difference between $x^2$ and $-1$, making them as close as possible.

\begin{example}[Altering the solution set by left multiplication]
Consider the  data matrix and target vector:
$\tiny
\bA=
\begin{bmatrix}
-3 & -4 \\
4 & 6  \\
1 & 1
\end{bmatrix}
$
and
$
\bb=
\tiny
\begin{bmatrix}
1  \\
-1   \\
0
\end{bmatrix}
.
$
It can be easily verified that the system $\bA\bx = \bb$ has no solution for $\bx$. 
However, if we multiply both sides on the left by
$
\bB=
\scriptsize
\begin{bmatrix}
0 & -1 & 6\\
0 & 1  & -4
\end{bmatrix},
$
then the solution $\bx_{LS} = [1/2, -1/2]^\top$ satisfies $\bB\bA\bx= \bB\bb$. 
This example illustrates why the normal equation can lead to the least squares solution. Multiplying a linear system on the left alters the solution set, effectively projecting the problem into a different subspace where a least squares solution exists.
\end{example}

\index{Rank-deficient}
\paragraph{Rank-deficiency.}
In this discussion, we assume that the matrix $\bA\in \real^{m\times n}$ has full rank with $m\geq n$, ensuring that $\bA^\top\bA$ is invertible. 
However, if two or more columns of $\bA$ are perfectly correlated, the matrix $\bA$ becomes deficient, and $\bA^\top\bA$ becomes singular. 
To address this issue, one can select the vector $\bx$ that minimizes $\normtwo{\bx_{LS}}^2$ while satisfying the normal equation. That is, we choose the least squares solution with the smallest magnitude. 
In Section~\ref{section:application-ls-qr}, we briefly discussed how to use UTV decomposition and SVD to solve such rank-deficient least squares problems.
See Problems~\ref{prob:als_pseudo1}$\sim$\ref{prob:als_pseudon} or the following paragraph for further insights.

\index{Condition number}
\index{Tikhonov regularization}
\index{$\ell_2$ regularization}
\paragraph{Regularizations and stability.}
A common  issue that arise in the ordinary least square solution is the near-singularity of the matrix $\bA$.
Let the SVD of $\bA$ be $\bA=\bU\bSigma\bV^\top\in\real^{m\times n}$, where $\bU\in\real^{m\times m}$ and $\bV\in\real^{n\times n}$ are orthogonal, and the main diagonal of $\bSigma\in\real^{m\times n}$ contains the singular values of $\bA$. 
Consequently, $\bA^\top\bA = \bV(\bSigma^\top\bSigma)\bV^\top = \bV\bS\bV^\top$, where $\bS= \bSigma^\top\bSigma  = \diag(\sigma_1^2, \sigma_2^2, \ldots,\sigma_n^2)\in\real^{n\times n}$ contains the squared singular values of $\bA$. When $\bA$ is nearly singular, $\sigma_n^2\approx 0$, making the inverse operation $(\bA^\top\bA)^{-1} = \bV\bS^{-1}\bV^\top$ numerically unstable. 
As a result, the least squares solution $\bx_{LS} =(\bA^\top\bA)^{-1}\bA^\top\bb $ may become highly sensitive or even diverge.
To address this instability, an $\ell_2$ regularization term is typically added, leading to the solution of the following optimization problem:
\begin{equation}
\bx_{Tik} = \mathop{\argmin}_{\bx} \normtwo{\bb-\bA\bx}^2 +\lambda\normtwo{\bx}^2.
\end{equation}
This approach is known as the  \textit{Tikhonov regularization method} (or simply the $\ell_2$ regularized method) \citep{tikhonov1963solution}.
The gradient of the problem is $2(\bA^\top\bA+\lambda\bI)\bx-2\bA^\top\bb$. Thus, the least squares solution is given by 
$
\bx_{Tik} = (\bA^\top\bA+\lambda\bI)^{-1}\bA^\top\bb.
$
The inverse operation becomes $(\bA^\top\bA+\lambda\bI)^{-1} = \bV(\bS+\lambda\bI)^{-1}\bV^\top$, where $\widetildebS=(\bS+\lambda\bI)=\diag(\sigma_1^2+\lambda, \sigma_2^2+\lambda, \ldots,\sigma_n^2+\lambda)$. 
The solutions for OLS and Tikhonov regularized LS are given, respectively, by 
\begin{equation}
\begin{aligned}
\bx_{LS} &= (\bA^\top\bA)^{-1}\bA^\top\bb = \bV\left(\bS^{-1}\bSigma\right)\bU^\top\bb;\\
\bx_{Tik} &= (\bA^\top\bA+\lambda\bI)^{-1}\bA^\top\bb = \bV\left((\bS+\lambda\bI)^{-1}\bSigma\right)\bU^\top\bb,\\
\end{aligned}
\end{equation}
where the main diagonals of $\left(\bS^{-1}\bSigma\right)$ are $\diag(\frac{1}{\sigma_1}, \frac{1}{\sigma_2}, \ldots, \frac{1}{\sigma_n})$; and the main diagonals of $\left((\bS+\lambda\bI)^{-1}\bSigma\right)$ are $\diag(\frac{\sigma_1}{\sigma_1^2+\lambda}, \frac{\sigma_2}{\sigma_2^2+\lambda}, \ldots, \frac{\sigma_n}{\sigma_n^2+\lambda})$. The latter solution is more stable if $\lambda$ is greater than the   smallest nonzero squared singular value.
The \textit{condition number}, which measures the sensitivity of the problem to perturbations, becomes smaller if  the smallest singular value $\sigma_n$ is close to zero:
$$
\kappa(\bA^\top\bA) = \frac{\sigma_1^2}{\sigma_n^2}
\qquad \rightarrow \qquad
\kappa(\bA^\top\bA+\lambda\bI) = \frac{\lambda+\sigma_1^2}{\lambda+\sigma_n^2}.
$$
Thus, Tikhonov regularization effectively prevents divergence in the least squares solution 
$\bx_{LS} = (\bA^\top\bA)^{-1} \bA^\top \bb$ when the matrix $\bA$ is nearly singular or even rank-deficient. This improvement enhances the convergence properties of both the LS algorithm and its variants, such as alternating least squares, while addressing identifiability issues  in various settings (see Section~\ref{section:regularization-extention-general}). As a result, Tikhonov regularization has become a widely applied technique.

\begin{exercise}
Use SVD to show that the optimum solution $\bx_{Tik} = (\bA^\top\bA+\lambda\bI)^{-1}\bA^\top\bb$
has non-increasing norm with increasing $\lambda$.
\end{exercise}

\index{Data least squares}
\index{Total least squares}
\paragraph{Data least squares.}
While the OLS method accounts for errors in the response variable $\bb$, the \textit{data least sqaures (DLS)} method considers errors in the predictor variables:
\begin{equation}
\bx_{DLS} = \mathop{\argmin}_{\bx, \widetildebA} \normf{\widetildebA}^2, \quad \text{s.t.}\quad \bb\in\cspace(\bA+\widetildebA),
\end{equation}
where $\widetildebA$ represents a perturbation in the matrix $\bA$ (i.e., a noise in the predictor variables).
That is, $(\bA+\widetildebA) \bx_{DLS} = \bb$, assuming the measured response $\bb$ is noise-free.
The Lagrangian function and its gradient w.r.t. $\bx$ are, respectively, given by 
$$
\begin{aligned}
L(\bx, \widetildebA, \blambda) &= \trace(\widetildebA\widetildebA^\top) +\blambda^\top (\bA\bx+\widetildebA\bx-\bb);\\
\nabla_{\widetildebA} L(\bx, \widetildebA,\blambda) &= \widetildebA+\blambda\bx^\top = \bzero \quad\implies\quad \widetildebA=-\blambda\bx^\top,
\end{aligned}
$$
where $\blambda\in\real^m$  is a vector of Lagrange multipliers.
Substituting the value of the vanishing gradient into $(\bA+\widetildebA) \bx = \bb$ yields $\blambda = \frac{\bA\bx-\bb}{\bx^\top\bx}$ and $\widetildebA=-\frac{(\bA\bx-\bb)\bx^\top}{\bx^\top\bx} $.
Therefore, using the invariance of the trace under cyclic permutations,  the objective function becomes 
$$
\mathop{\argmin}_{\bx}
\frac{(\bA\bx-\bb)^\top (\bA\bx-\bb)}{\bx^\top\bx} .
$$

\paragraph{Total least squares.} Similar to  data least squares, the \textit{total least squares (TLS)} method accounts for errors in both the predictor variables and the response variables. The TLS problem can be formulated as:
\begin{equation}
\bx_{TLS} = \mathop{\argmin}_{\bx, \widetildebA, \widetildebb} \normf{[\widetildebA, \widetildebb]}^2, 
\quad \text{s.t.}\quad (\bb+\widetildebb)\in\cspace(\bA+\widetildebA), 
\end{equation}
where $\widetilde{\bA}$ and $\widetilde{\bb}$ represent perturbations in the predictor variables and the response variable, respectively.
To simplify, define $\bC=[\bA,\bb]\in\real^{m\times (n+1)}$,  $\bD=[\widetildebA, \widetildebb]\in\real^{m\times (n+1)}$, and $\by\in\scriptsize\begin{bmatrix}
\bx\\
-1
\end{bmatrix}$, the problem can be equivalently stated as
\begin{equation}
\bx_{TLS} = \mathop{\argmin}_{\by, \bD} \normf{\bD}^2, 
\quad \text{s.t.}\quad \bD\by = -\bC\by, 
\end{equation}

\index{Decomposition: ALS}
\index{Netflix recommender}
\section{Netflix Recommender and Matrix Factorization}\label{section:als-netflix}
The rapid advancements in sensor technology and computer hardware have led to an explosion in the volume of data, presenting new challenges for data analysis. This data is often vast, noisy, and distorted, necessitating preprocessing to enable effective scientific inference. 
For instance, signals captured by antenna arrays are frequently contaminated by noise and other forms of degradation. To analyze such data effectively, it is essential to reconstruct or represent it in a manner that reduces inaccuracies while adhering to feasibility conditions.

In many cases, data collected from complex systems arises from multiple interrelated variables acting in unison. When these variables are not well-defined, the original data may contain overlapping or ambiguous information. 
By constructing a simplified system model, it is possible to achieve a level of accuracy comparable to that of the original system. A common approach to noise reduction, model simplification, data compression, and reconstruction is to replace the original data with a lower-dimensional representation obtained through subspace approximation.
As a result, \textit{low-rank matrix approximations (LRMA) or low-rank matrix decompositions} play a central role in many applications, such as data compression, feature selection, and noise filtering.~\footnote{Strictly speaking, the term ``approximation" usually refers to representing a matrix $\bA$ as $\bA \approx \bW\bZ$, where $\bW$ and $\bZ$ are matrices whose product approximates $\bA$. Conversely, the term ``decomposition" typically implies that $\bA$ is exactly represented as $\bA = \bW\bZ$. In this context, however, we use the terms approximation and decomposition interchangeably to refer to both exact and approximate matrix representations.}

Low-rank matrix decomposition is a powerful tool in machine learning and data mining for expressing a given matrix as the product of two or more matrices with lower dimensions. It captures the essential structure of a matrix while filtering out noise and redundancies. Common methods for low-rank matrix decomposition include singular value decomposition (SVD), principal component analysis (PCA), multiplicative update nonnegative matrix factorization (NMF), and the alternating least squares (ALS) approach, which will be introduced in this section.

\subsection*{Example: The Netflix Prize}

For example, in the Netflix Prize competition \citep{bennett2007netflix}, the goal is to predict the ratings of users for different movies, given the existing ratings (resp., interaction) of those users for other movies (resp., items).
We index $M$ movies with $m= 1, 2,\ldots,M$ and $N$ users
with $n = 1, 2,\ldots,N$. (In the matrix approximation context, lowercase letters e.g., $m,n,k$, are used for the subscripts in running indices, while  uppercase letters $M, N, K$ denote the upper bound of an index.) We denote the rating of the $n$-th user for the $m$-th movie by $a_{mn}$. 
Define $\bA$ as an $M \times N$ rating matrix (a \textit{movie-by-user matrix}) with columns $\{\ba_n\} \in \real^M$, 
each representing the ratings provided by the $n$-th user (also referred to as the \textit{preference matrix}). Note that many ratings $\{a_{mn}\}$ are missing, and our goal is to predict these missing ratings accurately, i.e., to complete the matrix.

It is  clear that without some inherent structure in the matrix, and consequently in the way users rate items, there would be no relationship between the observed and unobserved entries. This would mean there is no unique method to complete the matrix. Therefore, it is crucial to impose some structure on the matrix. A common structural assumption is that of low rank: we aim to fill in the missing entries of matrix $\bA$, assuming $\bA$ is a low-rank matrix. This assumption makes the problem well-posed and allows for a unique solution to some extent, as the low-rank structure establishes connections between the matrix entries (i.e., a \textit{matrix completion} problem). Consequently, the unobserved entries can no longer be independent of the observed values.~\footnote{It is worth noting that the low-rank assumption can be quite strong. For example, consider a rank-$r$ matrix $\bA = \sum_{i=1}^r \be_i \widetilde{\be}_j^\top$, where $\be_i$ and $\widetilde{\be}_j$ are the standard bases for $\real^M$ and $\real^N$, respectively. Such a matrix contains only $r$ nonzero entries. In real-world recommendation systems, we typically observe only a small fraction of matrix entries, which introduces the possibility that some entries may never be observed. This poses a significant challenge for matrix completion, but this topic is beyond the scope of this book.} 
It is important to note that, except for very special data structures, a matrix cannot be compressed/decomposed without incurring some compression error, since a low-rank matrix representation is only an approximation of the original matrix.
This procedure, often known as \textit{collaborative filtering}, seeks to exploit co-occurring patterns in the observed behaviors across users in order to predict future  behaviors of users.

\index{Matrix completion}
\index{Collaborative filtering}
\subsection*{Matrix Completion Formulation}
Consider the \textit{mask matrix} $\bM\in \{0,1\}^{M\times N}$, where $m_{mn}\in \{0,1\}$ indicates whether  user $n$ has rated  movie $m$ or not.
Then the low-rank matrix completion problem can be formulated as 
\begin{equation}
\widetildebA = \mathop{\argmin}_{\bX\in\real^{M\times N}} \sum_{m,n=1}^{M,N} (x_{mn} - a_{mn})^2\cdot m_{mn}\gap \text{s.t.} \gap \rank(\bX)\leq K.
\end{equation}
However, this problem is NP-hard (non-deterministic polynomial) \citep{hardt2014computational}. While it can be equivalently written (proof from singular value decomposition) in the following  unconstrained form:
\begin{equation}
\widetildebA =\widetildebW\widetildebZ = \mathop{\argmin}_{\substack{\bW\in\real^{M\times K}\\ \bZ\in\real^{K\times N}}} \sum_{m,n=1}^{M,N} ((\bW\bZ)_{mn}- a_{mn})^2\cdot m_{mn},
\end{equation}
which allows for indirect solution or approximation using alternate algorithms.

We then formally consider algorithms for solving the following problem: The matrix $\bA$ is approximately factorized into an $M\times K$ matrix $\bW$ and a $K \times  N$ matrix $\bZ$. 
Typically, $K$ is selected to be smaller than both $M$ and $N$, ensuring that $\bW$ and $\bZ$ have reduced dimensions compared to the original  matrix $\bA$. 
This reduction in dimensionality results in a compressed representation of the original data matrix. 
An appropriate decision on the value of $K$ is critical in practice; but the choice of $K$ is very
often problem-dependent.
The factorization is significant in the sense that if $\bA=[\ba_1, \ba_2, \ldots, \ba_N]$ and $\bZ=[\bz_1, \bz_2, \ldots, \bz_N]$ are the column partitions of $\bA$ and $\bZ$, respectively, then we have $\ba_n = \bW\bz_n$. This means each column $\ba_n$ is approximated by a linear combination of the columns of $\bW$, weighted by the components in $\bz_n$. 
Therefore, the columns of $\bW$ can be thought of as containing the column basis (\textit{template columns}, or the approximation of the column basis) of $\bA$; and $\bz_n$ indicates the coordinates (or \textit{activations}) of $\ba_n$ in the basis $\bW$. 
This concept is similar to the factorization methods discussed in the data interpretation part (Part~\ref{part:data-interation}). 
The key difference is that we do not restrict  $\bW$ to consist of exact columns from $\bA$.

\index{Cross-validation}

\index{Two-block coordinate descent}
\begin{algorithm}[H] 
\caption{2-Block Coordinate Descent: Framework of Most ALS and NMF Algorithms}
\label{alg:two_bcd_gen_inals}
\begin{algorithmic}[1] 
\Require A loss function for a variable with two blocks $\bX=(\bW,\bZ)$: $f(\bX)=f(\bW,\bZ)$, and data matrix $\bA$;
\Ensure Constraint on $\bW$ and $\bZ$;
\State Generate some initial matrices $\bW^{(0)}$ and $\bZ^{(0)}$;
\For{$t = 1, 2, \ldots$}
\State $\bW^{(t)} \leftarrow \text{update}\big(\bA, \bZ^{(t-1)}, \bW^{(t-1)}\big)$;
\State $\bZ^{(t)} \leftarrow \text{update}\big(\bA, \bW^{(t)}, \bZ^{(t-1)}\big)$;
\EndFor
\end{algorithmic} 
\end{algorithm}
However, in most cases, the resulting factorization problem has no exact solution, thus requiring optimization procedures to find suitable numerical approximations. The problem is usually solved using a \textit{two-block coordinate descent (2-BCD)} approach (see Algorithm~\ref{alg:two_bcd_gen_inals} for a general illustration).
In order to obtain the approximation $\bA\approx\bW\bZ$, we must establish a loss function such that the distance between $\bA$ and $\bW\bZ$ can be measured. 
In our discussion, the chosen loss function is the Frobenius norm  (a.k.a., the Euclidean distance, Definition~\ref{definition:frobernius-in-svd}) between two matrices, which vanishes to zero if $\bA=\bW\bZ$, and its advantages will become evident shortly.

To simplify the problem, let's first assume that there are no missing ratings. 
We project the data vectors $\ba_n\in\real^M$ into a lower dimension $\bz_n \in \real^K$  with $K<\min\{M, N\}$
in a way that the \textit{reconstruction error} (a.k.a., \textit{criterion function, objective function, cost function, or loss function}) as measured by the Frobenius norm (a.k.a., sum of squared loss) is minimized (assume $K$ is known):
\begin{equation}\label{equation:als-per-example-loss2}
L(\bW,\bZ)= D(\bA,\bW\bZ) = \frac{1}{2}\sum_{n=1}^N \sum_{m=1}^{M} \left(a_{mn} - \bw_m^\top\bz_n\right)^2 
=\frac{1}{2} \normf{\bW\bZ-\bA}^2,~\footnote{
Note that we include a scaling factor of $\frac{1}{2}$ for easier discussion of gradients. Minimizing over $\frac{1}{2}\normf{\bW\bZ-\bA}^2$ is equivalent to minimizing over $\normf{\bW\bZ-\bA}^2$ or $\normf{\bW\bZ-\bA}$.
The choice of the Frobenius norm assumes i.i.d. Gaussian noise on the data ($\bA=\bW\bZ+\bN$, where each entry of $\bN$ follows i.i.d. Gaussian noise) and leads to a smooth optimization via least squares. When the loss is measured by the $\ell_1$ matrix norm, one obtains a robust low-rank matrix factorization; and the noise is assumed i.i.d. Laplace. See \citet{lu2021numerical} for more details.}
\end{equation}
where $\bW=[\bw_1^\top; \bw_2^\top; \ldots; \bw_M^\top]\in \real^{M\times K}$ and $\bZ=[\bz_1, \bz_2, \ldots, \bz_N] \in \real^{K\times N}$ contain $\bw_m$'s and $\bz_n$'s as \textbf{rows and columns}, respectively. 
In \eqref{equation:als-per-example-loss2},  $L(\bW,\bZ)$ indicates that it is a loss function w.r.t. $\bW$ and $\bZ$, and $D(\bA, \bW\bZ)$ implies it is a distance/divergence~\footnote{In words, the \textit{distance} $D(\bE,\bF)$ indicates $D(\bE,\bF)=D(\bF,\bE)\geq 0$ and the equality holds if and only if $\bE=\bF$; while the \textit{divergence} holds that  $D(\bE,\bF)\neq D(\bF,\bE)\geq 0$ and the equality holds if and only if $\bE=\bF$.} between $\bA$ and $\bW\bZ$ (we will use the two terms interchangeably  when necessary).

Moreover, the loss function $L(\bW,\bZ)=\frac{1}{2} \normf{\bW\bZ-\bA}^2$ is convex~\footnote{
A set $ \sS \subseteq \real^n $ is \textit{convex} if for all $ \bx, \by \in \sS $ and $ \lambda \in [0, 1] $, the point
$
(1 - \lambda)\bx + \lambda \by
$
also belongs to $ \sS $.

A function $ f: \sS \subseteq \real^n \to \real $ is \textit{convex} on a convex set $ \sS $ if for all $ \bx, \by \in \sS $ and $ \lambda \in [0, 1] $, it holds that
$
f\big( (1 - \lambda)\bx + \lambda \by \big) \leq (1 - \lambda) f(\bx) + \lambda f(\by).
$
If the inequality is strict for all $ \bx \ne \by $ and $ \lambda \in (0, 1) $, then $ f $ is \textit{strictly convex}.
} concerning $\bZ$ when $\bW$ is held constant, and analogously, convex with respect to $\bW$ when $\bZ$ is fixed. 
This characteristic  motivates an alternating algorithm that alternately fixes one of the variables and
optimizes over the other.
Therefore, we can first minimize the loss with respect to $\bZ$
while keeping $\bW$ fixed, and subsequently minimize it with respect to $\bW$ with $\bZ$ fixed.
This leads to two optimization subproblems, denoted by ALS1 and ALS2, respectively:
$$
\left\{
\begin{aligned}
	\bZ &\leftarrow \mathop{\arg \min}_{\bZ} L(\bW,\bZ);    \qquad \text{(ALS1)} \\ 
	\bW &\leftarrow \mathop{\arg \min}_{\bW} L(\bW,\bZ). \qquad \text{(ALS2)}
\end{aligned}
\right.
$$
This approach is known as the \textit{two-block coordinate descent (2-BCD)  algorithm} as mentioned previously, where we alternate between optimizing the least squares with respect to $\bW$ and $\bZ$. 
Hence, it is also referred to as the \textit{alternating least squares (ALS)} algorithm \citep{comon2009tensor, takacs2012alternating, giampouras2018alternating}. Convergence is guaranteed if the loss function $L(\bW,\bZ)$ decreases at each iteration, and we shall discuss this further  in the sequel.

\index{Coordinate descent algorithm}
\index{Convexity}
\index{Global minimum}
\index{ALS}
\begin{remark}[Convexity and global minimum]
Although the loss function defined by the Frobenius norm $\frac{1}{2}\normf{\bW\bZ-\bA}^2$ is convex either with respect to  $\bW$ when $\bZ$ is fixed or vice versa (called \textit{marginally convex}), it is not \textit{jointly convex} in both variables simultaneously. Therefore, locating the global minimum is generally infeasible. 
However, the algorithm is guaranteed to converge to a local minimum.

More generally, let $D(\bA, \bB)$ be convex in the second argument $\bB$. Then, $D(\bA,\bW\bZ)$ is convex in $\bW$ when $\bZ$ is fixed, and vice versa; see Problem~\ref{prob:sep_conv}.
\end{remark}

\subsection*{Given $\bW$, Optimizing $\bZ$}

Now, let's examine the problem of $\bZ \leftarrow \mathop{\argmin}_{\bZ} L(\bW,\bZ)$. When there exists a unique minimum of the loss function $L(\bW,\bZ)$ with respect to $\bZ$, we refer to it as the \textit{least squares} minimizer of $\mathop{\argmin}_{\bZ} L(\bW,\bZ)$. 
With $\bW$ fixed, $L(\bW,\bZ)$  can be represented as $L(\bZ|\bW)$ (or more compactly, as $L(\bZ)$) to emphasize  its dependence on $\bZ$:
$$
\begin{aligned}
2L(\bZ|\bW) &= \normf{\bW\bZ-\bA}^2= \left\Vert\bW[\bz_1,\bz_2,\ldots, \bz_N]-[\ba_1,\ba_2,\ldots,\ba_N]\right\Vert^2=
\normf{
\scriptsize
\begin{bmatrix}
\bW\bz_1 - \ba_1 \\
\bW\bz_2 - \ba_2\\
\vdots \\
\bW\bz_N - \ba_N
\end{bmatrix}
}^2. 
\end{aligned}
$$
Now, if we define 
$$
\footnotesize
\widetildebW =
\begin{bmatrix}
	\bW & \bzero & \ldots & \bzero\\
	\bzero & \bW & \ldots & \bzero\\
	\vdots & \vdots & \ddots & \vdots \\
	\bzero & \bzero & \ldots & \bW
\end{bmatrix}
\in \real^{MN\times KN}, 
\gap 
\widetildebz=
\begin{bmatrix}
	\bz_1 \\ \bz_2 \\ \vdots \\ \bz_N
\end{bmatrix}
\in \real^{KN},
\gap 
\widetildeba=
\begin{bmatrix}
	\ba_1 \\ \ba_2 \\ \vdots \\ \ba_N
\end{bmatrix}
\in \real^{MN},
$$
then the (ALS1) problem can be reduced to the ordinary least squares problem for minimizing  $\big\Vert{\widetildebW \widetildebz - \widetildeba}\big\Vert_2^2$ with respect to $\widetildebz$. And the solution is given by 
$
\widetildebz = (\widetildebW^\top\widetildebW)^{-1} \widetildebW^\top\widetildeba.
$
However, it is not advisable  to obtain the result using this approach, as computing the inverse of  $\widetildebW^\top\widetildebW$ requires $2(KN)^3$ flops \citep{lu2021numerical}.
Alternatively, a more direct way to solve the  (ALS1)  problem  is to find the gradient of $L(\bZ|\bW)$ with respect to $\bZ$ (assuming all  partial derivatives of this function exist): 
\begin{equation}\label{equation:givenw-update-z-allgd}
\begin{aligned}
\nabla_{\bZ} L(\bZ|\bW) &= \frac{1}{2}
\frac{\partial \,\,\trace\left((\bW\bZ-\bA)(\bW\bZ-\bA)^\top\right)}{\partial \bZ}
\stackrel{\star}{=}  \bW^\top(\bW\bZ-\bA) \in \real^{K\times N},
\end{aligned}
\end{equation}
where the first equality arises from the definition of the Frobenius norm (Definition~\ref{definition:frobernius-in-svd}) such that $\normf{\bA} = \sqrt{\sum_{m=1,n=1}^{M,N} (a_{mn})^2}=\sqrt{\trace(\bA\bA^\top)}$, and the equality ($\star$)  is a consequence of the fact that $\frac{\partial \trace(\bA\bA^\top)}{\partial \bA} = 2\bA$. When the loss function is a differentiable function of $\bZ$, we can determine the least squares solution using differential calculus. 
Since we optimize over an open set $\real^{K\times N}$,  any minimum of the function 
$L(\bZ|\bW)$ must satisfy the condition:
$$
\nabla_{\bZ} L(\bZ|\bW)  = \bzero.
$$
Solving this equation yields the ``candidate" update for $\bZ$ that minimizes $L(\bZ|\bW)$:
\begin{equation}\label{equation:als-z-update}
\textbf{(``Candidate" update for $\bZ$)}: \gap {\bZ = (\bW^\top\bW)^{-1} \bW^\top \bA  \leftarrow \mathop{\arg \min}_{\bZ} L(\bZ|\bW).}
\end{equation}
This computation requires $2K^3$ flops to compute the inverse of $\bW^\top\bW$, compared to $2(KN)^3$ flops to get the inverse of $\widetildebW^\top\widetildebW$.
Prior to confirming that a root of the equation above is indeed a minimizer (as opposed to a maximizer, hence the term ``candidate" update), it is imperative to establish the convexity of the function. 
For a twice continuously differentiable function, this verification  can be equivalently achieved by confirming (see Problem~\ref{problem:pos_hessian} for more details): 
$$
\nabla^2_{\bZ} L(\bZ|\bW) \succ 0.~
\footnote{In short, a twice continuously differentiable function $f$ over an open convex set $\sS$ is called \textit{convex} if and only if $\nabla^2f(\bx)\succeq  \bzero $ for any $\bx\in \sS$ (sufficient and necessary for convex); and called \textit{strictly convex} if $\nabla^2f(\bx)\succ \bzero$ for any $\bx\in \sS$ (only sufficient for strictly convex, e.g., $f(x)=x^6$ is strictly convex, but $f^{\prime\prime}(x)=30x^4$ is equal to zero at $x=0$.). 
And when the convex function $f$ is a continuously differentiable function over a convex set $\sS$, the stationary point $\nabla f(\bx^\star)=\bzero$ of $\bx^\star\in\sS$ is  a \textit{global minimizer} of $f$ over $\sS$.
In our context, when given $\bW$ and updating $\bZ$, the function is defined over the entire space $\real^{K\times N}$.
}
$$
That is, the Hessian matrix is positive definite (Definition~\ref{definition:psd-pd-defini}; see, for example, \citet{beck2014introduction}). To demonstrate this, we explicitly express the Hessian matrix as
\begin{equation}\label{equation:als-z-update_hessian}
\nabla^2_{\bZ} L(\bZ|\bW)= \widetildebW^\top\widetildebW \in \real^{KN\times KN},
~\footnote{
A block-diagonal matrix whose block matrix on the diagonal is $\bW^\top\bW$. And it can be equivalently denoted as $\nabla^2_{\bZ} L(\bZ|\bW) = \diag(\bW,\bW,\ldots,\bW)^\top\diag(\bW,\bW,\ldots,\bW)$.
Using the Kronecker product ``$\kronecker$", this can be equivalently written as $\nabla^2_{\bZ} L(\bZ|\bW) = \bI_{N} \kronecker (\bW^\top\bW)$, where $\bI_N$ is the $N\times N$ identity matrix.
}
\end{equation}
which maintains full rank if $\bW\in \real^{M\times K}$ has full rank  and $K<M$ (Lemma~\ref{lemma:rank-of-ata}).

\begin{remark}[Positive definite Hessian if $\bW$ has full rank]
We  claim that if $\bW\in\real^{M\times K}$ has full rank $K$ with $K<M$, then $\nabla_{\bZ}^2 L(\bZ|\bW)$ is positive definite. This can be demonstrated by confirming that when $\bW$ has full rank, the equation $\bW\bx=\bzero$  holds true only when $\bx=\bzero$, since the null space of $\bW$ has dimension zero. Therefore, 
$$
\bx^\top (\bW^\top\bW)\bx >0, \qquad \text{for any nonzero vector $\bx\in \real^K$}.
$$ 
And this in turn implies $\widetildebW^\top\widetildebW \succ \bzero $.
\end{remark}
Now, the problem becomes  \textcolor{black}{\textbf{whether $\bW$ has full rank so that the Hessian of $L(\bZ|\bW)$ is positive definite}}; otherwise, we cannot claim the update of $\bZ$ in Equation~\eqref{equation:als-z-update} reduces the loss (due to convexity) so that the matrix decomposition progressively improves the approximation of the original matrix $\bA$ by $\bW\bZ$ in each iteration.
We will address the positive definiteness of the Hessian matrix shortly, relying on the following lemma.
\begin{lemma}[Rank of $\bZ$ after updating]\label{lemma:als-update-z-rank}
	Suppose $\bA\in \real^{M\times N}$ has full rank with \textcolor{mylightbluetext}{$M\leq N$} and $\bW\in \real^{M\times K}$ has full rank with $K<M$ (i.e., $K<M\leq N$). Then the update of $\bZ=(\bW^\top\bW)^{-1} \bW^\top \bA \in \real^{K\times N}$ in Equation~\eqref{equation:als-z-update} has full rank.
\end{lemma}
\begin{proof}[of Lemma~\ref{lemma:als-update-z-rank}]
Since $\bW^\top\bW\in \real^{K\times K}$ has full rank if $\bW$ has full rank (Lemma~\ref{lemma:rank-of-ata}), it follows that $(\bW^\top\bW)^{-1} $ has full rank. 

Suppose $\bW^\top\bx=\bzero$. This implies that $(\bW^\top\bW)^{-1} \bW^\top\bx=\bzero$. Thus, the following two null spaces satisfy:
$
\nspace(\bW^\top) \subseteq \nspace\left((\bW^\top\bW)^{-1} \bW^\top\right).
$
Moreover, suppose $\bx$ lies in the null space of $(\bW^\top\bW)^{-1} \bW^\top$ such that $(\bW^\top\bW)^{-1} \bW^\top\bx=\bzero$. And since $(\bW^\top\bW)^{-1} $ is invertible, it implies $ \bW^\top\bx=(\bW^\top\bW)\bzero=\bzero$, leading to
$
\nspace\left((\bW^\top\bW)^{-1} \bW^\top\right)\subseteq \nspace(\bW^\top).
$
Consequently, through ``sandwiching," it follows that 
\begin{equation}\label{equation:als-z-sandiwch1}
\nspace(\bW^\top) = \nspace\left((\bW^\top\bW)^{-1} \bW^\top\right).
\end{equation}
Therefore, $(\bW^\top\bW)^{-1} \bW^\top$ has full rank $K$. Let $\bT=(\bW^\top\bW)^{-1} \bW^\top\in \real^{K\times M}$, and suppose $\bT^\top\bx=\bzero$. This implies $\bA^\top\bT^\top\bx=\bzero$, yielding 
$
\nspace(\bT^\top) \subseteq \nspace(\bA^\top\bT^\top).
$
Similarly, suppose $\bA^\top(\bT^\top\bx)=\bzero$. Since $\bA$ has full rank with the dimension of the null space being 0: $\dim\left(\nspace(\bA^\top)\right)=0$, $(\bT^\top\bx)$ must be zero. The claim follows  since $\bA$ has full rank $M$ with the row space of $\bA^\top$ being equal to the column space of $\bA$, where $\dim\left(\cspace(\bA)\right)=M$ and  $\dim\left(\nspace(\bA^\top)\right) = M-\dim\left(\cspace(\bA)\right)=0$. 
Consequently, $\bx$ is in the null space of $\bT^\top$ if $\bx$ is in the null space of $\bA^\top\bT^\top$:
$
\nspace(\bA^\top\bT^\top)\subseteq \nspace(\bT^\top).
$
By ``sandwiching" again, we obtain
\begin{equation}\label{equation:als-z-sandiwch2}
\nspace(\bT^\top) = \nspace(\bA^\top\bT^\top).
\end{equation}
Since $\bT^\top$ has full rank $K<M\leq N$, it follows that $\dim\left(\nspace(\bT^\top) \right) = \dim\left(\nspace(\bA^\top\bT^\top)\right)=0$.
Therefore,
$\bZ^\top=\bA^\top\bT^\top$ has full rank $K$.
We complete the proof.
\end{proof}

\subsection*{Given $\bZ$, Optimizing $\bW$}
The matrix factorization problem exhibits symmetry: $\bA=\bW\bZ$ if and only if $\bA^\top=\bZ^\top\bW^\top$ such that $D(\bA, \bW\bZ)=D(\bA^\top, \bZ^\top\bW^\top)$. The analysis of optimizing $\bW$ given $\bZ$ directly follows  from the previously discussed methodology.
Below, we provide a brief outline of the results.
With $\bZ$ fixed, $L(\bW,\bZ)$ can be expressed as $L(\bW|\bZ)$ (or more compactly, as $L(\bW)$)  to emphasize  the dependence on $\bW$:
$
\begin{aligned}
	L(\bW|\bZ) &= \frac{1}{2}\normf{\bW\bZ-\bA}^2.
\end{aligned}
$
To solve the optimization problem (ALS2) directly, we  compute the gradient of  $L(\bW|\bZ)$ with respect to $\bW$:
\begin{equation}\label{equation:givenz-update-w-allgd}
\begin{aligned}
	\nabla_{\bW} L(\bW|\bZ) &= \frac{1}{2}
	\frac{\partial \trace\left((\bW\bZ-\bA)(\bW\bZ-\bA)^\top\right)}{\partial \bW}
	= (\bW\bZ-\bA)\bZ^\top \in \real^{M\times K}.
\end{aligned}
\end{equation}
Similarly, the ``candidate" update for  $\bW$ can be obtained by identifying the root of the gradient $\nabla_{\bW} L(\bW|\bZ)$:
\begin{equation}\label{equation:als-w-update}
\textbf{(``Candidate" update for $\bW$)}:\gap{\bW^\top = (\bZ\bZ^\top)^{-1}\bZ\bA^\top  \leftarrow \mathop{\arg\min}_{\bW} L(\bW|\bZ).}
\end{equation}
Once more, we emphasize that the update is merely a ``candidate" update. 
Further validation is necessary  to ascertain the positive definiteness of the Hessian matrix.
The Hessian matrix is given by:
\begin{equation}\label{equation:als-w-update_hessian}
\begin{aligned}
\nabla_{\bW}^2 L(\bW|\bZ) =\widetildebZ\widetildebZ^\top \in \real^{KM\times KM},
\end{aligned}
\end{equation}
where $\widetildebZ=\diag(\bZ,\bZ,\ldots,\bZ)\in\real^{KM\times NM}$ is defined analogously to $\widetildebW$ in \eqref{equation:als-z-update_hessian}.
Therefore, by similar reasoning, if $\bZ$ has full rank with $K<N$, the Hessian matrix is positive definite.

In Lemma~\ref{lemma:als-update-z-rank}, we proved that $\bZ$ has full rank under certain conditions, ensuring that  the Hessian matrix in Equation~\eqref{equation:als-w-update_hessian} is positive definite, and the update in Equation~\eqref{equation:als-w-update} exists.
We now prove that $\bW$ also has full rank under certain conditions, such that the Hessian in Equation~\eqref{equation:als-z-update_hessian} is positive definite, and the update in  Equation~\eqref{equation:als-z-update} exists.
\begin{lemma}[Rank of $\bW$ after updating]\label{lemma:als-update-w-rank}
Suppose $\bA\in \real^{M\times N}$ has full rank with \textcolor{mylightbluetext}{$M\geq N$} and $\bZ\in \real^{K\times N}$ has full rank with $K<N$ (i.e., $K<N\leq M$). Then the update of $\bW^\top = (\bZ\bZ^\top)^{-1}\bZ\bA^\top$ in Equation~\eqref{equation:als-w-update} has full rank.
\end{lemma}
The proof of Lemma~\ref{lemma:als-update-w-rank} follows the same reasoning as that of Lemma~\ref{lemma:als-update-z-rank}, so we omit the details.

\paragraph{Key observation.}
Combining the observations from Lemmas~\ref{lemma:als-update-z-rank} and \ref{lemma:als-update-w-rank}, as long as we \textbf{initialize $\bZ$ and $\bW$ to have full rank}, the updates in Equations~\eqref{equation:als-z-update}  and \eqref{equation:als-w-update} are well-defined \textbf{since the Hessians in Equations~\eqref{equation:als-z-update_hessian} and \eqref{equation:als-w-update_hessian} are positive definite}. 
\textbf{Note that we need an additional condition to satisfy  
both Lemma~\ref{lemma:als-update-z-rank} 
and Lemma~\ref{lemma:als-update-w-rank}: $M=N$, meaning there must be an equal number of movies and   users.} 
We will relax this condition in the next section through regularization.
(Alternatively,  Problems~\ref{prob:als_pseudo1}$\sim$\ref{prob:als_pseudon} relax this condition using the pseudo-inverse.)
We summarize the process in Algorithm~\ref{alg:als}.
Since the loss $\frac{1}{2}\normf{\bA-\bW\bZ}^2$ in each iteration is monotonically nonincreasing and bounded below, it converges \citep{lu2021numerical, gillis2020nonnegative}.
In particular, $\nabla_{\bZ} L(\bZ|\bW)$ and $\nabla_{\bW} L(\bW|\bZ)$ approach zero when $iter\rightarrow \infty$.

\begin{algorithm}[H] 
\caption{Alternating Least Squares}
\label{alg:als}
\begin{algorithmic}[1] 
\Require Matrix $\bA\in \real^{M\times N}$ \textcolor{mylightbluetext}{with $M= N$};
\State Initialize $\bW\in \real^{M\times K}$, $\bZ\in \real^{K\times N}$ \textcolor{mylightbluetext}{with full rank and $K<M= N$}; 
\State Choose a stop criterion on the approximation error $\delta$;
\State Choose the maximal number of iterations $C$;
\State $iter=0$; \Comment{Count for the number of iterations}
\While{$\normf{\bA-\bW\bZ}>\delta $ and $iter<C$} 
\State $iter=iter+1$;
\State $\bZ \leftarrow (\bW^\top\bW)^{-1} \bW^\top \bA  \leftarrow \mathop{\arg \min}_{\bZ} L(\bZ|\bW)$;
\State $\bW^\top \leftarrow (\bZ\bZ^\top)^{-1}\bZ\bA^\top  \leftarrow \mathop{\arg\min}_{\bW} L(\bW|\bZ)$;
\EndWhile
\State Output $\bW,\bZ$.
\end{algorithmic} 
\end{algorithm}

\index{Regularization}
\section{Regularization and Identifiability: Extension to General Matrices}\label{section:regularization-extention-general}

\textit{Regularization} is a machine learning technique employed to prevent overfitting and improve the generalization of models.  
Overfitting occurs when a model is overly complex and fits the training data too closely, resulting in poor performance on  unseen data. 
To mitigate this issue, regularization introduces a constraint or a penalty term into the loss function used for model optimization, discouraging the development of overly complex models. 
This creates  a trade-off between having a simple, generalizable model and fitting the training data well. 
Common types of regularization include $\ell_1$ regularization, $\ell_2$ regularization (Tikhonov regularization), and elastic net regularization (a combination of $\ell_1$ and $\ell_2$ regularizations). 
Regularization finds extensive applications in machine learning algorithms such as linear regression, logistic regression, and neural networks.

In the context of the alternating least squares problem, we can incorporate an $\ell_2$ regularization term  to minimize the following regularized loss function:
\begin{equation}\label{equation:als-regularion-full-matrix}
L(\bW,\bZ)  = \frac{1}{2}\normf{\bW\bZ-\bA}^2 +\frac{1}{2}\lambda_w \normf{\bW}^2 + \frac{1}{2}\lambda_z \normf{\bZ}^2, \qquad \lambda_w>0, \lambda_z>0,
\end{equation}
where the gradient with respect to $\bZ$ and $\bW$ are given, respectively, by 
\begin{equation}\label{equation:als-regulari-gradien}
\left\{
\begin{aligned}
\nabla_{\bZ} L(\bZ|\bW) &= \bW^\top(\bW\bZ-\bA) + \lambda_z\bZ \in \real^{K\times N};\\
\nabla_{\bW} L(\bW|\bZ)  &= (\bW\bZ-\bA)\bZ^\top + \lambda_w\bW \in \real^{M\times K}.
\end{aligned}
\right.
\end{equation}
The corresponding Hessian matrices are given, respectively, by 
$$
\left\{
\begin{aligned}
\nabla^2_{\bZ} L(\bZ|\bW) &= \widetildebW^\top\widetildebW+ \lambda_z\bI \in \real^{KN\times KN};\\
\nabla^2_{\bW} L(\bW|\bZ)  &= \widetildebZ\widetildebZ^\top + \lambda_w\bI \in \real^{KM\times KM}, \\
\end{aligned}
\right.
$$
which are positive definite due to the perturbation by the regularization:
$$
\left\{
\begin{aligned}
\bx^\top (\widetildebW^\top\widetildebW +\lambda_z\bI)\bx 
&= \underbrace{\bx^\top\widetildebW^\top\widetildebW\bx}_{\geq 0} + \lambda_z \normtwo{\bx}^2>0, \gap \text{for nonzero $\bx$};\\
\bx^\top (\widetildebZ\widetildebZ^\top +\lambda_w\bI)\bx 
&= \underbrace{\bx^\top\widetildebZ\widetildebZ^\top\bx}_{\geq 0} + \lambda_w \normtwo{\bx}^2>0,\gap \text{for nonzero $\bx$}.
\end{aligned}
\right.
$$
\textbf{The regularization ensures that the Hessian matrices remain positive definite, even if $\bW$ and $\bZ$ are rank-deficient}. 
Consequently, matrix decomposition can be extended to any matrix, regardless of whether $M>N$ or $M<N$. In rare cases, $K$ even can be chosen as $K>\max\{M, N\}$ to obtain a \textit{high-rank approximation} of $\bA$. However, in most scenarios, we aim to find a \textit{low-rank approximation} of $\bA$ with $K<\min\{M, N\}$. 
Therefore, the minimizers can be determined by identifying the roots of the gradients:
\begin{equation}\label{equation:als-regular-final-all}
\left.
\begin{aligned}
\bZ &= (\bW^\top\bW+ \lambda_z\bI)^{-1} \bW^\top \bA 
\gap \text{and}\gap 
\bW^\top = (\bZ\bZ^\top+\lambda_w\bI)^{-1}\bZ\bA^\top .
\end{aligned}
\right.
\end{equation}
The regularization parameters $\lambda_z, \lambda_w\in \real_{++}$ are used to balance the trade-off
between the accuracy of the approximation and the smoothness of the computed solution. The selection of these parameters is typically problem-dependent and can be determined through \textit{cross-validation} (CS). Again, we summarize the regularized ALS procedure in Algorithm~\ref{alg:als-regularizer}.
We will also introduce the \textit{alternating direction methods of multipliers  (ADMM)} for solving  matrix factorization problems with $\ell_2$ or $\ell_1$ regularization in Section~\ref{section:nmf_admm_all}, where the method can be extended to other types of regularizations and constraints, such as  nonnegativity   constraints.

The $\ell_2$ (or $\ell_1$ ) regularizations can be applied to generalize the ALS problem to general matrices.
However, we will consider the case where some entries of the matrix $\bA$ are missing. This leads to the matrix completion problem.
In this sense, the $\ell_1$ and $\ell_2$ regularizations  are not the only applicable regularizations; for example, the \textit{nuclear norm} \footnote{Also called the \textit{Schatten 1-norm} or \textit{trace norm}.} of $\bW\bZ$ (the sum of singular values of the matrix) can be applied, for which the \textit{Soft-Impute for matrix completion} algorithm guarantees the recovery of the matrix when the number of observed entries $z$ satisfies
$
z\geq C r n \log n,
$
where the underlying matrix $\bA$ is of size $\real^{n\times n}$ and $C > 0$ is a fixed universal constant \citep{gross2011recovering, hastie2015statistical}. 
However, the $\ell_2$ regularization on $\bW$ and $\bZ$ can somehow be reformulated into the nuclear norm form (see Problem~\ref{problem:nuclear_equi}).

\index{Cross-validation}
\begin{algorithm}[H] 
\caption{Alternating Least Squares with Regularization}
\label{alg:als-regularizer}
\begin{algorithmic}[1] 
\Require Matrix $\bA\in \real^{M\times N}$;
\State Initialize $\bW\in \real^{M\times K}$, $\bZ\in \real^{K\times N}$ \textcolor{mylightbluetext}{randomly without condition on the rank and the relationship between $M, N, K$}; 
\State Choose a stop criterion on the approximation error $\delta$;
\State Choose regularization parameters $\lambda_w, \lambda_z$;
\State Choose the  maximal number of iterations $C$;
\State $iter=0$; \Comment{Count for the number of iterations}
\While{$\normf{\bA-\bW\bZ}>\delta $ and $iter<C$}
\State $iter=iter+1$; 
\State $\bZ \leftarrow (\bW^\top\bW+ \lambda_z\bI)^{-1} \bW^\top \bA  \leftarrow \mathop{\arg \min}_{\bZ} L(\bZ|\bW)$;
\State $\bW^\top \leftarrow (\bZ\bZ^\top+\lambda_w\bI)^{-1}\bZ\bA^\top  \leftarrow \mathop{\arg\min}_{\bW} L(\bW|\bZ)$;
\EndWhile
\State Output $\bW,\bZ$.
\end{algorithmic} 
\end{algorithm}

\paragraph{Regularization as constraints and identifiability.} 
Regularization terms, such as $\lambda_w\normf{\bW}^2$ in \eqref{equation:als-regularion-full-matrix}, can be interpreted as  constraints like $\normf{\bW}\leq C$, where $C$ is a constant, via Lagrangian multipliers (see, for example, \citet{boyd2004convex} or Section~\ref{section:reg_geom_inter}). 
Different constraints can be placed on the factors $\bW$ and $\bZ$. For example, the nonnegativity constraint discussed in Chapter~\ref{chapter:nmf} and the sparsity constraint discussed in Section~\ref{section:reg_geom_inter}.
Moreover, the two matrices $\bW\in\real^{M\times K}$ and $\bZ\in\real^{K\times N}$ have $(M+N)K$ degrees of freedom.
However, due to the scaling degree of freedom of the columns of $\bW$ and rows of $\bZ$ in $\bA=\bW\bZ$, the factorization $\bW\bZ$ has $(M+N-1)K$ degrees of freedom: $\bW[:, k]\bZ[k,:] = (\gamma\bW[:, k])(\frac{1}{\gamma}\bZ[k,:])$ for any scalar $\gamma\neq 0$ and $k\in\{1,2,\ldots,K\}$.
Therefore, the factorization is not identifiable.  Regularization helps reduce overfitting and addresses the issue of identifiability by incorporating prior information through constraints.

\index{Missing entries}
\index{Hadamard product}
\index{Netflix recommender}
\section{Missing Entries and Rank-One Update}\label{section:alt-columb-by-column}
Matrix decomposition via  ALS is extensively used in the context of Netflix-style recommender data, where a substantial number of entries are missing due to users not having watched certain movies or choosing not to rate them for various reasons.
In this scenario, the low-rank matrix decomposition problem is also known as \textit{matrix completion} that can help recover unobserved entries \citep{jain2017non}.
To model this, we can introduce an additional mask matrix $\bM\in \{0,1\}^{M\times N}$, where  each entry $m_{mn}\in \{0,1\}$ indicates whether  user $n$ has rated  movie $m$ or not. 
Using this mask, the loss function can be defined as:
$$
L(\bW,\bZ) = \frac{1}{2}\normf{\bM\hadaprod  \bA- \bM\hadaprod (\bW\bZ)}^2,
$$
where $\hadaprod$ represents the \textit{Hadamard product} between matrices. 
The above formulation concisely expresses our goal of finding a completion of the ratings matrix that is both of low rank and consistent with observed user ratings.
To find the solution to this problem, we decompose the updates in Equation~\eqref{equation:als-regular-final-all} into:
\begin{equation}\label{equation:als-ori-all-wz}
	\left\{
	\begin{aligned}
		\bz_n &= (\bW^\top\bW+ \lambda_z\bI)^{-1} \bW^\top \ba_n, &\gap& \text{for $n\in \{1,2,\ldots, N\}$}  ;\\
		\bw_m &= (\bZ\bZ^\top+\lambda_w\bI)^{-1}\bZ\bb_m,  &\gap& \text{for $m\in \{1,2,\ldots, M\}$} ,
	\end{aligned}
	\right.
\end{equation}
where $\bZ=[\bz_1, \bz_2, \ldots, \bz_N]$ and $\bA=[\ba_1,\ba_2, \ldots, \ba_N]$ represent the column partitions of $\bZ$ and $\bA$, respectively. Similarly, $\bW^\top=[\bw_1, \bw_2, \ldots, \bw_M]$ and $\bA^\top=[\bb_1,\bb_2, \ldots, \bb_M]$ are the column partitions of $\bW^\top$ and $\bA^\top$, respectively. This decomposition of the updates indicates that the updates can be performed in a column-by-column fashion (the rank-one updates).

\paragraph{Given $\bW$.}
Let $\bo_n\in \{0,1\}^M$ represent the movies rated by user $n$, where $o_{nm}=1$ if user $n$ has rated movie $m$, and $o_{nm}=0$ otherwise. Then the $n$-th column of $\bA$ without missing entries can be denoted using the Matlab-style notation as $\ba_n[\bo_n]$. 
And we want to approximate the existing entries of the $n$-th column by $\ba_n[\bo_n] \approx \bW[\bo_n, :]\bz_n$, which is indeed a rank-one least squares problem:
\begin{equation}\label{equation:als-ori-all-wz-modif-z}
\begin{aligned}
\bz_n &= \left(\bW[\bo_n, :]^\top\bW[\bo_n, :]+ \lambda_z\bI\right)^{-1} \bW[\bo_n, :]^\top \ba_n[\bo_n], \quad \text{for $n\in \{1,2,\ldots, N\}$} .
\end{aligned}
\end{equation}
Moreover, the loss function with respect to $\bz_n$ and $\bZ$  can be described, respectively, by
$$
\begin{aligned}
L(\bz_n|\bW) &=\sum_{m\in \bo_n} \left(a_{mn} - \bw_m^\top\bz_n\right)^2
\gap \text{and}\gap
L(\bZ|\bW) =\sum_{n=1}^N\ \sum_{m\in \bo_n} \left(a_{mn} - \bw_m^\top\bz_n\right)^2.
\end{aligned}
$$

\paragraph{Given $\bZ$.}
Similarly, if $\bp_m \in\{0,1\}^{N}$ denotes the users who have rated  movie $m$, with $p_{mn}=1$ if  movie $m$ has been rated by user $n$, and $p_{mn}=0$ otherwise. Then the $m$-th row of $\bA$ without missing entries can be denoted by the Matlab-style notation as $\bb_m[\bp_m]$. We want to approximate the existing entries of the $m$-th row by $\bb_m[\bp_m] \approx \bZ[:, \bp_m]^\top\bw_m$, 
\footnote{Note that $\bZ[:, \bp_m]^\top$ is the transpose of $\bZ[:, \bp_m]$, which is equal to $\bZ^\top[\bp_m,:]$, i.e., transposing first and then selecting.}
which  is again a rank-one least squares problem:
\begin{equation}\label{equation:als-ori-all-wz-modif-w}
\begin{aligned}
\bw_m &= (\bZ[:, \bp_m]\bZ[:, \bp_m]^\top+\lambda_w\bI)^{-1}\bZ[:, \bp_m]\bb_m[\bp_m],  \quad \text{for $m\in \{1,2,\ldots, M\}$} .
\end{aligned}
\end{equation}
Similarly, the loss function with respect to $\bw_m$ and $\bW$ can be described, respectively,  by
$$
\begin{aligned}
L(\bw_m|\bZ) &=\sum_{n\in \bp_m} \left(a_{mn} - \bw_m^\top\bz_n\right)^2 
\gap \text{and}\gap
L(\bW|\bZ) =\sum_{m=1}^M  \sum_{n\in \bp_m} \left(a_{mn} - \bw_m^\top\bz_n\right)^2 .
\end{aligned}
$$
The procedure is once again presented in Algorithm~\ref{alg:als-regularizer-missing-entries}.
Other approaches, such as \textit{singular value projection (SVP)}, also exist to address the matrix completion problem. At a high level, SVP is a type of projected gradient descent (PGD) method that updates iteratively via gradient descent, projecting the updated matrix into a low-rank form through singular value decomposition at each step. However, the alternating least squares approach generally  outperforms SVP in the context of matrix completion, so we will not delve into SVP here. For more details, refer to \citet{jain2017non} and the references therein.

\begin{algorithm}[h] 
\caption{Alternating Least Squares with Missing Entries and Regularization}
\label{alg:als-regularizer-missing-entries}
\begin{algorithmic}[1] 
\Require Matrix $\bA\in \real^{M\times N}$;
\State Initialize $\bW\in \real^{M\times K}$, $\bZ\in \real^{K\times N}$ \textcolor{mylightbluetext}{randomly without condition on the rank and the relationship between $M, N, K$}; 
\State Choose a stop criterion on the approximation error $\delta$;
\State Choose regularization parameters $\lambda_w, \lambda_z$;
\State Compute the mask matrix $\bM$ from $\bA$;
\State Choose the maximal number of iterations $C$;
\State $iter=0$; \Comment{Count for the number of iterations}
\While{\textcolor{mylightbluetext}{$\normf{\bM\hadaprod  \bA- \bM\hadaprod (\bW\bZ)}^2>\delta $} and $iter<C$}
\State $iter=iter+1$; 
\For{$n=1,2,\ldots, N$}
\State $\bz_n \leftarrow \left(\bW[\bo_n, :]^\top\bW[\bo_n, :]+ \lambda_z\bI\right)^{-1} \bW[\bo_n, :]^\top \ba_n[\bo_n]$; \Comment{$n$-th column of $\bZ$}
\EndFor

\For{$m=1,2,\ldots, M$}
\State $\bw_m \leftarrow (\bZ[:, \bp_m]\bZ[:, \bp_m]^\top+\lambda_w\bI)^{-1}\bZ[:, \bp_m]\bb_m[\bp_m]$;\Comment{$m$-th column of $\bW^\top$}
\EndFor
\EndWhile
\State Output $\bW^\top=[\bw_1, \bw_2, \ldots, \bw_M],\bZ=[\bz_1, \bz_2, \ldots, \bz_N]$.
\end{algorithmic} 
\end{algorithm}

\index{Hidden features}
\index{Inner product}
\section{Vector Inner Product and Hidden Vectors}\label{section:als-vector-product}
We observe that the ALS algorithm seeks to find lower-dimensional matrices $\bW$ and $\bZ$ such that their product $\bW\bZ$ can approximate $\bA\approx \bW\bZ$ in terms of the  squared loss:
$
\mathop{\min}_{\bW,\bZ}  \sum_{n=1}^N \sum_{m=1}^{M} \left(a_{mn} - \bw_m^\top\bz_n\right)^2.
$
That is, each entry $a_{mn}$ in $\bA$ can be approximated as the inner product of  two vectors: $\bw_m^\top\bz_n$. The geometric interpretation of the vector inner product is given by 
$$
\bw_m^\top\bz_n = \normtwo{\bw_m}\cdot \normtwo{\bz_n} \cos \theta,
$$
where $\theta$ represents the angle between the vectors $\bw_m$ and $\bz_n$. Thus, if the vector norms of $\bw_m$ and $\bz_n$ are determined, a smaller angle between them results in a larger inner product.

In the context of Netflix-style recommendation systems,   movie ratings typically range from 0 to 5,
with higher ratings indicating a stronger user preference for the movie. 
If $\bw_m$ and $\bz_n$ fall sufficiently ``close" in direction, the value of $\bw_m^\top\bz_n$ becomes larger. 
This reflects a stronger match between the user's preferences and the movie's characteristics.

This concept elucidates the essence of ALS, where $\bw_m$ represents the features or attributes of movie $m$, while $\bz_n$ encapsulates  the features or preferences of user $n$. 
In other words,  ALS associates each user with a \textit{latent vector of preference} and each movie with a \textit{latent vector of attributes}.
Furthermore, each element in $\bw_m$ and $\bz_n$ signifies a specific feature. For example, it could be that the second feature $w_{m2}$ ($w_{m2}$ denotes the second element of the vector $\bw_{m}$) represents whether the movie is an action movie or not, and $z_{n2}$ might denote whether  user $n$ has a preference for action movies. When this holds true, then the inner product $\bw_m^\top\bz_n$ becomes large and provides a good approximation of the observed rating $a_{mn}$.

In the matrix decomposition $\bA\approx \bW\bZ$, it is established that the rows of $\bW$ contain the hidden features of the movies, and the columns of $\bZ$ contain the hidden features of the users. 
Nevertheless, the explicit meanings of the rows in $\bW$ or the columns in $\bZ$ remain undisclosed.
Although they might correspond to categories or genres of the movies, fostering underlying connections between users and movies, their precise nature remains uncertain.
It is precisely this ambiguity that gives rise to the terminology ``latent" or ``hidden."

\index{Stochastic gradient descent}
\index{Gradient descent}
\index{Matrix inverse}
\index{Decomposition: LU}
\section{Gradient Descent}\label{section:als-gradie-descent}
In Algorithms~\ref{alg:als}, \ref{alg:als-regularizer}, and \ref{alg:als-regularizer-missing-entries}, we minimize the loss function through the inversion of matrices (e.g., using LU decomposition). 
The reality, however, is frequently far from straightforward, particularly in the big data era of today. As data volumes explode, the size of the inversion matrix will grow at a pace proportional to the cube of the number of samples,  which poses a great challenge to the storage and computational resources.
This complexity has led to the ongoing development of gradient-based optimization techniques.
Among these, the \textit{gradient descent (GD)} method and its variant, the \textit{stochastic gradient descent (SGD)} method, are among  the simplest, fastest, and most efficient methods \citep{lu2022gradient}. These methods are particularly effective for solving convex optimization problems. We now provide a more detailed explanation of their underlying principles.

In Equation~\eqref{equation:als-ori-all-wz}, we derived the column-by-column update rules directly from the full matrix approach outlined in Equation~\eqref{equation:als-regular-final-all} (with regularization taken into account). 
To understand the underlying concept, consider the loss function with regularization, as given by Equation~\eqref{equation:als-regularion-full-matrix}.
When minimizing the  loss in \eqref{equation:als-regularion-full-matrix} with respect to $\bz_n$, we can break down the loss as follows:
\begin{equation}\label{als:gradient-regularization-zn}
\footnotesize
\begin{aligned}
L(\bz_n)  &=\frac{1}{2}\normf{\bW\bZ-\bA}^2 +\frac{1}{2}\lambda_w \normf{\bW}^2 + \frac{1}{2}\lambda_z \normf{\bZ}^2
= \frac{1}{2}\normtwo{\bW\bz_n-\ba_n}^2 + \frac{1}{2}\lambda_z \normtwo{\bz_n}^2 + C_{z_n},
\end{aligned}
\end{equation}
where $C_{z_n}$ is a constant with respect to $\bz_n$, and $\bZ=[\bz_1, \bz_2, \ldots, \bz_N]$ and $\bA=[\ba_1,\ba_2, \ldots, \ba_N]$ represent the column partitions of $\bZ$ and $\bA$, respectively. 
The gradient and the root are given, respectively, by 
$$
\begin{aligned}
\nabla_{\bz_n} L(\bz_n) = \bW^\top\bW\bz_n - \bW^\top\ba_n + \lambda_z\bz_n
\,\,\implies \,\,
\bz_n = (\bW^\top\bW+ \lambda_z\bI)^{-1} \bW^\top \ba_n, \,\,  \forall\,n.
\end{aligned}
$$
This solution corresponds to the first update rule in the column-wise updates of Equation~\eqref{equation:als-ori-all-wz}.
Similarly, when minimizing the loss with respect to $\bw_m$, we have:
\begin{equation}\label{als:gradient-regularization-wd}
\footnotesize
\begin{aligned}
L(\bw_m )  
&
=\frac{1}{2}\normf{\bZ^\top\bW-\bA^\top}^2 +\frac{1}{2}\lambda_w \normf{\bW^\top}^2 + \frac{1}{2}\lambda_z \normf{\bZ}^2
= \frac{1}{2}\normtwo{\bZ^\top\bw_m-\bb_n}^2 + \frac{1}{2}\lambda_w \normtwo{\bw_m}^2 + C_{w_m},
\end{aligned}
\end{equation}
where $C_{w_m}$ is a constant with respect to $\bw_m$, and $\bW^\top=[\bw_1, \bw_2, \ldots, \bw_M]$ and $\bA^\top=[\bb_1,\bb_2, \ldots,$ $\bb_M]$ represent the column partitions of $\bW^\top$ and $\bA^\top$, respectively. 
Analogously, taking the gradient with respect to $\bw_m$, it follows that
$$
\begin{aligned}
\nabla_{\bw_m} L(\bw_m) = \bZ\bZ^\top\bw_m - \bZ\bb_n + \lambda_w\bw_m
\,\,\implies\,\,
\bw_m = (\bZ\bZ^\top+\lambda_w\bI)^{-1}\bZ\bb_m, \,\, \forall \, m.
\end{aligned}
$$
This solution corresponds to the second update rule in the column-wise updates of Equation~\eqref{equation:als-ori-all-wz}:

Now suppose we express the iteration number ($t=1,2,\ldots$)  as the superscript, and we want to find the updates $\{\bz^{(t+1)}_n, \bw^{(t+1)}_m\}$  at the $(t+1)$-th iteration  base on $\{\bZ^{(t)}, \bW^{(t)}\}$  from the $t$-th iteration:
$$
\left.
\begin{aligned}
\bz^{(t+1)}_n    &\leftarrow \mathop{\arg \min}_{\bz_n^{(t)}} L(\bz_n^{(t)})
\qquad\text{and}\qquad
\bw_m^{(t+1)}    \leftarrow \mathop{\arg\min}_{\bw_m^{(t)}} L(\bw_m^{(t)}).
\end{aligned}
\right.
$$
For simplicity, we will only derive for $\bz^{(t+1)}_n    \leftarrow \mathop{\arg \min}_{\bz_n^{(t)}} L(\bz_n^{(t)})$, and the derivation for the update on $\bw_m^{(t+1)}$  follows a similar approach.

\index{Linear approximation}
\index{Linear update}
\index{Greedy search}
\index{Gradient descent}
\paragraph{Approximation by linear update.} 
Suppose we want to approximate $\bz^{(t+1)}_n$ using a \textit{linear update} based on $\bz^{(t)}_n$:
$$
\textbf{(Linear Update)}: \qquad {\bz^{(t+1)}_n = \bz^{(t)}_n + \eta \bv.}
$$
The problem now becomes finding the solution of $\bv$ such that
$$
\bv=\mathop{\arg \min}_{\bv} L(\bz^{(t)}_n + \eta \bv) .
$$
By Taylor's formula, $L(\bz^{(t)}_n + \eta \bv)$ can be approximated by 
$$
L(\bz^{(t)}_n + \eta \bv) \approx L(\bz^{(t)}_n ) + \eta \bv^\top \nabla  L(\bz^{(t)}_n ),
$$
where $\eta$ is a small value, and $\nabla  L(\bz^{(t)}_n )$ represents the gradient of $L(\bz)$ evaluated at $\bz^{(t)}_n$. 
To find $\bv$ under the constraint $\normtwo{\bv}=1$ for a positive $\eta$, we perform the following minimization:
$$
\bv=\mathop{\argmin}_{\normtwo{\bv}=1} L(\bz^{(t)}_n + \eta \bv) 
\approx\mathop{\argmin}_{\normtwo{\bv}=1}
 \left\{L(\bz^{(t)}_n ) + \eta \bv^\top \nabla  L(\bz^{(t)}_n )\right\}.
$$
This strategy is known as  \textit{greedy search}. The optimal $\bv$ can be obtained by 
$$
\bv = -\nabla L(\bz^{(t)}_n )\big/{\big\Vert{\nabla L(\bz^{(t)}_n )}\big\Vert_2},
$$
which means that $\bv$ points in the opposite direction to the gradient $\nabla L(\bz^{(t)}_n )$. Therefore, it is reasonable to update  $\bz_n^{(t+1)}$ as follows:
$$
\bz^{(t+1)}_n =\bz^{(t)}_n + \eta \bv = \bz^{(t)}_n - \eta {\nabla L(\bz^{(t)}_n )}\big/{\big\Vert{\nabla L(\bz^{(t)}_n )}\big\Vert_2},
$$
which is commonly referred to as    \textit{gradient descent} (GD). Similarly, the gradient descent update for $\bw_m^{(t+1)}$ is given by
$$
\bw^{(t+1)}_m =\bw^{(t)}_m + \eta \bv = \bw^{(t)}_m - \eta {\nabla L(\bw^{(t)}_m )}\big/{\big\Vert{\nabla L(\bw^{(t)}_m )}\big\Vert_2}.
$$
The revised procedure for Algorithm~\ref{alg:als-regularizer} employing a gradient descent approach is presented in Algorithm~\ref{alg:als-regularizer-missing-stochas-gradient}.

It's noteworthy that the ALS without GD (Algorithm~\ref{alg:als-regularizer}) does not involve explicit parameters such as step size  $\eta$. 
This characteristic can be both advantageous and disadvantageous. 
On one hand, it absolves the  user from the time-consuming task of fine-tuning parameters, making the method more accessible and less demanding. 
On the other hand, this absence of adjustable parameters also restricts the user's control to directly influence the progression of the algorithm, leaving the convergence of ALS entirely contingent upon the inherent structure of the optimization problem at hand.

In practical applications, it is customary to alternate between the pure ALS iterations outlined in Algorithm~\ref{alg:als-regularizer} and the modified, gradient-descent variants  discussed in this section. These descent-based adaptations offer the user a degree of control through a tunable step length parameter, allowing for a more customized approach to the optimization process.

\begin{algorithm}[h] 
\caption{Alternating Least Squares with Full Entries and Gradient Descent}
\label{alg:als-regularizer-missing-stochas-gradient}
\begin{algorithmic}[1] 
\Require Matrix $\bA\in \real^{M\times N}$;
\State Initialize $\bW\in \real^{M\times K}$, $\bZ\in \real^{K\times N}$ \textcolor{mylightbluetext}{randomly without condition on the rank and the relationship between $M, N, K$}; 
\State Choose a stop criterion on the approximation error $\delta$;
\State Choose regularization parameters $\lambda_w, \lambda_z$, and step sizes $\eta_w, \eta_z$;
\State Choose the maximal number of iterations $C$;
\State $iter=0$; \Comment{Count for the number of iterations}
\While{$\normf{\bA- (\bW\bZ)}^2>\delta $ and $iter<C$}
\State $iter=iter+1$; 
\For{$n=1,2,\ldots, N$}
\State $\bz^{(t+1)}_n \leftarrow\bz^{(t)}_n - \eta_z {\nabla L(\bz^{(t)}_n )}\big/{\big\Vert{\nabla L(\bz^{(t)}_n )}\big\Vert_2}$; \Comment{$n$-th column of $\bZ$}
\EndFor

\For{$m=1,2,\ldots, M$}
\State $\bw^{(t+1)}_m  \leftarrow \bw^{(t)}_m - \eta_w {\nabla L(\bw^{(t)}_m )}\big/{\big\Vert{\nabla L(\bw^{(t)}_m )}\big\Vert_2}$;\Comment{$m$-th column of $\bW^\top$}
\EndFor
\EndWhile
\State Output $\bW^\top=[\bw_1, \bw_2, \ldots, \bw_M],\bZ=[\bz_1, \bz_2, \ldots, \bz_N]$.
\end{algorithmic} 
\end{algorithm}

\index{Level curves}
\index{Level surfaces}
\subsection*{Geometric Interpretation of Gradient Descent}
\begin{lemma}[Direction of gradients]\label{lemm:direction-gradients}
The gradient of a function at a given point is perpendicular to the level curve (or level surface in higher dimensions) passing through that point.
\end{lemma}
\begin{proof}[of Lemma~\ref{lemm:direction-gradients}, the informal proof]
This proof involves showing that the gradient is orthogonal to the tangent vector of the level curve.  For simplicity, let's start with the two-dimensional case.
Suppose the level curve takes the form $f(x,y)=c$. 
This implicitly establishes a relationship between $x$ and $y$ such that $y=y(x)$, where $y$ can be regarded as a function of $x$~\footnote{This is known as the implicit function  theorem, provided that the partial derivative is nonzero and the function is smooth.}. Therefore, the level curve can be expressed as 
$
f(x, y(x)) = c.
$
Applying the chain rule, we get:
$$
\frac{\partial f}{\partial x} \underbrace{\frac{dx}{dx}}_{=1} + \frac{\partial f}{\partial y} \frac{dy}{dx}=0
\gap \implies \gap 
\left\langle \frac{\partial f}{\partial x}, \frac{\partial f}{\partial y}\right\rangle
\cdot 
\left\langle \frac{dx}{dx}, \frac{dy}{dx}\right\rangle=0.
$$
That is, the gradient is perpendicular to the tangent.

In full generality, consider the level curve of a vector $\bx\in \real^n$: $f(\bx) = f(x_1, x_2, \ldots, x_n)=c$. Each variable $x_i$ can be regarded as a function of a parameter $t$ on the level curve $f(\bx)=c$: $f(x_1(t), x_2(t), \ldots, x_n(t))=c$. Differentiating the equation with respect to $t$ using the chain rule:
$$
\frac{\partial f}{\partial x_1} \frac{dx_1}{dt} + \frac{\partial f}{\partial x_2} \frac{dx_2}{dt}
+\ldots + \frac{\partial f}{\partial x_n} \frac{dx_n}{dt}
=0.
$$
Thus, the gradient is perpendicular to the tangent in the $n$-dimensional case:
$$
\left\langle \frac{\partial f}{\partial x_1}, \frac{\partial f}{\partial x_2}, \ldots, \frac{\partial f}{\partial x_n}\right\rangle
\cdot 
\left\langle \frac{dx_1}{dt}, \frac{dx_2}{dt}, \ldots \frac{dx_n}{dt}\right\rangle=0.
$$
This completes the proof.
\end{proof}

This lemma provides a key geometric insight into gradient descent. When minimizing a convex function $L(\bz)$, gradient descent moves in the direction opposite to the gradient, which corresponds to the steepest descent direction. This direction ensures a decrease in the value of the loss function.
Figure~\ref{fig:alsgd-geometrical} illustrates this concept in two dimensions, where the vector $-\nabla L(\bz)$ points in the direction of maximum decrease of the convex function $L(\bz)$.

\begin{figure}[h]
\centering  
\vspace{-0.15cm}    
\subfigtopskip=2pt  
\subfigbottomskip=2pt 
\subfigcapskip=-5pt  
\subfigure[A two-dimensional convex function $L(\bz)$.]{\label{fig:alsgd1}
\includegraphics[width=0.47\linewidth]{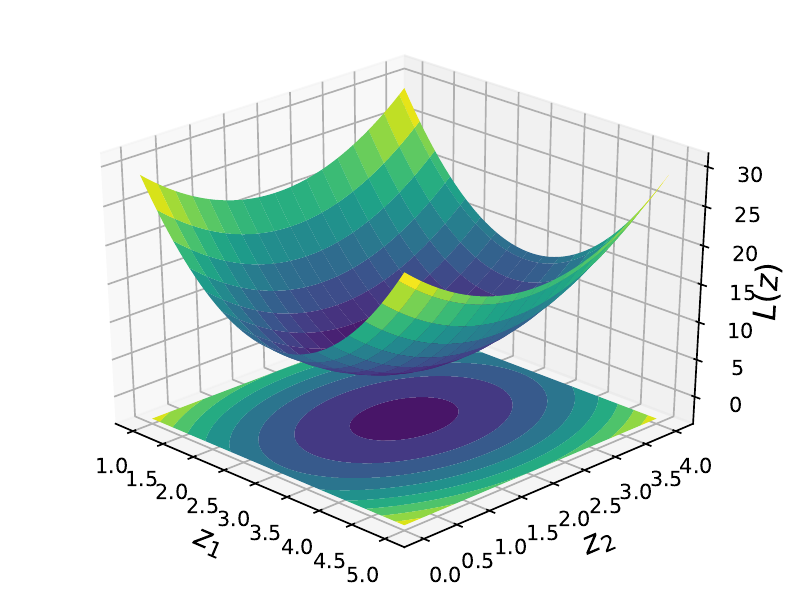}}
\subfigure[$L(\bz)=c$ is a constant.]{\label{fig:alsgd2}
\includegraphics[width=0.44\linewidth]{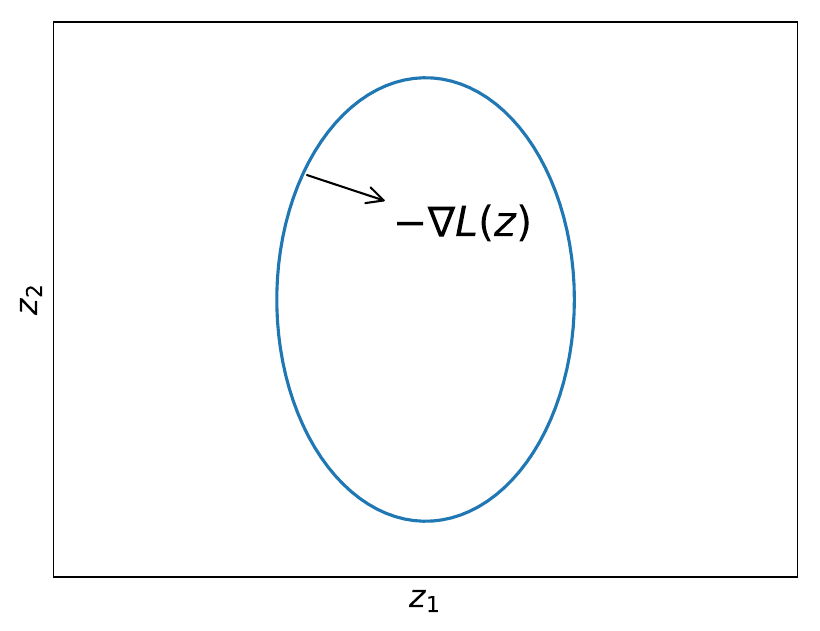}}
\caption{Figure~\ref{fig:alsgd1} shows  surface and  contour plots for a specific function (\textcolor{mydarkblue}{blue}=low, \textcolor{mydarkyellow}{yellow}=high), where the upper graph is the surface plot, and the lower one is its projection  (i.e., contour). Figure~\ref{fig:alsgd2}: $-\nabla L(\bz)$ pushes the loss to decrease for the convex function $L(\bz)$.}
\label{fig:alsgd-geometrical}
\end{figure}
\index{Convex function}
\index{Contour plot}
\index{Contour plot}
\index{Regularization}

\index{Geometric interpretation}
\index{Projection gradient descent}
\index{Overfitting}
\section{Regularization: A Geometric Interpretation}\label{section:reg_geom_inter}
\begin{figure}[h]
\centering
\includegraphics[width=0.95\textwidth]{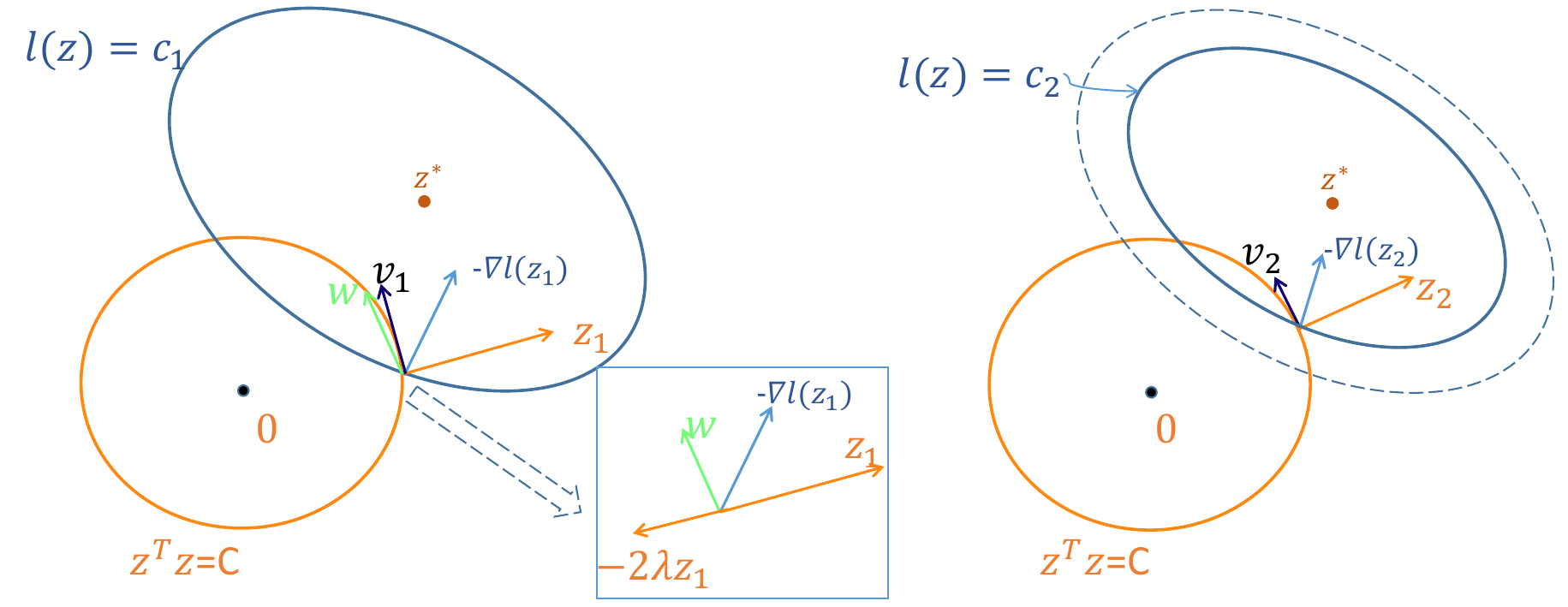}
\caption{Constrained gradient descent with $\bz^\top\bz\leq C$. The \textcolor{mydarkgreen}{green} vector $\bw$ represents the projection of $\bv_1$ onto the set $\bz^\top\bz\leq C$, where $\bv_1$ is the component of $-\nabla l(\bz)$ that is perpendicular to $\bz_1$. 
The image on the right illustrates the next step after the update in the left image. $\bz^\star$ denotes the optimal solution of \{$\min l(\bz)$\}.}
\label{fig:alsgd3}
\end{figure}
In Section~\ref{section:regularization-extention-general}, we discussed how  regularization can extend the ALS algorithm to general matrices.
Gradient descent offers a geometric interpretation of regularization.
To avoid confusion, we denote the loss function without regularization as $l(\bz)$ and the loss function with regularization as $L(\bz) = \l(\bz)+\lambda_z \normtwo{\bz}^2$, where $l(\bz): \real^n \rightarrow \real$. When minimizing $l(\bz)$, a descent method typically searches for a solution in $\real^n$. However, in machine learning, searching across the entire space $\real^n$ can lead to overfitting. 
One way to mitigate this is to restrict the search to a subset of the vector space, such as searching in $\bz^\top\bz < C$ for some constant $C$. 
This can be formulated as the constrained optimization problem:
$$
\mathop{\arg\min}_{\bz} \,\, l(\bz), \gap \text{s.t.,} \gap \bz^\top\bz\leq C.
$$
As demonstrated above, a standard gradient descent method updates $\bz$ by moving in the direction of steepest descent, i.e., update $\bz$ as $\bz\leftarrow \bz-\eta \nabla l(\bz)$ for a small step size $\eta$. When the level curve is $l(\bz)=c_1$ and the current position of parameter $\bz$ is $\bz=\bz_1$, where $\bz_1$ lies at the intersection of $\bz^\top\bz=C$ and $l(\bz)=c_1$, the descent direction $-\nabla l(\bz_1)$ will be perpendicular to the level curve of $l(\bz_1)=c_1$, as shown in the left image of Figure~\ref{fig:alsgd3} (by Lemma~\ref{lemm:direction-gradients}). However, if we further restrict that the optimal value must lie  within $\bz^\top\bz\leq C$, the standard descent direction $-\nabla l(\bz_1)$ will lead the update $\bz_2=\bz_1-\eta \nabla l(\bz_1)$ beyond the boundary of $\bz^\top\bz\leq C$. 
One solution is to decompose the step $-\nabla l(\bz_1)$ into 
$$
-\nabla l(\bz_1) = a\bz_1 + \bv_1,
$$ 
where $a\bz_1$ represents the component perpendicular to the curve of $\bz^\top\bz=C$, and $\bv_1$ is the component parallel to the curve of $\bz^\top\bz=C$. By keeping only the step $\bv_1$, the update becomes
$$
\bz_2 = \text{project}(\bz_1+\eta \bv_1) = \text{project}\bigg(\bz_1 + \eta
\underbrace{(-\nabla l(\bz_1) -a\bz_1)}_{\bv_1}\bigg),~\footnote{where the operation project($\bx$) 
will project the vector $\bx$ to the closest point inside $\bz^\top\bz\leq C$. Notice here the unprojected update $\bz_2 = \bz_1+\eta \bv_1$ can still make $\bz_2$ fall outside the curve of $\bz^\top\bz\leq C$.}
$$ 
which will lead to a smaller loss from $l(\bz_1)$ to $l(\bz_2)$ while maintaining the constraint $\bz^\top\bz\leq C$. This approach is known as  \textit{projection gradient descent (PGD)}. 
It is not hard to see that the update $\bz_2 = \text{project}(\bz_1+\eta \bv_1)$ can be understood as finding a vector $\bw$ (represented by the \textcolor{mydarkgreen}{green} vector in the left image of Figure~\ref{fig:alsgd3}) such that $\bz_2=\bz_1+\bw$ lies within the constraint set $\bz^\top\bz\leq C$. Mathematically, the vector $\bw$ can be determined as $-\nabla l(\bz_1) -2\lambda \bz_1$ for some $\lambda$, as illustrated in the middle image of Figure~\ref{fig:alsgd3}. This corresponds precisely to the negative gradient of the regularized loss function $L(\bz)=l(\bz)+\lambda\normtwo{\bz}^2$, so that 
$$
\begin{aligned}
\bw=-\nabla L(\bz) &= -\nabla l(\bz) - 2\lambda \bz 
\gap \implies \gap
\bz_2 = \bz_1+ \bw =\bz_1 -  \nabla L(\bz).
\end{aligned}
$$
And in practice, using a small step size $\eta$ prevents the trajectory from moving outside the constraint set $\bz^\top\bz\leq C$:
$$
\bz_2  =\bz_1 -  \eta\nabla L(\bz),
$$
which aligns with the regularization term discussed in Section~\ref{section:regularization-extention-general}.

\index{Sparsity}
\index{$\ell_1$ regularization}
\begin{figure}[h]
\centering
\includegraphics[width=0.95\textwidth]{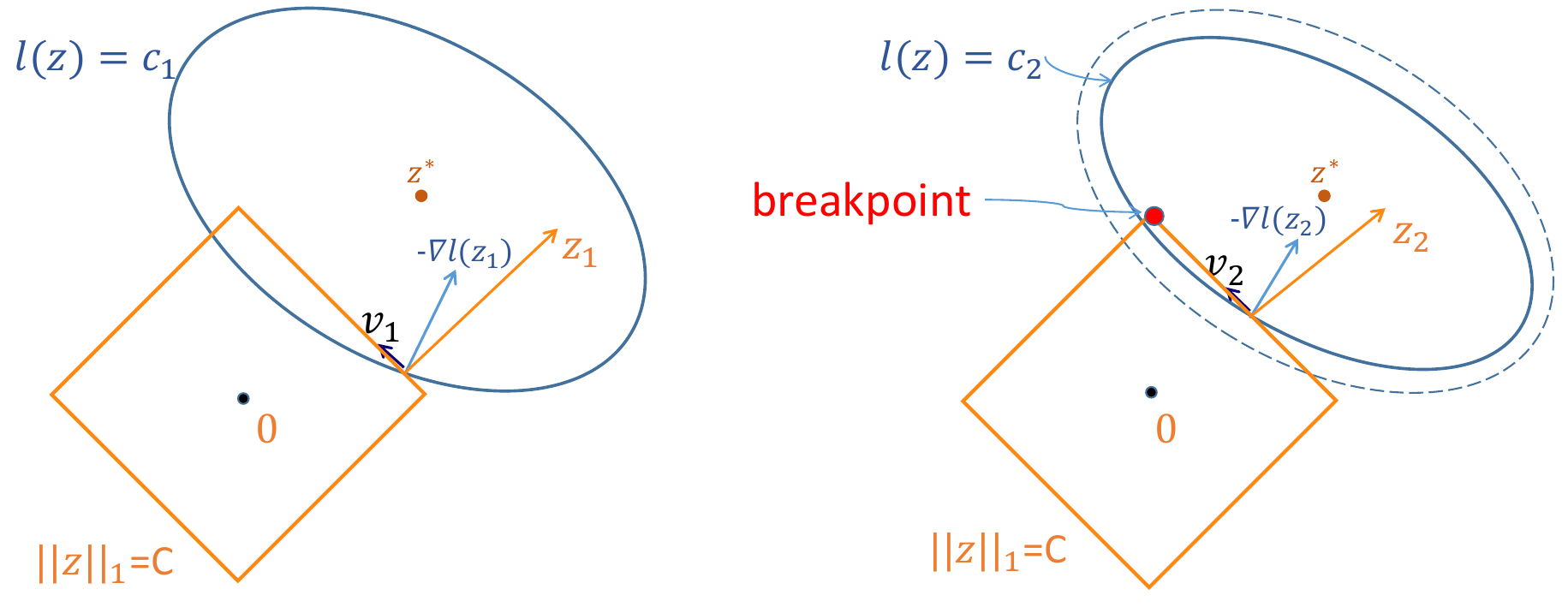}
\caption{Constrained gradient descent with $\norm{\bz}_1\leq C$, where the \textcolor{winestain}{red} dot denotes the breakpoint in the $\ell_1$ norm. The right image illustrates the next step after the update in the left image. $\bz^\star$ denotes the optimal solution of \{$\min l(\bz)$\}.}
\label{fig:alsgd4}
\end{figure}
\paragraph{Sparsity.}
In certain scenarios, we seek to identify a sparse solution $\bz$ such that $l(\bz)$ is minimized. 
For example, in facial feature extraction, sparsity leads to more localized features, meaning that fewer features are used to reconstruct each input image.
Regularization to be constrained in $\norm{\bz}_1 \leq C$ exists to this purpose, where $\norm{\cdot}_1$ denotes the $\ell_1$ norm of a vector or a matrix. 
Similar to the previous case, the $\ell_1$ constrained optimization pushes the gradient descent towards the border of the level set $\norm{\bz}_1=C$. The situation in the two-dimensional case is shown in Figure~\ref{fig:alsgd4}. 
In  high-dimensional cases, many elements in $\bz$ will be driven towards the breakpoint of $\norm{\bz}_1=C$, promoting sparsity in the solution, as shown in the right image of Figure~\ref{fig:alsgd4}.

\index{Stochastic gradient descent}
\index{Stochastic coordinate descent}
\section{Stochastic Gradient Descent}
The gradient descent method is a valuable optimization algorithm; however, it exhibits certain limitations in practical applications. 
To understand these limitations, consider the mean squared error (MSE) derived from  Equation~\eqref{equation:als-per-example-loss2}:
\begin{equation}\label{equation:als-per-example-loss_mse}
\frac{1}{MN}\mathop{\min}_{\bW,\bZ}  \sum_{n=1}^N \sum_{m=1}^{M} \left(a_{mn} - \bw_m^\top\bz_n\right)^2.
\end{equation}
The MSE requires calculating the residual $e_{mn} = (a_{mn} - \bw_m^\top\bz_n)^2$ for each observed entry $a_{mn}$, representing the squared difference between the predicted and actual values.  The total  sum of squared residuals is denoted by $e = \sum_{m,n=1}^{MN}e_{mn}$.
When the number of training entries is large (i.e., $MN$ is large), computing the full gradient over all entries becomes computationally expensive and slow. Moreover, gradients from different samples may cancel each other out, leading to small net updates and slow convergence.
To address these issues, researchers have enhanced the gradient descent method with the \textit{stochastic gradient descent (SGD)} method (see, for example, \citet{lu2022gradient}). 
In the SGD algorithm, instead of calculating the full gradient of the objective function with respect to the parameters  across all samples in the data set, which can be computationally expensive, the algorithm takes a more efficient approach. It randomly chooses one sample and calculates the gradient of the objective function with respect to the parameters using only  this single sample. 
This gradient estimate is then used to update the parameters in the direction that minimizes the objective function. 
By using a single sample at each iteration, the SGD algorithm provides a fast and often sufficient approximation of the full gradient, making it  particularly well-suited for large-scale data sets.

In particular, we consider again the per-example loss:
$$
L(\bW,\bZ)= \frac{1}{2} \sum_{n=1}^N \sum_{m=1}^{M} \left(a_{mn} - \bw_m^\top\bz_n\right)^2 +\frac{1}{2} \lambda_w\sum_{m=1}^{M}\normtwo{\bw_m}^2 +\frac{1}{2}\lambda_z\sum_{n=1}^{N}\normtwo{\bz_n}^2.
$$
As we iteratively minimize the  loss term $l(\bw_m, \bz_n)=\frac{1}{2}\left(a_{mn} - \bw_m^\top\bz_n\right)^2+\frac{1}{2} \lambda_w\normtwo{\bw_m}^2 +\frac{1}{2}\lambda_z\normtwo{\bz_n}^2$ for all $m\in \{1,2,\ldots, M\}, n\in\{1,2,\ldots,N\}$ (referred to as the per-example loss term), the overall loss $L(\bW,\bZ)$ decreases accordingly.
This approach is also known as  \textit{stochastic coordinate descent}. The gradients  with respect to $\bw_m$ and $\bz_n$, and their roots are given, respectively, by 
$$
\left\{
\begin{aligned}
\nabla_{\bz_n} l(\bz_n) &= \bw_m\bw_m^\top \bz_n +\lambda_z\bz_n  -a_{mn} \bw_m 
 &\implies&\,\, \bz_n= a_{mn}(\bw_m\bw_m^\top+\lambda_z\bI)^{-1}\bw_m;\\
\nabla_{\bw_m} l(\bw_m) &= \bz_n\bz_n^\top\bw_m +\lambda_w\bw_m - a_{mn}\bz_n &\implies&\,\, \bw_m= a_{mn}(\bz_n\bz_n^\top+\lambda_w\bI)^{-1}\bz_n.
\end{aligned}
\right.
$$
Alternatively, the update can be performed using gradient descent for the per-example loss. Since we update based on the per-example loss, this approach is thus known as the \textit{stochastic gradient descent (SGD)}:
$$
\left.
\begin{aligned}
\bz_n&\leftarrow \bz_n - \eta_z \frac{\nabla_{\bz_n} l(\bz_n)}{\normtwo{\nabla_{\bz_n} l(\bz_n)}}
\gap \text{and}\gap 
\bw_m\leftarrow \bw_m - \eta_w \frac{\nabla_{\bw_m} l(\bw_m)}{\normtwo{\nabla_{\bw_m} l(\bw_m)}}.
\end{aligned}
\right.
$$
The stochastic gradient descent update for ALS is formulated in Algorithm~\ref{alg:als-regularizer-missing-stochas-gradient-realstoch}. 
It is possible that the gradient descent or stochastic gradient descent algorithm may fail to converge. In such cases, it is advisable  to re-run the algorithm using a smaller step size.
And in practice, the indices  $m$ and $n$ in the algorithm can be randomly generated, which is why the method is termed ``stochastic." ~\footnote{When we iteratively choose the values of $m$ and $n$ from $\{1,2,\ldots, M\}$ and $\{1,2,\ldots, N\}$ in a deterministic cyclic order, respectively, the stochastic method can be referred to as  ``\textit{incremental gradient descent}."}

\begin{algorithm}[h] 
\caption{Alternating Least Squares with Full Entries and SGD}
\label{alg:als-regularizer-missing-stochas-gradient-realstoch}
\begin{algorithmic}[1] 
\Require  Matrix $\bA\in \real^{M\times N}$;
\State Initialize $\bW\in \real^{M\times K}$, $\bZ\in \real^{K\times N}$ \textcolor{mylightbluetext}{randomly without condition on the rank and the relationship between $M, N, K$}; 
\State Choose a stop criterion on the approximation error $\delta$;
\State Choose regularization parameters $\lambda_w, \lambda_z$, and step sizes $\eta_w, \eta_z$;
\State Choose the maximal number of iterations $C$;
\State $iter=0$; \Comment{Count for the number of iterations}
\While{$\normf{  \bA- (\bW\bZ)}^2>\delta $ and $iter<C$}
\State $iter=iter+1$; 
\For{$n=1,2,\ldots, N$}
\For{$m=1,2,\ldots, M$} \Comment{in practice, $m,n$ can be randomly produced}
\State $\bz_n\leftarrow \bz_n - \eta_z {\nabla l(\bz_n)}/{\normtwo{\nabla l(\bz_n)}}$;\Comment{$n$-th column of $\bZ$}
\State $\bw_m\leftarrow \bw_m - \eta_w {\nabla l(\bw_m)}/{\normtwo{\nabla l(\bw_m)}}$;\Comment{$m$-th column of $\bW^\top$}
\EndFor
\EndFor

\EndWhile
\State Output $\bW^\top=[\bw_1, \bw_2, \ldots, \bw_M],\bZ=[\bz_1, \bz_2, \ldots, \bz_N]$.
\end{algorithmic} 
\end{algorithm}

\section{Bias Term}

\begin{figure}[htp]
\centering
\includegraphics[width=0.95\textwidth]{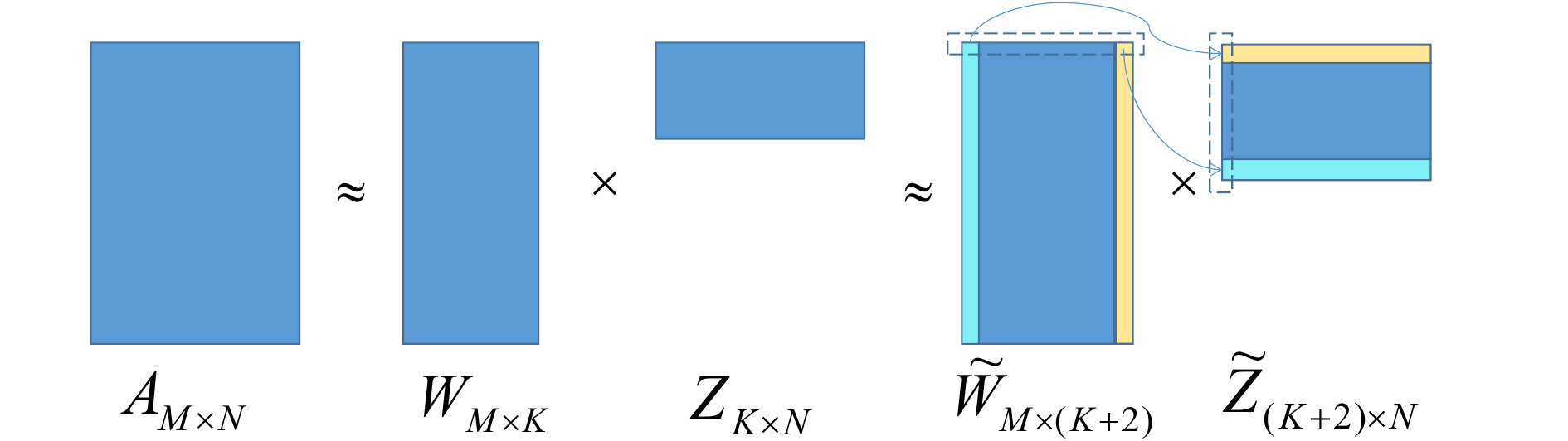}
\caption{Bias terms in alternating least squares, where the \textcolor{mydarkyellow}{yellow} entries denote ones (which are fixed), and the \textcolor{cyan}{cyan} entries denote the added features to fit the bias terms. The dotted boxes provide an example of how the bias terms work.}
\label{fig:als-bias}
\end{figure}
In ordinary least squares models, a bias term is usually incorporated into the raw matrix to improve model performance, as illustrated in Equation~\eqref{equation:ls-bias}. 
A similar approach can be applied to the ALS algorithm. 
Specifically, a fixed column filled with all ones can be appended to the \textbf{last column} of matrix $\bW$. 
To account for this, an extra row should be added to the last row of matrix $\bZ$ to fit the features introduced by the bias term in $\bW$. 
Analogously, a fixed row with all ones can be added to the \textbf{first row} of $\bZ$, and an extra column in the first column of $\bW$ can be added to fit the features. This configuration is illustrated in Figure~\ref{fig:als-bias}.

Given the loss function with respect to the columns of $\bZ$ in Equation~\eqref{als:gradient-regularization-zn}, let 
$
\widetildebz_n
=
\scriptsize
\begin{bmatrix}
	1\\
	\bz_n
\end{bmatrix}
\in\real^{K+2}
$ be the $n$-th column of $\widetildebZ$. Then we have:
\begin{equation}
\footnotesize
\begin{aligned}
2L(\bz_n) 
&=\normf{\widetildebW\widetildebZ-\bA}^2 +\lambda_w \normf{\widetildebW}^2 + \lambda_z \normf{\widetildebZ}^2
= 
\left\Vert
\widetildebW
\begin{bmatrix}
	1 \\
	\bz_n
\end{bmatrix}-\ba_n
\right\Vert_2^2 
+ 
\underbrace{\lambda_z \normtwo{\widetildebz_n}^2}_{=\lambda_z \normtwo{\bz_n}^2+\lambda_z}
+ 
C_{z_n}\\
&= 
\left\Vert
\begin{bmatrix}
	\widebarbw_0 & \widebarbW
\end{bmatrix}
\begin{bmatrix}
	1 \\
	\bz_n
\end{bmatrix}-\ba_n
\right\Vert_2^2 
+ \lambda_z \normtwo{\bz_n}^2 + C_{z_n}
= 
\bigg\Vert
\widebarbW \bz_n - 
\underbrace{(\ba_n-\widebarbw_0)}_{\widebarba_n}
\bigg\Vert_2^2 
+ \lambda_z \normtwo{\bz_n}^2 + C_{z_n},
\end{aligned}
\end{equation}
where $\widebarbw_0$ represents the first column of $\widetildebW$, $\widebarbW$ denotes the remaining $K+1$ columns of $\widetildebW$ (i.e., $\widetildebW=[\widebarbw_0, \widebarbW]$), and $C_{z_n}$ is a constant with respect to $\bz_n$. Let $\widebarba_n = \ba_n-\widebarbw_0$, the update for $\bz_n$ is just similar to the one in Equation~\eqref{als:gradient-regularization-zn}, with the gradient given by
$$
\nabla_{\bz_n} L(\bz_n) = \widebarbW^\top\widebarbW\bz_n - \widebarbW^\top\widebarba_n + \lambda_z\bz_n.
$$
Therefore, the update for $\bz_n$ is given by determining the root of the  gradient above:
$$
\textbf{(update for $\widetildebz_n$)}: \quad \bz_n = (\widebarbW^\top\widebarbW+ \lambda_z\bI)^{-1} \widebarbW^\top \widebarba_n
\gap 
\implies 
\gap 
\widetildebz_n = \begin{bmatrix}
	1\\\bz_n 
\end{bmatrix},
\,\forall n.
$$
Similarly, following the loss with respect to each row of $\bW$ in Equation~\eqref{als:gradient-regularization-wd}, let $\widetildebw_m =
\scriptsize
\begin{bmatrix}
	\bw_m \\
	1
\end{bmatrix} \in\real^{K+2}$ be the $m$-th row of $\widetildebW$ (or $m$-th column of $\widetildebW^\top$). Then we have:
\begin{equation}
\footnotesize
\begin{aligned}
&2L(\bw_m ) 
=\normf{\widetildebZ^\top\widetildebW^\top-\bA^\top}^2 +\lambda_w \normf{\widetildebW^\top}^2 + \lambda_z \normf{\widetildebZ}^2
= 
\normtwo{\widetildebZ^\top\widetildebw_m-\bb_m}^2 + 
\underbrace{\lambda_w \normtwo{\widetildebw_m}^2}_{=\lambda_w \normtwo{\bw_m}^2+\lambda_w}
+ 
C_{w_m} \\
&= 
\bigg\Vert
\begin{bmatrix}
\widebarbZ^\top&
\widebarbz_0
\end{bmatrix}
\scriptsize
\begin{bmatrix}
\bw_m \\
1
\end{bmatrix}
\footnotesize
-\bb_m
\bigg\Vert_2^2 
+ 
\lambda_w \normtwo{\bw_m}^2
+ 
C_{w_m}
= 
\left\Vert
\widebarbZ^\top\bw_m 
-(\bb_m-\widebarbz_0) 
\right\Vert_2^2+ 
\lambda_w \normtwo{\bw_m}^2
+ 
C_{w_m},
\end{aligned}
\end{equation}
where $\widebarbz_0$ represents the last column of $\widetildebZ^\top$, $\widebarbZ^\top$ contains the remaining $K+1$ columns of $\widetildebZ^\top$ (i.e., $\widetildebZ^\top=[\widebarbZ^\top, \widebarbz_0]$),
and $C_{w_m}$ is a constant with respect to $\bw_m$.  
$\bW^\top=[\bw_1, \bw_2, \ldots, \bw_M]$ and $\bA^\top=[\bb_1,\bb_2, \ldots, \bb_M]$ are the column partitions of $\bW^\top$ and $\bA^\top$, respectively. Let $\widebarbb_m = \bb_m-\widebarbz_0$. The update for $\bw_m$ is again just similar to  the one in Equation~\eqref{als:gradient-regularization-wd}, with the gradient given by
$$
\nabla_{\bw_m} L(\bw_m) = \widebarbZ\cdot \widebarbZ^\top\bw_m - \widebarbZ\cdot \widebarbb_m + \lambda_w\bw_m.
$$
Therefore, the update for $\bw_m$ is given by the root of the  gradient above:
$$
\textbf{(update for $\widetildebw_m$)}:\quad 
\bw_m=(\widebarbZ\cdot \widebarbZ^\top+\lambda_w\bI)^{-1}\widebarbZ\cdot \widebarbb_m
\gap \implies \gap 
\widetildebw_m = \begin{bmatrix}
	\bw_m \\ 1
\end{bmatrix}, 
\forall m.
$$
Similar updates can be derived using gradient descent, taking into account the bias terms and handling missing entries (see Section~\ref{section:als-gradie-descent} and \ref{section:alt-columb-by-column} for a reference).

\section{Low-Rank Hadamard Decomposition}\label{section:low_rank_hadamard}
In the fields of linear algebra and data analysis, matrix decomposition techniques are essential for extracting meaningful information from complex datasets. 
As discussed above, one common objective is to approximate a given matrix using a lower-rank representation, which simplifies the data while preserving its key characteristics. The Hadamard product, also known as the element-wise product, provides an alternative to traditional matrix multiplication in matrix decomposition.

As discussed  previously, the alternating least squares (ALS) algorithm is an iterative method used to find a suboptimal low-rank approximation of a matrix by decomposing it into two or more matrices. ALS is particularly advantageous for large-scale problems, such as those found in recommender systems, where the goal is to predict missing entries in a user-item interaction matrix. During each iteration, the ALS algorithm alternates between updating one matrix while keeping the other fixed, thereby minimizing the reconstruction error at every step.
Nonnegative matrix factorization (NMF), introduced in Chapter~\ref{chapter:nmf}, is a variant of matrix factorization where both the original matrix and the resulting factorized matrices have nonnegative entries. This constraint makes NMF especially suitable for applications where the data represents quantities that cannot be negative, such as images, audio signals, or document-term matrices in text mining.

Ws further explore  the  Hadamard decomposition of a matrix $\bA$, where $\bA$ can be expressed as the Hadamard product of two low-rank matrices: $\bA=\bA_1\hadaprod \bA_2$. This type of decomposition is advantageous when the data exhibits multiplicative relationships, and a low-rank approximation is desired to reduce complexity or enhance interpretability.
\paragraph{Non-Factorizability  Issue.}
When $\bA_1\in\real^{n^2\times n^2}$ and $\bA_2\in\real^{n^2\times n^2}$ share the same rank $n$, the Hadamard product $\bA_1\hadaprod \bA_2$ can achieve a maximum rank of $n^2$ (Problem~\ref{problem:rank_hada_prod}).
However, not all matrices $\bA\in\real^{n^2\times n^2}$ of rank $n^2$  can be  represented as the Hadamard product of two lower-rank matrices:
\begin{itemize}
	\item The Hadamard decomposition $\bA = \bA_1 \hadaprod  \bA_2$, where $\bA_1$ and $\bA_2$ are rank-$n$ factors, encodes a system of nonlinear equations.
	\item This system comprises $n^2 \times n^2 = n^4$ equations (one per entry of $ \bA$) and, due to the low-rank constraint on the two Hadamard factors $\bA_1$ and $ \bA_2$, only $(n^2 n + n n^2) = 2n^3$ variables exist.
	\item For $n > 2$, there are more equations than variables, suggesting that all the equations will be simultaneously satisfied only in special cases.
	For example, if the matrix $\bA$ includes a row or a column with all but a single entry being zero, then not all the equations in the system can be satisfied \citep{ciaperoni2024hadamard}.
	
\end{itemize}

Therefore, we focus on solving the low-rank reconstruction problem for the Hadamard decomposition.
Assuming that  $\bA_1$ and $\bA_2$ share the same rank $K$, our aim is to reconstruct the design matrix $\bA$ through the Hadamard product $\bA_1\hadaprod\bA_2$. 
Building upon the matrix factorization method used in   alternating least squares (Section~\ref{section:als-netflix}),
we now concentrate on algorithms for solving the \textit{low-rank Hadamard decomposition} problem:
\begin{itemize}
	\item Given a real matrix $\bA\in \real^{M\times N}$, find  matrix factors $\bA_1\in \real^{M\times N}$ and $\bA_2\in \real^{M\times N}$ such that: 
	\begin{equation}
		\min\,\,L(\bC_1, \bD_1, \bC_2,\bD_2) = \normf{\bA_1\hadaprod \bA_2-\bA}^2= \normf{(\bC_1\bD_1)\hadaprod (\bC_2\bD_2) -\bA}^2,
	\end{equation}
	where $\bC_1, \bC_2\in\real^{M\times K}$, and $\bD_1, \bD_2\in\real^{K\times N}$: $\bA_1=\bC_1\bD_1$ and $\bA_2=\bC_2\bD_2$ such that $\bA_1$ and $\bA_2$ are rank-$K$ matrices.
\end{itemize}
Low-rank (Hadamard) decomposition is often  necessary because many natural phenomena exhibit multiplicative or conjunctive relationships \citep{ciaperoni2024hadamard}.
For instance, consider a study on risk factors for a disease with two predictors: smoking status (yes/no) and alcohol consumption (yes/no). The multiplicative model would account  not only for the individual effects of smoking and alcohol consumption but also for their interaction.
The (low-rank) Hadamard decomposition offers an alternative approach to modeling such relationships.

Following the alternating descent framework using gradient descent, at each iteration, the matrices $\bC_1, \bD_1, \bC_2$, and $\bD_2$ are updated sequentially by taking a step in the direction opposite to the gradient of the objective function.
It then can be shown that 
$$
\nabla L(\bC_1)=\nabla L(\bC_1|\bD_1, \bC_2,\bD_2)=2\big(\left((\bC_1\bD_1)\hadaprod (\bC_2\bD_2) -\bA\right)\hadaprod (\bC_2\bD_2)\big)\bD_1^\top.
$$
\begin{proof}
	For simplicity, we derive the gradient of $\bE$ for $f(\bE) = \normf{\bE\bF\hadaprod \bC - \bD}^2$.
	We have 
	$$
	\begin{aligned}
		f(\bE)&=\normf{\bE\bF\hadaprod \bC - \bD}^2 
		= \trace\left((\bE\bF\hadaprod \bC - \bD)^\top(\bE\bF\hadaprod \bC - \bD) \right)\\
		&= \trace\left((\bE\bF\hadaprod \bC)^\top (\bE\bF\hadaprod \bC)\right)-2\trace\left((\bE\bF\hadaprod \bC)^\top\bD\right) + \trace(\bD^\top\bD).
	\end{aligned}
	$$
	Considering the first term, we get 
	$$
	\frac{\partial \trace\left((\bE\bF\hadaprod \bC)^\top (\bE\bF\hadaprod \bC)\right)}{\partial \bE}
	=2(\bE\bF)\hadaprod \bC\hadaprod \bC \cdot \bF^\top.
	~\footnote{Use the fact that $\frac{\partial \trace\left((\bE\hadaprod \bC)^\top(\bE\hadaprod \bC) \right)}{\partial \bE}=2\bE\hadaprod \bC\hadaprod \bC$, which can be derived element-wise.}
	$$
	For the second term, it follows that 
	$$
	-2\frac{\partial \trace\left((\bE\bF\hadaprod \bC)^\top\bD\right)}{\partial \bE}
	=-2\bD\hadaprod \bC \cdot \frac{\partial \bE\bF}{\partial \bE}
	=-2\bD\hadaprod \bC \cdot\bF^\top.
	~\footnote{Use the fact that $\frac{\partial\trace( (\bE\hadaprod \bC)^\top\bD )}{\partial \bE} = \bD\hadaprod \bC$, which can be derived element-wise. Since $\trace( (\bE\hadaprod \bC)^\top\bD)=\sum_{i,j} d_{ij}a_{ij}c_{ij}$ and thus $\frac{\partial \trace( (\bE\hadaprod \bC)^\top\bD)}{\partial a_{ij}}=d_{ij}c_{ij}$.}
	$$
	The third term is a constant w.r.t. to $\bE$. 
	Therefore, $\frac{\partial f(\bE)}{\partial \bE} = 2(\bE\bF)\hadaprod \bC\hadaprod \bC \cdot \bF^\top-2\bD\hadaprod \bC \cdot\bF^\top
	=2\big((\bE\bF)\hadaprod \bC -\bD \big)\hadaprod \bC\cdot\bF^\top
	.$
	Substituting  $\bE=\bC_1$, $\bF=\bD_1$, $\bC=\bC_2\bD_2$, and $\bD=\bA$ completes the proof.
\end{proof}
The gradients with respect to $\bD_1, \bC_2$, and $\bD_2$ can be derived analogously.
Thus, the alternating descent method for obtaining the low-rank approximation of Hadamard decomposition can be described by  Algorithm~\ref{alg:ad_hadamad_svd}.

\begin{algorithm}[h] 
	\caption{Alternating Descent with Gradient Descent for Low-Rank Hadamard Decomposition: A regularization can also be added into the gradient descent update (see Section~\ref{section:regularization-extention-general}).}
	\label{alg:ad_hadamad_svd}
	\begin{algorithmic}[1] 
		\Require Matrix $\bA\in \real^{M\times N}$;
		\State Initialize $\bC_1,\bC_2\in \real^{M\times K}$, and $\bD_1,\bD_2\in \real^{K\times N}$; 
		\State Choose a stoping criterion on the approximation error $\delta$;
		\State Choose  step size $\eta$;
		\State Choose the maximum number of iterations $C$;
		\State $iter=0$; \Comment{Count for the number of iterations}
		\While{$\normf{(\bC_1\bD_1)\hadaprod (\bC_2\bD_2) -\bA}^2>\delta $ and $iter<C$}
		\State $iter=iter+1$; 
		\State $\Delta \leftarrow \left((\bC_1\bD_1)\hadaprod (\bC_2\bD_2) -\bA\right)$;
		\State $\bC_1 \leftarrow \bC_1-\eta \nabla L(\bC_1)=\bC_1-\eta\cdot 2\left(\Delta \hadaprod (\bC_2\bD_2)\right)\bD_1^\top$;
		\State $\bD_1 \leftarrow \bD_1-\eta \nabla L(\bD_1)=\bD_1-\eta\cdot 2 \left\{\left(\Delta^\top \hadaprod (\bC_2\bD_2)^\top\right)\bC_1\right\}^\top$;
		\State $\bC_2 \leftarrow \bC_2-\eta \nabla L(\bC_2)=\bC_2 - \eta \cdot 2 \left(\Delta \hadaprod (\bC_1\bD_1)\right)\bD_2^\top$;
		\State $\bD_2 \leftarrow \bD_2-\eta \nabla L(\bD_2)=\bD_2-\eta \cdot 2 \left\{\left( \Delta^\top \hadaprod (\bC_1\bD_1)^\top \right)\bC_2\right\}^\top$;

		\EndWhile
		\State Output $\bC_1,\bD_1, \bC_2,\bD_2$.
	\end{algorithmic} 
\end{algorithm}

\subsection{Rank-One Update }
Following the rank-one update approach used in  ALS (Section~\ref{section:alt-columb-by-column}), we consider updating the $n$-th column $\bd_{1,n}$ of $\bD_1$, $n\in\{1,2,\ldots, N\}$.
Analogously, the gradient with respect to $\bd_{1,n}$ can be derived as:
\begin{equation}\label{equation:hada_rkone1}
	\begin{aligned}
		\nabla L(\bd_{1,n}) 
		=\frac{\partial L(\bd_{1,n})}{\partial \bd_{1,n}}
		&= 2\bC_1^\top \left((\bC_1\bd_{1,n}) \hadaprod \ba_{2,n} \hadaprod \ba_{2,n}\right) - 2\bC_1^\top (\ba_n \hadaprod \ba_{2,n})\\
		&=2\bC_1^\top \left(\left[(\bC_1\bd_{1,n}) \hadaprod \ba_{2,n}-\ba_n \right] \hadaprod\ba_{2,n}\right), \gap n\in\{1,2,\ldots, N\}, 
	\end{aligned}
\end{equation}
where $\ba_{2,n}$ denotes the $n$-th column of $\bA_2=\bC_2\bD_2$. The gradients for the columns of $\bD_2$ can be computed in a similar manner.

Suppose further that $\bC_1^\top=[\bc_{1,1}, \bc_{1,2}, \ldots, \bc_{1,M}]\in\real^{K\times M}$,  $\bB=\bA^\top=[\bb_1, \bb_2, \ldots, \bb_M]\in\real^{N\times M}$, and $\bB_2=\bA_2^\top=(\bC_2\bD_2)^\top=[\bb_{2,1}, \bb_{2,2}, \ldots, \bb_{2,M}]\in\real^{N\times M}$, i.e., the row partitions of $\bC_1$, $\bA$, and $\bA_2=(\bC_2\bD_2)$, respectively.
Then, the gradient with respect to  $\bc_{1,m}$ is given by:
\begin{equation}\label{equation:hada_rkone2}
	\nabla L(\bc_{1,m}) 
	=\frac{\partial L(\bc_{1,m})}{\partial \bc_{1,m}}
	= 2\bD_1 \left([(\bD_1^\top\bc_{1,m}) \hadaprod \bb_{2,m}-\bb_m ] \hadaprod\bb_{2,m}\right), \, m\in\{1,2,\ldots, M\}.
\end{equation}
The gradient for the rows of $\bC_2$ can be obtained analogously.
Therefore,  Algorithm~\ref{alg:ad_hadamad_svd} can be adapted to update the columns of $\bD_1, \bD_2$ and the rows of $\bC_1, \bC_2$ iteratively (referred to as  rank-one updates).

\subsection{Missing Entries}
The rank-one update framework can be extended to settings like the Netflix problem, in which case many entries of $\bA\in\real^{M\times N}$ are missing. 
Assuming $\bA$ is a low-rank matrix, we aim to fill in the missing entries of matrix $\bA$ (where $M$ represents the number of movies, and $N$ represents the number of users).

Let $\bo_n\in \{0,1\}^M, n\in\{1,2,\ldots, N\}$, represent the movies rated by user $n$, where $o_{nm}=1$ if user $n$ has rated movie $m$, and $o_{nm}=0$ otherwise.
Similarly, let $\bp_m \in\{0,1\}^{N}, m\in\{1,2,\ldots,M\}$ denote the users who have rated  movie $m$, with $p_{mn}=1$ if the movie $m$ has been rated by user $n$, and $p_{mn}=0$ otherwise.
Then, Equations~\eqref{equation:hada_rkone1} and~\eqref{equation:hada_rkone2} become
\begin{align}
	\nabla L(\bd_{1,n}) 
	&=2\bC_1[\bo_n,:]^\top \left(\left[(\bC_1[\bo_n,:]\bd_{1,n}) \hadaprod \ba_{2,n}[\bo_n]-\ba_n[\bo_n] \right] \hadaprod\ba_{2,n}[\bo_n]\right),\nonumber \\
	&\gap\gap\gap\gap\gap\gap\gap\gap\gap\gap\gap\gap\gap\gap\gap\gap n\in\{1,2,\ldots, N\};\label{equation:hada_rkone3} \\
	\nabla L(\bc_{1,m}) 
	&= 2\bD_1[:,\bp_m] \left(\left[(\bD_1[:,\bp_m]^\top\bc_{1,m}) \hadaprod \bb_{2,m}[\bp_m]-\bb_m[\bp_m] \right] \hadaprod\bb_{2,m}[\bp_m]\right),\nonumber\\
	&\gap\gap\gap\gap\gap\gap\gap\gap\gap\gap\gap\gap\gap\gap\gap\gap  m\in\{1,2,\ldots, M\} \label{equation:hada_rkone4}.
\end{align}
Since the Hadamard product commutes,  the gradients for $L(\bd_{2,n})$, $n\{1,2,\ldots, N\}$ and $L(\bc_{2,m})$,  $m\in\{1,2,\ldots, M\}$ can be obtained similarly due to symmetry. 
The complete procedure for predicting missing entries in $\bA$ using low-rank Hadamard decomposition is summarized in Algorithm~\ref{alg:ad_hadamad_missen}.

\begin{algorithm}[h] 
	\caption{Alternating Descent with Gradient Descent for Hadamard Decomposition with Missing Entries: A regularization can also be added into the gradient descent update (see Section~\ref{section:regularization-extention-general}).}
	\label{alg:ad_hadamad_missen}
	\begin{algorithmic}[1] 
		\Require Matrix $\bA\in \real^{M\times N}$;
		\State Initialize $\bC_1,\bC_2\in \real^{M\times K}$, and $\bD_1,\bD_2\in \real^{K\times N}$; 
		\State Choose a stoping criterion on the approximation error $\delta$;
		\State Choose  step size $\eta$;
		\State Choose the maximum number of iterations $C$;
		\State $iter=0$; \Comment{Count for the number of iterations}
		\While{$\normf{(\bC_1\bD_1)\hadaprod (\bC_2\bD_2) -\bA}^2>\delta $ and $iter<C$}
		\State $iter\leftarrow iter+1$; 
		\For{$n=1,2,\ldots, N$}
		\State $\bd_{1,n}\leftarrow\bd_{1,n}-\eta \nabla L(\bd_{1,n}) $;  \Comment{Equation~\eqref{equation:hada_rkone3}}
		\State $\bd_{2,n}\leftarrow\bd_{2,n}-\eta \nabla L(\bd_{2,n}) $;
		\EndFor
		\For{$m=1,2,\ldots, M$}
		\State $\bc_{1,m}\leftarrow\bc_{1,m}-\eta \nabla L(\bc_{1,m})  $; \Comment{Equation~\eqref{equation:hada_rkone4}}
		\State $\bc_{2,m}\leftarrow\bc_{2,m}-\eta \nabla L(\bc_{2,m})  $;
		\EndFor
		\EndWhile
		\State Output $\bC_1,\bD_1, \bC_2,\bD_2$.
	\end{algorithmic} 
\end{algorithm}

\section{Application: Movie Recommender}\label{section:movie_rec_als}
The ALS algorithm has been extensively developed for movie recommendation systems. 
To illustrate its application, we use  the ``MovieLens 100K" data set from MovieLens \citep{harper2015movielens}~\footnote{http://grouplens.org}. 
This data set is widely recognized and used in the field of recommender systems research due to its comprehensive set of user ratings for movies. 
It consists of 100,000 ratings from 943 users for 1,682 movies, with rating values ranging from 0 to 5. The data was collected through the MovieLens website
over a seven-month period from September 19th, 
1997 to April 22nd, 1998. This data has been cleaned up---users
who had less than 20 ratings or did not have complete demographic
information were removed from this data set such that simple demographic info for the users (age, gender, occupation, zip) can be obtained. However, our focus will solely be on the raw rating matrix  to evaluate how well the low-rank ALS approach can capture the underlying structure of the data, leading to accurate and meaningful recommendations.

The data set is split into training and validation set, comprising approximately 95,015 and 4,985 ratings, respectively, for fitting the ALS algorithm.
The error is quantified using the \textit{root mean squared error (RMSE)}. The RMSE is a common measure of the difference between actual and predicted values. For a set of values $\{x_1, x_2, \ldots, x_n\}$ and their predictions $\{\hat{x}_1, \hat{x}_2, \ldots, \hat{x}_n\}$, the RMSE can be described as 
$
\text{RMSE}(\bx, \hat{\bx}) = \sqrt{\frac{1}{n} \sum_{i=1}^{n}(x_i-\hat{x}_i)^2}.
$
For evaluating the ALS algorithm, the minimum RMSE for the validation set is achieved with $K=62$ and $\lambda_w=\lambda_z=0.15$, resulting in an RMSE of $0.806$ (less than 1), as shown in Figure~\ref{fig:movie100k}. 
Given that ratings range from 0 to 5, the ALS algorithm can predict whether a user is likely to enjoy a movie (e.g., ratings of 4 to 5) or not  (e.g., ratings of 0 to 2) on average due to the  RMSE score. 
\begin{figure}[h]
\centering  
\vspace{-0.35cm} 
\subfigtopskip=2pt 
\subfigbottomskip=2pt
\subfigcapskip=-5pt
\subfigure[Training set.]{\label{fig:movie100k1}
\includegraphics[width=0.47\linewidth]{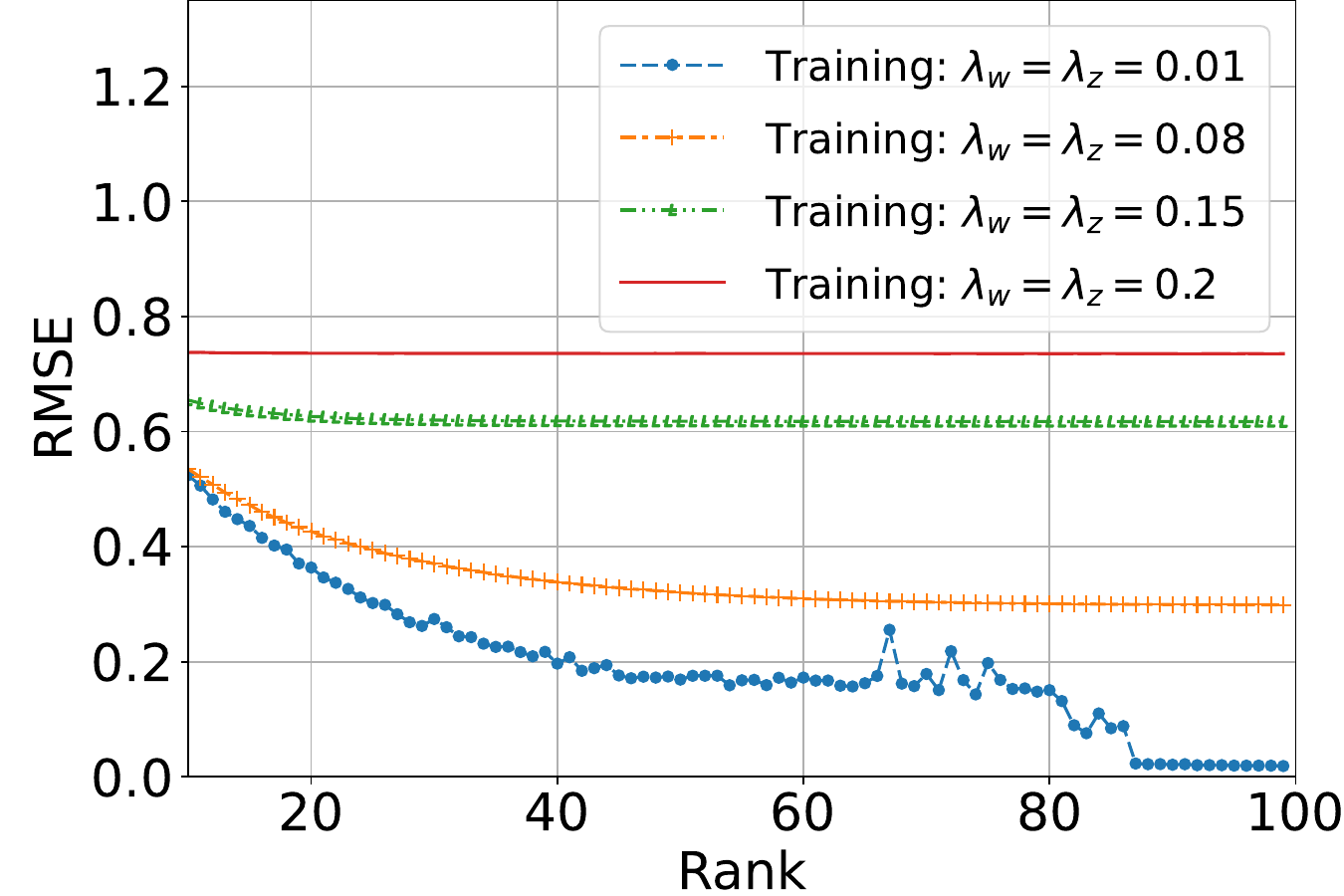}}
\quad 
\subfigure[Validation set.]{\label{fig:movie100k2}
\includegraphics[width=0.47\linewidth]{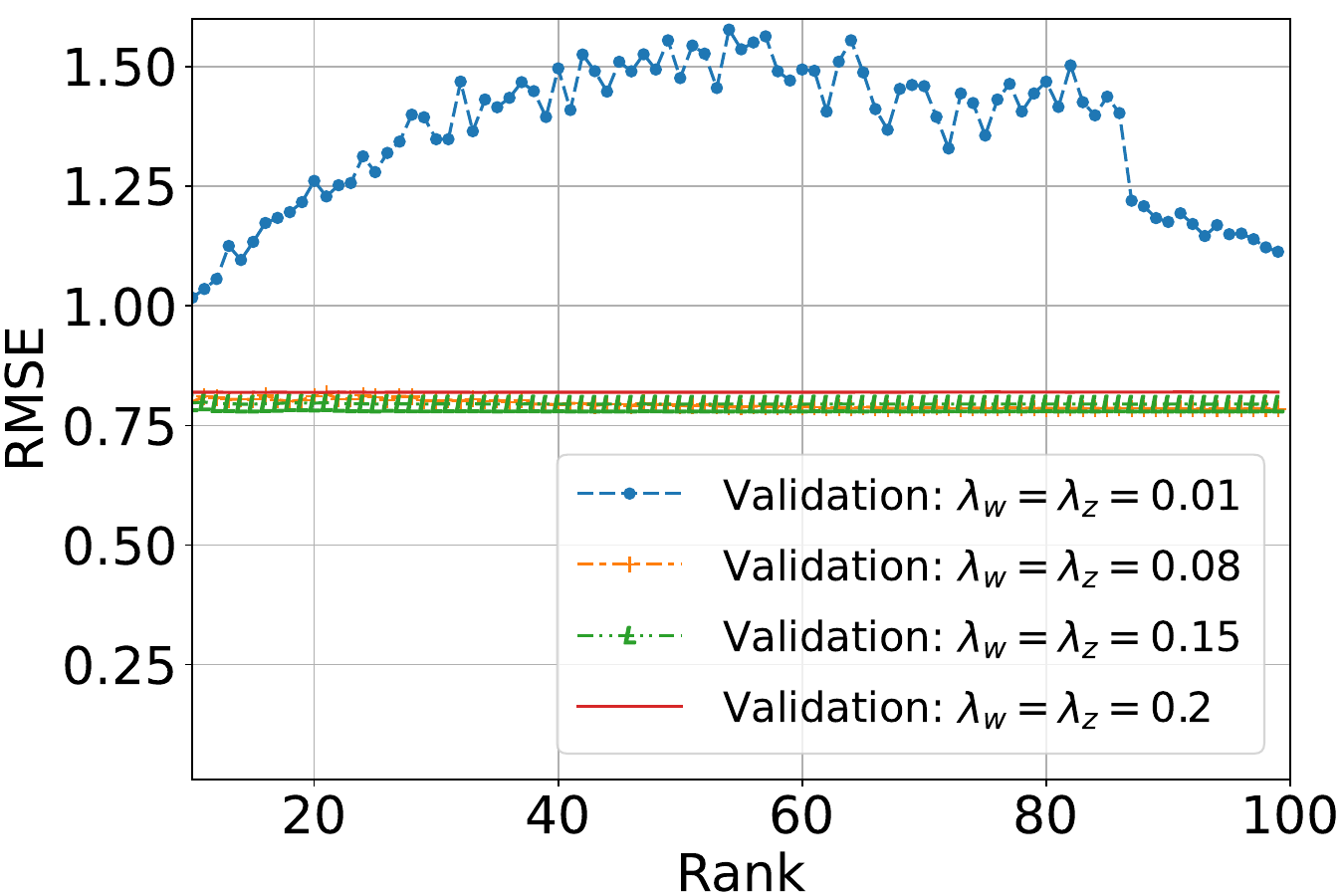}}
\caption{Comparison of training and validation error for the ``MovieLens 100K" data set with different reduction dimensions and regularization parameters.}
\label{fig:movie100k}
\end{figure}

\paragraph{Recommender 1.}
A simple recommender system suggests movie $m$ to user $n$ if $a_{mn}\geq4$ and user $n$ has not yet rated movie $m$. 
\paragraph{Recommender 2.}
Alternatively, we can recommend similar movies to those highly rated by the user. 
Suppose user $n$ has rated movie $m$ with a 5 ($a_{mn}=5$). Under the ALS approximation $\bA=\bW\bZ$, where each row of $\bW$ represents the hidden features of each movie (see Section~\ref{section:als-vector-product} on  vector inner products), the solution involves identifying the most similar movies that user  $n$ has not rated (or watched), to movie $m$. Mathematically, this is expressed as:
$$
\mathop{\argmax}_{\bw_i} \gap \text{similarity}(\bw_i, \bw_m), \qquad \text{for all} \gap i \notin \bo_n,
$$
where $\bw_i$'s are the rows of $\bW$, each representing the hidden features of movie $i$, and $\bo_n$ represents a mask vector, indicating the movies that user $n$ has already rated.

The method described above relies on a similarity function applied to two vectors. The \textit{cosine similarity} is the most commonly used measure. It is defined as the cosine of the angle between the two vectors:
$$
\cos(\bx, \by) = \frac{\bx^\top\by}{\normtwo{\bx}\cdot \normtwo{\by}},
$$
where the value ranges from $-1$ to 1, with $-1$ representing perfectly dissimilar and 1 being perfectly similar. Based on this definition, it follows that the cosine similarity depends only on the angle between the two nonzero vectors, but not on their magnitudes since it can be regarded as the inner product between the normalized  versions of these vectors. 
Another measure for calculating similarity is the \textit{Pearson similarity}:
$$
\text{Pearson}(\bx,\by) =\frac{\Cov(\bx,\by)}{\sigma_x \cdot \sigma_y}
= \frac{\sum_{i=1}^{n} (x_i - \bar{x} ) (y_i -\bar{y})}{ \sqrt{\sum_{i=1}^{n} (x_i-\bar{x})^2}\sqrt{ \sum_{i=1}^{n} (y_i-\bar{y})^2 }}.
$$
It is calculated as the ratio between the covariance of two variables and the product of their standard deviations, whose range varies between $-1$ and 1, where $-1$ is perfectly dissimilar, 1 is perfectly similar, and 0 indicates no linear relationship.  Pearson similarity is commonly used to measure the linear correlation between two sets of data. 

Both Pearson correlation and cosine similarity are widely used in  machine learning and data analysis. Pearson correlation is often used in regression analysis, while cosine similarity is commonly used in recommendation systems and information retrieval tasks. 
In our context, cosine similarity performs better in \textit{precision-recall (PR) curve} analysis. 
\begin{figure}[h]
\centering
\vspace{-0.35cm}
\subfigtopskip=2pt
\subfigbottomskip=2pt
\subfigcapskip=-5pt
\subfigure[Cosine Bin Plot.]{\label{fig:als-cosine}%
\includegraphics[width=0.32\linewidth]{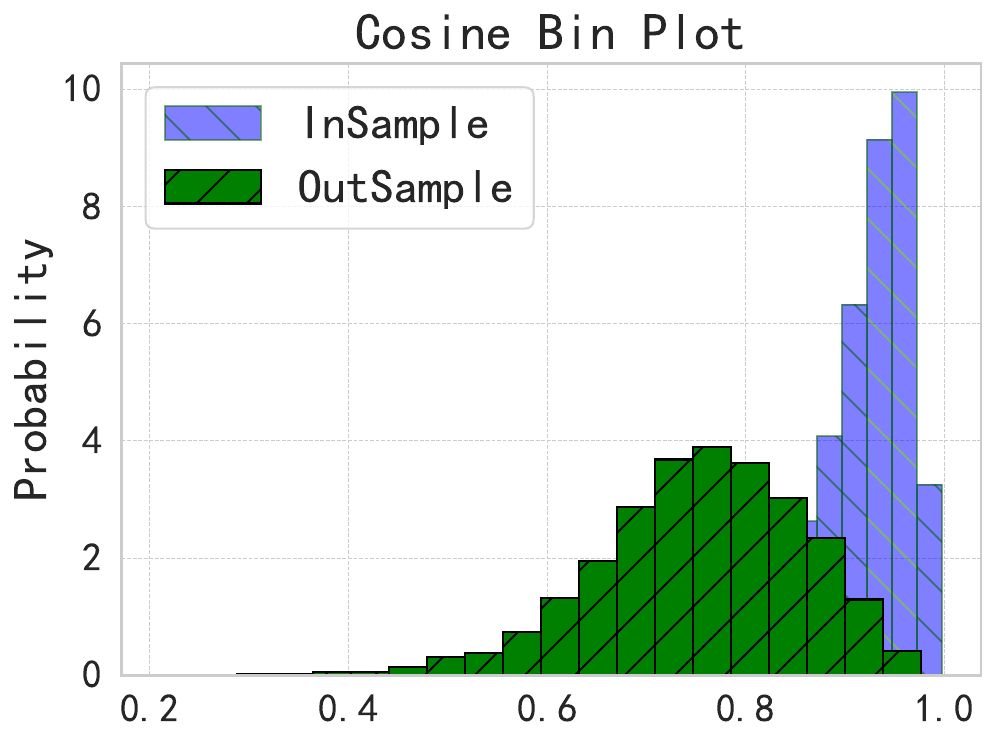}}%
\subfigure[Pearson Bin Plot.]{\label{fig:als-pearson}%
\includegraphics[width=0.32\linewidth]{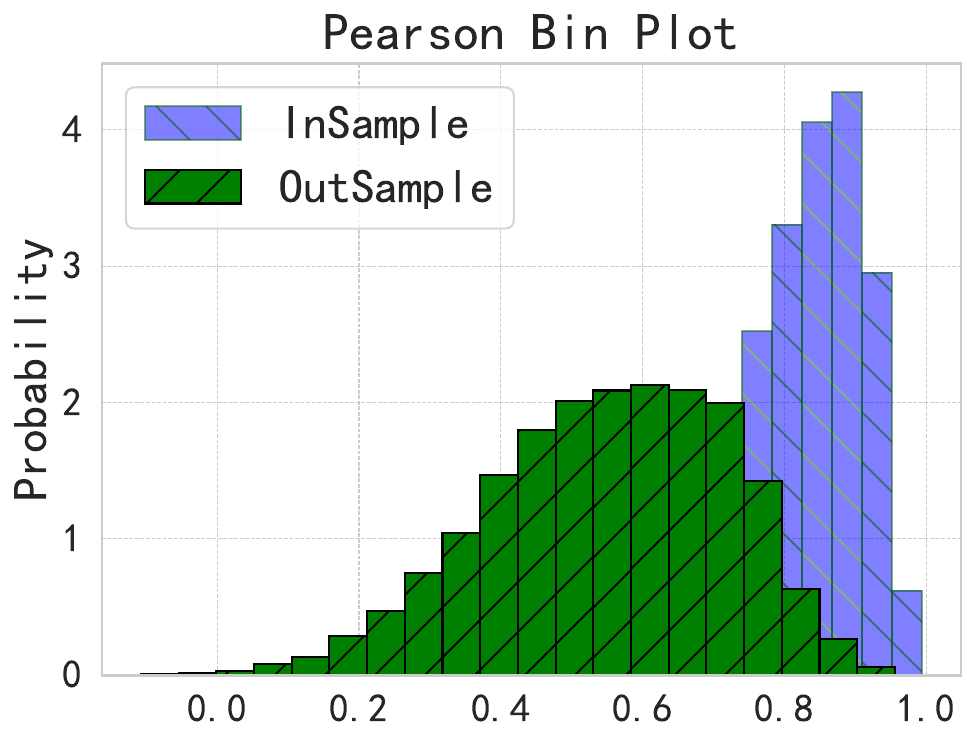}}%
\subfigure[PR Curve.]{\label{fig:als-prcurve}%
\includegraphics[width=0.35\linewidth]{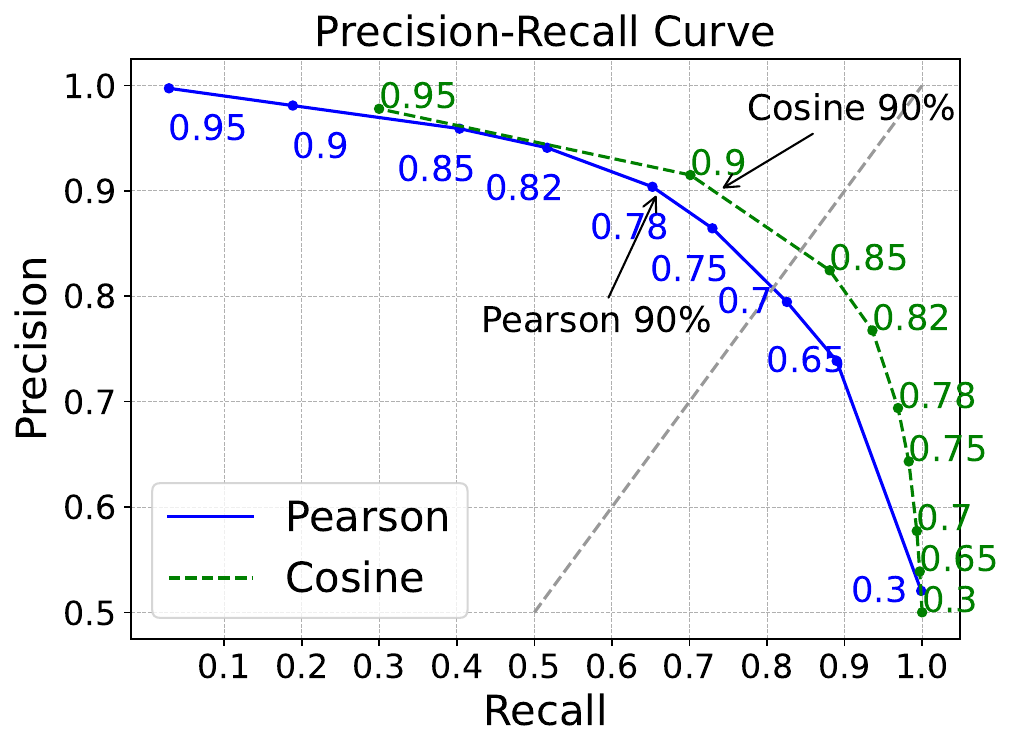}}%
\caption{Distribution of the insample and outsample using cosine and Pearson similarities, and the Precision-Recall curves for both.}
\label{fig:als-prcurive-bin}
\end{figure}

Building upon the previous example  using the MovieLens 100K data set, we set $\lambda_w=\lambda_z=0.15$ for  regularization and a rank of $62$ to minimize  RMSE. 
We aim to analyze the similarity between different movie hidden vectors, and the goal of Recommender 2 is to see whether the matrix factorization can help differentiate high-rated from low-rated movies, thereby recommending movies correlated with the user's high-rated ones.
Define further the term ``insample" as the similarity between the movies having rates $5$ for each user, and ``outsample" as the similarity between the movies having rates $5$ and $1$ for each user. Figure~\ref{fig:als-cosine} and \ref{fig:als-pearson} depict the bin plots of the distributions of insample and outsample under cosine and Pearson similarities, respectively.
In both scenarios, a clear distinction is observed between the distributions of the ``insample" and ``outsample" data, indicating that  ALS decomposition can actually find the hidden features of different movies for each user.
Figure~\ref{fig:als-prcurve} displays the \textit{precision-recall (PR) curve} for these scenarios, where we find  cosine similarity outperforms Pearson similarity, achieving over $73\%$ recall with  $90\%$ precision. However, Pearson similarity can  identify only about $64\%$ of the high-rated movies with the same precision. 
In practice, other measures, such as  \textit{negative Euclidean distance}, can also be explored. The Euclidean distance  measures the ``dissimilarity" between two vectors; and a negative value thus represents their similarity.

The ALS method for recommendation discussed here is designed for \textit{explicit data}, where the ratings provided by each user have a clear hierarchical meaning. In contrast, there are also recommendation systems for \textit{implicit data}, where the system automatically infers users' preferences by tracking their actions, such as which items they viewed, where they clicked, which products they purchased, or how long they spent on a web page.
In such cases,  ALS can be extended to more complex models, such as using a dictionary matrix to transform  the explicit data into user and item latent vectors \citep{he2017neural}, incorporating multinomial prior into a variational auto-encoder,  and enhancing the model's ability to handle implicit feedback by leveraging probabilistic modeling techniques \citep{liang2018variational}.

\index{Outlier detection}
\section{Application: Outlier Detection}
We can also use matrix decomposition algorithms to detect outlier entries in a matrix. 
Given an $ M \times N $ input matrix $ \bA $, the objective is to identify entries that significantly deviate from their reconstructed counterparts based on a low-rank approximation. Begin by selecting a rank $ K $ such that $ K < \min(M,N) $, which captures the dominant structure of the data while filtering out noise or anomalies. 
Then, we employ a matrix factorization method such as ALS, SVD, nonnegative matrix factorization (NMF; see Chapter~\ref{chapter:nmf}), or another suitable decomposition technique to approximate $ \bA $ as $ \bW\bZ $, where $ \bW \in \real^{M \times K} $ and $ \bZ \in \real^{K \times N} $. Subsequently, reconstruct the approximation $ \widetildebA = \bW\bZ $ and compute the element-wise reconstruction error matrix $ \bE \in \real^{M\times N} $ defined by $ e_{ij} = \abs{a_{ij} - \widetildea_{ij}} $ for all $ i,j $. Establish a threshold $ \tau $ either through statistical means---such as mean absolute deviation or quantile-based methods---or via domain-specific criteria. 
Finally, we can declare an entry $ a_{ij} $ as an outlier if $ e_{ij} > \tau $. The identified outliers correspond to those indices where the deviation between the original and reconstructed values exceeds the expected variability captured by the low-rank model. 
This approach exploits the ability of matrix factorization to capture the main patterns in the data. As a result, large reconstruction errors often indicate potential anomalies or unusual behavior.

\section{Application: Spectral Clustering and Link Prediction}\label{section:spec_clust}

In Section~\ref{section:feat_eng_scatt}, we discussed how spectral decomposition can be applied to perform spectral clustering on adjacency matrices of undirected graphs.
However, adjacency matrices of directed graphs are asymmetric.
For example, in a social network like Twitter, a teenager might follow a famous rock star, but the rock star may not follow the teenager in return.

Given an adjacency matrix $\bA\in\real^{n\times n}$ of a directed graph, ALS can be used to find the low-dimensional representation for each node such that $\bA=\bU\bV^\top$, where $\bU, \bV\in\real^{n\times k}$.
Unlike the symmetric case, the $k$-dimensional row vectors of $\bU$ and $\bV$ can be concatenated to form $2k$-dimensional embedded representations for each node.

These $2k$-dimensional embeddings can then be used for spectral clustering of the nodes, similar to the approach outlined in Section~\ref{section:feat_eng_scatt}.
Alternatively, once the factorization $\bA=\bU\bV^\top$ is obtained, the reconstructed matrix $\bU\bV^\top$ can be used to predict links---just as in the Netflix recommendation problem.
In such cases, when the connections between nodes are binary (i.e., an implicit data set), \textit{logistic matrix decomposition} using ALS can also be applied to model the probability of a link; see Problem~\ref{prob:logis_mf}.

\begin{problemset}

\item \label{prob:als_pseudo1} \textbf{Least squares for rank-deficiency \citep{lu2021rigorous}.} Let $\bA\in\real^{m\times n}$ and $\bb\in\real^m$. Show that the least squares problem $L(\bx)=\normtwo{\bA\bx-\bb}^2$ has a minimizer $\bx^*\in\real^n$ if and only if there exists a vector $\by\in\real^n$ such that $\bx^*=\bA^+\bb+(\bI-\bA^+\bA)\by$, where $\bA^+$ denotes the pseudo-inverse of $\bA$. 
Show that:
\begin{itemize}
\item The least squares has a \textbf{unique} minimizer of $\bx^*=\bA^+\bb$ only when $\bA^+$ is a left inverse of $\bA$ (i.e., $\bA^+\bA=\bI$). The solution in Lemma~\ref{lemma:ols} is a special case.
\item The optimal value is $L(\bx^*)=\bb^\top(\bI-\bA\bA^+)\bb$.
\item If $\by\neq \bzero$, then $\normtwo{\bA^+\bb}\leq \normtwo{\bA^+\bb+(\bI-\bA^+\bA)\by}$.
\end{itemize}
\textit{Hint: See LS via SVD in Section~\ref{section:application-ls-qr}.}

\item \label{prob:als_pseudo2} \textbf{Least squares for rank-deficiency.} Let  $\bA\in\real^{m\times n}$ and $\bB\in\real^{m\times p}$. Show that the least squares problem $L(\bX) = \normf{\bA\bX-\bB}^2$ has a minimizer $\bX^*=\bA^+\bB\in\real^{n\times p}$. Determine all the minimizers using Problem~\ref{prob:als_pseudo1}.

\item \label{prob:als_pseudon} \textbf{Least squares for rank-deficiency.}   Let $\bA\in\real^{m\times n}$ and $\bB\in\real^{p\times n}$. Show that the least squares problem $L(\bX) = \normf{\bX\bA-\bB}^2$ has a minimizer $\bX^*=\bB\bA^+\in\real^{p\times m}$.
	
\item Prove Lemma~\ref{lemma:als-update-w-rank}.

\item \label{prob:sep_conv} \textbf{Marginally convex.} Let $D(\bA, \bB)$ be convex in the second argument $\bB$. Show that $D(\bA,\bW\bZ)$ is convex in $\bW$ for a  fixed $\bZ$, and vice versa.

\item Derive the column-by-column update rules for Algorithm~\ref{alg:als-regularizer}.

\item \label{prob:ortho_mf} \textbf{Orthogonal and projective  matrix factorization.} Consider the optimization $\mathopmin{\bW}\normf{\bA-\bW\bZ}^2$ such that $\bZ\bZ^\top=\bI_K$, where $\bA\in\real^{M\times N}, \bW\in\real^{M\times K}, \bZ\in\real^{K\times N}$, and $K\leq \min\{M,N\}$. Show that the optimal value $\bW^*$ given $\bZ$ is $\bA\bZ^\top$.
This indicates that the matrix factorization optimization can be equivalently stated as $\mathopmin{\bZ\bZ^\top=\bI_K}\normf{\bA-\bA\bZ^\top\bZ}^2$. And the relaxed problem is called the \textit{projective matrix factorization} \citep{yuan2005projective, yang2010linear}:
$$
\mathopmin{\bZ}\normf{\bA-\bA\bZ^\top\bZ}^2,
$$
where each row of $\bA$ is projected onto a $K$-dimensional subspace, hence the name.
The interpretations of orthogonal and projective matrix factorizations are further discussed in Problem~\ref{prob:ortho_nmf}.

\item \label{problem:rls} \textbf{Regularized least squares (RLS).} 
Given $\bA\in\real^{m\times n}, \bb\in\real^{m}, \bB\in\real^{p\times n}$, and $\lambda\in\real_{++}$, we consider the regularized least squares  (RLS) problem:
$$
\mathop{\min}_{\bx\in\real^n} \normtwo{\bA\bx-\bb}^2 + \lambda\normtwo{\bB\bx}^2.
$$
Show that this regularized least squares problem has a unique solution if and only if $\nspace(\bA)\cap \nspace(\bB) = \{\bzero\}$.

\item \label{prob:denoise_rls} \textbf{Denoising via RLS.} Consider a noisy measurement of a signal $\bx\in\real^n$:
$
\by = \bx+\be,
$
where $\by$ is the observed measurement, and $\be$ is the noise vector. We want to find an estimate $\bx$ of the observed measurement $\by$ such that $\bx \approx \by$:
$
\min \normtwo{\bx-\by}^2.
$
Apparently, the optimal solution of this optimization is given by $\bx=\by$; however, it is meaningless.
To improve the estimate,  we can add a penalty term for the differences between  consecutive observations:
$
R(\bx) = \sum_{i=1}^{n-1} (x_i - x_{i+1})^2.
$
Then, 
\begin{itemize}
\item Find the regularized least squares representation for this problem and derive the regularized least squares solution.
\item Find some applications of this denoising problem. For example, when we model the profit and loss signal of a financial asset, the two observations over consecutive days of the underlying asset should exhibit smooth transitions rather than abrupt changes.
\end{itemize}


\item \textbf{Weighted least squares (WLS).}
Building upon the assumptions in Lemma~\ref{lemma:ols}, we consider further that each data point $i\in\{1,2,\ldots, m\}$ (i.e., each row of $\bA$) has a weight $w_i$. 
This means some  data points may carry greater significance than others, and we can produce approximate minimizers that reflect this.
Show that the value $\bx_{WLS} = (\bA^\top\bW^2\bA)^{-1}\bA^\top\bW^2\bb$ serves as the \textit{weighted least squares (WLS)}  estimate of $\bx$, where $\bW=\diag(w_1, w_2, \ldots, w_m)\in\real^{m\times m}$. \textit{Hint: Find the normal equation for this problem.}

\item \textbf{Positive definite  weighted least squares (PDWLS).}
Building upon the assumptions in Lemma~\ref{lemma:ols}, we consider further  the matrix equation $\bA\bx + \be =\bb$, where $\be$ is an error vector. Define the weighted error squared sum $E_w = \be^\top \bW \be$, where the weighting matrix $\bW$ is  positive definite. 
Show that the positive definite weighted least squares solution is $\bx^* = (\bA^\top\bW\bA)^{-1}\bA^\top\bW\bb$. \textit{Hint: Compute the gradient of $E_w = (\bb-\bA\bx)^\top\bW(\bb-\bA\bx)$.}

\item \textbf{Weighted color noise least squares.}
Building upon the assumptions in Lemma~\ref{lemma:ols}, we consider  the matrix equation $ \bA\bx + \be = \bb $, where $\be$ is an additive color noise vector satisfying the conditions $\Exp[\be] = \bzero$ and $\Exp[\be\be^\top] = \bSigma$, where $\bSigma$ is known. Use the weighting error function $E_w = \be^\top \bW \be$ as the loss function for finding the optimal estimate $\bx^*$. Show that
$\bx^* = (\bA^\top \bW \bA)^{-1} \bA^\top \bW \bb$,
where the optimal choice of the weighting matrix $\bW$ is $\bW^* = \bSigma^{-1}$.
\textit{Hint: Compute the gradient of $E_w = (\bb-\bA\bx)^\top\bW(\bb-\bA\bx)$.}

\item \label{problem:tls} \textbf{Transformed least squares (TLS).}
Building upon the assumptions in Lemma~\ref{lemma:ols}, we consider further the restriction $\bx=\bC\bgamma+\bc$, where $\bC\in\real^{n\times k}$ is a known matrix such that $\bA\bC$ has full rank, $\bc$ is a known vector, and $\bgamma$ is an unknown vector.
Show that the value $\bx_{TLS}=\bC(\bC^\top\bA^\top\bA\bC)^{-1}(\bC^\top\bA^\top)(\bb-\bA\bc) +\bc$ serves as the \textit{transformed least squares (TLS)} estimate of $\bx$.


\item \label{problem:twls2} Find the transformed weighted least squares estimate.

\index{Weighted matrix decomposition}
\item \label{prob:weight_mf} \textbf{ALS for weighted matrix decomposition.} Let $\bA\in\real^{m\times n}$ be the data matrix and $\bM\in\real^{m\times n}$ be the corresponding weight matrix, where each entry of $\bM$ represents the weight associated with the corresponding entry in $\bA$. Derive the ALS algorithm of the \textit{weighted matrix decomposition} problem:
\begin{equation}
	L(\bW, \bZ) = \normf{\bM\hadaprod(\bA-\bW\bZ)}^2.
\end{equation}
In the context of  \textit{implicit feedback data}, each entry of $\bA$ takes on a binary value; for example,  1 if the user rates a movie with 3, 4, or 5 stars, and 0 if the rating is 0, 1, or 2 stars.
Suppose that matrix $\bB\in\real^{m\times n}$ contains the original raw ratings of the matrix (ranging from 0 to 5). A relaxation of this implicit feedback representation can be achieved using a weight matrix  $\bM$, defined as:
$$
m_{ij} = 1+ \gamma \cdot b_{ij}, \quad \forall\,i,j,
$$
where the parameter $\gamma$ is typically set to a large value, such as $\gamma=40$. 
Discuss the relationship between this formulation of weighted matrix decomposition and the standard matrix decomposition used for implicit feedback data.

\index{Logistic matrix decomposition}
\item  \label{prob:logis_mf} \textbf{ALS for logistic matrix decomposition.}  Let $\bA\in\{0,1\}^{m\times n}$ be the observed binary  data matrix. The logistic matrix decomposition aims to factor $\bA$ into $\bW\in\real^{m\times k}$ and $\bZ\in\real^{k\times n}$ such that 
$$
L(\bW,\bZ) = \normf{\sigma(\bW\bZ)-\bA}^2 
$$
is minimized, where $\sigma(x) = 1/(1+\exp\{-x\})$ denotes the logistic sigmoid function, applied element-wise to the matrix 
$\bW\bZ$.
Alternatively, since we want to learn $\bW$ and $\bZ$ such that $\sigma(\bW\bZ)_{ij}$  has a large value when $a_{ij}$ is 1, and a small value when $a_{ij}$ is 0, we can use the \text{log-likelihood loss function}:
\begin{equation}
	L(\bW, \bZ) = -\sum_{i,j=1}^{m,n}\big[a_{ij} \ln\big(\sigma(\bw_i, \bz_j)\big) + (1-a_{ij})\ln\big(1-\sigma(\bw_i, \bz_j)\big)\big], 
\end{equation}
where $\bw_i$ and $\bz_j$ denote the $i$-th row of $\bW$ and the $j$-th column of $\bZ$, respectively.
Derive the ALS steps for this loss function, and discuss how it can be applied to the implicit feedback data introduced in Problem~\ref{prob:weight_mf}.

\index{Shared matrix decomposition}
\item \label{prob:shared_mf} \textbf{ALS for shared matrix decomposition.} In the main text, we consider a matrix, $\bA\in\real^{m\times n}$, representing the use-movie matrix, where   rows correspond to movies, columns to users, and entries contain the observed ratings. The product $\bW\bZ$ reconstructs the full rating matrix, including predictions for missing entries.
Suppose we are further given a user-book matrix $\bB\in\real^{p\times n}$, where the rows correspond to book items, and the columns contain the same set of users. To leverage both types of data (movies and books), we consider the \textit{shared matrix decomposition} problem:
\begin{equation}
	\min_{\bW,\bY,\bZ} L(\bW, \bY, \bZ) = \normf{\bA-\bW\bZ}^2 + \normf{\bB-\bY\bZ}^2 + \lambda(\normf{\bW}^2+\normf{\bY}^2+\normf{\bZ}^2).
\end{equation}
Since the columns of $\bZ$ represent latent features of users, these representations gain information from both the movie and book data.
Derive the ALS steps for this problem.

\index{First-order optimality condition}
\index{Fermat's theorem}
\item \label{problem:fist_opt} \textbf{First-order optimality condition for local optima points.} 
Consider  \textit{Fermat's theorem}: for a one-dimensional function $g(\cdot)$ defined and differentiable over an interval ($a, b$), if a point $x^*\in(a,b)$ is a local maximum or minimum, then $g^\prime(x^*)=0$. 
Prove the first-order optimality conditions for multivariate functions based on  Fermat's theorem for one-dimensional functions.
That is, let  $f: \sS\rightarrow \real$ be a function defined on a set $\sS\subseteq \real^n$. Suppose that $\bx^*\in\text{int}(\sS)$, i.e., in the interior point of the set, is a local optimum point and that all the partial derivatives of $f$ exist at $\bx^*$. Then $\nabla f(\bx^*)=\bzero$, i.e., the gradient vanishes at all local optimum points. (Note that, this optimality condition is a necessary condition but not sufficient; however, there could be vanished points which are not local maximum or minimum points.)
\textit{Hint: Consider the one-dimensional function $g(t)=f(\bx^* + t\be_i)$ for $i\in\{1,2,\ldots,n\}$.}

\item \label{problem:pos_hessian} \textbf{Global minimum point of convex functions.} Let the function $f$ be a twice continuously differentiable function defined over $\real^n$. Suppose that the Hessian $\nabla^2f(\bx) \succeq 0$ for any $\bx\in\real^n$ (i.e., the Hessian is always positive semidefinite~\footnote{Instead, if we assume the Hessian is positive semidefinite at a given point, then the point is a local minimum point.}).
This property is also referred to as the \textit{convexity}.
Show that $\bx^*$ is a global minimum point of $f$ if $\nabla f(\bx^*)=\bzero$. \textit{Hint: Use the linear approximation theorem from Taylor's expansion.}

\item \textbf{Two-sided matrix least squares \citep{friedland2007generalized, aggarwal2020linear}.} Let $\bB$ be an $M\times K$ matrix and $\bC$ be a $P\times N$ matrix. Find the $K\times P$ matrix $\bX$ such that $L(\bX)=\norm{\bA - \bB\bX\bC}_F^2$ is minimized, where $\bA\in\real^{M\times N}$ is known. 
\begin{itemize}
\item Derive the derivative of $L$ with respect to $\bX$ and the optimality conditions. 
\item Show that one possible solution to the optimality conditions is $\bX^*=\bB^+\bA\bC^+$, where $\bB^+$ and $\bC^+$ are the pseudo-inverses of $\bB$ and $\bC$, respectively.
\end{itemize}
Similarly, consider the optimization with $\rank(\bX)\leq p$:
$
L(\bX)=\norm{\bA - \bB\bX\bC}_F^2$, s.t. 
$\rank(\bX)\leq p$.
Show that 
\begin{itemize}
\item One possible solution to this is $\bX^*=\bB^+\bA_p\bC^+$, where $\bA_p$ a truncated SVD of $\bB\bB^+\bA\bC^+\bC$ by replacing all but the $p$ largest singular values by zero.
\item  $\bX^*$ also minimizes $\normf{\bX}$, i.e., has the smallest magnitude among all  solutions.
\item $\bX^*$ is the \textbf{unique} solution if and only if either $\rank(\bB\bB^+\bA\bC^+\bC)\leq p$ or both $\rank(\bB\bB^+\bA\bC^+\bC)\geq p$ and $\sigma_{p+1}(\bB\bB^+\bA\bC^+\bC) < \sigma_{p}(\bB\bB^+\bA\bC^+\bC)$.
\end{itemize}

\item \label{problem:mono_gd} \textbf{Monotonic progress of gradient descent.} Consider the gradient descent for a differentiable function $f(\bx):\real^n\rightarrow \real$ that is $L$-strongly smooth~\footnote{A continuously differentiable function $f:\real^n\rightarrow \real$ is called  \textit{$L$-strongly smooth (SS)} if, for every $\bx,\by\in\real^n$, it follows that 
	$
	f(\by)-f(\bx)-f(\bx)^\top (\by-\bx)
	\leq \frac{L}{2} \norm{\bx-\by}^2$.}. Suppose the iterate $\bx^{(t+1)}$ is obtained from iterate $\bx^{(t)}$ by 
$
\bx^{(t+1)} = \bx^{(t)} - \eta\nabla f(\bx^{(t)}).
$
Show that 
\begin{itemize}
\item If the step size $\eta \leq \frac{2}{L}$,  the function value $f$ is nonincreasing: $f(\bx^{(t+1)})\leq f(\bx^{(t)})$.
\item If the step size $\eta \in [\frac{1}{2L}, \frac{1}{L}]$,  the gradient satisfies $\normtwo{\nabla f(\bx^{(t)})}\leq \epsilon$ after $T=\mathcalO(\frac{1}{\epsilon^2})$ steps.
\end{itemize}

\item \label{problem:nuclear_equi} \citep{rennie2005fast, mazumder2010spectral} Consider the nuclear norm~\footnote{The nuclear norm is defined as the sum of singular values of a matrix and provides the tightest convex envelope of the rank function of a matrix.} $\norm{\bA}_n$ of any matrix $\bA\in\real^{m\times n}$ of rank $r$. Show that 
$$
\norm{\bA}_n = \mathop{\min}_{\substack{\bW\in\real^{m\times r} \\ 
\bZ\in\real^{r\times n} \\
}}
\frac{1}{2} (\normf{\bW}^2 + \normf{\bZ}^2)
\gap \text{s.t.} \gap 
\bA=\bW\bZ.
$$

\item \label{problem:rank_hada_prod} Let $\bA_1,\bA_2\in\real^{m\times n}$ be any $m\times n$ matrices of rank $r_1$ and rank $r_2$, respectively.
Show that their Hadamard product $\bA_1\hadaprod\bA_2$ has rank at most $r_1\cdot r_2$: $\rank(\bA_1\hadaprod\bA_2)\leq \rank(\bA_1)\rank(\bA_2)$.

\item \textbf{Modified LS.} Consider a modified least squares problem of minimizing $\normtwo{\bA \bx - \bb}^2 + \bc^\top \bx$, where $\bA\in\real^{m\times n}$, $\bx, \bc\in\real^n$, and $\bb\in\real^m$. 
Show that the problem can be reduced to the standard least squares problem as long as $\bc$ lies in the row space of $\bA$. What happens when $\bc$ does not lie in the row space of $\bA$? \textit{Hint: First examine the univariate version of this problem.}

\end{problemset}

\newpage
\chapter{Nonnegative Matrix Factorization (NMF)}\index{NMF}\label{chapter:nmf}

\index{Decomposition: NMF}
\index{Sparsity}
\index{Nonnegativity constraint}
\section{Nonnegative Matrix Factorization}
In the era of big data, extracting meaningful patterns and latent structures from high-dimensional data sets has become a central challenge in various scientific and technological fields. 
Singular value decomposition (SVD) is supported by strong theoretical foundations and is applicable in a wide range of contexts. However, it  has certain limitations; for example, when applied to a nonnegative matrix~\footnote{Nonnegative matrices possess unique properties in linear algebra and are crucial for theoretical analysis; see Problems~\ref{prob:nonn_lin_1}$\sim$\ref{prob:nonn_lin_12}.},  SVD may produce negative values in the resulting factors, which can be difficult to interpret meaningfully.
To overcome this limitation, \textit{nonnegative matrix factorization (NMF)} has emerged as a powerful and interpretable tool for dimensionality reduction, feature extraction, and discovering latent structures within complex data.
Early consideration of the NMF problem was due to \citet{paatero1994positive, cohen1993nonnegative}, who referred to it as \textit{positive matrix factorization}. 
Later, \citet{lee2001algorithms} popularized the problem with the introduction of the \textit{multiplicative update} rule.

Following the discussion of matrix factorization using the alternating least squares (ALS) method, we now turn to algorithms for solving the NMF problem: 
\begin{itemize}
\item Given a nonnegative matrix $\bA\in \real_+^{M\times N}$ of rank $r$, find nonnegative matrix factors $\bW\in \real_+^{M\times K}$ and $\bZ\in \real_+^{K\times N}$ such that: 
$
\bA\approx\bW\bZ.
$
\end{itemize}

As discussed in the ALS section, a fundamental challenge in linear data analysis involves transforming or decomposing a high-dimensional data vector into a linear combination of lower-dimensional vectors. This transformation captures the essential characteristics of the original data, making it suitable for tasks such as pattern recognition. Consequently, these lower-dimensional vectors are often referred to as \textit{``hidden vectors," ``pattern vectors," or ``feature vectors."}
When conducting data analysis, building models, and processing information, two primary requirements for a pattern vector are essential:
\begin{itemize}
\item \textit{Interpretability}. Each component of a pattern vector should possess clear physical or physiological significance, allowing for a meaningful interpretation  of the underlying data.
\item \textit{Statistical fidelity}. In cases where the data are reliable and contain minimal error or noise, the components of a pattern vector should effectively capture the variability within the data, reflecting its primary structure and distribution of information.
\end{itemize}
The NMF approach addresses these issues in various applications. For example:
\begin{itemize}
\item In document collections, documents are represented as vectors, with each vector element indicating the frequency (often weighted) of a specific term within the document. Arranging these document vectors sequentially forms a nonnegative term-by-document matrix, which provides a numerical representation of the entire document collection.
\item In image collections, each image is depicted by a vector, where each vector element represents a pixel. The value of each element, a nonnegative number, reflects the intensity and color of the corresponding pixel, leading to a nonnegative pixel-by-image matrix.

\item In gene expression analysis, observations from gene sequences under different experimental conditions are compiled into gene-by-experiment matrices. These matrices encapsulate the variations in gene expression across experiments.

\item For item sets or recommendation systems, customer purchase histories or ratings for a selection of items are recorded in a nonnegative sparse matrix. This matrix efficiently captures the sparse nature of user interactions with a large number of potential items.
\end{itemize}

Unlike arbitrary linear combinations, the linear combinations in the NMF context involve only nonnegative weights of nonnegative \textit{template vectors} (or \textit{basis vectors}, i.e., the columns of $\bW$). 
This prevents phenomena such as \textit{destructive interference}, where a positive component could be canceled out by adding a negative component.  Instead,  data vectors must be explained using purely constructive methods, involving only positive components.
The nonnegativity constraint  inherently imposes   \textbf{sparsity}, enabling the factorization to capture additive features, which is especially advantageous in applications where parts-based representations are meaningful. 
This property has led to its widespread use in fields such as text mining, image processing, document analysis, and bioinformatics, where the identified components often correspond to distinct parts or features.
For example, in image processing, NMF has proven valuable for tasks such as  object detection,  image segmentation, and facial recognition \citep{lee2001algorithms, gillis2014and, gillis2020nonnegative}. The decomposition into nonnegative components aligns with the intuitive notion that images are composed of identifiable parts.
In the topic recovery problem, each column of $\bA$ denotes a document;  NMF aligns with a soft clustering approach, where each column of $\bW$ represents a topic, and the positive entries of each column of $\bZ$ represent the  positive weights of each document for those topics \citep{shahnaz2006document}.
On the other hand, a nonnegative matrix factorization $\bA\approx \bW\bZ$ can be applied directly for clustering algorithms. 
Specifically, the data vector $\ba_j$ is assigned to cluster $i$ if $z_{ij}$ is the largest element in column $j$ of $\bZ$ \citep{brunet2004metagenes, gao2005improving}.
For further applications, see the survey by \citet{berry2007algorithms}.
In conclusion, the popularity of NMF stems from its ability to automatically extract sparse and easily interpretable factors.

To measure the quality of the  approximation, we evaluate the loss by computing  the Frobenius norm of the difference between the original matrix and its reconstruction:
\begin{equation}\label{equation:frob_nmf}
L(\bW,\bZ) = D(\bA, \bW\bZ) = \frac{1}{2}\normf{\bW\bZ-\bA}^2,~\footnote{
Note that the factor $\frac{1}{2}$ is included for analytical convenience in derivative calculations.}
\end{equation}
where $L(\bW,\bZ)$ indicates it is a loss function w.r.t. $\bW$ and $\bZ$, and $D(\bA, \bW\bZ)$ implies it is a distance/divergence between $\bA$ and $\bW\bZ$ (we will use the two notations interchangeably  as needed).
The Frobenius norm is arguably the most widely used norm for NMF because it corresponds to Gaussian additive noise, which is reasonable in many situations and allows for the design of particularly efficient algorithms.
For nonnegative data, Gaussian noise can be interpreted as a truncated version of standard Gaussian noise \citep{lu2023bayesian}.
In later sections, we will extend this approach to include more general  $\beta$-divergences (Section~\ref{section:beta_div_altmu}).

 When we want to find two nonnegative matrices $\bW\in\real^{M\times r}_+$ and $\bZ\in\real_+^{r\times N}$ such that $\bA=\bW\bZ$, the problem is known as the \textit{Exact NMF} of $\bA$ of size $r$. However, exact NMF is NP-hard \citep{vavasis2010complexity, gillis2020nonnegative}.
Therefore, we focus on the approximate NMF formulation in this discussion.
\index{Overfitting}
In the context of collaborative filtering, it is recognized  that  NMF via multiplicative updates can result in overfitting  despite favorable convergence properties.
The overfitting issue can be partially mitigated through regularization, but its out-of-sample performance may still be limited. 
Bayesian optimization through the use of generative models, on the other hand, can effectively prevent overfitting in nonnegative matrix factorization \citep{brouwer2017comparative, lu2022flexible, lu2023bayesian}.

In the following sections, we introduce several methods for solving NMF problems and provide a brief overview  of their applications.

\index{Bayesian inference}
\index{Bayesian optimization}
\index{Bayesian matrix decomposition}

\begin{algorithm}[h] 
\caption{Projected Gradient Descent Method}
\label{alg:pgd_gen}
\begin{algorithmic}[1] 
\Require A function $f(\bx)$ and a set $\sS$; 
\For{$t=1,2,\ldots$}
\State Pick a step size $\eta_t$;
\State Set $\bx^{(t+1)} \leftarrow \mathcalP_{\sS}(\bx^{(t)} - \eta_t \nabla f(\bx^{(t)}))$;
\EndFor
\State Output final  $\bx$;
\end{algorithmic} 
\end{algorithm}
\section{NMF via Alternating Projected Gradient Descent (APGD)}\label{section:nmf_apgd}
The projected gradient descent (PGD, Algorithm~\ref{alg:pgd_gen}) is designed to minimize a function over a constraint set $\sS$:
$$
\mathopmin{\bx\in\sS} f(\bx).
$$
The \textit{orthogonal projection} onto $\sS$ is defined as
$
\mathcalP_{\sS} (\bx) = \mathop{\argmin}_{\by\in\sS} \normtwo{\by-\bx}.
$
When $\sS$ is the nonnegative orthant, the projection $\mathcalP_{\sS} (\bx)$ simplifies to $\mathcalP_{\sS} (\bx) = \max\{\bzero, \bx\}$, where the max operator is applied componentwise.

Therefore, the \textit{alternating PGD (APGD)} approach for NMF updates the factored components iteratively by 
$$
\bZ\leftarrow \max\bigg\{\bzero, \mathop{\argmin}_{\bZ\in\real^{K\times N}}\normf{\bW\bZ-\bA} \bigg\}
\gap\text{and}\gap
\bW\leftarrow \max\bigg\{\bzero, \mathop{\argmin}_{\bW\in\real^{M\times K}}\normf{\bW\bZ-\bA} \bigg\},
$$
where each update can be solved using a least squares method followed by projection onto the nonnegative orthant.
However, due to the projection, the solution may not be properly scaled. A closed-form scaling factor $\gamma$ can be applied at each iteration to improve the approximation:
$$
\gamma^* = \mathop{\argmin}_{\gamma\geq 0} \normf{\gamma\bW\bZ-\bA} 
=
\frac{\langle \bA, \bW\bZ\rangle}{\langle \bW\bZ, \bW\bZ\rangle}
=
\frac{\langle \bA\bZ^\top, \bW\rangle}{\langle \bW^\top\bW, \bZ\bZ^\top\rangle}.
$$
While it is generally not advised to use APGD due to its convergence challenges, APGD can be quite effective as an initialization method. This approach involves running a few iterations of APGD before switching to a different NMF algorithm, which is particularly beneficial for sparse matrices \citep{gillis2014and}.

\index{KKT condition}
\index{Nonnegative least squares}
\index{NNLS|see {Nonnegative least squares}}
\index{ANLS|see {Nonnegative least squares}}
\section{NMF via Alternating Nonnegative Least Squares (ANLS)}\label{section:nmf_anls}
A fundamental component of  the ALS approach is the least squares problem (Lemma~\ref{lemma:ols}). 
For NMF, we focus on the \textit{nonnegative least squares (NNLS)} problem:
\begin{equation}
\mathopmin{\bx\geq \bzero } f(\bx) = \mathopmin{\bx\geq \bzero } \frac{1}{2}\normtwo{\bb-\bM\bx}^2
\gap
\text{with }\bM\in\real^{m\times n}, \bb\in\real^m, \bx\in\real_+^n. 
\end{equation}
The KKT conditions  imply the complementary slackness condition $\lambda_ix_i^*=0, \forall i$, where $\lambda_i$ is the Lagrangian multiplier; and the optimal condition $\nabla f(\bx^*) -\sum_{i}\lambda_i \be_i=\bzero$, where $\bx^*$ denotes the optimal solution of the NNLS problem.
Together, the complementary slackness and the optimal condition indicate that:
$$
\nabla f(\bx^*)
=
\sum_{i:x_i^*=0} \lambda_i \be_i.
$$
From this, we derive the following equivalent KKT conditions for NNLS:
\begin{equation}\label{equation:kkv_nnn_raw}
(\textbf{KKT of NNLS})\gap 
	\bx^*\geq \bzero, 
	\gap
	\nabla f(\bx^*)\geq 0, 
	\gap
	\text{and}
	\gap
	x_i^* (\nabla f(\bx^*))_i=0,\, \forall i.
\end{equation}
These conditions imply sparsity when the nonnegative constraint     is applied, meaning the NNLS or NMF problem inherently  imposes a \textbf{sparsity constraint}.

Assume we are given the inactive set $\sI\subseteq \{1,2,\ldots,n\}$:
$$
\sI = \left\{ i \mid x_{i}^{*} > 0, \,\forall i \in\{1,2,\ldots,n\} \right\}.
$$
The complement of $\sI$ is the so-called \textit{active set}, where the corresponding constraints are active. That is, the active set contains  indices $i$ such that $x_{i}^{*} = 0$. The nonzero entries of $\bx^{*}$ can be determined by solving the following reduced linear system:
$$
\begin{aligned}
[\nabla_{\bx} f(\bx)]_{\sI} = \bzero 
\gapthree\Longleftrightarrow\gapthree [\bM^\top(\bM \bx - \bb)]_{\sI} = \bzero 
\gapthree\Longleftrightarrow\gapthree \bM[:, \sI]^\top \bM[:, \sI] \bx[\sI] = \bM[:, \sI]^\top \bb.
\end{aligned}
$$
This is precisely the normal equation for the unconstrained least squares problem w.r.t. $\bx[\sI]$, that is,
$$
\min_{\bx[\sI]} \frac{1}{2}\normtwo{\bb -\bM[:, \sI] \bx[\sI]}^2.
$$
This observation forms the basis of the \textit{active-set method}, which iteratively updates the active set through pivoting (that is, entering and removing variables from the active set) to ensure the objective function  decreases \citep{lawson1995solving}; see Algorithm~\ref{alg:nmf_anls}. 

\index{Normal equation}
\paragraph{Alternating nonnegative least squares (ANLS).} 
Once we have the active-set method for NNLS problems,  NMF can be achieved by replacing OLS in ALS algorithms with NNLS, known as \textit{alternating nonnegative least squares (ANLS)} \citep{kim2011fast}.
Given a fixed $\bW$, the NMF objective can be solved for each column of $\bZ$ separately:
$$
\frac{1}{2}\normf{\bA-\bW\bZ}^2 =
\frac{1}{2}\sum_{n=1}^{N}\normtwo{\ba_n - \bW\bz_n}^2,
$$
where each subproblem $\mathopmin{\bz_n\geq \bzero}\normtwo{\ba_n - \bW\bz_n}^2$ can be solved using NNLS.
Since the NMF problem is symmetric: $\bA=\bW\bZ$ if and only if $\bA^\top=\bZ^\top\bW^\top$ such that $D(\bA, \bW\bZ)=D(\bA^\top, \bZ^\top\bW^\top)$. The analysis of optimizing $\bW$ given $\bZ$  follows directly from the previous methodology.
We should also note that  since the initial guess of $\bW$ and $\bZ$ typically offers a poor approximation of  $\bA$, solving the NNLS subproblems exactly in the early stages of the alternating algorithms is often unnecessary. 
Instead, it can be more efficient to use ANLS as a refinement step within a less computationally expensive NMF algorithm, such as  APGD or MU (discussed in later sections)

\begin{algorithm}[h] 
\caption{Nonnegative Least Squares (NNLS) via Active-Set Method}
\label{alg:nmf_anls}
\begin{algorithmic}[1] 
\Require A real-valued matrix $\bM\in\real^{m \times n}$, a real-valued vector $\bb\in\real^m$;
\State Initialize index sets $\sI = \emptyset$ and $\sJ = \{1, \ldots, n\}$;
\State Initialize unknown $\bx\in\real^n$ to an all-zero vector and let $\bw \leftarrow \bM^\top(\bb - \bM\bx)$;
\State Let $\bw[\sJ]$ denote the sub-vector with indices from $\sJ$;
\State Choose a stopping criterion on the approximation error $\delta$;
\State Choose the maximal number of iterations $C$;
\State $iter=0$; \Comment{Count for the number of iterations}
\While{$\sJ \neq \emptyset$ and $\max(\bw[\sJ]) > \delta$ and $iter<C$}
\State $iter=iter+1$; 
\State Let $j$ in $\sJ$ be the index of $\max(\bw[\sJ])$ in $\bw$: $j=\mathop{\argmax}_{j\in \sJ} w_j$;
\State Add $j$ to $\sI$  and remove $j$ from $\sJ$ such that $\sI\cup \sJ = \{1,2,\ldots,n\}$;
\State Let $\bM[:,\sI]$ be $\bM$ restricted to the variables/columns included in $\sI$;
\State \parbox[t]{\dimexpr\linewidth-\algorithmicindent}{Let $\bs$ be vector of same length as $\bx$;
Let $\bs[\sI]$ denote the sub-vector with indices from $\sI$, and let $\bs[\sJ]$ denote the sub-vector with indices from $\sJ$;}
\State Set $\bs[\sI] \leftarrow ((\bM[:,\sI])^\top \bM[:,\sI])^{-1} (\bM[:,\sI])^\top \bb$ and $\bs[\sJ]$ to zero;
\While{$\min(\bs[\sI]) \leq 0$}
\State Let $\alpha \leftarrow \min \frac{x_i}{x_i - s_i}$ for $i$ in $\sI$ where $s_i \leq 0$;
\State Set $\bx\leftarrow \bx + \alpha(\bs - \bx)$;
\State Move to $\sJ$ all indices $j$ in $\sI$ such that $x_j \leq 0$;
\State Set $\bs[\sI] \leftarrow ((\bM[:,\sI])^\top \bM[:,\sI])^{-1} (\bM[:,\sI])^\top \bb$;
\EndWhile
\State Set $\bs[\sJ]$ to zero;
\State Set $\bx\leftarrow \bs$;
\State Set $\bw\leftarrow \bM^\top(\bb - \bM\bx)$;
\EndWhile
\State Output $\bx$.
\end{algorithmic} 
\end{algorithm}

\index{Hierarchical ANLS}
\section{NMF via Hierarchical Alternating Nonnegative Least Squares}
Let $\ba, \bb\in\real_+^n$ be two nonnegative vectors. The \textit{univariate NNLS} problem can be formulated as 
$$
\mathopmin{x\geq 0} \normtwo{\ba-x\bb}^2.
$$
This problem admits a closed-form solution: $x=\max\big\{0, \frac{\bb^\top\ba}{\normtwo{\bb}^2}\big\}$ if $\normtwo{\bb}\neq 0$.
With this univariate NNLS solution in mind, considering the $k$-th row of $\bZ$ for $k\in\{1,2,\ldots,K\}$, the subproblem in NMF is
\begin{equation}\label{equation:llipschi_hianls}
\mathopmin{\bZ[k,:]\geq \bzero} \bigg\Vert\underbrace{\big(\bA-\sum_{p\neq k}^{K} \bW[:,p]\bZ[p,:]\big)}_{=\bA_k} - \bW[:,k]\bZ[k,:]\bigg\Vert_F^2, 
\gap \forall k,~\footnote{This subproblem is convex and is $L$-Lipschitz gradient continuous/$L$-strongly smooth (definition in Problem~\ref{problem:mono_gd}); see Problem~\ref{prob:llipschi_hianls}.}
\end{equation}
which indicates the entries in a row of $\bZ$ do not interact (similarly, entries in a column of $\bW$ do not interact). Therefore, the optimization of each entry in a row of $\bZ$ can be decoupled. 
Let $\bA_k=\big(\bA-\sum_{p\neq k}^{K} \bW[:,p]\bZ[p,:]\big)$. Then, the NMF problem becomes a set of rank-one updates on $\bA_k$, for $k\in\{1,2,\ldots,K\}$.
The solution is 
$$
\bZ^*[k,:]=\mathop{\argmin}_{\bZ[k,:]\geq \bzero} \normf{\bA_k -\bW[:,k]\bZ[k,:]}^2
=
\max\left(
\bzero, 
\frac{\bW[:,k]^\top\bA_k}{\normtwo{\bW[:,k]}^2}
\right),
\gap \forall k,
$$
where the max operator is applied componentwise.
This derivation leads to the \textit{hierarchical ANLS (Hi-ANLS)} solution for NMF problems, which iteratively solves a univariate NNLS problem.
The procedure is described in Algorithm~\ref{alg:hie_anls}, where we note that $\bZ[k,:]^\top = \bZ^\top[:,k]$.
In the algorithm, we update the $k$-th row of $\bZ$ and $k$-th column of $\bW$ in an interleaved manner. \citet{gillis2012accelerated} show that updating $\bZ$ several times before updating $\bW$  can significantly improve the performance  since this reuses the results of $\bW^\top\bA$ and $\bW^\top\bW$.

\begin{algorithm}[h] 
\caption{NMF via Hierarchical Alternating Nonnegative Least Squares (Hi-ANLS)}
\label{alg:hie_anls}
\begin{algorithmic}[1] 
\Require Matrix $\bA\in \real_+^{M\times N}$;
\State Initialize $\bW\in \real_{++}^{M\times K}$, $\bZ\in \real_{++}^{K\times N}$ randomly with positive entries;
\State Choose a stop criterion on the approximation error $\delta$;
\State Choose maximal number of iterations $C$;
\State $iter=0$; \Comment{Count for the number of iterations}
\While{$\normf{\bA- (\bW\bZ)}^2>\delta $ and $iter<C$}
\State $iter=iter+1$;  
\For{$k=1$ to $K$}
\State $
\bZ[k,:]\leftarrow 
\max\left(
\bzero, 
\frac{\bW[:,k]^\top\bA_k}{\normtwo{\bW[:,k]}^2}
\right)$; \Comment{$\bA_k=\big(\bA-\sum_{p\neq k}^{K} \bW[:,p]\bZ[p,:]\big)$}

\State $
\bW[:, k]\leftarrow 
\max\left(
\bzero, 
\frac{\bA_k\bZ[k,:]^\top}{\normtwo{\bZ[k,:]}^2}
\right)$;
\EndFor
\EndWhile
\State Output $\bW,\bZ$.
\end{algorithmic} 
\end{algorithm}

\section{NMF via Alternating Direction Methods of Multipliers (ADMM)}\label{section:nmf_admm_all}
We briefly introduce the \textit{alternating direction methods of multipliers (ADMM)} method and then  discuss its applications in matrix factorization and NMF.
\paragraph{ADMM.}
ADMM is designed to solve convex optimization problems of the form:
\begin{equation}\label{equation:admm_prob}
\mathopmin{\bx, \bz} f(\bx)+g(\bz), \gap \text{s.t.}\gap \bD\bx+\bE\bz=\bff.
\end{equation}
Given a penalty parameter $\rho>0$, the \textit{augmented Lagrangian} of \eqref{equation:admm_prob} is 
\begin{equation}
L_\rho (\bx, \bz, \bl) = f(\bx)+g(\bz) +\langle \bl, \bD\bx+\bE\bz-\bff\rangle + \frac{\rho}{2}\normtwo{\bD\bx+\bE\bz-\bff}^2.
\end{equation}
When $\rho=0$, the augmented Lagrangian function reduces to the Lagrangian function; when $\rho>0$, the augmented Lagrangian function  acts as  a penalized version of the Lagrangian function.
The \textit{augmented Lagrangian method} solves the problem iteratively.  At the $(t+1)$-th iteration,  it performs the following updates:
$$
\text{augmented Lagrangian:}
\gap 
\left\{
\begin{aligned}
(\bx^{(t+1)}, \bz^{(t+1)}) &\in \mathop{\argmin}_{\bx, \bz} L_\rho (\bx, \bz, \bl);\\
\bl^{(t+1)}&=\bl^{(t)} + \rho(\bD\bx^{(t+1)}+\bE\bz^{(t+1)} -\bff ),
\end{aligned}
\right.
$$
where the update on $\bl^{(t+1)}$ is derived from the \textit{conjugate subgradient theorem} (see, for example, \citet{bach2011convex}), and 
the symbol `$\in$' indicates that the minimum points may not be uniquely determined.
One source of difficulty is the coupling term between the $\bx$ and  $\bz$ variables, which is of the form $\rho(\bx^\top\bD^\top\bE\bz)$.
ADMM tackles this difficulty by replacing the exact minimization of $(\bx,\bz)$ with one iteration of the alternating minimization method.
To be more specific, at the  $(t+1)$-iteration,  ADMM performs the following updates:
\begin{equation}
\text{ADMM:}\gap
\left\{
\begin{aligned}
	\bx^{(t+1)}&\in\mathop{\argmin}_{\bx} \left\{ f(\bx)+ \frac{\rho}{2}\normtwo{\bD\bx+\bE\bz^{(t)} -\bff +\frac{1}{\rho}\bl^{(t)}}^2 \right\};\\
	\bz^{(t+1)}&\in\mathop{\argmin}_{\bz} \left\{ g(\bz)+ \frac{\rho}{2}\normtwo{\bD\bx^{(t+1)}+\bE\bz -\bff +\frac{1}{\rho}\bl^{(t)}}^2 \right\};\\
	\bl^{(t+1)}&=\bl^{(t)} + \rho(\bD\bx^{(t+1)}+\bE\bz^{(t+1)} -\bff ).
\end{aligned}
\right.
\end{equation}
By defining $\widetildebl=\frac{1}{\rho}\bl$, this  can be equivalently stated as (this form will be used in the sequel):
\begin{equation}\label{equation:admm_gen_up}
\text{ADMM:}\gap
\left\{
\begin{aligned}
\bx^{(t+1)}&\in\mathop{\argmin}_{\bx} \left\{ f(\bx)+ \frac{\rho}{2}\normtwo{\bD\bx+\bE\bz^{(t)} -\bff +\widetildebl^{(t)}}^2 \right\};\\
\bz^{(t+1)}&\in\mathop{\argmin}_{\bz} \left\{ g(\bz)+ \frac{\rho}{2}\normtwo{\bD\bx^{(t+1)}+\bE\bz -\bff +\widetildebl^{(t)}}^2 \right\};\\
\widetildebl^{(t+1)}&=\widetildebl^{(t)} + (\bD\bx^{(t+1)}+\bE\bz^{(t+1)} -\bff ).
\end{aligned}
\right.
\end{equation}
That is, ADMM alternately updates $\bx, \bz$, and $\bl$ (or the scaled dual variable $\widetildebl$).

\paragraph{ADMM applied to matrix factorization.}
We return to the problem discussed in ALS (Equation~\eqref{equation:als-per-example-loss2}, i.e., matrix factorization with Frobenius norm; not necessarily a NMF problem), along with a regularization function $r(\bZ)$:
$$
\mathopmin{\bZ} \frac{1}{2}\normf{\bA-\bW\bZ}^2+r(\bZ).
$$
The problem can be equivalently stated with an auxiliary variable $\widetildebZ\in\real^{K\times N}$:
\begin{equation}\label{equation:mf_admm_prob1}
\mathopmin{\bZ} \frac{1}{2}\normf{\bA-\bW\bZ}^2+r(\widetildebZ), 
\gap 
\text{s.t.}
\gap 
\bZ=\widetildebZ.
\end{equation}
Following \eqref{equation:admm_gen_up}, let a. \{$\bx\leftarrow \bZ$, $\bz\leftarrow \widetildebZ$, $\widetildebl\leftarrow \bL$, $\bD=-\bI$,  $\bE=\bI$\} or b. \{$\bx\leftarrow \bZ$, $\bz\leftarrow \widetildebZ$, $\widetildebl\leftarrow \bL$, $\bD=\bI$,  $\bE=-\bI$\}, 
the resulting ADMM updates  for \eqref{equation:mf_admm_prob1} are: 
\begin{equation}\label{equation:admm_gen_als}
\left\{
\begin{aligned}
\bZ 
&\stackrel{(a)}{\leftarrow} (\bW^\top\bW+\rho \bI)^{-1} \left[ \bW^\top\bA +\rho(\widetildebZ+\bL) \right]
&\stackrel{(b)}{\leftarrow}& (\bW^\top\bW+\rho \bI)^{-1} \left[ \bW^\top\bA +\rho(\widetildebZ-\bL) \right];\\
\widetildebZ
&\stackrel{(a)}{\leftarrow}\mathop{\argmin}_{\widetildebZ} r(\widetildebZ) + \frac{\rho}{2}\normf{-\bZ+\widetildebZ + \bL}^2
&\stackrel{(b)}{\leftarrow}&\mathop{\argmin}_{\widetildebZ} r(\widetildebZ) + \frac{\rho}{2}\normf{\bZ-\widetildebZ + \bL}^2\\
\bL&\stackrel{(a)}{\leftarrow}\bL -\bZ+\widetildebZ &\stackrel{(b)}{\leftarrow}& \bL +\bZ-\widetildebZ.
\end{aligned}
\right.
\end{equation}
In practice, the Cholesky decomposition of $(\bW^\top\bW+\rho \bI)$ can be calculated such that the update can be obtained by forward and backward substitutions.
The update for $\bW$ can be obtained similarly due to symmetry.
In the following discussion, we adopt setting (a) from \eqref{equation:admm_gen_als}.

%

\paragraph{ADMM applied to $\ell_1$ regularization.}
We may also consider the $\ell_1$ regularization (see Section~\ref{section:regularization-extention-general}): $r(\widetildebZ)=\lambda \Vert\widetildebZ\Vert_1$. The update for each element $(k,n)$ of $\widetildebZ$ is $\widetildez_{kn}\leftarrow \max(0, 1-\frac{\lambda}{\rho} \abs{h_{kn}}^{-1}) h_{kn}$ for all $k\in\{1,2,\ldots,K\}$ and $n\in\{1,2,\ldots,N\}$, where $h_{kn} = z_{kn}-l_{kn}$ (i.e., the elements of $\bH=\bZ-\bL$). 

\paragraph{ADMM applied to smoothness/denoising regularization.}
A smoothness regularization on $\bZ$ can be defined as $r(\widetildebZ)=\frac{\lambda}{2} \Vert\bT\widetildebZ^\top\Vert_F^2$, where $\bT$ is an $N\times N $ tridiagonal matrix with 2 on the main diagonal and $-1$ on the superdiagonal and subdiagonal. This regularization ensures the proximal components in each row of $\widetildebZ$ is smooth (see Problem~\ref{prob:denoise_rls}). The update  for $\widetildebZ$ becomes $\widetildebZ\leftarrow \rho\bZ(\lambda \bT^\top\bT +\rho\bI)^{-1}$ \citep{huang2016flexible}.

\paragraph{ADMM applied to NMF.}
The NMF with ADMM is achieved simply by replacing $r(\bZ)$ with an indicator function.
The update for $\widetildebZ$ becomes $\max\left(\bzero, \bZ-\bL\right)$, where the max operator is applied  componentwise.
However, unlike the methods discussed earlier (such as NNLS) or the MU approach introduced in the next section, ADMM updates are generally not monotonically nonincreasing in terms of the objective function. This is an important consideration when monitoring convergence.


\index{Alternating update}
\index{Kullback-Leibler divergence}
\index{Multiplicative update}
\section{NMF via Multiplicative Update (MU)}\label{section:nmf_frob_mu}
We consider an alternative alternating update approach for NMF. 
The hidden features in $\bW$ and $\bZ$ are modeled as nonnegative vectors in a low-dimensional space. These latent vectors are randomly initialized and iteratively updated via an alternating \textit{multiplicative update} rule to minimize the Frobenius norm distance between the observed and modeled matrices. 
Following Section~\ref{section:als-netflix}, we consider the low-rank with $K$ components; given $\bW\in \real_+^{M\times K}$, we aim to update $\bZ\in \real_+^{K\times N}$. The gradient of the loss function $L(\bW, \bZ)=\frac{1}{2}\normf{\bA-\bW\bZ}^2$ with respect to $\bZ$ is given by Equation~\eqref{equation:givenw-update-z-allgd}:
$
\begin{aligned}
\nabla_{\bZ} L(\bW, \bZ) =\bW^\top(\bW\bZ-\bA) \in \real^{K\times N}.
\end{aligned}
$
Applying the gradient descent idea discussed in Section~\ref{section:als-gradie-descent}, a straightforward update for $\bZ$ is:
$$
(\text{GD on $\bZ$})\gap \bZ \leftarrow \bZ - \eta \big(\nabla_{\bZ} L(\bW, \bZ)\big)=\bZ - \eta \nabla_{\bZ} L(\bW, \bZ),
$$
where $\eta$ represents a small positive step size (learning rate). 
\paragraph{Multiplicative update (MU).}
If we allow a different step size for each entry of $\bZ$, the update can be written as:
$$
(\text{GD$^\prime$ on $\bZ$})\gap 
\begin{aligned}
	z_{kn} &\leftarrow z_{kn} - {\eta_{kn}} \big(\nabla_{\bZ} L(\bW, \bZ)\big)_{kn}
	=z_{kn} - \eta_{kn}(\bW^\top\bW\bZ-\bW^\top\bA)_{kn}, \,\, \forall k,n,
\end{aligned}
$$
where $z_{kn}$ denotes the $(k,n)$-th entry of $\bZ$. 
To proceed, we further rescale the step size:
$$
\eta_{kn} = \frac{z_{kn}}{(\bW^\top\bW\bZ)_{kn}}.
$$
Then we obtain the update rule:
\begin{equation}\label{equation:multi-update-z}
(\text{MU on $\bZ$})\gap 
\bZ \leftarrow \bZ\hadaprod \frac{[\bW^\top\bA]}{[\bW^\top\bW\bZ]}
\stackrel{*}{=}
\bZ - \frac{[\bZ]}{[\bW^\top\bW\bZ]}\hadaprod \nabla_{\bZ} L(\bW, \bZ) 
,
\end{equation}
where $\frac{[\cdot]}{[\cdot]}$ represents the componentwise division between two matrices. This is known as the \textit{multiplicative update (MU)}, and is first developed in \citet{lee2001algorithms} for NMF problems. 
Analogously, the multiplicative update for $\bW$ can be obtained by 
\begin{equation}\label{equation:multi-update-w}
(\text{MU on $\bW$})\gap
\bW \leftarrow \bW \hadaprod \frac{[\bA\bZ^\top]}{[\bW\bZ\bZ^\top]} 
\stackrel{*}{=} 
\bW - \frac{[\bW]}{[\bW\bZ\bZ^\top]}\hadaprod \nabla_{\bW} L(\bW, \bZ) .
\end{equation}
The factors $\frac{(\bW^\top\bA)_{kn}}{(\bW^\top\bW\bZ)_{kn}}$ and $\frac{(\bA\bZ^\top)_{mk}}{(\bW\bZ\bZ^\top)_{mk}}$ for all $m,k,n$ in \eqref{equation:multi-update-z} and \eqref{equation:multi-update-w} are called \textit{multiplicative factors}.
When $\bA=\bW\bZ$, these multiplicative factors reduce to one, indicating that the corresponding gradients vanish.

\paragraph{MU vs gradient descent.}
The above derivation shows that multiplicative update algorithms are fundamentally similar to gradient descent algorithms, differing primarily in step size selection. With an appropriate choice of step size, the multiplicative algorithm can transform the subtraction update rule of the standard gradient descent method into a multiplicative update rule.

In the gradient descent algorithm, a fixed or adaptive step length is typically used, and this step length is independent of the specific variable being updated. In other words, the step size may vary over time, but at any given update step, all entries of the  matrix variable are updated using the same step size.
In contrast, the multiplicative algorithm uses different step sizes ($\eta_{kn}$ above) for different entries of the factor matrix. This means that the step length is adaptive to each matrix entry. This adaptability is a key reason why the multiplicative algorithm can outperform the gradient descent algorithm in NMF algorithms.

\paragraph{KKT conditions for NMF with Frobenius norm.}
The KKT conditions  indicate that (see derivation in \eqref{equation:kkv_nnn_raw}):
\begin{equation}\label{equation:nmf_fro_kkt1}
\begin{aligned}
\bZ\geq \bzero,\gap&\nabla_{\bZ} L(\bW,\bZ)&\geq& \bzero, \gap \langle \bZ, \nabla_{\bZ} L(\bW,\bZ)\rangle &=&\bzero_{K\times N}; \\
\bW\geq \bzero,\gap&\nabla_{\bW} L(\bW,\bZ)&\geq& \bzero, \gap \langle \bW, \nabla_{\bW} L(\bW,\bZ)\rangle &=&\bzero_{M\times K}.
\end{aligned}
\end{equation}
This also implies
\begin{equation}
\begin{aligned}
\min\{\bZ, \nabla_{\bZ} L(\bW,\bZ)\} = \bzero_{K\times N}
\gap \text{and}\gap 
\min\{\bW, \nabla_{\bW} L(\bW,\bZ)\} = \bzero_{M\times K},
\end{aligned}
\end{equation}
where the min operator $\min\{\cdot, \cdot \}$ is applied componentwise. Any pair $(\bW,\bZ)$ satisfying the KKT conditions is a stationary point of the NMF problem in \eqref{equation:frob_nmf}.

\paragraph{Problems in MU.}
The equality ($*$) in \eqref{equation:multi-update-z} indicates a rescaled gradient descent update in the MU rules, which also implies 
$$
\frac{[\bW^\top\bA]_{kn}}{[\bW^\top\bW\bZ]_{kn}}\geq 1 \gap\Longleftrightarrow\gap (\nabla_{\bZ} L(\bW, \bZ))_{kn}\leq 0, \gap  \forall k, n.
$$
Therefore, the MU algorithm induces three-fold rules: (i) Increase if its partial derivative is negative; (ii) Decrease it if its partial derivative is positive; (iii) Keep it unchanged if its partial derivative is zero.
However, if an element of $\bZ$ is equal to zero, the MU rule cannot modify it. Therefore, it is possible for an entry of $\bZ$  to be zero while its partial derivative is negative, which would violate the KKT conditions in \eqref{equation:nmf_fro_kkt1}.
As a result, the iterates from the MU rule are not guaranteed to converge to a stationary point. There are several ways to address this issue, such as: (i) Using a small positive lower bound for the entries of $\bZ$ and  $\bW$\citep{gillis2012accelerated}.
(ii) Using the MU rule while reinitializing zero entries of $\bZ$ and  $\bW$ to a small positive constant when their partial derivatives become negative \citep{chi2012tensors}.

We now prove that the MU rule monotonically decreases the loss function.
\begin{theorem}[Monotonically nonincreasing of multiplicative update]\label{theorem:conv_mu_fro}
The loss $L(\bW,\bZ)=\frac{1}{2}\normf{\bW\bZ-\bA}^2$ remains nonincreasing under the following multiplicative update rules:~\footnote{More general results for $\beta$-divergences are discussed in Theorem~\ref{theorem:conv_mu_beta}.}
$$
\begin{aligned}
\bZ &\leftarrow \bZ\hadaprod \frac{[\bW^\top\bA]}{[\bW^\top\bW\bZ]}
\qquad\text{and}\qquad 
\bW &\leftarrow \bW \hadaprod \frac{[\bA\bZ^\top]}{[\bW\bZ\bZ^\top]},
\end{aligned}
$$
where $\bA\in\real_+^{M\times N}, \bW\in\real_+^{M\times K}$, and $\bZ\in\real_+^{K\times N}$. 
The operator $\frac{[\cdot]}{[\cdot]}$ represents the componentwise division between two matrices, and $\hadaprod$ denotes the Hadamard product between two matrices.

The MU update requires that $\bZ$ and $\bW$ should be initialized with positive (nonzero) entries; otherwise, zeros will persist due to the multiplicative nature of the update.
\end{theorem}

The MU method sparked significant interest in NMF and has since become a cornerstone in the field, due to several advantages: (i) The update rules are extremely easy to implement; (ii) In practice, the convergence is relatively faster compared to many other methods; (iii) Nonnegativity is automatically preserved during updates.
To prove the monotonicity of the MU rules, we use the auxiliary function framework.
\begin{definition}[Auxiliary function (majorizer)]\label{definition:aux_func}
$G(\bx, \widetildebx)$ is called an \textit{auxiliary function} for $F(\bx)$ (or a majorizer of $F$ at $\widetildebx$) if the conditions~\footnote{$\bx$ can be scalars, vectors, or matrices.}
$$
G(\bx, \widetildebx) \geq F(\bx)
\qquad 
\text{and}
\qquad 
G(\bx, \bx) = F(\bx)
$$
are satisfied.
In other words, the auxiliary function $G(\bx, \widetildebx)$ is an upper bound of $F(\bx)$, and the bound is tight when $\widetildebx=\bx$.
\end{definition}

\begin{lemma}[Nonincreasing in auxiliary functions]\label{lemma:noninmuaux}
If $ G $ is an auxiliary function for $F$, then $F$ is nonincreasing under the update
\begin{equation}\label{equation:aux_update}
\bx^{(t+1)} = \mathop{\argmin}_{\bx} \, G(\bx, \bx^{(t)}).
\end{equation}
\end{lemma}
\begin{proof}[of Lemma~\ref{lemma:noninmuaux}]
The definition of the auxiliary function indicates that $ F(\bx^{(t+1)}) \leq G(\bx^{(t+1)}, \bx^{(t)}) \leq G(\bx^{(t)}, \bx^{(t)}) = F(\bx^{(t)})$.
\end{proof}

Note that $ F(\bx^{(t+1)}) = F(\bx^{(t)})$ only if $ \bx^{(t)}$ is a local minimum of $ G(\bx, \bx^{(t)})$ w.r.t. $\bx$. If the partial derivatives of $ F$ exist and are continuous in a small neighborhood of $ \bx^{(t)}$, this also implies that the gradient $ \nabla F(\bx^{(t)}) = \bzero$. Thus, by iterating the update in \eqref{equation:aux_update}, we obtain a sequence of estimates that converge to a local minimum $ \bx_{\min} = \argmin_{\bx} F(\bx)$ of the objective function:
\begin{equation}
 F(\bx^{(0)}) \geq 	 F(\bx^{(1)}) \geq  F(\bx^{(2)})\geq \ldots  \geq  F(\bx^{(t)}) \geq  F(\bx^{(t+1)})\geq \ldots \geq F(\bx_{\min}).
\end{equation}
Definition~\ref{definition:aux_func} finds a majorizer $G$ of $F$, and Lemma~\ref{lemma:noninmuaux} shows the minimization property in $G$, hence the algorithm is often referred to as the \textit{majorization-minimization (MM) framework}.
The update benefits when the global minimizer of $G$ has a closed-form solution or can be computed efficiently.

Therefore, if we can construct an appropriate auxiliary function $ G(\bx, \bx^{(t)})$ for both variables in $\normf{\bA-\bW\bZ}$, the update rules in  Theorem~\ref{theorem:conv_mu_fro} follow from  \eqref{equation:aux_update}.
To apply the auxiliary function method to the NMF problem, we focus on a single column of $\bA$ or $\bZ$: $\ba=\ba_n$ and $\bz=\bz_n$  in the following lemma, where $n\in\{1,2,\ldots,N\}$.

\begin{lemma}[Auxiliary function for NMF]\label{lemma:aux_nmf}
Let $\bW\in\real^{K\times N}, \ba\in\real^M$, and $\bz\in\real^{K}$.
Let further $\bD\in\real^{K\times K}$ be a diagonal matrix with the $(k,k)$-th entry being $d_{kk}=\frac{(\bW^\top\bW \bz)_k}{z_k}=\frac{\bw_k^\top\bW\bz}{z_k} = \frac{\sum_{j=1}^{K} (\bW^\top\bW)_{kj}z_j}{z_k}, \, \,\forall k\in\{1,2,\ldots,K\}$, where $\bw_k$ is the $k$-th column of $\bW$ and $z_k$ is the $k$-th component of $\bz$.
Then, the following function is an auxiliary function for  $F(\bz)=\frac{1}{2}\normtwo{\ba-\bW\bz}^2$:
\[ 
G(\bz, \bz^{(t)}) = F(\bz^{(t)}) + (\bz-\bz^{(t)})^\top \nabla F(\bz^{(t)}) +\frac{1}{2}(\bz-\bz^{(t)})^\top \bD (\bz-\bz^{(t)}).
\]
\end{lemma}
\begin{proof}[of Lemma~\ref{lemma:aux_nmf}]
Since the third-order partial derivatives of $F(\bz)$ vanish (see Problem~\ref{prob:third_order_nmf}), $F(\bz)$ can be factored as 
$$
F(\bz) = F(\bz^{(t)})+(\bz-\bz^{(t)})^\top \nabla F(\bz^{(t)}) +\frac{1}{2}(\bz-\bz^{(t)})^\top \bW^\top\bW (\bz-\bz^{(t)}).
$$
Apparently, $G(\bz, \bz)=F(\bz) $. To complete the proof, we need to show that $G(\bz, \bz^{(t)}) \geq F(\bz)$; that is, $\bD-\bW^\top\bW$ is positive semidefinite.
To prove this, consider the matrix $\bM\in\real^{K\times K}$ whose entries are $m_{ij}=z_i (\bD-\bW^\top\bW)_{ij}z_j$ for all $i,j\in\{1,2,\ldots,K\}$, which is a rescaling of the components of $\bD-\bW^\top\bW$. Then $\bD-\bW^\top\bW$ is positive semidefinite if and only if $\bM$ is: 
$$
\begin{aligned}
\bx^\top& \bM\bx
= \sum_{i,j=1}^{K,K} x_i m_{ij}x_j
\stackrel{*}{=}\sum_{i,j=1}^{K,K} \left\{(\bW^\top\bW)_{ij} z_i z_j  x_i^2 - (\bW^\top\bW)_{ij}z_i z_j  x_i x_j\right\}\\
&\stackrel{\dag}{=}\sum_{i,j=1}^{K,K} (\bW^\top\bW)_{ij}z_i z_j \left( \frac{1}{2}x_i^2 + \frac{1}{2}x_j^2 - x_ix_j\right)
=\sum_{i,j=1}^{K,K} (\bW^\top\bW)_{ij}z_i z_j\frac{1}{2} \left( x_i-x_j\right)^2 \geq 0,
\end{aligned}
$$
where the equality $(\dag)$ follows from the symmetry of $\bM$, and the equality ($*$) follows from the diagonality of $\bD$:
$$
\sum_{i,j=1}^{K,K} x_iz_i d_{ij}z_jx_j
=
\sum_{i=1}^{K} x_iz_i d_{ii}z_ix_i
=
\sum_{i=1}^{K} x_i^2 z_i^2 \frac{\sum_{j=1}^{K}(\bW^\top\bW)_{ij}z_j}{z_i}
=
\sum_{i,j=1}^{K,K} (\bW^\top\bW)_{ij} z_i z_j  x_i^2.
$$
This completes the proof.
\end{proof}

The proof of the monotonicity of MU updates in  Theorem~\ref{theorem:conv_mu_fro} follows directly from the above lemmas.
Clearly, the approximations $\bW$ and $\bZ$ remain nonnegative during the updates.
It is generally better to update $\bW$ and $\bZ$ ``simultaneously” rather than ``sequentially," i.e., updating each matrix completely before the other. In this case, after updating a row of $\bZ$, we update the corresponding column of $\bW$.  
In the implementation, it is advisable to introduce a small positive quantity, say the square root of the machine precision, to the denominators in the approximations of $\bW$ and $\bZ$
at each iteration. 
And a trivial value like $\epsilon=10^{-9}$  suffices. The full procedure is shown in Algorithm~\ref{alg:nmf-multiplicative}.
In practice, the algorithm can also be accelerated by updating $\bW$ several times before updating $\bZ$, during which process we can reuse the result of $\bA\bZ^\top$ and $\bZ\bZ^\top$, and vice versa.

\index{Machine precision}
\begin{algorithm}[h] 
\caption{NMF via Multiplicative Updates}
\label{alg:nmf-multiplicative}
\begin{algorithmic}[1] 
\Require Matrix $\bA\in \real_+^{M\times N}$;
\State Initialize $\bW\in \real_{++}^{M\times K}$, $\bZ\in \real_{++}^{K\times N}$ randomly with positive entries;
\State Choose a stop criterion on the approximation error $\delta$;
\State Choose maximal number of iterations $C$;
\State $iter=0$; \Comment{Count for the number of iterations}
\While{$\normf{\bA- (\bW\bZ)}^2>\delta $ and $iter<C$}
\State $iter=iter+1$; 
\State $\bZ \leftarrow \bZ\hadaprod \frac{[\bW^\top\bA]}{[\bW^\top\bW\bZ]+\epsilon}$;
\State $\bW \leftarrow \bW \hadaprod \frac{[\bA\bZ^\top]}{[\bW\bZ\bZ^\top]+\epsilon}$;
\EndWhile
\State Output $\bW,\bZ$.
\end{algorithmic} 
\end{algorithm}


\index{Regularization}
\subsection{Regularization}
As mentioned in \eqref{equation:kkv_nnn_raw},  the NNLS or NMF problem implicitly imposes a \textbf{sparsity constraint}. 
On the other hand, similar to the ALS method with regularization discussed in Section~\ref{section:regularization-extention-general} (recall that the regularization can help extend the applicability of ALS to general matrices),
a regularization term can be incorporated into the NMF framework to enhance its performance:
$$
L(\bW,\bZ)  =\frac{1}{2}\normf{\bW\bZ-\bA}^2 +\frac{1}{2}\lambda_w \normf{\bW}^2 + \frac{1}{2}\lambda_z \normf{\bZ}^2, \qquad \lambda_w>0, \lambda_z>0,
$$
where the employed matrix norm is still the Frobenius norm. The gradient with respect to $\bZ$ given $\bW$ is the same as that in Equation~\eqref{equation:als-regulari-gradien}:
$$
\begin{aligned}
	\frac{\partial L(\bZ|\bW)}{\partial \bZ} =\bW^\top(\bW\bZ-\bA) + \textcolor{mylightbluetext}{\lambda_z\bZ}  \in \real^{K\times N}.
\end{aligned}
$$
The gradient descent update can be obtained by 
$$
(\text{GD on }\bZ) \gap \bZ \leftarrow \bZ - \eta \left(\frac{\partial L(\bZ|\bW)}{\partial \bZ}\right)=\bZ - \eta \left(  \bW^\top\bW\bZ-\bW^\top\bA+\textcolor{mylightbluetext}{\lambda_z\bZ}\right),
$$
Analogously, if we assume a different step size for each entry of $\bZ$, the update can be obtained by
$$
(\text{GD$^\prime$ on $\bZ$})\gap 
\begin{aligned}
z_{kn} 
&=z_{kn} - \eta_{kn}(\bW^\top\bW\bZ-\bW^\top\bA+\textcolor{mylightbluetext}{\lambda_z\bZ})_{kn}, 
 \,\, \forall k,n.
\end{aligned}
$$
We again rescale the step size:
$
\eta_{kn} = \frac{z_{kn}}{(\bW^\top\bW\bZ)_{kn}}.
$
Then we obtain the MU rules for $\bZ$ and $\bW$ (due to symmetry):
$$
\begin{aligned}
\bZ \leftarrow  \bZ\hadaprod \frac{[\bW^\top\bA-\lambda_z\bZ]}{[\bW^\top\bW\bZ]}
\gap\text{and}\gap 
\bW \leftarrow \bW \hadaprod \frac{[\bA\bZ^\top-\lambda_w\bW]}{[\bW\bZ\bZ^\top]}.
\end{aligned}
$$

\paragraph{Modified MU.}
Since the update for the above regularized NMF can result in negative values, a modified MU can be applied such that
$$
\textbf{(MMU1):}\quad  \bZ \leftarrow \left[\bZ\hadaprod \frac{[\bW^\top\bA-\lambda_z\bZ]}{[\bW^\top\bW\bZ]}\right]_+
\gap\text{and}\gap 
\bW \leftarrow \left[\bW \hadaprod \frac{[\bA\bZ^\top-\lambda_w\bW]}{[\bW\bZ\bZ^\top]}\right]_+,
$$
where $[x]_+ = \max\{x, \epsilon\}$. The parameter $\epsilon$ is usually a very small positive number that prevents the emergence of negative update.
That is, we add a small lower bound for entries of $\bW$ and $\bZ$. An alternative  applies the nonnegativity constraint only to the numerator:
$$
\begin{aligned}
\textbf{(MMU2):}\quad 
\bZ \leftarrow \bZ\hadaprod \frac{[\bW^\top\bA-\lambda_z\bZ]_+}{[\bW^\top\bW\bZ]}
\gap\text{and}\gap 
\bW \leftarrow\bW \hadaprod \frac{[\bA\bZ^\top-\lambda_w\bW]_+}{[\bW\bZ\bZ^\top]}.
\end{aligned}
$$

\section{NMF with Three Factors}
The NMF method, when extended to incorporate three factor matrices, is referred to as  \textit{nonnegative matrix trifactorization (tri-NMF)}. 
This approach introduces an additional factor:
\begin{equation}\label{equation:tri_nmv}
\bA \approx \bW\bU\bZ,
\end{equation}
where $\bW\in\real_+^{M\times K}$, $\bU\in\real_+^{K\times J}$, and $\bZ\in\real_+^{J\times N}$.
Consider  the item-by-user matrix $\bA\in\real^{M\times N}$, where each element is a binary number $\{0,1\}$. 
This type of data is referred to as \textit{implicit feedback}, in contrast to \textit{explicit feedback} (such as numerical ratings) used in other contexts.~\footnote{For example, in datasets like Netflix or MovieLens, ratings above 4 can be mapped to 1, while ratings below 1 can be mapped to 0 to obtain an implicit data set.}
Standard NMF on this matrix provides a sum of  $K$ rank-one matrices $\bA\approx\sum_{k=1}^{K} \bW[:,k]\bZ[k,:]$.
In the context of implicit data, each rank-one matrix can be interpreted as finding a subset of users and a subset of items (e.g., movies)  that interact strongly with each other.
In contrast, tri-NMF yields the following approximation:
$$
\bA\approx\sum_{k=1}^{K}\sum_{j=1}^{J} \bW[:,k]\bU[k,j]\bZ[j,:].
$$
This formulation can be interpreted as identifying  separately $J$ subsets of movies that are watched together (the rows of $\bZ$) and $K$ subset of users that behave similarly (the columns of $\bW$); while the matrix $\bU$ tells us how these subsets interact together.
If $u_{kj}>0$, then the $k$-th subset of users (corresponding to the positive entries of $\bW[:,k]$) watches the movies from the $j$-th subset of movies (corresponding to the positive entries of $\bZ[j,:]$).

In other words, tri-NMF identifies groups of users who exhibit similar behavior (by watching the same movies) and groups of movies that are similar (because they are watched by the same users), while connecting these groups through the nonnegative interaction matrix $\bU$. 
This model is also applicable in text mining, where it can identify groups of documents that contain similar sets of words (columns of $\bW$) and groups of words that commonly appear together in the same documents (rows of $\bZ$), with 
$\bU$ encoding the relationships between these groups \citep{brouwer2017comparative, gillis2020nonnegative, lu2023bayesian}.

\section{$\beta$-Divergence, Alternative Perspectives of MU}\label{section:beta_div_altmu}
 
As mentioned previously, the sum of squared loss, as given in  \eqref{equation:als-per-example-loss2} or \eqref{equation:frob_nmf}, is convex when one of the factors is held constant, leading to a smooth optimization process.
This type of loss function falls under a broader class of distance/divergence estimators known as \textit{$\beta$-divergence} in the context of NMF.
Given two nonnegative scalars $x$ and $y$, the $\beta$-divergence between $x$ and $y$ is defined as follows:
\begin{equation}\label{equation:scalar_beta_div}
d_{\beta}(x, y)
=
\left\{
\begin{aligned}
&\frac{x}{y}-\ln \frac{x}{y}-1, &\text{if } \beta=0;\\
&x\ln\frac{x}{y} -x+y , &\text{if } \beta=1;\\
&\frac{1}{\beta^2-\beta}(x^\beta+(\beta-1)y^\beta - \beta xy^{\beta-1} ) , &\text{otherwise}.\\
\end{aligned}
\right.
\end{equation}
The $\beta$-divergence is continuous in $\beta$ since $\mathoplim{\beta\rightarrow 0}(x^\beta -y^\beta)/\beta=\ln(x/y)$. 
When $\beta=0, 1,$ and $2$, the $\beta$-divergences are also known as the \textit{Itakura-Saito (IS), KL, and Frobenius/Euclidean distances/divergences}, respectively.
The $\beta$-divergence between two matrices $\bB$ and $\bC$ is 
\begin{equation}\label{equation:beta_div_mat_def}
D_{\beta}(\bB,\bC)=
\sum_{j}d_{\beta}(\bb_j, \bc_j)
=
\sum_{i,j} d_{\beta}({b_{ij}, c_{ij}}).
\end{equation}

The analysis of $\beta$-divergence is complex. When the first argument is fixed at 1,  smaller values are less penalized as  the $\beta$ value increases; however, when the first argument is $2$,  smaller values  are more penalized as  the $\beta$ value increases.
In both cases, larger values are more heavily penalized as the $\beta$ value increases. See Figure~\ref{fig:beta_divergence_all}.
\begin{figure}[h]
\centering  
\vspace{-0.15cm}    
\subfigtopskip=2pt  
\subfigbottomskip=2pt 
\subfigcapskip=-5pt  
\subfigure[$\beta$-divergence for $d_{\beta}(1,y)$.]{\label{fig:beta_divergence_1}
\includegraphics[width=0.47\linewidth]{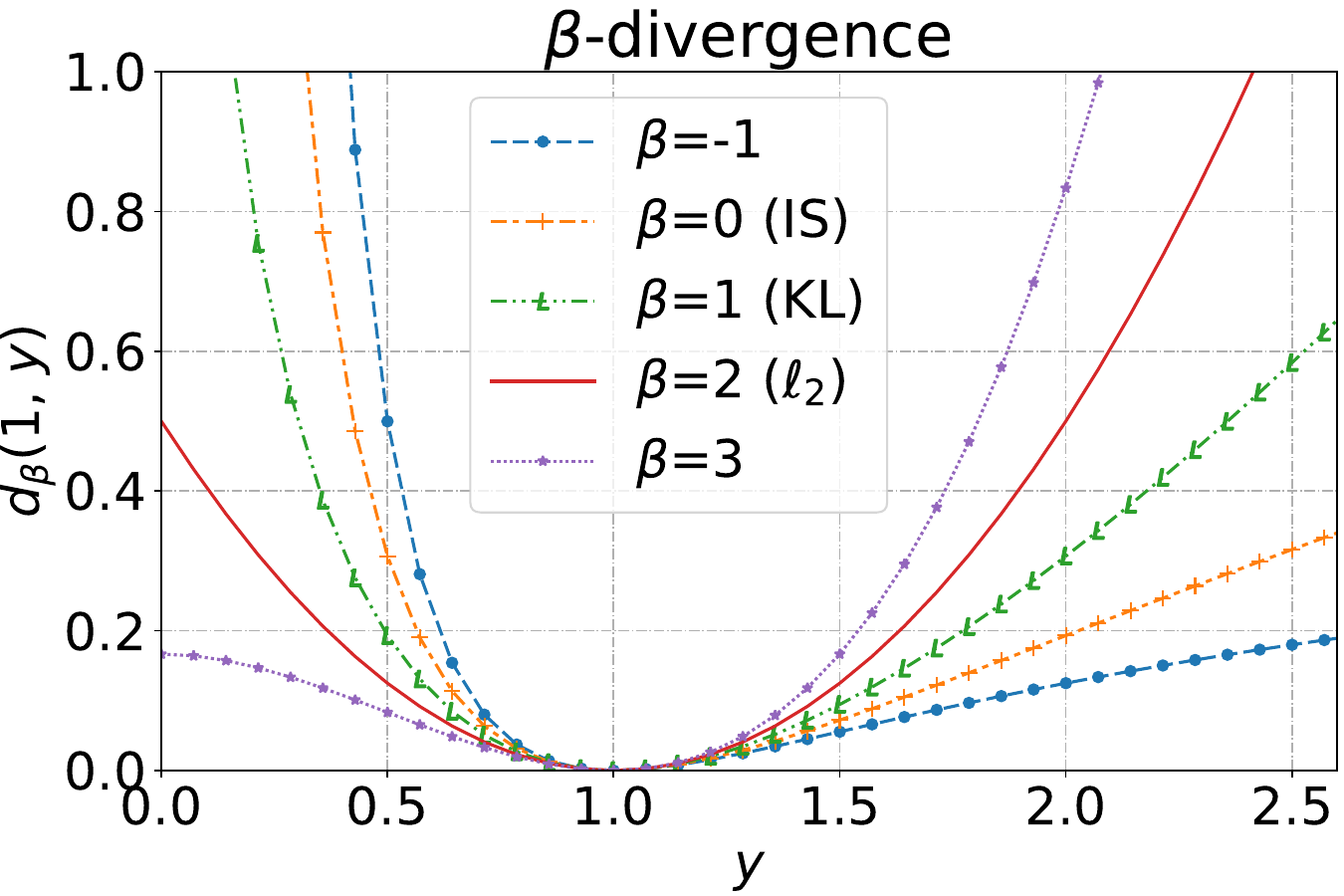}}
\subfigure[$\beta$-divergence for $d_{\beta}(2,y)$.]{\label{fig:beta_divergence_2}
\includegraphics[width=0.47\linewidth]{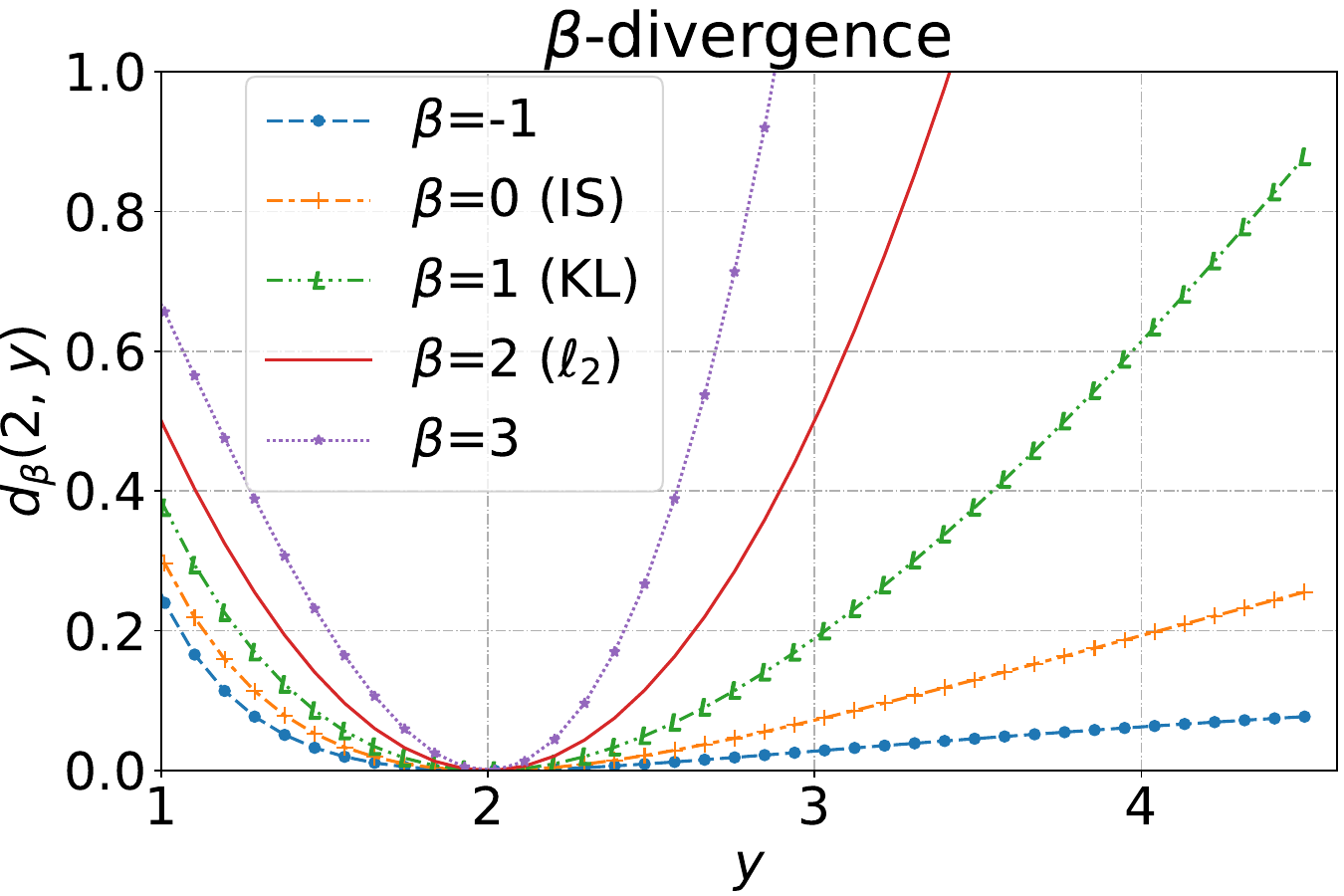}}
\caption{
The analysis of $\beta$-divergence is complex. When the first argument is fixed at 1,  smaller values are less penalized as  the $\beta$ value increases; however, when the first argument is $2$,  smaller values  are more penalized as  the $\beta$ value increases.
In both cases, larger values are more heavily penalized as the $\beta$ value increases.
}
\label{fig:beta_divergence_all}
\end{figure}

\paragraph{Convexity of $\beta$-divergence.}
When $\beta\in[1,2]$, the function $d_{\beta}(x,y)$ is convex in the second argument $y$. This implies $D_{\beta}(\bA,\bW\bZ)$ is convex in $\bW$ when fixing $\bZ$, and vice versa (Problem~\ref{prob:sep_conv}).  
Therefore, coordinate descent algorithms can be effectively applied for NMF using the $\beta$-divergence.

\paragraph{Scaling in arguments.} Let $\gamma>0$ be a scale factor, then 
\begin{equation}
d_{\beta}(\gamma x, \gamma y) = \gamma^\beta d_{\beta}(x,y).
\end{equation}
This indicates that the larger the $\beta$, the more sensitive the $\beta$-divergence is to large values of $x$ or $y$; on the contrary, $\beta$-divergence with small $\beta<0$ values relies more heavily on the smallest data values.
However, when $\beta=0$ (called \textit{Itakura-Saito divergence, IS divergence}), the $\beta$-divergence is not related to the $\beta$ value, and thus it is invariant to scaling. 
What matters is the ratio between $x$ and $y$; see \eqref{equation:scalar_beta_div}.

\paragraph{Gradient.} 
Since we consider a nonnegative matrix $\bA$ for NMF problems, we should note that $d_{\beta}(x, \cdot)$ for $x=0$ is not defined for all values of $\beta$:
$$
d_{\beta}(0, y)=
\left\{
\begin{aligned}
\text{not defined}, \gap &\text{if $\beta\leq 0$};\\
\frac{1}{\beta} y^\beta, \gap \gap &\text{if $\beta> 0$},
\end{aligned}
\right.
\gapthree \implies \gapthree
d'_{\beta}(0, y)=
\left\{
\begin{aligned}
\text{not defined}, \gap &\text{if $\beta\leq 0$};\\
 y^{\beta-1}, \gap \gap &\text{if $\beta> 0$},
\end{aligned}
\right.
$$
where the derivative $d'_{\beta}(0, y)$ corresponds to the second variable $y$.
Therefore, the algorithm developed in the following sections can only be applied to positive matrices when $\beta\leq 0$.
Table~\ref{table:dom_beta_1} and Table~\ref{table:dom_beta_2} present the domains of $d_\beta(x, \cdot)$ and $d'_{\beta}(x, \cdot)$, respectively, for different values of $\beta$ and $x$.
\noindent
\begin{table}[h]
\centering
\begin{minipage}{0.37\textwidth}
\centering
\caption{Domain of $d_\beta(x, \cdot)$.}
\label{table:dom_beta_1}
\begin{tabular}{c|c|c|c}
\hline
& $\beta \leq 0$ & $\beta \in (0,1]$ & $\beta > 1$ \\
\hline
$x = 0$ & $\varnothing$ & $\real_+$ & $\real_+$ \\
$x > 0$ & $\real_{++}$ & $\real_{++}$ & $\real_+$ \\
\hline
\end{tabular}
\end{minipage}
\quad
\begin{minipage}{0.54\textwidth}
\centering
\caption{Domain of $d'_{\beta}(x, \cdot)$.}
\label{table:dom_beta_2}
\begin{tabular}{c|c|c|c|c}
\hline
& $\beta \leq 0$ & $\beta \in (0,1)$ & $\beta \in [1,2)$ & $\beta \geq 2$ \\
\hline
$x = 0$ & $\varnothing$ & $\real_{++}$ & $\real_+$ & $\real_+$ \\
$x > 0$ & $\real_{++}$ & $\real_{++}$ & $\real_{++}$ & $\real_+$ \\
\hline
\end{tabular}
\end{minipage}
\end{table}

On the other hand, the gradients of  $D_{\beta}(\bA,\bW\bZ)$ w.r.t. $\bZ$ and $\bW$ are given by (if exist) 
\begin{equation}
\begin{aligned}
\nabla_{\bZ} D_{\beta}(\bA,\bW\bZ) &= \bW^\top \big( (\bW\bZ)^{\beta-2} \hadaprod (\bW\bZ-\bA) \big);\\
\nabla_{\bW} D_{\beta}(\bA,\bW\bZ) &=  \big( (\bW\bZ)^{\beta-2} \hadaprod (\bW\bZ-\bA) \big)\bZ^\top, 
\end{aligned}
\end{equation}
where $(\bW\bZ)^{\beta-2}$ denotes the componentwise exponent.
When $\beta=2$, the gradient reduces to the one in \eqref{equation:givenw-update-z-allgd} and \eqref{equation:givenz-update-w-allgd}.

\paragraph{Decomposition of $\beta$-divergence.}
The $\beta$-divergence can be divided into three parts: convex, concave, and constant terms. 
We should note that this decomposition is not unique since any affine term is both convex and concave. 
We follow the convention in  \citet{fevotte2011algorithms}:
\begin{equation}\label{equation:beta_decom_defi}
d_{\beta}(x,y) = \convd_{\beta}(x,y)+\concd_{\beta}(x,y)+\cnstd_{\beta}(x,y),
\end{equation} 
where $\convd_{\beta}(x,y)$ is convex in $y$, $\concd_{\beta}(x,y)$ is concave in $y$, and $\cnstd_{\beta}(x,y)$ is constant in $y$; see Table~\ref{table:beta_decom} for different $\beta$ values.

\begin{table}[h]
\setlength{\tabcolsep}{6.pt}    
\centering
\begin{tabular}{c|c|c|c}
\hline
& $\convd_{\beta}(x,y)/\convd'_{\beta}(x,y)$, convex & $\concd_{\beta}(x,y)/\concd'_{\beta}(x,y)$, concave & $\cnstd_{\beta}(x,y)$, constant \\ \hline\hline
$\beta < 1, \beta \neq 0$ & $-\frac{1}{\beta-1}xy^{\beta-1}/ -xy^{\beta-2}$ & $\frac{1}{\beta}y^{\beta}/y^{\beta-1}$ & $\frac{1}{\beta(\beta-1)}x^{\beta}$ \\ \hline
$\beta = 0$ & $xy^{-1}/-xy^{-2}$ & $\ln y/y^{-1}$ & $x(\ln x - 1)$ \\ \hline
$1 \leq \beta \leq 2$ & $d_{\beta}(x,y)/d'_{\beta}(x,y)$ & 0/0 & 0 \\ \hline
$\beta > 2$ & $\frac{1}{\beta}y^{\beta}/y^{\beta-1}$ & $-\frac{1}{\beta-1}xy^{\beta-1}$ & $\frac{1}{\beta(\beta-1)}x^{\beta}$ \\ \hline
\end{tabular}
\caption{Scalar convex-concave-constant decomposition  of $d_{\beta}(x,y)$ with respect to the second variable $y$, and the corresponding derivatives with respect to the second variable $y$.}
\label{table:beta_decom}
\end{table}

\paragraph{KKT conditions for NMF with $\beta$-divergence.}

The KKT conditions  indicate that (see derivation in \eqref{equation:kkv_nnn_raw}):
\begin{equation}\label{equation:nmf_beta_kkt1}
\begin{aligned}
\bZ\geq \bzero,\gap&\nabla_{\bZ} D_{\beta}(\bA, \bW\bZ) &\geq& \bzero, \gap \langle \bZ, \nabla_{\bZ} D_{\beta}(\bA, \bW\bZ)\rangle &=&\bzero_{K\times N}; \\
\bW\geq \bzero,\gap&\nabla_{\bW} D_{\beta}(\bA, \bW\bZ) &\geq& \bzero, \gap \langle \bW, \nabla_{\bW} D_{\beta}(\bA, \bW\bZ)\rangle &=&\bzero_{M\times K}.
\end{aligned}
\end{equation}
This also implies
\begin{equation}
\begin{aligned}
\min\{\bZ, \nabla_{\bZ} D_{\beta}(\bA, \bW\bZ)\} = \bzero_{K\times N}
\gap \text{and}\gap 
\min\{\bW, \nabla_{\bW} D_{\beta}(\bA, \bW\bZ)\} = \bzero_{M\times K},
\end{aligned}
\end{equation}
where the min operator $\min\{\cdot, \cdot \}$ is applied componentwise.

\subsection{MU for $\beta$-Divergence Obtained by Gradient Ratio Heuristic}\label{section:mu_gd_ratio}
We have shown that the MU update for the Frobenius norm can be derived from rescaled gradient descent. 
For brevity, let $\nabla_{\bZ}=\nabla_{\bZ} D_{\beta}(\bA,\bW\bZ) = \nabla_{\bZ}^+ - \nabla_{\bZ}^-$, where 
\begin{equation}\label{equation:mu_grati_decom}
\nabla_{\bZ}^+ = \bW^\top \big( (\bW\bZ)^{\beta-1} \big)
\gap \text{and}\gap 
\nabla_{\bZ}^- = \bW^\top \big( (\bW\bZ)^{\beta-2} \hadaprod \bA \big).
\end{equation}
When $z_{kn}>0, \forall k,n$, the KKT conditions show that $(\nabla_{\bZ}^+)_{kn} = (\nabla_{\bZ}^-)_{kn}$. The rule from gradient descent (i.e., $\bZ^{(t+1)}=\bZ^{(t)} - \eta\nabla_{\bZ}$) indicates a small decrease (resp., increase) of $z_{kn}$ will lead to a  decrease of the loss function if $(\nabla_{\bZ})_{kn}>0$ (resp., $<0$).
Therefore, it is reasonable to update $z_{kn}$ using the componentwise ratio between $\nabla_{\bZ}^-$ and $\nabla_{\bZ}^+$:
\begin{equation}\label{equation:mu_grati_z}
\bZ\leftarrow  \bZ \hadaprod \frac{[\nabla_{\bZ}^-]}{[\nabla_{\bZ}^+]},
\end{equation}
where $\frac{[\cdot]}{[\cdot]}$ represents the componentwise ratio of two matrices.
This update rule also corresponds to a multiplicative update (MU).
When $\beta=2$, the MU algorithm reduces to the case in Theorem~\ref{theorem:conv_mu_fro}.
When $\beta=1$, the loss function becomes the KL divergence, and the update for $\bZ$ is 
$$
\textbf{($\beta=1$)}:\gap \bZ\leftarrow \bZ \hadaprod \frac{[\bW^\top \frac{[\bA]}{[\bW\bZ]}]}{[\bW^\top \bone_{M\times N}]}.
$$
It can be shown that when $\beta\in[1,2]$, the MU algorithms derived for $\beta$-divergence will monotonically decrease $D_{\beta}(\bA, \bW\bZ)$.

\subsection{MU for $\beta$-Divergence  Obtained by Rescaled PGD}
As discussed in Section~\ref{section:nmf_apgd},
the PGD approach involves projecting the gradient descent update onto the feasible set \citep{lu2025practical}. Consider a standard GD update on $f(\bx)$: $\bx^{(t+1)}=\bx^{(t)}-\eta\nabla f(\bx^{(t)})$, where $\eta$ is a step size and  $-\nabla f(\bx^{(t)})$ is a \textit{descent direction} ($\bg$ is a descent direction if $\bg^\top\nabla f(\bx^{(t)})<0$).
Consider further  a diagonal $\bD$ such that $-\eta \nabla f(\bx^{(t)}) \rightarrow -\bD \nabla f(\bx^{(t)})$ is also a descent direction (replacing the step size by a diagonal matrix)~\footnote{$\bD$ can be relaxed to any positive definite matrices.}.
In this case, if the feasible set of $\bx$ is nonnegative, then the PGD is useful: $\bx^{(t+1)}=\mathcalP(\bx^{(t)}-\bD\nabla f(\bx^{(t)}))$, where $\mathcalP(x)=\max\{x, 0\}$~\footnote{see, for example, \citet{beck2017first}.}.
If we further decompose the gradient into positive and negative parts: $\nabla f(\bx^{(t)}) = \nabla^+f(\bx^{(t)}) - \nabla^-f(\bx^{(t)})$ with $\nabla^+f(\bx^{(t)})>0 $ and $\nabla^-f(\bx^{(t)})>0$, taking $\bD=\diag\big( \frac{[\bx^{(t)}]}{[\nabla^+f(\bx^{(t)}))]}  \big)$, the rescaled PGD update becomes a MU rule:
\begin{equation}
\bx^{(t+1)}=\mathcalP\bigg(\bx^{(t)}-\diag\big( \frac{[\bx^{(t)}]}{[\nabla^+f(\bx^{(t)})]}  \big)\nabla f(\bx^{(t)})\bigg)
=
\mathcalP\bigg(\bx^{(t)} \hadaprod \frac{[\nabla^-f(\bx^{(t)})]}{[\nabla^+f(\bx^{(t)})]}\bigg).
\end{equation}
If we use the decomposition of gradient in \eqref{equation:mu_grati_decom}, the rescaled PGD becomes the MU update for NMF in \eqref{equation:mu_grati_z}.
If we further incorporate a step size $\eta$ in the rescaled PGD update, it becomes 
\begin{equation}
\bx^{(t+1)}
=
\mathcalP\bigg((1-\eta)\bx^{(t)} + \eta\bx^{(t)} \hadaprod \frac{[\nabla^-f(\bx^{(t)})]}{[\nabla^+f(\bx^{(t)})]}\bigg).
\end{equation}
Since $-\bD \nabla f(\bx^{(t)})$ is a descent direction, the step size $\eta \in(0,1)$ can ensure that the update is monotonically nonincreasing.
Note that the projection operator can be omitted since all  updates are nonnegative.

\subsection{MU for $\beta$-Divergence Obtained by MM Framework}
The $\beta$-divergence between two matrices can be defined columnwise (Equation~\eqref{equation:beta_div_mat_def}), and the $\beta$-divergence can be divided into three parts (convex, concave, and constant, Equation~\eqref{equation:beta_decom_defi}). Thus, the loss function in NMF can be decomposed into (note the loss function can be further divided componentwise):
$$
D_{\beta}(\bA, \bW\bZ)
=
\sum_{n=1}^{N}d_{\beta}(\ba_n, \bW\bz_n)
=
\sum_{n=1}^{N}\left(\convd_{\beta}(\ba_n, \bW\bz_n) + \concd_{\beta}(\ba_n, \bW\bz_n) + \cnstd_{\beta}(\ba_n, \bW\bz_n)\right).
$$
For each column $n$, the MM framework involves finding auxiliary functions for the three components separately.
To see this, we need the following lemma:
\begin{lemma}[Auxiliary function by parts]
Let $F(\bx)=\sum_{i=1}^{n} F_i(\bx)$, and  let $G_i(\bx, \widetildebx)$ be an auxiliary function for $F_i(\bx)$ at $\widetildebx$ for all $i$. 
Then, $G(\bx, \widetildebx)=\sum_{i=1}^{n}G_i(\bx, \widetildebx)$ is an auxiliary function for $F(\bx)$ at $\widetildebx$.
\end{lemma}
This lemma indicates that if the auxiliary function is constructed separately for each component, it allows us to decouple the optimization.

\paragraph{Constant part.}
There is no need to find an auxiliary function for the constant term  $\cnstd_{\beta}(\ba_n, \bW\bz_n)$, since it does not influence the minimization of $d_{\beta}(\ba_n, \bW\bz_n)$ with respect to $\bz_n$.
\paragraph{Concave part.}
Any concave function can be upper-bounded using  linearization (the tangent plane):
$$
\concd_{\beta}(x, y) 
\leq 
\concd_{\beta}(x, \widetildey)
+
(y-\widetildey) \concd_{\beta}'(x, \widetildey), 
$$
where $\concd_{\beta}'(x, \widetildey)$ denotes the gradient of $\concd(x, \widetildey)$ with respect to the second component $\widetildey$.
Therefore, for any $\widetildebz_n\in\real^{K}$, the auxiliary function for the concave component $\concd_{\beta}(\ba_n, \bW\bz_n)$ can be constructed by  
$$
\concG(\bz_n, \widetildebz_n) = \concd_{\beta}(\ba_n, \bW\widetildebz_n) + (\bW\bz_n-\bW\widetildebz_n) \hadaprod \concd_{\beta}' (\ba_n, \bW\widetildebz_n).
$$

\paragraph{Convex part.}
The auxiliary function for the convex part follows from the convexity inequality~\footnote{Let $f:\sS\rightarrow \real$ be a convex function, and let $p\geq 2$ be any integer. Then, 
$
f\big(\sum_{i=1}^{p} \lambda_i\bx_i\big) \leq \sum_{i=1}^{p}\lambda_i f(\bx_i),
$
if $\lambda_i\geq 0$ and $\sum_{i=1}^{p}\lambda_i=1$.}.
Construct a matrix $\bP\in\real^{M\times K}$ as follows:
\begin{equation}
p_{mk} = \frac{w_{mk}{\widetildez}_{kn}}{\sum_{j}w_{mj}{\widetildez}_{jn}}
=
\frac{w_{mk}{\widetildez}_{kn}}{\bW[m,:]\widetildebz_n}
\gap \implies\gap
\bP\geq\bzero \text{ and } \bP\bone = \bone.
\end{equation}
That is, each row of $\bP$ belongs to the unit simplex in $\real^K$.
Therefore, we have
$$
\begin{aligned}
\convd_{\beta}(a_{mn}, \bW[m,:]\bz_n ) 
&=
\convd_{\beta}\big(a_{mn}, \sum_{k=1}^{K}w_{mk}z_{kn} \big) 
=
\convd_{\beta}\big(a_{mn}, \sum_{k=1}^{K} p_{mk}\frac{w_{mk}z_{kn}}{p_{mk}} \big) \\
&\leq 
\sum_{k=1}^{K} p_{mk}\convd_{\beta}\big(a_{mn}, \frac{w_{mk}z_{kn}}{p_{mk}} \big). 
\end{aligned}
$$

This decomposition finds an auxiliary function for $D_{\beta}(\bA,\bW\bZ)$ w.r.t. $\bZ$.
\begin{theorem}[Auxiliary function for $D_{\beta}(\bA,\bW\bZ)$ w.r.t. $\bZ$]\label{theorem:aux_db_z}
Let $\widetildeba_n=\bW\widetildebz_n$ with $\widetildea_{mn}=\bW[m,:]\widetildebz_n$ for all $m, n$, where $\widetildebz_n$ is any vector in $\real^K$. Then, $G(\bZ,\widetildebZ)=\sum_{n=1}^{N} G_{n}(\bz_n, \widetildebz_n) = \sum_{n=1}^{N} \sum_{m=1}^{M} G_{mn}$ is an auxiliary function for $D_{\beta}(\bA,\bW\bZ)$ w.r.t. $\bZ$, where 
$$
\footnotesize
\begin{aligned}
G_{mn}
&=
\cnstd_{\beta}(a_{mn}, \widetildea_{mn})
+
\concd_{\beta}(a_{mn}, \widetildea_{mn}) + \sum_{k=1}^{K}w_{mk}(z_{kn}-\widetildez_{kn}) \concd_{\beta}'(a_{mn}, \widetildea_{mn})
+ \sum_{k=1}^{K} \frac{w_{mk}\widetildez_{kn}}{\widetildea_{mn}}\convd_{\beta}\big(a_{mn},  \frac{\widetildea_{mn}z_{kn}}{\widetildez_{kn}} \big).
\end{aligned}
$$
\end{theorem}

\begin{exercise}[Gradient and Hessian of auxiliary functions]\label{exercise:gra_hes_aux}
Consider the setting and notations in Theorem~\ref{theorem:aux_db_z}. 
Let $G_n(\bz_n, \widetildebz_n)= \sum_{k=1}^{K} G_k (\bz_{kn}, \widetildebz_n) + C(\bz_n)$ where $C(\bz_n)$ is a constant w.r.t. $\bz_n$. 
That is, 
$$
G_k (\bz_{kn}, \widetildebz_n) 
= 
\sum_{m=1}^{M}w_{mk} z_{kn}\concd_{\beta}'(a_{mn}, \widetildea_{mn})
+ 
\sum_{m=1}^{M} \frac{w_{mk}\widetildez_{kn}}{\widetildea_{mn}}\convd_{\beta}\big(a_{mn},  \frac{\widetildea_{mn}z_{kn}}{\widetildez_{kn}} \big).
$$
Show that the gradient of the auxiliary function is 
$$
\nabla_{z_{kn}}G_n(\bz_n, \widetildebz_n) 
=
\sum_{m=1}^{M} 
w_{mk}
\bigg(
\concd_{\beta}'(a_{mn}, \widetildea_{mn})
+
\convd_{\beta}'\big(a_{mn},  \frac{\widetildea_{mn}z_{kn}}{\widetildez_{kn}} \big)
\bigg), 
$$
and the Hessian matrix is diagonal with entries
$$
\nabla^2_{z_{kn}}G_n(\bz_n, \widetildebz_n) 
=
\sum_{m=1}^{M} 
w_{mk}
\frac{\widetildea_{mn}}{\widetildez_{kn}}
\bigg(
\convd_{\beta}{''}\big(a_{mn},  \frac{\widetildea_{mn}z_{kn}}{\widetildez_{kn}} \big)
\bigg).
$$
Note in all cases, the first-order derivative or the second-order derivative corresponds to the second argument of $d_{\beta}(\cdot, \cdot)$.
\end{exercise}

Since $\convd_{\beta}(\cdot, \cdot)$ is convex in the second argument, the Hessian is positive definite. Thus, the auxiliary function is convex.
These constructions result in the following theorem by minimizing the  auxiliary function obtained in Theorem~\ref{theorem:aux_db_z}.
\begin{theorem}[Nonincreasing of MU for $\beta$-divergence \citep{fevotte2011algorithms, gillis2020nonnegative}]\label{theorem:conv_mu_beta}
Let $\bA\in\real_+^{M\times N}$, $\bW\in\real_{++}^{M\times K}$, and $\bZ\in\real_{++}^{K\times N}$. 
The  loss $D_{\beta}(\bA,\bW\bZ)$ remains nonincreasing under the following multiplicative update rules:
$$
\footnotesize
\begin{aligned}
\bZ\leftarrow \bZ\hadaprod \left(\frac{\left[\bW^\top \left\{ (\bW\bZ)^{(\beta-2)} \hadaprod \bA \right\}\right]}{[\bW^\top (\bW\bZ)^{(\beta-1)}]}\right)^{m(\beta)},
\gapthree \text{and}\gapthree
\bW\leftarrow \bW\hadaprod \left(\frac{\left[\left\{ (\bW\bZ)^{(\beta-2)} \hadaprod \bA \right\} \bZ^\top \right]}{[ (\bW\bZ)^{(\beta-1)}\bZ^\top]}\right)^{m(\beta)},
\end{aligned}
$$
where 
$$
m(\beta)=\left\{
\begin{aligned}
&\frac{1}{2-\beta}, \gap &\textit{if }&\beta<1; \\
&1,                  &\textit{if }& 1\leq\beta\leq 2; \\
&\frac{1}{\beta-1}, &\textit{if }& \beta>1.
\end{aligned}
\right.
$$
When $\beta=2$, the result reduces to Theorem~\ref{theorem:conv_mu_fro}.
When $1\leq \beta\leq 2$, the MU obtained via the MM framework coincides with heuristic described in Section~\ref{section:mu_gd_ratio}.
\end{theorem}
The update in Theorem~\ref{theorem:conv_mu_beta} ensures nonnegativity of the parameter updates, provided they are initialized with positive values.

\paragraph{Choice of $\beta$ for NMF.}
The choice of $\beta$-divergence for NMF is problem-dependent.
\citet{fevotte2009nonnegative} present  results of decomposing a piano power spectrogram using $\beta=0$ and demonstrate that components corresponding to very low residual noise and hammer strikes on the strings are extracted with great accuracy; these components are either ignored or severely degraded when using Euclidean or KL distances/divergences. 
\citet{fitzgerald2009use} show that  $\beta=0.5$ is optimal for  music source separation problems.

\paragraph{Convergence.}
An algorithm is said to be  \textit{convergent} if it produces a sequence of iterates $\{\bZ^{(t)}\}_{t\geq 1}$ or $\{\bW^{(t)}\}_{t\geq 1}$ that converges to a limit point $\bW^*$ or $\bZ^*$ satisfying the KKT conditions in \eqref{equation:nmf_beta_kkt1}. Monotonic nonincreasingness does not imply convergence in general, and neither is monotonicity necessary for convergence. Proving convergence of the MU methods is beyond the scope of this book; we refer the readers to \citet{gillis2020nonnegative, fevotte2011algorithms} and references therein for more details.

\subsection{Initialization of NMF}
A significant challenge in NMF is the lack of guaranteed convergence to a global minimum.  
Often, the convergence process is slow, and the algorithm may reach a suboptimal approximation.
In the preceding discussion, we initialized $\bW$ and $\bZ$ randomly. 
To mitigate this issue, there are also alternative strategies designed to obtain better initial estimates in the hope of converging more rapidly to a good solution \citep{boutsidis2008svd, gillis2014and}. We sketch the methods as follows for reference:
\begin{itemize}
\item \textit{Clustering techniques.} Apply some clustering methods to the columns of $\bA$, set the cluster means of the top $K$ clusters as the columns of $\bW$, and initialize $\bZ$ as a proper scaling of
the cluster indicator matrix (that is, $z_{kn}\neq 0$ indicates  that $\ba_n$ belongs to the $k$-th cluster);
\item \textit{Subset selection.} Pick $K$ columns of $\bA$, and set those as the initial columns for $\bW$. And analogously, $K$ rows of $\bA$ are selected to form the rows of $\bZ$;
\item \textit{SVD-based approach.} Suppose the optimal rank-$K$ approximation of $\bA$ is $\bA=\sum_{i=1}^{K}\sigma_i\bu_i\bv_i^\top$, where each factor $\sigma_i\bu_i\bv_i^\top$ is a rank-one matrix with possible negative values in $\bu_i$ and $\bv_i$, and nonnegative $\sigma_i$. Denote $[x]_+=\max(x, 0)$, we notice 
$$
\bu_i\bv_i^\top = [\bu_i]_+[\bv_i]_+^\top+[-\bu_i]_+[-\bv_i]_+^\top-[-\bu_i]_+[\bv_i]_+^\top-[\bu_i]_+[-\bv_i]_+^\top,
$$
where the first two rank-one factors in this decomposition are nonnegative.
Then, either $[\bu_i]_+[\bv_i]_+^\top$ or $[-\bu_i]_+[-\bv_i]_+^\top$ can be selected to replace the factor $\bu_i\bv_i^\top$. \citet{boutsidis2008svd} suggest to replace each rank-one factor in $\sum_{i=1}^{K}\sigma_i\bu_i\bv_i^\top$ with  either $[\bu_i]_+[\bv_i]_+^\top$ or $[-\bu_i]_+[-\bv_i]_+^\top$, selecting the one with the larger norm and scaling it properly.
In other words, if we select $[\bu_i]_+[\bv_i]_+^\top$, then $\sigma_i\cdot [\bu_i]_+  $ can be initialized as the $i$-th column of $\bW$, and $[\bv_i]_+^\top$ can be chosen as the $i$-th row of $\bZ$.
\end{itemize}
However, these techniques are not guaranteed to yield better performance theoretically. 
We recommend referring to the aforementioned papers for more detailed information.

\index{Implicit hierarchy}
\section{Movie Recommender Context}
Both  NMF and  ALS methods approximate a matrix and reconstruct its entries using a set of basis/template vectors. 
The key difference lies in the nature of these basis vectors and how the approximation is carried out.
The basis  in  NMF is composed of vectors with nonnegative elements while the basis vectors in  ALS can have positive or negative values.
In  NMF, each vector is reconstructed  as a nonnegative summation of the basis vectors with  ``relatively" small components in the direction of each basis vector.
In contrast, in the ALS approximation, the data is modeled as a linear combination of the basis vector such that we can add or subtract vectors as needed; and the components in the direction of each basis vector can be large positive values or negative values. Therefore, depending on the application, one or the other factorization can be utilized to describe the data with different meanings.

\paragraph{Movie recommender context.}
In the context of a movie recommender system,  the rows of $\bW$ represent the hidden features of  movies, while the columns of $\bZ$ represent the hidden features of users. 
For example, in NMF, a movie might be described as 0.5 comedy, 0.002 action, and 0.09 romantic. However, in the ALS approach, we can get combinations such as 4 comedy, $-0.05$ action, and $-3$ drama, indicating positive or negative contributions to each feature.

\paragraph{Implicit hierarchy.}
Both ALS and NMF do not rank the importance of each basis vector hierarchically. In contrast, singular value decomposition (SVD) ranks the importance of each basis vector based on the corresponding singular value.
In the SVD representation of $\bA=\sum_{i=1}^{r}\sigma_i\bu_i\bv_i^\top$, 
this usually means that the reconstruction $\sigma_1\bu_1\bv_1^\top$ via the first set of basis vectors dominates and is the most used set to reconstruct data, followed by the second set, and so on. 
This creates an implicit hierarchy in the SVD basis that doesn't happen in the ALS or the NMF approach.

\paragraph{Interpretability of basis vectors.}
In  SVD, the basis vectors can be statistically interpreted as the directions of maximum variance, but many of these directions lack a clear visual or intuitive interpretation due to the presence of zero, positive, and negative entries. 
When these basis vectors are used in a linear combination, the combination involves complex cancellations of positive and negative values, which can obscure the intuitive physical meaning of individual basis vectors. As a result, many basis vectors do not provide a meaningful explanation for nonnegative data, such as  pixel values in a color image.
On one hand, the entries of a nonnegative pattern vector should all be nonnegative values. On the other hand, mutually orthogonal singular vectors must contain negative entries. For example, if all  entries of the singular vector  $\bu_1$ corresponding to the maximum singular value are nonnegative, then any other singular vector orthogonal to $\bu_1$ must contain at least one negative entry; otherwise, the orthogonality condition $\bu_1^\top\bu_j=0$ for $j\neq 1$ cannot be satisfied. This indicates that mutually orthogonal singular vectors are not suitable as pattern vectors or basis vectors in nonnegative data analysis.

\section{Other Applications}

\paragraph{Music spectral reconstruction.}
To illustrate the application of NMF, we demonstrate how this technique can decompose a spectrogram of a music recording into components that carry musical significance \citep{muller2015fundamentals}. As an example, let's examine the opening measures of \textit{Frédéric Chopin's Prélude Op. 28, No. 4}. Figure~\ref{fig:nmf_music_note} presents the musical notation alongside a piano-roll visualization that is synchronized with an audio recording of the piece. For clarity, all information pertaining to the note numbered $p=71$ are emphasized with red rectangular frames.
\begin{figure}[h]
\centering       
\subfigtopskip=2pt               
\subfigbottomskip=-2pt         
\subfigcapskip=-30pt      
\includegraphics[width=0.98\textwidth]{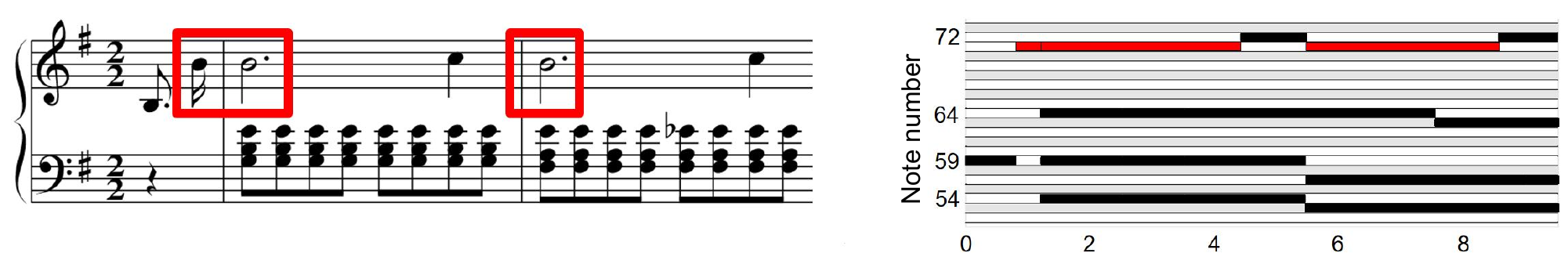}
\caption{Musical score and piano-roll representation. Figure is adapted from \citet{muller2015fundamentals}.}
\label{fig:nmf_music_note}
\end{figure}

Regarding the original data matrix $\bA$, we utilize the magnitude STFT (see, for example, \citet{lopez2019nmf}), which consists of a series of spectral vectors. By applying NMF, this matrix can be decomposed into two nonnegative matrices, $\bW$ and $\bZ$. Ideally,  $\bW$ encapsulates the spectral patterns corresponding to the pitches of the notes present in the musical piece, whereas $\bZ$ indicates the temporal points at which these patterns appear in the audio recording. Figure~\ref{fig:nmf_music_decom} illustrates such a decomposition applied to the Chopin prelude.
\begin{figure}[h]
	\centering       
	\vspace{-0.35cm}                 
	\subfigtopskip=2pt               
	\subfigbottomskip=-2pt         
	\subfigcapskip=-10pt      
	\includegraphics[width=0.98\textwidth]{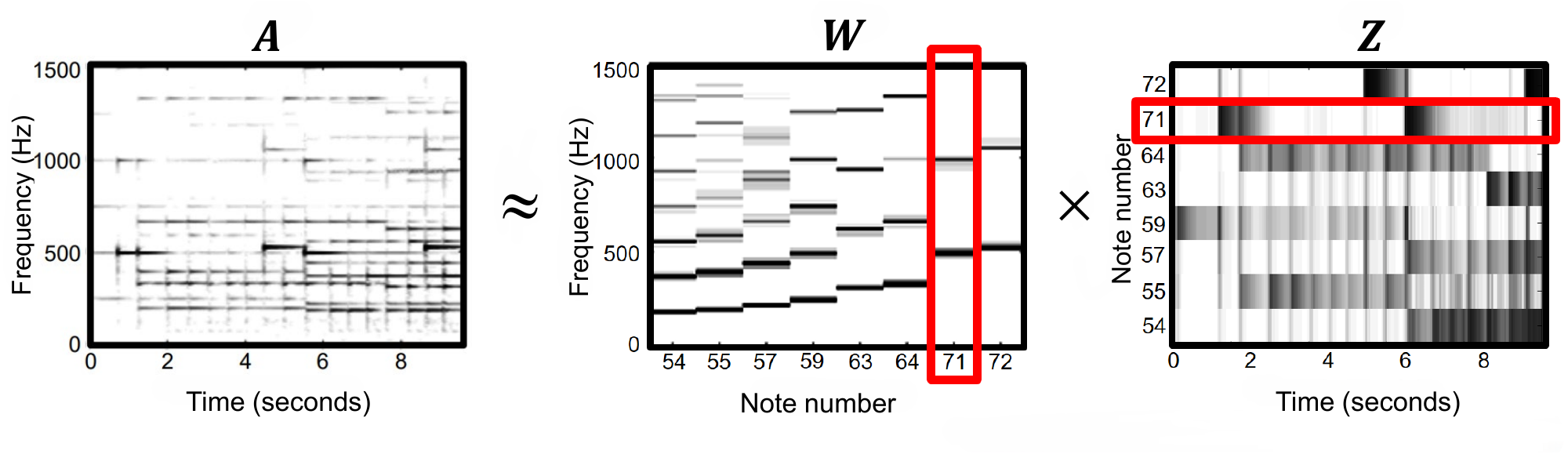}
	\caption{Ideal NMF of the spectrogram using NMF. Figure is adapted from \citet{muller2015fundamentals}.}
	\label{fig:nmf_music_decom}
\end{figure}

In this scenario, each template represented by the matrix  $\bW$ corresponds to the spectral manifestation of a specific pitch within  $\bA$, and the activation matrix $\bZ$ resembles the piano-roll representation of the musical score. 
Therefore, the advantages of NMF over general matrix factorization are evident:
\begin{itemize}
\item \textbf{Nonnegativity constraint.} NMF enforces nonnegativity on both the factorization matrices $\bW$ and $\bZ$. This constraint makes the resulting matrices more interpretable because they can be directly related to physical or perceptual quantities in the domain of interest. In the case of music, the nonnegative factors correspond to meaningful musical elements like notes or chords.
\item \textbf{Interpretability.} In NMF, the matrix $\bW$ represents the spectral patterns (timbres) of the notes present in the music piece, and $\bZ$ indicates the temporal activations of these patterns. This leads to a more interpretable decomposition compared to unconstrained matrix factorization methods, where the factors might not have a clear physical or musical interpretation.
\end{itemize}

\begin{figure}[h]
\centering       
\vspace{-0.25cm}                 
\subfigtopskip=2pt               
\subfigbottomskip=-2pt         
\subfigcapskip=-10pt      
\includegraphics[width=0.98\textwidth]{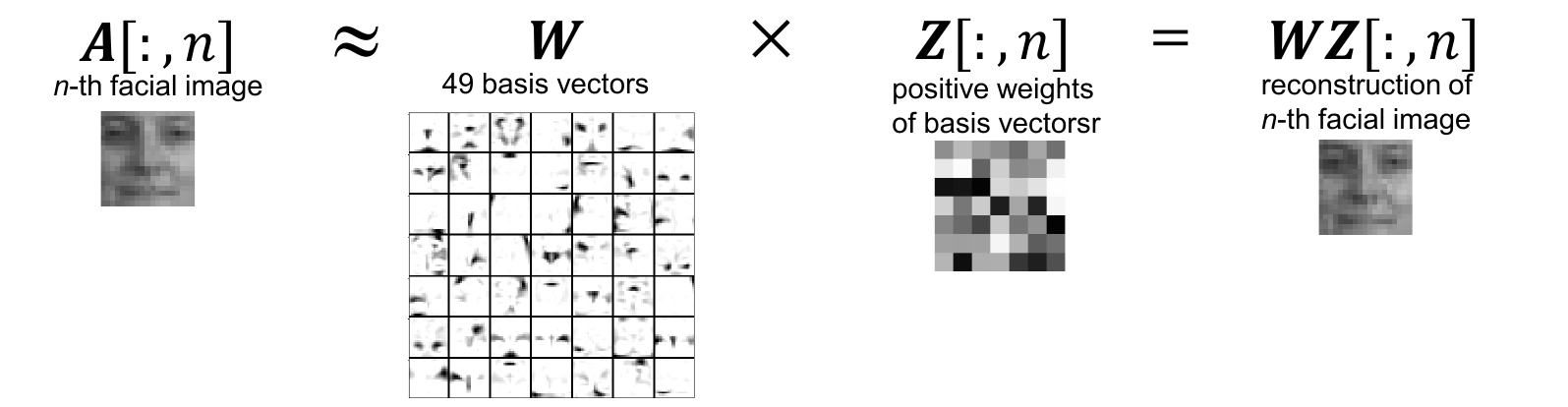}
\caption{NMF of the CBCL face database with $K=49$. The basis vectors in $\bW$ are reshaped into $19\times 19$ images. Facial features can be observed from these reshaped basis vectors, e.g., eyes, noses, nasolabial folds, and lips. Figure is adapted from \citet{lee1999learning, gillis2014and}.}
\label{fig:nmf_face_decom}
\end{figure}

\paragraph{Facial feature extraction and reconstruction.}
Suppose each column of the data matrix $\bA \in \real_+^{M\times N}$ represents a vectorized grayscale image of a face, where the $(m,n)$-th entry of $\bA$ corresponds to the intensity of the $m$-th pixel in the $n$-th face. NMF decomposes $\bA$ into two nonnegative matrices $\bW$ and $\bZ$ such that each image $\ba_n$ can be approximated by a linear combination of the columns of $\bW$. Since $\bW$ is nonnegative, its columns can be interpreted as images, referred to as \textit{template images} or \textit{basis images}, which are vectors of pixel intensities. The nonnegative weights in $\bZ$ ensure that these template images are combined additively to reconstruct each original face image. Given that the number of basis images $K$ is typically much smaller than the number of images $N$, the basis images should capture localized and sparse features that are common across multiple images. For facial images, these basis images often represent features such as eyes, noses, nasolabial folds, and lips (see Figure~\ref{fig:nmf_face_decom}, NMF for  the CBCL face data~\footnote{http://cbcl.mit.edu/software-datasets/FaceData2.html}), while the columns of $\bZ$ indicate the presence of these features in each image \citep{lee1999learning, gillis2014and}.

On the other hand, if each column of $\bA$ indicates a facial image of a single person, the  NMF approach can be utilized for \textit{face recognition}. NMF has been shown to be more robust to \textit{occlusions} compared to PCA or ALS, which generates dense factors. For instance, if a new face with occlusions (e.g., wearing sunglasses or distortions) needs to be mapped into the NMF basis, the non-occluded parts (e.g., the lips or the noise) can still be accurately approximated \citep{jain2017non}.

\paragraph{Topic recovery.} 
As mentioned at the very beginning of this chapter, NMF can be effectively applied to topic recovery problems. 
Typically, this application involves constructing a term-document matrix  $\bA$, where the rows correspond to terms (words or phrases) and the columns correspond to documents. Each entry  $a_{mn}$ in the matrix represents the weight or importance of term  $m$ in document $n$. 
This weight could be binary (presence/absence), \textit{term frequency (TF)}, or \textit{term frequency-inverse document frequency (TF-IDF)} \citep{shahnaz2006document}.
In this framework, each column of $\bW$ can be interpreted as a topic, while each column of $\bZ$ indicates the positive weight of each topic for the given document;  NMF in this context aligns with a \textit{soft clustering} approach where each data point can belong to multiple clusters. 
NMF is particularly well-suited for topic recovery because it captures the additive nature of topics in text data and produces interpretable results. However, the choice of the number of topics $K$  and the initialization of $\bW$  and  $\bZ$ can significantly affect the quality of the results. Additionally, NMF is sensitive to the scaling of the input data, so it's important to preprocess the term-document matrix appropriately.

\index{$L$-strongly smoothness}
\begin{problemset}

\item \label{prob:llipschi_hianls} \textbf{$L$-strongly smooth and PGD in Hi-ANLS problems.} A function $f(\bx): \real^n\rightarrow \real$ is called an $L$-Lipschitz gradient continuous function (a.k.a., a $L$-strongly smooth function) if it satisfies that $\normtwo{\nabla f(\bx) - \nabla f(\by)} \leq L\normtwo{\bx-\by}$ for all $\bx,\by$. Show that the subproblem~\eqref{equation:llipschi_hianls} in Hi-ANLS is $L$-strongly smooth with constant $L=\normtwo{\bW[:,k]}^2$. Therefore, the subproblem can be treated as a \textit{projected gradient descent (PGD)} update with a step size $\eta=\frac{1}{L}$, i.e., using gradient descent update with a step size $\eta=\frac{1}{L}$ first and projecting the update onto the nonnegative orthant afterwards \citep{lu2025practical}.

\item \label{prob:lsmooth_dslemma} \textbf{Descent lemma for $L$-strongly smooth functions.} Let $f:\sS\rightarrow (-\infty, \infty]$ be a function defined over a convex set $\sS$ such that $\normtwo{\nabla f(\bx) - \nabla f(\by)} \leq L\normtwo{\bx-\by}$ for all $\bx$ and $\by$.
Show that 
$
f(\by)\leq f(\bx)+ \nabla f(\bx)^\top (\by-\bx) + \frac{L}{2}\normtwo{\bx-\by}^2.
$
\textit{Hint: Use the fundamental theorem of calculus: $\nabla f(\bx+\alpha\bd) - \nabla f(\bx) = \int_{0}^{\alpha} \nabla^2 f(\bx+t\bd)\bd dt$.}

\item \label{prob:third_order_nmf} Let $\ba\in\real^M$, $\bz\in\real^K$, and $\bW\in\real^{K\times L}$. Show that the third-order partial derivatives of  $F(\bz)=\frac{1}{2}\normtwo{\ba-\bW\bz}^2$ vanish.

\item \textbf{MM applied to $L$-strongly smooth functions.} Let $f(\bx): \real^n\rightarrow \real$ be a $L$-strongly smooth function such that $\normtwo{\nabla f(\bx) - \nabla f(\by)} \leq L\normtwo{\bx-\by}$ for all $\bx,\by$. Show that $g(\bx, \widetildebx) = f(\widetildebx)+\nabla f(\widetildebx)^\top (\bx-\widetildebx) +\frac{L}{2} \normtwo{\bx-\widetildebx}^2$ is an auxiliary function for $f(\bx)$.
Find the update rule for this problem.

\item Derive the gradients and gradient descent updates for the tri-NMF problem in \eqref{equation:tri_nmv}.

\item \label{prob:projpro0} \textbf{Projection property-O.} Let $\sS\subset \real^n$ be \textbf{any set} and $\by\in\real^n$ such that $\widetildeby=\mathcalP_{\sS}(\by)$ is the projection of $\by$ onto set $\sS$. Show that for all $\bx\in\sS$, we have $\normtwo{\widetildeby - \by}\leq \normtwo{\bx-\by}$.

\item \label{prob:projpro1} \textbf{Projection property-I.} Let $\sS\subset \real^n$ be a \textbf{convex set} and $\by\in\real^n$ such that $\widetildeby=\mathcalP_{\sS}(\by)$. Show that for all $\bx\in\sS$, we have $\langle \bx-\widetildeby, \by-\widetildeby \rangle\leq 0$, i.e., the angle between the two vectors is greater than 90\textdegree.

\item \label{prob:projpro2} \textbf{Projection property-II.} Let $\sS\subset \real^n$ be a \textbf{convex set} and $\by\in\real^n$ such that $\widetildeby=\mathcalP_{\sS}(\by)$. Show that for all $\bx\in\sS$, we have $\normtwo{\widetildeby - \bx} \leq \normtwo{\by-\bx}$ and $\normtwo{\widetildeby - \bx}^2\leq \normtwo{\by-\bx}^2 - \normtwo{\by-\widetildeby}^2$ (the latter is related to the Pythagorean theorem). \textit{Hint: Examine $\normtwo{\by-\bx}^2=\normtwo{(\widetildeby-\bx)-(\widetildeby-\by)}^2$ and Problem~\ref{prob:projpro1}.}

\item \textbf{Linear feasibility.} Let $\sS=\{\bx\in\real^n: \bA\bx=\bb\}$ with full row rank $\bA$. Show that $\mathcalP_{\sS}(\bx) = \bx-\bA^\top(\bA\bA^\top)^{-1}(\bA\bx-\bb)$.

\item \label{prob:ab_diverg} \textbf{AB divergence \citep{amari2000methods}.} Let the \textit{$\alpha$-$\beta$ (AB) divergence} be given as follows:
\begin{align*}
d_{\alpha,\beta}(x,y)&=\begin{cases}
	-\frac{1}{\alpha\beta}(x^\alpha y^\beta-\frac{\alpha}{\alpha+\beta}x^{\alpha+\beta}-\frac{\beta}{\alpha+\beta}y^{\alpha+\beta}),&\alpha,\beta,\alpha+\beta\neq 0;\\
	\frac{1}{\alpha^2}(x^\alpha\ln(\frac{x^\alpha}{y^\alpha})-x^\alpha+y^\alpha),&\alpha\neq 0,\beta=0;\\
	\frac{1}{\alpha^2}(\ln(\frac{y^\alpha}{x^\alpha})+(\frac{y^\alpha}{x^\alpha})^{-1}-1),&\alpha=-\beta\neq 0;\\
	\frac{1}{\beta^2}(y^\beta\ln(\frac{y^\beta}{x^\beta})-y^\beta+x^\beta),&\alpha=0,\beta\neq 0;\\
	\frac{1}{2}(\ln (x)-\ln (y))^2,&\alpha=0,\beta=0.
\end{cases}
\end{align*}
When $\alpha+\beta=1$, it is called the \textit{$\alpha$-divergence}. Discuss under what conditions it reduces to the $\beta$-divergence.
Show that $d_{\alpha,\beta}(x,y)\geq 0$ and the equality holds if and only if $x=y$.

\item \label{prob:ortho_nmf} \textbf{Orthogonal and projective NMF, and clustering.} Consider the same setting as the  orthogonal or projective matrix factorization in Problem~\ref{prob:ortho_mf},  and suppose further that $\bA,\bW$, and $\bZ$ are nonnegative. 
Show that there is only one positive entry in each column of $\bZ$ in this case. 
How  is this related to the K-means problem?
When each column of $\bA$ represents a data point, discuss the interpretation of $z_{kn}$ (the $(k,n)$-th entry of $\bZ$) as the importance of the $k$-th cluster to the $n$-th data point in the projective NMF case; that is, each data point can belong to several clusters.


\item \label{prob:nonn_lin_1}  Suppose $\bA \geq \bzero_{n}$ is nonnegative~\footnote{$\bA \geq \bzero_{n}$ indicates that  $\bA$ is an $n\times n$ nonnegative matrix, and ${\bA}\geq \bzero_{m,n}$ indicates  that ${\bA}$ is an $m\times n$ nonnegative matrix.
Note that $\bC=\abs{\bA}$ is defined as the matrix obtained by setting each entry of $\bC$ as the absolute value of $\bA\in\real^{n\times n}$.
} and has a positive row. If $\abs{\bA\bx}=\bA\abs{\bx}$, where $\bx\in\complex^n$, then there exists a real $\theta\in[0, 2\pi)$ such that $e^{-i\theta}\bx = \abs{\bx}$, where $e^{-i\theta}\bx$ indicates a complex vector with $j$-th element being $e^{-i\theta} x_j$. \textit{Hint: Use triangle inequality $\abs{\bA\bx}\leq \abs{\bA}\abs{\bx}$, and examine the positive row. In the polar coordinate notation, $e^{i\theta}=\cos\theta+i\sin\theta$ and $\abs{e^{i\theta}x}=\abs{x} \implies \normtwo{e^{i\theta}\bx}=\normtwo{\bx}$.}

\item \label{prob:nnga_algebra} \textbf{Nonnegative algebra.} A bounty of results can be harvested from nonnegative conditions. We investigate several of them in this problem. Given square matrices $\bA,\bB,\bC,\bD\in\real^{n\times n}$, show that 
\begin{itemize}
\item \textbf{Triangle inequality.} $\abs{\bA\bB}\leq \abs{\bA}\abs{\bB}$.
\item \textbf{Nonexpansiveness.} $\abs{\bA^k}\leq \abs{\bA}^k$, for all $k=\{1,2,\ldots\}$.
\item \textbf{Equal norm.} $\normf{\bA}=\normf{\abs{\bA}}$.
\item If $\abs{\bB}\geq \abs{\bA}$, then $\normf{\bB}\geq \normf{\bA}$.
\item If $ \bB \geq \bA\geq \bzero$ and $ \bD \geq \bC\geq \bzero$, then $ \bB\bD \geq \bA\bC\geq \bzero$.
\item If $ \bB \geq \bA\geq \bzero$, then $\bB^k\geq \bA^k\geq \bzero$, for all $k=\{1,2,\ldots\}$,
\end{itemize}
where $\bB \geq \bA$ indicates that $\bB-\bA$ is a nonnegative matrix.
Given rectangular matrices $\bA,\bB\in\real^{m\times n}$, show that 
\begin{itemize}
\item $\abs{\bA+\bB}\leq \abs{\bA}+\abs{\bB}$.
\end{itemize}

\item$^\ast$ \label{prob:nnga_algebra2} \textbf{Eigenvalue interlacing in nonnegative matrices.}  Let   $\bB - \abs{\bA}\in\real_+^{n\times n}$ be nonnegative. Show  that 
$$
\rho(\bA) \leq \rho(\abs{\bA}) \leq \rho(\bB),
$$
where $\rho(\bX)$ represents the spectral radius of  $\bX$ (Definition~\ref{definition:spectrum}).
\textit{Hint: Use Problem~\ref{prob:nnga_algebra} and Gelfand formula; show that $\normf{\bA^k}\leq \normf{\abs{\bA}^k} \leq \normf{\bB^k}$.}

\item \label{prob:nnga_algebra3} Use Problem~\ref{prob:nnga_algebra2} to show that $\rho(\bB)\geq \rho(\bA)$ if $\bB\geq \bA\geq \bzero$.

\item Let $\bA\in\real_+^{n\times n}$ be nonnegative,  let $\bB=\bA[1:k,1:k], \,\forall k\in\{1,2,\ldots,n\}$  (i.e., any leading principal submatrix of $\bA$, Definition~\ref{definition:leading-principle-minors}), and let $\bC\in\real^{k\times k}, \,\forall k\in\{1,2,\ldots,n\}$ be any principal submatrix of $\bA$ (Definition~\ref{definition:principle-minors}). Show that 
\begin{itemize}
\item $\rho(\scriptsize\begin{bmatrix}
\bB & \bzero \\
\bzero & \bzero 
\end{bmatrix}
\normalsize
) 
\leq 
\rho (\bA)
$
$\implies\rho(\bB)\leq \rho(\bA)$.
\item Use the first result to prove $\rho(\bC)\leq \rho(\bA)$. \textit{Hint: Use permutation transformations.}
\item $\mathopmax{i=1,2,\ldots,n}a_{ii} \leq \rho(\bA)$.
\end{itemize}

\item$^\ast$  \label{prob:nonn_lin_12} Let $\bA\in\real_+^{n\times n}$ be nonnegative. Show that 
$$
\begin{aligned}
\text{Row sum: }\gap  \mathop{\min}_{1\leq i \leq n} \sum_{j=1}^{n} a_{ij} 
&\leq \rho(\bA)
\leq 
\mathop{\max}_{1\leq i \leq n} \sum_{j=1}^{n} a_{ij}; \\
\text{Column sum: }\gap \mathop{\min}_{1\leq j \leq n} \sum_{i=1}^{n} a_{ij} 
&\leq \rho(\bA)
\leq 
\mathop{\max}_{1\leq j \leq n} \sum_{i=1}^{n} a_{ij}. \\
\end{aligned}
$$

\end{problemset}

\newpage
\chapter{Biconjugate Decomposition}
\section{Existence of the Biconjugate Decomposition}
The concept of \textit{biconjugate decomposition} was introduced by \citet{chu1995rank}.
However, its underlying principle---the rank-diminishing operator---on the other hand, has roots in the work of  \citet{egervary1960rank, householder1964theory, stewart1973conjugate}.
A variety of matrix decomposition methods can be unified through  this biconjugate decomposition. 
In Section~\ref{section:conn_bi}, biconjugate decomposition is put into perspective by providing connections with standard decompositional methods, namely LDU, Cholesky, QR, and SVD decompositions.
The existence of the biconjugate decomposition is supported by the rank-one reduction theorem, as presented below.

\begin{theorem}[(Wedderburn's) rank-one reduction\index{Rank-one reduction}]\label{theorem:rank-1-reduction}
Let $\bA\in \real^{m\times n}$  be an $m\times n$ matrix of rank $r$, and let  $\bx\in \real^n$ and $\by\in \real^m$ be a pair of vectors such that $w=\by^\top\bA\bx \neq 0$. Then the matrix
\begin{equation}\label{equation:rank_reduc_operat}
\bB=\bA-w^{-1}\bA\bx\by^\top\bA
\end{equation}
has  rank  $r-1$, which is exactly one less than the rank of $\bA$, i.e., $\rank(\bB)=\rank(\bA)-1$.
\end{theorem}
A generalization of the rank-one reduction is discussed in Problem~\ref{problem:rk_reduc}.
\begin{proof}[of Theorem~\ref{theorem:rank-1-reduction}]
To prove the theorem, it suffices to show that the dimension of the null space of $\bB$ is one greater than that of $\bA$, indicating that $\bB$ has a rank exactly one less than the rank of $\bA$.

For any vector $\bn \in \nspace(\bA)$, i.e., $\bA\bn=\bzero$, we  have $\bB\bn =\bA\bn-w^{-1}\bA\bx\by^\top\bA\bn=\bzero $, implying that $\nspace(\bA)\subseteq \nspace(\bB)$.

Now, consider any vector $\bmm \in \nspace(\bB)$, i.e., $\bB\bmm=\bzero$. We have $\bB\bmm = \bA\bmm-w^{-1}\bA\bx\by^\top\bA\bmm =\bzero$.

Let $k= w^{-1}\by^\top\bA\bmm$, which is a scalar. 
Therefore, $\bB\bmm=\bA(\bmm - k\bx)=\bzero$, i.e., for any vector $\bn\in \nspace(\bA)$, we could find a vector $\bmm\in \nspace(\bB)$ such that $\bn=(\bmm - k\bx)\in \nspace(\bA)$. 
Note that $\bA\bx\neq \bzero$ based on the definition of $w$. Thus, the null space of $\bB$ is therefore obtained from the null space of $\bA$ by adding $\bx$ to its basis, which will increase the order of the space by one. 
Consequently, the dimension of $\nspace(\bA)$ is smaller than the dimension of $\nspace(\bB)$ by one, which completes the proof.
\end{proof}

The converse of the above theorem is also true, as stated  in the following corollary.
\begin{corollary}[Rank-one reduction, \citep{egervary1960rank}]\label{corollary:rk_redu_one}
Let  $\bA\in \real^{m\times n}$ be any $m\times n$ matrix, and let $\bu\in\real^m$ and $\bv\in\real^n$ be two vectors.
Then, the rank of the matrix $\bB = \bA-\sigma^{-1}\bu\bv^\top$ is less than that of $\bA$ if and only if there exist vectors $\bx\in\real^n$ and $\by\in\real^m$ such that $\bu=\bA\bx, \bv=\bA^\top\by$, and $\sigma=\by^\top\bA\bx\neq 0$. In this case, it holds that $\rank(\bB)=\rank(\bA)-1$.
\end{corollary}

More generally, the rank-one reduction can be extended to reductions involving matrices of higher rank.
\begin{corollary}[Rank-$k$ reduction, \citep{cline1979rank}]\label{corollary:rk_K_redu_one}
Let $\bA\in \real^{m\times n}$ be any $m\times n$ matrix. Let further $\bP\in\real^{m\times k}$,  $\bU\in\real^{k\times k}$ be nonsingular, and $\bQ\in\real^{n\times k}$. Then,
$$
\rank(\bA-\bP\bU^{-1}\bQ^\top) = \rank(\bA)-\rank(\bP\bU^{-1}\bQ^\top)
$$
if and  only if there exist $\bX\in\real^{n\times k}$ and $\bY\in\real^{m\times k}$ such that 
$$
\bP=\bA\bX, \gap \bQ = \bA^\top\bY, \gap \text{and} \gap\bU=\bY^\top\bA\bX.
$$
\end{corollary}

Suppose a matrix $\bA\in \real^{m\times n}$ has rank $r$. We can define a rank-reducing process to generate a sequence of matrices $\{\bA_k\}$, known as \textit{Wedderburn matrices} or \textit{Wedderburn sequence}:
\begin{equation}\label{equation:rank_reduc_proces}
\bA_1 = \bA \gap \text{and}\gap \bA_{k+1} = \bA_k-w_k^{-1}\bA_k\bx_k\by_k^\top\bA_k, \quad \forall k\in\{1,2,\ldots, r\},
\end{equation}
where $\bx_k \in \real^n$ and $\by_k\in \real^m$ are any vectors satisfying $w_k = \by_k^\top\bA_k\bx_k \neq 0$.
The operator in Equation~\eqref{equation:rank_reduc_operat} is known as a \textit{rank-diminishing operator}, and the process described by Equation~\eqref{equation:rank_reduc_proces} is referred to the \textit{rank-reducing process}.
And the sets $\{\bx_1,\bx_2,\ldots, \bx_r\}$ and $\{\by_1,\by_2,\ldots, \by_r\}$ are called the \textit{vectors associated with the rank-reducing process}. Alternatively, if we let $\bX=[\bx_1,\bx_2,\ldots, \bx_r]$ and $\bY=[\by_1,\by_2,\ldots, \by_r]$, then the pair ($\bX,\bY$) is said to  \textit{effect a rank-reducing process} for $\bA$. 

The sequence will terminate after $r$ steps since the rank of $\bA_k$ decreases by exactly one at each step. 
The sequence can be written out as follows:
$$
\begin{aligned}
	\bA_1 &= \bA, \\
\bA_1-\bA_{2} &= w_1^{-1}\bA_1\bx_1\by_1^\top\bA_1,\\
\bA_2-\bA_3 &=w_2^{-1}\bA_2\bx_2\by_2^\top\bA_2, \\
\bA_3-\bA_4 &=w_3^{-1}\bA_3\bx_3\by_3^\top\bA_3, \\
\vdots &=\vdots\\
\bA_{r-1}-\bA_{r} &=w_{r-1}^{-1}\bA_{r-1}\bx_{r-1}\by_{r-1}^\top\bA_{r-1}, \\
\bA_r-\bzero &=w_{r}^{-1}\bA_{r}\bx_{r}\by_{r}^\top\bA_{r}.
\end{aligned}
$$
By summing up the sequence, we obtain
\begin{equation}\label{equation:bi_rreduc_equa}
\begin{aligned}
\textbf{(Rank-reducing)}: &\quad (\bA_1-\bA_2)+(\bA_2-\bA_3)+\ldots+(\bA_{r-1}-\bA_{r})+(\bA_r-\bzero ) \\
&=\bA_1=\bA= \sum_{i=1}^{r}w_i^{-1}\bA_i\bx_i\by_i^\top\bA_i.
\end{aligned}
\end{equation}
Therefore, we can derive the following decomposition directly from this rank-reducing process.
\begin{theoremHigh}[Biconjugate decomposition: form 1\index{Decomposition: Biconjugate}]\label{theorem:biconjugate-form1}
Let $\bA\in\real^{m\times n}$ any matrix of rank $r$.
This equality~\eqref{equation:bi_rreduc_equa},  derived from the rank-reducing process, implies the following matrix decomposition
$$
\bA = \bPhi \bOmega^{-1} \bPsi^\top,
$$
where $\bOmega=\diag(w_1, w_2, \ldots, w_r)$, $\bPhi=[\bphi_1,\bphi_2, \ldots, \bphi_r]\in \real^{m\times r}$, and $\bPsi=[\bpsi_1, \bpsi_2, \ldots, \bpsi_r]$ with 
$$
\bphi_k = \bA_k\bx_k 
\qquad \text{and}\qquad 
\bpsi_k=\bA_k^\top \by_k, 
\qquad \forall k\in\{1,2,\ldots,r\}.
$$
\end{theoremHigh}
Thus, different choices of the vectors $\bx_k$ and $\by_k$ will result in different biconjugate factorizations, making this factorization  quite general and  versatile. In the following sections, we will explore its connections to several well-known matrix factorizations.

\index{Wedderburn sequence}

\begin{remark}\label{remark:weded_space}
Regarding the vectors $\bx_k$ and $\by_k$ in the Wedderburn sequence, the following orthogonality properties hold:
$$
\begin{aligned}
\bx_k &\in \nspace(\bA_{k+1})      \quad \implies \quad   \bx_k\bot \cspace(\bA_{k+1}^\top), \\
\by_k &\in \nspace(\bA_{k+1}^\top) \quad \implies \quad   \by_k\bot \cspace(\bA_{k+1}).
\end{aligned}
$$
To verify this, observe that:
$$
\begin{aligned}
\bA_{k+1} \bx_k 
&=  (\bA_k-w_k^{-1}\bA_k\bx_k\by_k^\top\bA_k) \bx_k\\
&=  \bA_k(\bx_k - \frac{\by_k^\top\bA_k\bx_k}{\by_k^\top\bA_k\bx_k}\bx_k)=\bzero,
\end{aligned}
$$
since $w_k=\by_k^\top\bA_k\bx_k \neq 0$. Hence, $\bx_k \in \nspace(\bA_{k+1}) $.
Similarly, it can be shown that $ \bA_{k+1}^\top\by_k=\bzero$.
\end{remark}

\begin{lemma}[General term formula of Wedderburn sequence: V1]\label{lemma:wedderburn-sequence-general}
Let $\bA\in\real^{m\times n}$ be any matrix of rank $r$, and let $\bA_1=\bA$.
For each matrix in the sequence defined by $\bA_{k+1} = \bA_k -w_k^{-1} \bA_k\bx_k\by_k^\top \bA_k$ ($k\in\{1,2,\ldots,r-1\}$), the matrix $\bA_{k+1} $ can be expressed as 
$$
\bA_{k+1} = \bA - \sum_{i=1}^{k}w_{i}^{-1} \bA\bu_i \bv_i^\top \bA,
\qquad \forall k\in\{1,2,\ldots,r-1\},
$$
where 
$$
\bu_k=\bx_k -\sum_{i=1}^{k-1}\frac{\bv_i^\top \bA\bx_k}{w_i}\bu_i
\qquad \text{and}\qquad 
\bv_k=\by_k -\sum_{i=1}^{k-1}\frac{\by_k^\top \bA\bu_i}{w_i}\bv_i.
$$
Let $\bX=[\bx_1,\bx_2,\ldots,\bx_r]$, $\bY=[\by_1,\by_2,\ldots,\by_r]$, $\bU=[\bu_1,\bu_2,\ldots,\bu_r]$, and $\bV=[\bv_1,\bv_2,\ldots,\bv_r]$ be the column partitions for each set of vectors.
Then, the rank-reducing process can be viewed as transforming the  matrix pair $(\bX,\bY)$ into the pair $(\bU,\bV)$.~\footnote{It can be shown that if $\bA$ is symmetric and $\bX=\bY$, then $\bU=\bV$.}
\end{lemma}

The proof of this lemma is deferred to Section~\ref{section:wedderburn-general-term}. We notice that $w_i =\by_i^\top\bA_i\bx_i$ in the general term formula is related to $\bA_i$, which means the expression is not the true general term formula. 
We will later reformulate $w_i$ in terms of the original matrix $\bA$ rather than $\bA_i$.
From the general term formula of the  Wedderburn sequence, we have: 
$$
\begin{aligned}
\bA_{k+1} &= \bA - \sum_{i=1}^{k}w_{i}^{-1} \bA\bu_i \bv_i^\top \bA, \\
\bA_{k} &= \bA - \sum_{i=1}^{k-1}w_{i}^{-1} \bA\bu_i \bv_i^\top \bA.
\end{aligned}
$$
Subtracting these two equations yields: $\bA_{k+1} - \bA_{k} = -w_{k}^{-1} \bA\bu_k \bv_k^\top \bA$. 
Since  the sequence is defined as $\bA_{k+1} = \bA_k -w_k^{-1} \bA_k\bx_k\by_k^\top \bA_k$, we can deduce that $w_{k}^{-1} \bA\bu_k \bv_k^\top \bA = w_k^{-1} \bA_k\bx_k\by_k^\top \bA_k$. 
Consequently, it follows that
\begin{equation}\label{equation:wedderburn-au-akxk}
\begin{aligned}
\bA\bu_k &=\bA_k\bx_k,\\
\bv_k^\top \bA&=\by_k^\top \bA_k.
\end{aligned}
\end{equation}
Let $z_{k,i} = \frac{\bv_i^\top \bA\bx_k}{w_i}$, which is a scalar. Referring to the definitions of $\bu_k$ and $\bv_k$ in the  lemma above, we can express them explicitly as follows:
\begin{itemize}
\item $\bu_1=\bx_1$; 

\item $\bu_2 = \bx_2 - z_{2,1}\bu_1 \implies $ $\bx_2$ is a linear combination of $\bu_1$ and $\bu_2$;

\item $\bu_3 = \bx_3 - z_{3,1}\bu_1-z_{3,2}\bu_2 \implies$ $\bx_3$ is a linear combination of $\bu_1,\bu_2$, and $\bu_3$;

\item $\ldots$.
\end{itemize}
Each coefficient $z_{k,i}$ ($i< k$) encodes the component of $\bx_k$ in that of $\bu_i$.
This process bears resemblance to the Gram--Schmidt process (Section~\ref{section:gram-schmidt-process}).
However, in this process,  we do not perform an \textbf{orthogonal projection} of $\bx_2$ onto $\bx_1$ to find the vector component of $\bx_2$ along $\bx_1$, as we would do in an orthogonal projection (Section~\ref{section:qr-gram-compute}). 
Instead, the vector of $\bx_2$ along $\bx_1$ is now defined by $z_{2,1}$ (i.e., an \textbf{oblique projection}; see Section~\ref{section:qr-gram-compute}). 
This process is illustrated in Figure~\ref{fig:projection-wedd}. 

In Figure~\ref{fig:project-line-wedd}, $\bu_2$ is \textbf{not} perpendicular to $\bu_1$ (in the Gram--Schmidt process, $\bu_2$ would be perpendicular to $\bu_1$ via orthogonal projections).
Nevertheless, $\bu_2$ does not lie on the same line as $\bu_1$, so $\{\bu_1, \bu_2\}$  can still  span a $\real^2$ subspace. Similarly, in Figure~\ref{fig:project-space-wedd}, $\bu_3= \bx_3 - z_{3,1}\bu_1-z_{3,2}\bu_2$ does not lie in the subspace spanned by $\{\bu_1, \bu_2\}$, allowing $\{\bu_1, \bu_2, \bu_3\}$  to span a $\real^3$ subspace.
\index{Gram--Schmidt}

A moment of reflexion would reveal that the span of $\{\bx_2, \bx_1\}$ is the same as the span of $\{\bu_2, \bu_1\}$. This equivalence extends to the $\bv_i$ vectors and $\by_i$ vectors as well.  
We can express this property as follows:
\begin{equation}\label{equation:wedderburn-span-same}
\left\{
\begin{aligned}
\spn\{\bx_1, \bx_2, \ldots, \bx_j\} &= \spn\{\bu_1, \bu_2, \ldots, \bu_j\},\gap \forall j\in\{1,2,\ldots, r\};\\
\spn\{\by_1, \by_2, \ldots, \by_j\} &= \spn\{\bv_1, \bv_2, \ldots, \bv_j\},\gap \forall j\in\{1,2,\ldots, r\}.\\
\end{aligned}
\right.
\end{equation}

\begin{figure}[h]
\centering  
\vspace{-0.35cm} 
\subfigtopskip=2pt 
\subfigbottomskip=2pt 
\subfigcapskip=-5pt 
\subfigure[``Project" onto a line.]{\label{fig:project-line-wedd}
\includegraphics[width=0.47\linewidth]{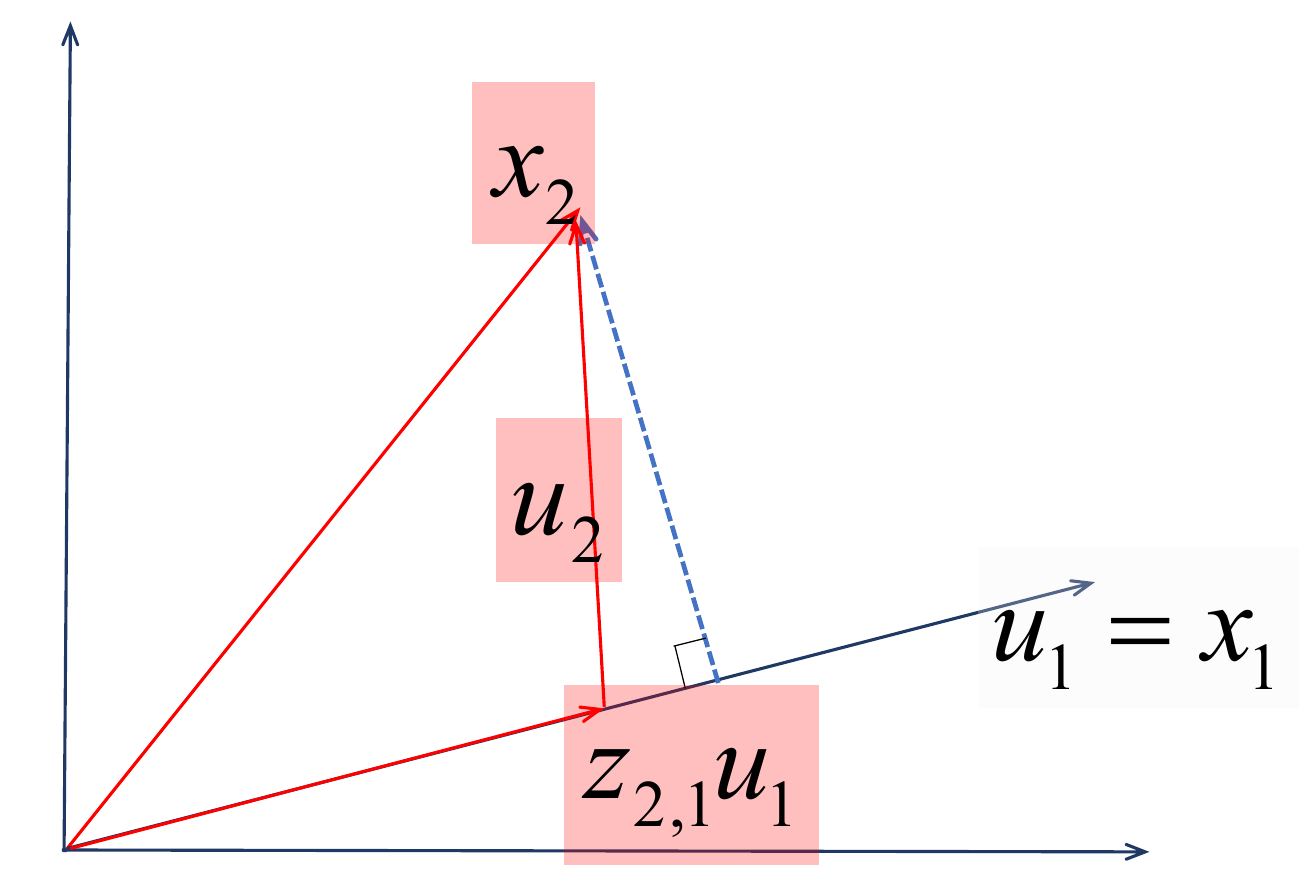}}
\quad 
\subfigure[``Project" onto a space.]{\label{fig:project-space-wedd}
\includegraphics[width=0.47\linewidth]{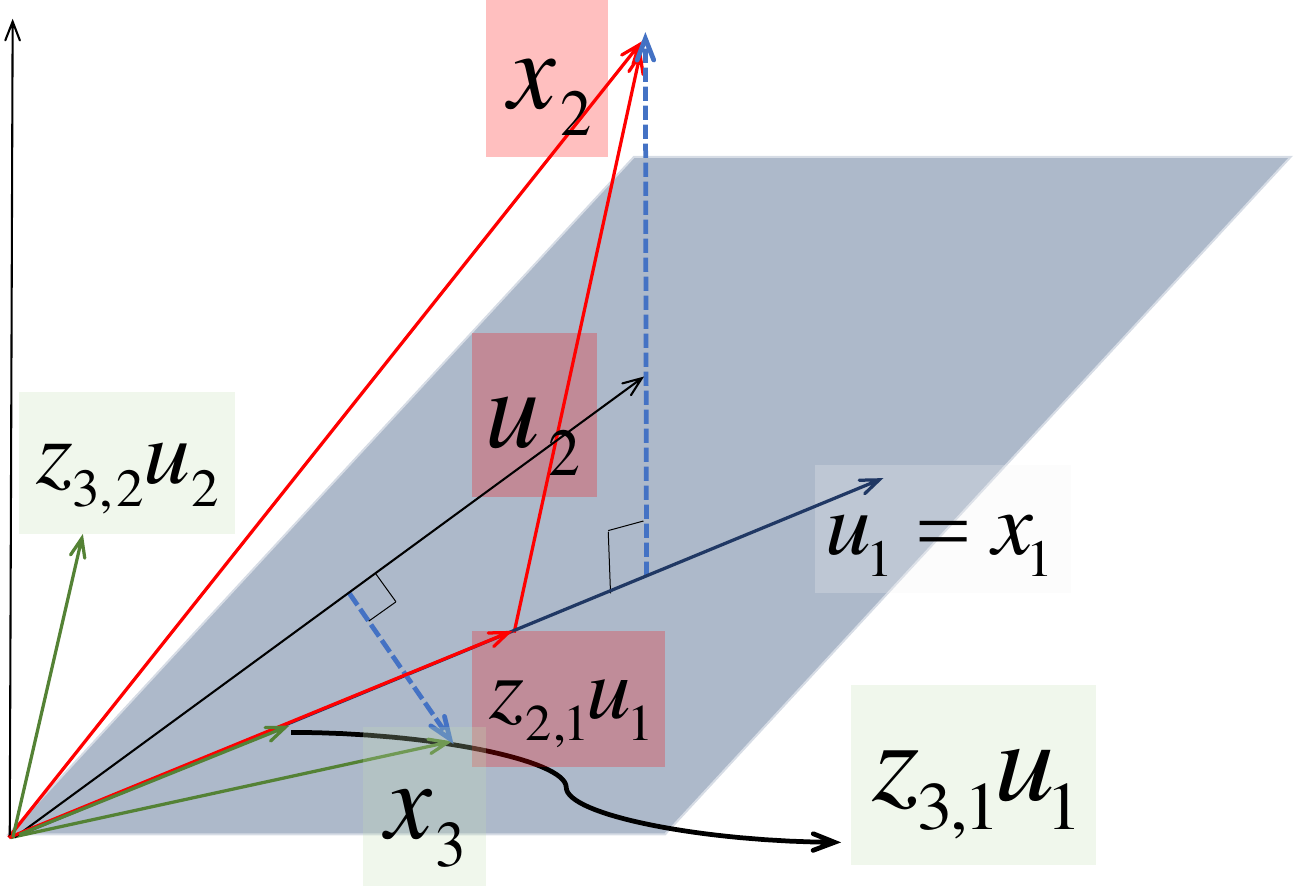}}
\caption{``Project" a vector onto a line and onto a space. Compare with the Gram--Schmidt process in Figure~\ref{fig:gram-schmidt-12}.}
\label{fig:projection-wedd}
\end{figure}

Furthermore, from the rank-reducing property of the Wedderburn sequence, we have the following nested inclusions:
$$
\left\{
\begin{aligned}
	\cspace(\bA_1) &\supset \cspace(\bA_2) \supset \cspace(\bA_3) \supset \ldots;\\
	\nspace(\bA_1^\top) &\subset \nspace(\bA_2^\top) \subset \nspace(\bA_3^\top) \subset \ldots.
\end{aligned}
\right.
$$
Since $\by_k \in \nspace(\bA_{k+1}^\top)$, it then follows that $\by_j \in \nspace(\bA_{k+1}^\top)$ for all $j<k+1$, i.e., $\bA_{k+1}^\top \by_j=\bzero$ for all $j<k+1$. 
This also holds true for $\bx_{k+1}^\top \bA_{k+1}^\top \by_j=0$ for all $j<k+1$. From Equation~\eqref{equation:wedderburn-au-akxk}, we also have $\bu_{k+1}^\top \bA^\top \by_j=0$ for all $j<k+1$. Following Equation~\eqref{equation:wedderburn-span-same}, we obtain 
\begin{equation}
\bx_{k+1}^\top \bA_{k+1}^\top \by_j=0
\,\,\stackrel{\eqref{equation:wedderburn-au-akxk}}{\implies}\,\,
\bu_{k+1}^\top \bA^\top \by_j=0
\,\,\stackrel{\eqref{equation:wedderburn-span-same}}{\implies}\,\,
\bv_j^\top\bA\bu_{k+1}=0\quad  \text{ for all } j<k+1.
\end{equation}
Similarly, we can prove 
\begin{equation}
\bv_{k+1}^\top\bA\bu_{j}=0\quad \text{ for all } j<k+1.
\end{equation}
Moreover, since $w_k = \by_k^\top \bA_k\bx_k$, according to Equation~\eqref{equation:wedderburn-au-akxk}, we can express  $w_k$ as:
$$
\begin{aligned}
w_k &= \by_k^\top \bA_k\bx_k
=\bv_k^\top \bA\bx_k \\
&=\bv_k^\top\bA (\bu_k +\sum_{i=1}^{k-1}\frac{\bv_i^\top \bA\bx_k}{w_i}\bu_i) \qquad &(\text{by the definition of }\bu_k \text{ in Lemma~\ref{lemma:wedderburn-sequence-general}})\\
&=\bv_k^\top\bA \bu_k,\qquad &(\text{by } \bv_{k}^\top\bA\bu_{j}=0 \text{ for all } j<k)
\end{aligned}
$$
which can be utilized to substitute the value of $w_k$ in Lemma~\ref{lemma:wedderburn-sequence-general}. We then have the full version of the general term formula of the Wedderburn sequence. 
In this form, the formula no longer depends on matrices $\bA_k$ (in the form of $w_k$'s): 
\begin{equation}\label{equation:uk-vk-to-mimic-gram-process}
\bu_k=\bx_k -\sum_{i=1}^{k-1}\frac{\bv_i^\top \bA\bx_k}{\textcolor{mylightbluetext}{\bv_i^\top\bA\bu_i}}\bu_i
\qquad \text{and}\qquad 
\bv_k=\by_k -\sum_{i=1}^{k-1}\frac{\by_k^\top \bA\bu_i}{\textcolor{mylightbluetext}{\bv_i^\top\bA\bu_i}}\bv_i.
\end{equation}

\index{Gram--Schmidt}
\paragraph{Gram--Schmidt process from Wedderburn sequence.}
Suppose the matrices $\bX=[\bx_1, 
\bx_2, \ldots, \bx_r]\in \real^{n\times r}$ and $\bY=[\by_1, \by_2, \ldots, \by_r]\in \real^{n\times r}$ effect a rank-reducing process for $\bA\in\real^{n\times n}$. 
If $\bA=\bI\in\real^{n\times n}$ is the identity matrix and $(\bX=\bY)$ are identical, containing the vectors for which an orthogonal basis is desired (i.e., we aim to obtain an orthogonal basis spanning the same column space of $\bX=\bY$), then  the result of the rank-reducing process, $(\bU = \bV)$, gives the resultant orthogonal basis (but not an orthonormal basis as that in the Gram--Schmidt process):
\begin{itemize}
\item To see this, we follow the computation of the Gram--Schmidt process as given in Equation~\eqref{equation:qr-gsp-equation}, where we replace the data matrix with $\bX=\bY\in\real^{m\times n}$ (assuming $\bX$ has full rank for simplicity), and we use the matrix $\bQ=[\bq_1,\bq_2, \ldots, \bq_n]$ to denote the orthonormal basis.
Then, for any $k\in\{1,2,\ldots,n\}$, the Gram--Schmidt process admits
\begin{equation}\label{equation:bicon_cgs}
	\text{Gram--Schmidt process}=\left\{
	\begin{aligned}
		\bx_k^\perp&= \left(\bI-\sum_{i=1}^{k-1}\bq_i\bq_i^\top\right)\bx_k;\\
		\bq_k &= \frac{\bx_k^\perp}{\norm{\bx_k^\perp}}.
	\end{aligned}
	\right.
\end{equation}
\item Considering the rank-reducing process with $\bA=\bI$, $\bX=\bY$. Since $\bA$ is symmetric and $\bX=\bY$, we obtain $\bU=\bV$. We can therefore focus on the analysis of $\bU$. For any $k\in\{1,2,\ldots,n\}$, the rank-reducing process yields 
\begin{equation}\label{equation:bicon_wedd}
	\text{Rank-reducing process}=\left\{
	\begin{aligned}
		\bu_k&=\bx_k -\sum_{i=1}^{k-1}\frac{\bv_i^\top \bA\bx_k}{\bv_i^\top\bA\bu_i}\bu_i
		=\left(\bI -\sum_{i=1}^{k-1}\frac{\bu_i\bu_i^\top }{\bu_i^\top\bu_i}\bu_i\right)\bx_k ;\\
		\widetilde{\bu}_k &= \frac{\bu_k}{\norm{\bu_k}}.
	\end{aligned}
	\right.
\end{equation}
\end{itemize}
Therefore, by comparing Equations~\eqref{equation:bicon_cgs} and \eqref{equation:bicon_wedd}, we can find the equivalence between $\bq_k$ and $\widetilde{\bu}_k$ for $k\in\{1,2,\ldots,n\}$. Thus, the Wedderburn sequence effects a Gram--Schmidt process when $\bX=\bY$ is the data matrix and $\bA=\bI$.


\paragraph{``Projection" notations.}
The expressions for $\bu_k$ and $\bv_k$ in Equation~\eqref{equation:uk-vk-to-mimic-gram-process} closely resemble the projection onto the perpendicular space in the Gram--Schmidt process, as shown in Equation~\eqref{equation:gram-schdt-eq2}. We then define the bilinear form \fbox{$<\bx, \by>=\by^\top\bA\bx$} to explicitly mimic the form of projection in Equation~\eqref{equation:gram-schdt-eq2}.

We consolidate the results established so far into the following lemma, which provides a concise overview of what we have been working on. These results will be extensively utilized  in the sequel.
\begin{lemma}[Properties of Wedderburn sequence]\label{lemma:wedderburn-sequence-general-v2}
Let $\bA\in\real^{m\times n}$ be any matrix of rank $r$, and let $\bA_1=\bA$.
For each matrix in the sequence defined by $\bA_{k+1} = \bA_k -w_k^{-1} \bA_k\bx_k\by_k^\top \bA_k$ ($k\in\{1,2,\ldots,r-1\}$), the matrix $\bA_{k+1} $ can be expressed as 
$$
\bA_{k+1} = \bA - \sum_{i=1}^{k}w_{i}^{-1} \bA\bu_i \bv_i^\top \bA,
$$
where 
\begin{equation}\label{equation:properties-of-wedderburn-ukvk}
\bu_k=\bx_k -\sum_{i=1}^{k-1}\frac{\textcolor{mylightbluetext}{<\bx_k, \bv_i>}}{\textcolor{mylightbluetext}{<\bu_i,\bv_i>}}\bu_i
\qquad \text{and}\qquad 
\bv_k=\by_k -\sum_{i=1}^{k-1}\frac{\textcolor{mylightbluetext}{<\bu_i,\by_k>}}{\textcolor{mylightbluetext}{<\bu_i,\bv_i>}}\bv_i.
\end{equation}
Furthermore, we can observe the following properties:
\begin{equation}\label{equation:wedderburn-au-akxk-2}
\begin{aligned}
	\bA\bu_k &=\bA_k\bx_k;\\
	\bv_k^\top \bA&=\by_k^\top \bA_k;
\end{aligned}
\end{equation}
\begin{equation}\label{equation:wedderburn-au-akxk-33}
<\bu_k, \bv_j>=<\bu_j, \bv_k>=0 \text{ for all } j<k;
\end{equation}
\begin{equation}\label{equation:wk-by-ukvk}
w_k = \by_k^\top\bA_k\bx_k = <\bu_k, \bv_k>.
\end{equation}

\end{lemma}

By substituting Equation~\eqref{equation:wedderburn-au-akxk-2} into Form 1 of the biconjugate decomposition and using Equation~\eqref{equation:wk-by-ukvk}, which implies $w_k = \bv_k^\top\bA\bu_k$, we obtain the Form 2 and Form 3 of this decomposition:
\begin{theoremHigh}[Biconjugate decomposition: Form 2 and Form 3]\label{theorem:biconjugate-form2}
Let $\bA\in\real^{m\times n}$ be any matrix of rank $r$.
This equality~\eqref{equation:bi_rreduc_equa}, which results from the rank-reducing process, implies the following matrix decomposition:
\begin{equation}\label{equation:bicon_form2_1}
\bA = \bA\bU_r \bOmega_r^{-1} \bV_r^\top\bA,
\end{equation}
where $\bOmega_r=\diag(w_1, w_2, \ldots, w_r)$, $\bU_r=[\bu_1,\bu_2, \ldots, \bu_r]\in \real^{n\times r}$, and  $\bV_r=[\bv_1, \bv_2,$ $\ldots,$ $\bv_r] \in \real^{m\times r}$ with  
\begin{equation}\label{equation:properties-of-wedderburn-ukvk2-inform2}
\bu_k=\bx_k -\sum_{i=1}^{k-1}\frac{<\bx_k, \bv_i>}{<\bu_i,\bv_i>}\bu_i
\qquad \text{and}\qquad 
\bv_k=\by_k -\sum_{i=1}^{k-1}\frac{<\bu_i,\by_k>}{<\bu_i,\bv_i>}\bv_i.
\end{equation}
Additionally, for any $\gamma \leq r$, the following decomposition holds:
\begin{equation}\label{equation:wedderburn-vgamma-ugamma}
\bV_\gamma^\top \bA \bU_\gamma = \bOmega_\gamma,
\end{equation}
where $\bOmega_\gamma=\diag(w_1, w_2, \ldots, w_\gamma)$, $\bU_\gamma=[\bu_1,\bu_2, \ldots, \bu_\gamma]\in \real^{n\times \gamma}$, and $\bV_\gamma=[\bv_1, \bv_2,$ $\ldots,$ $\bv_\gamma]\in \real^{m\times \gamma}$. Note the difference between the subscripts $r$ and $\gamma$  employed here, where $\gamma \leq r$.
\end{theoremHigh}
Note that Equation~\eqref{equation:bicon_form2_1} is derived from  \eqref{equation:wedderburn-au-akxk-2}, 
and Equation~\eqref{equation:wedderburn-vgamma-ugamma}  is a consequence of \eqref{equation:wedderburn-au-akxk-33}. 
Importantly, these two forms of the biconjugate decomposition no longer depend on the intermediate Wedderburn matrices  $\{\bA_k\}$. 

\paragraph{Notation.} 
In the following discussion, we will use subscripts to indicate the dimensions of matrices to avoid ambiguity. For example, the use of $r$ and $\gamma$ in the above theorem highlights the size of the constructed matrices.

\section{Properties of the Biconjugate Decomposition}
The following corollary  establishes a connection between matrices $\bU_\gamma$ and $\bX_\gamma$ through unique unit upper triangular matrices derived from the Wedderburn sequence.

\begin{corollary}[Connection of $\bU_\gamma$ and $\bX_\gamma$]\label{corollary:biconjugate-connection-u-x}
Let $\bA\in\real^{m\times n}$ be any matrix of rank $r\geq \gamma$.
If $(\bX_\gamma, \bY_\gamma) \in \real^{n\times \gamma}\times \real^{m\times \gamma}$ effects a rank-reducing process for $\bA$, then there exist unique unit upper triangular matrices $\bR_\gamma^{(x)}\in\real^{\gamma\times \gamma}$ and $\bR_\gamma^{(y)}\in\real^{\gamma\times \gamma}$ such that 
$$
\bX_\gamma = \bU_\gamma \bR_\gamma^{(x)}
\qquad \text{and} \qquad 
\bY_\gamma=\bV_\gamma\bR_\gamma^{(y)},
$$
where $\bU_\gamma$ and $\bV_\gamma$ are matrices whose columns are derived from the Wedderburn sequence, as described in Equation~\eqref{equation:wedderburn-vgamma-ugamma}.
\end{corollary} 
\begin{proof}[of Corollary~\ref{corollary:biconjugate-connection-u-x}]
The proof follows directly from the definitions of $\bu_k$ and $\bv_k$ in Equations~\eqref{equation:properties-of-wedderburn-ukvk} or \eqref{equation:properties-of-wedderburn-ukvk2-inform2}. We construct   the $j$-th columns of $\bR_\gamma^{(x)}$ and $\bR_\gamma^{(y)}$ as follows:
$$
\left[\frac{<\bx_j,\bv_1>}{<\bu_1,\bv_1>}, \frac{<\bx_j,\bv_2>}{<\bu_2,\bv_2>}, \ldots, \frac{<\bx_j,\bv_{j-1}>}{<\bu_{j-1},\bv_{j-1}>}, 1, 0, 0, \ldots, 0  \right]^\top,
$$
and 
$$
\left[\frac{<\bu_1, \by_j>}{<\bu_1,\bv_1>},\frac{<\bu_2, \by_j>}{<\bu_2,\bv_2>}, \ldots, \frac{<\bu_{j-1}, \by_j>}{<\bu_{j-1},\bv_{j-1}>}, 1, 0, 0, \ldots, 0  \right]^\top.
$$
And the uniqueness stems  from the fact that the matrices $\bU_\gamma$ and $\bV_\gamma$  have independent columns from the rank-reducing process.
This completes the proof.
\end{proof}

The pair $(\bU_\gamma, \bV_\gamma) \in \real^{m\times \gamma}\times \real^{n\times \gamma}$ in Theorem~\ref{theorem:biconjugate-form2} is called a \textbf{biconjugate pair} with respect to $\bA$ if $\bOmega_\gamma$ is nonsingular and diagonal. 
Furthermore, suppose the pair $(\bX_\gamma, \bY_\gamma) \in \real^{n\times \gamma}\times \real^{m\times \gamma}$ effects a rank-reducing process for $\bA$. 
Then, the pair $(\bX_\gamma, \bY_\gamma)$ is said to be \textbf{biconjugatable} and can be \textbf{biconjugated into a biconjugate pair} of matrices $(\bU_\gamma, \bV_\gamma)$, if there exist unit upper triangular matrices $\bR_\gamma^{(x)}$ and $\bR_\gamma^{(y)}$ such that $\bX_\gamma = \bU_\gamma \bR_\gamma^{(x)}$ and $\bY_\gamma=\bV_\gamma\bR_\gamma^{(y)}$.

\section{Connection to Well-Known Decomposition Methods}\label{section:conn_bi}
In this section, we demonstrate how biconjugate decomposition relates to well-known matrix factorization methods.

\subsection{LDU Decomposition}
\index{Decomposition: LDU}
\begin{theorem}[LDU, \cite{chu1995rank} Theorem 2.4]\label{theorem:biconjugate-ldu}
Let $\bA\in\real^{m\times n}$ be any matrix of rank $r\geq \gamma$.
Let further $(\bX_\gamma, \bY_\gamma) \in \real^{n\times \gamma}\times \real^{m\times \gamma}$ 
with $\gamma\in\{1, 2, \ldots, r\}$. Then, the pair  $(\bX_\gamma, \bY_\gamma)$ is biconjugatable if and only if $\bY_\gamma^\top\bA\bX_\gamma$ admits an LDU decomposition. 
\end{theorem}
\begin{proof}[of Theorem~\ref{theorem:biconjugate-ldu}]
Suppose $\bX_\gamma$ and $\bY_\gamma$ are biconjugatable. Then there exist unit upper triangular matrices $\brx$ and $\bry$ such that $\bxgamma = \bugamma\brx$, $\bygamma = \bvgamma\bry$, and $\bvgamma^\top\bA\bugamma = \bomegagamma$ is a nonsingular diagonal matrix. It follows that 
$$
\bygamma^\top\bA\bxgamma = \bryt \bvgamma^\top \bA \bugamma\brx = \bryt \bomegagamma \brx
$$
is the unique  LDU decomposition of $\bygamma^\top\bA\bxgamma$. This expression can be regarded as the \textbf{fourth form of biconjugate decomposition}.

Conversely, suppose $\bygamma^\top\bA\bxgamma = \bR_2^\top \bD\bR_1$ is an LDU decomposition, with both $\bR_1$ and $\bR_2$ being unit upper triangular matrices. Since the inverses $\bR_1^{-1}$ and $\bR_2^{-1}$ are also unit upper triangular matrices, the pair $(\bX_\gamma, \bY_\gamma)$ can be biconjugated into $(\bX_\gamma\bR_1^{-1}, \bY_\gamma\bR_2^{-1})$.
This completes the proof.
\end{proof}

\index{Determinant}
\index{Principal minor}
\begin{corollary}[Determinant]\label{corollary:lu-determinant}
Let $\bA\in\real^{m\times n}$ be any matrix of rank $r\geq \gamma$.
Suppose the pair $(\bX_\gamma, \bY_\gamma) \in \real^{n\times \gamma}\times \real^{m\times \gamma}$ can be biconjugated into $(\bU_\gamma, \bV_\gamma)$ such that $\bvgamma^\top\bA\bugamma = \bomegagamma=\diag(w_1,w_2,\ldots,w_\gamma)$ is a nonsingular diagonal matrix. 
Then it follows that
$$
\det(\bygamma^\top \bA\bxgamma) = \prod_{i=1}^{\gamma} w_i.
$$
\end{corollary}
\begin{proof}[of Corollary~\ref{corollary:lu-determinant}]
By Theorem~\ref{theorem:biconjugate-ldu}, since $(\bX_\gamma, \bY_\gamma)$ are biconjugatable, then there exist unit upper triangular matrices $\brx$ and $\bry$ such that $\bygamma^\top\bA\bxgamma = \bryt \bomegagamma \brx$. The determinant is simply the product of the diagonal elements.
\end{proof}

\begin{lemma}[Biconjugatable in principal minors]\label{lemma:Biconjugatable-in-Principal-Minors}
Let $\bA\in\real^{m\times n}$ be any matrix of rank $r\geq \gamma$.
In the Wedderburn sequence, we choose $\bx_i$ as the $i$-th standard basis in $\real^n$ for $i \in \{1, 2, \ldots, \gamma\}$ (i.e., $\bx_i = \be_i \in \real^n$), and $\by_i$ as the $i$-th standard basis in $\real^m$ for $i \in \{1, 2, \ldots, \gamma\}$ (i.e., $\by_i=\be_i \in 
\real^m$). 
That is, $\bygamma^\top \bA\bxgamma$ corresponds to the leading principal submatrix of $\bA$, i.e., $\bygamma^\top \bA\bxgamma = \bA[1:\gamma, 1:\gamma]$. 
Then, $(\bX_\gamma, \bY_\gamma)$ is biconjugatable  into $(\bU_\gamma, \bV_\gamma)$ such that $\bvgamma^\top\bA\bugamma = \bomegagamma=\diag(w_1,w_2,\ldots,w_\gamma)$ is a nonsingular diagonal matrix
if and only if the $\gamma$-th leading principal minor of $\bA$ is nonzero, i.e.,  $\det(\bA[1:\gamma, 1:\gamma])\neq 0$. In this case, the $\gamma$-th leading principal minor of $\bA$ is given by $\prod_{i=1}^{\gamma} w_i$.\index{Leading principal minor}
\end{lemma}
\begin{proof}[of Lemma~\ref{lemma:Biconjugatable-in-Principal-Minors}]
The proof is straightforward that the $\gamma$-th leading principal minor of $\bA$ being nonzero will imply that $w_i \neq 0$ for all $i\leq \gamma$. Thus, the Wedderburn sequence can be successfully obtained. The converse holds because Corollary~\ref{corollary:lu-determinant} implies that $\det(\bygamma^\top \bA\bxgamma)$ is nonzero.
\end{proof}

We have now arrived at the LDU decomposition for square matrices. 
\begin{theorem}[LDU: Biconjugate decomposition for square matrices]\label{theorem:biconjugate-square-ldu}
For any matrix $\bA\in \real^{n\times n}$, the pair $(\bI_n, \bI_n)$ is biconjugatable if and only if all the leading principal minors of $\bA$ are nonzero. In this case, $\bA$ can be factored as 
$$
\bA = \bV_n^{-\top} \bOmega_n \bU_n^{-1} = \bL\bD\bU,
$$
where $\bOmega_n = \bD$ is a diagonal matrix with nonzero values along its diagonal, $\bV_n^{-\top} = \bL$ is a unit lower triangular matrix, and $\bU_n^{-1} = \bU$ is a unit upper triangular matrix.
\end{theorem}
\begin{proof}[of Theorem~\ref{theorem:biconjugate-square-ldu}]
As per Lemma~\ref{lemma:Biconjugatable-in-Principal-Minors}, it is evident that the pair $(\bI_n, \bI_n)$ is biconjugatable. 
Based on Corollary~\ref{corollary:biconjugate-connection-u-x}, we have $\bU_n \bR_n^{(x)} = \bI_n$ and $\bI_n=\bV_n\bR_n^{(y)}$. Thus, $\bR_n^{(x)} = \bU_n^{-1}$ and $\bR_n^{(y)} = \bV_n^{-1}$ are well defined. This completes the proof.
\end{proof}

\subsection{Cholesky Decomposition}
For symmetric and positive definite matrices, all leading principal minors are always positive. 
The proof for this statement can be found in Section~\ref{appendix:leading-minors-pd}.
The following theorem shows how the Cholesky decomposition arises naturally from biconjugate decomposition in the case of positive definite matrices.
\begin{theorem}[Cholesky: Biconjugate decomposition for PD matrices]\label{theorem:biconjugate-square-cholesky}
For any symmetric and positive definite matrix $\bA\in \real^{n\times n}$, the Cholesky decomposition of $\bA$ can be derived from the Wedderburn sequence by setting $(\bX_n, \bY_n)$ as the pair $(\bI_n, \bI_n)$.
In this case, $\bA$ can be factored as 
$$
\bA = \bU_n^{-\top} \bOmega_n \bU_n^{-1} = (\bU_n^{-\top} \bOmega_n^{1/2})( \bOmega_n^{1/2} \bU_n^{-1}) =\bR^\top\bR,
$$
where $\bOmega_n$ is a diagonal matrix with positive values along the diagonal, and $\bU_n^{-1}$ is a unit upper triangular matrix.
\end{theorem}
\begin{proof}[of Theorem~\ref{theorem:biconjugate-square-cholesky}]
Given that the leading principal minors of positive definite matrices are positive, $w_i>0$ for all $i\in \{1, 2,\ldots, n\}$. 
It follows from the LDU factorization via biconjugation and the symmetry of $\bA$ that $\bA = \bU_n^{-\top} \bOmega_n \bU_n^{-1}$.  
Since $w_i$'s are positive,  $\bOmega_n$ is positive definite and  can be factored as $\bOmega_n = \bOmega_n^{1/2}\bOmega_n^{1/2}$. This implies that $\bOmega_n^{1/2} \bU_n^{-1}$ is the Cholesky factor. 
\end{proof}

\subsection{QR Decomposition}

Without loss of generality, we  assume that $\bA\in \real^{n\times n}$ has full rank, which allows for the QR decomposition: $\bA=\bQ\bR$, where $\bQ\in\real^{n\times n}$ is orthogonal, and $\bR \in \real^{n\times n}$ is upper triangular with  full rank and positive diagonal values. 
We now show how this decomposition can be obtained through biconjugate decomposition.
\begin{theorem}[QR: Biconjugate decomposition for nonsingular matrices]\label{theorem:biconjugate-square-qr}
For any nonsingular matrix $\bA\in \real^{n\times n}$, the QR decomposition of $\bA$ can be obtained from the Wedderburn sequence by setting $(\bX_n, \bY_n)$ as $(\bI_n, \bA)$.
Thus, $\bA$ can be factored as 
$$
\bA = \bQ\bR,
$$
where $\bQ=\bV_n \bOmega_n^{-1/2}$ is an orthogonal matrix, and $\bR = \bOmega_n^{1/2}\bR_n^{(x)}$ is an upper triangular matrix, according to the \textcolor{black}{\textbf{Form 4}} in Theorem~\ref{theorem:biconjugate-ldu}, with $\gamma=n$:
$$
\bY_n^\top\bA\bX_n = \bR_n^{(y)\top} \bV_n^\top \bA \bU_n\bR_n^{(x)}=\bR_n^{(y)\top} \bOmega_n \bR_n^{(x)},
$$
where we set $\gamma=n$ because $\gamma$ can be any value such that $\gamma\leq r$, and the rank $r=n$.
\end{theorem}

\begin{proof}[of Theorem~\ref{theorem:biconjugate-square-qr}]
Since $(\bX_n, \bY_n) = (\bI_n, \bA)$, applying Theorem~\ref{theorem:biconjugate-ldu}, we have the decomposition
$
\bY_n^\top\bA\bX_n = \bR_n^{(y)\top} \bV_n^\top \bA \bU_n\bR_n^{(x)}=\bR_n^{(y)\top} \bOmega_n \bR_n^{(x)}.
$
Substituting $(\bX_n,\bY_n)=(\bI_n, \bA)$ into the  decomposition above, we obtain: 
\begin{equation}\label{equation:biconjugate-qr-ata1}
\begin{aligned}
	\bY_n^\top\bA\bX_n &= \bR_n^{(y)\top} \bV_n^\top \bA \bU_n\bR_n^{(x)} = \bR_n^{(y)\top} \bOmega_n \bR_n^{(x)};\\
	\bA^\top \bA &= \bR_n^{(y)\top} \bOmega_n \bR_n^{(x)}; \\
	\bA^\top \bA &= \bR_1^\top \bOmega_n \bR_1; \qquad (\text{$\bA^\top\bA$ is symmetric and let $\bR_1=\bR_n^{(x)}=\bR_n^{(y)}$})\\
	\bA^\top \bA &= (\bR_1^\top \bOmega_n^{1/2\top}) (\bOmega_n^{1/2}\bR_1);\\
	\bA^\top \bA &= \bR^\top\bR. \qquad \qquad\qquad\qquad(\text{let $\bR = \bOmega_n^{1/2}\bR_1$})
\end{aligned}
\end{equation}
To see why $\bOmega_n$ can be factored as $\bOmega_n = \bOmega_n^{1/2\top}\bOmega_n^{1/2}$, we consider the following steps. 
Suppose $\bA=[\ba_1, \ba_2, \ldots, \ba_n]$ is the column partition of $\bA$. We obtain $w_i = \by_i^\top\bA\bx_i=\ba_i^\top \ba_i>0$, since $\bA$ is nonsingular. Therefore, $\bOmega_n=\diag(w_1, w_2, \ldots, w_n)$ is positive definite and it can be factored as 
\begin{equation}\label{equation:omega-half-qr}
\bOmega_n = \bOmega_n^{1/2}\bOmega_n^{1/2}= \bOmega_n^{1/2\top}\bOmega_n^{1/2}.
\end{equation}
By $\bxgamma = \bugamma\brx$ in Theorem~\ref{theorem:biconjugate-ldu} for all $\gamma\in \{1, 2, \ldots, n\}$, we have 
$$
\begin{aligned}
\bX_n &= \bU_n\bR_1; \\
\bI_n &= \bU_n\bR_1; \qquad (\text{Since $\bX_n = \bI_n$}) \\
\bU_n &= \bR_1^{-1}.
\end{aligned}
$$
By $\bygamma = \bvgamma\bry$ in Theorem~\ref{theorem:biconjugate-ldu}  for all $\gamma\in \{1, 2, \ldots, n\}$, we have
$$
\begin{aligned}
\bY_n &= \bV_n\bR_1;\\
\bA &= \bV_n\bR_1; \qquad &(\text{$\bA=\bY_n$}) \\
\bA^\top \bA &= \bR_1^\top\bV_n^\top \bV_n\bR_1; \\
\bR_1^\top \bOmega_n \bR_1&=\bR_1^\top\bV_n^\top \bV_n\bR_1; \qquad &(\text{Equation~\eqref{equation:biconjugate-qr-ata1}})\\
(\bR_1^\top \bOmega_n^{1/2\top}) (\bOmega_n^{1/2}\bR_1) &= (\bR_1^\top \bOmega_n^{1/2\top} \bOmega_n^{-1/2\top})\bV_n^\top \bV_n (\bOmega_n^{-1/2}\bOmega_n^{1/2}\bR_1); \qquad &\text{(Equation~\eqref{equation:omega-half-qr})} \\
\bR^\top\bR &= \bR^\top (\bOmega_n^{-1/2\top} \bV_n^\top) (\bV_n \bOmega_n^{-1/2})  \bR.
\end{aligned}
$$
Thus, $\bQ=\bV_n \bOmega_n^{-1/2}$ is an orthogonal matrix.
\end{proof}

\subsection{SVD}\label{section:bicon_SVD}
To explore  the SVD of a square matrix $\bA\in\real^{n\times n}$ within the biconjugation decomposition, we introduce the following notation: let $\bA=\bU^\svd \bSigma^\svd \bV^{\svd\top}$ be the SVD of $\bA$, where $\bU^\svd = [\bu_1^\svd, \bu_2^\svd, \ldots, \bu_n^\svd]$ is orthogonal, $\bV^\svd = [\bv_1^\svd, \bv_2^\svd, \ldots, \bv_n^\svd]$ is orthogonal, and $\bSigma^\svd = \diag(\sigma_1, \sigma_2, \ldots, \sigma_n)$ is diagonal. Without loss of generality, we assume $\bA\in \real^{n\times n}$ and $\rank(\bA)=n$. Readers can verify the equivalence for a general matrix $\bA\in \real^{m\times n}$.

If the pair ($\bX_n=\bV^\svd$, $\bY_n=\bU^\svd$) effects a rank-reducing process for $\bA$.
From the definitions of $\bu_k$ and $\bv_k$ in Equation~\eqref{equation:properties-of-wedderburn-ukvk} or Equation~\eqref{equation:properties-of-wedderburn-ukvk2-inform2}, we have
$$
\bu_k = \bv_k^\svd \qquad \text{and} \qquad \bv_k = \bu_k^\svd \qquad \text{and} \qquad w_k = \by_k^\top \bA \bx_k=\sigma_k.
$$
This implies $\bV_n = \bU^\svd$, $\bU_n = \bV^\svd$, and $\bOmega_n = \bSigma^\svd$, where we set $\gamma=n$ because $\gamma$ can be any value such that $\gamma\leq r$, and the rank $r=n$.

By $\bX_n= \bU_n\bR_n^{(x)}$ in Theorem~\ref{theorem:biconjugate-ldu}, we have 
$$
\bX_n = \bU_n\bR_n^{(x)}   \quad\implies\quad
\bV^\svd = \bV^\svd\bR_n^{(x)}	 \quad\implies\quad
\bI_n = \bR_n^{(x)}.
$$
By $\bY_n = \bV_n\bR_n^{(y)}$ in Theorem~\ref{theorem:biconjugate-ldu}, we have
$$
\bY_n = \bV_n\bR_n^{(y)}    \quad\implies\quad
\bU^\svd = \bU^\svd\bR_n^{(y)} \quad\implies\quad
\bI_n=\bR_n^{(y)}.
$$
Applying Theorem~\ref{theorem:biconjugate-ldu} again and setting $\gamma=n$, we have
$$
\bY_n^\top\bA\bX_n = \bR_n^{(y)\top} \bV_n^\top \bA \bU_n\bR_n^{(x)}=\bR_n^{(y)\top} \bOmega_n \bR_n^{(x)}.
$$
This simplifies to
$
\bU^{\svd\top}\bA \bV^\svd = \bSigma^\svd,
$
which corresponds precisely to the form of a SVD. This demonstrates the equivalence between the SVD and the biconjugate decomposition when the Wedderburn sequence is applied with $(\bV^\svd, \bU^\svd)$ as $(\bX_n, \bY_n)$.

\section{Proof: General Term Formula of Wedderburn Sequence}\label{section:wedderburn-general-term}
In Lemma~\ref{lemma:wedderburn-sequence-general}, we present the general term formula for the Wedderburn sequence.
Given any matrix $\bA\in\real^{m\times n}$, the Wedderburn sequence of $\bA$ is defined recursively  by $\bA_{k+1} = \bA_k -w_k^{-1} \bA_k\bx_k\by_k^\top \bA_k$ with $\bA_1 = \bA$. The proof of the general term formula for this sequence is as follows:
\begin{proof}[of Lemma~\ref{lemma:wedderburn-sequence-general}]
For $\bA_2$, let $\bu_1=\bx_1$ and $\bv_1 =\by_1$. We have:
$$
\begin{aligned}
\bA_2 &=\bA_1 -w_1^{-1} \bA_1\bx_1\by_1^\top \bA_1=\bA -w_1^{-1} \bA\bu_1\bv_1^\top \bA.
\end{aligned}
$$
For $\bA_3$, we can write out the equation as:
$$
\footnotesize
\begin{aligned}
&\bA_3 = \bA_2 -w_2^{-1} \bA_2\bx_2\by_2^\top \bA_2 \\
&=(\bA -w_1^{-1} \bA\bu_1\bv_1^\top \bA) - w_2^{-1}(\bA -w_1^{-1} \bA\bu_1\bv_1^\top \bA)\bx_2\by_2^\top(\bA -w_1^{-1} \bA\bu_1\bv_1^\top \bA) \qquad &\text{(substitute $\bA_2$)}\\
&=(\bA -w_1^{-1} \bA\bu_1\bv_1^\top \bA) 
- w_2^{-1}\textcolor{mylightbluetext}{\bA}(\textcolor{mylightbluetext}{\bx_2 }-w_1^{-1} \bu_1\bv_1^\top \bA\textcolor{mylightbluetext}{\bx_2 })(\textcolor{mylightbluetext}{\by_2^\top} -w_1^{-1}\textcolor{mylightbluetext}{\by_2^\top} \bA\bu_1\bv_1^\top )\textcolor{mylightbluetext}{\bA} \qquad &\text{(factor out $\bA$)}\\
&=\bA -w_1^{-1} \bA\bu_1\bv_1^\top \bA - w_2^{-1}\bA\bu_2\bv_2^\top\bA
=\bA -\sum_{i=1}^{2}w_i^{-1} \bA\bu_i\bv_i^\top \bA,
\end{aligned}
$$
where $\bu_2=\bx_2 -w_1^{-1} \bu_1\bv_1^\top \bA\bx_2=\bx_2 -\frac{\bv_1^\top \bA\bx_2}{w_1}\bu_1$, and $\bv_2=\by_2 -w_1^{-1}\by_2^\top \bA\bu_1\bv_1=\by_2 -\frac{\by_2^\top \bA\bu_1}{w_1}\bv_1$.
Similarly, we can find the expression of $\bA_4$ by $\bA$:
$$
\footnotesize
\begin{aligned}
&\bA_4  = \bA_3 -w_3^{-1} \bA_3\bx_3\by_3^\top \bA_3 \\
&=\bA -\sum_{i=1}^{2}w_i^{-1} \bA\bu_i\bv_i^\top \bA 
 - w_3^{-1} \big(\bA -\sum_{i=1}^{2}w_i^{-1} \bA\bu_i\bv_i^\top \bA\big)\bx_3\by_3^\top \big(\bA -\sum_{i=1}^{2}w_i^{-1} \bA\bu_i\bv_i^\top \bA\big) \gap &\text{{(substitute $\bA_3$)}}\\
&=\bA -\sum_{i=1}^{2}w_i^{-1} \bA\bu_i\bv_i^\top \bA 
- w_3^{-1}\textcolor{mylightbluetext}{\bA} \big(\textcolor{mylightbluetext}{\bx_3} -\sum_{i=1}^{2}w_i^{-1} \bu_i\bv_i^\top \bA\textcolor{mylightbluetext}{\bx_3}\big) \big(\textcolor{mylightbluetext}{\by_3^\top} -\sum_{i=1}^{2}w_i^{-1}\textcolor{mylightbluetext}{\by_3^\top} \bA\bu_i\bv_i^\top \big)\textcolor{mylightbluetext}{\bA} \gap &\text{(factor out $\bA$)} \\
&=\bA -\sum_{i=1}^{2}w_i^{-1} \bA\bu_i\bv_i^\top \bA - w_3^{-1}\bA\bu_3\bv_3^\top\bA
=\bA -\sum_{i=1}^{3}w_i^{-1} \bA\bu_i\bv_i^\top \bA,
\end{aligned}
$$	
where $\bu_3=\bx_3 -\sum_{i=1}^{2}\frac{\bv_i^\top \bA\bx_3}{w_i}\bu_i$, and $\bv_3=\by_3 -\sum_{i=1}^{2}\frac{\by_3^\top \bA\bu_i}{w_i}\bv_i$.
Continuing this process, we can define 
$$
\bu_k=\bx_k -\sum_{i=1}^{k-1}\frac{\bv_i^\top \bA\bx_k}{w_i}\bu_i
\qquad \text{and} \qquad
\bv_k=\by_k -\sum_{i=1}^{k-1}\frac{\by_k^\top \bA\bu_i}{w_i}\bv_i,
$$ 
and  the general term of the Wedderburn sequence can be proved by induction.
\end{proof}

\begin{problemset}
\item Following the proof of Theorem~\ref{theorem:rank-1-reduction}, prove Corollarys~\ref{corollary:rk_redu_one} and \ref{corollary:rk_K_redu_one}.

\item Discuss the Wedderburn sequence of $(\bX_n,\bY_n)=(\bI_n, \bA)$ in Theorem~\ref{theorem:biconjugate-square-qr} if $\bA$ is singular.

\item Following Section~\ref{section:bicon_SVD}, verify the equivalence between the SVD and the biconjugate decomposition for a general matrix $\bA$ of size $m\times n$.

\item \label{problem:rk_reduc} \textbf{Rank reduction theorem.} Let $\bA\in\real^{m\times n}$, $\bX\in\real^{n\times k}$, and $\bY\in\real^{m\times k}$ such that $\bW=\bY^\top\bA\bX$ is nonsingular. Show that
$$
\rank(\bA-\bA\bX\bW^{-1}\bY^\top\bA) = \rank(\bA)-\rank(\bA\bX\bW^{-1}\bY^\top\bA).
$$
When $k=1$, this is the rank-one reduction (Theorem~\ref{theorem:rank-1-reduction}).
Discuss how this general result relates to  Corollary~\ref{corollary:rk_K_redu_one}.

\item Show that if $\bA$ is symmetric and $\bX=\bY$, then $\bU=\bV$ in Lemma~\ref{lemma:wedderburn-sequence-general}.
\end{problemset}

\newpage
\bibliography{bib}

\clearpage
\printindex


\end{document}